\documentclass[11pt]{amsart}
\usepackage{amsmath,amssymb}
\usepackage{graphicx}
\usepackage{amsfonts,amscd,latexsym,epsfig, psfrag}

\usepackage{hyperref}
\usepackage[all]{xy}

\newcommand{\Aa}{{\mathcal A}}

\newcommand{\Mor}{{\rm Mor}}
\newcommand{\Obj}{{\rm Obj}}

\newcommand{\ul}{\underline}
\renewcommand{\Hat}{\widehat}

\newcommand{\HBb}{\Hat{\Bb}}

\newcommand{\PSL}{{\rm PSL}}
\newcommand{\pbar}{{\ov {\p}_J}}
\newcommand{\Cc}{{\mathcal C}}
\newcommand{\Kk}{{\mathcal K}}

\newcommand{\Jj}{{\mathcal J}}

\newcommand{\im}{{\rm im\,}}
\newcommand{\less}{{\smallsetminus}}

\newcommand{\bla}{{\bigl\langle}}
\newcommand{\bra}{{\bigl\rangle}}

\newcommand{\supp}{{\rm supp\,}}

\newcommand{\p}{{\partial}}
\newcommand{\al}{{\alpha}}

\newcommand{\be}{{\beta}}

\newcommand{\om}{{\omega}}
\newcommand{\eps}{{\varepsilon}}
\newcommand{\de}{{\delta}}
\newcommand{\De}{{\Delta}}
\newcommand{\ga}{{\gamma}}
\newcommand{\Ga}{{\Gamma}}
\newcommand{\io}{{\iota}}
\newcommand{\ka}{{\kappa}}
\newcommand{\la}{{\lambda}}
\newcommand{\La}{{\Lambda}}
\newcommand{\si}{{\sigma}}
\newcommand{\Si}{{\Sigma}}

\newcommand{\Uu}{{\mathcal U}}
\newcommand{\Bb}{{\mathcal B}}
\newcommand{\Ww}{{\mathcal W}}

\newcommand{\Mm}{{\mathcal M}}

\newcommand{\Tt}{{\mathcal T}}
\newcommand{\oMm}{{\overline {\Mm}}}
\newcommand{\ov}{\overline}

\newcommand{\id}{{\rm id}}
\newcommand{\rd}{{\rm d}}
\newcommand{\rT}{{\rm T}}
\newcommand{\Exp}{{\rm Exp}}
\renewcommand{\Tilde}{\widetilde}

\newcommand{\coker}{{\rm coker\,}}

\newcommand{\Ee}{{\mathcal E}}
\newcommand{\Ii}{{\mathcal I}}

\newcommand{\Vv}{{\mathcal V}}

\newcommand{\N}{{\mathbb N}}

\newcommand{\Q}{{\mathbb Q}}
\newcommand{\R}{{\mathbb R}}
\newcommand{\C}{{\mathbb C}}

\newcommand{\Z}{{\mathbb Z}}

\newcommand{\Hom}{{\rm Hom}}

\newcommand{\Nn}{{\mathcal N}}
\newcommand{\Pp}{{\mathcal P}}

\newcommand{\ev}{{\rm ev}}

\newcommand{\bB}{{\bf B}}
\newcommand{\bC}{{\bf C}}
\newcommand{\bG}{{\bf G}}

\newcommand{\bE}{{\bf E}}
\newcommand{\bK}{{\bf K}}
\newcommand{\bX}{{\bf X}}
\newcommand{\bz}{{\bf z}}
\newcommand{\bZ}{{\bf Z}}

\newtheorem{theorem}{Theorem}[subsection]
\newtheorem{thm}[theorem]{Theorem}

\newtheorem{lemma}[theorem]{Lemma}

\newtheorem{prop}[theorem]{Proposition}
\newtheorem{definition}[theorem]{Definition}
\newtheorem{defn}[theorem]{Definition}
\newtheorem{example}[theorem]{Example}

\newtheorem{remark}[theorem]{Remark}
\newtheorem{rmk}[theorem]{Remark}

\numberwithin{figure}{subsection}
\numberwithin{equation}{subsection}

\newcommand{\MS}{{\medskip}}

\newcommand{\NI}{{\noindent}}

\newcommand{\ti}{\tilde}
\newcommand{\Ti}{\widetilde}

\newcommand{\gr}{\operatorname{graph}}
\newcommand{\pr}{{\rm pr}}

\newcommand{\lm}{\Lambda^{\rm max}\,}

\newenvironment{enumlist}
   { \begin{list} {}
         {  \setlength{\itemsep}{.5ex} \setlength{\leftmargin}{0ex} } }
   { \end{list} }

   \newcounter{qcounter}

\newenvironment{itemlist}
   { \begin{list} {$\bullet$}
         {  \setlength{\itemsep}{.5ex} \setlength{\leftmargin}{2.5ex} } }
   { \end{list} }

\newcommand*{\longhookleftarrow}{\ensuremath{\leftarrow\joinrel\relbar\joinrel\rhook}}
\newcommand*{\longhookrightarrow}{\ensuremath{\lhook\joinrel\relbar\joinrel\rightarrow}}

\newcommand{\leftsub}[2]{{\vphantom{#2}}_{#1}{#2}}

\newcommand\quotient[2]{
        \mathchoice
            {
                \text{\raise1ex\hbox{$#1$}\Big/\lower1ex\hbox{$#2$}}%
            }
            {
                #1\,/\,#2
            }
            {
                #1\,/\,#2
            }
            {
                #1\,/\,#2
            }
    }

\newcommand\quot[2]{
                \text{\raise1ex\hbox{$#1$}/\lower1ex\hbox{$\scriptstyle#2$}}
  }

\newcommand\quo[2]{
                \text{\raise1ex\hbox{$#1\!\!$}/\lower1ex\hbox{$\!\scriptstyle#2$}}
  }

\newcommand\qu[2]{
                \text{\raise.8ex\hbox{$\scriptstyle#1\!$}/\lower.8ex\hbox{$\!\scriptstyle#2$}}
  }

\newcommand\qq[2]{
                \text{\raise.8ex\hbox{$#1\!$}/\lower.8ex\hbox{$#2$}}
}

 \title[Smooth Kuranishi atlases with trivial isotropy]{Smooth Kuranishi atlases with trivial isotropy}
 \author{Dusa McDuff}
\author{Katrin Wehrheim}\thanks{partially supported by NSF grants  
DMS 0905191 and DMS 0844188}

\begin{document}
\maketitle

\tableofcontents
\section{Introduction}

Kuranishi structures were introduced to symplectic topology by Fukaya and Ono \cite{FO},
and recently refined by Joyce \cite{J}, in order to extract homological data from compactified moduli spaces of holomorphic maps in cases where geometric regularization approaches such as
perturbations of the almost complex structure do not yield a smooth structure on the moduli space.
These geometric methods generally cannot handle curves that are nowhere injective.  The first instance in which it was important to overcome these limitations
was the case of nowhere injective spheres, which are then multiply covered and
have nontrivial isotropy.\footnote
{This is not the case for discs. For example a disc with boundary on the equator can wrap two and a half times around the sphere.
This holomorphic curve, called the lantern, has trivial isotropy. }  Because of this, the development of virtual transversality techniques in \cite{FO}, and the related work by Li and Tian \cite{LT}, was focussed on dealing with finite isotropy groups, while some topological and analytic issues were 
not resolved. 

The goal of this paper is to explain these issues and provide the beginnings of a framework for resolving them.
To that end we focus on the most fundamental issues, which are already present in applying virtual transversality techniques to moduli spaces of holomorphic spheres without nodes or nontrivial isotropy.
We give a general survey of regularization techniques in symplectic topology in Section~\ref{s:fluff}, pointing to some general analytic issues in Sections~\ref{ss:geom} --\ref{ss:kur}, and discussing the specific 
algebraic and
topological issues of the Kuranishi approach in Sections~\ref{ss:alg} and \ref{ss:top}.
In the main body of the paper we provide an abstract framework of Kuranishi atlases which separates the analytic and topological issues as outlined in the following.

The main analytic issue in each regularization approach is in the construction of transition maps for a given moduli space, where one has to deal with the lack of differentiability of the reparametrization action on infinite dimensional function spaces discussed in Section~\ref{s:diff}.
When building a Kuranishi atlas on a moduli space, this issue also appears in a sum construction for basic charts on overlaps, and has to be dealt with separately for each specific moduli space.
We explain the construction of basic Kuranishi charts, their sums, and transition maps in the case of spherical Gromov--Witten moduli spaces in Section~\ref{s:construct}, outlining the proof of a more precise version of the following in Proposition~\ref{prop:A1} and Theorem~\ref{thm:A2}.
\MS

\NI {\bf Theorem A.}\,\,{\it  Let $(M,\om,J)$ be a symplectic manifold with tame almost complex structure, and let $\Mm(A,J)$ be the space of simple $J$-holomorphic maps $S^2\to M$ in class $A$ with one marked point, modulo reparametrization. 
If $\Mm(A,J)$ is compact (e.g.\ if $A$ is 
\lq\lq $\omega$-minimal"), 
then there exists an open cover
$\Mm(A,J)= \bigcup_{i=1,\ldots,N} F_i$ by ``footprints'' of basic Kuranishi charts $(\bK_i)_{i=1,\ldots,N}$.

Moreover, for any tuple $(\bK_i)_{i\in I}$ of basic charts whose obstruction spaces $E_i$ satisfy a ``transversality condition'' we can construct transition data as follows.
There exists a ``sum chart'' $\bK_I$ with obstruction space  $\prod_{i\in I}E_i$ and footprint  $F_{I}=\bigcap_{i\in I} F_i \subset \Mm(A,J)$, such that a restriction of each basic chart $\bK_i|_{F_I}$ includes into $\bK_I$ by a coordinate change.}\MS

The notions of a Kuranishi chart, a restriction, and a coordinate change are defined in detail in Section~\ref{s:chart}. Throughout, we simplify the discussion by assuming that all isotropy groups are trivial.
In that special case our definitions largely follow \cite{FO,J}, though instead of working with germs, we define a Kuranishi atlas as a covering family of basic charts together with transition data satisfying a cocycle condition involving an inclusion requirement on the domains of the coordinate changes.
In addition to Theorem~A, this requires the construction of further coordinate changes satisfying the cocycle condition.
At this point one could already use ideas of \cite{LiuT} to reduce the cover and construct compatible transverse perturbations of the sections in each Kuranishi chart.
However, there is no guarantee that the perturbed zero set modulo transition maps is a closed manifold, in particular Hausdorff.
This is an essential requirement in the construction of a {\it virtual moduli cycle}
(VMC for short), which, as its name indicates, is a cycle in an appropriate homology theory
 representing  the {\it virtual fundamental class}  (or VFC)  $[X]^{vir}_\Kk$ of $X$.
We remedy this situation in the main technical part of this paper by constructing a {\it virtual neighbourhood} of the moduli space, with paracompact Hausdorff topology, in which the perturbed zero set modulo transition maps is a closed subset.
The construction of this virtual neighbourhood requires a tameness property of the domains of the coordinate changes, in particular a strong cocycle condition which requires equality of domains.
However, the coordinate changes arising from sum constructions as in Theorem~A can usually only be made to satisfy a weak cocycle condition on the overlap of domains.
On the other hand, these constructions naturally provide an additivity property for the obstruction spaces. In Sections~\ref{ss:tame} and \ref{ss:Kcobord} we develop these notions,
proving in Propositions~\ref{prop:proper} and~\ref{prop:cobord2}
that a suitable shrinking of such an additive weak
Kuranishi atlas induces a tame Kuranishi atlas that is well defined up to cobordism.
We can then ``reduce" the Kuranishi atlas in order to facilitate the construction of sections.
As a result, in Section~\ref{s:VMC} we obtain a precise version of the following.\footnote{To see why we use rational, rather than integer, \v{C}ech homology see Remark~\ref{rmk:Cech}.
}

\MS

\NI {\bf Theorem B.}\,\,{\it
Let $\Kk$ be an oriented,  $d$-dimensional, 
weak, additive Kuranishi atlas with trivial isotropy groups on a compact metrizable space $X$. Then $\Kk$ determines a cobordism class of smooth, oriented, compact manifolds, and an element 
$[X]^{vir}_\Kk$ in the \v{C}ech homology group $\check{H}_d(X;\Q)$.
Both depend only on the cobordism class of $\Kk$.}
\MS

Making our constructions applicable to general holomorphic curve moduli spaces will
require two generalizations of Theorem~B, which we are working on in \cite{MW:ku2}.
Firstly, a groupoid version of the theory (extending ideas from \cite{FO,J}), will allow for nontrivial isotropy groups. Secondly, allowing the structure maps in the Kuranishi category to be stratified smooth rather than smooth permits the construction of Kuranishi charts using standard gluing analysis, as established e.g.\ in \cite{MS}. With this framework in place, we hope to extend Theorem~A to giving a completely detailed construction of a Kuranishi atlas for spherical Gromov--Witten invariants in \cite{MW:gw}, 
by combining the ideas of finite dimensional reduction in \cite{FO} with the explicit obstruction bundles in \cite{LT} and the analytic results in \cite{MS}.
\MS

\noindent
{\bf Organization:} The following remarks together with Sections~\ref{s:fluff} and \ref{s:diff} provide a survey of regularization techniques in symplectic topology and their pitfalls. Section~\ref{s:construct} continues this discussion for the specific example of Kuranishi atlases for genus zero Gromov--Witten moduli spaces, and also outlines an approach to proving Theorem~A.
All of these sections are essentially self-contained and can be read in any order.
The main technical parts of the paper, Sections~\ref{s:chart}, \ref{s:Ks}, and  \ref{s:VMC} , are independent of the previous sections, but strongly build on each other towards a proof of Theorem~B. 
Much of the work here concerns the topological underpinnings of the theory.
An introduction and outline for these technical parts can be found in Sections~\ref{ss:kur} and \ref{ss:top}.

\smallskip

Since this project revisits fifteen year old, much used theories, let us explain some motivations, relations to other work, and give an outlook for further development. 

\smallskip
\noindent
{\bf Background:}
In a 2009 talk at MSRI \cite{w:msritalk}, KW posed some basic questions on the
currently available abstract transversality approaches.
DM, who had been uneasily aware for some time of analytic problems with the approach  of Liu--Tian~\cite{LiuT}, the basis of her expository article \cite{Mcv}, decided that now was the time to clarify the constructions once and for all.
The issue here is the lack of differentiability of the reparametrization action, which enters if constructions are transferred between different infinite dimensional local slices of the action, or if a differentiable Banach manifold structure on a quotient space of maps by reparametrization is assumed. The same issue is present for Deligne--Mumford type spaces of domains and maps, and is discussed in detail in Section~\ref{s:diff}.
In studying the construction of a virtual moduli cycle for Gromov--Witten invariants via a Kuranishi atlas we soon encountered the same differentiability issue.

Extrapolating remarks by Aleksey Zinger and Kenji Fukaya, we next realized that the geometric construction of obstruction spaces in \cite{LT} could be used to construct smoothly compatible Kuranishi charts.
However, in making these constructions explicit, we needed to resolve ambiguities in the definition of a Kuranishi structure, concerning the precise meaning of germ of coordinate changes and the cocycle condition, discussed in Section~\ref{ss:alg}.
More generally, we found it difficult to find a definition of Kuranishi structure that on the one hand clearly has a virtual fundamental class (a generalization of Theorem B), and on the other hand arises from fairly simple analytic techniques for holomorphic curves (a generalization of Theorem A).
One issue that we will touch on only briefly in Section~\ref{ss:approach} is the lack of smoothness of the standard gluing constructions, which affects the smoothness of the Kuranishi charts near nodal or broken curves.
A more fundamental topological issue is the necessity to ensure that the zero set of a transverse perturbation is not only smooth, but also remains compact as well as Hausdorff, and does not acquire boundary in the regularization.
These properties, as far as we are aware nowhere previously addressed in the literature, are crucial for obtaining a global triangulation and thus a fundamental homology class.
Another topological issue is the necessity of refining the cover by Kuranishi charts to a ``good cover'', in which the unidirectional transition maps allow for an iterative construction of perturbations.
(These topological issues will be discussed in Section~\ref{ss:top}.)
So we decided to start from the very basics and give a definition of Kuranishi atlas and a completely explicit construction of a VFC in the simplest nontrivial case.

\smallskip
\noindent
{\bf Relation to other Kuranishi notions:}
As this work was nearing completion, we alerted Fukaya et al and Joyce to some of the issues we had uncovered, and the ensuing discussion eventually resulted in \cite{FOOO12} and some parts of \cite{JD}.
To clarify the relation between these approaches and ours, we now use the language of atlases, which in fact is more descriptive.
While the previous definitions of Kuranishi structures in \cite{FO,J} are algebraically inconsistent as explained in Section~\ref{ss:alg}, our approach is compatible with the notions of \cite{FOOO,FOOO12}.  Indeed we show in Remark~\ref{rmk:otherK} how to construct a Kuranishi structure in this sense from a weak Kuranishi atlas.  
Similarly, when there is no isotropy there is a relation between the ideas behind a ``good coordinate system" and our notion of a reduction. However, as will be clear when \cite{MW:ku2} is completed, the two approaches differ significantly when there is nontrivial isotropy.  
One can make an analogy with the development of the theory of orbifolds:  The approach of \cite{FOOO12} is akin to Satake's definition of a $V$-manifold, while our definitions are much closer to  the idea of describing an orbifold  as the realization of an \'etale proper groupoid.
  
In our view, weak atlases in the sense of Definition~\ref{def:Kwk} are the natural outcomes of constructions of compatible finite dimensional reductions, and we see a clear abstract path from an atlas to a VMC.   Constructing a weak atlas involves checking only a finite number of consistency conditions for the coordinate changes, while uncountably many such conditions must be checked if one tries to construct a Kuranishi structure directly.  
We will further compare our approach to that of Fukaya et al in Remark~\ref{rmk:otherK}.

Another approach to constructing virtual fundamental classes from local finite dimensional reductions is proposed in \cite{JD} using so-called ``d-orbifolds'' which have more algebraic properties than Kuranishi structures. We cannot comment on the details of this approach apart from noting that it does not offer a direct approach to regularizing moduli spaces, but instead seems to require special types of Kuranishi structures or polyfold Fredholm sections as a starting point.\footnote{
\cite[Thm.15.6]{JD} claims ``virtual class maps'' for d-orbifolds under an additional ``(semi)effectiveness'' assumption, which in our understanding could only be obtained from the constructions of \cite[Thm.16.1]{JD} under the assumption of obstruction spaces invariant under the isotropy action. However, this is almost a case of equivariant transversality in which e.g.\ the polyfold setup should allow for a global finite dimensional reduction as smooth section of an orbifold bundle along the lines of \cite{DZ}. 
}

\smallskip
\noindent
{\bf Outlook:}
Based on our results in the present paper, we have a good idea of how to generalize and apply these constructions to all genus zero Gromov--Witten moduli spaces.
We believe that Kuranishi atlases for other moduli spaces of closed holomorphic curves
can be constructed analogously, though each case would require a geometric construction of
local slices as well as
obstruction bundles specific to the setup, and careful gluing analysis.
An alternative route is provided by the construction of Kuranishi structures from a proper Fredholm section in a polyfold bundle, as announced in \cite{DZ}.
In fact, this approach would induce a smooth, rather than stratified smooth Kuranishi structure.

The case of moduli spaces with boundary given as the fiber product of other moduli spaces, as required for the construction of $A_\infty$-structures, is beyond the scope of our project.
While finishing this manuscript, we learned that Jake Solomon \cite{Sol} has been developing an approach to dealing with boundaries and making ``pull-push'' construction as required for chain level theories.
Such a framework will need to generalize our notion of Kuranishi cobordism on $X\times [0,1]$ to underlying moduli spaces with a much less natural ``boundary and corner stratification'', in particular facing a further complication in the already highly nontrivial construction of relative perturbations in Proposition~\ref{prop:ext2}.
Moreover, it has to solve the additional task of constructing regularizations that respect the fiber product structure on the boundary. This issue, also known as  constructing coherent perturbations, has to be addressed separately in each specific geometric setting, and requires a hierarchy of moduli spaces which permits one to construct the regularizations iteratively.
In the construction of the Floer differential on a finitely generated complex, such an iteration can be performed using an energy filtration thanks to the algebraically simple gluing operation.
In more `nonlinear' algebraic
settings, such as $A_\infty$ structures, one needs to artificially deform the gluing operation, e.g.\ when dealing with homotopies of data \cite{SEID}.
It seems to us that in such situations, especially when one needs to understand the moduli space's boundary and corners, it is more efficient to use the polyfold theory of Hofer-Wysocki-Zehnder [HWZ1--4] since this gives a cleaner structure.
However, in the Gromov--Witten setting, the less technologically sophisticated approach via Kuranishi atlases still has value for applications, specially in very geometric situations such as \cite{McT}, or in situations in which the symplectic situation is very close to that in complex algebraic geometry
such as in \cite{Z2} or \cite{Mcu}.

%

Another fundamental issue surfaced when we tried to understand how Floer's proof of the Arnold conjecture is extended to general symplectic manifolds using abstract regularization techniques. In the language of Kuranishi structures, it argues that a Kuranishi space $X$ of virtual dimension $0$, on which $S^1$ acts such that the fixed points $F\subset X$ are isolated solutions, allows for a perturbation whose
zero set is given by the fixed points. At that abstract level, the argument in both \cite{FO} and \cite{LiuT} is that
$(X\less F)/S^1$ can be equipped with a Kuranishi structure of virtual dimension $-1$, which induces a trivial fundamental cycle for $X\less F$. 
However, they give no details of this construction.
It seems that any such pullback construction would require the $S^1$-action to extend to an ambient space $X\less F\subset \Bb$ such that the Kuranishi structure for $(X\less F)/S^1$ has domains in 
the quotient space $\Bb/S^1$.
While polyfold theory offers a direct approach to this question of equivariant transversality, we believe that any VMC approach will require, as a starting point, the construction of a virtual neighbourhood as introduced in the present paper. Such a VMC approach -- which however we cannot comment on -- is now announced in \cite{FOOO12}.

\smallskip
\noindent
{\bf Acknowledgements:}
We would like to thank
Mohammed Abouzaid,
Kenji Fukaya,
Tom Mrowka,
Kaoru Ono,
Yongbin Ruan,
Dietmar Salamon,
Bernd Siebert,
Cliff Taubes,
Gang Tian,
and
Aleksey Zinger
for encouragement and enlightening discussions about this project.
We moreover thank MSRI, IAS, and BIRS for hospitality.

\section{Regularizations of holomorphic curve moduli spaces}  \label{s:fluff}

One of the central technical problems in the theory of holomorphic curves, which provides many of the modern tools in symplectic topology, is to construct algebraic structures by extracting homological information from moduli spaces of holomorphic curves in general compact symplectic manifolds $(M,\om)$.
We will refer to this technique as {\it regularization} and note that it requires two distinct components. On the one hand, some perturbation technique is used to achieve {\it transversality}, which gives the moduli space a smooth structure that induces a count or chain. On the other hand some cobordism technique is used to achieve {\it invariance}, i.e.\ independence of the resulting algebraic structure from the choices involved in achieving transversality.

The aim of this section is to give an overview of the different regularization approaches in the example of genus zero Gromov--Witten invariants $\bla \al_1,\ldots,\al_k\bra_{A} \in \Q$.
These are defined as a generalized count of $J$-holomorphic genus $0$ curves in class $A\in H_2(M;\Z)$ that meet $k$ representing cycles of the homology classes $\al_i\in H_*(M)$.
This number should be independent of the choice of $J$ in the contractible space of $\om$-compatible almost complex structures, and of the cycles representing $\alpha_i$.
For complex structures $J$ one can work in the algebraic setting, in which the curves are cut out by holomorphic functions on $M$, but general symplectic manifolds do not support an integrable $J$.  For non-integrable $J$, the approach introduced by Gromov \cite{GRO} is to view the (pseudo-)holomorphic curves as maps to $M$ satisfying the Cauchy--Riemann PDE, modulo reparametrizations by automorphisms of the complex domain.

To construct the Gromov--Witten moduli spaces of holomorphic curves, one starts out with the typically noncompact quotient space
$$
\Mm_k(A,J) := \bigl\{ \bigl( f: S^2 \to M, \bz\in (S^2)^k
\less \Delta \bigr) \,\big|\, f_*[S^2]=A , \pbar f = 0 \bigr\} / \PSL(2,\C)
$$
of equivalence classes of tuples $(f,\bz)$, where $f$ is a $J$-holomorphic map, the marked points
$\bz= (z_1,\ldots z_k)$ are pairwise disjoint, and the equivalence relation is given by the reparametrization action $\ga\cdot(f,\bz)=(f\circ\ga,\ga^{-1}(\bz))$
of the M\"obius group $\PSL(2,\C)$.
This space is contained (but not necessarily dense) in the compact moduli space $\oMm_{k}(A,J)$ formed by the equivalence classes of $J$-holomorphic genus $0$ stable maps $f:\Si \to M$ in class $A$ with $k$ pairwise disjoint marked points.  There is a natural evaluation map
\begin{equation} \label{eq:ev}
\ev: \oMm_{k}(A,J)\to M^k, \quad
[\Si,f,(z_1,\ldots,z_k)]\mapsto \bigl(f(z_1),\ldots,f(z_k)\bigr),
\end{equation}
and one expects the Gromov--Witten invariant
$$
\bla \al_1,\ldots,\al_k\bra_{A}:=\ev_*[\oMm_{k}(A,J)]\cap (\al_1\times\ldots\times \al_k)
$$
to be defined as intersection number of a homology class $\ev_*[\oMm_{k}(A,J)]\in H_*(M;\Q)$ with the class $\al_1\times\ldots\times \al_k$.
The construction of this homology class requires a {\it regularization} of $\oMm_{k}(A,J)$.
In the following Sections~\ref{ss:geom} - \ref{ss:kur} we give a brief overview of the approaches using geometric means or an abstract polyfold setup, and review the fundamental ideas behind Kuranishi atlases.
Sections ~\ref{ss:alg} and~\ref{ss:top} then discuss the algebraic and topological issues in constructing a virtual fundamental class from a Kuranishi atlas.

\subsection{Geometric regularization} \hspace{1mm}\\ \vspace{-3mm}  \label{ss:geom}

For some special classes of symplectic manifolds, the regularization of holomorphic curve moduli spaces can be achieved by a choice of the almost complex structure $J$, or more generally a perturbation of the Cauchy--Riemann equation $\pbar f =0$ that preserves the symmetry under reparametrizations, and whose Hausdorff compactification is given by nodal solutions.
Note that these properties are generally not preserved by perturbations of a nonlinear Fredholm operator such as $\pbar$, so this approach requires a class of perturbations
that preserves the geometric properties of $\pbar$.

The construction of Gromov--Witten invariants from a regularization of $\oMm_{k}(A,J)$ most easily fits into this approach if $A$ is a  homology  class on which $\om(A)>0$ is minimal, since then $A$
cannot be represented by a multiply covered or nodal holomorphic sphere.
For short, we call such $A$ $\om$-{\it minimal.}
In this case $\Mm_{k}(A,J')$ is smooth for generic $J'$, and compact if $k\le 3$.
More generally, this approach applies to all spherical Gromov--Witten invariants in semipositive symplectic manifolds, since in this case it is possible to compactify the image $\ev(\Mm_{k}(A,J'))$ by adding codimension-$2$ strata.
Full details for this construction can be found in \cite{MS}.
The most general form of this {\bf geometric regularization approach} proceeds in the following steps.

\begin{enumlist}
\item{\bf Fredholm setup:}
Write the (not necessarily compact) moduli space $\Mm = \si^{-1}(0)/{\rm Aut}$ as the
quotient, by an appropriate reparametrization group ${\rm Aut}$, of an equivariant smooth Fredholm section $\si:\Hat\Bb\to\Hat\Ee$ of a Banach vector bundle $\Hat\Ee\to\Hat\Bb$.
For example, $\Mm=\Mm_k(A,J)$ is cut out from $\Hat\Bb=W^{m,p}(S^2,M)$ by the Cauchy--Riemann operator $\si=\pbar$, which is equivariant with respect to ${\rm Aut}=\PSL(2,\C)$.

\item {\bf Geometric perturbations:}
Find a Banach manifold $\Pp\subset\Ga^{\rm Aut}(\Hat\Ee)$ of {\it equivariant }
sections for which the perturbed sections $\si+p$ have the same compactness properties as $\si$.
For example, the contractible set $\Jj^\ell$ of compatible $\Cc^\ell$-smooth almost complex structures for $\ell\geq m$ provides equivariant sections $p=\overline{\partial}_{J'}-\pbar$ for all $J'\in\Jj^\ell$. Moreover, $J'$-holomorphic curves also have a Gromov compactification $\oMm_k(A,J')$.

\item{\bf Sard--Smale:}
Check transversality of the section $(p,f)\to (\si + p)(f)$ to deduce that the universal moduli space
$$
{\textstyle \bigcup_{p\in\Pp}} \; \{p\}\times (\si + p)^{-1}(0) \;\subset\; \Pp\times\Hat\Bb
$$
is a differentiable Banach manifold. (In the example it is $\Cc^\ell$-differentiable.)
Then the regular values of the projection to $\Pp$ (guaranteed by the Sard--Smale theorem for sufficiently high differentiability of the universal moduli space) provide a comeagre subset $\Pp^{\rm reg}\subset\Pp$ for which the perturbed sections $\si_p:=\si+p$ are transverse to the zero section.
For holomorphic curves and perturbations given by $\Jj^\ell$, this transversality holds if all holomorphic maps are {\it somewhere injective}.
For $\om$-minimal  $A$ this weak form of injectivity
is a consequence of unique continuation (cf. \cite[Chapter~2]{MS}),
but for general Gromov--Witten moduli spaces this step only
applies to the subset $\Mm_k^*(A,J)$ of simple (i.e.\ not multiply covered) curves.
\item {\bf Quotient:}
For $p\in \Pp^{\rm reg}$, the perturbed zero set $\si_p^{-1}(0)\subset \Hat\Bb$  is
a smooth manifold by the implicit function theorem.
If, moreover, the action of ${\rm Aut}$ on $\si_p^{-1}(0)$ is smooth, free, and properly discontinuous, then the moduli space $\Mm^p := \si_p^{-1}(0) / {\rm Aut}$ is a smooth manifold.
For holomorphic curves, the smoothness of the action can be achieved if all solutions of $\overline{\partial}_{J'}f=0$ are smooth. For that purpose one can use e.g.\ the Taubes' trick to find regular perturbations given by smooth $J'$.

\item {\bf Compactification:}
For Gromov--Witten moduli spaces with $A$ $\om$-minimal and $k\le 3$, the previous steps already give $\Mm_k(A,J')$ the structure of a compact smooth manifold.
Thus the Gromov--Witten invariants can be defined using its fundamental class $[\Mm_{k}(A,J')]$.
In the semipositive case, the previous steps give $\Mm_k^*(A,J')$ a smooth structure such that
$\ev:\Mm_k^*(A,J')\to M$ defines a pseudocycle.
Indeed, its image is compact up to $\ev\bigl(\oMm_{k}(A,J) \less  \Mm_{k}^*(A,J)\bigr)$ which is given by the images of nodal and multiply covered maps.
Since the underlying simple curves are regular and of lower Fredholm index, these additional sets are smooth and of codimension at least $2$.
A more general approach for showing this pseudocycle property is to use {\bf gluing techiques}, which also apply to moduli spaces 
whose regularization is expected to have boundary.

Generally, one obtains a compactification $\oMm\,\!^p$ of the perturbed moduli space by constructing gluing maps into $\Mm^p$, whose images cover the complement of a compact set, and which are compatible on their overlaps.
In the Gromov--Witten case the gluing construction provides local homeomorphisms
$$
((1,\infty)\times S^1)\,\!^{N} \times \oMm\,\!^{*N}_k(A,J') \;\hookrightarrow\; \Mm_{k}(A,J')
$$
to neighbourhoods of the strata $\oMm\,\!^{*N}_k(A,J')$ of simple stable curves with $N$ nodes.
The Gromov compactification $\oMm\,\!^p=\oMm_k\,\!^*(A,J')$ is then constructed by completing
each cylinder  to a disc $\bigl((1,\infty)\times S^1\bigr) \cup \{\infty\}$,
where we identify the added set $ \{\infty\}\times  \oMm\,\!^{*N}_k(A,J') $ with the
stratum $ \oMm\,\!^{*N}_k(A,J') \subset \oMm_{k}(A,J)$.
Then  $\oMm_{k}(A,J)$ is compact up to  stable maps containing
multiply covered  components.

\item {\bf Invariance:}
To prove that invariants extracted from the perturbed moduli space $\oMm\,\!^p$ are well defined, one chooses $\Pp$ to be a connected neighbourhood of the zero section and constructs a cobordism between $\oMm\,\!^{p_0}$ and $\oMm\,\!^{p_1}$ for any regular pair $p_0,p_1\in\Pp^{\rm reg}$ by repeating the last five steps for the section $[0,1]\times\Hat\Bb\to\Hat\Ee$, $(t,b)\mapsto (\si+p_t)(b)$ for any smooth path $(p_t)_{t\in[0,1]}\in\Pp$.
In the semipositive Gromov--Witten example, the same argument is applied to find a pseudochain with boundary $\ev(\Mm_{k}^*(A,J_0)) \sqcup \ev(\Mm_{k}^*(A,J_1))$.
\end{enumlist}

\begin{remark} \rm
For Gromov-Witten theory, the evaluation map \eqref{eq:ev} generally does not represent a well defined rational homology class in $M^k$.  Although $\oMm_{k}(A,J)$ is compact and has a well understood formal dimension $d$ (given by the Fredholm index of $\pbar$ minus the dimension of the automorphism group), it need not be a manifold or orbifold of dimension $d$ for any~$J$.  Indeed it may contain subsets of dimension larger than $d$ consisting of stable maps with a component that is a
multiple cover on which $c_1(f^*\rT S^2)$ is negative.
In the case of spherical Gromov--Witten theory on manifolds with $[\om]\in H^2(M;\Q)$, it is possible to avoid this problem by first finding a consistent way to ``stabilize the domain'' to obtain a global description of the moduli space that involves no reparametrizations, and then allowing a richer class of perturbations; cf.\ \cite{CM, Io}. This approach has recently been extended in \cite{IoP}.  
\end{remark}

The main nontrivial steps in the geometric approach, which need to be performed in
careful detail for any given moduli space, are the following.
\begin{itemlist}
\item Each setting requires a different, precise definition of a Banach space of perturbations. Note in particular that spaces of maps with compact support in a given open set are not compact. The proof of transversality of the universal section is very sensitive to the specific geometric setting, and in the case of varying $J$ requires each holomorphic map to have suitable injectivity properties.
\item The gluing analysis is a highly nontrivial Newton iteration scheme and
should have an abstract framework that does not seem to be available at present. In particular, it requires surjective linearized operators, and so only applies after perturbation. Moreover, gluing of noncompact spaces requires uniform quadratic estimates, which do not hold in general.
Finally, injectivity of the gluing map does not follow from the Newton iteration and needs to be checked in each geometric setting.
\end{itemlist}

\subsection{Approaches to abstract regularization} \hspace{1mm}\\ \vspace{-3mm}
\label{ss:approach}

In order to obtain a regularization scheme that is generally applicable to holomorphic curve moduli spaces, it seems to be necessary to work with abstract perturbations $f\mapsto p(f)$ that need not be differential operators. Thus we recast the question more abstractly into one of regularizing a
compactification of the quotient of the zero set of a Fredholm operator.\footnote{
As pointed out by Aleksey Zinger, the ``Gromov compactification'' of a moduli space of holomorphic curves in fact need not even be a compactification in the sense of containing the holomorphic curves with smooth domains as dense subset. For example, it could contain an isolated nodal curve.
}
From this abstract differential geometric perspective, the geometric regularization scheme provides a highly nontrivial generalization of the well known finite dimensional regularization based on Sard's theorem, see e.g.\ \cite[ch.2]{GuillP}.

\label{finite reg}
\medskip
\noindent
{\bf Finite Dimensional Regularization Theorem:}  {\it
Let $E\to B$ be a finite dimensional vector bundle, and let $s:B\to E$ be a smooth section such that $s^{-1}(0)\subset B$ is compact.
Then there exists a compactly supported, smooth perturbation section $p:B\to E$ such that $s+p$ is transverse to the zero section, and hence $(s+p)^{-1}(0)$ is a smooth manifold.
Moreover, $[(s+p)^{-1}(0)]\in H_*(B,\Z)$ is independent of the choice of such perturbations.
}

\begin{remark}\rm   \label{equivariant BS}
\begin{itemlist}
\item
Using multisections, this theorem generalizes to equivariant sections under a finite group action, yielding orbifolds as regularized spaces and thus a well defined rational homology class $[(s+p)^{-1}(0)]\in H_*(B,\Q)$.
\item
For nontrivial Lie groups $G$ acting by bundle maps on $E$, equivariance and transversality are in general contradictory requirements on a section.
Only if $G$ acts smoothly, freely, and properly on $B$ and $E$,  can
one obtain $G$-equivariant transverse sections by pulling back transverse sections of $E/G\to B/G$.
%
%
%
\item
Finite dimensional regularization also holds for noncompact zero sets $s^{-1}(0)$,
but the homological invariance of the zero set fails in the simplest examples.

\item
There have been several attempts to extend this theorem to the case of a Fredholm section $s:\Hat\Bb\to \Hat \Ee$ of a Banach (orbi)bundle.
The paper by Lu--Tian~\cite{LuT} develops some abstract analysis, which we have not studied in detail since it does not apply to Gromov--Witten moduli spaces; see below.
Similarly,  Chen--Tian present in~\cite[\S5]{CT}  an idea of a \lq\lq Fredholm system"  that in its global form (even after replacing the nonsensical properness assumption by compactness of the zero set) is irrelevant to most Gromov--Witten moduli spaces, and when localized runs into the same problems to do with smoothness of coordinate changes and lack of suitable cut off functions that we discuss in detail in Remark~\ref{LTBS} below.
 \end{itemlist}
\end{remark}

Note here that no typical moduli space of holomorphic curves, nor even the moduli spaces in gauge theory or Morse theory, has a currently available description as the zero set of a Fredholm section in a Banach groupoid bundle.
In the case of holomorphic curves or Morse trajectories, the first obstacle to such a description is the differentiability failure of the action of ${\rm Aut}$ on any Sobolev space of maps $\Hat\Bb$ explained in Section~\ref{ss:nodiff}.
In gauge theory, the action of the gauge group typically is smooth, but in all 
theories the typical moduli spaces are compactified by gluing constructions, for which there is not even a natural description as a zero set in a topological vector bundle.

In comparison, the geometric regularization approach works with a smooth section $\si:\Hat\Bb\to\Hat\Ee$ of a Banach bundle, which has a noncompact solution set $\si^{-1}(0)$ and is equivariant under the action of a noncompact Lie group.
From an abstract topological perspective, the nontrivial achievement of this approach is that it produces equivariant transverse perturbations and a well defined homology class by compactifying quotients of perturbed spaces, rather than by directly perturbing the compactified moduli space.

\begin{remark}\rm
Another notable analytic feature of the perturbations obtained by altering $J$ is that they preserve the compactness and Fredholm properties of the nonlinear differential operator, despite changing it nonlinearly in highest order. Indeed, in local coordinates, $\si= \partial_s + J \partial_t$ is a first order operator, and changing $J$ to $J'$ amounts to adding another first order operator $p=(J'-J)\partial_t$.
This preserves the Fredholm operator since it preserves ellipticity of the symbol. In general, one retains Fredholm properties only with lower order perturbations, i.e.\ by adding a compact operator to the linearization.
For the Cauchy--Riemann operator, that would mean an operator involving no derivatives, e.g.\ $p(f)=X\circ f$ given by a vector field.
Note also that the compactness properties of solution sets of nonlinear operators are generally not even preserved under lower order perturbations that are supported in a neighbourhood of a compact solution set, since in the infinite dimensional setting such neighbourhoods are never compact.
\end{remark}

This discussion shows that a regularization scheme for general holomorphic curve moduli spaces needs to work with more abstract perturbations and directly on the compactified moduli space -- i.e.\ after quotienting and taking the Gromov compactification. The following approaches are currently used in symplectic topology.

\begin{enumlist}
\item{The {\bf global obstruction bundle approach}} as in \cite{LiuT, Sieb, Mcv} aims to extend successful techniques from algebraic geometry and gauge theory to the holomorphic curve setting, by means of a weak orbifold structure on a suitably stratified Banach space completion of the space of equivalence classes of smooth stable maps.

\item{The {\bf Kuranishi approach}}, introduced by Fukaya-Ono \cite{FO} and implicitly Li-Tian \cite{LT} in the 1990s, aims to construct a virtual fundamental class from finite dimensional reductions of the equivariant Fredholm problem and gluing maps near nodal curves.

\item{The {\bf polyfold approach}}, developed by Hofer-Wysocki-Zehnder since 2000, aims to generalize the finite dimensional regularization theorem so that it applies directly to the compactified moduli space, by expressing it as the zero set of a smooth section.
\end{enumlist}

\MS

The first two approaches construct a virtual fundamental class $[\oMm_{k}(A,J)]^{vir}$ and are hence also referred to as virtual transversality.
They have been used for concrete calculations of Gromov--Witten invariants, e.g.\ \cite{McT,Mcu,Z} by building a VMC using geometrically meaningful perturbations.
The third approach is more functorial and produces a VMC with significantly more structure, e.g.\ as a smooth orbifold in the case of Gromov--Witten invariants \cite{HWZ:gw}. This allows to define e.g.\ symplectic field theory (SFT) invariants on chain level. The book \cite{FOOO} uses the Kuranishi approach to a similar end in the construction of chain level Lagrangian Floer theory.
We will make no further comments on chain level theories.
Instead, let us compare how the different approaches handle the fundamental analytic issues.

\medskip\noindent
{\bf Dividing by the automorphism group:}
Unlike the smooth action of the (infinite dimensional) gauge group on Sobolev spaces of connections, the reparametrization groups (though finite dimensional) do not act differentiably on any known Banach space completion of spaces of smooth maps (or pairs of domains and maps from them); see Section~\ref{s:diff}.
In the global obstruction bundle approach this causes a significant differentiability failure in the relation between local charts in Liu--Tian \cite{LiuT}, and hence in the survey article McDuff~\cite{Mcv} and subsequent papers such as Lu~\cite{GLu}.
For more details of the problems here see Remark~\ref{LTBS}.
This differentiability issue was not mentioned in \cite{FO,LT}.
However, as we explain in detail in Section~\ref{s:construct} below, it needs to be addressed when defining charts that combine two or more basic charts since this must be done in the Fredholm setting {\it before} passing to a finite dimensional reduction.
We make this explicit in our notion of ``sum chart", but the same construction is used implicitly in \cite{FO,FOOO}.
In this setting, it can be overcome by working with special obstruction bundles, as we outline in Section~\ref{ss:gw}.
In the polyfold approach, this issue is resolved by replacing the notion of smoothness in Banach spaces by a notion of scale-smoothness which applies to the reparametrization action.
To implement this, one must redevelop linear as well as nonlinear functional analysis in the scale-smooth category.

\medskip
\noindent
{\bf Gromov compactification:}
Sequences of holomorphic maps can develop various kinds of singularities: bubbling (energy concentration near a point), breaking (energy diverging into a noncompact end of the domain -- sometimes also induced by stretching in the domain), buildings (parts of the image diverging into a noncompact end of the target -- sometimes also induced by stretching the domain at a hypersurface).
These limits are described as tuples of maps from various domains, capturing the Hausdorff limit of the images. In quotienting by reparametrizations, note that for the limit object this group is a substantially larger product of various reparametrization groups.

In the geometric and virtual regularization approaches, charts near the singular limit objects are constructed by gluing analysis, which involves a pregluing construction and a Newton iteration. The pregluing creates from a tuple of holomorphic maps a single map from a nonsingular domain, which solves the Cauchy--Riemann equation up to a small error. The Newton iteration then requires quadratic estimates for the linearized Cauchy--Riemann operator to find a unique exact solution nearby.
In principle, the construction of a continuous gluing map should always be possible along the lines of \cite{MS}, though establishing the quadratic estimates is nontrivial in each setting. However, additional arguments specific to each setting are needed to prove surjectivity, injectivity, and openness of the gluing map.
Moreover, while homeomorphisms to their image suffice for the geometric regularization approach, the virtual regularization approaches 
all require stronger differentiability of the gluing map; e.g.\ smoothness in \cite{FO,FOOO,J}.

None of \cite{LT,LiuT,FO,FOOO} give all details for the construction of a gluing map.
In particular, \cite{FO,FOOO} construct gluing maps with image in a space of maps,
 but give few details on the induced map to the quotient space,
even in the nonnodal case as discussed in Remark~\ref{FOglue}.
For closed nodal curves, \cite[Chapter~10]{MS} constructs continuous gluing maps in full detail, but 
does not claim that the glued curves depend differentiably 
on the gluing parameter $a\in \C$ as $a\to 0$.
By rescaling $|a|$, it is possible to establish more differentiability for $a\to 0$. For example
Ruan~\cite{Ruan} uses local $\Cc^1$ gluing maps. However, as pointed out by Chen--Li~\cite{ChenLi}, 
this $\Cc^1$ structure is not intrinsic, so may not be preserved under coordinate changes.
This problem was ignored in \cite{FO}, but discussed in both in the appendix to \cite{FOOO} and more recently in \cite{FOOO12}. 

The polyfold approach reinterprets the pregluing construction as the chart map for an ambient space $\Tilde \Bb$ which contains the compactified moduli space, essentially making the quadratic estimates part of the definition of a Fredholm operator on this space. The Newton iteration is replaced by an abstract implicit function theorem for transverse Fredholm operators in this setting. The injectivity and surjectivity issues then only need to be dealt with at the level of pregluing. Here injectivity fails dramatically but in a way that can be reinterpreted in terms of a generalization of a Banach manifold chart, where the usual model domain of an open subset in a Banach space is replaced by a relatively open subset in the image of a scale-smooth retraction of a scale-Banach space.
This makes it necessary to redevelop differential geometry in the context of retractions and scale-smoothness.

\subsection{Regularization via polyfolds}  \hspace{1mm}\\ \vspace{-3mm}
\label{ss:poly}

In the setting of holomorphic maps with trivial isotropy (but allowing for general compactifications by e.g.\ nodal curves), the result of the entirely abstract development of scale-smooth nonlinear functional analysis and retraction-based differential geometry is the following direct generalization
of the finite dimensional regularization theorem, see \cite{HWZ3}.
The following is the relevant version for trivial isotropy, in which ambient spaces have the structure of an M-polyfold --- a generalization of the notion of Banach manifold, essentially given by charts in open subsets of images of retraction maps, and scale-smooth transition maps between the ambient spaces of the retractions.

\medskip
\noindent
{\bf M-polyfold Regularization Theorem:} {\it
Let $\Tilde{\Ee}\to\Tilde{\Bb}$ be a strong M-polyfold bundle, and let $s:\Tilde{\Bb}\to\Tilde{\Ee}$ be a scale-smooth Fredholm section such that $s^{-1}(0)\subset\Tilde{\Bb}$ is compact.
Then there exists a class of perturbation sections $p:\Tilde{\Bb}\to\Tilde{\Ee}$ supported near $s^{-1}(0)$ such that $s+p$ is transverse to the zero section, and hence $(s+p)^{-1}(0)$ carries the structure of a smooth finite dimensional manifold.
Moreover, $[(s+p)^{-1}(0)]\in H_*(\Tilde\Bb,\Z)$ is independent of the choice of such perturbations.
}

\medskip
For dealing with nontrivial, finite isotropies, \cite{HWZ3} transfers this theory to a groupoid setting to obtain a direct generalization of the orbifold version of the finite dimensional regularization theorem.
It is these groupoid-type ambient spaces $\Tilde\Bb$, whose object and morphism spaces are M-polyfolds, that are called polyfolds.
These abstract regularization theorems should be compared with the definition of Kuranishi atlas and the abstract construction of a virtual fundamental class for any Kuranishi atlas that will be outlined in the following sections.
While the language of polyfolds and the proof of the regularization theorems in \cite{HWZ1,HWZ2,HWZ3} is highly involved, it seems to be developed in full detail and is readily quotable.
A survey of the basic philosophy and language is now available in \cite{gffw}.

Just as in the construction of a Kuranishi atlas for a given holomorphic curve moduli space discussed in Section~\ref{s:construct} below, the application of the polyfold regularization approach still requires a description of the compactified moduli space as the zero set of a Fredholm section in a polyfold bundle.
It is here that the polyfold approach promises the most revolutionary advance in regularization techniques. Firstly, fiber products of moduli spaces with polyfold descriptions are naturally described as zero sets of a Fredholm section over a product of polyfolds. For example, one can obtain a polyfold setup for the PSS morphism by combining the polyfold setup for SFT with a smooth structure on Morse trajectory spaces, see \cite{afw:arnold}.
Secondly, Hofer--Wysocki--Zehnder are currently working on formalizing a ``modular''
approach to the polyfold axioms in such a way that the analytic setup can be given locally in domain and target for every singularity type. With that, the polyfold setup for a new moduli space that combines previously treated singularities in a different way would merely require a Deligne--Mumford type theory for the underlying spaces of domains and targets.

\begin{remark} \rm \label{polyfold BS checklist}
While the polyfold framework is a very powerful method for constructing algebraic invariants from holomorphic curve moduli spaces, it also has some pitfalls in geometric applications.
\begin{itemlist}
\item
Some caution is required with arguments involving the geometric properties of solutions after regularization.
The reason for this is that
the perturbed solutions do not solve a PDE but an abstract compact perturbation of the Cauchy--Riemann equation. Essentially, one can only work with the fact that the perturbed solutions can be made to lie arbitrarily close to the unperturbed solutions in any metric that is compatible with the scale-topology (e.g.\ any $\Cc^k$-metric in the case of closed curves).
\item
Despite reparametrizations acting scale-smoothly on spaces of maps, the question of equivariant regularization for smooth, free, proper actions remains nontrivial due to the interaction with retractions, i.e.\ gluing constructions.
In the example of the $S^1$-action on spaces of Floer trajectories for an autonomous Hamiltonian, the unregularized compactified Floer trajectory spaces of virtual dimension $0$ may contain broken trajectories. The corresponding stratum of the quotient space,
$$
\oMm(p_-,p_+)/S^1 \;\supset\; {\textstyle \bigcup_q} \bigl(\oMm(p_-,q) \times \oMm(q,p_+)\bigr)/S^1 ,
$$
is an $S^1$-bundle over the fiber product $\oMm(p_-,q)/S^1\times \oMm(q,p_+)/S^1$ of quotient spaces, rather than the fiber product itself.
Due to these difficulties, as yet, there is no quotient theorem for polyfolds, and hence no understanding of when a description of $\oMm$ as zero set of an $S^1$-equivariant Fredholm section would induce a description of $\oMm/S^1$ as zero set of a Fredholm section with smaller Fredholm index. Moreover, such a quotient would not even immediately induce an equivariant regularization of the Floer trajectory spaces compatible with gluing.
\end{itemlist}
\end{remark}

\subsection{Regularization via Kuranishi atlases}  \hspace{1mm}\\ \vspace{-3mm} \label{ss:kur}

Continuing the notation of Section~\ref{ss:geom}, the basic idea of regularization via Kuranishi atlases is to describe the compactified moduli space $\oMm$ by local finite dimensional reductions of the ${\rm Aut}$-equivariant section $\si:\Hat\Bb \to \Hat\Ee$, and by gluing maps near the nodal curves. 
There are different  ways to formalize the construction.  We will proceed  via a notion of {\bf Kuranishi atlas}
in the following steps.\MS

\begin{enumlist}
\item {\bf Compactness:}
Equip the compactified moduli space $\oMm$ with a compact, metrizable topology; namely as Gromov compactification of $\Mm=\si^{-1}(0)/{\rm Aut}$.
\MS

\item {\bf Equivariant Fredholm setup:}
As in the geometric approach, a significant subset $\Mm\subset\oMm$ of the compactified moduli space is given as the zero set of a Fredholm section modulo a finite dimensional Lie group,
\vspace{-10mm}
\[
\begin{aligned}
&\phantom{\Mm} \\
&\phantom{\Mm} \\
& \Mm = \frac{\si^{-1}(0)}{{\rm Aut}} ,
 \end{aligned}
 \qquad\qquad
\xymatrix{
 \Hat\Ee   \ar@(ul,dl)_{\textstyle \rm Aut}\ar@{->}[d]     \\
 \Hat\Bb \ar@(ul,dl)_{\textstyle \rm Aut} \ar@/_1pc/[u]_{\textstyle \si}
}
\]
One can now relax the assumption of $\rm Aut$ acting freely to the requirement that
 the isotropy subgroup $\Ga_{f}:=\{ \ga \in {\rm Aut} \,|\, \ga\cdot f = f \}$ be finite for every solution
 $f\in\si^{-1}(0)$.\MS

\item {\bf Finite dimensional reduction:}
Construct {\bf basic Kuranishi charts}
for every ${[f]\in\Mm}$,
\vspace{-5mm}
\[
\begin{aligned}
&\phantom{\Mm} \\
&\phantom{\Mm} \\
 \Mm \; \overset{\psi_f}{\longhookleftarrow} \;\frac{\ti s_f^{-1}(0)}{\Ga_f} ,
 \end{aligned}
 \qquad\qquad
\xymatrix{
*+[r]{\Tilde E_f }  \ar@(ul,dl)_{\textstyle \Ga_f}
  \ar@{->}[d]    \\
 *+[r]{U_f} \ar@(ul,dl)_{\textstyle \Ga_f} \ar@/_1pc/[u]_{\textstyle \ti s_f}
}
\]
which depend on a choice\footnote{
In practice the Kuranishi data will be constructed from many choices, including that of a representative. So we try to avoid false impressions by using the subscript $f$ rather than $[f]$.
}
of representative $f$, and consist of the following data:
\begin{itemize}
\item
the {\bf domain} $U_f$ is a finite dimensional smooth manifold (constructed from a local slice of the ${\rm Aut}$-action on a thickened solution space
$\{g \,|\, \pbar g \in \Hat E_{f}\}$);
\item
the {\bf obstruction bundle} $\Ti E_f =\Hat E_f|_{U_f}\to U_f$, a finite rank vector bundle (constructed from the cokernel of the linearized Cauchy--Riemann operator at $f$), which is isomorphic $\Ti E_f\cong U_f\times E_f$ to a trivial bundle, whose fiber $E_f$ we call the {\bf obstruction space};
\item
the {\bf section} $\ti s_{f}: U_{f} \to \Tilde E_{f}$ (constructed from $g \mapsto \pbar g$),
which induces a smooth map $s_f : U_f\to E_f$ in the trivialization;
\item
the {\bf isotropy group} $\Ga_{f}$ acting on $U_{f}$ and $E_{f}$ such that $\ti s_{f}$ is equivariant; \item
the {\bf footprint map} $\psi_{f}: \ti s_{f}^{-1}(0)/\Ga_{f} \to \oMm$, a homeomorphism to a neighbourhood of $[f]\in\oMm$ (constructed from $\{g \,|\, \pbar g=0 \}
\ni g \mapsto [g]$).
\end{itemize}
(See Section~\ref{ss:Kchart} for a detailed outline of this construction for $\oMm_1(A,J)$  with $\Ga_f=\{\rm id\}$.)\MS

\item {\bf Gluing:}
Construct basic Kuranishi charts covering $\oMm\less \Mm$ by combining finite dimensional reductions with gluing analysis similar to the geometric approach.\MS

\item {\bf Compatibility:}
Given a finite cover of $\oMm$ by the footprints of basic Kuranishi charts $\bigl(\bK_i = (U_i,E_i,\Ga_i,s_i,\psi_i)\bigr)_{i=1,\ldots,N}$, construct transition data satisfying suitable compatibility conditions.
In the case of trivial isotropies $\Ga_i=\{{\rm id}\}$, any notion of compatibility will have to induce the following minimal transition data for any element $[g]\in \im\psi_i\cap \im\psi_j\subset \oMm$ in an overlap of two footprints:
\begin{itemize}
\item
a
{\bf transition Kuranishi chart} $\bK^{ij}_g = (U^{ij}_g,E^{ij}_g,s^{ij}_g,\psi^{ij}_g)$ whose footprint $\im\psi^{ij}_g \subset  \im\psi_i\cap \im\psi_j$ is a neighbourhood of $[g]\in\oMm$;
\item
{\bf coordinate changes} $\Phi^{i,ij}_g : \bK_i \to \bK^{ij}_g$ and $\Phi^{j,ij}_g :\bK_j \to \bK^{ij}_g$ consisting of embeddings and linear injections
$$
\phi^{\bullet,ij}_g :\;  U_\bullet \supset V^{\bullet,ij}_g\; \longhookrightarrow\; U^{ij}_g , \qquad
\Hat\phi^{\bullet,ij}_g :\;  E_\bullet \; \longhookrightarrow\; E^{ij}_g \qquad
\text{for}\; \bullet = i,j
$$
which extend $\phi^{\bullet,ij}_g|_{\psi_\bullet^{-1}(\im\psi^{ij}_g)} = (\psi^{ij}_g)^{-1}\circ \psi_\bullet$ to open subsets $V^{\bullet,ij}_g \subset U_\bullet$ such that
$$
s^{ij}_g \circ \phi^{\bullet,ij}_g   \; =\;  \Hat\phi^{\bullet,ij}_g\circ s_\bullet  \qquad
\text{for}\; \bullet = i,j .
$$
\end{itemize}
Further requirements on the domains, an index or ``tangent bundle"
condition, coordinate changes between multiple overlaps, and cocycle 
conditions are discussed in Section~\ref{ss:top}.
Sections~\ref{s:diff}, \ref{s:construct} discuss the relevant smoothness issues.
The collection of such data -- basic charts, transition charts and coordinate changes -- will be called a {\bf Kuranishi atlas}.
\MS

\item {\bf Abstract Regularization:}
For suitable transition data, (multivalued) perturbations $s_f': U_f\to E_f$ of the sections in the 
Kuranishi charts (both basic charts and the transition charts) should yield the following regularization theorem:
\MS

\noindent
{\it Any Kuranishi atlas $\Kk$ on a compact space $\oMm$ induces a virtual fundamental class $[\oMm]_\Kk^{\rm vir}$.}

\item {\bf Invariance:}
Prove that $[\oMm]_\Kk^{\rm vir}$ is independent of the different choices in the previous steps, in particular the choice of local slices and obstruction bundles. This involves the construction of a Kuranishi  atlas for 
$\oMm\times [0,1]$ that restricts to two given choices $\Kk^0$ on $\oMm\times \{0\}$ and $\Kk^1$ on $\oMm\times \{1\}$. Then an abstract cobordism theory for Kuranishi atlases should imply 
$[\oMm]_{\Kk^0}^{\rm vir}=[\oMm]_{\Kk^1}^{\rm vir}$.
\end{enumlist}

The construction of a Kuranishi  atlas for a given holomorphic curve moduli space is explained in more detail in Section~\ref{s:construct}.
The rest of the paper (Sections~\ref{s:chart},~\ref{s:Ks} and~\ref{s:VMC}) then discusses
the abstract regularization theorem underlying this approach.
For that purpose we restrict to the case of trivial isotropy groups $\Ga_{f}=\{{\rm id}\}$ in all
Kuranishi charts. This simplifies constructions in two ways. First, it guarantees
the existence of restrictions of Kuranishi charts to any open subset of the footprint.
(In general, restrictions of Kuranishi charts with nontrivial isotropy will only exist as generalized Kuranishi charts whose domain is a groupoid.)
Second, for trivial isotropy one can construct the virtual fundamental class from the zero sets of perturbed sections
$s_f + \nu_f \approx s_f$
that are transverse,
$s_f+\nu_f\pitchfork 0$,
rather than replacing each $\Ga_f$-equivariant section $s_f$ with a transverse multisection.

\subsection{Algebraic issues in the use of germs for Kuranishi structures}  \hspace{1mm}\\ \vspace{-3mm}  \label{ss:alg}

A natural approach, adopted in \cite{FO,J} for formalizing the compatibility of Kuranishi charts is to work with germs of charts and coordinate changes.
Recent discussions have led to an agreement that this approach has serious algebraic issues in making sense of a cocycle condition for germs of coordinate changes, which we explain here. This issue is rooted in the fact that only the footprints of Kuranishi charts have invariant meaning, so that
the coordinate changes between Kuranishi charts are fixed by the charts only on the zero sets. Thus in the definition of a germ of charts the traditional equivalence of maps with common restriction to a smaller domain is extended by equivalence of maps that are intertwined by a diffeomorphism of the domains. This leads to an ambiguity in the definition of germs of coordinate changes between germs of charts.
As a result, germs of coordinate changes are defined as conjugacy classes of coordinate changes with respect to diffeomorphisms of the domains that fix the zero sets. However, in this setting the composition of germs is ill defined, so there is no meaningful cocycle condition.
Alternatively, one might want to view (charts, coordinate changes, equivalences of coordinate changes) as a $2$-category with ill-defined $2$-composition. Either way, there is no general procedure for extracting the data necessary for a construction of a VFC:
 a finite set of charts and coordinate changes that satisfy the cocycle condition.
In the following, we spell out in complete detail the usual definitions of germs and point out the algebraic issues that arise from the equivalence under conjugation.

\MS

To simplify notation let us (falsely) pretend that all obstruction spaces are finite rank subspaces $E_f\subset\Ee$ of the same space and the linear maps $\Hat\phi$ in the coordinate changes are restrictions of the identity. We moreover assume that all isotropy groups are trivial $\Ga_f=\{{\rm id}\}$ and only consider germs of charts and coordinate changes at a fixed point $p\in\oMm$.
In the following all neighbourhoods are required to be open.

\MS\NI
To the best of our understanding, \cite{FO,J} define a germ of Kuranishi chart as follows.

\begin{itemlist}
\item
A Kuranishi chart consists of a neighbourhood $U\subset\R^k$ of $0$ for some $k\in\N$, a map $s:U\to E \subset\Ee$ with $s(0)=0$, and an embedding $\psi:s^{-1}(0)\to\oMm$ with $\psi(0)=p$,
$$
\oMm \;\overset{\psi}{\longhookleftarrow}\; s^{-1}(0)
\;\subset\;
U \;\overset{s}{\longrightarrow}\; E.
$$
\item
Two Kuranishi charts $(U_1,s_1,\psi_1)$, $(U_2,s_2,\psi_2)$ are equivalent if the transition map
$$
\psi_2^{-1}\circ\psi_1 :\; s_1^{-1}(0) \;\supset\; \psi_1^{-1}(\im\psi_2)
\;\longrightarrow \; \psi_2^{-1}(\im\psi_1) \;\subset\; s_2^{-1}(0)
$$
extends to a diffeomorphism $\theta: U'_1\to U'_2$ between neighbourhoods $U_i' \subset U_i$ of $0$ that intertwines the sections $s_1|_{U'_1}= s_2|_{U'_2}\circ \theta$.
\item
A germ of Kuranishi chart at $p$ is an equivalence class of Kuranishi charts.
\item[$\mathbf{\bigtriangleup} \hspace{-2.08mm} \raisebox{.3mm}{$\scriptscriptstyle !$}\,$]
Note that $s_1= s_2\circ \theta$ does not necessarily determine the diffeomorphism $\theta$ except on the (usually singular and not dense) zero set.
Hence there may exist auto-equivalences, i.e.\ a nontrivial diffeomorphism $\theta: U'_1\to U'_2$ between restrictions $U'_1,U'_2\subset U$ of the same Kuranishi chart $(U,\ldots)$, satisfying $\theta|_{s^{-1}(0)}={\rm id}$ and $s= s\circ \theta$.
\end{itemlist}

Next, one needs to define the notion of a coordinate change between two germs of Kuranishi charts $[U_I,s_I,\psi_I]$ and $[U_J,s_J,\psi_J]$.\footnote{
In the notation of the previous section,
an example of a required coordinate change is one for index sets $I=\{i\}$, $J=\{i,j\}$, where
$[U_I, \ldots]$ denotes the germ at $[f]$ induced by $(U_i,\ldots)$, and $[U_J, \ldots]$ denotes the germ at $[f]$ induced by $(U^{ij}_f, \ldots)$.
}
It is here that ambiguities in the compatibility conditions appear, so we give what seems like the most natural definition, which is at least closely related to \cite{FO,J}.

\begin{itemlist}
\item
A coordinate change $(U_{IJ},\phi_{IJ}) : (U_I,s_I,\psi_I)\to (U_J,s_J,\psi_J)$ between Kuranishi charts consists of a neighbourhood $U_{IJ}\subset U_I$ of $0$ and an embedding $\phi_{IJ}:U_{IJ} \hookrightarrow U_J$ that extends the natural transition map and intertwines the sections,
$$
\phi_{IJ}|_{s_J^{-1}(0)\cap U_{IJ}} \;=\; \psi_J^{-1}\circ\psi_I , \qquad
s_J|_{U_{IJ}}  \; =\; s_I \circ \phi_{IJ} .
$$
\item
Let $(U_{I,1},s_{I,1},\psi_{I,1})\sim(U_{I,2},s_{I,2},\psi_{I,2})$ and $(U_{J,1},s_{J,1},\psi_{J,1})\sim (U_{J,2},s_{J,2},\psi_{J,2})$ be two pairs of equivalent Kuranishi charts. Then two coordinate changes
\begin{align*}
& (U_{IJ,1},\phi_{IJ,1}) : (U_{I,1},s_{I,1},\psi_{I,1})\to (U_{J,1},s_{J,1},\psi_{J,1}) \\
\text{and}\quad &
(U_{IJ,2},\phi_{IJ,2}) : (U_{I,2},s_{I,2},\psi_{I,2})\to (U_{J,2},s_{J,2},\psi_{J,2})
\end{align*}
are equivalent if there exist diffeomorphisms $\theta_I: U'_{I,1}\to U'_{I,2}$ and $\theta_J: U'_{J,1}\to U'_{J,2}$ between smaller neighbourhoods of $0$ as in the definition of equivalence of Kuranishi charts (i.e.\
$s_{I,1}|_{U'_{I,1}}= s_{I,2}|_{U'_{I,2}}\circ \theta_I$ and $s_{J,1}|_{U'_{J,1}}= s_{J,2}|_{U'_{J,2}}\circ \theta_J$) that intertwine the coordinate changes on a neighbourhood of $0$,
$$
\theta_J \circ \phi_{IJ,1} =  \phi_{IJ,2} \circ \theta_I  .
$$
\item
A germ of coordinate changes between germs of Kuranishi structures at $p$ is an equivalence class of coordinate changes.
\item[$\mathbf{\bigtriangleup} \hspace{-2.08mm} \raisebox{.3mm}{$\scriptscriptstyle !$}\,$]
As a special case, two coordinate changes $\phi_{IJ}, \phi'_{IJ} : (U_{I},\ldots )\to (U_{J},\ldots)$
 between the same Kuranishi charts are equivalent
if there exist auto-equivalences $\theta_I: U'_{I,1}\to U'_{I,2}$ and $\theta_J: U'_{J,1}\to U'_{J,2}$ such that
$\theta_J \circ \phi_{IJ} =  \phi'_{IJ} \circ \theta_I  $.
\item[$\mathbf{\bigtriangleup} \hspace{-2.08mm} \raisebox{.3mm}{$\scriptscriptstyle !$}\,$]
Given a germ of coordinate change $[U_{IJ},\phi_{IJ}] : [U_I,\ldots]\to [U_J,\ldots]$ and choices of representatives $(U'_I,\ldots), (U'_J,\ldots)$ of the germs of charts, a representative of the coordinate change now only exists between suitable restrictions $(U''_I\subset U'_I,\ldots), (U''_J\subset U'_J,\ldots)$, and even with fixed choice of restrictions may not be uniquely determined.
\end{itemlist}

\MS\NI
Finally, it remains to make sense of the cocycle condition for germs of coordinate changes.
At this point \cite{FO,J} simply write equations such as $[\Phi_{JK}]\circ [\Phi_{IJ}] = [\Phi_{IK}]$ on the level
of conjugacy classes of maps, which do not make strict sense.
The following is an attempt to phrase the cocycle condition on the level of germs, but we will see that it falls short of implying the existence of compatible choices of representatives that is required for the construction of a VMC.

\begin{itemlist}
\item
Let $[U_{I},s_{I},\psi_{I}]$, $[U_{J},s_{J},\psi_{J}]$, and $[U_{K},s_{K},\psi_{K}]$ be germs of Kuranishi charts. Then we say that a triple of germs of coordinate changes $[U_{IJ},\phi_{IJ}], [U_{JK},\phi_{JK}],[U_{IK},\phi_{IK}]$ satisfies the cocycle condition if there exist representatives of the coordinate changes
between representatives $(U_{I},\ldots)$, $(U_{J},\ldots)$, $(U_{K},\ldots)$ of the charts,
\begin{align*}
 (U_{IJ},\phi_{IJ}) : \;(U_{I},s_{I},\psi_{I})&\to (U_{J},s_{J},\psi_{J}) , \\
 (U_{JK},\phi_{JK}) : (U_{J},s_{J},\psi_{J})&\to (U_{K},s_{K},\psi_{K}) , \\
 (U_{IK},\phi_{IK}) : \;(U_{I},s_{I},\psi_{I})&\to (U_{K},s_{K},\psi_{K}) ,
\end{align*}
such that on a neighbourhood of $0$ we have
\begin{align} \label{algcc}
\phi_{JK} \circ \phi_{IJ} = \phi_{IK} .
\end{align}
\item[$\mathbf{\bigtriangleup} \hspace{-2.08mm} \raisebox{.3mm}{$\scriptscriptstyle !$}\,$]
Note that the above cocycle condition for some choice of representatives does not imply a cocycle condition for different choices of representatives. For example, suppose that $\phi_{IJ},\phi_{JK},\phi_{IK}$ satisfy \eqref{algcc}, and consider other representatives
$$
\phi_{IJ}' = \theta_J \circ \phi_{IJ} \circ \theta_I^{-1} ,  \quad
\phi_{JK}' = \Theta_K \circ \phi_{JK} \circ \Theta_J^{-1}
$$
given by auto-equivalences $\theta_I,\theta_J,\Theta_J, \Theta_K$.
Then these fit into a cocycle condition
$$
\phi_{JK}' \circ \phi_{IJ}' \;=\;
\bigl( \Theta_K \circ \phi_{JK} \circ \Theta_J^{-1} \bigr)
\circ
\bigl( \theta_J \circ \phi_{IJ} \circ \theta_I^{-1} \bigr)  \;=\; \phi'_{IK} \;\in\; [\phi_{IK}]
$$
only if $\Theta_J=\theta_J$ and $\phi'_{IK} = \Theta_K \circ \phi_{IK} \circ \theta_I^{-1}$.
That is, the choice of one representative in the cocycle condition between three germs of coordinate changes essentially fixes the choice of the other two representatives.
This causes problems as soon as one considers the
compatibility of four or more coordinate changes.
\end{itemlist}
\MS

Now suppose that a Kuranishi structure on $\oMm$ is given by germs of charts at each point and germs of coordinate changes between each suitably close pair of points, satisfying a cocycle condition. Then the fundamentally important first step towards the construction of a VMC is the claim of \cite[Lemma~6.3]{FO} that any such Kuranishi structure has a ``good coordinate system".
The latter, though the definitions in \cite{FO,FOOO} are slightly ambiguous, is a finite cover of $\oMm$ by partially ordered charts (where two charts should be comparable iff the footprints intersect) with coordinate changes according to the partial order, and satisfying a weak cocycle condition.
In order to extract such a finite cover from a tuple of germs of charts and germs of coordinate changes, one makes a choice of representative in each equivalence class of charts and picks a finite subcover.
The first nontrivial step is to make sure that these representatives were chosen sufficiently small for coordinate changes between them to exist in the given germs of coordinate changes.
The second  crucial step is to make specific choices of representatives of the coordinate changes such that the cocycle condition is satisfied.
However, \cite[(6.19.4)]{FO} does not address the need to choose specific, rather than just sufficiently small, representatives.
 In order to reduce the number of constraints,  this would require a rather special structure of the overlaps of charts.
In general, the choice of a representative for $[\phi_{AB}]$ would affect the choice of
representatives for $[\phi_{CA}]$ or $[\phi_{AC}]$ for all $C$ with $\dim U_C\leq \dim U_B$, and for $[\phi_{BC}]$ or $[\phi_{CB}]$ when $\dim U_C \geq \dim U_A$.
These are algebraic issues, governed by the intersection pattern of the charts.

One approach to solving these algebraic issues could be to replace the definition of Kuranishi structure by that of a  good coordinate system. However, we know of no direct way to construct such ordered covers and explicit cocycle conditions for a given moduli space $\oMm$.
We solve both problems by defining the notion of a Kuranishi atlas as a weaker version of a good coordinate system --- without a partial ordering on the charts, but satisfying an explicit cocycle condition --- that can in practice be constructed.
We then construct a good coordinate system in Proposition~\ref{prop:red} by an abstract refinement of the Kuranishi atlas.

\begin{rmk}\rm
One potential attraction of the notion of germs of Kuranishi charts is that for moduli spaces $\oMm$ arising from a Fredholm problem, there could be the notion of a \lq\lq natural germ" of charts at a point
$[f]\in\oMm$ given by the finite dimensional reductions at any representative $f$.
However, the present definition of germ does not provide a notion of equivalence between finite dimensional reductions with obstruction spaces of different dimension.
So the only natural choice would be to require obstruction spaces to have minimal rank at $f$.
But with such a choice it is not clear how to make compatible choices of  the needed coordinate changes.   As we will see, given two different charts at $[f]$ there is usually no natural  choice of a coordinate change from one to another; the natural maps arise by including each of them into a bigger chart (here called their sum).
Such a construction takes one quickly out of the class of minimal germs.
\end{rmk}

\subsection{Topological issues in the construction of a virtual fundamental class}  \hspace{1mm}\\ \vspace{-3mm}  \label{ss:top}

After one has solved the analytic issues involved in constructing compatible basic Kuranishi
charts as defined in Section~\ref{ss:kur} for a given moduli space $\oMm$, the further difficulties in constructing the virtual fundamental class $[\oMm]^{\rm virt}$ are all essentially topological, though their solution will impose further requirements on the construction of a Kuranishi atlas.
The basic idea for constructing a VMC is to make transverse perturbations $s_i+\nu_i\approx s_i$ of the section in each basic chart, such that the smooth zero sets modulo a relation given by the transition maps provide a regularization of the moduli space
$$
\oMm \,\!^\nu := \; \quotient{{\underset{{i=1,\ldots,N}}{\textstyle \bigsqcup}
}\; (s_i+\nu_i)^{-1}(0)} { \sim}
$$
There are various notions of regularizations; the common features (in the case of trivial isotropy and empty boundary) are that $\oMm\,\!^\nu$ should be a CW complex with a distinguished homology class $[\oMm\,\!^\nu]$ (e.g.\ arising from an orientation and triangulation), and that in some sense this class should be independent of the choice of perturbation~$\nu$.
For example, \cite{FO,FOOO} require that for any CW complex $Y$ and continuous map $f:\oMm \to Y$ that extends compatibly to the Kuranishi charts, the induced map $f:\oMm\,\!^\nu \to Y$ is a cycle, whose homology class is independent of the choice of $\nu$ and extension of $f$.
The basic issues in any regularization are that we need to make sense of the equivalence relation and ensure that the zero set of a transverse perturbation is not just locally smooth (and hence can be triangulated locally), but also that the transition data glues these local charts to a compact Hausdorff space without boundary. These properties are crucial for obtaining a global triangulation and thus well defined cycles.
For simplicity we aim here for the strongest version of regularization, giving $\oMm \,\!^\nu$ the structure of an oriented, compact, smooth manifold, which is unique up to cobordism.
That is, we wish to realize $\oMm \,\!^\nu$ as an abstract compact 
manifold as follows.
(We simplify here by deferring the discussion of orientations to the end of this section.)

\begin{definition} \label{def:mfd}
An abstract compact smooth manifold of dimension $d$ consists of
\begin{itemlist}
\item[{\bf (charts)}]
a finite disjoint union $\underset {{i=1,\ldots,N}}{\bigsqcup} V_i$ of open subsets $V_i\subset \R^d$,
\item[{\bf (transition data)}] 
for every pair $i,j\in\{1,\ldots,N\}$ an open subset $V_{ij} \subset V_i$ 
and a smooth embedding $\phi_{ij} : V_{ij} \hookrightarrow V_j$
such that $V_{ji} = \phi_{ij}(V_{ij})$ and $V_{ii} = V_i$,
\end{itemlist}
satisfying the {\bf cocycle condition}
$$
\phi_{jk} \circ \phi_{ij} = \phi_{ik} \qquad\text{on}\;\; \phi_{ij}^{-1}(V_{jk}) \subset V_{ik}
\qquad\quad
\forall i,j,k\in\{1,\ldots,N\},
$$
and such that the induced topological space
\begin{equation} \label{quotient}
\quotient{{\textstyle \underset{{i=1,\ldots,N}}{\bigsqcup}} V_i}{\sim}
\qquad\text{with}\quad
  x \sim y \;:\Leftrightarrow\; \exists \; i,j :  y=\phi_{ij}(x)
\end{equation}
is Hausdorff and compact.
\end{definition}

Note here that it is easy to construct examples of charts and transition data that satisfy the cocycle condition but fail to induce a Hausdorff  space, e.g.\ $V_1=V_2=(0,2)$ with $V_{12}=V_{21}=(0,1)$ and $\phi_{12}(x)=\phi_{21}(x)=x$ does not separate the points $1\in V_1$ and $1\in V_2$.
However, if we rephrase the data of charts and transition maps in terms of groupoids, then, as we now show, the Hausdorff property of the quotient is simply equivalent to a properness condition.
In this paper we take a groupoid to be a topological category whose morphisms are invertible, whose spaces of objects and morphisms are smooth manifolds, and whose structure maps (encoding source, target, composition, identity, and inverse) are local diffeomorphisms.
 Such groupoids are often called {\it \'etale}.
For further details see e.g.\ \cite{Moe,Mbr}.

\begin{rmk}\label{rmk:grp}\rm
A collection of charts and transition data satisfying the cocycle condition as in Definition~\ref{def:mfd} induces a topological groupoid $\bG$, that is a category with
\begin{itemlist}
\item
the topological space of objects $\Obj=\Obj_\bG = {\textstyle{\bigsqcup}_{i=1,\ldots,N}} V_i$
induced by the charts,
\item
the topological space of morphisms $\Mor=\Mor_\bG= {\textstyle{\bigsqcup}_{i,j=1,\ldots,N}} V_{ij}$
induced by the transition domains, with
\begin{itemize}
\item[-] source map $s: \Mor \to \Obj$, $(x\in V_{ij}) \mapsto (x\in V_i)$,
and target map $ t: \Mor \to \Obj$, $(x\in V_{ij}) \mapsto (\phi_{ij}(x)\in V_j)$ induced
by the transition maps,
\item[-] composition $\Mor \leftsub{t}{\times}_s \Mor\to \Mor$,
$\bigl( (x\in V_{ij}), (y\in V_{jk})\bigr) \mapsto (x\in V_{ik})$ if $\phi_{ij}(x)=y$,
which is well defined by the cocycle condition,
\item[-]
identities $\Obj \to \Mor, x \mapsto x\in V_{ii} \cong V_i$, and inverses $\Mor \to \Mor, (x\in V_{ij}) \mapsto (\phi_{ij}(x)\in V_{ji})$, again well defined by the cocycle condition.
\end{itemize}
\end{itemlist}
Moreover, $\bG$ has the following properties.
\begin{itemlist}
\item[{\bf (nonsingular)}]
For every $x\in\Obj$ the isotropy group $\Mor(x,x)=\{\rm id_x\}$ is trivial.
\item[{\bf (smooth)}]
The object and morphism spaces $\Obj$ and $\Mor$ are smooth manifolds.
\item[{\bf (\'etale)}]
All structure maps are local diffeomorphisms.
\end{itemlist}

\NI
The quotient space \eqref{quotient} is now given as the
{\it realization}
of the groupoid $\bG$, that is
$$
|\bG| := \Obj_\bG/\sim \qquad\text{with}\qquad
x \sim y \; \Leftrightarrow\; \Mor(x,y) \neq \emptyset .
$$
This realization is a compact manifold iff $\bG$ has the following additional properties.
\begin{itemlist}
\item[{\bf (proper)}]  The product of the 
source and target map $s\times t : \Mor \to \Obj \times \Obj$ is proper,  i.e.\ preimages of compact sets are compact.
%
%
\item[{\bf (compact)}]
$|\bG|$ is compact.
%
%
\end{itemlist}
\end{rmk}

Now let $(\bK_i = (U_i,E_i,s_i,\psi_i))_{i=1,\ldots,N}$ be a cover of  
a compact moduli space 
$\oMm$ by basic Kuranishi charts with footprints $F_i: = \psi_i(s_i^{-1}(0))\subset \oMm$.
Our guiding idea is to make from these charts two categories, the base category called $\bB_\Kk$, formed from the domains
$U_i$, and the 
bundle category $\bE_\Kk$ formed from the obstruction bundles.
The morphism spaces in both will arise from some type of transition maps between the basic charts.
The projections $U_i\times E_i \to U_i$ and sections $s_i$ should then induce a projection functor $\pi_\Kk:\bE_\Kk\to \bB_\Kk$ and a section functor $s_\Kk: \bB_\Kk \to \bE_\Kk$.
Further, the footprint maps $\psi_i$ induce a surjection $\psi_\Kk: s_\Kk^{-1}(0)\to \oMm$ from the zero set onto the moduli space. This induces natural morphisms in the subcategory $s_\Kk^{-1}(0)$, given by
$$
\psi_j^{-1} \circ \psi_i \, : \; s_i^{-1}(0)\cap \psi_i^{-1}(F_i\cap F_j) \;\longrightarrow\; s_j^{-1}(0) .
$$
If we use only these morphisms and their lifts to $s_i^{-1}(0)\times\{0\}\subset U_i\times E_i$, then composition in the categories $\bB_\Kk, \bE_\Kk$ is well defined, $\pi_\Kk, s_\Kk, \psi_\Kk$ are functors, and $|\psi_\Kk|:|s_\Kk^{-1}(0)|\to \oMm$ is a homeomorphism, which identifies the unperturbed moduli space $\oMm$ with a subset of the realization $|\bB_\Kk|$. However, these morphism spaces may be highly singular, so that the structure maps are merely local homeomorphisms between topological spaces.
This structure is insufficient for a regularization by transverse perturbation of the section $s_\Kk$, hence a Kuranishi atlas requires an extension of the transition maps $\psi_j^{-1} \circ \psi_i$ to diffeomorphisms between submanifolds of the domains. 

Recall here that the domains of the charts $U_i$ may not have the same dimension, since one can only expect the Kuranishi charts $\bK_i $ to have constant index $d= \dim U_i-\dim E_i$. 
Hence transition data is generally given by ``transition charts'' $\bK^{ij}$ and coordinate changes  $\bK_i \to \bK^{ij} \leftarrow \bK_j$ which in particular involve embeddings from open subsets $U_i^{ij}\subset U_i, U_j^{ij}\subset U_j$ into $U^{ij}$. (Here we simplify the notion from Section~\ref{ss:kur} by assuming that a single transition chart covers the overlap $F_i\cap F_j$.)
Now one could appeal to Sard's theorem to find a transverse perturbation in each basic chart,
$$
\nu=(\nu_i:U_i\to E_i)_{i=1,\ldots,N} \qquad\text{with}\quad s_i+\nu_i \pitchfork 0  \quad\forall i=1,\ldots,N ,
$$
and use this to regularize $\oMm \cong |s_\Kk^{-1}(0)|=\qu{s_\Kk^{-1}(0)}{\sim}$. Here the relation $\sim$ is given by morphisms, so the regularization ought to be the realization $|(s_\Kk+\nu)^{-1}(0)| = \qu{(s_\Kk + \nu)^{-1}(0)}{\sim}$ of a subcategory. 
Hence the perturbations $\nu_i$ need to be compatible with the morphisms, i.e.\ transition maps.
Given such compatible transverse perturbations, one obtains the charts and transition data for an abstract manifold as in Definition~\ref{def:mfd}, but still needs to verify the cocycle condition, Hausdorffness, compactness, and an invariance property to obtain a generalization of the finite dimensional regularization theorem on page \pageref{finite reg} to Kuranishi atlases along the following lines.

\medskip
\noindent
{\bf Kuranishi Regularization:} {\it
Let $s_\Kk: \bB_\Kk\to\bE_\Kk$ be a Kuranishi section of index~$d$ such that $|s_\Kk^{-1}(0)|$ is compact. Then there exists a class of smooth perturbation functors $\nu:\bB_\Kk\to\bE_\Kk$ such that the subcategory $\bZ_\nu:=(s_\Kk+\nu)^{-1}(0)$ carries the structure of an abstract compact smooth manifold of dimension $d$ in the sense of Definition~\ref{def:mfd}.
Moreover, $[\, |\bZ_\nu | \,]\in H_d(\bB_\Kk,\Z)$ is independent of the choice of such perturbations.
}

\medskip

However, in general there is no theorem of this precise form, since the topological issues  discussed below require various refinements of the perturbation construction.

\MS\NI
{\bf Compatibility:}
In order to obtain well defined transition maps, i.e.\ a space of morphisms in $\bZ_\nu$ with well defined composition, the perturbations $\nu_i$ clearly need to be compatible.
Since $(s_i + \nu_i)^{-1}(0)$ and $(s_j + \nu_j)^{-1}(0)$ are not naturally identified via $\psi_j^{-1} \circ \psi_i$ for $\nu\not\equiv 0$, this requires that one include in $\bB_\Kk$  the choice of specific transition data between the basic charts. 
Next, since the intersection of these embeddings $\phi_i^{ij}:U_i^{ij} \to U^{ij}$ and $\phi_j^{ij}:U_j^{ij} \to U^{ij}$ in the ``transition chart'' $\bK^{ij}$ is not controlled, the direct transition map
$(\phi_j^{ij})^{-1} \circ \phi_i^{ij}$ may not have a smooth extension that is defined on an open set.  
Therefore, we do not want to consider such maps to be morphisms in $\bB_\Kk$ since that would violate the \'etale property.
Instead, we include $\bK^{ij}$ into the set of charts, and ask that the pushforward of each perturbation $\nu_i,\nu_j$ extend to to a perturbation $\nu^{ij}$.  But now one must consider triple composites, and so on.

The upshot is that, as well as the system of basic charts $(\bK_i)_{i=1,\ldots,N}$ with footprints $F_i$,  one is led to consider a full collection of transition charts $(\bK_I)_{I\subset\{1,\ldots,N\}}$ with footprints $F_I: = \cap _{i\in I} F_i$.
To make a category, each of these ``sum charts" should have a chosen domain $U_I$, which is a smooth manifold, and the objects in $\bB_\Kk$ should be $\sqcup_I U_I$.  Further
the morphisms in the category should come from coordinate changes 
between these charts, which in particular involve embeddings $\phi_{IJ}:U_{IJ} \to U_J$ of open subsets $U_{IJ}\subset U_I$. 
Thus the space  $\Mor_{\bB_\Kk}$ will be the disjoint union of the domains $U_{IJ}$ of these coordinate changes over all relevant pairs $I,J$.\footnote{
See Definitions~\ref{def:Ku} and~\ref{def:catKu} for detailed definitions of the categories $\bB_\Kk$ and $\bE_\Kk$ along these lines.
It is worth noting here that we do not require the domains of charts to be open subsets of Euclidean space.  We could achieve this for the basic charts, since these can be arbitrarily small. However, for the transition charts one may need to make a choice between having a single sum chart for each overlap and having sum charts whose domains are topologically trivial.
We construct the former in Theorem~\ref{thm:A2}.
}

For this to form a category, all composites must exist, which is equivalent to the    
cocycle condition $\phi_{JK}\circ\phi_{IJ}=\phi_{IK}$ including the condition that domains be chosen such that $\phi_{IJ}^{-1}(U_{JK}) \subset U_{IK}$.
However, natural constructions as in Section~\ref{ss:gw} only satisfy the cocycle condition on the overlap $\phi_{IJ}^{-1}(U_{JK}) \cap U_{IK}$.
Thus already the construction of an equivalence relation from the transition data requires a refinement of the choice of domains, which we achieve in Theorem~\ref{thm:K} by iteratively choosing subsets of each $U_I$ and $U_{IJ}$. 
If we assume the cocycle condition, then 
$\bB_\Kk$ satisfies all 
properties of a nonsingular
 groupoid except 
\begin{itemize}
\item[-]  we do not assume that inverses exist;
\item[-] the \'etale condition is relaxed to require that the structure maps are smooth embeddings
(as spelled out in Definition~\ref{def:map}) rather than diffeomorphisms.
\end{itemize}
We write $|\Kk|$ for the realization $\Obj(\bB_\Kk)/\!\sim$ of $\bB_\Kk$, where $\sim$ is the equivalence relation generated by the morphisms, and denote by $\pi_\Kk:\Obj(\bB_\Kk)\to|\Kk|$ the projection.
We show in Lemma~\ref{le:Knbhd1} that 
the natural inclusion $\io_\Kk:
\oMm\to |\Kk|$ that is a homeomorphism to its image. Therefore we think of $|\Kk|$ as a {\it virtual neighbourhood} of $\oMm$.

This categorical framework and the resulting virtual neighbourhood of the moduli space is new in the Kuranishi context. The approach of both \cite{FO} (which \cite{FOOO} builds on though using different definitions) and \cite{J} is to work with equivalence classes of charts  at every point. It runs into the algebraic difficulties discussed in Section~\ref{ss:alg}.

\MS\NI
{\bf Hausdorff property:}  For a category such as  $\bB_\Kk$, it is no longer true that
the properness of $s\times t$ implies that its realization
$|\Kk|$ is Hausdorff; cf. Example~\ref{ex:Haus}.
Therefore, the easiest way to ensure that  
the realization of the perturbed zero set $|\bZ_\nu|$ 
is Hausdorff is to make $|\Kk|$ Hausdorff and check that the inclusion $|\bZ_\nu|\to |\Kk|$ is 
continuous.

The Hausdorff property (or more general properness conditions in the case of nontrivial isotropy) are not addressed in \cite{FO,FOOO,J}. Our attempts to deal with these requirements motivated the introduction of categories and a virtual neighbourhood.
Given this framework, most of Section~\ref{s:Ks} is devoted to finding a way to shrink the domains of the charts to achieve not only the cocycle condition but also ensure that $|\Kk|$ is Hausdorff.  To this end we introduce the notion of  {\it tameness} in Definition~\ref{def:tame}, which is a very strong form of the cocycle condition that gives great control on the morphisms in $\bB_\Kk$, cf.\ Lemma~\ref{le:Ku2}. 
We can achieve this if the original Kuranishi charts are additive, that is the obstruction spaces $E_i$ of the basic charts are suitably transverse. 
We then show in Proposition~\ref{prop:Khomeo} that the realization of a tame Kuranishi atlas is not only Hausdorff, but also has the property that the natural maps $U_I\to |\Kk|$ are homeomorphisms to their image. This means that we can construct a perturbation over $|\Kk|$ by working with its pullbacks to  each chart.

\MS\NI
{\bf Compactness:}   Unfortunately, even when we can make $|\Kk|$ Hausdorff, 
it is almost never locally compact or metrizable.  In fact, a typical local model is the subset $S$ of $\R^2$ formed by the union of the line $y=0$ with the half plane  $x>0$, with $\io_\Kk(\oMm) = \{y=0\}$, but given the  topology as a quotient of the disjoint union $\{y=0\}\sqcup \{x>0\}$. As we show
in Example~\ref{ex:Khomeo}, the quotient topology on $S$ is not metrizable, and even in the 
weaker subspace topology from $\R^2$ the zero set $\io_\Kk(\oMm)$ does not have a locally compact  neighbourhood in $|\Kk|$.
Therefore ``sufficiently small'' perturbations $\nu$ cannot guarantee compactness of the 
perturbed zero set $|\bZ_\nu|$.
Instead, the challenge is to find subsets of $|\Kk|$ containing $\io_\Kk(\oMm)$ that are compact 
and -- while not open -- are still large enough to contain the zero sets of appropriately perturbed sections $s+\nu$.

Similar to the Hausdorff property, this compactness is asserted in \cite{FO,FOOO,J} by quoting an analogy to the construction of an Euler class of orbibundles. However, we demonstrate in Examples~\ref{ex:Haus} and \ref{ex:Khomeo} that nontrivial Kuranishi atlases (involving domains of different dimension) -- unlike orbifolds -- never provide locally compact Hausdorff ambient spaces for the perturbation theory.

To solve the compactness issue, we introduce in Section \ref{ss:red} precompact subsets $\pi_\Kk(\Cc)$ of $|\Kk|$ with these properties; cf.\ Proposition~\ref{prop:zeroS0}.
In fact, for reasons explained below, we are forced to consider nested pairs $\Cc\subset \Vv\subset \Obj(\bB_\Kk)$ of such subsets, where the perturbation $\nu$ is defined over $\Vv$ so that the realization of its zero set is contained in $\pi_\Kk(\Cc)$.
One can think of $\pi_\Kk(\Cc)$ as a kind of neighbourhood of $\io_\Kk(\oMm)$, 
but, even though $\Cc$ is an open subset of $\Obj(\bB_\Kk)$, the image 
$\pi_\Kk(\Cc)$  is not open in $|\Kk|$ because the different components of $\Obj(\bB_\Kk)$
have different dimensions. For example, if $|\Kk|=S$ as above then $\pi_\Kk(\Cc)$ could be the union of 
$\{y=0, x<2\}$ with the set 
$\{x>1, |y|<1\}$. 
(In this example, since $\oMm$ is not compact, we 
cannot expect $\pi_\Kk(\Cc)$ to be precompact, but its closure is locally compact.)
 
\MS\NI
{\bf Construction of sections:}  We aim to construct the perturbation $\nu$ by finding
a compatible family of local perturbations $\nu_I$ in each chart $\bK_I$. 
Thus, if 
basic charts $\bK_i$ and $\bK_j$ have nontrivial overlap and we start by defining $\nu_i$,
the most naive approach is to try to  extend the partially defined perturbation
$\nu_i\circ (\phi_i^{ij})^{-1} \circ \phi_j^{ij}$  over $U_j$.
But, as seen above, the image of $(\phi_j^{ij})^{-1} \circ \phi_i^{ij}$ might be too singular to allow for an extension, and since $\phi_i^{ij}$ and $\phi_j^{ij}$ have overlapping images in $U^{ij}$ it does not help to rephrase this in terms of finding an extension of the pushforwards of these sections to $U^{ij}$.
Thus one needs some notion of a \lq\lq good coordinate system" on $\oMm$ such as in \cite{FO}, in which all compatibility conditions between the perturbations are given by pushforwards with embeddings. That is, two charts $\bK_I$ and $\bK_J$ of a \lq\lq good coordinate system" have either no overlap or all morphisms are given by an embedding $\phi_{IJ}$ or $\phi_{JI}$.

An approach towards extracting a ``good coordinate system'' from a Kuranishi atlas is given in \cite{FO} and built on by \cite{FOOO,J} but does not address compatibility with overlaps or the cocycle condition. 

We achieve this ordering by constructing a {\it reduction} $\Vv\subset\Obj(\bB_\Kk)$ in Proposition~\ref{prop:cov2}. This does not provide another Kuranishi atlas or collection of compatible charts (though see Proposition~\ref{prop:red}), but merely is a subset of the domain spaces that covers the unperturbed moduli space $\pi_\Kk\bigl(\Vv\cap s_\Kk^{-1}(0)\bigr)=\oMm$ 
and whose parts project to disjoint subsets $\pi_\Kk(\Vv\cap U_I)\cap\pi_\Kk(\Vv\cap U_J)=\emptyset$ in the virtual neighbourhood $|\Kk|$ unless there is a direct coordinate change between $\bK_I$ and $\bK_J$. 
Since the unidirectional coordinate changes induce an ordering, this allows for an iterative approach to constructing compatible perturbations $\nu_I$.
However, this construction in Proposition~\ref{prop:ext} is still very delicate and requires great control over the perturbation $\nu$ since, to ensure compactness of the zero set, 
we must construct it so that $\pi_\Kk\bigl((s+\nu)^{-1}(0)\bigr)$ lies in a precompact but generally not open set $\pi_\Kk(\Cc)$.   
In particular, this construction requires a suitable metric on $\pi_\Kk(\Vv)$, cf.\ Definition~\ref{def:metric}, which raises the additional difficulty of working with different topologies since -- as explained above -- the natural quotient topologies are almost never metrizable.

\MS\NI{\bf Regularity of sections:}   
In order to deduce the existence of transverse perturbations in a single chart from Sard's theorem, the section must be $\Cc^k$, where $k\ge 1$ is larger than the index of the Kuranishi atlas.
(This was overlooked in \cite{FO} and comments of \cite{J}.)
For applications to pseudoholomorphic curve moduli spaces this means that either a refined gluing theorem with controls of the derivatives must be proven, or a theory of stratified smooth Kuranishi is needed, which we are developing in \cite{MW:ku2}.

Moreover, when extending a transverse section $(\phi_{IJ})_*(s_I+\nu_I)$ from the image of the embedding $\phi_{IJ}$ to the rest of $U_J$, we must control its behavior in directions normal to the submanifold $\im \phi_{IJ}$ so that zeros of $s_I+\nu_I:U_I\to E_I$ in $U_{IJ}$ correspond to transverse zeros of $s_J+\nu_J:U_J\to E_J$.
That is, the derivative $\rd (s_J+\nu_J)$ must induce a surjective map from the normal bundle of $\phi_{IJ}(U_{IJ})\subset U_J$ to $E_J/\Hat\phi_{IJ}(E_I)$.
If this is to be satisfied at the intersection of several embeddings to $U_J$, then the construction of transverse sections necessitates a tangent bundle condition, which was introduced in \cite{J} and then adopted in \cite{FOOO}.
We reformulate it as an {\it index condition} relating the kernel and cokernels of $\rd s_I$ and $\rd s_J$ and can then extend transverse perturbations by requiring that $\rd \nu_J=0$ in the normal directions to all embeddings $\phi_{IJ}$. (See the notion of {\it admissible sections} in Definition~\ref{def:sect}.)

\MS\NI{\bf Uniqueness  up to cobordism:}  
Another crucial requirement on the perturbation constructions is that the resulting manifold (in the case of trivial isotropy) is unique modulo cobordism. This requires considerable effort since it does not just pertain to nearby sections of one bundle, but to sections constructed with respect to different metrics in different shrinkings and reductions of the charts.
Finally, in applications to pseudoholomorphic curve moduli spaces, a notion of equivalence between different Kuranishi atlases is needed. Contrary to the finite dimensional charts for transverse moduli spaces, or the Banach manifold charts for ambient spaces of maps, two Kuranishi charts for the same moduli space may not be directly compatible. 
Section \ref{ss:Kcobord} instead introduces a notion of {\it commensurability} by a common extension. In the application to Gromov-Witten moduli spaces \cite{MW:gw}, we expect to obtain this equivalence from an infinite dimensional index condition relating the linearized Kuranishi section to the linearized Cauchy-Riemann operator.

In order to prove invariance of the abstract VFC construction, however, we need to work with a weaker notion of cobordism of Kuranishi atlases -- a very special case of Kuranishi atlas with boundary, namely for $\oMm\times[0,1]$.
As a result, we need to repeat all shrinking, reduction, and perturbation constructions in Section~\ref{s:VMC} in a relative setting to interpolate between fixed data for $\oMm\times\{0,1\}$. Again, the rather general categorical setting -- rather than a base manifold with boundary -- introduces unanticipated subtleties into these constructions.

\MS
\NI {\bf Orientability:}
The dimension condition $\dim U_I-\dim E_I = \dim U_J-\dim E_J=:d$ together with the fact that each $s_I+\nu_I$ is transverse to $0$ implies that
the zero sets  of $s_I+\nu_I$ and $s_J+\nu_J$ both have dimension $d$, so that the embedding
$\phi_{IJ}$ does restrict to a local diffeomorphism between these local zero sets.
Thus, if all the above conditions hold, then the zero sets $(s_I+\nu_I)^{-1}(0)$ and morphisms induced by the coordinate changes do form an \'etale proper groupoid $\bZ_\nu$.
To give its realization $|\bZ_\nu|$ a well defined fundamental cycle, it remains to orient the local zero sets compatibly, i.e.\ to pick compatible nonvanishing sections of the determinant line bundles $\La^d\bigl(\ker \rd(s_I+\nu_I)\bigr)$.
These should be induced from a notion of orientation of the Kuranishi atlas, i.e.\ of sections of the unperturbed determinant line bundles $\La^{\rm max} \ker \rd s_I \otimes \bigl(\La^{\rm max} \coker \rd s_I\bigr)^*$, which are compatible with fiberwise isomorphisms induced by the embeddings $\phi_{IJ}: U_{IJ}\to U_J$.

To construct this {\it determinant line bundle} $\det(s_\Kk)$ of the Kuranishi atlas in Proposition~\ref{prop:det0}, we have to compare trivializations of determinant line bundles that arise from stabilizations by trivial bundles of different dimension.  
As recently pointed out by Zinger~\cite{Z3}, there are several ways to choose local trivializations  that are compatible with all necessary structure maps. We shall use  one that is different from both the original and the revised construction in \cite[Theorem~A.2.2]{MS}, since these lead to sign incompatibilities.
Our construction, though, does seem to coincide with ordering conventions in the construction of a canonical K-theory class on the space of linear operators between fixed finite dimensional spaces.\footnote{Thanks to Thomas Kragh for illuminating discussions on the topic of determinant bundles.
}
Finally, we use intermediate determinant bundles $\La^{\rm max} \rT U_I \otimes \bigl(\La^{\rm max} E_I\bigr)^*$ in Proposition~\ref{prop:orient1} to transfer an orientation of $\det(s_\Kk)$ to $\bZ_\nu$.

\MS
Putting everything together, we finally conclude in Theorem~\ref{thm:VMC1} that every
oriented weak additive $d$-dimensional Kuranishi atlas $\Kk$ with trivial isotropy determines a unique cobordism class of  oriented $d$-dimensional compact manifolds, that is represented by the zero sets of a suitable class of  admissible sections.
Theorem~\ref{thm:VMC2} interprets this result in more intrinsic terms, defining
a \v{C}ech homology class on $\oMm$, which we call the {\it virtual fundamental class} (VFC).
This is a stronger notion than in \cite{FO,FOOO}, where a virtual fundamental cycle is supposed to associate to any ``strongly continuous map'' $f:\oMm\to Y$ a cycle in $Y$. On the other hand, our notion of VFC does not yet provide a ``pull-push'' construction as needed for e.g.\ the construction of a chain level $A_\infty$ algebra in \cite{FOOO} by pullback of cycles via evaluation maps $\ev_1,\ldots,\ev_k:\oMm\to L$ and push forward by another evaluation $\ev_0:\oMm\to L$. 

Finally, note that our definition of a Kuranishi atlas is designed to make it possible both to construct them in applications, such as Gromov--Witten moduli spaces, and to prove that they have natural virtual fundamental cycles.  Its basic ingredients (charts, coordinate changes) are closely related to those in \cite{FO,FOOO, J}, yet we already need small variations. We make essential changes to almost all global notions and constructions and compare our notion of Kuranishi atlas to 
the various notions of Kuranishi structures
 in Remark~\ref{rmk:otherK}.

\begin{rmk}\rm  This paper makes rather few references to Li--Tian~\cite{LT} because that deals mostly with
 gluing and isotropy; in other respects it is very sketchy.  For example, it does not mention any of the analytic details in Section \ref{s:construct} below. Its Theorem~1.1  constructs the oriented Euler class of a \lq\lq generalized Fredholm bundle" $[s:\Tilde\Bb\to \Tilde\Ee]$, avoiding the Hausdorff question
 by assuming that there is a global finite dimensional bundle $\Tilde\Bb\times F$ that maps onto the local approximations.  However, in the Gromov--Witten situation this is essentially never the case (even if there is no gluing) since $\Tilde\Bb$ is a quotient of the form $\Hat\Bb/G$.  Therefore, we must work in the situation described in \cite[Remark~3]{LT}, and here they just say that the extension to this case is easy, without further comment.
Also, the proof that the structure described in Remark 3 actually exists even in the simple Gromov--Witten case that we consider in Section \ref{s:construct} lacks almost all detail. The only reference to this question is on page 79 in the course of the proof of Proposition 2.2 (page 38 in the arxiv preprint). 
Their idea is to build a global object from a covering family of basic charts using sum charts (see condition (iv) at the beginning of Section~1) and partitions of unity to extend sections. 
The paper \cite{LiuT} explains this idea with much more clarity, but unfortunately, because it does not pass to  finite  dimensional reductions, it makes serious analytic errors; cf.\ Remark~\ref{LTBS}.
As we point out in this remark, there are also serious difficulties with using partitions of unity in this context that cannot be easily circumvented by passing to finite dimensional reductions. Therefore at present it is unclear to us whether this construction can be correctly carried out.
\end{rmk}

%
%
%

\section{Differentiability issues in abstract regularization approaches}
\label{s:diff}

Any abstract regularization procedure for holomorphic curve moduli spaces needs to deal with the fundamental analytic difficulty of the reparametrization action, which has been often overlooked in symplectic topology. We thus explain in Section~\ref{ss:nodiff} the relevant differentiability issues in the example of spherical curves with unstable domain.
In a nutshell, the reparametrization $f\mapsto f\circ\ga$ with a fixed diffeomorphism $\ga$ is smooth on infinite dimensional function spaces, but the action $(\ga,f)\mapsto f\circ\ga$ of any nondiscrete family of diffeomorphisms fails even to be differentiable in any standard Banach space topology.
In geometric regularization techniques, this difficulty is overcome by regularizing the space of parametrized holomorphic maps in such a way that it remains invariant under reparametrizations. Then the reparametrization action only needs to be considered on a finite dimensional manifold, where it is smooth. It has been the common understanding that by stabilizing the domain or working in finite dimensional reductions one can overcome this differentiability failure in more general situations.
We will explain in Section~\ref{ss:DMdiff} that reparametrizations nevertheless need to be dealt with in establishing compatibility of constructions in local slices, in particular between charts near nodal curves and local slices of regular curves.
In particular, we will show the difficulties in the
global obstruction bundle approach in Section~\ref{ss:nodiff}, and for the Kuranishi atlas approach will see explicitly in Section \ref{ss:Kcomp} that the action on infinite dimensional function spaces needs to be dealt with when establishing compatibility of local finite dimensional reductions.
Finally, Section~\ref{ss:eval} explains additional smoothness issues in dealing with evaluation maps.

\subsection{Differentiability issues arising from reparametrizations}  \hspace{1mm}\\ \vspace{-3mm}
\label{ss:nodiff}

The purpose of this section is to explain the implications of the fact that the action of a nondiscrete automorphism group ${\rm Aut}(\Si)$ on a space of maps $\{ f: \Si \to M \}$ by reparametrization is not continuously differentiable in any known Banach metric.
In particular,
the space
$$
\{ f: \Si \to M \,|\, f_*[\Si]\neq 0 \}/{\rm Aut}(\Si),
$$
of equivalence classes of (nonconstant) smooth maps from a fixed domain modulo repara\-metrization of the domain, has no known completion with differentiable Banach orbifold structure.
We discuss the issue in the concrete case of the moduli space $\oMm_{1}(A,J)$ of $J$-holomorphic spheres with one marked point.\footnote{
In order to understand how any given abstract regularization technique deals with the differentiability issues caused by reparametrizations, one can test it on the example of spheres with one marked point. This is a realistic test case since since sphere bubbles will generally appear in any compactified moduli space (before regularization).
}
For the sake of simplicity let us assume that the nonzero class $A\in H_2(M)$ is such that it excludes bubbling and multiply covered curves a priori, so that no nodal solutions are involved and all isotropy groups are trivial. In that case one can describe the moduli space
\begin{align*}
\oMm_{1}(A,J)
&:= \bigl\{ (z_1,f) \in S^2 \times \Cc^\infty(S^2,M) \,\big|\, f_*[S^2]=A, \pbar f = 0 \bigr\} / \PSL(2,\C) \\
& \cong
\bigl\{ f \in \Cc^\infty(S^2,M) \,\big|\, f_*[S^2]=A, \pbar f = 0 \bigr\} / G_\infty
\end{align*}
as the zero set of the section $f\mapsto \pbar f$ in an appropriate bundle over the quotient
$$
\HBb/G_\infty \quad\text{of}\quad
\Hat\Bb:= \bigl\{ f \in \Cc^\infty(S^2,M) \,\big|\, f_*[S^2]=A \bigr\}$$
by the reparametrization action $f\mapsto f\circ\ga$ of
$$G_\infty : = \{\ga\in \PSL(2,\C) \,|\, \ga(\infty)=\infty\}.
$$
The quotient space $\HBb/G_\infty$ inherits the structure of a Fr\'echet manifold, but note that the action on any Sobolev completion
\begin{align} \label{action}
\Theta: G_\infty \times W^{k,p}(S^2,M) \to W^{k,p}(S^2,M), \quad
(\gamma,f) \mapsto f\circ\gamma
\end{align}
does not even have directional derivatives at maps $f_0\in W^{k,p}(S^2,M) \less W^{k+1,p}(S^2,M)$ since the differential\footnote{
Here the tangent space to the automorphism group $\rT_{\rm Id}G_\infty \subset\Ga(\rT S^2)$ is the finite dimensional space of holomorphic (and hence smooth) vector fields $X:S^2 \to \rT S^2$ that vanish at $\infty\in S^2$.
}
\begin{align}\label{eq:actiond}
{\rm D}\Theta ({\rm Id},f_0) : \;
\rT_{\rm Id}G_\infty
\times W^{k,p}(S^2, f_0^*\rT M) &\;\longrightarrow\; W^{k,p}(S^2, f_0^*\rT M) \\
 {(X,\xi)} \qquad\qquad\qquad\;\;\; &\;\longrightarrow\; \;\;   \xi + \rd f_0 \circ X \nonumber
\end{align}
is well defined only if $\rd f_0$ is of class $W^{k,p}$.
In fact, even at smooth points $f_0\in{\mathcal C}^\infty(S^2,M)$, this ``differential'' only provides directional derivatives of \eqref{action}, for which the rate of linear approximation depends noncontinuously on the direction. Hence \eqref{action} is not classically differentiable at any point.

\begin{remark} \rm  \label{BSdiff}
To the best of our knowledge, the differentiability failure of \eqref{action} persists in all other completions of $\Cc^\infty(S^2,M)$ to a Banach manifold -- e.g.\ using H\"older spaces.
The restriction of \eqref{action} to $\Cc^\infty(S^2,M)$ does have directional derivatives, and the differential is continuous in the $\Cc^\infty$ topology. Hence one could try to deal with \eqref{action} as a smooth action on a Fr\'echet manifold.
Alternatively, one could equip $\Cc^\infty(S^2,M)$ with a (noncomplete) Banach metric. Then
$$
\Theta:G_\infty \times \Cc^\infty(S^2,M) \to \Cc^\infty(S^2,M)
$$  has a bounded differential operator
$$
{\rm D}\Theta (\ga,f) : \rT_\ga G_\infty \times \Cc^\infty(S^2, f^*\rT M) \to \Cc^\infty(S^2, \ga^*f^*\rT M).
$$
 However, the differential fails to be continuous with respect to $f\in  \Cc^\infty(S^2,M)$ in the Banach metric.
Now continuous differentiability could be achieved by restricting to a submanifold $\Cc\subset\Cc^\infty(S^2,M)$ on which the map $f\mapsto \rd f$ is continuous. However, in e.g.\ a Sobolev or
H\"older metric, the identity operator  $\rT\Cc\to\rT\Cc$ would then be compact,
 so that $\Cc$ would have to be finite dimensional.

Finally, one could observe that the action \eqref{action} is in fact $\Cc^\ell$ when considered as
a map $G_\infty \times W^{k+\ell,p}(S^2,M) \to W^{k,p}(S^2,M)$. This might be useful for fixing the differentiability issues in the virtual regularization approaches with additional analytic arguments. In fact, this is essentially the definition of scale-smoothness developed in \cite{HWZ1} to deal with reparametrizations directly in the infinite dimensional setting.
\end{remark}

It has been the common understanding that virtual regularization techniques deal with the differentiability failure of the reparametrization action by working in finite dimensional reductions, in which the action is smooth.
We will explain below for the global obstruction bundle approach, and in Section \ref{ss:Kcomp} for the Kuranishi atlas approach, that the action on infinite dimensional spaces nevertheless needs to be dealt with in establishing compatibility of the local finite dimensional reductions.
In fact, as we show in Section~\ref{s:construct}, the existence of a consistent set of such finite dimensional reductions with finite isotropy groups for a Fredholm section that is equivariant under a nondifferentiable group action is highly nontrivial.
For most holomorphic curve moduli spaces, even the existence of not necessarily compatible reductions relies heavily on the fact that,
despite the differentiability failure, the action of the reparametrization groups generally do have local slices.
However, these do not result from a general slice construction for Lie group actions on a Banach manifold, but from an explicit geometric construction using transverse slicing conditions.

We now explain this construction,
and subsequently
show that it only defers the differentiability failure to the transition maps \eqref{transition} between different local slices.

\medskip

In order to construct local slices for the action of $G_\infty$ on a Sobolev completion of $\Hat\Bb$,
$$
\Hat\Bb^{k,p}:= \bigl\{ f \in W^{k,p}(S^2,M) \,\big|\, f_*[S^2]=A \bigr\},
$$
we will assume $(k-1)p>2$ so that $W^{k,p}(S^2)\subset \Cc^1(S^2)$.
Then any element of $\Hat\Bb^{k,p}/G_\infty$ can be represented as $[f_0]$, where the parametrization $f_0\in W^{k,p}(S^2,M)$ is chosen so that $\rd f_0(t)$ is injective for $t=0, 1 \in S^2=\C\cup\{\infty\}$.
With such a choice, a neighbourhood of $[f_0]\in \Hat\Bb^{k,p}/G_\infty$ can be parametrized by $[\exp_{f_0}(\xi)]$, for $\xi$ in a small ball in the subspace
$$
\bigl \{  \xi\in W^{k,p}(S^2,f_0^*\rT M) \ | \  \xi(t)\in \im \rd f_0(t)^\perp \; \text{for}\; t=0,1\bigr\}.
$$
Moreover, the map $\xi \mapsto [\exp_{f_0}(\xi)]$ is injective up to an action of the finite {\bf isotropy group}
$$
G_{f_0} = \{ \ga \in G_\infty \,|\, f_0\circ\ga = f_0 \} .
$$
In other words, for sufficiently small $\eps>0$, a $G_{f_0}$-quotient of
\begin{equation}\label{eq:slice}
\Bb_{f_0}:= \bigl\{ f\in \Hat\Bb^{k,p} \,\big|\, d_{W^{k,p}}(f,f_0)<\eps , f(0)\in Q_{f_0}^0 , f(1) \in Q_{f_0}^1 \bigr\}
\end{equation}
is a local slice for the action of $G_\infty$, where for some $\delta>0$
\begin{equation} \label{eq:hypsurf}
Q_{f_0}^t=\bigl\{\exp_{f_0(t)} (\xi) \,\big|\,  \xi \in \im \rd f_0(t)^\perp , |\xi|<\delta \bigr\}
\subset M
\end{equation}
 are codimension $2$ submanifolds transverse to the image of $f_0$ in two extra marked points $t=0,1$.
 For simplicity we will in the following assume that the isotropy group $G_{f_0} =\{{\rm id}\}$ is trivial, and that the submanifolds $Q_{f_0}^{t}$ can be chosen so that $f_0^{-1}(Q_{f_0}^{t})$ is unique for $t=0,1$.
Then, for sufficiently small $\eps>0$, the intersections $\im f\pitchfork Q_{f_0}^t$ are unique and transverse for all elements of $\Bb_{f_0}$. This proves that $\Bb_{f_0}$ is a {\it local slice} to the action of $G_\infty$ in the following sense.

\begin{lemma} \label{lem:slice}
For every $f_0\in\Hat\Bb^{k,p}$ such that $\rd f_0(t)$ is injective for $t=0,1$ and $G_{f_0} =\{{\rm id}\}$, there exist $\eps,\delta>0$ such that $\Bb_{f_0} \to \Hat\Bb^{k,p}/G_\infty$, $f\mapsto [f]$ is a homeomorphism to its image.
\end{lemma}

\begin{remark} \rm \label{rmk:unique}
If $f_0$ is pseudoholomorphic with closed domain, then trivial isotropy implies somewhere injectivity, see \cite[Chapter~2.5]{MS}; however this is not true for general smooth maps or other domains.
Thus to prove Lemma~\ref{lem:slice} for general $f_0$ with trivial isotropy, we must deal with the case of non-unique intersections. In that case one obtains unique transverse intersections for $f\approx f_0$ in a neighbourhood of the chosen points in $f_0^{-1}(Q_{f_0}^{t})$ and can prove the same result. We defer the details to \cite{MW:gw}, where we also prove an orbifold version of Lemma~\ref{lem:slice} in the case of  nontrivial isotropy.
In that case, one must define the local action of the isotropy group with some care.  However, it is always defined by a formula such as \eqref{transition}, and so in general is no more differentiable than the transition maps \eqref{transition} below.
\end{remark}

The topological embeddings $\Bb_{f}\to \Hat\Bb^{k,p}/G_\infty$ of the local slices provide a cover of $\Hat\Bb^{k,p}/G_\infty$ by Banach manifold charts.
The transition map between two such Banach manifold
charts centered at $f_0$ and $f_1$ is given in terms of the local slices by
\begin{equation}\label{transition}
\Gamma_{f_0,f_1} :\;
\Bb_{f_0,f_1}:=
 \Bb_{f_0}\cap G_\infty\Bb_{f_1}
\; \longrightarrow \;
\Bb_{f_1}, \qquad f\longmapsto f\circ \gamma_f ,
\end{equation}
where $\gamma_f\in G_\infty$ is uniquely determined by $\gamma_f(t)\in f^{-1}(Q_{f_1}^{t})$ for $t=0,1$ by our choice of $\Bb_{f_1}$.
Here the differentiability of the map
\begin{equation} \label{gf}
W^{k,p}(S^2,M)\to G_\infty, \quad f\mapsto \gamma_f
\end{equation}
is determined by that of the intersection points  with the slicing conditions for $t=0,1$,
$$
 W^{k,p}(S^2,M)\to S^2, \quad f \mapsto f^{-1}(Q_{f_1}^{t}) .
$$
By the implicit function theorem, these maps are $\Cc^\ell$-differentiable if $k>\ell + 2/p$ such that $W^{k,p}(S^2)\subset \Cc^\ell(S^2)$.
However, the transition map also involves the action \eqref{action}, and thus is non-differentiable at some simple examples of $f\in W^{k,p}\less W^{k+1,p}$, no matter how we pick $k,p$.

\begin{lemma} \label{le:Gsmooth}
Let $B\subset \Bb_{f_0}$ be a finite dimensional submanifold of $\Bb_{f_0}$ with the $W^{k,p}$-topology, and assume that it lies in the subset of smooth maps, $B\subset\Cc^\infty(S^2,M)\cap  \Bb_{f_0}$.
Then the transition map \eqref{transition} restricts to a smooth map
$$
B \cap G_\infty\Bb_{f_1}
\; \longrightarrow \; \Bb_{f_1}, 
\qquad f \;\longmapsto\; \Gamma_{f_0,f_1}(f) = f\circ \gamma_f .
$$
\end{lemma}
\begin{proof}
Since $B$ is finite dimensional, all norms on $\rT B$ are equivalent. In particular, we equip $B$ with the $W^{k,p}$-topology in which it is a submanifold of $\Bb_{f_0}$. Then the embeddings $B\to \Cc^\ell(S^2,M)$ for all $\ell\in\N$ are continuous and hence the above discussion shows that the map $B\to G_\infty$, $f\mapsto\ga_f$ given by restriction of \eqref{gf} is smooth.
To prove smoothness of $\Gamma_{f_0,f_1}|_{B \cap G_\infty\Bb_{f_1}}$ it remains to establish smoothness of the restriction of the action $\Theta$ in \eqref{action} to 
$$
\Theta_B \,: \; G_\infty \times B \;\longrightarrow \; W^{k,p}(S^2,M), \qquad (\ga,f) \;\longmapsto\;  f\circ \ga .
$$
For that purpose first note that continuity in $f\in B$ is elementary since, after embedding $M\hookrightarrow \R^N$, this is a linear map in $f$. Continuity in $\ga$ for fixed $f\in\Cc^\infty(S^2,M)$ follows from uniform bounds on the derivatives of $f$ (and could also be extended to infinite dimensional subspaces of $W^{k,p}(S^2,M)$ by density of the smooth maps). This proves continuity of $\Theta_B$.

Generalizing \eqref{eq:actiond}, with $\rT_{\ga_0} G_\infty\subset\Ga(\ga_0^*\rT S^2)$ the space of holomorphic (and hence smooth) sections $X:S^2 \to \ga_0^*\rT S^2$ that vanish at $\infty\in S^2$, the differential of $\Theta_B$ is 
\begin{align*}
{\rm D}\Theta_B (\ga_0,f_0) : \;
\rT_{\ga_0} G_\infty \times \rT_{f_0} B &\;\longrightarrow\; W^{k,p}(S^2, f_0^*\rT M) \\
 {(X,\xi)} \qquad&\;\longrightarrow\; \;\;   \xi\circ\ga_0 + \rd f_0 \circ X .
\end{align*}
It exists and is a bounded operator at all $(\ga_0,f_0)\in G_\infty \times B$ since by assumption $f_0$ is smooth, so it remains to analyze the regularity of this operator family under variations in $G_\infty \times B$.
Denoting by $L(E,F)$ the space of bounded linear operators $E\to F$, the second term,
$$
B \;\to\; L\bigl(\rT_{\ga_0} G_\infty , W^{k,p}(S^2, f_0^*\rT M)\bigr), \quad
f_0 \;\mapsto\; (\rd f_0)_* 
\qquad\text{given by}\;   (\rd f_0)_* X = \rd f_0 \circ X ,
$$
is smooth on the finite dimensional submanifold $B$ because 
$\Cc^\ell(S^2,\R^N) \to W^{k,p}(\rT S^2,\R^N)$, $f_0\mapsto \rd f_0$ is a bounded linear map for sufficiently large $\ell$, and the $\Cc^\ell$-norm on $B\subset \Cc^\infty(S^2,\R^N)$ is equivalent to the $W^{k,p}$-norm.
The first term,
\begin{equation}\label{thetag}
G_\infty \;\to\; L\bigl(\rT_{f_0} B , W^{k,p}(S^2, f_0^*\rT M)\bigr) , \quad
\ga_0 \;\mapsto\; \theta_{\ga_0} \qquad\text{given by}\;\theta_{\ga_0}(\xi) = \xi\circ\ga_0 ,
\end{equation}
is of the same type as $\Theta_B$, hence continuous by the above arguments. 

This proves continuous differentiability of $\Theta_B$. Then continuous differentiability of the first term \eqref{thetag} follows from the same general statement about differentiability of reparametrization by $G_\infty$, and thus implies continuous differentiability of ${\rm D}\Theta_B$. Iterating this argument, we see that all derivatives of $\Theta_B$ are continuous, and hence $\Theta_B$ is smooth, as claimed.
Note however that this argument crucially depends on the finite dimensionality of $B$ to obtain continuity for the second term of ${\rm D}\Theta_B$.
\end{proof}

An important observation here is that the Cauchy--Riemann operator
$$
\pbar : \; \Hat\Bb^{k,p} \;\longrightarrow\; \Hat\Ee:= {\textstyle \bigcup_{f\in\Hat\Bb^{k,p}}} W^{k-1,p}(S^2,\Lambda^{0,1}f^*\rT M)
$$
restricts to a smooth section $\pbar:\Bb_{f_i}\to\Hat\Ee|_{\Bb_{f_i}}$ in each local slice.
The bundle map
$$
\Hat\Ga_{f_0,f_1} : \; \Hat\Ee|_{\Bb_{f_0,f_1}}  \;\longrightarrow\;  \Hat\Ee|_{\Bb_{f_1}}, \qquad
 \Hat\Ee_f  \;\ni\; \eta \;\longmapsto\; \eta \circ \rd \ga_f^{-1}  \;\in\; \Hat\Ee_{f\circ\ga_f} ,
$$
intertwines the Cauchy--Riemann operators in different local slices,
$$
\Hat\Ga_{f_0,f_1} \circ \pbar = \pbar \circ \Ga_{f_0,f_1} .
 $$
However, general perturbations of the form $\pbar + \nu : \Bb_{f_1} \to \Hat\Ee|_{\Bb_{f_1}}$, where $\nu$ is a  $\Cc^1$ section of the bundle $\Hat\Ee|_{\Bb_{f_1}}$, are {\it not} pulled back to $\Cc^1$ sections of $\Hat\Ee|_{\Bb_{f_0,f_1}}$ by $\Hat\Ga_{f_0,f_1}$ since
\begin{equation}\label{trans2}
\Hat\Ga_{f_0,f_1}^{-1}\circ \nu \circ \Ga_{f_0,f_1} : \;  f \;\longmapsto\; \nu(f\circ\ga_f) \circ \rd\ga_f^{-1}
\end{equation}
does not depend differentiably on the points $f$ in the base $\Bb_{f_0}$.
In equations \eqref{graphsp} and \eqref{graph} we  give a geometric construction of a special class of sections $\nu$ that do behave well under this pullback.

\begin{remark}\rm \label{LTBS}
The lack of differentiability in \eqref{transition} and \eqref{trans2}
poses a significant problem in the global obstruction bundle approach to regularizing holomorphic curve moduli spaces.
This approach views the Cauchy--Riemann operator $\pbar:\Ti\Bb\to\Ti\Ee$ as a section of a topological vector bundle over an ambient space $\widetilde\Bb$ of stable $W^{k,p}$-maps modulo reparametrization, with a $W^{k,p}$-version of Gromov's topology. It requires a ``partially smooth structure'' on this space, in particular a smooth structure on each stratum. For example, the open stratum in the present Gromov--Witten example is $\Hat\Bb^{k,p}/G_\infty\subset\Ti\Bb$, for which smooth orbifold charts, isotropy actions, and transition maps are explicitly claimed in \cite[Proposition~2.15]{GLu} and implicitly in \cite{LiuT}.
The latter paper does not even prove continuity of isotropy and transition maps, though an argument was supplied by Liu for the 2003 revision of \cite[\S6]{Mcv}.  However, continuity does not suffice to preserve the differentiability of perturbation sections in local trivializations of $\Ti\Ee\to\Ti\Bb$ under pullback by isotropy or transition maps.

Another serious problem with this approach is its use of cutoff functions to extend  sections
defined on infinite dimensional local slices such as $\Bb_{f_0}$ to other local slices.
Since these cutoff sections are still intended to give Fredholm perturbations of $\pbar$, the cutoff functions must be $\Cc^1$ and remain so under coordinate changes.  The paper \cite{LiuT} gives no details here.
A construction is given in \cite[Appendix~D]{LuT},  but this paper implicitly assumes an invariant notion of smoothness on the strata of $\Ti\Bb$.  It is possible that one can avoid these problems by first passing to a finite dimensional reduction as in \cite{LT}.  However, as far as we are aware,
details of such an approach have not been worked out.
If they were to be worked out, we would consider them more as an approach of Kuranishi type than an approach using a global obstruction bundle.

The global obstruction bundle approach requires a smooth structure on the ambient infinite dimensional spaces such as $\Hat\Bb^{k,p}/G_\infty$.
One way to resolve this issue would be to use the scale calculus of polyfold theory, in which the action \eqref{action} and the coordinate changes \eqref{transition} are scale-smooth, hence $\Hat\Bb^{k,p}/G_\infty$ has the structure of a scale-Banach manifold -- the simplest nontrivial example of a polyfold. It is conceivable that the constructions of \cite{LiuT,Mcv} can be made rigorous by replacing Banach spaces with scale-Banach spaces, smoothness with scale-smoothness, and all standard calculus results (e.g.\ chain rule and implicit function theorem) with their correlates in scale calculus.
It is however also conceivable that the construction of a global obstruction bundle near nodal holomorphic curves requires a compatibility of strata-wise smooth structures, along the lines of a polyfold structure on $\widetilde\Bb$.

Siebert \cite[Theorem 5.1]{Sieb} also aims for a Banach orbifold structure on a space of equivalence classes of maps.
However, his notion is that of topological orbifold, i.e.\ with continuous transition maps. Indeed, his construction of local slices uses a (problematic for other reasons) averaged version of the slicing condition in \eqref{eq:slice}; thus the transition maps have the same form as \eqref{transition}, and hence fail differentiability.\footnote{
Instead, Siebert realized that differentiability does hold in all but finitely many directions.
This local classical differentiability can also be observed in all current applications of the polyfold approach. Differing from this approach, the construction of a ``localized Euler class'' in \cite[Thm.1.21]{Sieb} requires a section whose differential varies continuously in the operator norm, even in the nondifferentiable directions. However, at least in the fairly standard analytic setup of e.g.\ \cite{HWZ:gw, w:fred}, this is not the case.
}

\end{remark}

\subsection{Differentiability issues in general holomorphic curve moduli spaces}  \hspace{1mm}\\ \vspace{-3mm}
\label{ss:DMdiff}

The purpose of this section is to explain that the differentiability issues discussed in the previous section pertain to any holomorphic curve moduli space for which regularization is a nontrivial question.
The only exception to the differentiability issues are compactified moduli spaces that can be expressed as subspace of tuples of maps and complex structures on a {\it fixed domain},
\begin{equation}\label{safe}
\oMm = \bigl\{ (f,j) \in \Cc^\infty(\Si,M)\times {\mathfrak C}_\Si   \,\big|\, \overline{\partial}_{j,J} f = 0 \bigr\} ,
\end{equation}
where ${\mathfrak C}_\Si$ is a compact manifold of complex structures on a fixed smooth surface $\Si$.
In particular, this does not allow one to divide out by any equivalence relation of the type
\begin{equation} \label{rep}
(f,j) \sim (f\circ\phi, \phi^* j) \qquad \forall \phi\in{\rm Diff}(\Si) .
\end{equation}
For moduli spaces of this form, regularization can be achieved by the simplest geometric approach; namely choosing a generic domain-dependent almost complex structure $J:\Si \to \Jj(M,\om)$,
 with no further quotient or compactification needed.
One rare example of this setting is the $3$-pointed spherical Gromov--Witten moduli space $\oMm_3(A,J)$  for a class $A$ which excludes bubbling by energy arguments, since the parametrization can be fixed by putting the marked points at $0,1,\infty\in \C\cup\{\infty\}\cong S^2$, thus setting $\Si=S^2$ and ${\mathfrak C}_\Si=\{i\}$ in \eqref{safe}.
A similar setup exists for tori with $1$ or disks with $3$ marked points in the absence of bubbling, but we are not aware of further meaningful examples.
Generally, the compactified holomorphic curve moduli spaces are of the form
$$
\bigl\{ (\Si,\bz,f) \,\big|\, (\Si,\bz) \in {\mathfrak R}, f: \Si \to M, \overline{\partial}_{J} f = 0 \bigr\}  / \sim
$$
with
$$
 (\Si,\bz, f) \sim (\Si' , \phi^{-1}(\bz), f\circ\phi) \qquad \forall \phi:\Si'\to\Si .
$$
Here ${\mathfrak R}$ is some space of Riemann surfaces $\Si$ with a fixed number $k\in\N_0$ of pairwise distinct marked points $\bz\in\Si^k$, which contains regular as well as broken or nodal surfaces.
In important examples (Floer differentials and one point Gromov--Witten invariants arising from disks or spheres) all domains $(\Si,\bz)$ are unstable, i.e.\ have infinite automorphism groups.
If the regular domains are stable, unless bubbling is a priori excluded, ${\mathfrak R}/\!\!\sim$ is still not a 
Deligne--Mumford space since one has to allow nodal domains $(\Si,\bz)$ with unstable components to describe sphere or disk bubbles.

On the complement of nodal surfaces, these moduli spaces have local slices of the form \eqref{safe}
with additional marked points $\bz\in \Si^k \less \Delta$.
In the case of unstable domains, the slices are constructed by stabilizing the domain with additional marked points given by intersections of the map with auxiliary hypersurfaces.
In the case of stable domains, the slices are constructed by pullback of the complex structures to a fixed domain $\Si$, or fixing some of the marked points.
In fact, stable spheres, tori, and disks have a single slice covering the interior of the Deligne--Mumford space $\{(\Si,\bz) \;\text{regular, stable}\}/\!\!\sim$ given by fixing the surface and letting all but $3$ resp.\ $1$ marked point vary.
Using such slices, the differentiability issue of reparametrizations still appears in many guises:
\begin{enumerate}
\item
The transition maps between different local slices
-- arising from different choices of fixed marked points or auxiliary hypersurfaces --
are reparametrizations by biholomorphisms that vary with the marked points or the maps.
The same holds for local slices arising from different reference surfaces, unless the two families of diffeomorphisms to the reference surface are related by a fixed diffeomorphism, and thus fit into a single slice.
\item
A local chart for ${\mathfrak R}$ near a nodal domain is constructed by gluing the components of the nodal domain to obtain regular domains. Transferring maps from the nodal domain to the nearby regular domains involves reparametrizations of the maps that vary with the gluing parameters.
\item
The transition map between a local chart near a nodal domain and a local slice of regular domains is given by varying reparametrizations. This happens because the local chart produces a family of Riemann surfaces that varies with gluing parameters, whereas the local slice has a fixed reference surface.
\item
Infinite automorphism groups act on unstable components of nodal domains.
\end{enumerate}

The geometric regularization approach deals with issues (i), (iii), and (iv) by dealing with the biholomorphisms between domains only after equivariant transversality is achieved. This is possible only in restricted geometric settings; in particular it fails unless multiply covered spheres can be excluded in (iv).
Similarly, the geometric approach deals with issue (ii) by making gluing constructions only on finite dimensional spaces of smooth solutions that are cut out transversely.
We show in Remark~\ref{LTBS} and Section~\ref{ss:Kcomp} that these issues are highly nontrivial to deal with in abstract regularization approaches.
In the polyfold approach described in Section~\ref{ss:poly}, it is solved by introducing the notion of scale-smoothness for maps between scale-Banach spaces, in which the reparametrization action is smooth.
The other approaches have no systematic way of dealing with a symmetry group that acts nondifferentiably.

\begin{remark}\rm
One notable partial solution of the differentiability issues is the construction of Cieliebak-Mohnke \cite{CM} for Gromov--Witten moduli spaces in integral symplectic manifolds.\footnote
{
Ionel lays the foundations for a related approach in \cite{Io}.}
They use a fixed set of symplectic hypersurfaces to construct a global slice to the equivalence relation \eqref{rep}.
This reduces the differentiability issues to the gluing analysis near nodal curves, where the construction of a pseudocycle does not require differentiability.
This method fits into the geometric approach as described in Section~\ref{ss:geom} by working with a larger set of perturbations. The existence of suitable hypersurfaces is a special geometric property of the symplectic manifold and the type of curves considered.
\end{remark}

\subsection{Smoothness issues arising from evaluation maps} \hspace{1mm}\\ \vspace{-3mm} \label{ss:eval}

Another less dramatic differentiability issue in the regularization of holomorphic curve moduli spaces arises from evaluating maps at varying marked points.
This concerns evaluation maps of the form
$$
{\rm ev_i}: \;
\bigl\{ \bigl(\Si, \bz=(z_1,\ldots,z_k) ,f \bigr) \,\big|\,  \ldots \bigr\}/\!\!\sim \;\;\longrightarrow\;  M ,
\qquad
[\Si,\bz,f] \;\longmapsto\;  f(z_i)
$$
in situations when they need to be regularized while the moduli space is being constructed,
e.g.\ if they need to be transverse to submanifolds of $M$ or are involved in its definition via fiber products.
In those cases, the evaluation map needs to be included in the setup of a Fredholm section. However, on infinite dimensional function spaces its regularity depends on the Banach norm on the function space.
As a representative example, the map
\begin{equation} \label{evmap}
{\rm ev}: \; S^2 \times \Cc^\infty(S^2,M) \;\longrightarrow\; M , \qquad
(z,f) \;\longmapsto\; f(z)
\end{equation}
is $\Cc^\ell$ with respect to a Banach norm on $\Cc^\infty(S^2,M)$ only if the corresponding Banach space of functions, e.g.\ $\Cc^k(S^2)$ or $W^{k,p}(S^2)$, embeds continuously to $\Cc^\ell(S^2)$, e.g.\ if $k\geq\ell$ resp.\ $(k-\ell)p>2$. This can be seen from the explicit form of the differential
$$
{\rm D}_{(z_0,f_0)}{\rm ev}: \; T_{z_0}S^2 \times \Cc^\infty(S^2,f_0^*\rT M) \;\longrightarrow\; \rT_{f_0(z_0)}M , \qquad
(Z,\xi) \;\longmapsto\;  \rd f_0 (Z) + \xi(z_0)  ,
$$
whose regularity is ruled by the regularity of $\rd f_0$.
We will encounter this issue in the construction of a smooth domain for a Kuranishi chart in Section~\ref{ss:gw}, where the evaluation maps are used to express the slicing conditions that provide local slices to the reparametrization group. There we are able to deal with the lack of smoothness of \eqref{evmap} by first constructing a ``thickened solution space'', which is a finite dimensional manifold consisting of smooth maps and marked points that do not satisfy the slicing condition yet. Then the slicing conditions can be phrased in terms of the evaluation restricted to a finite dimensional submanifold of $\Cc^\infty(S^2,M)$.
This operator is smooth, but now it is nontrivial to establish its  transversality.

\begin{lemma} \label{le:evsmooth}
Let $B\subset W^{k,p}(S^2,M)$ be a finite dimensional submanifold, and assume that it lies in the subset of smooth maps, $B\subset\Cc^\infty(S^2,M)$.
Then the evaluation map \eqref{evmap} restricts to a smooth map
$$
\ev_B \; : \; S^2 \times B
\; \longrightarrow \; M , 
\qquad (z,f) \;\longmapsto\; f(z) .
$$
\end{lemma}
\begin{proof}
We will prove this by an iteration similar to the proof of Lemma~\ref{le:Gsmooth}, with Step $k$ asserting that maps of the type
\begin{equation}\label{type}
{\rm Ev} \;:\;
 \C \times \Cc^k(\C, \R^n )  \; \longrightarrow \; \R^n , \qquad (z,f) \;\longmapsto\; f(z) 
\end{equation}
are $\Cc^k$. 
In Step $0$ this proves continuity of the evaluation \eqref{evmap} on Sobolev spaces $W^{k,p}(S^2,M)$ that continuously embed to $\Cc^0(S^2,M)$.
In Step $k$ this proves that $\ev_B$ is $\Cc^k$ if we can check that the inclusion $B\hookrightarrow \Cc^k(S^2, M)$ is smooth if $B$ is equipped with the $W^{k,p}$-topology. Indeed, embedding $M\hookrightarrow\R^N$, this is the restriction of a linear map, which is bounded (and hence smooth) since $B$ is finite dimensional.
Hence to prove smoothness of $\ev_B$ it remains to perform the iteration.

Continuity in Step $0$ holds since we can estimate, given $\eps>0$,
\begin{align*}
\bigl| f(z) - f'(z') \bigr|
&\;\le\; 2 \| f - g\|_{\Cc^0} + \bigl| g(z) - g(z') \bigr| + \bigl|f(z') - f'(z')\bigr| \\
&\;\le\; \tfrac 12 \eps + \|\rd g\|_\infty |z-z'| + \|f - f'\|_{\Cc^0}  \;\le\;  \eps ,
\end{align*}
where we pick $g\in\Cc^1(\C,\R^n)$ sufficiently close to $f$, and then obtain the $\eps$-estimate for $(f',z')$ sufficiently close to $(f,z)$.

To see that Step $k$ implies Step $k+1$ we express the differential ${\rm D}_{(z_0,f_0)}\,{\rm Ev}:(Z,\xi)\mapsto \xi(z_0) + \rd_{z_0} f_0(Z)$ as sum of two operator families. The first family, 
$$
\C  \;\longrightarrow\; L\bigl( \Cc^{k+1}(\C, \R^n )  , \R^n \bigr) , \qquad
z_0 \;\longmapsto\;  {\rm Ev}(\cdot, z_0)  ,
$$
can be written as composition of $\C  \to L\bigl( \Cc^{k}(\C, \R^n )  , \R^n \bigr)$, $z_0 \mapsto {\rm Ev}(\cdot, z_0)$, which is $\Cc^k$ by Step $k$, and the bounded linear operator $L\bigl( \Cc^k(\C, \R^n )  , \R^n \bigr) \to L\bigl( \Cc^{k+1}(\C, \R^n )  , \R^n \bigr)$.
The second family,
$$
\C \times \Cc^{k+1}(\C, \R^n )  \;\longrightarrow\;  L\bigl( \C , \R^N \bigr) , \qquad
(z_0, f_0) \;\longmapsto\;  \rd_{z_0} f_0,
$$
can be written as composition of the linear map 
\begin{equation}\label{bugger}
\C \times \Cc^{k+1}(\C, \R^n )  \;\longrightarrow\; \C\times \Cc^k(\C,\bigl(\C, L(\C,\R^n) \bigr), \qquad 
(z_0, f_0) \;\longmapsto\;  (z_0, \rd f_0) ,
\end{equation}
which is a bounded linear operator hence smooth, and the evaluation map
$$
\C \times  \Cc^k(\C,\bigl(\C, L(\C,\R^n) \bigr) \;\longrightarrow\; L(\C,\R^n), \qquad
(z_0, \eta) \;\longmapsto\; \eta(z_0) ,
$$
which is of the type \eqref{type} dealt with in Step $k$, hence also $\Cc^k$ by iteration assumption.
This proves that the differential of evaluation maps of type ${\rm Ev} : \C \times \Cc^{k+1}(\C,\R^n) \to \R^n$  is $\Cc^k$, i.e.\ th emaps are $\Cc^{k+1}$, which finishes the iteration step and hence proof of smoothness of $\ev_B$.

Again note that this argument makes crucial use of the finite dimensionality of $B$ to obtain continuity of the embeddings $B\hookrightarrow \Cc^k(S^2, M)$ to prove that $\ev_B$ is $\Cc^k$. Here the increase in differentiability index $k$ is necessary  to obtain boundedness of \eqref{bugger} in the iteration step.
\end{proof}

\section{On the construction of compatible finite dimensional reductions}  \label{s:construct}

This section gives a general outline of the construction of a Kuranishi atlas on a given moduli space of holomorphic curves, concentrating on the issues of  dividing by the reparametrization action and making charts compatible. We thus use the example of
the Gromov--Witten moduli space $\oMm_{1}(A,J)$ of $J$-holomorphic curves of genus $0$ with one marked point, and assume that the nonzero class $A\in H_2(M)$ is such that it excludes bubbling and multiply covered curves a priori.
(For example, $A$ could be $\om$-minimal as assumed in Section~\ref{ss:geom}.)
This allows us to use the framework of smooth Kuranishi atlases with trivial isotropy, that is developed in
Sections ~\ref{s:chart}--\ref{s:VMC}
of this paper.
The additional difficulties of finite isotropy groups and nodal curves require a stratified smooth
groupoid setting and will be developed in \cite{MW:ku2,MW:gw}.
We do not claim that this is a general procedure for regularizing other moduli
spaces of holomorphic curves, but it does provide a guideline for similar constructions.

Recall that in this simplified setting the compactified Gomov--Witten moduli space
\begin{align*}
\oMm_{1}(A,J)
&:=  \bigl\{ (z_1=\infty, f) \in S^2\times \Cc^\infty(S^2,M) \,\big|\, f_*[S^2]=A, \pbar f = 0 \bigr\} / G_\infty
\end{align*}
is the solution space of the Cauchy--Riemann equation modulo reparametrization by
$$
G_\infty : = \{\ga\in \PSL(2,\C) \,|\, \ga(\infty)=\infty\}.
$$
We begin by discussing the construction of basic Kuranishi charts for $\oMm_{1}(A,J)$ in Section~\ref{ss:Kchart}, where we find that an abstract approach runs into differentiability issues in reducing to a local slice of the action of $G_\infty$. However, this can be overcome by using the infinite dimensional local slices that are constructed geometrically in Section~\ref{ss:nodiff}.
In Section~\ref{ss:Kcomp} we discuss the compatibility of a pair of basic Kuranishi charts, showing again that a  sum chart and coordinate changes cannot be constructed abstractly (e.g.\ from the given basic charts), but require specifically constructed obstruction bundles, which transfer well under the action of $G_\infty$.
Finally in Section~\ref{ss:gw} we give an outline of the construction of a full Kuranishi atlas for $\oMm_1(A,J)$.

\subsection{Construction of basic Kuranishi charts} \label{ss:Kchart}  \hspace{1mm}\\ \vspace{-3mm}

The construction of basic Kuranishi charts for the Gromov--Witten moduli space $\oMm_{1}(A,J)$ requires local finite dimensional reductions of the Cauchy--Riemann operator
\begin{equation} \label{eq:dbar}
\pbar : \; \Hat\Bb^{k,p}= W^{k,p}(S^2,M) \;\longrightarrow\; \Hat\Ee:= {\textstyle \bigcup_{f\in\Hat\Bb^{k,p}}} W^{k-1,p}(S^2,\Lambda^{0,1}f^*\rT M),
\end{equation}
and simultaneously a reduction of the noncompact Lie group $G_\infty$ to a finite isotropy group; namely the trivial group in the case considered here.
We begin by giving an abstract construction of a finite dimensional reduction for an abstract equivariant Fredholm section. Note that by the previous discussion, holomorphic curve moduli spaces do not exactly fall into this abstract setting.
However, our purpose is to demonstrate the need to deal with the reparametrization group in infinite dimensional settings.

\begin{remark}\rm
To simplify the reading of the following sections, let us explain our notational philosophy.
We use curly letters for locally noncompact spaces and roman letters for finite dimensional spaces.
We also use the hat superscript to denote spaces on which an automorphism group acts, or the slicing conditions are not (yet) applied. For example, $\Bb_{f_0} \subset \Hat \Bb^{k,p}$ is an infinite dimensional local slice in a Banach manifold $\Hat \Bb^{k,p}$ of maps,
the {\it local thickened solution space} $\Hat U$ is a finite dimensional submanifold of $\Hat \Bb$,
 and $U\subset \Hat U$ is the subset satisfying a slicing condition.

For bundles we again use curly letters if the fibers are infinite dimensional and roman letters if they are finite dimensional, with hats indicating that the base is infinite dimensional and tildes indicating that it is finite dimensional.
For example, $\Hat\Ee\to\Hat\Bb^{k,p}$ is a bundle with infinite dimensional fibers over a Banach manifold, while $\Hat E \subset \Hat\Ee|_{\Hat\Bb}$ has finite dimensional fibers $\Hat E |_f$ over points $f\in\Hat\Bb$ in an open subset $\Hat\Bb \subset \Hat\Bb^{k,p}$.
We will always write the fiber at a point as a restriction $\Ti E|_f$, since we require subscripts for other purposes.
Namely, when constructing a finite dimensional reduction near a point $f_0$, we use $f_0$ as subscript for the domains $U_{f_0}$ and restrictions of the bundles $\Ti E_{f_0}=\Hat E|_{U_{f_0}}$.
Moreover, we denote by $E_{f_0}$ a finite dimensional vector space isomorphic to the fibers
$(\Ti E_{f_0})|_f$ of $\Ti E_{f_0}$.

Finally, the symbol $\approx$  is used to mean ``sufficiently close to". Thus for $\ga\in G_\infty$, the set
$\{\ga\approx id\}$ is a neighbourhood of the identity.
\end{remark}

\begin{lemma} \label{le:fobs}
Suppose that $\si:\Hat\Bb\to\Hat\Ee$ is a smooth Fredholm section that is equivariant under the smooth, free, proper action of a finite dimensional Lie group $G$.
For any $f\in\si^{-1}(0)$ let $E_f\subset \Hat\Ee|_f$
be a finite rank complement of $\im {\rm D}_f\si\subset \Hat\Ee|_f$,
and let $\rT_f (Gf)^\perp \subset \ker{\rm D}_f\si$ be a complement of the tangent space of the $G$-orbit inside the kernel.
There exists a smooth map $s_f: W_f \to E_f$ on a neighbourhood $W_f\subset \rT_f (Gf)^\perp$ of~$0$ and a homeomorphism $\psi_f: s_f^{-1}(0)\to \si^{-1}(0)/G$ to a neighbourhood of $[f]$.
\end{lemma}
\begin{proof}
Let $\Hat E \subset\Hat\Ee|_{\Hat \Vv}$ be the trivial extension of $E_f\subset \Hat\Ee|_f$
given by a local trivialization $\Hat\Ee|_{\Hat \Vv} \cong \Hat\Vv \times \Hat\Ee|_f$
over an open neighbourhood $\Hat\Vv\subset\Hat\Bb$ of $f$.
Then $\Pi\circ\si : \Hat\Vv \to \Hat\Ee_{\Hat\Vv}/\Hat E$
is a smooth Fredholm operator that is transverse to the zero section. Thus by the implicit function theorem the thickened solution space
$$
\Hat U_f := \{ \, g\in \Hat\Vv \,|\, \si(g)\in \Hat E \, \} \; \subset \Hat\Bb
$$
is a submanifold of finite dimension ${\rm ind}\,{\rm D}_f\si + {\rm rk}\,E_f$. In particular, for small $\Hat\Vv$, there is an exponential map $\rT_f \Hat U_f \supset \Hat W_f \to \Hat U_f$. More precisely, this is a diffeomorphism
$$
\exp_f: \; \ker {\rm D}_f \si \;\supset\; \Hat W_f \;\overset{\cong}{\longrightarrow} \; \{\, g\in \Hat\Vv \,|\, \si(g)\in \Hat E \,\} \;=\; \Hat U_f
$$
from a neighbourhood $\Hat W_f\subset \ker {\rm D}_f \si$
of $0$ with $\exp_f(0)=f$ and $\rd_0\exp_f : \ker{\rm D}_f\si \to \rT_f \Hat\Bb$ the inclusion.
Note here that we chose the minimal obstruction space $E_f$ so that
$$
\rT_f \Hat U_f \;=\;
({\rm D}_f\si)^{-1}(E_f) \;=\; \ker {\rm D}_f(\Pi\circ\si) \;=\;  \ker {\rm D}_f\si.
$$
Via this exponential map we then obtain maps
\begin{align*}
\Hat s : \;\Hat W_f &\to \exp_f^*\Hat E,
\qquad\qquad\;\, \xi\mapsto \si(\exp_f(\xi))  , \\
\Hat\psi : \;\Hat s^{-1}(0) &\to \si^{-1}(0)/G, \quad\quad \xi\mapsto [\exp_f(\xi)]
\end{align*}
such that the section $\Hat s$ is smooth and  $\Hat\psi$ is continuous with image $[\Hat\Vv\cap\si^{-1}(0)]$.
Restricting to the complement of the infinitesimal action, $W_f:= \Hat W_f \cap \rT_f (Gf)^\perp$, and
trivializing
$\exp_f^*\Hat E \cong \Hat W_f \times E_f$ we obtain a smooth map $s_f$ and a continuous map $\psi_f$,
\begin{align*}
s_f:= \Hat s|_{\rT_f (Gf)^\perp} \;\; &: \; \quad W_f \to E_f,  \\
\psi_f:= \Hat\psi_f|_{\rT_f (Gf)^\perp} &: \; s_f^{-1}(0) \to \si^{-1}(0)/G .
\end{align*}
We need to check that $\psi_f$ is injective i.e.\  that every orbit of $G$ in $\Hat W_f$ intersects $\exp_f(s_f^{-1}(0))$ at most once.
We claim that this holds for $\Hat \Vv$ sufficiently small. By contradiction suppose $s_f^{-1}(0) \ni \xi_i, \xi'_i\to 0$, $\ga_i\in G\less\{{\rm id}\}$ satisfy $\ga_i\cdot \exp_f(\xi_i)= \exp_f( \xi'_i)$.
By continuity of $\exp_f$ this implies $(\ga_i\cdot\exp_f(\xi_i),\exp_f(\xi'_i))\to (f,f)$, and properness of the action implies $\ga_i\to\ga_\infty\in G$ for a subsequence. Since the action is also free, we have $\ga_\infty={\rm id}$.
This will constitute a contradiction once we have proven that the ``local action'' $\{\ga\approx{\rm id}\} \times W_f \to \si^{-1}(0)/G$ is injective on a sufficiently small neighbourhood of $({\rm id},0)$.
So far we have only used the differentiability of the $G$-action at a fixed point $f\in\Hat\Bb$ to define
$\rT_f(Gf)$. However, the proof of injectivity of the local action as well as local surjectivity of $\psi_f$ will rely heavily on the continuous differentiability of the $G$-action $G\times\Hat\Bb \to \Hat\Bb$.
(Intuitively, the problem is that our slice is given by a condition involving
a derivative of the $G$ action at $f$, and so is well behaved only if this derivative varies continuously with $f$.)

To finish the proof of the homeomorphism property of $\psi_f$ we pick $\Hat\Vv$ sufficiently small such that $\Hat U_f$ is covered by a single submanifold chart (i.e.\  a chart for $\Hat\Bb$ in a Banach space, within which $\Hat U_f$ is mapped to a finite dimensional subspace).
Then we can extend $\exp_f$ to an exponential map on the ambient space, i.e.\ a diffeomorphism
from a neighbourhood $\Hat\Ww_f\subset \rT_f\Hat\Bb$ of $\Hat W_f$,
$$
\Exp_f: \;   \Hat\Ww_f \;\overset{\cong}{\longrightarrow} \; \Hat\Vv   \qquad \text{with} \quad \Exp_f|_{\Hat W_f} = \exp_f , \quad \rd_0\Exp_f ={\rm id}_{\rT_f \Hat\Bb}.
$$
Note that
the existence of such an extension at least requires
continuous differentiability of $\Hat U_f$ resp.\ $\exp_f$.
Next, we also crucially use the continuous differentiability of the action $G\times\Hat\Bb\to\Hat\Bb$ to deduce that, for $\Hat \Vv$ sufficiently small, by the implicit function theorem
\begin{equation} \label{GBS}
\{ \ga \in G \,|\, \ga \approx {\rm id} \} \;\times\; \bigr(\Hat\Ww_f \cap  \rT_f (Gf)^\perp\bigl) \;\longrightarrow\; \Hat\Bb , \qquad
(\ga, \xi ) \;\longmapsto\; \ga\cdot \Exp_f(\xi)
\end{equation}
is a diffeomorphism to a neighbourhood of $f\in\Hat\Bb$.
The injectivity of \eqref{GBS} then implies that $\ga_i\cdot \exp_f(\xi_i) \neq \exp_f( \xi'_i)$ for  $\ga_i\neq{\rm id}$, which finishes the proof of injectivity of $\psi_f$.
More generally, the local diffeomorphism \eqref{GBS} implies that
$$
\Psi : \,\Hat\Ww_f \cap  \rT_f (Gf)^\perp \;\to\; \Hat\Bb/G, \qquad \xi\mapsto [\Exp_f(\xi)]
$$
is a homeomorphism to a neighbourhood $\Uu \subset  \Hat\Bb/G$ of $[f]$ (which in general is a proper subset of $[\Hat\Vv]$).
In particular, its image contains $\Hat\psi_f(\Hat s_f^{-1}(0))\cap\Uu=[\si^{-1}(0)]\cap\Uu$, and by construction
$$
\Psi \bigl( \Hat\Ww_f \cap  \rT_f (Gf)^\perp \bigr)
\;\cap\; \Hat\psi_f(\Hat s_f^{-1}(0))
\;=\; \Hat\psi_f\Bigl( \Hat\Ww_f \cap  \rT_f (Gf)^\perp \cap \Hat s_f^{-1}(0) \Bigr)
\;=\; \psi_f( s_f^{-1}(0) ) .
$$
This finally implies that the restriction $\psi_f =  \Psi|_{s_f^{-1}(0)}$ is a homeomorphism from  $s_f^{-1}(0)$ to the neighbourhood $\Uu\cap[\si^{-1}(0)] \subset \si^{-1}(0)/G$ of $[f]$, which completes the proof.
\end{proof}

\begin{remark} \rm \label{FOglue}
The above proof translates the construction of basic Kuranishi charts in \cite{FO}
in the absence of nodes and Deligne--Mumford parameters into a formal setup.
In \cite[12.23]{FO} this construction is described in the presence of nodes, in which case the construction of $\exp_f$ involves gluing analysis rather than just an exponential map.
Then the injectivity of $\psi_f$ is analogous to the claim of \cite[12.24]{FO}, where an argument is only given in the case of nontrivial Deligne--Mumford parameters.
In the case of a remaining nondiscrete automorphism group such as $G_\infty$, an abstract argument would have to proceed along the lines of Lemma~\ref{le:fobs}.
However, the map $(\ga,\xi)\mapsto \ga\cdot\exp_f(\xi)$ is continuously differentiable only if $\exp_f$ is $\Cc^1$ (which excludes most current gluing constructions) and has image in the smooth maps (which requires a very special construction of $\Hat E$). Moreover, one would at least need $\im \rd_0\exp_f(\rT_f (G_\infty f)^\perp ) + \rT_f (G_\infty f)$ to have maximal rank. This is not necessarily satisfied even for smooth gluing constructions  of $\exp_f$ since e.g.\ $\rd_0\exp_f$ could have nontrivial kernel. (In fact, one obvious method for making the gluing map smooth is to scale the gluing parameter such that $\rd_0\exp_f$ vanishes in that direction at the node.)
But note that this maximal rank does not seem to be sufficient to achieve the homeomorphism property of $\psi_f$.

Even in the absence of nodes, \cite{FO} construct the maps $\Hat s$ and $\Hat\psi$ on a ``thickened Kuranishi domain'' analogous to $\Hat W_f$ and thus need to make the same restriction to an ``infinitesimal local slice'' as in Lemma~\ref{le:fobs}.
Again, the argument for injectivity of $\psi_f$ given in Lemma~\ref{le:fobs} does not apply due to the differentiability failure of the reparametrization action of $G=G_\infty$ discussed in Section~\ref{ss:nodiff}.
One could apply the same argument to the embedding obtained by restricting \eqref{GBS} to the finite dimensional subspace $\Hat W_f$, as long as $\Hat U_f$ is contained in the smooth maps, 
and use the additional geometric information that the action of $G_\infty$ restricts to a smooth map from any finite dimensional submanifold consisting of smooth maps to $\Hat\Bb$.
(It does not restrict to a smooth action unless we find a finite dimensional, $G_\infty$-invariant submanifold of $\Hat\Bb$ consisting of smooth maps.)

Finally, the claim that $\psi_f$ has open image in $\si^{-1}(0)/G$ is analogous to \cite[12.25]{FO}, which asserts that ``$\Hat\psi(\Hat s^{-1}(0)\cap \exp_f(W_f))=\Hat\psi(\Hat s^{-1}(0))$ by definition''.\footnote{
Arguments towards a weaker localized version are now proposed in \cite[Prop.34.2]{FOOO12}.
}
A natural approach to proving this would use $G_\infty$-invariance of $\Hat U_f$. 
However, $G_\infty$-invariance of $\Hat U_f$ requires $G_\infty$-equivariance of $\Hat E$, i.e.\ an equivariant extension of $E_f$ to the infinite dimensional domain $\Hat\Vv$. A general construction of such extensions does not exist due to the differentiability failure of the $G_\infty$-action.
And again, the arguments of Lemma~\ref{le:fobs} do not apply since they use a local diffeomorphism to the infinite dimensional quotient space $\Hat\Bb/G$.
Now a finite dimensional version of these arguments would require an embedding of a finite dimensional slice into a $G_\infty$-invariant, smooth target space that contains a neighbourhood of $f$ in the solution set $\si^{-1}(0)$.  But there is no suitable candidate for such a space.
The unperturbed solution space $\si^{-1}(0)$ is $G$-invariant, so contains $\{\ga\approx {\rm id}\} \cdot \exp_f(s_f^{-1}(0))$, but may be highly singular, while the thickened solution space $\Hat U_f$ is smooth but generally not invariant under $G_\infty$, and so does not contain $\{\ga\approx {\rm id}\}\cdot \exp_f(W_f)$. 
Finally, some argument for the continuity of $\psi_f^{-1}$ is needed, though
 not mentioned in \cite{FO}; 
in Lemma~\ref{le:fobs} this also requires differentiability of the $G$-action on $\Hat\Bb$.
\end{remark}

In contrast to the differentiability failure of the reparametrization action discussed above, note that the gauge action on spaces of connections is generally smooth.
Hence Lemma~\ref{le:fobs} applies in gauge theoretic settings, with an infinite dimensional group $G$, and abstractly provides finite dimensional reductions or Kuranishi charts for the moduli spaces.
On the other hand, the differentiability issues in the construction of Kuranishi charts (and in particular coordinate changes between them) can only be resolved by using a geometrically explicit local slice $\Bb_{f}\subset \Hat\Bb^{k,p}$ as in \eqref{eq:slice}.
This is briefly mentioned in various places throughout the literature, e.g.\ \cite[Appendix]{FO}, but we could not find the analytic details that will be given in the following.

More precisely, the construction of a Kuranishi chart near $[f_0]\in\oMm_1(A,J)$ will depend on the choice~of
\begin{itemize}
\item
a representative $f_0\in[f_0]$;
\item
hypersurfaces $Q^0:=Q_{f_0}^0,Q^1:=Q_{f_0}^1 \subset M$ as in \eqref{eq:hypsurf}, and $\eps_{f_0}>0$ inducing a local slice
$$
\Bb_{f_0}:= \bigl\{ f\in \Hat\Bb^{k,p} \,\big|\, d_{W^{k,p}}(f,f_0)<\eps_{f_0} , f(0)\in Q_{f_0}^0 , f(1) \in Q_{f_0}^1 \bigr\} \;\subset \; \Hat\Bb^{k,p};
$$
\item
an obstruction space $E_{f_0}\subset\Hat\Ee|_{f_0}$ 
that covers the cokernel of the linearization at $f_0$ of the Cauchy-Riemann section \eqref{eq:dbar}, that is 
$\im {\rm D}_{f_0}\pbar  + E_{f_0} = \Hat\Ee|_{f_0}$;
\item
an extension of $E_{f_0}$ to a trivialized finite rank obstruction bundle $\Hat\Vv_{f_0}\times E_{f_0} \cong \Hat E_{f_0}  \subset\Hat\Ee|_{\Hat\Vv_{f_0}}$ over a neighbourhood $\Hat\Vv_{f_0}\subset\Hat \Bb^{k,p}$ of the slice $\Bb_{f_0}$.
\end{itemize}

With that we can construct the Kuranishi chart as a local finite dimensional reduction of the Cauchy--Riemann operator
$\pbar : \Bb_{f_0} \to \Hat\Ee|_{\Bb_{f_0}}$ in the slice to the action of $G_\infty$.
Note in the following that this construction requires the extension of the obstruction bundle $\Hat E_{f_0}$ to an open set of $\Hat\Bb^{k,p}$.

\begin{prop} \label{prop:A1}
For a sufficiently small slice $\Bb_{f_0}$, the subspace of generalized holomorphic maps with respect to the obstruction bundle $\Hat E_{f_0}$ is a finite dimensional manifold
\begin{equation}\label{Uf0}
U_{f_0}:=\bigl\{ f\in\Bb_{f_0} \,\big|\, \pbar f \in \Hat E_{f_0} \bigr\} .
\end{equation}
Moreover, $\Ti E_{f_0}:=\Hat E_{f_0}|_{U_{f_0}}\cong U_{f_0}\times E_{f_0}$ forms the bundle of a Kuranishi chart, whose smooth section and footprint map (a homeomorphism to a neighbourhood of $[f_0]$) are
$$
\begin{array}{rll}
\ti s_{f_0} \,:\; U_{f_0} &\to \; \Hat E_{f_0}|_{U_{f_0}}, & \quad f\mapsto \pbar f , \\
\psi_{f_0} \,: \; \ti s_{f_0}^{-1}(0) = \bigl\{ f\in \Bb_{f_0} \,\big|\, \pbar f =0\bigr\}&\to\; \oMm_{1}(A,J),& \quad
f\mapsto [f] .
\end{array}
$$
\end{prop}

\begin{proof}
%
%
%
We combine the local slice conditions and the perturbed Cauchy--Riemann equation to express $U_{f_0}$ as the zero set of
\begin{align*}
\Hat\Bb^{k,p} \;\supset\;
 \bigl\{ f \,\big|\, d_{W^{k,p}}(f,f_0)<\eps_{f_0} \bigr\}
&\;\longrightarrow\;
 \bigl( \Hat\Ee / \Hat E_{f_0}\bigr) \times
(\rT_{f_0(0)} Q^{0})^\perp\times (\rT_{f_0(1)} Q^{1})^\perp ,\\
f \quad &\longmapsto
\Bigl([\pbar f], \Pi^\perp_{Q^{0}}(f(0)), \Pi^\perp_{Q^{1}}(f(1))\Bigr),
\end{align*}
with projections $\Pi^\perp_{Q^t}$ near $f_0(t)$ along $Q^t$ to $T_{f_0(t)}(Q^t)^\perp$. 
Since the choice of $\Hat E_{f_0}$ guarantees that the linearized Cauchy-Riemann operator ${\rm D}_f\pbar$ maps onto $\Hat\Ee_f/\Hat E_{f_0}$ for $f=f_0$, and thus for nearby $f\approx f_0$, we obtain transversality of the full operator for sufficiently small $\eps_{f_0}>0$ if the linearized operator at $f_0$ maps the kernel of ${\rm D}_{f_0}\pbar$ onto the second and third factor. That is, we claim surjectivity of the map
$$
R_{f_0} : \;
\ker{\rm D}_{f_0}\pbar \;\ni\; \delta f \mapsto \bigl(\rd\Pi^\perp_{Q^{0}}(\delta f(0)), \rd\Pi^\perp_{Q^{1}}(\delta f(1))\bigr) .
$$
To check this, we can use the inclusion $\rT_{f_0} (G_\infty\cdot f_0)\subset\ker{\rm D}_{f_0}\pbar$ of a tangent space to the orbit together with the surjectivity of the infinitesimal action on two
points,
$$
\rT_{\rm id} G_\infty \; \to \; \rT_0 S^2 \times \rT_1 S^2
,\qquad \xi \; \mapsto \; \bigl(\xi(0), \xi(1) \bigr) .
$$
Combining these facts with $(\rT_{f_0(t)} Q^{t})^\perp =\im \rd f_0(t)$ we obtain transversality from
$$
R_{f_0} \bigl( \rT_{f_0} (G_\infty\cdot f_0) \bigr) \;=\;
\bigl(\rd\Pi^\perp_{Q^{0}} \times \rd\Pi^\perp_{Q^{1}}\bigr)
\bigl(\im \rd f_0(0) \times \im \rd f_0(1)\bigr) .
$$
This approach circumvents the differentiability failure of the $G_\infty$-action by working with the explicit local slice $\Bb_{f_0}$, which is analytically better behaved.
Moreover, the homeomorphism $\psi_{f_0}$ is given by restriction of the local homeomorphism $\Bb_{f_0}\to\Hat\Bb^{k,p}/G_\infty$ from Lemma~\ref{lem:slice}.
Finally, we need to find a trivialization of the obstruction bundle $\Ti E_{f_0}:=\Hat E_{f_0}|_{U_{f_0}}\cong U_{f_0}\times E_{f_0}$. For that purpose we choose $\eps_{f_0}>0$ even smaller. The effect of this on the bundle $\Ti E_{f_0}$ is a restriction to smaller neighbourhoods of $f_0$. Thus for sufficiently small $\eps_{f_0}>0$ the bundle over a smaller domain $U_{f_0}$ can be trivialized.
\end{proof}

A Kuranishi chart in the exact sense of  Definition~\ref{def:chart} can be obtained from the trivialization
$\Ti E_{f_0}\cong U_{f_0} \times E_{f_0}$.
However, to emphasize the geometric meaning of our constructions we
continue to use the notation for Kuranishi charts given in Section~\ref{ss:kur} in terms of a
bundle $\Tilde E_f\to U_f$ with section $\tilde s$.

\subsection{Compatibility of Kuranishi charts} \label{ss:Kcomp}  \hspace{1mm}\\ \vspace{-3mm}

As in Section~\ref{ss:kur} we oversimplify the formalism by saying that basic Kuranishi charts
$$
\bigl( \; \ti s_{f_i} : U_{f_i}\to \Ti E_{f_i}  \;,\;  \psi_{f_i} : \ti s_{f_i}^{-1}(0)\hookrightarrow \oMm_{1}(A,J) \;\bigr) \qquad \text{for}\; i=0,1 ,
$$
as constructed in the previous section from obstruction bundles  $\Hat E^i:=\Hat E_{f_i}$ over neighbourhoods of local slices $\Bb_{f_i}$,
are {\bf compatible} if the following transition data exists for every element in the overlap $[g_{01}]\in \im\psi_{f_0}\cap \im\psi_{f_1}\subset \oMm_{1}(A,J)$:
\begin{enumerate}
\item
a Kuranishi chart
$\quad\displaystyle
\bigl( \; \ti s_{g_{01}} : U_{g_{01}}\to \Ti E_{g_{01}}  \;,\;  \psi_{g_{01}} : \ti s_{g_{01}}^{-1}(0)\hookrightarrow \oMm_{1}(A,J) \;\bigr)
$\\
whose footprint $\im\psi_{g_{01}} \subset \im\psi_{f_0}\cap \im\psi_{f_1}$ is a neighbourhood of $[g_{01}]\in\oMm_1(A,J)$;
\item
for $i=0,1$ the transition map
arising from the footprints,
$$
\phi|_{\psi_{f_i}^{-1}(\im\psi_{g_{01}})} := \;
\psi_{g_{01}}^{-1}\circ \psi_{f_i} : \; \ti s_{f_i}^{-1}(0) \;\supset\; \psi_{f_i}^{-1}(\im\psi_{g_{01}}) \; \overset{\cong}{\longrightarrow}\; \ti s_{g_{01}}^{-1}(0)
$$
extends to a coordinate change consisting of an open neighbourhood $V_i\subset U_{f_i}$
of $\psi_{f_i}^{-1}(\im\psi_{g_{01}})$ and an embedding and linear injection in the trivialization $\Ti E_{f_i}\cong U_{f_i}\times E_{f_i}$ that intertwine the sections $\ti s_\bullet$,
$$
\phi :\; U_{f_i} \supset V_i \; \longhookrightarrow\; U_{g_{01}} , \qquad
\Hat\phi :\; E_{f_i} \; \longrightarrow\; E_{g_{01}} .
$$
\end{enumerate}

For notational convenience we will continue to construct the Kuranishi charts such that the domains have a canonical embedding $U \hookrightarrow \Bb^{k,p}/G_\infty$ (given by $f\mapsto [f]$ from a local slice $\Bb\subset \Bb^{k,p}$) which identifies the homeomorphism $\psi : s^{-1}(0) \hookrightarrow \oMm_1(A,J)$ with the identity on $\oMm_1(A,J)\subset \Bb^{k,p}/G_\infty$.
However, we will not use this ambient space for other purposes, since it has no direct generalization in the case of nodal curves.
In particular, the new domain $U_{g_{01}}$ cannot be constructed as an overlap of the domains $U_{f_i}$ since only the intersection of the possibly highly singular footprints $\im\psi_{f_0}\cap \im\psi_{f_1}\subset\oMm_1(A,J)$ has invariant meaning. Indeed, because the bundles $\Hat E^0, \Hat E^1$ may be quite different, the intersection $[U_{f_0}]\cap[U_{f_1}]\subset \Bb^{k,p}/G_\infty$ may only contain the intersection of footprints.
Moreover, the domains $U_{f_0},U_{f_1}\subset \Bb^{k,p}$ have no relation to each other beyond the fact that they are both spaces of perturbed solutions of the Cauchy--Riemann equation in a local slice.
Hence the Kuranishi chart (i) cannot be abstractly induced from the basic Kuranishi charts but needs to be constructed as another finite dimensional reduction of the Cauchy--Riemann operator.
With such a chart given, the transition map $\psi_{g_{01}}^{-1}\circ \psi_{f_i}$ between the zero sets is well defined, but its extension to a neighbourhood of $\psi_{f_i}^{-1}(\im\psi_{g_{01}})\subset \ti s_{f_i}^{-1}(0)$ in the domain $U_{f_i}$ also needs to be constructed. In fact, the need for this extension guides the construction of the chart.

For the rest of this section we will assume that the Kuranishi chart required in (i) can be constructed in the same way as the basic charts in Section~\ref{ss:Kchart}, and explain which extra requirements are necessary to guarantee the existence of a coordinate change (ii).
The chart (i) will be determined by the following data:

\begin{itemize}
\item
a representative $g_{01}\in[g_{01}]$;
\item
hypersurfaces $Q_{g_{01}}^0,Q_{g_{01}}^1 \subset M$ and $\eps_{g_{01}}>0$ inducing a local slice $\Bb_{g_{01}}\subset\Hat\Bb^{k,p}$;
\item
a finite rank subspace $E_{g_{01}}\subset\Hat\Ee |_{g_{01}}$ such that $\im {\rm D}_{g_{01}}\pbar  + E_{g_{01}} = \Hat\Ee|_{g_{01}}$;
\item
an extension  to a trivialized finite rank subbundle $\Hat\Vv_{g_{01}}\times E_{g_{01}} \cong \Hat E^{01}: =  \Hat E_{g_{01}} \subset\Hat\Ee|_{\Hat\Vv_{g_{01}}}$ over a neighbourhood $\Hat\Vv_{g_{01}}\subset\Hat\Bb^{k,p}$ of $\Bb_{g_{01}}$.
\end{itemize}

\NI
The coordinate change (ii) requires the construction of the following for $i=0,1$
\begin{itemize}
\item
open neighbourhoods $V_i\subset U_{f_i}$ of $\psi_{f_i}^{-1}(\im\psi_{g_{01}})$;
\item
embeddings $\phi_i : V_i \hookrightarrow U_{g_{01}}$ and a bundle map
$\Hat\phi_i : \Ti E_{f_i}|_{V_i} \to \Ti E_{g_{01}}$
covering $\phi_i$ and constant on the fibers in a trivialization, such that
$$
\Hat\phi_i \circ \ti s_{f_i} = \ti s_{g_{01}} \circ \phi_i , \qquad  \psi_{f_i} = \psi_{g_{01}} \circ \phi_i .
$$
\end{itemize}

In the explicit construction, we have
$V_i\subset U_{f_i} \subset \Bb_{f_i}$ and $U_{g_{01}} \subset\Bb_{g_{01}}$
both identified
with subsets of $\Bb^{k,p}/G_\infty$, and in this identification the embedding $\phi_i:V_i \hookrightarrow U_{g_{01}}$ is required to restrict to the identity on $\im\psi_{g_{01}}\subset\im\psi_{f_i}$.
So the natural extension of $\phi_i$ to a neighbourhood of $\psi_{f_i}^{-1}(\im\psi_{g_{01}})\subset U_{f_i}$ should lift the identity on $\Bb^{k,p}/G_\infty$. That is, with the domains $V_i\subset U_{f_i}$ still to be determined, we fix $\phi_i$ to be the transition map \eqref{transition}
between the local slices,
$$
\phi_i:= \Ga_{f_i,g_{01}}|_{V_i} : \; V_i \to \Bb_{g_{01}} , \quad f \mapsto f\circ\ga^{01}_f  ,
$$
where $\ga^{01}_f\in G_\infty$ is determined by $f\circ\ga^{01}_f\in\Bb_{g_{01}}$ .
Now in order for $\phi_i(V_i)$ to take values in $U_{g_{01}}$ we must have
\begin{equation}\label{givestrans}
\pbar f \in \Hat E^i|_{f} \;\Longrightarrow\; \pbar f\circ\rd\ga^{01}_f \in \Hat E^{01}|_{f\circ\ga^{01}_f} \qquad\forall \; f\in V_i.
\end{equation}
In particular for all $g\in \ti s_{g_{01}}^{-1}(0)$ we must have
\begin{equation} \label{E01 req}
\Hat E^0|_{g\circ \ga^0_g} \circ (\rd\ga^0_g)^{-1}
\;+\; \Hat E^1|_{g\circ\ga^1_g} \circ (\rd\ga^1_g)^{-1}  \;\subset\; \Hat E^{01}|_{g} ,
\end{equation}
where $\ga^i_g\in G_\infty$ is determined by $g\circ\ga^i_g\in\Bb_{f_i}$ and $\Hat E^i\subset\Hat \Ee|_{\Hat\Vv_{f_i}}$ is the obstruction bundle extending $E_{f_i}$.
Note here that we at least have to construct $\Hat E^{01}\to\Bb_{g_{01}}$ as a smooth obstruction bundle over an infinite dimensional slice, since this induces the smooth structure on the domain
$U_{g_{01}} =\{g\in\Bb_{g_{01}} \,|\, \pbar g \in \Hat E^{01} \}$.
(In fact, the proof of Lemma~\ref{Uf0} uses the obstruction bundle over an open set in $\Hat\Bb^{k,p}$.)
However, we encounter several obstacles in constructing $\Hat E^{01}$ such that \eqref{E01 req} is satisfied near ${g_{01}\in\Bb_{g_{01}}}$.
\MS

\begin{itemlist}
\item[{\bf \qquad\, 1.)}]
The left hand side of \eqref{E01 req} involves the pullbacks of $(0,1)$-forms by the transition map $\Ga_{g_{01}, f_i} : \Bb_{g_{01}} \to \Bb_{f_i}$ between local slices.
In fact, it is no surprise that the reparametrizations enter crucially, since $\Hat E^0$ and $\Hat E^1$ are bundles over neighbourhoods of the local slices $\Bb_{f_0}$ and $\Bb_{f_1}$ respectively, which may have no intersection in $\Hat\Bb^{k,p}$ at all, although they do have an open intersection in the quotient $\Hat\Bb^{k,p}/G_\infty$.
Since the transition maps are not continuously differentiable, the pullback bundles
$$
\Ga_{g_{01}, f_i}^*\Hat E^i
\,:= \;{\textstyle \bigcup_{g\in\Bb_{g_{01}}}} \Hat E^i|_{g\circ\ga^i_g} \circ (\rd\ga^i_g)^{-1}
$$
will not be differentiable in general. Thus we must find a special class of obstruction bundles, on which the pullback by reparametrizations acts smoothly.
\item[{\bf \qquad\, 2.)}]
Even if the pullback bundles $\Ga_{g_{01}, f_0}^*\Hat E^0$ and $\Ga_{g_{01}, f_1}^*\Hat E^1$ are differentiable, their fibers can have wildly varying intersections over $\Bb_{g_{01}}$. Here the diameter of the local slice can be chosen arbitrarily small, but it will always be locally noncompact. So it is unclear whether there even exists a finite rank subbundle of $\Hat\Ee|_{\Bb_{g_{01}}}$ that contains both pullback bundles.
To ensure this we must assume transversality at $g_{01}$,
$$
\bigl(\Hat E^0|_{g_{01}\circ(\ga^0_{g_{01}})^{-1}} \circ\rd\ga^0_{g_{01}} \bigr)
\cap
\bigl( \Hat E^1|_{g_{01}\circ(\ga^1_{g_{01}})^{-1}} \circ\rd\ga^1_{g_{01}} \bigr)  \;=\; \{ 0 \}  .
$$
\end{itemlist}

If the requirements in 1.) and 2.) are satisfied, then the sum of obstruction bundles
\begin{align*}
\Hat E^{01} &\,:=\;
\Ga_{g_{01}, f_0}^*\Hat E^0 \oplus \Ga_{g_{01}, f_1}^*\Hat E^1\\
&\;=\;
{\textstyle \bigcup_{g\in\Bb_{g_{01}}} } \bigl\{ \nu^0\circ (\rd\ga^0_g)^{-1} + \nu^1\circ (\rd\ga^1_g)^{-1}  \,\big|\, \nu^i\in \Hat E^i|_{g\circ\ga^{01}_g} \bigr\}
\end{align*}
is a smooth, finite rank subbundle of $\Hat\Ee$ over a local slice $\Bb_{g_{01}}$ of sufficiently small diameter $\eps_{g_{01}}>0$.
Under these assumptions, the constructions of Section~\ref{ss:Kchart} provide a Kuranishi chart for a neighbourhood of $[g_{01}]\in \oMm_1(A,J)$, which we also call {\bf sum chart} since it is given by a sum of obstruction bundles. Its domain and section are
$$
\ti s_{g_{01}} : \; U_{g_{01}} :=\{g\in\Bb_{g_{01}} \,|\, \pbar g \in \Hat E^{01} \}  \;\to\; \Hat E^{01} , \qquad g \mapsto \pbar g ,
$$
and the embedding $\Bb_{g_{01}}\to\Hat\Bb^{k,p}/G_\infty$ of the local slice restricts to a homeomorphism into the moduli space,
$$
\psi_{g_{01}}: \ti s_{g_{01}}^{-1}(0) \to  \oMm_1(A,J) , \qquad g\mapsto [g].
$$
Moreover, we already fixed the embeddings $\phi_i = \Ga_{f_i,g_{01}}$ and can read off from \eqref{givestrans} the corresponding embedding of obstruction bundles
$$
\Hat\phi_i:
\Hat E^i|_{f} \to \Hat E^{01}|_{f\circ\ga_f},
\qquad
\nu \mapsto  \nu\circ \rd\ga_f.
$$
Since this should be a constant linear map $E_{f_i}\to E_{g_{01}}$ in some trivialization
$\Hat E^{01}\cong U_{g_{01}}\times E^{01}_{g_{01}}$, the  trivialization map
$T^{01}(g) :\Hat E^{01}|_g \to E_{g_{01}}$ must be given by
$$
T^{01}(g) \,:\;
\sum_{i=0,1} \nu^i\circ (\rd\ga^i_g)^{-1}
\;\mapsto\;
\sum_{i=0,1} \Bigl( T^i(g_{01}\circ\ga^i_{g_{01}} ) ^{-1} \,T^i(g\circ\ga^i_g ) \; \nu^i \Bigr)
\circ (\rd\ga^i_{g_{01}})^{-1}
$$
in terms of the trivializations $T^i(f) :\Hat E^i|_f\overset{\cong}\to E_{f_i}$ of its factors.
In fact, this shows exactly what it means for the sum bundle $\Hat E^{01} = \Ga_{g_{01}, f_0}^*\Hat E^0 \oplus \Ga_{g_{01}, f_1}^*\Hat E^1$ to be smooth.

\MS

We now summarize the preceding discussion in the context of a tuple of $N$ charts  $\bigl(\bK_i = (U_{f_i},E_{f_i},s_{f_i},\psi_{f_i})\bigr)_{i=1,\ldots,N}$.
Generalizing conditions (i) and (ii) at the beginning of this section, we find that if these arise from obstruction bundles $\Hat E^i\to \Hat\Vv_{f_i}$ over neighbourhoods of local slices $\Bb_{f_i}$, the minimally necessary compatibility conditions require us to construct for every index subset $I\subset\{1,\ldots,N\}$ and every element $[g_0]\in\bigcap_{i\in I}\im\psi_i\subset\oMm_1(A,J)$ in the overlap of footprints
\begin{enumerate}
\item
 a {\bf sum chart} $\bK_{I,g_0}$ with obstruction space $E_{I,g_0} \cong
 \prod_{i\in I} E_{f_i}$, whose footprint $\im\psi_{I,g_0} \subset \bigcap_{i\in I}\im\psi_{f_i}$ is a neighbourhood of $[g_0]$;
\item
coordinate changes $\bigl(\bK_i \to \bK_{I,g_0}\bigr)_{i\in I}$ that extend the
transition maps
$\psi_{I,g_0}^{-1}\circ \psi_{f_i}$.
\end{enumerate}

\NI
The construction of a virtual fundamental class $[\oMm_1(A,J)]^{\rm vir}$ from a cover by compatible basic Kuranishi charts $\bigl(\bK_i \bigr)_{i=1,\ldots,N}$ in addition requires fixed choices of the above transition data, and further coordinate changes $\bK_{I,g_0}\to\bK_{J,h_0}$ satisfying a cocycle condition; see Section~\ref{ss:top}.
The main difficulty is to ensure that the sum charts are well defined.
The details of their construction are dictated by the existence of coordinate changes from the basic charts.
This construction is so canonical that coordinate changes between different sum charts exist essentially automatically, and satisfy the weak cocycle condition.
By the discussion in the case of two charts,
the following conditions on the choice of basic Kuranishi charts $\bigl(\bK_i = (U_{f_i},E_{f_i},s_{f_i},\psi_{f_i})\bigr)_{i=1,\ldots,N}$ ensure the existence of the sum charts (i) and transition maps (ii).

\begin{itemlist}
\item[{\bf \qquad\, Sum Condition I:}]
{\it For every $i\in \{1,\ldots,N\}$ let
$T^i(f) :\Hat E^i|_f\overset{\cong}\to E_{f_i}$ be induced by  the trivialization of the obstruction bundle.
Then for every $[g_0]\in\im\psi_i \cap\bigcap_{j\neq i}\im\psi_j$ and representative $g_0$ with sufficiently small local slice $\Bb_{g_0}$, the map}
\begin{align*}
\Bb_{g_{0}} \times E_{f_i} &\;\longrightarrow\; \quad \Hat\Ee \\
 (g, \nu_i )\quad & \;\longmapsto \; \bigl( T^i(g\circ\ga^i_g ) \, \nu_i \bigr) \circ (\rd\ga^i_g)^{-1}
\end{align*}
{\it is required to be smooth, despite the differentiability failure of $g\mapsto g\circ\ga^i_g$.}

An approach for satisfying this condition will be given in the next section.

\item[{\bf \qquad\, Sum Condition II:}]
{\it For every $I\subset\{1,\ldots,N\}$ and $[g]\in\bigcap_{i\in I}\im\psi_i$ we must ensure transversality
of the vector spaces
$\Hat E^i|_{g\circ(\ga^i_{g})^{-1}} \circ\rd\ga^i_{g}  = \bigl( T^i(g\circ(\ga^i_{g})^{-1})^{-1} E_{f_i}  \bigr) \circ\rd\ga^i_{g}$ for $i\in I$. That is, their sum needs to be a direct sum,
$$
\sum_{i\in I} \Hat E^i|_{g\circ(\ga^i_{g})^{-1}} \circ\rd\ga^i_{g}  \;=\;
\bigoplus_{i\in I} \Hat E^i|_{g\circ(\ga^i_{g})^{-1}} \circ\rd\ga^i_{g}    \quad\subset\;\Hat \Ee|_g .
$$
}

This means that, no matter how the obstruction bundles are constructed for each chart, the choices for a tuple need to be made ``transverse to each other''  along the entire intersection of the footprints
before transition data can be constructed.
\end{itemlist}

\subsection{Sum construction for genus zero Gromov--Witten moduli spaces}  \hspace{1mm}\\ \vspace{-3mm}  \label{ss:gw}

The purpose of this section is to explain the basic ideas of our project \cite{MW:gw} of constructing a Kuranishi atlas for the genus zero Gromov--Witten moduli spaces by combining the geometric perturbations of \cite{LT} with the gluing analysis of \cite{MS}.
A natural idea (suggested to us by e.g.\ Kenji Fukaya,
see \cite[Appendix]{FO}, and Cliff Taubes) for dealing with the failure of differentiability in the pullback construction for obstruction bundles is to introduce varying marked points so that the pullback by $\Gamma_{g_{01},f_i}$ no longer depends on the infinite dimensional space of maps, instead depending on a finite number of parameters.
For the sum construction of two Kuranishi charts $(U_{f_i},\ldots)_{i=0,1}$ arising from finite rank bundles $\Hat E^i\to \Hat\Vv_{f_i}$ over neighbourhoods of local slices $\Bb_{f_i}$, let us for simplicity of notation work in a slice $\Bb_{g_{01}}\subset\Bb_{f_0}$ so that $\ga^0_g\equiv {\rm id}$.
Thus we construct the domain of the sum chart as
$$
U_{g_{01}} = \bigl\{  \bigl( g , \ul{w} \bigr) \in \Bb_{g_{01}}\times (S^2)^2 \,\big|\, \pbar g \in \Hat E^0 + \Gamma_{\ul w}^* \Hat E^1 ,
 \ul{w}=(w^{0} , w^{1}) \in D_{01},
g(w^{t})\in Q_{f_1}^{t} \bigr\} .
$$
Here $D_{01}\subset (S^2)^2$ is a neighbourhood of $\ul{w}_{01}:=(w_{01}^0, w_{01}^1)$ with $w_{01}^t=g_{01}^{-1}(Q_{f_1}^t)$, and $\Gamma_{\ul w} : g\mapsto g\circ \gamma_{\ul w}$ is the reparametrization with
\begin{equation}\label{gaw}
\gamma_{\ul w}\in G_\infty  \quad \text{given by} \quad
\gamma_{\ul w}(t)=w^{t} \quad \text{for}\; t=0,1.
\end{equation}
Observe that, with varying marked points, the map $(g,w) \mapsto g(w)$ still only has the regularity of $g$, see Section~\ref{ss:eval}.
So the above $U_{g_{01}}$ is not cut out by a single smooth Fredholm section.
However, we may now consider the intermediate {\it thickened solution space}
$$
\Hat U_{g_{01}} = \bigl\{  ( g , \ul{w} ) \in \Bb_{g_{01}}\times D_{01} \,\big|\, \pbar g \in \Hat E^0 + \Gamma_{\ul w}^* \Hat E^1  \bigr\} \;\subset\; \Bb_{f_0}\times (S^2)^2,
$$
where we have not yet imposed the slicing conditions at the points $\ul{w}$.
Then the domain $U_{g_{01}}={\rm ev}^{-1}(Q_{f_1}^0\times Q_{f_1}^1)$ is cut out by the slicing conditions, which use the evaluation map on the finite dimensional thickened solution space:
$$
{\rm ev} : \; \Hat U_{g_{01}} \to (S^2)^2, \qquad  ( g , w^0, w^1 ) \mapsto  ( g(w^0) , g(w^1) ) .
$$
To check that this map is transverse to $Q_{f_1}^0\times Q_{f_1}^1$ at $(g_{01}, w_{01}^0, w_{01}^1)$, note that $\{0\}\times (\rT S^2)^2$ is tangent to the thickened solution space at this point
(crucially using the fact that the solution space
$\{g \,|\, \pbar g=0\}$
is $G_\infty$-invariant so that there is an infinitesimal action at $g_{01}$).
Moreover, at every point in the local slice $g\in\Bb_{f_1}$ we have $\im \rd_t g \pitchfork \rT_{g(t)} Q_{f_1}^t$, in particular at $g_{01}\circ\ga_{\ul w_{01}}$ with $\im \rd_t (g_{01}\circ\ga_{\ul w_{01}})= \im \rd_{w_{01}^t}g_{01}$.
Moreover, the evaluation map is smooth if we can ensure that the thickened solution space $\Hat U_{g_{01}}\subset \Cc^\infty(S^2,M)\times (S^2)^2$ contains only smooth functions.
Continuing the list of conditions on the choice of summable obstruction bundles from the previous section, this adds the following regularity requirement.

\begin{itemlist}
\item[{\bf \qquad\, Sum Condition III:}]
{\it The obstruction bundles $\Hat E^i\subset\Hat\Ee|_{\Hat\Vv_{f_i}}$ need to satisfy regularity,}
$$
\pbar g \in {\textstyle \sum_i } \, \Ga_{\ul w_i}^* \Hat E^i
\;\Longrightarrow \; g \in\Cc^\infty(S^2,M) .
$$
By elliptic regularity for $\pbar$, this holds if
$\Hat E^i|_{W^{\ell,p}\cap \Hat\Vv_{f_i}}\in W^{\ell,p}\cap\Hat\Ee$ for all $\ell\in\N$, or in terms of the trivializations $T^i(f):\Hat E^i_f \to E_{f_i}$ if the elements of $E_{f_i}$ are smooth
$1$-forms in $\Hat\Ee|_{f_i}$  and
$$
f\in W^{\ell,p} \; \Longrightarrow \;\im T^i(f)\subset W^{\ell,p} .
$$
This means that sections of $\Hat E^i$ are lower order, compact perturbations for $\pbar$, i.e.\ they are $sc^+$ in the language of scale calculus \cite{HWZ1}.
\end{itemlist}

\MS\NI
Finally, we need to ensure smoothness of the thickened solution space $\Hat U_{g_{01}}$, which can be viewed as the zero set of the section
$$
\Bb_{g_{01}} \times D_{01} \;\longrightarrow \; \Hat\Ee / ( \Hat E^0 + \Gamma^* \Hat E^1 ) , \qquad
(g,\ul w) \;\longmapsto \; \pbar g .
$$
Here the form of the summed obstruction bundle,
\begin{align*}
\Ga^* \Hat E^1 &\;=\; { \underset{\ul w \in (S^2)^2}{\textstyle{\bigcup}}}\Gamma_{\ul w}^* \Hat E^1 \; \longrightarrow \; \Bb_{g_{01}}\times D_{01}, \\
\bigl(\Ga^* \Hat E^1 \bigr)|_{(g,\ul w)} &\;=\; \bigl\{ \nu \circ \rd \ga_{\ul w}^{-1} \,\big|\, \ga_{\ul w} (t)= w^{t} , \nu \in \Hat E^1|_{g\circ\ga_{\ul w}} \bigr\},
\end{align*}
is dictated by fixing the natural embedding
$\phi_1 : U_{f_1}\cap G_\infty U_{g_{01}} \to U_{g_{01}}$
given by  $f\mapsto (f\circ\ga_f^{-1}, \ga_f(0), \ga_f(1))$, where $f\circ\ga_f^{-1}\in\Bb_{g_{01}}$.
Its inverse map is $(g,\ul w)\mapsto g\circ\ga_{\ul w}$, which maps to a neighbourhood of $\Bb_{f_1}$.
While the extension of $\Hat E^1$ to a neighbourhood of $\Bb_{f_1}\subset\Hat\Bb^{k,p}$ so far was mostly for convenience in the proof of Lemma~\ref{Uf0}, it now becomes crucial for the construction of this ``decoupled sum bundle''.
In fact, as in that lemma, we will also extend $\Hat E_{g_{01}}= \Hat E^0 + \Gamma^* \Hat E^1$ to a neighbourhood $\Hat\Vv_{g_{01}}$ of $\Bb_{g_{01}}$ to induce the smooth structure on $\Hat U_{g_{01}}$.
With this setup, Sum Condition I becomes smoothness of the map involving the trivialization
$T^1(f):\Hat E^1(f)\to E_{f_1}$,
\begin{equation}\label{wantsmooth}
\Hat\Vv_{g_{01}} \times D_{01} \times E_{f_1} \;\longrightarrow\; \Hat\Ee , \qquad
 (g, \ul w ,\nu ) \;\longmapsto \; \bigl( T^1(g\circ\ga_{\ul w} ) \, \nu \bigr) \circ \rd \ga_{\ul w}^{-1} .
\end{equation}
This still involves reparametrizations $(g,\ga_{\ul w})\mapsto g\circ \ga_{\ul w}$, which are not differentiable in any Sobolev topology on $\Hat\Vv_{g_{01}}$, since $\ul w\in D_{01}$ and thus $\ga_{\ul w}$ is allowed to vary. Thus the compatibility of Kuranishi atlases requires a very special form of the trivialization $T^1$, i.e.\ very special obstruction bundles $\Hat E^i$.

\MS\NI
{\bf Geometric construction of obstruction bundles:}
To solve the remaining differentiability issue, we now follow the more geometric approach of \cite{LT} and construct obstruction bundles by pulling back finite rank subspaces
\begin{equation}\label{graphsp}
E^i\subset \Cc^\infty(\Hom^{0,1}_J(S^2,M))
\end{equation}
of the space of smooth sections of the bundle over $S^2\times M$ of $(j,J)$-antilinear maps $\rT S^2 \to \rT M$.
Given such a subspace and a neighbourhood $\Hat\Vv_{f_i}$ of a local slice, we hope to obtain an obstruction bundle
\begin{equation}\label{graph}
\Hat E^i := \;{\textstyle \bigcup_{f\in\Hat\Vv_{f_i}} } \bigl\{ \nu|_{\gr f} \;\big|\; \nu \in E^i\bigr\} \;\subset\; \Hat\Ee|_{\Hat\Vv_{f_i}}
\end{equation}
by restriction to the graphs $\nu|_{\gr f} \in \Hat\Ee|_f = W^{k-1,p}(S^2, \Lambda^{0,1}f^* \rT M )$
given by
$$
\nu|_{\gr f} (z) = \nu(z,f(z)) \in \Hom^{0,1}_J(\rT_zS^2,\rT_{f(z)}M) .
$$
The disadvantage of this construction is that we need to assume injectivity of the map
$$
E^i\ni \nu\mapsto \nu|_{\gr f}\in \Hat\Ee|_f
$$
for each $f\in\Hat\Vv_{f_i}$ to obtain fibers of constant rank.
On the other hand, the inverse trivialization of the obstruction bundle
$$
(T^i)^{-1}:
\Hat\Vv_{f_i} \times E^i   \to \Hat E^i|_f  , \qquad (f,\nu) \mapsto \nu|_{\gr f}
$$
is now a smooth map, satisfying the regularity requirement in Sum Condition III, since
on the finite dimensional space $E^i$ consisting of smooth sections
the composition on the domain with $f\in\Hat\Vv_{f_i}\subset W^{k,p}(S^2,M)$ is smooth.
In fact, the pullback $\Ga^*\Hat E^1$ in \eqref{wantsmooth} now takes the special form, with $\ga_{\ul w}$ from \eqref{gaw},
$$
(g, \ul w ,\nu ) \mapsto \ga_{\ul w}^*\nu |_{\gr g}, \qquad
\ga_{\ul w}^*\nu (z,x)
 =  \nu ( \ga_{\ul w}^{-1}(z) , x ) \circ \rd_z \ga_{\ul w}^{-1} .
$$
This eliminates composition on the domain of infinite dimensional function spaces.
Indeed, we now have
$$
\ga_{\ul w}^*\nu |_{\gr g} (z)
 =  \nu ( \ga_{\ul w}^{-1}(z) , g(z) ) \circ \rd_z \ga_{\ul w}^{-1} ,
$$
whose derivatives in the directions of $g$ and $\ul{w}$ take forms that, unlike \eqref{eq:actiond}, do not involve derivatives of $g$.
Moreover, we will later make use of the special transformation of these obstruction bundles under the action of $\ga\in G_\infty$,
\begin{equation} \label{Eequivariant}
\ga_{\ul w}^*\nu |_{\gr g} \circ \rd\ga
\;=\;  \nu \bigl( \ga_{\ul w}^{-1}\circ\ga (\cdot) , g\circ \ga (\cdot) \bigr) \circ \rd \ga_{\ul w}^{-1} \circ \rd \ga
\;=\;  (\ga^{-1}\circ\ga_{\ul w})^*\nu |_{\gr g\circ \ga} .
\end{equation}
Thus we have replaced Sum Conditions I--III, including the highly nontrivial smoothness requirement in the previous section, by the following requirement for the compatibility of the geometrically constructed obstruction bundles.

\begin{itemlist}
\item[{\bf \qquad Sum Condition I$'$:}]
{\it For every $i\in \{1,\ldots,N\}$ the obstruction bundle $\Hat E^i \subset \Hat\Ee|_{\Hat\Vv_{f_i}}$ is given by \eqref{graph} from a subspace $E^i\subset\Cc^\infty(\Hom^{0,1}_J(S^2,M))$ such that}
$$
E^i \to \Hat\Ee|_f , \;\;  \nu \mapsto \nu|_{\gr f} \quad \text{is injective} \quad \forall \; f\in\Hat\Vv_{f_i} .
$$
\item[{\bf \qquad Sum Condition II$'$:}]
{\it For every $I\subset\{1,\ldots,N\}$, $[g]\in\bigcap_{i\in I}\im\psi_i$ with representative $g\in\Bb_{f_{i_0}}$ for some $i_0\in I$, and marked points $\ul w_i \in D_{i_0 i}\subset (S^2)^2$ in neighbourhoods of $(g^{-1}(Q_{f_i}^t))_{t=0,1}$ resp.\  $D_{i_0 i_0}=\{(0,1)\}$,
we must ensure linear independence of $\bigl\{ \ga_{\ul w_i}^* \nu^i |_{\gr g}  \;\big|\;  \nu^i\in E^i \bigr\}$ for $i\in I$. That is, their sum must be a direct sum}
$$
\sum_{i\in I} \bigl\{ \ga_{\ul w_i}^* \nu^i |_{\gr g}  \;\big|\;  \nu^i\in E^i \bigr\}
 \;=\; 
\bigoplus_{i\in I} \bigl\{ \ga_{\ul w_i}^* \nu^i |_{\gr g}  \;\big|\;  \nu^i\in E^i \bigr\}
 \quad\subset\;\Hat \Ee|_g.
$$

\end{itemlist}

Satisfying these two conditions always requires making the choices of the obstruction spaces $E^i$ ``suitably generic''. If they are satisfied, then they provide a construction of sum charts and coordinate changes as we will state next.
At this point, we can also incorporate a further requirement from Section~\ref{ss:top} into the compatibility condition (i) for a tuple of charts $(\bK_i)_{i=1,\ldots,N}$ by
constructing a single sum chart $\bK_{I,g_0}=\bK_I$ for each $I\subset\{1,\ldots, N\}$, whose footprint is the entire overlap of footprints $F_I:=\im\psi_I =  \bigcap_{i\in I} \im\psi_i$.
Moreover, we construct coordinate changes between any pair of tuples $I,J\subset\{1,\ldots,N\}$ with nonempty overlap $F_I\cap F_J\neq\emptyset$ that are, up to a choice of domains, directly induced from the basic charts. Thus our construction naturally satisfies the weak cocycle condition,
i.e.\ equality on overlap of domains as in Section~\ref{ss:top}.
Note here that $J\subset\{1,\ldots,N\}$ has a very different meaning from the almost complex structure which determines the Gromov--Witten moduli space $\oMm_1(A,J)$.
To avoid confusion, we will sometimes abbreviate $\overline\partial:=\pbar$.

For the construction of sum charts, we will moreover make the following simplifying assumption that all intersections with the slicing hypersurfaces are unique. This can be achieved in sufficiently small neighbourhoods of any holomorphic sphere with trivial isotropy, see Remark~\ref{rmk:unique}.

 \begin{itemlist}
\item[{\bf \qquad Sum Condition IV$'$:}]
{\it For every $i\in \{1,\ldots,N\}$ we assume that the representative $[f_i]$, slicing conditions $Q^t_{f_i}$, size $\eps>0$ of local slice $\Bb_{f_i}$, and its neighbourhood $\Hat\Vv_{f_i}\subset\Hat\Bb^{k,p}$ are chosen such that
for all $g\in \Hat\Vv_{f_i}$ and $t=0,1$ the intersection $g^{-1}(Q^t_{f_i}) =: \{w_i^t(g)\}$ is a unique point and transverse, i.e.\ $\im\rd_{w_i^t(g)}g\pitchfork \rT_{w_i^t(g)} Q^t_{f_i}$.
}

Then the same holds for $g\in G_\infty \Hat\Vv_{f_i}$. Hence for any $i_0\in I \subset \{1,\ldots,N\}$
the local slice $\Bb_{f_{i_0}}$ embeds topologically (as a homeomorphism to its image, with inverse given by  the projection $\Bb_{f_{i_0}} \times (S^2)^{2|I|} \to \Bb_{f_{i_0}}$) into a space of maps and marked points by
\begin{align}
\label{embed}
\iota_{i_0, I} : \; \Bb_{f_{i_0}} &\;\longhookrightarrow\; \Hat\Bb^{k,p} \times (S^2)^{2|I|}   \\
g &\;\longmapsto\; \bigl( g ,
\ul w(g) \bigr) ,
\qquad\qquad\quad
\ul w(g):= \bigl( g^{-1}(Q_{f_i}^t) \bigr)_{i\in I,t=0,1}.
 \nonumber
\end{align}
\end{itemlist}

\NI
Note that the elements of $\im \io_{i_0,I}$ have the form $\bigl(g,\ul w(g)=(\ul w_i)_{i\in I}\bigr)$ with $\ul w_{i_0} = (0,1)$.
In the following we denote by $\ul w = (\ul w_i)_{i\in I}\in (S^2)^{2|I|}$ any tuple of $\ul w_i = (w_i^0, w_i^1)\in S^2\times S^2$, even if it is not determined by a map $g$.
Then ``$\forall i, t$'' will be shorthand for ``$\forall i\in I, t\in \{0,1\}$''.

\begin{thm} \label{thm:A2}
Suppose that the tuple of basic Kuranishi charts
$$
\bigl(\bK_i = (U_{f_i},E_{f_i},s_{f_i},\psi_{f_i})\bigr)_{i=1,\ldots,N}
$$
is constructed as in Proposition~\ref{prop:A1} from local slices $\Bb_{f_i}$ and subspaces
$$
E^i\subset \Cc^\infty(\Hom^{0,1}_J(S^2,M)),
$$
that  induce obstruction bundles $\Hat E^i$ over neighbourhoods
$\Hat\Vv_{f_i}\subset\Hat\Bb^{k,p}$ of $\Bb_{f_i}$.
Assume moreover that this data satisfies Sum Conditions {\rm I}$'$, {\rm II}$'$, and {\rm IV}$'$.
Then for every index subset $I\subset\{1,\ldots,N\}$ with nonempty overlap of footprints
$$
F_I := \;{\textstyle \bigcap_{i\in I}}\im \psi_i  \;\neq\; \emptyset
$$
we obtain the following transition data.
\begin{enumerate}
\item
Corresponding to each choice of $i_0\in I$ and sufficiently small open set
 $$
 \Hat\Ww_{I,i_0} \subset \Bigl(\Hat\Bb^{k,p} \times (S^2)^{2|I|}\Bigr)\cap \bigl\{(g,\ul w)
 \,\big|\, \ul w_{i_0} = (0,1)\bigr\}
 $$
that covers a neighbourhood of the footprint $F_I$ in the sense that
\begin{align*}
\quad
 \bigl\{ ( g , \ul w ) \in \Hat\Ww_{I,i_0} \,\big|\, \pbar g = 0 , \;
 g(w_i^t)\in Q_{f_i}^t,\; \forall i, t \,\bigr\}
= \iota_{i_0,I}( \psi_{i_0}^{-1}(F_I) )
\end{align*}
there is a {\bf sum chart} $\bK_{I}: = \bK_{I,i_0}$ with
\begin{itemize}
\item
domain
\[
\qquad\qquad
U_{I} := \bigl\{ \bigl( g , \ul w \bigr) \in \Hat\Ww_{I,i_0}\,\big|\,
\pbar g \in {\textstyle \sum_{i\in I}} \Gamma_{\ul w_i}^* \Hat E^i,   \,
 g(w_i^t)\in Q_{f_i}^t  \ \forall i,t \,
\bigr\} ,
\]
\item
obstruction space $\displaystyle \; E_I:= {\textstyle \prod_{i\in I}} E^i$,
\item
section
$\displaystyle \;
s_{I} : U_{I} \to E_I , \;  ( g , \ul{w})   \mapsto  (\nu^i)_{i\in I}$
given by
$$
\pbar g = {\textstyle \sum_{i\in I}} \, \ga_{\ul w_i}^* \nu^i |_{\gr g} ,
$$
\item
footprint map
$\displaystyle\; \psi_{I} : s_{I}^{-1}(0) \overset{\cong}\to F_I, \;
(g,\ul w) \mapsto [g]$.
\end{itemize}\MS
\item
For every $I\subset J$ and choice of $i_0\in I, j_0\in J$ as above,
a coordinate change
$\Hat\Phi_{IJ}:
\bK_{I} \to \bK_{J}$ is given by
\begin{itemize}
\item
a choice of domain
$\displaystyle \; V_{IJ} \subset U_{I}$ such that
\begin{align}
\label{choice VIJ}
\qquad\qquad\qquad  V_{IJ}\cap s_{I}^{-1}(0) = \psi_{I}^{-1}(F_J), \qquad V_{IJ} \subset
\iota_{i_0,I} \bigl( \Ga_{f_{j_0},f_{i_0}} \bigl( \iota_{j_0,J}^{-1} ( U_{J} ) \bigr)\bigr)
\end{align}
with the embeddings \eqref{embed} and the reparametrization ${\Ga_{f_{j_0},f_{i_0}}:\Bb_{f_{j_0}}\to \Bb_{f_{i_0}}}$ as in \eqref{transition},
\item
embedding  $\phi_{IJ} := \iota_{j_0,J} \circ \Ga_{f_{i_0},f_{j_0}} \circ \iota_{i_0,I}^{-1}$, that is
\footnote{
This map also equals
$\; \phi_{IJ}\bigl( g , (w_i^t)_{i\in I,t=0,1} \bigr) \;=\;
\bigl(\, g\circ \ga \,,\,  (\ga^{-1}(w_i^t))_{i\in I,t=0,1} \cup (g^{-1}(Q^t_{f_j}))_{j\in J\less I, t=0,1} \,\bigr)$,
where $\ga=\ga_{\ul w_{j_0}}=\ga_g\in G_\infty$ is determined by
$\Ga_{f_{i_0},f_{j_0}}(g) = g\circ \ga_g$ or equivalently $\ga_{\ul w_{j_0}}(t)=w_{j_0}^t \;\forall t$.
}
$$
\qquad\qquad
\phi_{IJ} : V_{IJ} \to U_{J}, \quad
\bigl( g ,
\ul w \bigr)
\mapsto
\bigl(\, \Ga_{f_{i_0},f_{j_0}}(g) \,,\, (g^{-1}(Q^t_{f_j}))_{j\in J, t=0,1} \,\bigr) ,
$$
\item
linear embedding $\Hat\phi_{IJ}:E_I \hookrightarrow E_J$ given by the natural inclusion.
\end{itemize}
\end{enumerate}
Moreover, any choice of $i_0\in I$ and open sets $\Hat\Ww_{I,i_0}$ for each $F_I\neq\emptyset$, and domains $V_{IJ}$ for each $F_I\cap F_J\neq\emptyset$ forms a {\it weak Kuranishi atlas}
${(\bK_I, \Hat\Phi_{IJ})}$ in the sense of Definition~\ref{def:Kwk}; in particular satisfying the weak cocycle condition
$$
\phi_{JK} \circ \phi_{IJ} = \phi_{IK} \qquad\text{on}\;\; V_{IK} \cap \phi_{IJ}^{-1}(V_{JK}) .
$$
\end{thm}

\begin{proof}
The sum charts $\bK_{I}$ will be constructed as in Proposition~\ref{prop:A1}. In fact, let us begin by showing that the necessary choices of neighbourhoods in (i) always exist.
Since $F_I\subset\oMm_1(A,J)$ is open and $\psi_{i_0}$ is a homeomorphism to $s_{i_0}^{-1}(0)\subset U_{i_0}\subset\Bb_{f_{i_0}}$, there exists an open set $\Bb_{I,i_0}\subset\Bb_{f_{i_0}}$
such that $\Bb_{I,i_0}\cap s_{i_0}^{-1}(0) = \bigl\{  g \in \Bb_{I,i_0} \,\big|\, \pbar g = 0 \bigr\} = \psi_{i_0}^{-1}(F_I)$. Next, since $\iota_{I,i_0}$ is an embedding to
$\Hat\Bb^{k,p} \times \bigl\{ (\ul w_{i}) \in (S^2)^{2|I|}   \,\big|\, \ul w_{i_0} = (0,1) \bigr\}$,
it contains an open set $\Hat\Ww_{I,i_0}$ such that $\Hat\Ww_{I,i_0} \cap\im\iota_{I,i_0} = \iota_{i_0,I}( \Bb_{I,i_0})$. Together, this implies the requirement in (i).
Note moreover that elements $(g,\ul w)\in \im \iota_{i_0,I}$ satisfy $g(w_i^t)\in 
Q^t_{f_i}$ and hence $\Hat\Ww_{I,i_0}$ can be chosen such that $g(w_i^t)$ lies in a given neighbourhood of the hypersurface  
$Q^t_{f_i}$ near $f_i(t)$ for any $(g,\ul w)\in \Hat\Ww_{I,i_0}$.

Next, note that Sum Condition II$'$ is assumed to be satisfied for $g\in\psi_{i_0}^{-1}(F_I)\subset\Bb_{f_{i_0}}$, and hence continues to hold for $\bigl(g,(\ul w_i) \bigr) \in\Hat\Ww_{I,i_0}$ in a sufficiently small neighbourhood of $\iota_{i_0,I}(\psi_{i_0}^{-1}(F_I))$. Thus we obtain a well defined bundle
\begin{equation}\label{HEI}
\Hat E_I \; \to \; \Hat\Ww_{I,i_0}   , \qquad
\Hat E_I|_{(g,\ul w)}:=
{\textstyle \sum_{i\in I}} \bigl( \Gamma_{\ul w_i}^* \Hat E^i \bigr)|_{g}  \;\subset\; \Hat\Ee|_g.
\end{equation}
In order to construct a Kuranishi chart $\bK_{I}$ with footprint $F_I$ from $\Hat E_I$ along the lines of Proposition~\ref{prop:A1}, we need to express the domain $U_I$ as the zero set of a smooth transverse Fredholm operator.
Recall here from Section~\ref{ss:eval} that $\Cc^\infty(S^2,M) \times S^2 \ni (g, w_i^t) \mapsto g(w_i^t) \in M$ is not smooth in any standard Banach norm.
Hence we first
construct
the thickened solution space
\[
\Hat U_{I} := \bigl\{ ( g , \ul{w} ) \in \Hat\Ww_{I,i_0} \,\big|\, \; \pbar g \in \Hat E_I|_{(g,\ul w)}
\bigr\},
\]
which is the zero set of the smooth Fredholm operator
$$
\Hat\Ww_{I,i_0} \;\longrightarrow\;
\bigcup_{(g, \ul w)}\quotient{\Hat\Ee|_g }{ \Hat E_I |_{(g,\ul w)}}  ,
\qquad
( g , \ul{w})  \; \longmapsto\; [\pbar g] .
$$
We can achieve transversality of this operator by choosing $\Hat\Ww_{I,i_0}$ to be a sufficiently small neighbourhood of $\iota_{i_0,I}(\psi_{i_0}^{-1}(F_I))$, since $\pbar$ is transverse to $\Hat\Ee/\Hat E^{i_0}$ over $\psi_{i_0}^{-1}(F_I) \subset \Hat\Vv_{f_{i_0}}$, and for
$( g , \ul{w})\in  \iota_{i_0,I}(\psi_{i_0}^{-1}(F_I))$
we have
$\Hat E^{i_0}|_g \subset \Hat E_I |_{(g,\ul w)}$.

Finally, the domain $U_{I}\subset\Hat U_{I}$ is the zero set of the map
\begin{align} \label{BQW}
\Hat U_{I}  \quad &\;\longrightarrow\;  \underset{i\in I}{ \textstyle {\prod}}\bigl( (\rT_{f_i(0)} Q_{f_i}^{0})^\perp\times (\rT_{f_i(1)} Q_{f_i}^{1})^\perp \bigr) , \\
( g , \ul{w})
&\; \longmapsto\;
\underset{i\in I}{ \textstyle {\prod}} \bigl( \Pi^\perp_{Q_{f_i}^{0}}(g(w^0_i)), \Pi^\perp_{Q_{f_i}^{1}}(g(w^1_i))\bigr),
\nonumber
\end{align}
which is well defined for sufficiently small choice of $\Hat\Ww_{I,i_0}$, such that the $g(w^t_i)$ lie in the domain of definition of the projections $\Pi^\perp_{Q_{f_i}^{t}}$.
Moreover, this map is smooth, since by the regularity in Sum Condition III (which is satisfied by construction) we have $\Hat U_I\subset\Cc^\infty(S^2,M)$.
To see that it is transverse, it suffices to consider any given point $(g ,\ul{w}) \in \iota_{i_0,I}(\psi_{i_0}^{-1}(F_I))$, since transversality at these points persists in an open neighbourhood, and then $\Hat\Ww_{I,i_0}$ can be chosen sufficiently small to achieve transversality on all of $\Hat U_I$.
At these points we understand some parts of the tangent space $\rT_{(g ,\ul{w})}\Hat U_I$ because $\{ (f, \ul v ) \in \Hat\Ww_{I,i_0} \,|\,  \pbar f=0 \}$ is a subset of $\Hat U_I$ which contains $\iota_{i_0,I}(\psi_{i_0}^{-1}(F_I))$.
Hence we have $\bigl(\delta g, (\delta w_i)_{i\in I}\bigr) \in \rT_{(g, \ul w)} \Hat U_I$ for any $\delta g \in \ker{\rm D}_g\pbar$ and $\delta w_i \in \rT_{w_i}(S^2)^2$ with $\delta w_{i_0}=0$.
In particular, we have $\rT_{g}(G_\infty g) \times \{0\} \subset \rT_{(g, \ul w)} \Hat U_I$ since $\{ (f, \ul v ) \in \Hat\Ww_{I,i_0} \,|\,  \pbar f=0 \}$ is invariant under the action $\ga : (f,\ul v)\mapsto (f\circ \ga , \ul v)$ of $\{\ga\approx{\rm id}\}\subset G_\infty$, unlike the thickened solution space $\Hat U_I$ itself.
(Neither space is invariant under the more natural action $(f,\ul v)\mapsto (f\circ \ga , \ga^{-1}(\ul v))$ that will be important below, since at the moment $\ul w_{i_0}$ is fixed.)

Now the $i_0$ component of the linearized operator
of \eqref{BQW} at any point 
simplifies, since the marked points $w^t_{i_0}=t$ are fixed, to
\begin{equation}\label{linop0}
\rT_{(g, \ul w)} \Hat U_I \;\ni\; \bigl(\delta g, (\delta w_i)_{i\in I}\bigr)
 \;\mapsto\;  \bigl( \rd\Pi^\perp_{Q^{0}_{f_{i_0}}} \delta g (0) , \rd\Pi^\perp_{Q^{1}_{f_{i_0}}} \delta g (1) \bigr) .
\end{equation}
At points with $\pbar g=0$, its restriction to $\rT_{g}(G_\infty g) \times \{0\} \subset \rT_{(g, \ul w)} \Hat U_I$ is surjective by the same argument as in Proposition~\ref{prop:A1}, which uses the fact that $\im\rd_t g$ projects onto $(\rT_{f_{i_0}(t)} Q_{f_{i_0}}^{t})^\perp$ by the construction of the local slice $\Bb_{f_{i_0}}$ at $g\approx f_{i_0}$.
Next, 
the  
$j\in I\less\{i_0\}$ component of the linearized operator for fixed $\delta g$ is
\begin{equation}\label{linopi}
\rT_{(g, \ul w)} \Hat U_I \;\ni\; 
\bigl(\delta g, (\delta w_i)_{i\in I}\bigr)
\;\mapsto\;
\Bigl(  \rd\Pi^\perp_{Q^{t}_{f_j}}   \bigl(\delta g (w^t_j )  + \rd_{w^t_j} g ( \delta w^t_j)  \bigr)\Bigr)_{t=0,1}  .
\end{equation}
We claim that this is surjective for any given $\delta g \subset \rT_{g}(G_\infty g)$ (given by the surjectivity requirements for $i_0$), just by variation of $\delta w_j$.
Indeed, for $(g, \ul w)\in \iota_{i_0,I}(\psi_{i_0}^{-1}(F_I))$ we 
have $\bigl(\delta g, (\delta w_i)_{i\in I}\bigr) \in \rT_{(g, \ul w)} \Hat U_I$ for any $\delta g \subset \rT_{g}(G_\infty g)$ and $\delta w_i \in \rT_{w_i}(S^2)^2$.
Moreover, we have $\im\rd_{w^t_j} g = 
\im 
\rd_{t} (g\circ\ga_{\ul w_j})$, which projects onto $(\rT_{f_j(t)} Q_{f_j}^{t})^\perp$ by the construction of the local slice $\Bb_{f_j}$ at $g\circ\ga_{\ul w_j}\approx f_j$.
This proves surjectivity of \eqref{linopi} for $j\neq i_0$  
by variation of $\delta w_j$, 
and together with the surjectivity of\eqref{linop0} 
by variation of $\delta g$
proves transversality of \eqref{BQW} for sufficiently small $\Hat\Ww_{I, i_0}$.

Now that the domain $U_I$ is equipped with a smooth structure, we can construct a Kuranishi atlas $\bK_I$ as in Proposition~\ref{prop:A1} by pulling back the smooth section
$$
\ti s_I : U_I \;\to\; \Hat E_I|_{U_I} , \qquad  (g, \ul w) \;\mapsto\; \pbar g
$$
to the trivialization $\Hat E_I|_{U_I}\cong U_I \times E_I$ given by construction of the sum bundle.
The induced homeomorphism
$$
\psi_I :  \; \ti s_I^{-1}(0) \; \overset{\cong}{\longrightarrow} \;  F_I \;\subset\; \oMm_1(A,J) ,
\qquad  (g, \ul w)\;\mapsto\; [g]
$$
maps
$\ti s_I^{-1}(0) \subset \im\iota_{i_0,I}$
to the desired footprint since we chose the neighbourhoods $\Hat\Ww_{I, i_0}$ and $\Bb_{I,i_0}:=\iota_{i_0,I}^{-1}(\Hat\Ww_{I, i_0})\subset \Bb_{f_{i_0}}$ such that
$$
\psi_I( \ti s_I^{-1}(0) )
\;=\; \pr \bigl( \iota_{i_0,I}^{-1}(\ti s_I^{-1}(0)) \bigr)
\;=\;  \pr \bigl( \iota_{i_0,I}^{-1}(\Hat\Ww_{I,i_0}) \cap \pbar^{-1}(0) \bigr)
\;=\;  \pr \bigl( \psi_{i_0}^{-1}(F_I)  \bigr)
\;=\; F_I ,
$$
where $\pr :\Hat\Bb^{k,p}\to \Hat\Bb^{k,p}/G_\infty$ denotes the quotient.
This finishes the construction for (i).

To construct the coordinate changes, we can now forget the marked points, which were only a technical means to obtaining smooth sum charts.
For that purpose fix a pair $i_0\in I$ and note that the forgetful map $\Pi_I:\Hat\Bb^{k,p}\times (S^2)^{2|I|}\to\Hat\Bb^{k,p}$ is a left inverse to the embedding $\io_{i_0,I}:\Bb_{f_{i_0}}\hookrightarrow \Hat\Bb^{k,p}\times (S^2)^{2|I|}$ from \eqref{embed}, whose image contains the smooth finite dimensional domain $U_I\subset\Cc^\infty(S^2,M)\times (S^2)^{2|I|}$.
Hence it restricts to a topological embedding to a space of perturbed holomorphic maps in the slice,
\begin{align}\label{UB}
\Pi_I|_{U_I} : \;   U_I   \;\longrightarrow\;
B_{I,i_0} :=&\; \bigl\{ g\in \Bb_{f_{i_0}} \,\big|\, \exists \ul w \in (S^2)^{2|I|} : (g,\ul w) \in U_I \bigr\} \\
=&\; \bigl\{ g\in \Bb_{I,i_0} \,\big|\, \pbar g \in \Hat E_I |_{(g,\ul w(g))} \bigr\}   .  \nonumber
\end{align}
In fact, this is a smooth embedding since the forgetful map is smooth and we can check that the differential of the forgetful map $\Pi_I|_{U_I}$ is injective. Indeed, its kernel at $(g,\ul w)$  is the vertical part of the tangent space
$\rT_{(g,\ul w)} U_{I} \cap \bigl( \{0\} \times \rT_{\ul w} (S^2)^{2|I|}\bigr)$, which in terms of the linearized operators \eqref{linopi} is given by the kernel of
$$
\rT_{\ul w} (S^2)^{2|I|} \; \ni \;
(\delta w^t_i )_{i\in I, t=0,1} \;\longmapsto\;
\Bigl(  \rd_{g(w^t_i)}\Pi^\perp_{Q^{t}_{f_i}}   \bigl( \rd_{w^t_i} g ( \delta w^t_i)  \bigr)\Bigr)_{i\in I, t=0,1}
\;\in \; \underset{i\in I, t=0,1}{\textstyle\prod} \im \rd_t f_i .
$$
This operator is injective (and hence surjective) since by Sum Condition IV$'$
$$
\bigl(\im\rd_{w^t_i} g\bigr)\; \pitchfork\; \bigl(\rT_{w^t_i}Q^{t}_{f_i}\bigr)=\ker\rd_{g(w^t_i)}\Pi^\perp_{Q^{t}_{f_i}}.
$$
Thus $\io_{i_0,I}: B_{I,i_0}\to U_I$ is a diffeomorphism, and since it also intertwines the Cauchy--Riemann operator on the domains and the projection to $\oMm_1(A,J)$, this forms a map
$$
\Hat\Pi_{I,i_0}:=\bigl(\Pi_I|_{U_I} , \id_{E_I} \bigr): \; \bK_I \longrightarrow \bK^B_I,
$$
from the sum chart
$\bK_I=\bigl(\, U_I \,,\, \bigcup_{(g,\ul w)\in U_I} \Hat E_I|_{(g,\ul w)} \,,\, \ti s_I(g,\ul w)=\pbar g \,,\, \psi_I(g,\ul w)=[g] \,\bigr)$ to the Kuranishi chart
$$
\bK^B_I: = \bigl(\, B_{I,i_0} \,,\,
{\textstyle \bigcup_{g\in B_{I,i_0}}} \Hat E_I|_{(g,\ul w(g))} \,,\, \ti s(g)=\pbar g \,,\, \psi(g)=[g] \,\bigr).
$$
(Here we indicated the obstruction bundles before trivialization to $E_I$.)
The inverse map $\Hat\Pi_{I,i_0}^{-1}:=\bigl(\io_{i_0,I} , \id_{E_I} \bigr)$ is also a map between Kuranishi charts, and both are coordinate changes since the index condition is automatically satisfied when $\phi_{IJ}$ and $\Hat\phi_{IJ}$ are both diffeomorphisms.
Indeed, in this case, both target and domain in the tangent bundle condition \eqref{tbc} are trivial.

Next, we will obtain further coordinate changes $\Hat\Phi^I_{i_0 j_0} : (B_{I,i_0},\ldots) \to (B_{I,j_0},\ldots)$ for different choices of index $i_0,j_0\in I$. Here the choices of neighbourhoods $\Hat\Ww_{I,\bullet}$ induce neighbourhoods in the local slices $\Bb_{I,\bullet}\subset\Bb_{f_\bullet}$ such that
$B_{I,\bullet} :=\; \bigl\{ g\in \Bb_{f_\bullet} \,\big|\, \pbar g \in \Hat E_I |_{(g,\ul w(g))} \bigr\}$.
These domains are intertwined by the transition map between local slices $\Ga_{f_{i_0},f_{j_0}}$. Indeed, using the $G_\infty$-equivariance of the obstruction bundles \eqref{Eequivariant}, we have
$$
\pbar g = {\textstyle \sum_{i\in I}} \, \ga_{\ul w_i(g)}^*\nu^i |_{\gr g}
\quad\Longrightarrow\quad
\pbar(g\circ\ga) = {\textstyle \sum_{i\in I}}\,  \ga_{\ul w_i (g\circ \ga)}^*\nu^i |_{\gr g\circ\ga} ,
$$
where $\ga^{-1}\circ\ga_{\ul w_i} = \ga_{\ul w_i(g\circ\ga)}$ since $\ga^{-1}(\ga_{\ul w_i}(t))= \ga^{-1}( w_i^t)= w_i^t(g\circ\ga)$.
Thus we obtain a well defined map
$\Ga_{f_{i_0},f_{j_0}} : B_{I,i_0} \cap G_\infty\Bb_{I,j_0} \to  B_{I,j_0}$.
It is a topological embedding with open image, since its inverse is $\Ga_{f_{j_0},f_{i_0}} |_{B_{I,j_0} \cap G_\infty\Bb_{I,i_0}}$. In fact, it is a local diffeomorphism since both maps are smooth by Lemma~\ref{le:Gsmooth}.
The above also shows that this diffeomorphism intertwines the sections, given by the Cauchy--Riemann operator, and the footprint maps, given by the projection $g\mapsto [g]\in\oMm_1(A,J)$.
Since the index condition is automatic as above, we obtain the required coordinate change by
$\Hat\Phi^I_{i_0 j_0}:=\bigl(\, \Ga_{f_{i_0},f_{j_0}} \,,\, \id_{E_I} \,\bigr)$
with domain $B_{I,i_0} \cap G_\infty\Bb_{I,j_0} \subset B_{I,i_0}$.

With these preparations, a natural coordinate change  for $I\subsetneq J$ and any choice of $i_0\in I$, $j_0\in J$ arises from the composition of the above coordinate changes (all of which are local diffeomorphisms on the domains) with another natural coordinate change
$\Hat\Phi^{i_0}_{IJ} : (B_{I,j_0},\ldots) \to (B_{J,j_0},\ldots)$ given by the inclusion
$$
\phi^{j_0}_{IJ} := {\rm id}_{\Bb_{j_0}} : \; B_{I,j_0} \cap \Bb_{J,j_0} \;\hookrightarrow\; B_{J,j_0} .
$$
Again, this naturally intertwines the sections and footprint maps with
$$
s_I^{-1}(0)\cap B_{I,j_0} \cap \Bb_{J,j_0}= \psi_I^{-1}(F_J).
$$
To check the index condition for this embedding together with the linear embedding $\Hat\phi^{j_0}_{IJ} := {\rm id}_{E_I} : E_I \hookrightarrow E_J$ we express the tangent spaces to both domains in terms of the linearization of the Cauchy--Riemann operator on the local slice $\overline\partial: \Bb_{j_0} \to \Hat\Ee|_{\Bb_{j_0}}$. Comparing
$$
 \rT_g B_{I,j_0} =  ({\rm D}_g \overline{\partial}) ^{-1} \Bigl( \textstyle{\sum_{i\in I}} (\Ga_{\ul w_i(g)}^*\Hat E^i)|_g \Bigr) , \qquad
\rT_g B_{J,j_0} = ({\rm D}_g \overline\partial) ^{-1} \Bigl(  {\textstyle \sum_{j\in J}} (\Ga_{\ul w_j(g)}^*\Hat E^j)|_g \Bigr)
$$
as subsets of $\rT_g\Bb_{f_{j_0}}$, we can identify
$$
\quotient{\rT_g B_{J,j_0}}{\rd_g\phi^{j_0}_{IJ} \bigl( \rT_g B_{I,j_0} \bigr)}
\;=\;
 \quotient{ ({\rm D}_g \overline\partial) ^{-1}  \left( \textstyle{\sum_{j\in J\less I}} (\Ga_{\ul w_j(g)}^*\Hat E^j)|_g\right) }{\ker {\rm D}_g\overline\partial
}
$$
to see that the linearized section (given by the linearized Cauchy Riemann operator together with the trivialization of obstruction bundles) satisfies the tangent bundle condition \eqref{tbc}
\begin{eqnarray*}
{\rm D}_{g} \overline\partial \;:\;
\quotient{\rT_g B_{J,j_0}}{\rd_g\phi^{j_0}_{IJ} \bigl( \rT_g B_{I,j_0} \bigr)}
\;&\stackrel{\cong} \longrightarrow\; &
 {\textstyle\sum_{j\in J\less I}} (\Ga_{\ul w_j(g)}^*\Hat E^j)|_g \\
&&\qquad\quad \;\cong\;
\quotient{\Hat E_J|_{(g,(\ul w_j(g))_{j\in J})}}
{\Hat E_I|_{(g,(\ul w_i(g))_{i\in I})}}.
\end{eqnarray*}
Finally, we can compose the coordinate changes to
$$
\Hat\Phi_{IJ} \,:= \; \Hat\Pi_{J,j_0}^{-1} \circ \Hat\Phi^J_{i_0 j_0} \circ \Hat\Phi^{i_0}_{IJ} \circ \Hat\Pi_{I,i_0} \; : \;\;
\bK_I \; \to \; \bK_J .
$$
By Lemma~\ref{le:cccomp} this defines a coordinate change with the maximal domain
$$
\iota_{i_0,I}(B_{I,i_0}  \cap G_\infty \Bb_{J,j_0}) \;=\; \iota_{i_0,I} \bigl( \Ga_{f_{j_0},f_{i_0}} \bigl( \iota_{j_0,J}^{-1} ( U_{J} ) \bigr)\bigr),
$$
which we can restrict to any smaller choice of $V_{IJ}$ containing $\psi_I^{-1}(F_J)$.
The linear embedding, after the fixed trivialization of the bundle, is the trivial embedding $\Hat\phi_{IJ}: E_I\hookrightarrow E_J$, whereas the nonlinear embedding $\phi_{IJ}:=\iota_{i_0,I}^{-1}\circ \Ga_{f_{i_0},f_{j_0}}\circ \iota_{j_0,J}: V_{IJ} \to U_J$ of domains is given by the restriction to $V_{IJ}$ of the composition
$$
U_I \;\overset{\iota_{i_0,I}}{\longhookleftarrow}\; B_{I,i_0}  \cap G_\infty \Bb_{J,j_0} \; \xrightarrow[\cong]{\Ga_{f_{i_0},f_{j_0}}}\; B_{I,j_0} \cap \Bb_{J,j_0}
 \;\overset{\id_{\Bb_{j_0}}}{\longhookrightarrow}\; B_{J,j_0}
 \;\overset{\iota_{j_0,J}}{\longhookrightarrow}\; U_J .
$$
This completes the proof of (ii).
Finally, the cocycle condition on the level of the linear embeddings $\Hat\phi_{IJ}$ holds trivially,
whereas the weak cocycle condition for the embeddings between the domains follows, since $\iota_{j_0,J}^{-1}\circ\iota_{j_0,J}=\id_{\Bb_{J,j_0}}$ from the cocycle property of the local slices,
$$
\Ga_{f_{j_0},f_{k_0}}\circ \Ga_{f_{i_0},f_{j_0}}  = \Ga_{f_{i_0},f_{k_0}}
\qquad\text{on}\; \Bb_{f_{i_0}}\cap \bigl(G_\infty\cdot\Bb_{f_{j_0}}\bigr) \cap \bigl(G_\infty \cdot \Bb_{f_{k_0}}\bigr) .
$$
This completes the proof of Theorem~\ref{thm:A2}.
\end{proof}

Note that we crucially use the triviality of the isotropy groups, in particular in the proof of the cocycle condition. Nontrivial isotropy groups cause additional indeterminacy, which has to be dealt with in the abstract notion of Kuranishi atlases. The construction of Kuranishi atlases with nontrivial isotropy groups for Gromov--Witten moduli spaces will in fact require a sum construction already for the basic Kuranishi charts.
We will give a more detailed proof of Theorem~\ref{thm:A2} in \cite{MW:gw}, where we will also treat nodal curves and deal with the case of isotropy
or, more generally, nonunique intersections with the hypersurfaces $Q^t_{f_i}$.
Further, we will show that Sum Conditions I$'$ and II$'$ can always be satisfied by perturbing and shrinking a given set of basic charts.
Finally, in the language of Section~\ref{ss:top}, note that the Kuranishi charts and transition data
that we construct only satisfy the weak cocycle condition. However, our obstruction bundles are naturally additive. Therefore we obtain an additive weak Kuranishi atlas.
As we show in Theorem~\ref{thm:K} and Section~\ref{s:VMC}, this is precisely what we need to define the virtual fundamental class $[\oMm_1(A,J)]^{\rm virt}$.

\begin{remark} \rm  \label{rmk:smart}
 (i) The actual idea behind the choice of local slice conditions and the introduction of further marked points is of course a stabilization of the domain in order to obtain a theory over the Deligne--Mumford moduli space of stable genus zero Riemann surfaces with marked points. So one might want to rewrite this approach invariantly and, when summing two charts for example,
work over the Deligne--Mumford moduli space with five marked points instead of taking the points $(\infty,0,1,w_i^0,w_i^1)$.
When properly handled, this approach does give a good framework for discussing coordinate changes. However one does need to take care not to obscure the analytic problems by introducing these further abstractions and notations.
Moreover, this abstraction does not yield another approach to constructing the coordinate changes.
If there is a rigorous approach using the Deligne--Mumford formalism, then in a local model near $(\infty,0,1, w^0_{01},w^1_{01})$,  it would take exactly the form discussed above.\MS

\NI (ii)
The abstraction to equivalence classes of maps and marked points modulo automorphisms becomes crucial when one wants to extend the above approach to construct finite dimensional reductions near nodal curves, because the Gromov compactification  exactly mirrors the construction of Deligne--Mumford space. While we will defer the details of this construction to \cite{MW:gw}, let us note that the genus zero Deligne--Mumford spaces are defined by equivalence classes of pairwise distinct marked points on the sphere. Hence we will
need to make sure that the marked points $w^0,w^1\in S^2$ that we read off from intersection with the hypersurfaces $Q_{f_1}^{0},Q_{f_1}^{1}$ are disjoint from each other and from $\infty, 0,1$.
Thus, in order to be summable near nodal curves, the basic Kuranishi charts must be constructed from local slices with pairwise disjoint slicing conditions $Q_{f_i}^t$.
\MS

\NI (iii)
In view of Sum Condition II$'$ and the previous remark, one cannot expect any two given basic Kuranishi charts to have summable obstruction bundles and hence be compatible. This requires a perturbation of the basic Kuranishi charts, which is possible only when dealing with a compactified solution space, since each perturbation may shrink the image of a chart.
\MS

\NI (iv)
This discussion also shows that even a simple moduli space such as $\oMm_{1}(A,J)$ does not have a canonical Kuranishi atlas. Hence the construction of invariants from this space also involves constructing a Kuranishi atlas on the product cobordism $\oMm_{1}(A,J)\times [0,1]$ 
intertwining any two Kuranishi atlases for $\oMm_{1}(A,J)$ arising from different choices of basic charts and transition data. Note here that one could construct basic charts of the ``wrong dimension'' by simply adding trivial finite dimensional factors to the abstract domains or obstruction spaces.
A natural and necessary condition for constructing a well defined cobordism class of Kuranishi atlases with the ``expected dimension'' for $\oMm_{1}(A,J)$ is the following {\bf Fredholm
index condition for charts} pointed out to us by Dietmar Salamon:\MS

{\it Each Kuranishi chart must in some sense identify the kernel $\ker\rd s_{f_i}$ and cokernel $(\im\rd s_{f_i})^\perp$ of the finite dimensional reduction with the kernel modulo the infinitesimal action $\ker\rd_{f_i}\pbar/{\scriptstyle \rT_{f_i}(G_\infty f_i)}$ and cokernel $(\im\rd_{f_i}\pbar)^\perp$ of the Cauchy--Riemann operator.}\MS

In fact, one might argue that this identification should be part of a Kuranishi atlas on a moduli space. However, this would require giving the abstract footprint of a Kuranishi atlas more structure than that of a compact metrizable topological space,
in order to keep track of the kernel and cokernel of the Fredholm operator that arises by linearization from the PDE that defines the moduli space.
Whether the index condition picks out a unique cobordism class of Kuranishi atlases on $X$ is an interesting open question.
\MS

\NI (v)
The Fredholm index condition for Kuranishi charts, once rigorously formulated, should imply that any map between charts
which satisfy the index condition
should also satisfy the index condition for coordinate changes in Definition~\ref{def:change}
(a reformulation of the tangent bundle condition introduced by Joyce).
Conversely, a map between charts that satisfies the index condition for coordinate changes should also preserve the Fredholm index condition for charts.
More precisely, if $\Hat\Phi_{IJ} :\bK_I \to \bK_J$ is a map satisfying the index condition, and one of the charts $\bK_I$ or $\bK_J$ satisfies the Fredholm index condition, then both charts satisfy the Fredholm index condition (pending a rigorous definition of the latter).
\end{remark}

\section{Kuranishi charts and coordinate changes
with trivial isotropy
}
\label{s:chart}

Throughout this chapter, $X$ is assumed to be a compact and metrizable space.
This section defines Kuranishi charts with trivial isotropy for $X$ and coordinate changes between them.
The case of nontrivial isotropy is a fairly straightforward generalization using the language of groupoids, but the purpose of this paper is to clarify fundamental topological issues in the simplest example.
Hence we assume throughout that the charts have trivial isotropy and drop this qualifier from the wording.
Our definitions are motivated by \cite{FO};  an additional reference is \cite{J}.
Differing from both approaches, we work exclusively with charts whose domain and section are fixed, rather than with germs of charts as discussed in Section~\ref{ss:alg}. We moreover restrict our attention to charts 
without boundaries or corners. \cite[App.~A]{FOOO} also dispenses with germs and uses essentially the same basic definitions.  
However, our insistence on specifying the domain is new, as are our notion of Kuranishi atlas and
interpretation in terms of categories in Section~\ref{s:Ks} and our construction of the virtual fundamental class in Section~\ref{s:VMC}.

\subsection{Charts, maps, and restrictions}\label{ss:chart}

\begin{defn}\label{def:chart}
Let $F\subset X$ be a nonempty open subset.
A {\bf Kuranishi chart} for $X$ with {\bf footprint} $F$ is a tuple $\bK = (U,E,s,\psi)$ consisting of
\begin{itemize}
\item
the {\bf domain} $U$, which is an open smooth
$k$-dimensional manifold;
\item
the {\bf obstruction space} $E$, which is a finite dimensional real vector space;
\item
the {\bf section} $U\to U\times E, x\mapsto (x,s(x))$ which is given by a
smooth map $s: U\to E$;
 \item
the {\bf footprint map} $\psi : s^{-1}(0) \to X$, which is a homeomorphism to the footprint ${\psi(s^{-1}(0))=F}$.
\end{itemize}
The {\bf dimension} of $\bK$ is $\dim \bK: = \dim U-\dim E$.
\end{defn}

More generally, one could work with an obstruction bundle over the domain, but this complicates the notation and, by not fixing the trivialization, makes coordinate changes less unique.
In the application to holomorphic curve moduli spaces, there are natural choices of trivialized obstruction bundles.
The section $s$ is then given by the generalized Cauchy--Riemann operator, and elements in the footprint are $J$-holomorphic maps modulo reparametrization.

\MS\NI
Since we aim to define a regularization of $X$, the most important datum of a Kuranishi chart is its footprint.
So, as long as the footprint is unchanged, we can vary the domain $U$ and section $s$ without changing the chart in any important way. Nevertheless, we will always work with charts that have a fixed domain and section. In fact, our definition of a map between Kuranishi charts crucially involves
these domains.
The following definition is very general; the actual coordinate changes involved in a Kuranishi atlas will be a combination of a restriction, as defined below, and a map that satisfies an extra
index condition.

\begin{defn}\label{def:map} A {\bf map}  $\Hat\Phi : \bK \to \bK'$  between Kuranishi charts is a pair $(\phi,\Hat\phi)$ 
consisting of an embedding $\phi :U \to U'$ and a linear injection $\Hat \phi :E \to E'$ such that
\begin{enumerate}
\item the embedding restricts to $\phi|_{s^{-1}(0)}=\psi'^{-1} \circ\psi : s^{-1}(0) \to s'^{-1}(0)$, the transition map induced from the footprints in $X$;
\item
the embedding intertwines the sections, $s' \circ \phi  = \Hat\phi \circ s$, on the entire domain $U$.
\end{enumerate}
That is, the following diagrams commute:
\begin{equation}
 \begin{array} {ccc}
{U\times E}& \stackrel{\phi\times \Hat\phi} \longrightarrow &
{U'\times E'} 
\phantom{\int_Quark}  \\
\phantom{sp} \uparrow {s}&&\uparrow {s'} \phantom{spac}\\
\phantom{s}{U} & \stackrel{\phi} \longrightarrow &{U'} \phantom{spacei}
\end{array}
\qquad
 \begin{array} {ccc}
{s^{-1}(0)} & \stackrel{\phi} \longrightarrow &{s'^{-1}(0)} \phantom{\int_Quark} \\
\phantom{spa} \downarrow{\psi}&&\downarrow{\psi'} \phantom{space} \\
\phantom{s}{X} & \stackrel{{\rm Id}} \longrightarrow &{X}. \phantom{spaceiiii}
\end{array}
\end{equation}
\end{defn}

The dimension of the obstruction space $E$ typically varies as the footprint $F\subset X$ changes.
Indeed, the maps $\phi, \Hat\phi$ need not be surjective.  However, as we will see in Definition \ref{def:change}, the maps allowed as coordinate changes are carefully controlled in the normal direction.
Since we only defined maps of Kuranishi charts that induce an inclusion of footprints,
we now need to define a notion of restriction of a Kuranishi chart to a smaller subset of its footprint.

\begin{defn} \label{def:restr}
Let $\bK$ be a Kuranishi chart and $F'\subset F$
an open subset of the footprint.
A {\bf restriction of $\bK$ to $\mathbf{\emph F\,'}$} is a Kuranishi chart of the form
$$
\bK' = \bK|_{U'} := \bigl(\, U' \,,\, E'=E \,,\, s'=s|_{U'} \,,\, \psi'=\psi|_{s'^{-1}(0)}\, \bigr)
$$
given by a choice of open subset $U'\subset U$ of the domain such that $U'\cap s^{-1}(0)=\psi^{-1}(F')$.
In particular, $\bK'$ has footprint $\psi'(s'^{-1}(0))=F'$.
%
%
\end{defn}

The following lemma shows that we may easily restrict to any open subset of the footprint.
Moreover it provides a tool for restricting to precompact domains, which we require for refinements of Kuranishi atlases in Sections~\ref{ss:shrink} and \ref{ss:red}.
Here and throughout we will use the notation $V'\sqsubset V$ to mean that the inclusion $V'\hookrightarrow V$ is {\it precompact}. That is, ${\rm cl}_V(V')$ is compact, where ${\rm cl}_V(V')$ denotes the closure of $V'$ in the relative topology of $V$.  If both $V'$ and $V$ are contained in a compact space $X$, then $V'\sqsubset V$ is equivalent to the inclusion $\ov{V'}:={\rm cl}_X(V')\subset V$ of the closure of $V'$ with respect to the ambient topology.

\begin{lemma}\label{le:restr0}
Let $\bK$ be a Kuranishi chart. Then for any open subset $F'\subset F$ there exists a restriction $\bK'$ to $F'$ whose domain $U'$ is such that $\ov{U'}\cap s^{-1}(0) = \psi^{-1}(\ov{F'})$. If moreover $F'\sqsubset F$ is precompact, then $U'$ can be chosen to be precompact.
\end{lemma}

\begin{proof}
Since $F'\subset F$ is open and $\psi:s^{-1}(0)\to F$ is a homeomorphism in the relative topology of $s^{-1}(0)\subset U$, there exists an open set $V\subset U$ such that $\psi^{-1}(F')=U'\cap s^{-1}(0)$.
If $F'\sqsubset F$ then we claim that $U'\subset V$ can be chosen  so that  in addition its closure 
intersects $s^{-1}(0)$ in $\psi^{-1}(\ov{F'})$.
To arrange this we define
$$
U' := \bigl\{x\in V \,\big|\, d(x, \psi^{-1}(F')) < d(x, \psi^{-1}(F\less F'))\bigr\},
$$
where
$d(x,A): = \inf_{a\in A} d(x,a)$
denotes the distance between the point $x$ and the subset $A\subset U$ with respect to any metric $d$ on the finite dimensional manifold $U$. Then $U'$ is an open subset of $V$.
By construction, its intersection with $s^{-1}(0)=\psi^{-1}(F)$ is 
$V\cap \psi^{-1}(F')=\psi^{-1}(F')$.
To see that
$\ov{U'} \cap s^{-1}(0) \subset \psi^{-1}(\ov{F'})$, consider a sequence 
$x_n\in U'$ that converges to $x_\infty\in \psi^{-1}(F)$.
If $x_\infty\in \psi^{-1}(F\less F')$ then by definition of $U'$ there are points $y_n\in \psi^{-1}(F')$ such that $d(x_n,y_n)< d(x_n,x_\infty)$.
This implies $d(x_n,y_n)\to 0$, hence we also get convergence $y_n\to x_\infty$, which proves $x_\infty\in \ov{\psi^{-1}(F')} = \psi^{-1}(\ov{F'})$, where the last equality is by the homeomorphism property of $\psi$.
Finally, the same homeomorphism property implies the inclusion $\psi^{-1}(\ov{F'}) = \ov{\psi^{-1}(F')} \subset \ov{U'} \cap s^{-1}(0)$, and thus equality.
This proves the first statement.
The second statement will hold if we show that when $F$ is precompact, we may choose $V$, and hence $U'\subset V$ to be precompact in $U$.
For that purpose we use the homeomorphism property of $\psi$ and relative closedness of $s^{-1}(0)\subset U$ to deduce that $\psi^{-1}(F')\subset U$ is a precompact set in a finite dimensional manifold, hence has a precompact open neighbourhood $V\sqsubset U$.
To see this, note that  each point in $\ov{\psi^{-1}(F')}$ is the center of a precompact ball, by the Heine-Borel theorem; and a continuity and compactness argument provides a finite covering of $\ov{\psi^{-1}(F')}$ by precompact balls.
\end{proof}

\subsection{Coordinate changes}
\label{ss:coord}   \hspace{1mm}\\ \vspace{-3mm}

The following notion of coordinate change is key to the definition of Kuranishi atlases.
It involves a special kind of map from an intermediate chart that is obtained by restriction.
Here we begin using notation that will also appear in our definition of Kuranishi atlases. For now, $\bK_I=(U_I,E_I,s_I,\psi_I)$ and $\bK_J=(U_J,E_J,s_J,\psi_J)$ just denote different Kuranishi charts for the same space $X$.

\begin{figure}[htbp] 
   \centering
   \includegraphics[width=5in]{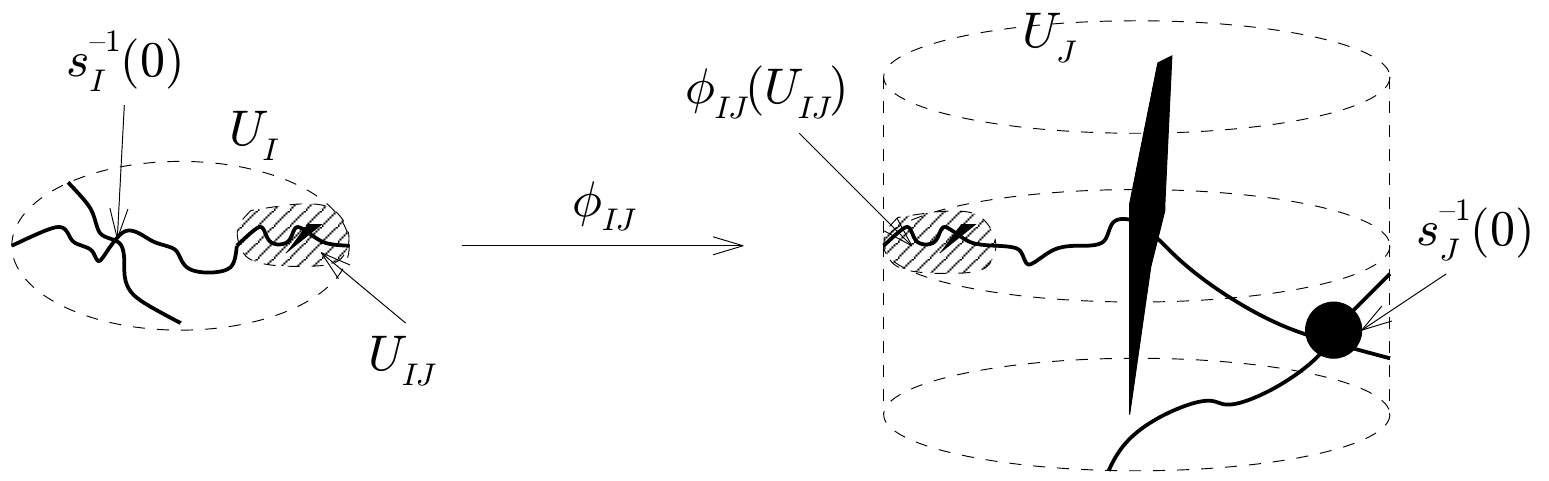}
   \caption{A coordinate change in which $\dim U_{J} =\dim U_I+ 1$.
   Both $U_{IJ}$ and its image $\phi(U_{IJ})$ are shaded.}
   \label{fig:2}
\end{figure}

\begin{defn}\label{def:change}
Let $\bK_I$ and $\bK_J$ be Kuranishi charts such that $F_I\cap F_J$ is nonempty.
A {\bf coordinate change} from $\bK_I$ to $\bK_J$ is a map
$\Hat\Phi: \bK_I|_{U_{IJ}}\to \bK_J$, which satisfies the {\bf index condition} in (i),(ii) below, and whose domain is a restriction of $\bK_I$ to $F_I\cap F_J$.
That is, the {\bf domain} of the coordinate change is an open subset $U_{IJ}\subset U_I$ such that
$\psi_I(s_I^{-1}(0)\cap U_{IJ}) = F_I\cap F_J$.
\begin{enumerate}
\item
The embedding $\phi:U_{IJ}\to U_J$ underlying the map $\Hat\Phi$ identifies the kernels,
$$
\rd_u\phi \bigl(\ker\rd_u s_I \bigr) =  \ker\rd_{\phi(u)} s_J    \qquad \forall u\in U_{IJ};
$$
\item
the linear embedding $\Hat\phi:E_I\to E_J$ given by the map $\Hat\Phi$ identifies the cokernels,
$$
\forall u\in U_{IJ} : \qquad
E_I = \im\rd_u s_I \oplus C_{u,I}  \quad \Longrightarrow \quad E_J = \im \rd_{\phi(u)} s_J \oplus \Hat\phi(C_{u,I}).
$$
\end{enumerate}
\end{defn}

Note that coordinate changes are in general {\it unidirectional} since the maps $U_{IJ}\to U_J$ and $E_{IJ}=E_I\to E_J$  are not assumed to have open images.
Note also that the footprint of the intermediate chart $\bK_I|_{U_{IJ}}$ is
always the full intersection $F_I\cap F_J$.
By abuse of notation, we often denote a
coordinate change by $\Hat\Phi: \bK_I\to \bK_J$, without specific mention of its domain $U_{IJ}$, even though it is really just a map from the restriction $\bK_I|_{U_{IJ}}$ of the chart $\bK_I$ to $F_I\cap F_J$.
This should cause no confusion with Definition~\ref{def:map} since the symbol $\Hat\Phi:\bK_I\to\bK_J$ can represent a map only in case $F_I\subset F_J$, in which case a map is the special case of a coordinate change with domain $U_{IJ}=U_I$. So in either case, the symbol $\Hat\Phi: \bK_I\to \bK_J$ means the choice of a domain $U_{IJ}\subset U_I$ and a map $\Hat\Phi:\bK_I|_{U_{IJ}}\to\bK_J$
from the restriction of $\bK_I$ with domain $U_{IJ}$.

The following lemma shows that the index condition is in fact equivalent to a tangent bundle condition which was first introduced, in a weaker version, by \cite{FO}, and formalized in the present version by \cite{J}.
We have chosen to present it as an index condition, 
since that is closer to the basic motivating question of how to associate canonical (equivalence classes of) Kuranishi atlases to moduli spaces described in terms of nonlinear Fredholm operators; see Remark~\ref{rmk:smart}~(iv).

\begin{lemma} \label{le:change}
The index condition is equivalent to the {\bf tangent bundle condition}, which requires isomorphisms for all $v=\phi(u)\in\phi(U_{IJ})$, 
\begin{equation}\label{tbc}
\rd_v s_J : \;\quotient{\rT_v U_J}{\rd_u\phi (\rT_u U_I)} \;\stackrel{\cong}\longrightarrow \; \quotient{E_J}{\Hat\phi(E_I)},
\end{equation}
or equivalently at all (suppressed) base points as above
\begin{equation}\label{inftame}
E_J=\im\rd s_J + \im\Hat\phi_{IJ} \qquad\text{and}\qquad
\im\rd s_J \cap \im\Hat\phi_{IJ} = \Hat\phi_{IJ}(\im\rd s_I).
\end{equation}
Moreover, the index condition implies that $\phi(U_{IJ})$ is an open subset of $s_J^{-1}(\Hat\phi(E_I))$,
and that the charts $\bK_I, \bK_J$ have the same dimension.
\end{lemma}

\begin{proof}
We will suppress most base points in the notation.
To see that the tangent bundle condition \eqref{tbc} implies the index condition, first note that the compatibility with sections, $\Hat\phi\circ\rd s_I = \rd s_J \circ \rd\phi$ implies
$$
\Hat\phi(\im\rd s_I)\subset\im\rd s_J , \qquad \rd\phi \bigl(\ker\rd s_I \bigr) \subset  \ker\rd s_J  .
$$
Since $\phi$ and $\Hat\phi$ are embeddings, this implies dimension differences $d,d'\geq 0$ in
$$
\dim\im\rd s_I  + d = \dim\im\rd s_J, \qquad
\dim\ker\rd s_I + d' = \dim\ker\rd s_J .
$$
The fact that \eqref{tbc} is an isomorphism implies that the Kuranishi charts have equal dimensions, and hence
\begin{align*}
\dim E_J  - \dim E_I
&= \dim U_J  - \dim U_I \\
&= \dim \ker\rd s_J + \dim\im\rd s_J - \dim \ker\rd s_I - \dim\im\rd s_I
\\ &= d+d' .
\end{align*}
Moreover, if we pick a representative space $C_I$ for the cokernel, i.e.\ $E_I = \im\rd s_I \oplus C_{I}$, then
then the surjectivity of $\rd s_J$ in \eqref{tbc}  gives
\begin{equation}\label{EJ}
E_J = \im \rd s_J + \Hat\phi(E_I)
= \im \rd s_J + \bigl( \Hat\phi(C_I) \oplus \Hat\phi(\im\rd s_I) \bigr) = \im \rd s_J + \Hat\phi(C_I) ,
\end{equation}
where
$$
\dim \Hat\phi(C_I) = \dim E_I - \dim\im\rd s_I
= \dim E_J - \dim\im\rd s_J - d' .
$$
Thus the sum \eqref{EJ} must be direct and $d'=0$, which implies the identification of cokernels and kernels.

Conversely, to see that the index condition implies the tangent bundle condition let
again
$C_I\subset E_I$ be a complement of $\im\rd s_I$. Then compatibility of the sections $s_J \circ\phi = \Hat\phi \circ s_I$ implies
$$
\Hat\phi(E_I) = \Hat\phi (\im\rd s_I) \oplus \Hat\phi(C_I)
= \rd s_J(\im\rd\phi) \oplus \Hat\phi(C_I)  .
$$
Moreover, let $N_u\subset \rT U_J$ be a complement of $\im \rd_u \phi$,
then the identification of cokernels takes the form
$$
E_J = \rd s_J(N_u) \oplus \rd s_J(\im\rd\phi) \oplus \Hat\phi(C_I) = \rd s_J(N_u) \oplus \Hat\phi(E_I).
$$
This shows that \eqref{tbc} is surjective, and for injectivity it remains to check injectivity of $\rd s_J |_{N_u}$. The latter holds since the identification of kernels implies $\ker\rd s_J \subset \im\rd\phi$.

To check the equivalence of \eqref{inftame} and \eqref{tbc} note that the first condition in \eqref{inftame} is the surjectivity of \eqref{tbc}, while the injectivity is equivalent to $\im\rd s_J \cap \im\Hat\phi_{IJ} \subset \rd s_J(\im\rd\phi_{IJ})$. 
The latter equals $\Hat\phi_{IJ}(\im\rd s_I)$ by the compatibility $s_J\circ\phi_{IJ}=\Hat\phi_{IJ}\circ s_I$ of sections. So \eqref{inftame} implies \eqref{tbc}, and for the converse it remains to check that \eqref{tbc} implies equality of the above inclusion. This follows from a dimension count as in \eqref{EJ} above.

Finally, to see that $\phi(U_{IJ})$ is an open subset of $s_J^{-1}(\Hat\phi(E_I))$, we may choose the complements $N_u$ above such that $\Nn:=\bigcup_{u\in U_{IJ}} N_u\subset {\rm T}U_J|_{\phi(U_{IJ})}$ is a normal bundle to $\phi(U_{IJ})$.
Then $\phi(U_{IJ})$ has a neighbourhood that is diffeomorphic to a neighbourhood of the zero section in $\Nn$, and we may pull back $s_J$ to a smooth map $\Nn\to E_J/\Hat\phi(E_I)$. It satisfies the assumptions of the implicit function theorem on the zero section by \eqref{tbc}, and hence for a sufficiently small open neighbourhood $\Uu\subset\Nn$ of $\phi(U_{IJ})$ we have $s_J^{-1}(\Hat\phi(E_I))\cap\Uu =\phi(U_{IJ})$.
\end{proof}

The next lemmas provide restrictions and compositions of coordinate changes.

\begin{lemma} \label{le:restrchange}
Let $\Hat\Phi:\bK_I|_{U_{IJ}}\to \bK_J$ be a coordinate change from $\bK_I$ to $\bK_J$, and let $\bK'_I=\bK_I|_{U'_I}$, $\bK'_J=\bK_J|_{U'_J}$ be restrictions of the Kuranishi charts to open subsets $F'_I\subset F_I, F'_J\subset F_J$ with $F_I'\cap F_J'\ne \emptyset$.
Then a {\bf restricted coordinate change} from $\bK'_I$ to $\bK'_J$ is given by
$$
\Hat\Phi|_{U'_{IJ}} := \bigl( \, \phi|_{U'_{IJ}} \,, \,\Hat\phi \, \bigr)
\; : \; \bK'_I|_{U'_{IJ}} \to \bK'_J
$$
for any choice of open subset $U'_{IJ}\subset U_{IJ}$ of the domain such that
$$
U'_{IJ} \subset U'_I\cap \phi^{-1}(U'_J) , \qquad \psi_I(s_I^{-1}(0)\cap U'_{IJ}) = F'_I \cap F'_J .
$$
\end{lemma}

\begin{proof}
First note that restricted domains $U'_{IJ}\subset U_{IJ}$ always exist
 since we can choose e.g.\ $U'_{IJ} = U'_I \cap\phi^{-1}(U'_J)$, which is open in $U_{IJ}$ by the continuity of $\phi$ and has the required footprint since
$$
\psi_I\bigl(s_I^{-1}(0) \cap U'_I \bigr) \cap  \psi_I\bigl(s_I^{-1}(0) \cap \phi^{-1}(U'_J) \bigr)
\;=\; F'_I\cap \psi_J\bigl(s_J^{-1}(0) \cap U'_J \bigr)
\;=\; F'_I\cap F'_J .
$$
Next, $\bK'_I|_{U'_{IJ}}=\bK_I|_{U'_{IJ}}$ is a restriction of $\bK'_I$ to $F'_I\cap F'_J$ since it has the required footprint
$$
\psi'_I({s_I'}\,\!^{-1}(0)\cap U'_{IJ}) \;=\; \psi_I\bigl(s_I^{-1}(0) \cap U'_{IJ} \bigr)
\;=\; F'_I\cap F'_J .
$$
Finally, $\Hat\Phi' := (\phi|_{U'_{IJ}} , \Hat\phi )$ is a map since it satisfies the conditions of  Definition~\ref{def:map},
\begin{enumerate}
\item
$\phi' |_{{s'_I}\,\!^{-1}(0)} = \phi|_{s_I^{-1}(0)\cap U'_{IJ}}
= \psi_J^{-1} \circ \psi_I |_{s_I^{-1}(0)\cap U'_{IJ}}
= {\psi'_J}\,\!^{-1} \circ \psi'_I $ ;
\item
$s'_J \circ \phi'
= s_J|_{U'_J} \circ \phi|_{U'_{IJ}}
= \Hat\phi \circ s_{IJ}|_{U'_{IJ}}
= \Hat\phi' \circ s'_{IJ}$.
\end{enumerate}
This completes the proof since the index condition is preserved under restriction.
\end{proof}

\begin{lemma} \label{le:cccomp}
Let $\bK_I,\bK_J,\bK_K$ be Kuranishi charts such that
$ F_I\cap F_K \subset F_J$, and let $\Hat\Phi_{IJ}: \bK_I\to \bK_J$ and $\Hat\Phi_{JK}: \bK_J\to \bK_K$ be coordinate changes.
(That is, we are given restrictions
$\bK_I|_{U_{IJ}}$ to $F_I\cap F_J$ and $\bK_J|_{U_{JK}}$ to $F_J\cap F_K$ and maps $\Hat\Phi_{IJ}: \bK_I|_{U_{IJ}}\to \bK_J$, $\Hat\Phi_{JK}: \bK_J|_{U_{JK}}\to \bK_K$
satisfying the index condition.)

Then the following holds.
\begin{itemize}
\item
The domain $U_{IJK}:=\phi_{IJ}^{-1}(U_{JK}) \subset U_I$ defines a restriction $\bK_I|_{U_{IJK}}$
to $F_I \cap F_K$.
\item
The compositions $\phi_{IJK}:=\phi_{JK}\circ\phi_{IJ}: U_{IJK}\to U_K$ and $\Hat\phi_{JK}\circ\Hat\phi_{IJ}: E_I\to E_K$ define a map $\Hat\Phi_{IJK}:\bK_I|_{U_{IJK}}\to\bK_K$
in the sense of Definition~\ref{def:map}.
\item $(\phi_{IJK},\Hat\phi_{IJK})$ satisfy the index condition, so define a coordinate change.
\end{itemize}
We denote the induced {\bf composite coordinate change} $\Hat\Phi_{IJK}=(\phi_{IJK},\Hat\phi_{IJK})$ by
$$
\Hat\Phi_{JK}\circ \Hat\Phi_{IJ}
:=\Hat\Phi_{IJK} : \; \bK_I|_{U_{IJK}} \; \to\; \bK_K.
$$
\end{lemma}
\begin{proof}
In order to check that $\bigl(U_{IJK},E_I,s_I|_{U_{IJK}}, \psi_I|_{s_I^{-1}(0)\cap U_{IJK}}\bigr)$
is the required restriction,
we need to verify that it has footprint $F_I\cap F_K$.
 Indeed, $\psi_I\bigl(s_I^{-1}(0)\cap U_{IJK}\bigr) =F_I\cap F_K$ holds since we may decompose
$\psi_I=\psi_J\circ\phi_{IJ}$ on $s_I^{-1}(0)\cap U_{IJ}$ with $U_{IJK}\subset U_{IJ}$,
and then combine the identities
$$
\phi_{IJ}( s_I^{-1}(0) \cap U_{IJK} )
=\phi_{IJ}( s_{IJ}^{-1}(0) ) \cap U_{JK} \subset s_J^{-1}(0),
$$
$$
\psi_J\bigl( \phi_{IJ}( s_{IJ}^{-1}(0) ) \bigr)
= \psi_I( s_{IJ}^{-1}(0) ) = F_I\cap F_J ,
$$
$$
\psi_J( s_{J}^{-1}(0)\cap U_{JK}) = F_J\cap F_K .
$$
Finally, our assumption
$F_I\cap F_K\subset F_J$
ensures that
$ (F_I\cap F_J)\cap ( F_J\cap F_K) = F_I\cap F_K $.
This proves (i).
To prove (ii) we check the conditions of Definition~\ref{def:map}, noting that injectivity and homomorphisms are preserved under composition.
\begin{itemize}
\item[-]
On $U_{IJK}\cap s_{I}^{-1}(0)$ we have
$$
\phi_{IJK}
= \phi_{JK} \circ \phi_{IJ}
= \bigl( \psi_K^{-1}\circ\psi_{J} \bigr)
\circ \bigl(\psi_{J}^{-1} \circ \psi_{I}\bigr)
= \psi_K^{-1}\circ \psi_{I} .
$$
\item[-]
The sections are intertwined,
$$
s_K \circ \phi_{IJK}
= s_K \circ \phi_{JK} \circ \phi_{IJ}
= \Hat\phi_{JK} \circ s_J \circ \phi_{IJ}
= \Hat\phi_{JK}\circ \Hat\phi_{IJ} \circ s_{I}
= \Hat\phi_{IJK} \circ s_{IK} .
$$
\end{itemize}

Next, the index condition is preserved under composition by the following.
\begin{itemize}
\item[-] The kernel identifications
$\rd\phi_{IJ} \bigl(\ker\rd s_I \bigr) =  \ker\rd s_J $ and $\rd\phi_{JK} \bigl(\ker\rd s_J \bigr) =  \ker\rd s_K$ imply
$$
\rd\bigl(\phi_{JK} \circ \phi_{IJ} \bigr) \bigl(\ker\rd s_I \bigr) =
\bigl(\rd\phi_{JK} \circ \rd\phi_{IJ} \bigr) \bigl(\ker\rd s_I \bigr) =  \ker\rd s_K .
$$
\item[-]
Assuming $E_I = \im\rd s_I \oplus C_{I}$, the cokernel identification of $\Phi_{IJ}$ implies
$ E_J = \im \rd s_J \oplus C_J$ with  $C_J=\Hat\phi(C_I)$, so that  the cokernel identification of $\Phi_{JK}$ implies
$$
E_K = \im \rd s_K \oplus \Hat\phi_K(C_J) = \im \rd s_K \oplus (\Hat\phi_K\circ\Hat\phi_J)(C_I).
$$
\end{itemize}
\end{proof}

Finally, we introduce  two notions of equivalence between coordinate changes that may not have the same domain, and show that they are compatible with composition.

\begin{defn} \label{def:overlap}
Let $\Hat\Phi^\al :\bK_I|_{U^\al_{IJ}}\to \bK_J$ and  $\Hat\Phi^\be:\bK_I|_{U^\be_{IJ}}\to \bK_J$ be coordinate changes.
\begin{itemlist}
\item
We say the coordinate changes are {\bf equal on the overlap} and write $\Hat\Phi^\al\approx\Hat\Phi^\be$, if the restrictions of Lemma~\ref{le:restrchange} to $U'_{IJ}:=U^\al_{IJ}\cap U^\be_{IJ}$ yield equal maps
$$
\Hat\Phi^\al|_{U'_{IJ}} = \Hat\Phi^\be|_{U'_{IJ}}  .
$$
\item
We say that $\Hat\Phi^\be$ {\bf extends} $\Hat\Phi^\al$ and write $\Hat\Phi^\al\subset\Hat\Phi^\be$,
if $U_{IJ}^\al\subset U_{IJ}^\be$ and the restriction of Lemma~\ref{le:restrchange} yields equal maps
$$
\Hat\Phi^\be|_{U_{IJ}^\al} = \Hat\Phi^\al .
$$
\end{itemlist}
\end{defn}

\begin{lemma}  \label{le:cccompoverlap}
Let $\bK_I,\bK_J,\bK_K$ be Kuranishi charts such that
$F_I\cap F_K\subset F_J$,  and suppose $\Hat\Phi^\al_{IJ}\approx \Hat\Phi^\be_{IJ}:\bK_I\to \bK_J$ and $\Hat\Phi^\al_{JK}\approx\Hat\Phi^\be_{JK}: \bK_J\to \bK_K$ are coordinate changes that are equal on the overlap.
Then their compositions as defined in Lemma~\ref{le:cccomp} are equal on the overlap, $\Hat\Phi^\al_{JK}\circ \Hat\Phi^\al_{IJ}\approx \Hat\Phi^\be_{JK}\circ \Hat\Phi^\be_{IJ}$.

Moreover, if $\Hat\Phi^\al_{IJ}\subset \Hat\Phi^\be_{IJ}$ and $\Hat\Phi^\al_{JK}\subset\Hat\Phi^\be_{JK}$ are extensions, then
$\Hat\Phi^\al_{JK}\circ \Hat\Phi^\al_{IJ}\subset \Hat\Phi^\be_{JK}\circ \Hat\Phi^\be_{IJ}$
is an extension as well.
\end{lemma}
\begin{proof}
The overlap of the domains of $\phi^\al_{IJK}=\phi^\al_{JK}\circ\phi^\al_{IJ}$ and $\phi^\be_{IJK}=\phi^\be_{JK}\circ\phi^\be_{IJ}$ is
\begin{align*}
U^\al_{IJK} \cap U^\be_{IJK}
\;=\; (\phi_{IJ}^\al)^{-1}(U^\al_{JK}) \cap (\phi_{IJ}^\be)^{-1}(U^\be_{JK})
\;\subset\; U^\al_{IJ}\cap U^\be_{IJ}.
\end{align*}
By assumption, the injective maps $\phi^\al_{IJ}$ and $\phi^\be_{IJ}$ agree on all points in $U^\al_{IJ}\cap U^\be_{IJ}$, and hence also map $U^\al_{IJK} \cap U^\be_{IJK}$ to $U^\al_{JK} \cap U^\be_{JK}$
The first claim follows because on this latter domain we have $\phi^\al_{JK}=\phi^\be_{JK}$.
To prove the second claim it remains to note that
$$
U^\al_{IJK}
= (\phi_{IJ}^\al)^{-1}(U^\al_{JK})
= (\phi_{IJ}^\be)^{-1}(U^\al_{JK})
\subset (\phi_{IJ}^\be)^{-1}(U^\be_{JK})
= U^\be_{IJK} .
$$
\end{proof}

\section{Kuranishi atlases with trivial isotropy}\label{s:Ks}

With the preliminaries of Section~\ref{s:chart}
in hand, we can now define a notion of Kuranishi atlas with trivial isotropy on a compact  metrizable space $X$, which will be fixed throughout this section.
As before, we work exclusively in the case of trivial isotropy and hence drop this qualifier from the wording.
We first define the notions of Kuranishi atlas $\Kk$ and Kuranishi neighborhood $|\Kk|$, showing in Examples~\ref{ex:Haus} and~\ref{ex:Knbhd} that $|\Kk|$ need not be Hausdorff and
that the maps from the domains of the charts into $|\Kk|$ need not be injective.
Further, as in Example~\ref{ex:Khomeo}, $|\Kk|$ need not be metrizable or locally compact.
Moreover, in practice one can construct only weak Kuranishi atlases in the sense of Definition~\ref{def:Kwk}, although they do often have the additivity property of Definition~\ref{def:Ku2}.  The main result of this section is Theorem~\ref{thm:K} which states that given a weak additive Kuranishi atlas one can construct a Kuranishi atlas $\Kk$, whose neighborhood $|\Kk|$ is Hausdorff and has the injectivity property, and that moreover is well defined up to a natural notion of cobordism.  This theorem is proved in Sections~\ref{ss:tame},~\ref{ss:shrink} and \ref{ss:Kcobord}.

\subsection{Covering families and transition data}\label{ss:Ksdef}  \hspace{1mm}\\ \vspace{-3mm}

We begin by introducing the notion of a Kuranishi atlas.
There are various ways that one might try to define a ``Kuranishi structure", but in practice every such structure on a compact moduli space of holomorphic curves is constructed from a covering family of basic charts with certain compatibility conditions 
akin to our notion of Kuranishi atlas.
We express the compatibility in terms of a further collection of charts for overlaps, and will discuss three different versions of a cocycle condition.
We compare our definition with others in Remark~\ref{rmk:otherK}.
The basic building blocks of our notion of Kuranishi atlases are the following.

\begin{defn}\label{def:Kfamily}
Let $X$ be a compact metrizable space.
\begin{itemlist}
\item
A {\bf covering family of basic charts} for $X$ is a finite collection $(\bK_i)_{i=1,\ldots,N}$ of Kuranishi charts for $X$ whose footprints cover $X=\bigcup_{i=1}^N F_i$.
\item
{\bf Transition data} for a covering family $(\bK_i)_{i=1,\ldots,N}$ is a collection of Kuranishi charts $(\bK_J)_{J\in\Ii_\Kk,|J|\ge 2}$ and coordinate changes $(\Hat\Phi_{I J})_{I,J\in\Ii_\Kk, I\subsetneq J}$ as follows:
\begin{enumerate}
\item
$\Ii_\Kk$ denotes the set of subsets $I\subset\{1,\ldots,N\}$ for which the intersection of footprints is nonempty,
$$
F_I:= \; {\textstyle \bigcap_{i\in I}} F_i  \;\neq \; \emptyset \;;
$$
\item
$\bK_J$ is a Kuranishi chart for $X$ with footprint $F_J=\bigcap_{i\in J}F_i$ for each $J\in\Ii_\Kk$ with $|J|\ge 2$, and for one element sets $J=\{i\}$ we denote $\bK_{\{i\}}:=\bK_i$;
\item
$\Hat\Phi_{I J}$ is a coordinate change $\bK_{I} \to \bK_{J}$ for every $I,J\in\Ii_\Kk$ with $I\subsetneq J$.
\end{enumerate}
\end{itemlist}
 \end{defn}

The transition data for a covering family automatically satisfies a cocycle condition on the zero sets, where due to the footprint maps to $X$ we have for $I\subset J \subset K$
$$
\phi_{J K}\circ \phi_{I J}
= \psi_K^{-1}\circ\psi_J\circ\psi_J^{-1}\circ\psi_I
= \psi_K^{-1}\circ\psi_I
= \phi_{I K}
\qquad \text{on}\; s_I^{-1}(0)\cap U_{IK} .
$$
Since there is no natural ambient topological space into which the entire domains of the Kuranishi charts map, the cocycle condition on the complement of the zero sets has to be added as axiom. For the linear embeddings between obstruction spaces, we will always impose $\Hat\phi_{J K}\circ \Hat\phi_{I J} = \Hat\phi_{I K}$.  However for the embeddings between domains of the charts there are three natural notions of cocycle condition with varying requirements on the domains of the coordinate changes.

\begin{defn}  \label{def:cocycle}
Let $\Kk=(\bK_I,\Hat\Phi_{I J})_{I,J\in\Ii_\Kk, I\subsetneq J}$ be a tuple of basic charts and transition data. Then for any $I,J,K\in\Ii_K$ with  $I\subsetneq J \subsetneq K$ we define the composed coordinate change $\Hat\Phi_{J K}\circ \Hat\Phi_{I J} : \bK_{I}  \to \bK_{K}$ as in Lemma~\ref{le:cccomp} with domain $\phi_{IJ}^{-1}(U_{JK})\subset U_I$.
We then use the notions of Definition~\ref{def:overlap} to say that the triple of coordinate changes
$\Hat\Phi_{I J}, \Hat\Phi_{J K}, \Hat\Phi_{I K}$ satisfies the
\begin{itemlist}
\item {\bf weak cocycle condition}
if $\Hat\Phi_{J K}\circ \Hat\Phi_{I J} \approx \Hat\Phi_{I K}$, i.e.\ the coordinate changes are equal on the overlap, in particular if
\begin{equation*}
\qquad
\phi_{J K}\circ \phi_{I J} = \phi_{I K}
\qquad \text{on}\;\;
\phi_{IJ}^{-1}(U_{JK}) \cap U_{IK} ;
\end{equation*}
\item {\bf cocycle condition}
if $\Hat\Phi_{J K}\circ \Hat\Phi_{I J} \subset \Hat\Phi_{I K}$, i.e.\  $\Hat\Phi_{I K}$ extends the composed coordinate change, in particular if
\begin{equation}\label{eq:cocycle}
\qquad
\phi_{J K}\circ \phi_{I J} = \phi_{I K}
\qquad \text{on}\;\;
\phi_{IJ}^{-1}(U_{JK}) \subset U_{IK} ;
\end{equation}
\item {\bf strong cocycle condition}
if $\Hat\Phi_{J K}\circ \Hat\Phi_{I J} = \Hat\Phi_{I K}$ are equal as coordinate changes, in particular if
\begin{equation}\label{strong cocycle}
\qquad
\phi_{J K}\circ \phi_{I J} = \phi_{I K}
\qquad \text{on}\; \;
\phi_{IJ}^{-1}(U_{JK}) = U_{IK} .
\end{equation}
\end{itemlist}
 \end{defn}

The relevance of the these versions is that the weak cocycle condition can be achieved in practice by constructions of finite dimensional reductions for holomorphic curve moduli spaces, whereas the strong cocycle condition is needed for our construction of a virtual moduli cycle in Section~\ref{s:VMC} from perturbations of the sections in the Kuranishi charts.
The cocycle condition is an intermediate notion which is too strong to be constructed in practice and too weak to induce a VMC, but it does allow us to formulate Kuranishi atlases categorically. This in turn gives rise, via a topological realization of a category, to a virtual neighbourhood of $X$ into which all Kuranishi domains map. In the following we use the intermediate cocycle condition to develop these concepts.

\begin{defn}\label{def:Ku}
A {\bf Kuranishi atlas of dimension $\mathbf d$} on a compact metrizable space
$X$ is a tuple
$$
\Kk=\bigl(\bK_I,\Hat\Phi_{I J}\bigr)_{I, J\in\Ii_\Kk, I\subsetneq J}
$$
consisting of a covering family of basic charts $(\bK_i)_{i=1,\ldots,N}$ of dimension $d$
and transition data $(\bK_J)_{|J|\ge 2}$, $(\Hat\Phi_{I J})_{I\subsetneq J}$ for $(\bK_i)$ as in Definition~\ref{def:Kfamily}, that satisfy the {\it cocycle condition} $\Hat\Phi_{J K}\circ \Hat\Phi_{I J} \subset \Hat\Phi_{I K}$ for every triple
$I,J,K\in\Ii_K$ with $I\subsetneq J \subsetneq K$.
\end{defn}

\begin{rmk}\label{rmk:Ku}\rm
We have assumed from the beginning that $X$ is compact and metrizable.
Some version of  compactness is essential in order for $X$ to define a VMC, but one might hope to weaken the metrizability assumption.  
However, for our charts to model open subsets of $X$ we require the footprint maps $\psi_i: s_i^{-1}(0)\to X$ to be homeomorphisms onto open subsets $F_i \subset X$.  Hence any space $X$ with a Kuranishi atlas is Hausdorff, and by compactness has a finite covering by the footprints $F_i$. Since each of these are regular and second countable, it follows that $X$ must also be regular and second countable, and hence metrizable.  For details of these arguments, see Proposition~\ref{prop:Ktopl1}~(iv).
\end{rmk}

It is useful to think of the domains and obstruction spaces of a Kuranishi atlas as forming the following categories.

\begin{defn}\label{def:catKu}
Given a Kuranishi atlas $\Kk$ we define its {\bf domain category} $\bB_\Kk$ to consist of
the space of objects\footnote{
When forming categories such as $\bB_\Kk$, we take always the space of objects 
to be the disjoint union of the domains 
$U_I$, even if we happen to have defined the sets $U_I$ 
as subsets of some larger space such as $\R^2$ 
or a space of maps as in the Gromov--Witten case.
Similarly, the morphism space is a disjoint union of the $U_{IJ}$ even though $U_{IJ}\subset U_I$ for all $J\supset I$.}
$$
\Obj_{\bB_\Kk}:= \bigcup_{I\in \Ii_\Kk} U_I \ = \ \bigl\{ (I,x) \,\big|\, I\in\Ii_\Kk, x\in U_I \bigr\}
$$
and the space of morphisms
$$
\Mor_{\bB_\Kk}:= \bigcup_{I,J\in \Ii_\Kk, I\subset J} U_{IJ} \ = \ \bigl\{ (I,J,x) \,\big|\, I,J\in\Ii_\Kk, I\subset J, x\in U_{IJ} \bigr\}.
$$
Here we denote $U_{II}:= U_I$ for $I=J$, and for $I\subsetneq J$ use
the domain $U_{IJ}\subset U_I$ of the restriction $\bK_I|_{U_{IJ}}$ to $F_J$
that is part of the coordinate change $\Hat\Phi_{IJ} : \bK_I|_{U_{IJ}}\to \bK_J$.

Source and target of these morphisms are given by
$$
(I,J,x)\in\Mor_{\bB_\Kk}\bigl((I,x),(J,\phi_{IJ}(x))\bigr),
$$
where $\phi_{IJ}: U_{IJ}\to U_J$ is the  embedding given by $\Hat\Phi_{I J}$, and we denote $\phi_{II}:={\rm id}_{U_I}$.
Composition is defined by
$$
\bigl(I,J,x\bigr)\circ \bigl(J,K,y\bigr)
:= \bigl(I,K,x\bigr)
$$
for any $I\subset J \subset K$ and $x\in U_{IJ}, y\in  U_{JK}$ such that $\phi_{IJ}(x)=y$.

The {\bf obstruction category} $\bE_\Kk$ is defined in complete analogy to $\bB_\Kk$ to consist of
the spaces of objects $\Obj_{\bE_\Kk}:=\bigcup_{I\in\Ii_\Kk} U_I\times E_I$ and morphisms
$$
\Mor_{\bE_\Kk}: = \bigl\{ (I,J,x,e) \,\big|\, I,J\in\Ii_\Kk, I\subset J,  x\in U_{IJ}, e\in E_I \bigr\}.
$$
\end{defn}

We also express the further parts of a Kuranishi atlas in categorical terms:

\begin{itemlist}
\item
The obstruction category $\bE_\Kk$ is a bundle over $\bB_\Kk$ in the sense that there is a functor
$\pr_\Kk:\bE_\Kk\to\bB_\Kk$ that is given on objects and morphisms by projection $(I,x,e)\mapsto (I,x)$ and $(I,J,x,e)\mapsto(I,J,x)$ with locally trivial fiber $E_I$.
\item
The sections $s_I$ induce a smooth
section of this bundle, i.e.\ a functor $s_\Kk:\bB_\Kk\to \bE_\Kk$ which acts smoothly
on the spaces of objects and morphisms, and whose composite with the projection
$\pr_\Kk: \bE_\Kk \to \bB_\Kk$ is the identity. More precisely, it is given by $(I,x)\mapsto (I,x,s_I(x))$ on objects and by $(I,J,x)\mapsto (I,J,x,s_I(x))$ on morphisms.
\item
The zero sets of the sections $\bigcup_{I\in\Ii_\Kk} \{I\}\times s_I^{-1}(0)\subset\Obj_{\bB_\Kk}$ form a very special strictly full subcategory $s_\Kk^{-1}(0)$ of $\bB_\Kk$. Namely, $\bB_\Kk$ splits into the subcategory $s_\Kk^{-1}(0)$ and its complement (given by the full subcategory with objects  $\{ (I,x) \,|\, s_I(x)\ne 0 \}$) in the sense that there are no morphisms of $\bB_\Kk$ between the underlying sets of objects.
(This is since given any morphism $(I,J,x)$ we have $s_I(x)=0$ if and only if $s_J(\phi_{IJ}(x))=\Hat\phi_{IJ}(s_I(x))=0$.)
\item
The footprint maps $\psi_I$ give rise to a surjective functor $\psi_\Kk: s_\Kk^{-1}(0) \to \bX$ 
to the category $\bX$ with object space $X$ and trivial morphism spaces.
It is given by $(I,x)\mapsto \psi_I(x)$ on objects and by $(I,J,x)\mapsto {\rm id}_{\psi_I(x)}$ on morphisms.
\end{itemlist}

\begin{lemma}\label{le:Kcat}  
The categories $\bB_{\Kk}$ and $\bE_{\Kk}$  are well defined. 
%
%
\end{lemma}
\begin{proof}  
We must check that the composition of morphisms in $\bB_{\Kk}$ is well defined and associative.
To see this, note that the composition $\bigl(I,J,x\bigr)\circ \bigl(J,K,y\bigr)$ only needs to be defined for $x=\phi_{IJ}^{-1}(y)\in\phi_{IJ}^{-1}(U_{JK})$, i.e.\ for $x\in U_{IJK}$ in the domain of the composed coordinate change $\Phi_{JK}\circ\Phi_{IJ}$, which by axiom (d) is contained in the domain of $\Phi_{IK}$, and hence $\bigl(I,K,x\bigr)$ is a well defined morphism.
With this said, identity morphisms are given by $\bigl(I,I,x\bigr)$ for all $x\in U_{II}=U_I$, and the composition is associative since for any $I\subset J \subset K\subset L$, and $x\in U_{IJ}, y\in U_{JK}, z\in U_{KL}$ the three morphisms
$\bigl(I,J,x\bigr), \bigl(J,K,y\bigr),\bigl(K,L,z\bigr)$ are composable iff
$y=\phi_{IJ}(x)$ and $z=\phi_{JK}(y)$. In that case we have
$$
\bigl(I,J,x\bigr)\circ \Bigl(\bigl(J,K,y\bigr) \circ \bigl(K,L,z)\bigr) \Bigr) \\
\;=\;
\bigl(I,J,x\bigr)\circ \bigl(J,L,\phi_{IJ}(x)\bigr)
\;=\;
\bigl(I,L,x\bigr)
$$
and $z=\phi_{JK}(\phi_{IJ}(x))=\phi_{IK}(x)$, hence
$$
\Bigl( \bigl(I,J,x\bigr)\circ \bigl(J,K,y\bigr) \Bigr) \circ \bigl(K,L,z\bigr)
\;=\; \bigl(I,K,x\bigr)\circ \bigl(K,L,\phi_{IK}(x)\bigr)
\;=\; \bigl(I,L,x\bigr)
$$
%
\end{proof}

\begin{remark}\label{rmk:Kgroupoid}\rm
Because $\Kk$ has trivial isotropy, all sets of morphisms in $\bB_\Kk$ between fixed objects consist of at most one element. However, $\bB_\Kk$ cannot be completed to an \'etale groupoid\footnote
{
The notion of \'etale groupoid is reviewed in Remark~\ref{rmk:grp}.}
 since the inclusion of inverses of the coordinate changes, and their compositions, may yield singular spaces of morphisms. Indeed, coordinate changes $\bK_I\to\bK_K$ and $\bK_J\to\bK_K$ with the same target chart are given by embeddings $\phi_{IK}:U_{IK}\to U_K$ and $\phi_{JK}:U_{JK}\to U_K$, whose images may not intersect transversely (for example, often their intersection is contained only in the zero set $s_K^{-1}(0)$), yet this intersection would be a component of the space of morphisms from $U_{I}$ to $U_{J}$.
Moreover, since the map $\phi_{IJ}:U_{IJ}\to U_J$ underlying a coordinate change could well have target of higher dimension than its domain,
the target map  $t:\Mor_{\bB_\Kk}\to \Obj_{\bB_\Kk}$
is not in general a local  diffeomorphism.
However,  the spaces of objects and morphisms in  $\bB_\Kk$  are smooth manifolds, and all structure maps are smooth embeddings.
Thus $\bB_\Kk$ is in some ways  similar to the \'etale groupoids considered  in \cite{Mbr}.
\end{remark}

The categorical formulation of a Kuranishi atlas $\Kk$ allows us to construct a topological space  $|\Kk|$ which contains a homeomorphic copy $\io_{\Kk}(X)\subset |\Kk|$ of $X$ and hence may be viewed as a virtual neighbourhood of $X$.

\begin{defn}  \label{def:Knbhd}
Let $\Kk$ be a Kuranishi atlas for the compact space $X$.
Then the 
{\bf Kuranishi neighbourhood} or {\bf virtual neighbourhood} of $X$,
$$
|\Kk| := \Obj_{\bB_\Kk}/{\scriptstyle\sim}
$$
is the topological realization\footnote
{
As is usual in the theory of \'etale groupoids we take the realization of
the category $\bB_\Kk$ to be a quotient of its space of objects rather than the classifying space
of the category $\bB_\Kk$ (which is also sometimes called the topological realization), cf.\ \cite{ALR}.} of the category $\bB_\Kk$, that is the quotient of the object space $\Obj_{\bB_\Kk}$ by the equivalence relation generated by
$$
\Mor_{\bB_\Kk}\bigl((I,x),(J,y)\bigr) \ne \emptyset \quad \Longrightarrow \quad
(I,x) \sim (J,y) .
$$
We denote by  $\pi_\Kk:\Obj_{\bB_\Kk}\to |\Kk|$ the natural projection $(I,x)\mapsto [I,x]$, where $[I,x]\in|\Kk|$ denotes the equivalence class containing $(I,x)$.
We moreover equip $|\Kk|$ with the quotient topology, in which $\pi_\Kk$ is continuous.
Similarly, we define
$$
|\bE_\Kk|:=\Obj_{\bE_\Kk} /{\scriptstyle\sim}
$$
to be the topological realization of the
obstruction category $\bE_\Kk$.  The natural projection $\Obj_{\bE_\Kk}\to |\bE_\Kk|$ is still denoted $\pi_\Kk$.
\end{defn}

\begin{lemma} \label{le:realization}
The functor ${\rm pr}_\Kk:\bE_\Kk\to\bB_\Kk$ induces a continuous map
$$
|{\rm pr}_\Kk|:|\bE_\Kk| \to |\Kk|,
$$
which we call the {\bf obstruction bundle} of $\Kk$, although its fibers generally do not have the structure of a vector space.\footnote
{
Proposition~\ref{prop:linear} shows that the fibers do have a natural linear structure when $\Kk$ satisfies a natural additivity condition on its obstruction spaces as well as taming conditions on its domains.
Both conditions are necessary, see Example~\ref{ex:nonlin} and Remark~\ref{rmk:LIN}.
}
However, it has a continuous zero section
$$
|0_\Kk| : \; |\Kk| \to |\bE_\Kk| , \quad [I,x] \mapsto [I,x,0] .
$$
Moreover, the section $s_\Kk:\bB_\Kk\to \bE_\Kk$ descends to a continuous section
$$
|s_\Kk|:|\Kk|\to |\bE_\Kk| .
$$
These maps are sections in the sense that $|\pr_\Kk|\circ|s_\Kk| =  |\pr_\Kk|\circ |0_\Kk|= {\rm id}_{|\Kk|}$.
Moreover, there is a natural homeomorphism from the realization of the subcategory $s_\Kk^{-1}(0)$
to the zero set of the section, with the relative topology induced from $|\Kk|$,
$$
\bigr| s_\Kk^{-1}(0)\bigr| \;=\; \quotient{s_\Kk^{-1}(0)}{\sim_{\scriptscriptstyle s_\Kk^{-1}(0)}}
\;\overset{\cong}{\longrightarrow}\;
|s_\Kk|^{-1}(0) \,:=\; \bigl\{[I,x] \,\big|\, s_I(x)=0  \bigr\}  \;\subset\; |\Kk| .
$$
\end{lemma}
\begin{proof}
The zero section is induced by the functor $0_\Kk: \Kk \to \bE_\Kk$ given by the map $(I,x) \mapsto (I,x,0)$ on objects and $(I,J,x) \mapsto (I,J,x,0)$ on morphisms.
The existence and continuity of $|\pr_\Kk|$, $|0_\Kk|$, and $|s_\Kk|$ then follows from continuity of the maps induced by $\pr_\Kk$, $0_\Kk$, and $s_\Kk$ on the object space and the following general fact: Any functor $f:A\to B$, which is continuous on the object space, induces a continuous map between the realizations
(where these are given the quotient topology of each category).
Indeed, $|f|:|A|\to|B|$ is well defined since the functoriality of $f$ ensures $a\sim a' \Rightarrow f(a)\sim f(a')$. Then by definition we have $\pi_B\circ f = |f| \circ \pi_A$ with the projections $\pi_A:A\to |A|$ and $\pi_B:B\to |B|$.
To prove continuity of $|f|$ we need to check that for any open subset $U\subset |B|$ the preimage $|f|^{-1}(U)\subset|A|$ is open, i.e.\ by definition of the quotient topology, $\pi_A^{-1}\bigl(|f|^{-1}(U)\bigr)\subset A$ is open. But
$\pi_A^{-1}\bigl(|f|^{-1}(U)\bigr) = f^{-1}\bigl(\pi_B^{-1}(U)\bigr)$, which is open by the continuity of $\pi_B$ (by definition) and $f$ (by assumption).

Next, recall that the equivalence relation $\sim$ on $\Obj_{\bB_\Kk}$ that defines $|\Kk|$ is given by the embeddings $\phi_{IJ}$, their inverses, and compositions. Since these generators intertwine the zero sets $s_I^{-1}(0)$ and the footprint maps $\psi_I:s_I^{-1}(0)\to F_I$, we have the useful observations
\begin{align} \label{eq:useful1}
\psi_I(x)=\psi_J(y) \quad &\Longrightarrow \quad \;  (I,x)\sim (J,y) , \\
 \label{eq:useful2}
(I,x)\sim (J,y) , \; s_I(x)=0 \quad &\Longrightarrow \quad \;  s_J(y)=0, \; \psi_I(x)=\psi_J(y).
\end{align}
In particular, \eqref{eq:useful2} implies that the equivalence relation $\sim$ on $\Obj_{\bB_\Kk}$ that defines $|\Kk|=\qu{\Obj_{\bB_\Kk}}{\sim}$ restricted to the objects $\{(I,x) \,|\, s_I(x)=0\}$ of the subcategory $s_\Kk^{-1}(0)$ coincides with the equivalence relation $\sim_{\scriptscriptstyle s_\Kk^{-1}(0)}$ generated by the morphisms of $s_\Kk^{-1}(0)$.
Hence the map $\bigr| s_\Kk^{-1}(0)\bigr| \to |s_\Kk|^{-1}(0)$, $[(I,x)]_{\scriptscriptstyle s_\Kk^{-1}(0)} \mapsto [(I,x)]_{\Kk}$ is a bijection. It also is continuous because it is the realization of the functor $s_\Kk^{-1}(0)\to \bB_\Kk$ given by the continuous embedding of the object space.

To check that the inverse is continuous, consider an open subset $Z\subset \bigr| s_\Kk^{-1}(0)\bigr|$, 
that is with open preimage $\pi_\Kk^{-1}(Z)\subset \{(I,x) \,|\, s_I(x)=0\}$.
The latter is given the relative topology induced from $\Obj_{\bB_\Kk}$, hence we have
$\pi_\Kk^{-1}(Z) = 
\Ww\cap \{(I,x) \,|\, s_I(x)=0\}$ for some open subset $\Ww\subset\Obj_{\bB_\Kk}$. Now we need to check that $\Ww$ can be chosen so that 
$\Ww=\pi_\Kk^{-1}(\pi_\Kk(\Ww))$, i.e.\ $\pi_\Kk(\Ww)\subset|\Kk|$ is open, and
$\pi_\Kk(\Ww)\cap |s_\Kk|^{-1}(0)=Z$.
For that purpose note that each footprint $\psi_I(\pi_\Kk^{-1}(Z)\cap U_I)\subset X$ is open since $\psi_I$ is a homeomorphism from $s_I^{-1}(0)$ to an open subset of $X$ and $Z_I:= \pi_\Kk^{-1}(Z)\cap U_I \subset s_I^{-1}(0)$ is open by assumption.
Hence the finite union $\bigcup_{I\in\Ii_\Kk}\psi_I(Z_I)\subset X$ is open, thus has a closed complement, so that each preimage $C_J := \psi_J^{-1}\bigl( X\less \bigcup_I \psi_I(Z_I)\bigr)\subset U_J$ is also closed by the homeomorphism property of the footprint map $\psi_I$. 
Moreover, by \eqref{eq:useful1} and \eqref{eq:useful2}  the morphisms in $\bB_\Kk$ on the zero sets are determined by the footprint functors, so that we have $\psi_J^{-1}(\psi_I(Z_I))=\pi_\Kk^{-1}(\pi_\Kk(Z_I))\cap U_J=\pi_\Kk^{-1}(Z \cap \pi_\Kk(U_I))\cap U_J$ for each $I,J\in\Ii_\Kk$, and thus
$C_J = s_J^{-1}(0) \less \pi_\Kk^{-1}(Z)$.
With that we obtain an open set
$\Ww: = {\textstyle \bigcup_{I\in\Ii_\Kk}} \bigl( U_I \less C_I \bigr) \subset \Obj_{\bB_\Kk}$ such that $\pi_\Kk(\Ww)\subset|\Kk|$ is open since $\Ww$ is invariant under the equivalence relation by $\pi_\Kk$, namely
$$
\pi_\Kk(\Ww) \;=\;  {\textstyle \bigcup_{I\in\Ii_\Kk}} \pi_\Kk( U_I ) \less \bigl( \pi_\Kk(s_I^{-1}(0))  \less Z\bigr)  \;=\; |\Kk| \less  \bigl( |s_\Kk^{-1}(0)| \less Z\bigr) 
$$
so that its preimage is $\Ww$ and hence open, by the identity
$$
\pi_\Kk^{-1}(\pi_\Kk(\Ww))
\;=\; {\textstyle \bigcup_{I\in\Ii_\Kk}}  U_I \cap \pi_\Kk^{-1}\bigl( |\Kk| \less  ( |s_\Kk^{-1}(0)| \less Z ) \bigr) 
 \;=\; {\textstyle \bigcup_{I\in\Ii_\Kk}}  U_I \less \bigl( s_I^{-1}(0) \less Z \bigr).
$$
Finally, using the above, we check that $\pi_\Kk(\Ww)$ has the required intersection
\begin{align*}
\pi_\Kk(\Ww) \cap  |s_\Kk|^{-1}(0)
\;=\; \bigl(  |\Kk| \less  \bigl( |s_\Kk^{-1}(0)| \less Z \bigr) \bigr)  \cap  |s_\Kk|^{-1}(0)
\;=\; Z .
\end{align*}
This proves the homeomorphism between $\bigl|s_\Kk^{-1}(0)\bigr|$ and $ |s_\Kk|^{-1}(0)$.
\end{proof}

The next lemma shows that the zero set 
$\bigl|s_\Kk^{-1}(0)\bigr|\cong |s_\Kk|^{-1}(0)$
is also naturally homeomorphic to $X$, and hence $X$ embeds into the virtual neighbourhood $|\Kk|$.

\begin{lemma} \label{le:Knbhd1}
The footprint functor $\psi_\Kk: s_\Kk^{-1}(0) \to \bX$ descends to a homeomorphism $|\psi_\Kk| :  |s_\Kk|^{-1}(0) \to X$.   Its inverse is given by
$$
\io_{\Kk}:= |\psi_\Kk|^{-1} : \; X\;\longrightarrow\;  |s_\Kk|^{-1}(0) \;\subset\; |\Kk|, \qquad
p \;\mapsto\; [(I,\psi_I^{-1}(p))] ,
$$
where $[(I,\psi_I^{-1}(p))]$ is independent of $I\in\Ii_\Kk$ with $p\in F_I$.
\end{lemma}

\begin{proof}
To begin, recall that $\psi_\Kk$ is a surjective functor from $s_\Kk^{-1}(0)$ to $\bX$ with objects $X$ (i.e.\ the footprints $F_I = \psi_I (s_I^{-1}(0))$ cover $X$).
Hence the argument of Lemma~\ref{le:realization} proves that $|\psi_\Kk|$ is well defined, surjective, and continuous when $|\psi_\Kk|$ is considered as a map from the quotient space $|s_\Kk^{-1}(0)|$ (with its quotient topology rather than the relative topology induced by $|\Kk|$)  to $X$.

The map $|\psi_\Kk|=\io_\Kk^{-1}$ considered here is given by composing this realization of the functor $\psi_\Kk$ with the natural homeomorphism $\bigl|s_\Kk^{-1}(0)\bigr|\overset{\cong}{\to} |s_\Kk|^{-1}(0)$ from Lemma~\ref{le:realization}.
So it remains to check continuity of $\io_{\Kk}$ with respect to the subspace topology on $|s_\Kk|^{-1}(0)\subset|\Kk|$. 
For that purpose we need to consider an open subset $V\subset|\Kk|$, that is $\pi_\Kk^{-1}(V)\subset \Obj_{\bB_\Kk}$ is open. Since $\Obj_{\bB_\Kk}$ is a disjoint union that means $\pi_\Kk^{-1}(V)=\bigcup_{I\in\Ii_\Kk} \{I\}\times W_I$ is a union of open subsets $W_I\subset U_I$.  So in the relative topology $(W_I\cap s_I^{-1}(0))\subset s_I^{-1}(0)$ is open, as is its image under the homeomorphism $\psi_I: s_I^{-1}(0) \to F_I \subset X$.
Therefore
$$
\io_{\Kk}^{-1}(V)
\;=\;  |\psi_{\Kk}|(V)
\;=\;\psi_\Kk\left(s_\Kk^{-1}(0)\cap
{\textstyle
\bigcup_{I\in\Ii_\Kk}} \{I\}\times W_I\right)
\,=\;
{\textstyle \bigcup_{I\in\Ii_\Kk}} \psi_I(W_I\cap s_I^{-1}(0))
$$
is open in $X$ since it is a union of open subsets.  This completes the proof.
\end{proof}

Note that the injectivity of $\io_{\Kk}:X\to|\Kk|$ could be seen directly from the injectivity property \eqref{eq:useful2} of the equivalence relation $\sim$ on $s_\Kk^{-1}(0)\subset\Obj_{\bB_\Kk}$.
In particular, this property implies injectivity of the projection of the zero sets in fixed charts, $\pi_\Kk :s_I^{-1}(0) \to |\Kk|$.
This injectivity however only holds on the zero set.
On $U_I\less s_I^{-1}(0)$, the projections $\pi_\Kk: U_I\to |\Kk|$ need not be injective,
as the following example shows.

\begin{figure}[htbp] 
   \centering
   \includegraphics[width=4in]{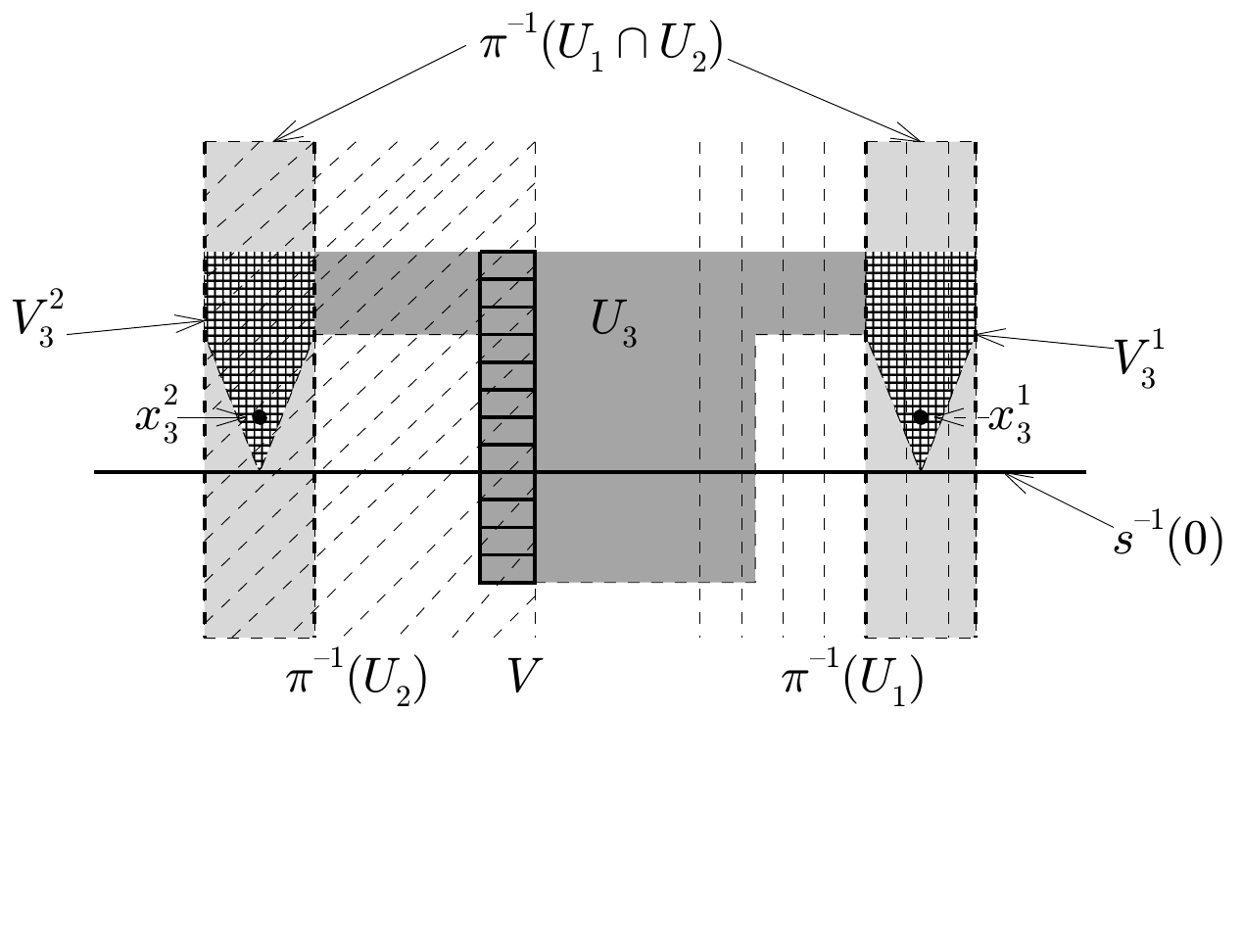}
   \caption{
The lift $\pi^{-1}(U_1\cap U_2)$ is shown as two light grey strips in $\R\times \R$, intersecting the dark grey region $U_3$ in the two shaded sets $V_3^1, V_3^2$.  The domains $U_1,U_2\subset S^1\times\R$ lift injectively to the dashed sets. The points $x_3^1\neq x_3^2 \in U_3$ have the same image in $|\Kk|\subset S^1\times\R$.
In Example~\ref{ex:nonlin} we will add another chart with domain $V\cup V^2_3$, where $V$ is barred.
}
   \label{fig:3}
\end{figure}

\begin{example}[Failure of Injectivity]\label{ex:Knbhd}\rm
The circle $X=S^1=\R/\Z$ can be covered by a single ``global'' Kuranishi chart $\bK_0$ of dimension $1$ with domain $U_0= S^1\times \R$, obstruction space $E_0 =\R$, section map $s_0 = \pr_\R$, and footprint map $\psi_i= \pr_{S^1}$.
A slightly more complicated Kuranishi atlas (involving transition charts but still no cocycle conditions) can be obtained by the open cover $S^1=F_1\cup F_2\cup F_3$ with $F_i=(\frac i3,\frac{i+2}3) \subset \R/\Z$ such that all pairwise intersections $F_{ij}:=F_i\cap F_j \neq \emptyset$ are nonempty, but the triple intersection $F_1\cap F_2\cap F_3$ is empty. We obtain a covering family of basic charts $\bigl(\bK_i:=\bK_0|_{U_i}\bigr)_{i=1,2,3}$ with these footprints by restricting $\bK_0$ to the open domains $U_i:=F_i\times (-1,1)\subset S^1\times \R$.
Similarly, we obtain transition charts $\bK_{ij}:=\bK_0|_{U_i\cap U_j}$ and coordinate changes $\Hat\Phi_{i, ij}:= \Hat\Phi_{0,0}|_{U_i\cap U_j}$ by restricting the identity map $\Hat\Phi_{0,0}=({\rm id}_{U_0}, {\rm id}_{E_0}):\bK_0 \to \bK_0$ to the overlap $U_{ij}:= U_i\cap U_j$.
These are well defined for any pair $i,j\in\{1,2,3\}$ (and satisfy all cocycle conditions), but for a Kuranishi atlas it suffices to restrict to $i<j$. That is, the transition charts $\bK_{12},\bK_{13},\bK_{23}$ and corresponding coordinate changes
$\Hat\Phi_{1,12}, \Hat\Phi_{2,12}, \Hat\Phi_{1,13}, \Hat\Phi_{3,13}, \Hat\Phi_{2,23}, \Hat\Phi_{3,23}$ form transition data, for which the cocycle condition is vacuous.
The realization of this Kuranishi atlas is $|\Kk|=U_1\cup U_2\cup U_3\subset S^1\times \R$, and the maps $U_i\to |\Kk|$ are injective.

However, keeping the same basic charts $\bK_1, \bK_2$, and transition data for $i,j\in\{1,2\}$, we may choose $\bK_3$ to have the same form as $\bK_0$ but with domain $U_3\subset (0,2)\times \R$ such that the projection $\pi:\R\times \R \to S^1\times \R$ embeds $U_3\cap (\R\times\{0\})=(1,\frac 23)\times\{0\}$ to $F_3\times\{0\}$. We can moreover choose $U_3$ so large that the inverse image of $U_1\cap U_2$ meets $U_3$ in two components $\pi^{-1}(U_1\cap U_2)\cap U_3 = V_3^1 \sqcup V_3^2$ with $\pi(V_3^1)=\pi(V_3^2)$, but there are continuous lifts $\pi^{-1}: U_i \cap \pi(U_3) \to U_3$ with $V^i_3\subset\pi^{-1}(U_i)$; cf.\ Figure~\ref{fig:3}.
These intersections $V^i_3\subset U_3$ necessarily lie outside of the zero section $s_3^{-1}(0)=F_3\times\{0\}$, though their closure might intersect it.
Then it remains to construct transition data from $\bK_i$ for $i=1,2$ to $\bK_3$.
We choose the transition charts as restrictions $\bK_{i3}:= \bK_3|_{U_{i3}}$ of $\bK_3$ to the domains $U_{i3}:= \pi^{-1}(U_1)\cap U_3$, with transition maps $\Hat\Phi_{3,i3}:=\Hat\Phi_{3,3}|_{U_{i3}}$. Finally, we construct the transition maps $\Hat\Phi_{i,i3}:\bK_i|_{U_{i,i3}}\to\bK_3$ for $i=1,2$ by the identity $\Hat\phi_{i,i3}:={\rm id}_{E_i}$ on the identical obstruction bundles $E_i=E_3=E_0$ and the lift
$\phi_{i,i3}:= \pi^{-1}$ on the domain $U_{i,i3}: = U_1\cap \pi(U_3)$.

This again defines a Kuranishi atlas with vacuous cocycle condition, but the map $\pi_\Kk: U_3\to |\Kk|$ is not injective.
Indeed any point $x_3^1\in V_3^1\subset U_3$ is identified $[x_3^1]=[x_3^2]\in|\Kk|$ with the corresponding point $x_3^2\in V_3^2$ with $\pi(x_3^1)=\pi(x_3^2)= y \in S^1\times\R$.
Indeed, denoting by $(ij, z)$ the point $z$ considered as an element of $U_{ij}$ (which is just a simplified version of the previous notation $(I,x)$ for a point $x\in U_I$), we have
\begin{equation}\label{eq:equiv123}
(3,x_3^1)\sim (13, x_3^1)\sim (1,y)\sim(12,y)\sim(2,y)\sim (23,x_3^2)\sim (3,x_3^2),
\end{equation}
 where each equivalence is induced by the relevant coordinate change.
Since there are such points $x_3^1$ arbitrarily close to the zero set $s_3^{-1}(0)=F_3\times\{0\}$, the projection $\pi_\Kk:U_3\to |\Kk|$ is not injective on any neighborhood of the zero set $s_3^{-1}(0)$.
\end{example}

We next adapt the above example so that the fibers of the bundle $\pr_\Kk:
|\bE_\Kk|\to |\Kk|$ also fail to have a linear structure.
(Remark~\ref{rmk:LIN} describes another scenario where linearity fails.)

\begin{example}
[Failure of Linearity]  \label{ex:nonlin}\rm
We can build on the construction of Example~\ref{ex:Knbhd} to obtain a Kuranishi atlas $\Kk'$ of dimension $0$ on $S^1$ with four basic charts, in which the fibers of $\pr_{\Kk'}$ are not even contractible.
The first three basic charts and the associated transition charts are obtained by replacing the obstruction space $\R$ with $\C$. That is, for $I\in \{1,2,3,12,13,23\}$, we identify $\R\subset\C$ to define the charts
$$
\bK_I': = \bigl( \, U'_I:=U_I \,,\, E'_I:=\C \,,\, s_I':= s_I \,,\, \psi_I':=\psi_I \bigr) ,
$$
where $(U_I,E_I=\R,s_I,\psi_I)$ are the charts of Example~\ref{ex:Knbhd} which yield a noninjective Kuranishi atlas.
We also define coordinate changes by the same domains and nonlinear embeddings as before, that is we set $U_{i,ij}':=U_{i,ij}$ and $\phi_{i,ij}':=\phi_{i,ij}$ for $i<j \in \{1,2,3\}$. However, we now have a choice for the linear embeddings $\Hat\phi\,'_{i,ij}$  since compatibility with the sections only requires $\Hat\phi\,'_{i,ij}|_{\R}={\rm id}_\R$.
We will take $\Hat\phi\,'_{i,ij} = \id_{\C}$ except for $\Hat\phi\,'_{2,23}(\alpha +\hat\iota \beta) :=\alpha + 2 \hat\iota \beta$ (where we denote the imaginary unit by $\hat\iota$ to prevent confusion with the index $i$).  As above the cocycle condition is trivially satisfied since there are no triple intersections of footprints.
Then with $x_3^i\in U'_3$ as in Example~\ref{ex:Knbhd} the chain of equivalences \eqref{eq:equiv123} lifts to the obstruction space $E_3'=\C$ as
\begin{equation} \label{fiber2}
(3, x_3^1,\alpha + \hat\iota \beta) \;\sim\; (3,x_3^2,\alpha + 2\hat\iota \beta) .
\end{equation}
In order to also obtain the equivalences
\begin{equation} \label{fiber1}
(3, x_3^1,\alpha + \hat\iota \beta) \;\sim\; (3,x_3^2,\alpha + \hat\iota \beta)
\end{equation}
we add another basic chart $\bK_4'= \bK_3'|_{U'_4}$ with
domain indicated in Figure~\ref{fig:3},
$$
U'_4: = V \cup V_3^2 ,\qquad
V: = \pi^{-1} \bigl(F_{13}\times \R\bigr) \cap U'_3.
$$
This chart has footprint $F'_4=F_{13}$, so it requires no compatibility with $\bK'_2$, and for $I\subset\{1,3,4\}$ we always have $F'_I = F_{13}$. We define the transition charts as restrictions
$$
\bK'_{14}: = \bK'_1|_{\pi(U'_4)},\qquad \bK'_{34} = \bK'_3|_{U'_4} , \qquad
\bK_{134}: = \bK_3|_V .
$$
Then we obtain the coordinate changes $\Hat\Phi_{I,J}$ for $I\subsetneq J \subset\{1,3,4\}$ by setting $\Hat\phi\,'_{IJ} := \id_\C$ and $\phi_{IJ}'$ equal either to the identity or to $\pi^{-1}$, as appropriate, on the domains
\begin{align*}
U'_{1,14} := \pi(U'_4), \qquad
&
U'_{4,14}= U'_{3,34}= U'_{4,34} := U'_4, \\
U'_{1,134} = U'_{14,134} := \pi(V), \qquad
&
U'_{3,134} = U'_{4,134} = U'_{13,134} := V .
\end{align*}
To see that the cocycle condition holds, note that we only need to check it for the triples $(i,34,134), i=3,4$, $(j,14,134), j=1,4$, and $(k, 13,134), k=1,3$, and in all of these cases both $\phi_{JK}'\circ\phi_{IJ}'$ and $\phi_{IK}'$ have equal domain, given by $V$ or $\pi(V)$. This provides another chain of morphisms between the same objects as in \eqref{eq:equiv123},
$$
 (3,x_3^2)\sim (34,x_3^2)\sim (4,x_3^2) \sim (14,y) \sim (1,y)\sim
 (13,x^1_3)\sim (3,x_3^1),
$$
whose lift to the obstruction space is \eqref{fiber1} since $\Hat\phi\,'_{IJ}= \id_\C$ for all coordinate changes involved.
Therefore, the fiber of $|\pi_{\Kk'}|: |\bE_{\Kk'}|\to |\Kk'|$ over $[3,x_3^1]=[3,x_3^2]$ is
$$
|\pr_\Kk|^{-1}([3,x_3^1]) \;=\;\;  \quotient{\C}{\scriptstyle \bigl( \alpha + \hat\iota \beta \;\sim\; \alpha + 2\hat\iota \beta \bigr)}
\quad \cong\;  \R \times S^1 ,
$$
which does not have the structure of a vector space, and in fact is not even contractible.
\end{example}

Finally, we  give a simple example where $|\Kk|$ is not Hausdorff in any neighbourhood of $\io_\Kk(X)$ even though the map $s\times t:\Mor_{\bB_\Kk}\to \Obj_{\bB_\Kk}\times \Obj_{\bB_\Kk}$ is
proper; cf.\ Section~\ref{ss:top}.

\begin{example}[Failure of Hausdorff property]  \label{ex:Haus}\rm
We construct a Kuranishi atlas for $X := \R$
by starting with a basic chart whose footprint $F_1=\R$ already covers $X$,
$$
\bK_1 := \bigl(\, U_1=\R^2 \,,\, E_1=\R \,,\, s_1(x,y)=y \,,\, \psi_1(x,0)=x \,\bigr) .
$$
We then construct a second basic chart $\bK_2:=\bK_1|_{U_2}$ with footprint $F_2=(0,\infty)\subset\R$ and the transition chart $\bK_{12}:=\bK_1|_{U_{12}}$ as restrictions of $\bK_1$ to the domains
$$
U_2: = \{-y<x\le0\}\cup \{x>0\}
, \qquad
U_{12} := \{ x>0\} .
$$
This induces coordinate changes $\Hat\Phi_{i,12}:=\Hat\Phi_{1,1}|_{U_{i,12}} : \bK_i|_{U_{i,12}}\to\bK_{12}$ for $i=1,2$ given by restriction of the trivial coordinate change $\bigl(\phi_{1,1}={\rm id}_{\R^2} , \Hat\phi_{1,1}={\rm id}_\R\bigr)$ to $U_{i,12} := U_{12}$.
This defines a Kuranishi atlas since there are no compositions of coordinate changes for which a cocycle condition needs to be checked.
Moreover, $s\times t$ is proper because
on each of the finitely many connected components of $\Mor_{\bB_\Kk}$ the target map $t$ restricts to a homeomorphism to a connected component of $\Obj_{\bB_\Kk}$. (For example, $t:\Mor_{\bB_\Kk} \supset U_{i,12} \to U_{12} \subset \Obj_{\bB_\Kk}$ is the identity.)

On the other hand the images in $|\Kk|$ of the points $(0,y)\in U_1$ and $(0,y)\in U_2$
for $y>0$ have no disjoint neighbourhoods
since for every $x>0$
$$
\bigl( 1, (x,y)\bigr) \sim
\bigl( 12, (x,y)\bigr) \sim
\bigl( 2, (x,y)\bigr) .
$$
Therefore $\io_\Kk(X)$ does not have a Hausdorff neighbourhood in $|\Kk|$.
\end{example}

In Section~\ref{ss:shrink} below we will achieve both the injectivity and the Hausdorff property by a subtle shrinking of the domains of charts and coordinate changes. However, we are still unable to make the 
Kuranishi
neighbourhood $|\Kk|$ locally compact or even metrizable, due to the following natural example.

\begin{example}
[Failure of metrizability and local compactness]  \label{ex:Khomeo}
\rm
For simplicity we will give an example with noncompact $X = \R$. (A similar example can be constructed with $X = S^1$.)
We construct a Kuranishi atlas $\Kk$ on $X$ by two basic charts, 
$\bK_1 = (U_1=\R, E_1=\{0\}, s=0,\psi_1=\id)$ and
$$
\bK_2 = \bigl(U_2=(0,\infty)\times \R,\ E_2=\R, \ s_2(x,y)= y,\ \psi_2(x,y)= x\bigr),
$$
one transition chart $\bK_{12} = \bK_2|_{U_{12}}$ with domain $U_{12} := U_2$, and the coordinate changes $\Hat\Phi_{i,12}$ induced by the natural embeddings of the domains $U_{1,12} := (0,\infty)\hookrightarrow (0,\infty)\times\{0\}$ and $U_{2,12} := U_2\hookrightarrow U_2$.
Then as a set $|\Kk| = \bigl(U_1\sqcup U_2\sqcup U_{12}\bigr)/\sim$
can be identified with $\bigl(\R\times\{0\}\bigr) \cup \bigl( (0,\infty)\times\R\bigr) \subset \R^2$.
However, the quotient topology at $(0,0)\in|\Kk|$ is strictly stronger than the subspace topology.
That is, for any $O\subset\R^2$ open the induced subset $O\cap|\Kk|\subset|\Kk|$ is open, but some open subsets of $|\Kk|$ cannot be represented in this way.
In fact,  
for any $\eps>0$ and continuous function $f:(0,\eps)\to (0,\infty)$,
the set
$$
U_{f,\eps} \, :=\; \bigl\{ [x] \,\big|\, x\in U_1, |x|< \eps \}  \;\cup\; \bigl\{ [(x,y)] \,\big|\, (x,y)\in U_2,  |x|< \eps , |y|<f(x)\} \;\subset\; |\Kk|
$$
is open in the quotient topology.  Moreover these sets form a basis for the 
neighbourhoods of 
$[(0,0)]$ in the quotient topology. 
To see this, let $V\subset |\Kk|$ be open in the quotient topology. Then,
since $\pi_\Kk^{-1}(V)\cap U_1$ is a neighbourhood of $0$, there is 
$\eps>0$ so that 
$\{ (x,0) \,|\,  |x|<\eps \} \subset V$.  Further, define 
$f:\{x\in\R \,|\, 0<x<\eps \} \to (0,\infty)$ by
$f(x) := \sup \{\de \,|\, B_\de(x,0)\subset V\}$, where $B_\de(x,0)$ is the open ball in $\R^2$ with radius $\de$. Then 
$f(x)>0$ for all 
$0<x<\eps$ because $\pi_\Kk^{-1}(V)\cap U_2$ is a neighbourhood of $(0,x)$.
The triangle inequality implies that 
$f(x')\ge f(x) - |x'-x|$ for all  $0<x,x'<\eps$. Hence $|f(x)- f(x')|\le |x'-x|$, so that $f$ is continuous.  
Thus we have constructed a neighbourhood $U_{f,\eps}\subset |\Kk|$ of $[(0,0)]$ of the above type with $U_{f,\eps}\subset V$.

We will use this to see that the point $[(0,0)]$ does not have a countable neighbourhood basis in the quotient topology.
Indeed, suppose by contradiction that $(U_k)_{k\in\N}$ is such a basis, then by the above we can iteratively find $1>\eps_k>0$ and $f_k:(0,\eps_k)\to(0,\infty)$ so that $U_{f_k,\eps_k}\subset U_k\cap U_{f_{k-1},\frac 12 \eps_{k-1}}$ (with $U_{f_0,\frac 12 \eps_0}$ replaced by $|\Kk|$). In particular, the inclusion $U_{f_k,\eps_k}\subset U_{f_{k-1},\frac 12 \eps_{k-1}}$ implies $\eps_k < \eps_{k-1}$.
Now there exists a continuous function $g:(0,1)\to (0,\infty)$ such that $g(\frac 12\eps_k) < f_k(\frac 12 \eps_k)$ for all $k\in\N$. 
Then the neighbourhood $U_{g,1}$ does not contain any of the $U_k$ because $U_{g,1}\supset U_k \supset U_{f_k,\eps_k}$ implies that $g(\frac 12\eps_k) \geq f_k(\frac 12 \eps_k)$.
This contradicts the assumption that $(U_k)_{k\in\N}$ is a neighbourhood basis of $[(0,0)]$, hence there exists no countable neighbourhood basis.

Note also that  the point $[(0,0)]\in|\Kk|$ has no compact neighbourhood with respect to the subspace topology from $\R^2$, and hence neither  with respect to the stronger quotient topology on $|\Kk|$.
\end{example}

\begin{rmk}\label{rmk:Khomeo}\rm  
For the Kuranishi atlas in Example~\ref{ex:Khomeo} there exists an exhausting sequence $\ov{\Aa^n}\subset \ov{\Aa^{n+1}}$ of closed subsets of $\bigcup_{I\in \Ii_\Kk} U_I$ with the properties
\begin{itemize}
\item  
each $\pi_\Kk(\ov{\Aa^n})$ contains $\iota_\Kk(X)$;
\item  
each $\pi_\Kk(\ov{\Aa^n})\subset |\Kk|$ is metrizable and locally compact in the subspace topology;
\item 
$\bigcup_{n\in\N} \ov{\Aa^n} = \bigcup_{I\in \Ii_\Kk} U_I$.
\end{itemize}
For example, we can take $\ov{\Aa^n}$ to be the disjoint union of the closed sets 
$$
\ov{A_1^n}= [-n,n]\subset U_1, \qquad 
\ov{A_{2}^n} : = \{(x,y)\in U_2 \,\big|\,  x \geq \tfrac 1n, |y| \leq  n\},
$$
and any closed subset $\ov{A_{12}^n} \subset \ov{A_2^n}$.
However, in the limit $[(0,0)]$ becomes a ``bad point'' because its neighbourhoods have to involve open subsets of $U_2$.  

In fact, if we altered Example~\ref{ex:Khomeo} to a Kuranishi atlas for the compact space $X=S^1$, then we could choose $\ov{\Aa^n}$ compact, so that the subspace and quotient topologies on  $\pi_\Kk(\ov{\Aa^n})$
coincide by Proposition~\ref{prop:Ktopl1}~(ii). We emphasize the subspace topology above because that is the one inherited by (open) subsets of $\ov{\Aa^n}$.  For example, the quotient topology on $\pi_\Kk(\Aa^n)$, where $\Aa^n: = \bigcup_I {\rm int}(\ov{A_I^n})$ has the same bad properties at $[(\frac 1n,0)]$ as the quotient topology on $|\Kk|$ has at $[(0,0)]$, while the subspace topology on $\pi_\Kk(\Aa_n)$ is metrizable.
 We prove in Proposition~\ref{prop:Ktopl1} that a similar statement holds for all $\Kk$,
 though there we only consider a fixed set $\ov\Aa$ since we have no need for an exhaustion of the domains.
 \end{rmk}

We end by comparing our choice of definition with the notions of Kuranishi structures in the current literature.

\begin{rmk}\label{rmk:otherK}\rm
(i)
We defined the notion of a Kuranishi atlas so that it is relatively easy to construct from an equivariant Fredholm section. The only condition that is difficult to satisfy is the cocycle condition since that involves making compatible choices of all the domains $U_{IJ}$.
However, we show in Theorem~\ref{thm:K} that, provided the obstruction bundles satisfy
an additivity condition, one can always construct a Kuranishi atlas from a tuple of charts and coordinate changes that satisfy the weak cocycle condition in Definition~\ref{def:cocycle}, which is much easier to satisfy in practice.
The additivity condition is also naturally satisfied by the sum constructions for finite dimensional reductions of holomorphic curve moduli spaces in e.g.\ \cite{FO} and Section~\ref{s:construct}.
\MS

\NI (ii)
A Kuranishi structure in the sense of \cite{FO,J} is given in terms of germs of charts at every point of $X$ and some set of coordinate changes.
While this is a natural idea, we were not able to find a meaningful notion of compatible coordinate changes; see the discussion in Section~\ref{ss:alg}.
Recently, there seems to be a general understanding in the field that  explicit charts and coordinate changes are needed.
\MS

\NI (iii)
A Kuranishi structure in the sense of \cite[App.~A]{FOOO} consists of a Kuranishi chart $\bK_p$ at every point $p\in X$ and coordinate changes $\bK_q|_{U_{qp}}\to \bK_p$
whenever $q\in F_p$, and requires the weak cocycle condition.
The idea from \cite{FO} for constructing such a Kuranishi structure also starts with a finite covering family of basic charts $(\bK_i)$. Then the chart at $p$ is obtained by a sum construction from the charts $\bK_i$ with $p\in F_i$.
We outline in Section~\ref{s:construct} how  the analytic aspects of this sum construction can be made rigorous in the case of genus zero Gromov--Witten moduli spaces. Moreover, the construction of the sum charts and coordinate changes needs to be essentially canonical in order to achieve even the weak cocycle condition.

In the case of trivial isotropy, an abstract weak Kuranishi atlas in the sense of Definition~\ref{def:Ku2} induces a Kuranishi structure in the sense of \cite[App.~A]{FOOO} as follows.
Given a covering family of basic charts $(\bK_i)_{i=1,\dots,N}$ with footprints $F_i$ and transition data $(\bK_I,\Hat\Phi_{IJ})$, choose a family of compact subsets $C_i\subset F_i$ that also cover $X$.
Then for any $p\in X$ one obtains a Kuranishi chart $\bK_p$ by choosing a restriction of $\bK_{I_p}$ to $F_p:=\cap_{i\in I_p} F_i\less \cup_{i\notin I_p} C_i$, where $I_p: =  \{i \,|\, p\in C_i\}$.
This construction guarantees that for $q\in F_p$ we have $I_q\subset I_p$  and thus can restrict the coordinate change $\Hat\Phi_{I_q I_p}$ to a coordinate change from $\bK_q$ to $\bK_p$. 
The weak cocycle condition is preserved by these choices. 
Note however that neither this notion of a Kuranishi structure nor a weak Kuranishi atlas is sufficient for our approach to the construction of a VMC. 
(We start from a weak Kuranishi atlas with extra additivity condition as in (i). This allows us to achieve the strong cocycle condition and a tameness property by a shrinking procedure. The latter are crucial to achieve compactness and Hausdorff properties of perturbed solution spaces.)

Essentially, Fukaya et al use the same approach for constructing a Kuranishi structure.
However, instead of formulating the notion of a (weak) atlas, they first work on the level of the infinite family of charts $(\bK_p)_{p\in X}$ and only later rebuild a finite covering by ``orbifold charts'' to form a ``good coordinate system".   There is some justification for this approach when there is isotropy, since in this case the notion of a weak Kuranishi atlas, although very natural, involves some new ideas such as coverings involved in the coordinate changes, see~\cite{MW:ku2}. However, when there is no isotropy it seems cleaner to work directly with the charts in the finite covering family rather than passing to the uncountably many small charts $\bK_p$.   

\MS
\NI (iv)
Some recent work uses the notion of  ``good coordinate system" from \cite{FO,FOOO,J} instead of a Kuranishi structure, which is introduced there as intermediate step in the construction of a VMC. 
The  
early versions 
of this notion had serious problems since the proof of existence
(in \cite[Lemma 6.3]{FO}) is based on notions of germs and does not address 
the cocycle condition. 
Moreover, it required a totally ordered set of charts but does not clarify the relationship between order and overlaps.
In its most recent version in~\cite{FOOO12} (and in the case when there is no isotropy), a good coordinate system requires a finite cover of $X$ by a partially ordered set of charts $(\bK_I)_{I\in\Pp}$ and coordinate changes $\bK_I \to \bK_J$ for $I\leq J$, where the order is compatible with the overlaps of the footprints in the sense that $F_I\cap F_J\ne \emptyset$ implies $I\leq J$ or $J\leq I$.
Moreover, a complicated set of extra conditions must be satisfied  to ensure that the resulting quotient space is well behaved.  The arguments for the existence of a good coordinate system are also very complicated because they must deal with two problems at once: In the language used here, the resulting 
type of atlas is in particular required to be both tame and reduced.  

In our approach these questions are separated in order to clarify exactly what choices and constructions are  needed. Thus we first establish tameness and then tackle the problem of reduction.
In  the case of trivial isotropy, after constructing a Kuranishi atlas $\Kk$ with additivity and 
the strong cocycle condition, we then refine the cover to obtain data with the most important properties of  a ``good coordinate system" as in \cite{FOOO12}.
More precisely, we construct in Proposition~\ref{prop:red} a Kuranishi atlas $\Kk^\Vv$, with basic charts $\bK^\Vv_I$ for $I\in\Ii_\Kk$ given by restriction of the charts in $\Kk$ to precompact subsets of the domain, such that the overlaps of footprints are compatible with the partial ordering by the inclusion relation on $\Ii_\Kk$.
We will show that the realization $|\Kk^\Vv|$ injects continuously into $|\Kk|$ and inherits the Hausdorff property from $|\Kk|$ as well as the homeomorphism property of the natural maps $U^\Vv_C\to |\Kk^\Vv|$ from the domain of each Kuranishi chart in $\Kk^\Vv$.
Here the advantage of constructing $\Kk^\Vv$ via $\Kk$ is that $\Kk$ has fewer basic charts and coordinate changes, each with large domain, which makes it relatively easy to analyze the properties of its realization $|\Kk|$. On the other hand, $\Kk^\Vv$ has smaller domains but many more coordinate changes, which makes it hard to deduce properties such as Hausdorffness directly. 
However, good topological properties transfer from $\Kk$ to $\Kk^\Vv$ because its coordinate changes are given by restriction from $\Kk$, hence are not independent of each other.
In fact, it turns out to be easier and perhaps more natural to deal with an associated subcategory $\bB_\Kk|_\Vv$ of $\bB_\Kk$, rather than with the Kuranishi atlas $\Kk^\Vv$ itself; cf.\ Definition~\ref{def:vicin}.
\end{rmk}

\subsection{Additivity, Tameness and the Hausdorff property}
\label{ss:tame}   \hspace{1mm}\\ \vspace{-3mm}

We begin by introducing the notion of an additive weak Kuranishi atlas, which can be constructed in practice on compactified holomorphic curve moduli spaces, as outlined in Theorem~A and Section~\ref{s:construct}.
In contrast, we then introduce tameness conditions for Kuranishi atlases that imply the Hausdorff property of the Kuranishi neighbourhood. Finally, we provide tools for refining Kuranishi atlases
to achieve the tameness condition.

\begin{defn}\label{def:Kwk}
A {\bf weak Kuranishi atlas of dimension $\mathbf d$} is a covering family of basic charts of dimension $d$ with transition data $\Kk=\bigl(\bK_I,\Hat\Phi_{I J}\bigr)_{I, J\in\Ii_\Kk, I\subsetneq J}$ as in Definition \ref{def:Kfamily}, that satisfy the {\it weak cocycle condition} $\Hat\Phi_{J K}\circ \Hat\Phi_{I J} \approx \Hat\Phi_{I K}$ for every triple $I,J,K\in\Ii_K$ with $I\subsetneq J \subsetneq K$.
\end{defn}

This weaker notion of Kuranishi atlas is crucial for two reasons. Firstly, in the application to moduli spaces of holomorphic curves, it is not clear how to construct Kuranishi atlases that satisfy the cocycle condition.
Secondly, it is hard to preserve the cocycle condition while manipulating Kuranishi atlases, for example by shrinking as we do below.
Note that if $\Kk$ is only a weak Kuranishi atlas then we cannot define its domain category $\bB_{\Kk}$ precisely as in Definition~\ref{def:catKu} since the given set of morphisms is not closed under composition.  We will deal with this by simply not considering this category unless $\Kk$ is a Kuranishi atlas, i.e.\ satisfies the standard cocycle condition \eqref{eq:cocycle}.

On the other hand, the constructions of transition data in practice, e.g.\ in Section~\ref{s:construct}, use a sum construction for basic charts, which has the effect of adding the obstruction bundles, and thus yields the following additivity property. Here we simplify the notation by writing $\Hat\Phi_{i I}:= \Hat\Phi_{\{i\} I}$ for the coordinate change $\bK_i =\bK_{\{i\}} \to \bK_I$ where $i\in I$.

\begin{defn}\label{def:Ku2}  
Let $\Kk$ be a weak Kuranishi atlas.  We say that $\Kk$ is {\bf additive} if for each $I\in \Ii_\Kk$ the linear embeddings $\Hat\phi_{i I}:E_i \to E_I$ induce an isomorphism
$$
{\textstyle \prod_{i\in I}} \;\Hat\phi_{iI}: \; {\textstyle \prod_{i\in I}} \; E_i \;\stackrel{\cong}\longrightarrow \; E_I  ,
\qquad\text{or equivalently} \qquad
E_I = {\textstyle \bigoplus_{i\in I}} \; \Hat\phi_{iI}(E_i) .
$$
In this case we abbreviate notation by $s_J^{-1}(E_I): = s_J^{-1}\bigl(\Hat\phi_{IJ}(E_I)\bigr)$ and $s_J^{-1}(E_\emptyset) := s_J^{-1}(0)$.
\end{defn}

The additivity property is useful since it extends the automatic control of transition maps on the
zero sets $s_J^{-1}(0)$ to a weaker control on larger parts of the Kuranishi domains $U_J$ as follows.

\begin{lemma}
Let $\Kk$ be an additive weak Kuranishi atlas. Then for any $H,I,J \in \Ii_\Kk$ with $H,I\subset J$ we have
\begin{align} \label{eq:addd}
\Hat\phi_{IJ}(E_I)  \; \cap \; \Hat\phi_{HJ}(E_H) &\;=\;  \Hat\phi_{(I\cap H) J}(E_{(I\cap H) J}) , \\
\label{eq:CIJ}
s_J^{-1}(E_I) \;\cap\; s_J^{-1}(E_H) &\;=\;
s_J^{-1}(E_{I\cap H}) .
\end{align}
In particular, we deduce
\begin{equation}\label{eq:CIJ0}
H\cup I\subset J, \;
I\cap H=\emptyset \quad \Longrightarrow \quad
s_J^{-1}(E_I)\cap s_J^{-1}(E_H)  = s_J^{-1}(0) .
\end{equation}
\end{lemma}
\begin{proof}
Generally, for $H,I\subset J$ we have a direct sum ${\textstyle\bigoplus_{i\in I\cup H}}\,\Hat\phi_{iI}(E_i) \subset E_J$ and hence
\begin{align*}
\Hat\phi_{IJ}(E_I)\cap \Hat\phi_{HJ}(E_H)
&=
\Hat\phi_{IJ}\Bigl({\textstyle\bigoplus_{i\in I}}\,\Hat\phi_{iI}(E_i)\Bigr)\cap \Hat\phi_{HJ}\Bigl({\textstyle\bigoplus_{i\in H}}\,\Hat\phi_{iH}(E_i)\Bigr)  \\
&=  \Bigl({\textstyle\bigoplus_{i\in I}}\,\Hat\phi_{iJ}(E_i)\Bigr)\cap \Bigl({\textstyle\bigoplus_{i\in H}}\,\Hat\phi_{iJ}(E_i)\Bigr)  \\
&={\textstyle\bigoplus_{i\in I\cap H}}\,\Hat\phi_{iJ}(E_i)
= \Hat\phi_{(I\cap H) J} \Bigl( {\textstyle\bigoplus_{i\in I\cap H}}\,\Hat\phi_{i (I\cap H)}(E_i) \Bigr) \\
& = \Hat\phi_{(I\cap H) J} (E_{I\cap H})   .
\end{align*}
This proves \eqref{eq:addd}.
Applying $s_J^{-1}$ to both sides and recalling our abbreviations then implies \eqref{eq:CIJ}.
If moreover $I\cap H = \emptyset$ then $E_{I\cap H}=E_\emptyset = \{0\}$, which implies
\eqref{eq:CIJ0}.
\end{proof}

Before stating the main theorem, we introduce 
a notion of metrics on Kuranishi atlases
that will be useful in the construction of 
perturbations in Section~\ref{ss:const}.

\begin{defn}\label{def:metric}  
A Kuranishi atlas $\Kk$ is said to be {\bf metrizable} if there is a bounded metric $d$ on the set $|\Kk|$ such that for each $I\in \Ii_\Kk$ the pullback metric $d_I:=(\pi_\Kk|_{U_I})^*d$ on $U_I$ induces the given topology on the manifold $U_I$.
In this situation we call $d$ an {\bf admissible metric} on $|\Kk|$. 
A {\bf metric Kuranishi atlas} is a pair $(\Kk,d)$ consisting of a metrizable Kuranishi atlas together with a choice of  admissible metric $d$.
For a metric Kuranishi atlas, we denote the $\de$-neighbourhoods of subsets $Q\subset |\Kk|$ resp.\ $A\subset U_I$ for $\de>0$ by
\begin{align*}
B_\de(Q) &\,:=\; \bigl\{w\in |\Kk|\ | \ \exists q\in Q : d(w,q)<\de \bigr\}, \\
B^I_\de(A) &\,:=\; \bigl\{x\in U_I\ | \ \exists a\in A : d_I(x,a)<\de \bigr\}.
\end{align*}
\end{defn}

We next show that if $d$ is an admissible metric on $|\Kk|$, then the metric topology on $|\Kk|$ is weaker (has fewer open sets) than the quotient topology, which generally is not metrizable by Example~\ref{ex:Khomeo}.

\begin{lemma}\label{le:metric}  
Suppose that $d$ is an admissible metric on the virtual neighbourhood $|\Kk|$ of a Kuranishi atlas $\Kk$.
Then the following holds.
\begin{enumerate}
\item
The identity $\id_{|\Kk|} :|\Kk| \to (|\Kk|,d)$ is continuous as a map from the quotient topology to the metric topology on $|\Kk|$.
\item
In particular, each set $B_\de(Q)$ is open in the quotient topology on $|\Kk|$, so that the 
existence of an admissible metric implies that
$|\Kk|$ is Hausdorff.
\item 
The embeddings $\phi_{IJ}$ that are part of the coordinate changes 
for $I\subsetneq J\in\Ii_\Kk$
are isometries when considered as maps $(U_{IJ},d_I)\to (U_J,d_J)$.
\end{enumerate}
\end{lemma}

\begin{proof} 
Since the neighbourhoods of the form $B_\de(Q)$ define the metric topology, it suffices to prove that these are also open in the quotient topology, i.e.\ that each subset $U_I\cap \pi_\Kk^{-1}(B_\de(Q))$ is open in $U_I$. 
So consider $x\in U_I$ with $\pi_\Kk(x)\in B_\de(Q)$. By hypothesis there is $q\in Q$ and $\eps>0$ such that $d(\pi_\Kk(x),q)<\de-\eps$, and compatibility of metrics and the triangle inequality then imply the inclusion $\pi_\Kk(B^I_\eps(x))\subset B_\eps(Q)\subset B_\de(Q)$.
Thus $B^I_\eps(x)$ is a neighbourhood of $x\in U_I$ contained in $U_I\cap \pi_\Kk^{-1}(B_\de(Q))$. 
This proves the openness required for (i) 
and (ii).
Since every metric space is Hausdorff, $|\Kk|$ is therefore Hausdorff in the quotient topology as stated in (ii).   Claim (iii) is immediate
from the construction.
\end{proof}

One might hope to achieve the Hausdorff property by constructing an admissible metric, but the existence of the latter is highly nontrivial. Instead, in a refinement process that will take up the next two sections, we will first construct a Kuranishi atlas whose virtual neighbourhood has the Hausdorff property, then prove metrizability of certain subspaces, and finally obtain an admissible metric by pullback to a further refined Kuranishi atlas. 
This process will prove the following theorem whose formulation uses the notions of shrinking from Definition~\ref{def:shr}, tameness from Definition~\ref{def:tame}, and cobordism from Definition~\ref{def:Kcobord}.
The formulation below is somewhat informal; more precise statements may be found in the results quoted in its proof.

\begin{thm}\label{thm:K}
Let $\Kk$ be an additive weak Kuranishi atlas on a compact metrizable space $X$.
Then an appropriate shrinking of $\Kk$ provides a  metrizable tame Kuranishi atlas $\Kk'$ with domains 
$(U'_I\subset U_I)_{I\in\Ii_{\Kk'}}$ such that the realizations $|\Kk'|$ and $|\bE_{\Kk'}|$ are Hausdorff in the quotient topology.
In addition, for each $I\in \Ii_{\Kk'} = \Ii_\Kk$ the projection maps $\pi_{\Kk'}: U_I'\to |\Kk'|$ and
$\pi_{\Kk'}:U'_I\times E_I\to |\bE_{\Kk'}|$ are homeomorphisms onto their images and fit into a commutative diagram
$$
\begin{array}{ccc}
U_I'\times E_I & \stackrel{\pi_{\Kk'}}\longhookrightarrow & |\bE_{\Kk'}|  \quad \\
 \downarrow & & \;\; \downarrow \scriptstyle |\pr_{\Kk'}| \\
U_I' &
\stackrel{\pi_{\Kk'}} \longhookrightarrow  &|\Kk'| \quad
\end{array}
$$
where the horizontal maps intertwine the vector space structure on $E_I$ with a vector space structure on the fibers of $|\pr_{\Kk'}|$.

Moreover, any two such shrinkings are cobordant by a metrizable tame Kuranishi cobordism whose realization also has the above Hausdorff, homeomorphism, and linearity properties.
\end{thm}

\begin{proof} 
The key step is Proposition~\ref{prop:proper}, which establishes the existence of a tame shrinking. As we show in   Proposition~\ref{prop:metric}, the existence of a metric tame shrinking is an easy consequence. 
Uniqueness up to metrizable tame cobordism is proven in Proposition~\ref{prop:cobord2}. 
By Proposition~\ref{prop:Khomeo}, tameness implies the Hausdorff and homeomorphism properties.
The diagram commutes since it arises as the realization of commuting functors to $\pr_{\Kk'}:\bE_{\Kk'}\to\bB_{\Kk'}$, and we prove the linearity property in Proposition~\ref{prop:linear}.
Finally, the last statement follows from Lemma~\ref{le:cob0} where we show that the realization of every tame Kuranishi cobordism has the Hausdorff, homeomorphism, and linearity properties.
\end{proof}

The Hausdorff property for the 
Kuranishi neighbourhood $|\Kk|$ will require the following control of the domains of coordinate changes, which we will achieve in Section~\ref{ss:shrink} by a shrinking from an additive weak Kuranishi atlas.

\begin{defn}\label{def:tame}
A weak Kuranishi atlas is {\bf tame} if it is additive, and for all $I,J,K\in\Ii_\Kk$ we have
\begin{align}\label{eq:tame1}
U_{IJ}\cap U_{IK}&\;=\; U_{I (J\cup K)}\qquad\qquad\;\;\;\;\,\qquad\forall I\subset J,K ;\\
\label{eq:tame2}
\phi_{IJ}(U_{IK}) &\;=\; U_{JK}\cap s_J^{-1}\bigl(\Hat\phi_{IJ}(E_I)\bigr) \qquad\forall I\subset J\subset K.
\end{align}
Here we allow equalities, using the notation $U_{II}:=U_I$ and $\phi_{II}:={\rm Id}_{U_I}$.
Further, to allow for the possibility that $J\cup K\notin\Ii_\Kk$, we define
$U_{IL}:=\emptyset$ for $L\subset \{1,\ldots,N\}$ with $L\notin \Ii_\Kk$.
Therefore \eqref{eq:tame1} includes the condition
$$
U_{IJ}\cap U_{IK}\ne \emptyset
\quad \Longrightarrow \quad F_J\cap F_K \ne \emptyset  \qquad \bigl( \quad \Longleftrightarrow\quad
J\cup K\in \Ii_\Kk \quad\bigr).
$$
\end{defn}

The first tameness condition \eqref{eq:tame1} extends the identity for footprints $\psi_I^{-1}(F_J)\cap \psi_I^{-1}(F_K) = \psi_I^{-1}(F_{J\cup K})$ to the domains of the transition maps in $U_I$. In particular, with $J\subset K$ it implies nesting of the domains of the transition maps,
\begin{equation}\label{eq:tame4}
U_{IK}\subset U_{IJ} \qquad\forall I\subset J \subset K.
\end{equation}
The second tameness condition \eqref{eq:tame2} extends the control of transition maps between footprints and zero sets
$\phi_{IJ}(\psi_I^{-1}(F_K)) = \psi_J^{-1}(F_K) = U_{JK}\cap s_J^{-1}(0) $ to the Kuranishi domains.
In particular, with $J=K$ it controls the image of the transition maps,
\begin{equation}\label{eq:tame3}
\im\phi_{IJ}:= \phi_{IJ}(U_{IJ}) =  s_J^{-1}(\Hat\phi_{IJ}(E_I)) \qquad\forall I\subset J.
\end{equation}
This implies that the image of $\phi_{IJ}$ is a closed subset of $U_J$,
and is a much strengthened form of the  "infinitesimal tameness"
 $ \im\rd_y\phi_{IJ}=(\rd s_J)^{-1}\bigr(\Hat\phi_{IJ}(E_I)\bigl)$
provided at the points $y\in \im\phi_{IJ}$ by Definition \ref{def:change}.
The next lemma shows, somewhat generalized, that every tame weak Kuranishi atlas
in fact satisfies the strong cocycle condition, so in particular is a Kuranishi atlas.

\begin{lemma}\label{le:tame0}
Suppose that the weak Kuranishi atlas $\Kk$ satisfies the tameness conditions \eqref{eq:tame1}, \eqref{eq:tame2} for all $I,J,K\in\Ii_\Kk$ with $|I|\leq k$. Then for all $I\subset J\subset K$ with $|I|\leq k$ the strong cocycle condition \eqref{strong cocycle} is satisfied, i.e.\ $\phi_{JK}\circ \phi_{IJ}=\phi_{IK}$ with equality of domains
$$
U_{IJ}\cap \phi_{IJ}^{-1}(U_{JK}) = U_{IK} .
$$
\end{lemma}
\begin{proof}
From the tameness conditions \eqref{eq:tame1} and \eqref{eq:tame3} we obtain for all $I\subset J\subset K$ with $|I|\leq k$
$$
\phi_{IJ}(U_{IK})
= U_{JK}\cap s_J^{-1}(\Hat\phi_{IJ}(E_I)) 
= U_{JK}\cap \phi_{IJ}(U_{IJ}) .
$$
Applying $\phi_{IJ}^{-1}$ to both sides and using \eqref{eq:tame4} implies equality of the domains.
Then the weak cocycle condition $\phi_{JK}\circ \phi_{IJ}=\phi_{IK}$ on the overlap of domains is identical to the strong cocycle condition.
\end{proof}

The above remarks do not use additivity.  However, we formulated Definition~\ref{def:tame} so that tameness implies additivity, because the most useful consequences come by using additivity.  
In particular, we obtain a limited transversality for the embeddings of the domains involved in coordinate changes.
This property is crucial to guarantee the existence of coherent (i.e.\ compatible with coordinate changes) perturbations of the canonical section $s_\Kk$ of a tame Kuranishi atlas.
However, due to further technical complications, we will not use it directly in the constructions of Section~\ref{ss:const}.

\begin{lemma}\label{le:phitrans} 
If $\Kk$ is a tame Kuranishi atlas, then for any  $H,I,J\in\Ii_\Kk$ with $H \cap I\ne \emptyset$ and $H\cup I\subset J$ the two submanifolds $\im \phi_{H J}$ and $\im \phi_{IJ}$ of  $\im \phi_{(H\cup I) J}$ intersect transversally in $\im \phi_{(H\cap I)J}$.
\end{lemma}
\begin{proof}  
We will make crucial use of tameness, which identifies 
$$
\im \phi_{L J}=s_J^{-1}(E_L):=s_J^{-1}\bigl(\Hat\phi_{LJ}(E_L)\bigr)
$$
for $L=H,I,H\cap I, H\cup I$, with the preimages under $s_J$ of the images of the linear embeddings, $\im\Hat\phi_{LJ}= \Hat\phi_{LJ}(E_L) \subset E_J$.
The inclusions $s_J^{-1}\bigl(\im\Hat\phi_{LJ}\bigr) \subset s_J^{-1}\bigl(\im\Hat\phi_{(H\cup I)J}\bigr)$ for $L=H,I$ then follow from the linear cocycle condition $\Hat\phi_{LJ} = \Hat\phi_{(H\cup I)J} \circ \Hat\phi_{L(H\cup I)}$, which implies $\im\Hat\phi_{LJ} \subset \im\Hat\phi_{(H\cup I)J}$. 
The intersection identity $s_J^{-1}\bigl(\im\Hat\phi_{HJ}\bigr)\cap s_J^{-1}\bigl(\im\Hat\phi_{IJ}\bigr) = s_J^{-1}\bigl(\im\Hat\phi_{(H\cap I)J}\bigr)$ 
follows by applying $s_J^{-1}$ to the additivity property \eqref{eq:addd}.

To prove the transversality of intersection, we use additivity and the linear cocycle condition to obtain the decomposition
\begin{align*}
\im\Hat\phi_{(H\cup I)J} 
&\;\cong\;  
\quotient{\im\Hat\phi_{HJ}}{\im\Hat\phi_{(H\cap I)J}} \;\oplus\;
\im\Hat\phi_{(H\cap I)J} \;\oplus\;\quotient{\im\Hat\phi_{IJ}}{\im\Hat\phi_{(H\cap I)J}} .
\end{align*}
Applying the isomorphism $\rd s_J^{-1}$ to a complement of $\ker\rd s_J \subset \rT U_J$, and adding this kernel to the middle factor, this implies 
\begin{align*}
\rd s_J^{-1}\bigl(\im\Hat\phi_{(H\cup I)J}\bigr)
&\;\cong\;  
\frac{\rd s_J^{-1}\bigl(\im\Hat\phi_{HJ}\bigr)}{\rd s_J^{-1}\bigl(\im\Hat\phi_{(H\cap I)J}\bigr)} \;\oplus\;
\rd s_J^{-1}\bigl(\im\Hat\phi_{(H\cap I)J}\bigr) \;\oplus\;\frac{\rd s_J^{-1}\bigl(\im\Hat\phi_{IJ}\bigr)}{\rd s_J^{-1}\bigl(\im\Hat\phi_{(H\cap I)J}\bigr)} .
\end{align*}
Since $\im\Hat\phi_{(H\cap I)J} \subset \im\Hat\phi_{LJ}$ for $L=H,I$ this implies transversality
\begin{align*}
\rd s_J^{-1}\bigl(\im\Hat\phi_{(H\cup I)J}\bigr)
\;=\; 
\rd s_J^{-1}\bigl(\im\Hat\phi_{HJ}\bigr) \;+\; \rd s_J^{-1}\bigl(\im\Hat\phi_{IJ}\bigr)
\end{align*}
as claimed.
\end{proof}

We now show that the additivity and tameness conditions give us very useful control over the equivalence relation $\sim$ on $\Obj_{\bB_\Kk}$, given in Definition~\ref{def:Knbhd} by abstractly inverting the morphisms.
We reformulate it here with the help of a partial order given by the morphisms -- more precisely the embeddings $\phi_{IJ}$ that are part of the coordinate changes.

\begin{definition} \label{def:preceq}
Let $\preceq$ denote the partial order on $\Obj_{\bB_\Kk}$ given by
$$
(I,x)\preceq(J,y) \quad :\Longleftrightarrow \quad \Mor_{\bB_\Kk}((I,x),(J,y))\neq\emptyset .
$$
That is, we have $(I,x)\preceq (I,y)$ iff $x\in U_{IJ}$ and $y=\phi_{IJ}(x)$.
Moreover, for any $I,J\in\Ii_\Kk$ and a subset $S_I\subset U_I$ we denote the subset of points in $U_J$ that are equivalent to a point in $S_I$ by
$$
\eps_J(S_I) \,:=\; \pi_\Kk^{-1}(\pi_\Kk(S_I)) \cap U_J \;=\;
\bigl\{y\in U_J  \,\big|\, \exists\, x\in S_I : (I,x)\sim (J,y) \bigr\} \;\subset\; U_J.
$$
There is a similar partial order  on $\Obj_{\bE_\Kk}$ given by
$$
(I,x,e)\preceq(J,y,f) \quad :\Longleftrightarrow \quad \Mor_{\bE_\Kk}((I,x,e),(J,y,f))\neq\emptyset .
$$
\end{definition}

The notation $\eps_J(S_I)$ will obtain a more useful interpretation in Lemma \ref{le:tame0} below.
The relation $\preceq$ on $\Obj_{\bE_\Kk}$ is very similar to that on $\Obj_{\bB_\Kk}$.  Indeed,
$(I,x,e)\preceq (J,y,f)$ implies  $(I,x)\preceq (J,y)$. Conversely, if $(I,x)\preceq (J,y)$ then for every $e\in E_I$ there is a unique $f\in E_J$ such that $(I,x,e)\preceq (J,y,f)$.  Thus to ease notation we mostly work with the relation on $\Obj_{\bB_\Kk}$ though any statement about this has an immediate analog for the relation on $\Obj_{\bE_\Kk}$ (and vice versa).

\begin{lemma} \label{lem:eqdef}
The equivalence relation $\sim$ on $\Obj_{\Bb_\Kk}$ of Definition~\ref{def:Knbhd} is equivalently defined
by $(I,x)\sim (J,y)$ iff there is a finite tuple of objects $(I_0, x_0), \ldots, (I_k, x_k)\in\Obj_{\bB_\Kk}$ such that
\begin{align}\label{eq:ch}
&(I,x) = (I_0,x_0)\preceq  (I_1,x_1) \succeq (I_2,x_2) \preceq  \ldots (I_k, x_k)=(J,y)  \\
\text{or}\qquad &
(I,x) = (I_0,x_0)\succeq  (I_1,x_1) \preceq (I_2,x_2) \succeq  \ldots (I_k, x_k)=(J,y) . \nonumber
\end{align}
\end{lemma}
\begin{proof}
The relation $\preceq$ is transitive by the cocycle condition \ref{def:Ku}~(d) and antisymmetric since the transition maps are directed. In particular, we have $(I,x)\preceq (I,y)$ iff $x=y$.
The two definitions of $\sim$ are equivalent since, if \eqref{eq:ch} had consecutive morphisms $(I_{\ell-1},x_{\ell-1})\preceq (I_{\ell},x_{\ell}) \preceq (I_{\ell+1},x_{\ell+1})$, these could be composed to a single morphism $(I_{\ell-1},x_{\ell-1})\preceq (I_{\ell+1},x_{\ell+1})$ by the cocycle condition.  Similarly, any consecutive morphisms $(I_{\ell-1},x_{\ell-1})\succeq (I_{\ell},x_{\ell}) \succeq (I_{\ell+1},x_{\ell+1})$ can be composed to a single morphism $(I_{\ell-1},x_{\ell-1})\succeq (I_{\ell+1},x_{\ell+1})$.
\end{proof}

When $\Kk$ is tame, the definition of $\sim$ in terms of $\preceq$ can be simplified, and thus has good topological properties, as follows.  Note that, because tame $\Kk$ are additive, we now denote $s_J^{-1}(E_I):=s_J^{-1}(\Hat\phi_{IJ}(E_I))\subset U_J$.

\begin{lemma} \label{le:Ku2} 
Let $\Kk$ be a tame Kuranishi atlas.
\begin{enumerate}
\item [(a)] 
For $(I,x),(J,y)\in\Obj_{\bB_\Kk}$ the following are equivalent.
\begin{enumerate}
\item[(i)] $(I,x)\sim (J,y)$;
\item[(ii)] there exists $z\in U_{I\cup J}$ such that $(I,x)\preceq (I\cup J,z) \succeq (J,y)$;
\item[(iii)]  there exists $w\in U_{I\cap J}$ such that $(I,x)\succeq (I\cap J,w) \preceq (J,y)$. $\phantom{\int_{x_x}}$
\end{enumerate}

\item[(b)]
For $(I,x,e),(J,y,f)\in\Obj_{\bE_\Kk}$ the following are equivalent.
\begin{enumerate}
\item[(i)] $(I,x,e)\sim (J,y,f)$;
\item[(ii)] $(I,x)\sim(J,y)$ and $\Hat\phi_{I (I\cup J)}(e) = g = \Hat\phi_{J (I\cup J)}(f)$ for some $g\in E_{I\cup J}$;
\item[(iii)]
$(I,x)\sim(J,y)$ and $\Hat\phi^{-1}_{(I\cup J)I}(e) = d = \Hat\phi_{(I\cap J) J}^{-1}(f)$ for some $d\in E_{I\cap J}$.
\end{enumerate}
\item[(c)]
$\pi_\Kk:U_I \to |\Kk|$ and $\pi_\Kk: U_I\times E_I \to |\Ee_\Kk|$ are injective for each $I\in\Ii_\Kk$, that is $(I,x,e)\sim (I,y,f)$ implies $x=y$ and $e=f$.
In particular, the elements $z$ and $w$ in (a) resp.\ $g$ and $d$ in (b) are automatically unique.
\item[(d)]
For any $I,J\in\Ii_\Kk$ and $S_I\subset U_I$ we have
$$
\eps_J(S_I) \;=\; \phi_{J(I\cup J)}^{-1}\bigl( \phi_{I (I\cup J)}(S_I) \bigr)
\;=\; \phi_{(I\cap J) J}\bigl( \phi_{(I\cap J)I}^{-1}(S_I) \bigr) ;
$$
in particular
$$
\eps_J(U_I) \;:=\; U_J\cap \pi_\Kk^{-1}\bigl(\pi_\Kk(U_I)\bigr) \;=\; 
 U_{J(I\cup J)}\cap s_J^{-1}(E_{I\cap J}).
$$
\item[(e)]
If $S_I\subset U_I$ is closed, then $\eps_J(S_I) \subset U_J$ is closed.
In particular we have $\ov{\eps_J(A_I)} \subset \eps_J(\ov{A_I})$ for any subset $A_I\subset U_I$.
\end{enumerate}
\end{lemma}
\begin{proof}
We first prove the following intermediate results:

\medskip
\noindent
{\bf Claim 1:} {\it Suppose $(I,x,e)\preceq (K,z,g) \succeq (J,y,f)$ for some $(K,z,g)\in\Obj_{\bE_\Kk}$, then there exists $w\in U_{I\cap J}$ and $d\in E_{I\cap J}$ such that $(I,x,e)\succeq (I\cap J,w,d) \preceq (J,y,f)$.}
\MS

\NI
Indeed, tameness \eqref{eq:tame3} and additivity \eqref{eq:CIJ} imply
$$
z \in \phi_{IK}(U_{IK})\cap \phi_{JK}(U_{JK})
= s_K^{-1}(E_I)\cap s_K^{-1}(E_J)
 = s_K^{-1}(E_{I\cap J})
 =  \phi_{(I\cap J) K}(U_{(I\cap J)K}).
$$
Therefore we have $z=\phi_{(I\cap J)K}(w)$ for some $w\in U_{(I\cap J)K}$.
We also have $z=\phi_{IK}(x)$ by assumption, and Lemma \ref{le:tame0} implies that $\phi_{(I\cap J)K}(w) = \phi_{IK}\bigl( \phi_{(I\cap J)I}(w)\bigr)$, so the elements $x$ and $\phi_{(I\cap J)I}(w)$ of $U_{IK}$ have the same image under $\phi_{IK}$. Since the latter is an embedding we deduce $x=\phi_{(I\cap J) I}(w)$.
Similarly, $y=\phi_{(I\cap J) J}(w)$ follows from $\phi_{(I\cap J)K} = \phi_{JK}\circ \phi_{(I\cap J)J}$.
Moreover, we have $\Hat\phi_{IK}(e) = g = \Hat\phi_{JK}(f)$ by assumption, so $g\in (\Hat\phi_{IK}(E_I)) \cap (\Hat\phi_{JK}(E_J)) =  \Hat\phi_{(I\cap J) K}(E_{(I\cap J) K})$ by additivity \eqref{eq:addd}. Now applying $\Hat\phi_{(I\cap J) K}^{-1}$ together with the cocycle conditions  $\Hat\phi_{(I\cap J) K}=  \Hat\phi_{\bullet K} \circ \Hat\phi_{(I\cap J) \bullet}$ for $\bullet = I, J$ we obtain $\Hat\phi_{(I\cap J)I}^{-1}(e) = \Hat\phi_{(I\cap J) J}^{-1}(f) = \Hat\phi_{(I\cap J) K}^{-1}(g) =:d$.

\medskip

\noindent
{\bf Claim 2:} {\it Suppose $(I,x,e)\succeq (H,w,d) \preceq (J,y,f)$ for some $(H,w,d)\in\Obj_{\bE_\Kk}$, then there exists $z\in U_{I\cup J}$ and $g\in E_{I\cup J}$ such that $(I,x,e)\preceq (I\cup J,z,g) \succeq (J,y,f)$.}

\MS\NI
Indeed, \eqref{eq:tame1} implies that $w\in U_{H (I\cup J)}$ so that $z:=\phi_{H (I\cup J)}w\in U_{I\cup J}$ is defined.
We also have $x = \phi_{HI}(w)$ by assumption, which together with Lemma \ref{le:tame0} implies
$$
z = \phi_{H (I\cup J)}(w) = \phi_{(I\cup J) I}\bigl(\phi_{H I}(w)\bigr) =  \phi_{(I\cup J) I}(x) .
$$
Similarly, $z= \phi_{(I\cup J) J}(y)$ follows from $y=\phi_{HJ}(w)$ and the strong cocycle condition
\eqref{strong cocycle}, proved in Lemma~\ref{le:tame0}.
Moreover, we have $e=\Hat\phi_{HI}(d)$ and $f=\Hat\phi_{HJ}(d)$, so can apply the cocycle conditions $\Hat\phi_{H(I\cup J)}=  \Hat\phi_{\bullet (I\cup J)}\circ \Hat\phi_{H \bullet}$ for $\bullet = I, J$ to obtain $$
{\Hat\phi_{I(I\cup J)}(e) = \Hat\phi_{J(I\cup J)}(f) = \Hat\phi_{H (I\cup J)}(d) =:g}.
$$

\medskip

Now to prove part (a), observe that Claims 1 and 2 imply the analogous statements on $\Obj_{\bB_\Kk}$ by picking the zero vector for each $e,f,g,d$.
With that, (ii) $\Rightarrow $ (iii)  follows from Claim 1, and (iii) $\Rightarrow$ (i) holds by definition of the equivalence relation $\sim$.
The implication (i) $\Rightarrow$  (ii) is proven by noting as above that consecutive morphisms in \eqref{eq:ch} in the same direction can be composed to a single morphism by the cocycle condition. Combining this with Claim 1 and Claim 2 above, any tuple of morphisms \eqref{eq:ch} can be replaced by two morphisms $(I,x)\succeq (H,w) \preceq (J,y)$ or $(I,x)\preceq (K,z) \succeq (J,y)$. We then use once more Claim 2 or Claims 1 and 2 to deduce the existence of morphisms $(I,x)\preceq (I\cup J,z) \succeq (J,y)$.
This proves (a), and (b) is proven in complete analogy.

Next, part (c) is a consequence of (i)$\Rightarrow$(ii) since $\phi_{II}={\rm Id}_{U_I}$ and 
$\Hat\phi_{II}={\rm Id}_{E_I}$.  
The formulas for $\eps_J(S_I)$ in (d) follow immediately from the equivalent definitions of $\sim$ in~(a). 
In case $S_I=U_I$ we can moreover use \eqref{eq:tame3} and \eqref{eq:CIJ} to obtain
\begin{align*}
\eps_J(U_I) &
\;=\;  \phi_{J(I\cup J)}^{-1}\bigl( s_{I\cup J}^{-1}(E_I) \bigr) 
\;=\;  \phi_{J(I\cup J)}^{-1}\bigl( s_{I\cup J}^{-1}(E_I) \cap  s_{I\cup J}^{-1}(E_J) \bigr) \\
&\;=\;  \phi_{J(I\cup J)}^{-1}\bigl( s_{I\cup J}^{-1}(E_{I\cap J}) \bigr) 
\;=\;  U_{J(I\cup J)} \cap s_J^{-1}(E_{I\cap J}) \bigr) .
\end{align*}
If in addition $S_I\subset U_I$ is closed, then $\eps_J(S_I) = \phi_{(I\cap J) J}\bigl( \phi_{(I\cap J)I}^{-1}(S_I) \bigr) \subset \im\phi_{(I\cap J) J}$ is closed since the transition maps are homeomorphisms to their images. The tameness assumption \eqref{eq:tame3} ensures that $\im\phi_{(I\cap J) J}\subset U_J$ is closed, and hence $\eps_J(S_I)\subset U_J$ is closed.
Now for any subset $A_I\subset U_I$ we have $\eps_J(A_I)$ contained in the closed subset $\eps_J(\ov{A_I})\subset U_J$, hence by definition the closure $\ov{\eps_J(A_I)} \subset U_J$ is contained in $\eps_J(\ov{A_I})$. This finishes the proof of (e).
\end{proof}

With these preparations we  show in the following propositions that the additivity and tameness conditions imply the Hausdorff, homeomorphism, and linearity properties claimed in Theorem~\ref{thm:K}.

\begin{prop}\label{prop:Khomeo}
Suppose that the Kuranishi atlas $\Kk$ is tame. Then $|\Kk|$ and  $|\bE_\Kk|$ are Hausdorff, and for each $I\in\Ii_\Kk$ the quotient maps $\pi_{\Kk}|_{U_I}:U_I\to |\Kk|$ and $\pi_{\Kk}|_{U_I\times E_I}:U_I\times E_I\to |\bE_\Kk|$ are homeomorphisms onto their image.
\end{prop}

\begin{proof}
The claims for $\Kk$ follow from that for $ |\bE_\Kk|$ since $|\Kk|$ can be identified with
the zero section of the bundle $\pr:  |\bE_\Kk| \to |\Kk|$, which is a closed subset.
However, to avoid carrying along unnecessary notation, we first prove the statement for $|\Kk|$,
and then sketch the necessary extensions of the argument for $|\bE_\Kk|$.

Since each component of $\Obj_{\bB_\Kk}=\bigcup_{I\in \Ii_\Kk} U_I$ is Hausdorff and locally compact,
its quotient $|\Kk|=\Obj_{\bB_\Kk}/\hspace{-1.5mm}\sim$ is Hausdorff exactly if the equivalence relation $\sim$ is closed. Since $\Ii_\Kk$ is finite, it suffices to consider sequences $U_I \ni x^\nu\to x^\infty$ and
$U_J\ni y^\nu\to y^\infty$ of equivalent objects $(I,x^\nu)\sim (J, y^\nu)$ for all $\nu\in\N$ and check that
 $(I, x^\infty)\sim (J, y^\infty)$.
For that purpose denote $H:=I\cap J$, then
by Lemma~\ref{le:Ku2}(a) there is a sequence $w^\nu\in U_H$ such that $x^\nu = \phi_{HI}(w^\nu)$ and $y^\nu = \phi_{HJ}(w^\nu)$.
Now it follows from the tameness condition \eqref{eq:tame3} that $x^\infty$ lies in the
relatively closed subset $\phi_{HI}(U_{HI})=s_I^{-1}(E_H)\subset U_I$, and since $\phi_{HI}$ is a homeomorphism to its image we deduce convergence $w^\nu\to w^\infty\in U_{HI}$ to a preimage of $x^\infty=\phi_{HI}(w^\infty)$.  Then by continuity of the transition map we obtain $\phi_{HJ}(w^\infty) = y^\infty$, so that $(I, x^\infty)\sim (J, y^\infty)$ as claimed. Thus $|\Kk|$ is Hausdorff.

To show that $\pi_\Kk|_{U_I}$ is a homeomorphism onto its image, first recall that it is injective by Lemma \ref{le:Ku2}~(c).
It is moreover continuous since $|\Kk|$ is equipped with the quotient topology. Hence it remains to show that $\pi_\Kk|_{U_I}$ is an open map to its image, i.e.\ for any given open subset $S_I\subset U_I$ we must find an open subset $\Ww\subset |\Kk|$ such that $\Ww\cap \pi_{\Kk}(U_I) = \pi_\Kk(S_I)$.
Since $|\Kk|$ is given the quotient topology, $\Ww$ is open precisely if
$\pi_{\Kk}^{-1}(\Ww)\cap U_J \subset U_J$  is open for each $J\in\Ii_\Kk$.
Equivalently, we could find open subsets $W_J\subset U_J$ such that $W_I=S_I$ and
$$
\pi_{\Kk}^{-1} \Bigl( \pi_\Kk \Bigl( {\textstyle \bigcup_{J\in\Ii_\Kk} W_J } \Bigr) \Bigr)  = {\textstyle \bigcup_{J\in\Ii_\Kk} W_J }  ,
$$
since then $\Ww:=\pi_\Kk\bigl( \bigcup_{J\in\Ii_\Kk} W_J \bigr)\subset |\Kk|$ is the required open set.
To construct these sets we order $\Ii_\Kk=\{I_1,I_2,\ldots,I_N\}$ in any way such that $I_1=I$. Then we will iteratively define open sets $W_{I_\ell}\subset U_{I_\ell}$ and not necessarily open sets
$$
\Ww_\ell:=  \pi_\Kk \bigl( {\textstyle \bigcup_{k=1}^\ell } W_{I_k}\bigr)\subset|\Kk|
$$
such that
\begin{equation} \label{wwell}
\pi_{\Kk}^{-1}(\Ww_\ell)\cap U_{I_k} = W_{I_k}  \quad \forall \;k=1,\ldots,\ell .
\end{equation}
Obviously, we also need to begin with $W_{I_1}:=S_I$, which is given as open.
Then $\Ww_1=\pi_\Kk(W_{I_1})$ satisfies \eqref{wwell} for $\ell=1$ as we check with the help of Lemma~\ref{le:Ku2}~(c),
$$
\pi_{\Kk}^{-1}(\Ww_1)\cap U_{I_1}
= \pi_{\Kk}^{-1}(\pi_\Kk (W_{I_1}))\cap U_{I_1}
= \eps_{I_1}(W_{I_1})
= W_{I_1} .
$$
Next, suppose that the open sets $W_{I_1},\ldots, W_{I_\ell}$ are defined and satisfy \eqref{wwell}.  Then we define
$$
W_{I_{\ell+1}}: = U_{I_{\ell+1}}\;\less\; {\textstyle \bigcup_{n\leq\ell}  \eps_{I_{\ell+1}}(U_{I_n}\less W_{I_n}) }
$$
by removing those points from the domain $U_{I_{\ell+1}}$ that would violate \eqref{wwell}.
This defines an open subset $W_{I_{\ell+1}}\subset U_{I_{\ell+1}}$ since by previous constructions $W_{I_n}$ is open, so $U_{I_n}\less W_{I_n}\subset U_{I_n}$ is closed, and by Lemma~\ref{le:Ku2}~(d) this implies closedness of $\eps_{I_{\ell+1}}(U_{I_n}\less W_{I_n})$.
To verify \eqref{wwell} with $\ell$ replaced by $\ell+1$ first note that for $k=1,\ldots,\ell$ we have
$$
\pi_\Kk^{-1}\bigl(\pi_\Kk(W_{I_{\ell+1}})\bigr)\cap U_{I_k}
= \eps_{I_k}(W_{I_{\ell+1}})
\subset
\eps_{I_k}\bigl( U_{I_{\ell+1}}\less \; \eps_{I_{\ell+1}}(U_{I_k}\less W_{I_k}) \bigr)
\subset W_{I_k}
$$
since the complicated expression consists of those $x\in U_{I_k}$ for which there exists $y \in U_{I_{\ell+1}}$ with $x \sim y$ and $y\not\sim z$ for any $z\in U_{I_k}\less W_{I_k}$, which implies $x\in W_{I_k}$.
Conversely, we have
$$
\pi_\Kk^{-1}(\Ww_\ell)\cap U_{I_{\ell+1}} \subset W_{I_{\ell+1}}
$$
since for every $n\leq\ell$ the intersection
$\pi_\Kk^{-1}(\Ww_\ell)\cap  \eps_{I_{\ell+1}}(U_{I_n}\less W_{I_n})$ is empty. Indeed, otherwise there exist $x\in U_{I_{\ell+1}}$ and $y\in U_{I_n}\less W_{I_n}$ with $x\sim y$, and on the other hand $\pi_\Kk(x)\in\Ww_\ell$. However, this implies $y\in \pi_\Kk^{-1}(\Ww_\ell)$ in contradiction to \eqref{wwell}.
Using these two inclusions we can now finish the proof by verifying the iteration of \eqref{wwell},
$$
\pi_\Kk^{-1}(\Ww_{\ell+1})\cap U_{I_k}
=
\bigl(\pi_\Kk^{-1}(\Ww_\ell)\cap U_{I_k}\bigr)
\cup
\bigl(\pi_\Kk^{-1}\bigl(\pi_\Kk(W_{I_{\ell+1}})\bigr)\cap U_{I_k}\bigr)
= W_{I_k}
$$
holds for $k=1,\ldots,\ell$, and for $k=\ell+1$ we have
$$
\pi_\Kk^{-1}(\Ww_{\ell+1})\cap U_{I_{\ell+1}}
=
\bigl(\pi_\Kk^{-1}(\Ww_\ell)\cap U_{I_{\ell+1}}\bigr)
\cup
\bigl(\pi_\Kk^{-1}\bigl(\pi_\Kk(W_{I_{\ell+1}})\bigr)\cap U_{I_{\ell+1}}\bigr)
= W_{I_{\ell+1}} .
$$
This completes the proof of the statements about $|\Kk|$.
The analogous statements for $|\bE_\Kk|$ hold by using part (b) of Lemma~\ref{le:Ku2} instead of part (a).
\end{proof}

\begin{prop}\label{prop:linear}
Let $\Kk$ be a tame Kuranishi atlas. Then there exists a unique linear structure on the fibers of  $|\pr_{\Kk}|: |\bE_\Kk| \to |\Kk|$ such that for every $I\in\Ii_\Kk$ the embedding $\pi_{\Kk} : U_I\times E_I \to  |\bE_{\Kk}|$ is linear on the fibers.
\end{prop}

\begin{proof}
For fixed $p\in |\Kk|$
denote the union of index sets for which $p\in \pi_\Kk(U_I)$ by
$$
I_p:= \bigcup_{I\in\Ii_\Kk, p\in \pi_\Kk(U_I)}  I
\qquad \subset \{1,\ldots,N\} .
$$
To see that $I_p\in\Ii_\Kk$ we repeatedly use the observation that Lemma~\ref{le:Ku2}~(a) implies
$$
p \in \pi_\Kk(U_I)\cap\pi_\Kk(U_J) \quad\Rightarrow\quad (I,\pi_\Kk^{-1}(p)\cap U_I)\sim (J,\pi_\Kk^{-1}(p)\cap U_J) \quad\Rightarrow\quad I\cup J \in \Ii_\Kk .
$$
Moreover, $x_p:=\pi_\Kk^{-1}(p)\cap U_{I_p}$ is unique by Lemma~\ref{le:Ku2}~(c).
Next, any element in the fiber $[I,x,e]\in |\pr_{\Kk}|^{-1}(p)$ is represented  by some vector over $(I,x)\in \pi_\Kk^{-1}(p)$, so we have $I\subset I_p$ and $\phi_{I I_p}(x)=x_p$, and hence
$(I,x, e)\sim (I_p,x_p,\Hat\phi_{I I_p}(e))$. Thus $\pi_\Kk:\{x_p\}\times E_{I_p}\to |\pr_\Kk|^{-1}(p)$ is surjective, and by Lemma~\ref{le:Ku2}~(c) also injective.
Thus the requirement of linearity for this bijection induces a unique linear structure on the fiber $|\pr_\Kk|^{-1}(p)$.
To see that this is compatible with the injections $\pi_\Kk:\{x\}\times E_{I}\to |\pr_\Kk|^{-1}(p)$ for $(I,x)\sim(I_p,x_p)$ note again that $I\subset I_p$ since $I_p$ was defined to be maximal, and hence by Lemma~\ref{le:Ku2}~(b)~(ii) the embedding factors as $\pi_\Kk|_{\{x\}\times E_{I}} = \pi_\Kk|_{\{x_p\}\times E_{I_p}} \circ \Hat\phi_{I I_p}$, where $\Hat\phi_{I I_p}$ is linear by definition of coordinate changes. Thus $\pi_\Kk|_{\{x\}\times E_{I}}$ is linear as well.
\end{proof}

\begin{rmk}\label{rmk:LIN}\rm
It is tempting to think that additivity alone is enough to imply that the fibers of $|\pr_\Kk|:|\bE_\Kk|\to|\Kk|$ are vector spaces. However, if the first tameness condition \eqref{eq:tame1} fails because  there is $x\in (U_{IJ}\cap U_{IK}) \less U_{I(J\cup K)}$, then both $E_J$ and $E_K$ embed into the fiber $|\pr_\Kk|^{-1}([I,x]))$, but may not be summable, since such sums are well defined by additivity only in $E_{J\cup K}$.
\end{rmk}

We end this section with further topological properties of the Kuranishi neighbourhood of a tame Kuranishi atlas that will be useful when constructing 
an admissible metric in Section~\ref{ss:shrink} and eventually 
the virtual fundamental class in Section~\ref{s:VMC}.  For that purpose we need to be careful in differentiating between the quotient and subspace topology on subsets of the Kuranishi neighbourhood, as follows.

\begin{definition} \label{def:topologies}
For any subset $\Aa\subset \Obj_{\bB_\Kk}$ of the union of domains of a Kuranishi atlas $\Kk$, we denote by 
$$
\|\Aa\|:=\pi_\Kk(\Aa)\subset|\Kk| , 
\qquad\qquad
|\Aa|:=\pi_\Kk(\Aa)\cong \quot{\Aa}{\sim}
$$
the set $\pi_\Kk(\Aa)$ equipped with its subspace topology induced from the inclusion $\pi_\Kk(\Aa)\subset|\Kk|$ resp.\ its quotient topology induced from the inclusion $\Aa\subset \Obj_{\bB_\Kk}$ and the equivalence relation $\sim$ on $\Obj_{\bB_\Kk}$ (which is generated by all morphisms in $\bB_\Kk$, not just those between elements of $\Aa$).
\end{definition}

\begin{remark} \label{rmk:hom} \rm
In many cases we will be able to identify different topologies on subsets of the Kuranishi neighbourhood $|\Kk|$ by appealing to the following elementary {\bf nesting uniqueness of compact Hausdorff topologies}:

Let $f:X\to Y$ be a continuous bijection from a compact topological space $X$ to a Hausdorff space $Y$. Then $f$ is in fact a homeomorphism. 
Indeed, it suffices to see that $f$ is a closed map, i.e.\ maps closed sets to closed sets, since that implies continuity of $f^{-1}$. But any closed subset of $X$ is also compact, and its image in $Y$ under the continuous map $f$ is also compact, hence closed since $Y$ is Hausdorff.

In particular, if $Z$ is a set with nested compact Hausdorff topologies $\Tt_1\subset\Tt_2$, then $\id_Z: (Z,\Tt_2)\to (Z,\Tt_1)$ is a continuous bijection, hence homeomorphism, i.e.\ $\Tt_1=\Tt_2$.
\end{remark}

\begin{prop}\label{prop:Ktopl1}  
Let $\Kk$ be a tame Kuranishi atlas.
\begin{enumerate}
\item
For any subset $\Aa\subset \Obj_{\bB_\Kk}$ the identity map $\id_{\pi_\Kk(\Aa)}: |\Aa| \to \|\Aa\|$ is continuous.
\item 
If $\Aa \sqsubset \Obj_{\bB_\Kk}$ is precompact, then both $|\ov\Aa|$ and $\|\ov\Aa\|$ are compact. In fact, the quotient and subspace topologies on $\pi_\Kk(\ov\Aa)$ coincide, that is $|\ov\Aa|=\|\ov\Aa\|$ as topological spaces.
\item
If $\Aa \sqsubset \Aa' \subset \Obj_{\bB_\Kk}$, then $\pi_\Kk(\ov{\Aa}) = \ov{\pi_\Kk(\Aa)}$ and $\pi_\Kk(\Aa) \sqsubset \pi_\Kk(\Aa')$ in the topological space $|\Kk|$.
\item
If  $\Aa \sqsubset \Obj_{\bB_\Kk}$ is precompact, then $\|\ov{\Aa}\|=|\ov\Aa|$ is metrizable; in particular this implies that $\|\Aa\|$ is metrizable.
\end{enumerate}

\end{prop}
\begin{proof}
To prove (i) recall that openness of $\Uu\subset\pi_\Kk(\Aa)$ in the subspace topology implies the existence of an open subset $\Ww \subset |\Kk|$ with $\Ww \cap \pi_\Kk(\Aa)=\Uu$. Then we have $\Aa \cap \pi_\Kk^{-1}(\Uu)=\Aa \cap \pi_\Kk^{-1}(\Ww)$, where $\pi_\Kk^{-1}(\Ww) \subset \bigcup_{I\in\Ii_\Kk}U_I$ is open by definition of the quotient topology on $|\Kk|$. However, that exactly implies openness of $\Aa \cap \pi_\Kk^{-1}(\Uu)\subset\Aa$ and thus of $\Uu$ in the quotient topology. This proves continuity.

The compactness assertions in (ii) follow from the compactness of $\ov\Aa$ together with the fact that both $\pi_\Kk: \Aa \to |\Kk|$ and $\pi_\Kk: \Aa \to {\Aa}/{\sim}$ are continuous maps.
Moreover, $\|\Aa\|$ is Hausdorff because its topology is induced by the Hausdorff topology on $|\Kk|$.
Therefore the identity map $|\ov\Aa|\to\|\ov\Aa\|$ is a continuous bijection from a compact space to a Hausdorff space, and hence a homeomorphism by Remark~\ref{rmk:hom}, which proves the equality of topologies.

In (iii), the continuity of $\pi_\Kk$ implies $\pi_\Kk(\ov \Aa)\subset \ov{\pi_\Kk(\Aa)}$ for the closure in $|\Kk|$. On the other hand, the compactness of $\ov\Aa$ implies that $\pi_\Kk(\ov \Aa)$ is compact by (ii), in particular it is closed and contains $\pi_\Kk(\Aa)$, hence also contains $ \ov{\pi_\Kk(\Aa)}$. This proves equality $\pi_\Kk(\ov \Aa)=\ov{\pi_\Kk(\Aa)}$.  The last claim of (iii) then holds because $\ov{\pi_\Kk(\Aa)}= \pi_\Kk(\ov \Aa)\subset \pi_\Kk(\Aa')$, and $\pi_\Kk(\ov \Aa)$ is compact by (ii). 

To prove the metrizability in (iv), we will use Urysohn's metrization theorem, which says that any regular and second countable topological space is metrizable.
Here $\|\ov{\Aa}\|\subset |\Kk|$ is regular (i.e. points and closed sets have disjoint neighbourhoods) since it is a compact subset of a Hausdorff space. 
So it remains to establish second countability, i.e.\ to find a countable base for the topology, namely a countable collection of open sets, such that any other open set can be written as a union of part of the collection.

For that purpose first recall that each $U_I$ is a manifold, so is second countable by definition. This property is inherited by the subsets $\ov{A}_I\subset U_I$ for $I\in\Ii_\Kk$, and by their images $\pi_\Kk(\ov{A}_I)\subset|\Kk|$ via the homeomorphisms $\pi_\Kk|_{U_I}$ of Proposition~\ref{prop:Khomeo}.
Moreover, each $\pi_\Kk(\ov{A}_I)$ is compact since it is the image under the continuous map $\pi_\Kk$ of the closed subset $\ov{A}_I=\ov{\Aa}\cap U_I$ of the compact set $\ov{\Aa}$.
So, in order to prove second countability of the finite union $\|\ov{\Aa}\|=\bigcup_{I\in\Ii_\Kk} \pi_\Kk(\ov{A}_I)$ iteratively, it remains to establish second countability for a union of two compact second countable subsets, as follows.

\MS
\NI
{\bf Claim:} Let $B,C \subset Y$ be compact subsets of a Hausdorff space $Y$ such that $B,C $ are second countable in their subspace topologies. Then $B \cup  C$ is second countable in the subspace topology.

\MS
To prove this claim, let $(V_i^B)_{i\in\N}$ resp.\ $(V_i^C)_{i\in\N}$ be countable neighbourhood bases for $B$ and $C$. Then $(V_i^B \cap (B\less C))_{i\in\N}$ resp.\ $(V_i^C \cap (C\less B))_{i\in\N}$ are countable neighbourhood bases for the open subsets $B\less C \subset B$ resp.\ $C\less B  \subset C$.
To finish the construction of a countable neighborhood basis for $B\cup C$ it then suffices to find a countable collection of open sets $W_j\subset B\cup C$ with the property
\begin{equation}\label{eq:RAB}
R\subset B\cup C \quad\text{open}\qquad \Longrightarrow \qquad
R\cap (B\cap C) \; \subset {\textstyle \underset{W_j\subset R}\bigcup} W_j .
\end{equation}
For then $R$ will be the union of these $W_j$ together with all the sets $V^B_i\less C$ and $V^C_i\less B$ that are contained in $R$.

To construct the $W_j$, choose metrics $d^B, d^C$ on $B,C$ respectively (which are guaranteed by Urysohn's metrization theorem). Since $B\cap C$ is compact, the restrictions of the metrics to this intersection are equivalent, i.e.\ there is $\kappa>1$ such that
$$
 \tfrac 1\kappa d^B(x,y)\le d^C(x,y)\le \kappa \ d^B(x,y) \qquad\forall \ x,y\in B\cap C.
$$
For any subset $S\subset B$ and $\eps>0$ denote the $\eps$-neighbourhood of $S$ in $B$ by
$$
\Nn^B_\eps(S): = \bigl\{ y\in B \,\big|\, \ {\textstyle \inf_{s\in S}} d^B(y,s)<\eps \bigr\}
$$
and similarly define $\Nn^C_\eps(T)$ for $T\subset C$.
These are open sets by the triangle inequality.
Moreover, the triangle inequality for $d^C$ together with the above equivalence of metrics gives the following nesting of neighbourhoods for all $S\subset B\cap C$ and $\eps,\delta>0$,
\begin{equation} \label{skewtriangle}
 \Nn^C_\eps\bigl(\Nn^B_\de(S)\cap C\bigr) \;\subset\; \Nn^C_{\kappa\de+\eps}(S).
\end{equation}
Now for any $S \subset B\cap C$ and $\eps>0$ we define an $\eps$-neighbourhood in $B\cup C$ by
$$
W_\eps(S)\,:=\; \Nn^B_\eps(S) \;\cup\; \Bigl(\Nn^C_\eps\bigl(\Nn^B_\eps(S)\cap C)\bigr)
\;\less\; \bigl( B \,\less\, \Nn^B_\eps(S)\bigr)\Bigr) \;\subset\; B\cup C.
$$
Note here that $S = S\cap B\subset \Nn^B_\eps(S)$ already implies the inclusion $S \subset W_\eps(S)$. Moreover, the definition is made to satisfy the nesting property
\begin{equation} \label{eq:incl}
T \subset S \quad\Longrightarrow \quad W_\eps(T) \subset W_\eps(S) .
\end{equation}
To see that $W_\eps(S)\subset B\cup C$ is open, it suffices to check relative openness of the intersections with $B$ and $C$, since then both $B\less W_\eps(S)$ and $C\less W_\eps(S)$ are compact in the relative topology, so is their union $(B\less W_\eps(S))\cup(C\less W_\eps(S))=
 (B\cup C)\less W_\eps(S)$, and hence $W_\eps(S)\subset B\cup C$ is the complement of a closed subset.
Indeed, $W_\eps(S) \cap B = \Nn^B_\eps(S)$ is open since the $\eps$-neighbourhoods were constructed open, and we use the inclusion $T\subset \Nn^C_\eps(T)$ for $T=\Nn^B_\eps(S)\cap C$ to express
$$
W_\eps(S) \cap C \;=\; \Nn^C_\eps\bigl(\Nn^B_\eps(S)\cap C)\bigr) \;\less\;
 \bigl(B \,\less\, \Nn^B_\eps(S)\bigr)
$$
as complement of a closed set in an open set. This shows openness of $W_\eps(S) \cap C$ and hence of $W_\eps(S)$. Moreover, the equivalence of metrics gives for any $S\subset B\cap C$
$$
W_\eps(S)\cap C\;\subset\; \Nn^C_\eps\bigl(\Nn^B_\eps(S)\cap C)\bigr) \;\subset\;
 \Nn^C_\eps\bigl(\Nn^C_{\kappa\eps}(S)\cap C)\bigr) \;\subset\;  \Nn^C_{(\kappa+1)\eps}(S).
$$
Next, recall that compact metrizable spaces are separable. Since we can equip $B\cap C$ with either metric $d^B$ or $d^C$, we obtain a dense sequence $(x_n)_{n\in\N}\subset B\cap C$.
Now we claim that the countable collection
$$
( W_j )_{j\in\N} \;=\; \bigl( W_{1/m}(x_n) \bigr)_{m,n\in\N}
$$
satisfies \eqref{eq:RAB}.  To see this, we must check that for every $r\in R\cap (B\cap C)$ there is $m,n\in\N$ with $r\in W_{1/m}(x_n)$.
For that purpose first choose $\eps>0$ so that $\Nn^B_\eps(r)\subset R\cap B$
and $\Nn^C_\eps(r)\subset R\cap C$. Then choose $m\in\N$ so that $m>(2\kappa+1)/\eps$,
and choose $x_n\in \Nn^B_{1/m}(r)$. Then we use \eqref{eq:incl} and the equivalence of metrics to obtain the inclusion
\begin{align*}
W_{1/m}(x_n)
&\;\subset\;
W_{1/m}\bigl(\Nn^B_{1/m}(r)\bigr) \\
& \;\subset\;
\Nn^B_{1/m}\bigl(\Nn^B_{1/m}(r)\bigr)\; \cup \; \Nn^C_{1/m}\bigl(\Nn^B_{1/m}\bigl(\Nn^B_{1/m}(r)\bigr) \cap C\bigr)\\
&\; \subset \; \Nn^B_{2/m}(r)\; \cup \; \Nn^C_{(2\kappa+1)/m}(r)
\; \subset \; \Nn^B_{\eps}(r)\; \cup \; \Nn^C_{\eps}(r) \;\subset\; R
\end{align*}
as required.
This proves \eqref{eq:RAB} and hence the claim, which finishes the proof of (iv).
In particular, $\|\Aa\|$ is metrizable in the subspace topology, by restriction of a metric on $\pi_\Kk(\ov\Aa)\subset|\Kk|$.
\end{proof}

\subsection{Shrinkings and tameness}
\label{ss:shrink}  \hspace{1mm}\\ \vspace{-3mm}

The purpose of this section is to prove Theorem~\ref{thm:K} and Proposition~\ref{prop:metric} by giving a general construction of a metric, tame Kuranishi atlas starting from an additive, weak Kuranishi atlas. The  construction will be a suitable shrinking of the footprints along with the domains of charts and transition maps, as follows.

\begin{defn}\label{def:shr0}
Let $(F_i)_{i=1,\ldots,N}$ be an open cover of a compact space $X$.  We say that $(F_i')_{i=1,\ldots,N}$ is a {\bf shrinking} of $(F_i)$ if $F_i'\sqsubset F_i$ are
precompact open subsets, 
which cover $X= \bigcup_{i=1,\ldots,N} F'_i$, and are such that for all subsets $I\subset \{1,\ldots,N\}$ we have
\begin{equation} \label{same FI}
F_I: = {\textstyle\bigcap_{i\in I}} F_i \;\ne\; \emptyset
\qquad\Longrightarrow\qquad
F'_I: = {\textstyle\bigcap_{i\in I}} F'_i \;\ne\; \emptyset .
\end{equation}
\end{defn}

Recall here that precompactness $V'\sqsubset V$ is defined as the relative closure of $V'$ in $V$ being compact.  If $V$ is contained in a compact space $X$, then $V'\sqsubset V$ is equivalent to the closure $\ov{V'}$ in the ambient space being contained in $V$.

\begin{defn}\label{def:shr}
Let $\Kk=(\bK_I,\Hat\Phi_{I J})_{I, J\in\Ii_\Kk, I\subsetneq J}$ be a weak Kuranishi atlas.   We say that a weak Kuranishi atlas $\Kk'=(\bK_I',\Hat\Phi_{I J}')_{I, J\in\Ii_{\Kk'}, I\subsetneq J}$ is a {\bf shrinking} of $\Kk$~if
\begin{enumerate}
\item  the footprint cover $(F_i')_{i=1,\ldots,N'}$ is a 
shrinking of the cover $(F_i)_{i=1,\ldots,N}$,
in particular the numbers $N=N'$ of basic charts agree, and so do the index sets $\Ii_{\Kk'} = \Ii_\Kk$;
\item
for each $I\in\Ii_\Kk$ the chart $\bK'_I$ is the restriction of $\bK_I$ to a precompact domain $U_I'\subset U_I$
as in Definition \ref{def:restr};
\item
for each $I,J\in\Ii_\Kk$ with $I\subsetneq J$ the coordinate change $\Hat\Phi_{IJ}'$ is the restriction of $\Hat\Phi_{IJ}$  to the open subset $U'_{IJ}: =  \phi_{IJ}^{-1}(U'_J)\cap U'_I$
 as in Lemma~\ref{le:restrchange}.
\end{enumerate}
\end{defn}

\begin{rmk}\rm \label{rmk:shrink}
(i)
Note that any shrinking of an additive weak Kuranishi atlas preserves the weak cocycle condition (since it only requires equality on overlaps) and also the additivity condition.
Moreover, a shrinking  is determined by the choice of domains $U'_I\sqsubset U_I$ and so can be considered as the restriction of $\Kk$ to 
the subset
$\bigcup_{I\in\Ii_\Kk} U_I'\subset\Obj_{\Bb_\Kk}$.
However, for a shrinking to satisfy a stronger form of the cocycle condition (such as tameness) the  domains $U'_{IJ}:= \phi_{IJ}^{-1}(U'_J)\cap U'_I$ of the coordinate changes must satisfy appropriate compatibility conditions, so that the domains $U_I'$ can no longer be chosen independently of each other.
Since the relevant conditions are expressed in terms of the $U_{IJ}'$, we will find that the construction of a 
tame shrinking in Proposition~\ref{prop:proper} can be achieved by iterative choice of these sets $U_{IJ}'$.  
These will form shrinkings in each step, though we prove it only up to the level of the iteration. 
Our construction is made possible only because of the additivity conditions on $\Kk$. 
\MS

\NI (ii)
Given two tame shrinkings $\Kk^0$ and $\Kk^1$ of the same weak Kuranishi atlas, one might hope to obtain a ``common refinement'' $\Kk^{01}$ by intersection $U^{01}_{IJ}:=U^0_{IJ}\cap U^1_{IJ}$ of the domains.
This could in particular simplify the proof of compatibility of the Kuranishi atlases $\Kk^0,\Kk^1$ in Proposition~\ref{prop:cobord2}.
However, for this to be a valid approach the footprint covers $(F^0_i)_{i=1,\ldots,N}$ and $(F^1_i)_{i=1,\ldots,N}$ 
would have to be comparable in the sense that their intersections still cover $X=\bigcup_{i=1,\ldots,N} (F^0_i\cap F^1_i)$ and have the same index set, i.e.\ $F^0_I \cap F^1_I \neq \emptyset$ for all $I\in\Ii_\Kk$.
Once this is satisfied, one can check that $\Kk^{01}$ defines another tame shrinking of $\Kk$.
\end{rmk}

We can now prove the main result of this section. Note that by the above remark, the main challenge is to achieve the tameness conditions \eqref{eq:tame1}, \eqref{eq:tame2}.

\begin{prop}  \label{prop:proper}
Every additive weak Kuranishi atlas $\Kk$ has a shrinking $\Kk'$ that is a tame Kuranishi atlas -- for short called a {\bf tame shrinking}.
\end{prop}

\begin{proof}
Since $X$ is compact and metrizable and the footprint open cover $(F_i)$ is finite, it has a shrinking $(F_i')$ in the sense of Definition~\ref{def:shr0}.
In particular we can ensure that $F_I'\ne \emptyset$ whenever $F_I\ne \emptyset$ by choosing $\de>0$ so that every nonempty $F_I$ contains some ball $B_\de(x_I)$ and then choosing the $F_i'$ to contain $B_{\de/2}(x_I)$ for each $I\ni i$ (i.e.\ $F_I\subset F_i$).
Then we obtain $F'_I\neq\emptyset$ for all $I\in \Ii_\Kk$ since $B_{\de/2}(x_I)\subset \bigcap_{i\in I} F_i =F_I'$.

In another preliminary step we now find precompact open subsets $U_I^{(0)}\sqsubset U_I$ and open sets $U_{IJ}^{(0)}\subset U_{IJ}\cap U_I^{(0)}$ for all $I,J\in\Ii_\Kk$ such that 
\begin{equation}\label{eq:U(0)}
U_I^{(0)}\cap s_I^{-1}(0) = \psi_I^{-1}(F_I'),\qquad U_{IJ}^{(0)}\cap s_I^{-1}(0) = \psi_I^{-1}(F'_I\cap F_J').
\end{equation}
The restricted domains $U_I^{(0)}$ are provided by Lemma~\ref{le:restr0}. 
Next, we may take
\begin{equation}\label{eq:UIJ(0)}
U_{IJ}^{(0)}:= U_{IJ}\cap U_I^{(0)}\cap \phi_{IJ}^{-1}( U_J^{(0)}),
\end{equation}
which is open because $U_J^{(0)}$ is open and $\phi_{IJ}$ is continuous.
It has the required footprint
$$
U_{IJ}^{(0)}\cap s_I^{-1}(0) =U_{IJ}\cap \psi_I^{-1}(F_I')\cap \phi_{IJ}^{-1} \bigl(\psi_J^{-1}(F_J')\bigr) = \psi_I^{-1}(F_I'\cap F_J') =  \psi_I^{-1}(F_J').
$$
Therefore, this defines an additive weak Kuranishi atlas with footprints $F_I'$, which satisfies the conditions of Definition~\ref{def:shr} and so is a shrinking of $\Kk$.

We will construct  
the required shrinking 
$\Kk'$ by choosing possibly smaller domains  $U_I'\subset U_I^{(0)}$ and $U_{IJ}'\subset U_{IJ}^{(0)}$ but will preserve the footprints $F_I'$.
We will also arrange $U_{IJ}' = U_I'\cap \phi_{IJ}^{-1}(U_J')$, so that $\Kk'$ is a shrinking of the original $\Kk$.  Since $\Kk'$ is automatically additive, we just need to make sure that it satisfies the tameness conditions~\eqref{eq:tame1}  and~\eqref{eq:tame2}.
By Lemma \ref{le:tame0} it will then satisfy the cocycle condition and hence will be a Kuranishi atlas.
We will construct the domains $U_I', U_{IJ}'$ by a finite iteration, starting with $U_I^{(0)}, U_{IJ}^{(0)}$.
Here we streamline the notation by setting $U_{I}^{(k)}:=U_{II}^{(k)}$ and extend the notation to all pairs of subsets $I\subset J\subset\{1,\ldots,N\}$ by setting $U_{IJ}^{(k)}=\emptyset$ if $J\notin\Ii_\Kk$. (Note that $J\in\Ii_\Kk$ and $I\subset J$ implies $I\in\Ii_\Kk$.)
Then in the $k$-th step we will construct open subsets $U_{IJ}^{(k)}\subset U_{IJ}^{(k-1)}$ for all $I\subset J\subset\{1,\ldots,N\}$ such that the following holds.

\begin{enumerate}
\item
The zero set conditions $U_{IJ}^{(k)}\cap s_I^{-1}(0) = \psi_I^{-1}(F_J')$ hold for all $I\subset J$.
\item
The first tameness condition \eqref{eq:tame1} holds for all $I\subset J,K$ with $|I|\le k$, that is
$$
U_{IJ}^{(k)}\cap U_{IK}^{(k)}= U_{I (J\cup K)}^{(k)} .
$$
In particular, we have $U_{IK}^{(k)} \subset U_{IJ}^{(k)}$ for $I\subset J \subset K$ with $|I|\le k$.
\item
The second tameness condition \eqref{eq:tame2} holds for all $I\subset J\subset K$ with $|I|\le k$, that is
$$
\phi_{IJ}(U_{IK}^{(k)}) = U_{JK}^{(k)}\cap s_J^{-1}(E_I) .
$$
In particular we have $\phi_{IJ}(U_{IJ}^{(k)}) = U_{J}^{(k)}\cap s_J^{-1}(E_I)$ for all $I\subset J$ with $|I|\le k$.
\end{enumerate}\MS

Note that after the $k$-th step, the domains $U^{(k)}_{IJ}$ form a shrinking ``up to order $k$'' in the sense that
\begin{equation}\label{kclaim}
U_{IJ}^{(k)} = U_I^{(k)} \cap \phi_{IJ}^{-1}(U_J^{(k)}) \qquad \forall\; |I|\le k , I\subsetneq J .
\end{equation}
Indeed, for any such pair $I\subsetneq J$, property (iii) with $J=K$ implies 
$$
\phi_{IJ}(U_{IJ}^{(k)}) \;=\; U_{J}^{(k)}\cap s_J^{-1}(E_I)
\;=\; U_{J}^{(k)}\cap \im\phi_{IJ},
$$ 
where the second equality is due to the first implying $U_{J}^{(k)}\cap s_J^{-1}(E_I) \subset \im\phi_{IJ}$, and $\im \phi_{IJ}\subset s_J^{-1}(E_I)$, which follows from $s_J\circ\phi_{IJ}=\Hat\phi_{IJ}\circ s_I$.
Since $\phi_{IJ}$ is injective, this implies $U_{IJ}^{(k)}=\phi_{IJ}^{-1}(U_J^{(k)})$.
Now  \eqref{kclaim} follows since (ii) with $K=I$ implies $U_{IJ}^{(k)} \subset U_{II}^{(k)} = U_I^{(k)}$.

Thus, when the iteration is complete, that is when $k=M: =\max_{I\in \Ii_\Kk} |I|$, then $\Kk'$ is a shrinking of $\Kk$.
Moreover, the tameness conditions hold on $\Kk'$ by (ii) and (iii), and Lemma~\ref{le:tame0} implies that $\Kk'$ satisfies the strong cocycle condition. Hence $\Kk'$ is the desired tame Kuranishi atlas.
So it remains to implement the iteration. 

Our above choice of the domains $U_{IJ}^{(0)}$ completes the $0$-th step since conditions (ii) (iii) are vacuous. Now suppose that the $(k-1)$-th step is complete for some $k\geq 1$.
Then we define $U_{IJ}^{(k)}:=U_{IJ}^{(k-1)}$ for all $I\subset J$ with $|I|\leq k-1$. For $|I|=k$ we also set $U_{II}^{(k)}:=U_{II}^{(k-1)}$.
This ensures that (i) and (ii) continue to hold for $|I|<k$. In order to preserve (iii) for triples $H\subset I\subset J$ with $|H|<k$ we then require that the intersection $U_{IJ}^{(k)}\cap s_I^{-1}(E_H)= U_{IJ}^{(k-1)}\cap s_I^{-1}(E_H)$ is fixed.
In case
$H=\emptyset$, this is condition (i), and since $U_{IJ}^{(k)}\subset U_{IJ}^{(k-1)}$ it can generally be phrased as inclusion (i$'$) below.
With that it remains to construct the open sets $U_{IJ}^{(k)}\subset U_{IJ}^{(k-1)}$ as follows.
\begin{itemize}
\item[(i$'$)]
For all $H\subsetneq I\subset J$ with $|H|<k$ and $|I|\geq k$ we have $U_{IJ}^{(k-1)}\cap s_I^{-1}(E_H)\subset U_{IJ}^{(k)}$. Here we include $H=\emptyset$, in which case the condition says that
$U_{IJ}^{(k-1)}\cap s_I^{-1}(0)\subset U_{IJ}^{(k)}$ (which implies $U_{IJ}^{(k)}\cap s_I^{-1}(0) = \psi_I^{-1}(F_J')$, as explained above).
\item[(ii$'$)]
For all $I\subset J,K$ with $|I|= k$ we have
$U_{IJ}^{(k)}\cap U_{IK}^{(k)}= U_{I (J\cup K)}^{(k)}$.
\item[(iii$'$)]
For all $I\subsetneq J\subset K$ with $|I|=k$ we have
$\phi_{IJ}(U_{IK}^{(k)}) = U_{JK}^{(k)}\cap s_J^{-1}(E_I)$.
\end{itemize}

Here we also used the fact that (iii) is automatically satisfied for $I=J$ and so stated (iii$'$) only for $J\supsetneq I$.
By construction, the domains $U^{(k)}_{II}$ for $|I|=k$ already satisfy (i$'$), so we may now do this iteration step in two stages:

\MS\NI
{\bf Step A}\; will construct $U_{IK}^{(k)}$ for $|I|=k$ and $I\subsetneq K$ satisfying (i$'$),(ii$'$) and

\begin{itemize}
\item[(iii$''$)]
$U_{IK}^{(k)} \subset \phi_{IJ}^{-1}(U_{JK}^{(k-1)})$
for all $I\subsetneq J\subset K$ .
\end{itemize}
{\bf Step B} will construct $U_{JK}^{(k)}$ for $|J|>k$ and $J\subset K$ satisfying (i$'$) and (iii$'$).
\MS

\bigskip\NI
{\bf Step A:} We will accomplish this construction by applying Lemma \ref{le:set} below for fixed $I$ with $|I|=k$ to the complete metric space $U:=U_I$, its precompact open subset $U':=U^{(k)}_{II}$, the relatively closed subset
$$
Z :=
\bigcup_{H\subsetneq I} \bigl( U_{II}^{(k-1)} \cap s_I^{-1}(E_H) \bigr)
\;=\; \bigcup_{H\subsetneq I} \phi_{HI}\bigl(U_{HI}^{(k-1)}\bigr)
\;\subset\; U'
$$
and the relatively open subsets for all $i\in\{1,\ldots,N\}\less I$
$$
Z_i := U^{(k-1)}_{I (I\cup\{i\})} \cap Z
\;=\; \bigcup_{H\subsetneq I} \bigl( U_{I(I\cup\{i\})}^{(k-1)} \cap s_I^{-1}(E_H) \bigr)
\;=\; \bigcup_{H\subsetneq I} \phi_{HI}\bigl(U^{(k-1)}_{H (I\cup\{i\})} \bigr) .
$$
Here, by slight abuse of language, we define $\phi_{\emptyset I}\bigl(U_{\emptyset J}^{(k-1)}\bigr): =
U_{IJ}^{(k-1)} \cap s_I^{-1}(0) $.
Note that $Z\subset U'$ is relatively closed since $U_{II}^{(k-1)}=U_{II}^{(k)}=U'$, and
$Z_i\subset Z$ is relatively open since $U^{(k-1)}_{I (I\cup\{i\})} \subset U'$ is open.
Also the above identities  for $Z$ and $Z_i$ in terms of $\phi_{HI}$
follow from (iii) for the triples $H\subset I \subset I$ and $H\subset I \subset I\cup\{i\}$ with $|H|<|I|=k$.

To understand the choice of these subsets, note that in case $k=1$ with $I=\{i_0\}$ the set $Z$ is given by $H=\emptyset$ and $U^{(0)}_{II}=U^{(0)}_{i_0} \sqsubset U_{i_0}$, hence $Z = U^{(0)}_{i_0} \cap s_{i_0}^{-1}(0) = \psi_{i_0}^{-1}(F_{i_0}')$ is the preimage of the shrunk footprint and for $i\neq{i_0}$ we have $Z_{i}
= \psi_{i_0}^{-1}(F_{i_0}'\cap F'_i)$.
When $k\ge 1$, the index sets $K\subset\{1,\ldots,N\}$ containing $I$ as proper subset are in one-to-one correspondence with the nonempty index sets $K'\subset\{1,\ldots,N\}\less I$ via $K=I\cup K'$ and give rise to the relatively open sets
$$
Z_{K'} \,:=\; \bigcap_{i\in K'} Z_i
\;=\; Z \cap \bigcap_{i\in K'} U^{(k-1)}_{I (I\cup\{i\})}
\;=\; Z \cap U^{(k-1)}_{IK} .
$$
Here we used (ii) for $|H| < k$. We may also use the identity $Z_i=\bigcup_{H\subsetneq I} \ldots$ together with (ii) and
(iii) for $|H|<k$ to identify these sets with
\begin{equation}\label{eq:ZzK}
Z_{K'}
= \bigcup_{H\subsetneq I} \phi_{HI}\Bigl(\; \underset{i\in K'}{\textstyle \bigcap} U^{(k-1)}_{H (I\cup\{i\})} \Bigr)
= \bigcup_{H\subsetneq I} \phi_{HI}\bigl( U^{(k-1)}_{H (I\cup K')} \bigr)
= \bigcup_{H\subsetneq I} \bigl( U_{IK}^{(k-1)} \cap s_I^{-1}(E_H) \bigr) ,
\end{equation}
which explains their usefulness:
If we construct the new domains such that $Z_{K'}\subset U_{IK}^{(k)}$, then this implies the inclusion $U_{IK}^{(k-1)}\cap s_I^{-1}(E_H) \subset Z_{K'} \subset U_{IK}^{(k)}$ required by (i$'$).

Finally, in order to achieve the inclusion condition
$U_{IK}^{(k)} \subset \phi_{IJ}^{-1}(U_{JK}^{(k-1)})$ of (iii$''$),
we fix the open subsets $W_{K'}$ for all $I\subsetneq K = I \cup K'$ as
\begin{equation}\label{eq:WwK}
W_{K'}: = \bigcap_{I\subset J\subset K} \bigl( \phi_{IJ}^{-1} (U^{(k-1)}_{JK}) \cap U^{(k-1)}_{IJ} \bigr) \;\subset\; U' .
\end{equation}
If we require $U_{IK}^{(k)}\subset W_{K'}$ then this ensures (iii$''$)
as well as $U_{IK}^{(k)}\subset U_{IK}^{(k-1)}$.
The latter follows from the inclusion $W_{K'}\subset U_{IK}^{(k-1)}$, which holds by definition \eqref{eq:WwK} with $J=K$.
Now if we can ensure that $W_{K'}\cap Z = Z_{K'}$, then Lemma \ref{le:set} provides choices of open subsets $U_{IK}^{(k)}\subset U'$ satisfying (ii$'$) and the inclusions $Z_{K'}\subset U_{IK}^{(k)}\subset W_{K'}$. The latter imply (i$'$) and the desired inclusion (iii$''$), as discussed above.

Hence it remains to check that the sets $W_{K'}$ in \eqref{eq:WwK} do satisfy the conditions
$W_{K'}\cap Z = Z_{K'}$.
To verify this, first note that $W_{K'}$ is contained in $U^{(k-1)}_{IJ}$ for all $J\supset I$, in particular
for
$J=K$, and hence
$$
W_{K'}\cap Z \;\subset\; U^{(k-1)}_{IK} \cap Z \;=\; Z_{K'}.
$$
It remains to check $Z_{K'} \subset W_{K'}$.
By~\eqref{eq:WwK} and the  expression for $Z_{K'}$ in the middle of~\eqref{eq:ZzK},
it suffices to show that for all $H\subsetneq I \subset J \subset K$
$$
\phi_{HI}\bigl( U^{(k-1)}_{H K} \bigr) \;\subset\; \phi_{IJ}^{-1} (U^{(k-1)}_{JK}) \cap U^{(k-1)}_{IJ} .
$$
But, (ii) for $H\subset J\subset K$ and (iii) for $H\subset I \subset J$ imply
$$
\phi_{HI}( U^{(k-1)}_{H K} ) \subset \phi_{HI}( U^{(k-1)}_{H J} ) \subset U^{(k-1)}_{IJ},
$$
so it remains
to check that
$\phi_{HI}\bigl( U^{(k-1)}_{H K} \bigr) \subset \phi_{IJ}^{-1} ( U^{(k-1)}_{JK} )$.
For that purpose we will use the weak cocycle condition $\phi_{IJ}^{-1}\circ\phi_{HJ}=\phi_{HI}$ on
$U_{H J}\cap \phi_{HI}(U_{IJ})$. Note that $U^{(k-1)}_{H K}$ lies in this domain since, by (ii) for $|H|<k$, it is a subset of $U^{(k-1)}_{H J}$, which by (iii) for $H\subset I\subset J$ is contained in $\phi_{HI}^{-1}(U_{IJ})$. This proves the first equality in
$$
\phi_{HI}\bigl( U^{(k-1)}_{H K} \bigr)
=\phi_{IJ}^{-1}\bigl( \phi_{HJ}\bigl( U^{(k-1)}_{H K} \bigr) \bigr)
\subset \phi_{IJ}^{-1} ( U^{(k-1)}_{JK} ),
$$
and the last inclusion holds by (iii) for $H\subset J\subset K$.
This finishes Step A.

\MS\NI
{\bf Step B:} The crucial requirement on the construction of the open sets $U_{JK}^{(k)}\subset U_{JK}^{(k-1)}$ for $|J|\geq k+1$ and $J\subset K$ is (iii$'$), that is
$$
U_{JK}^{(k)}\cap s_J^{-1}(E_I) = \phi_{IJ}(U_{IK}^{(k)})
$$
 for all $I\subsetneq J$ with $|I|=k$.
Here $U_{IK}^{(k)}$ is fixed by Step A and satisfies $\phi_{IJ}(U_{IK}^{(k)})\subset U_{JK}^{(k-1)}\cap s_J^{-1}(E_I)$ by (iii$''$),
where the second part of the inclusion is automatic by $\phi_{IJ}$ mapping to $s_J^{-1}(E_I)$. Hence the maximal subsets $U_{JK}^{(k)}\subset U_{JK}^{(k-1)}$ satisfying (iii$'$) are
\begin{equation}\label{eq:UJK(k)}
U_{JK}^{(k)}\,:=\; U_{JK}^{(k-1)}\less \bigcup_{I\subset J, |I|= k}
 \bigl( s_J^{-1}(E_I)\less \phi_{IJ}(U^{(k)}_{IJ}) \bigr) .
\end{equation}
It remains to check that these subsets are open and satisfy (i$'$).
Here $U_{JK}^{(k)}$ is open since $s_J^{-1}(E_I)\subset U_J$ is closed and
 $\phi_{IJ}(U^{(k)}_{IJ})\subset s_J^{-1}(E_I)$ is relatively open by
 the index condition in Definition~\ref{def:change}.
Finally, condition (i$'$), namely
$$
 U_{JK}^{(k-1)}\cap s_J^{-1}(E_H) \subset U_{JK}^{(k)},
$$
follows from the following inclusions for all $H\subsetneq I \subsetneq J \subset K$ with $|I|=k$.
On the one hand we have $U_{JK}^{(k-1)}\cap s_J^{-1}(E_H) \subset s_J^{-1}(E_I)$ from the additivity of $\Kk$; on the other hand (iii) for $|H|<k$ together with the weak cocycle condition on $U_{HK}^{(k-1)}\subset U_{HJ}\cap\phi_{HI}^{-1}(U_{IJ})$
and (i$'$) for $|I|=k$ imply
\begin{align*}
U_{JK}^{(k-1)}\cap s_J^{-1}(E_H)
= \phi_{HJ}( U_{HK}^{(k-1)} )
&= \phi_{IJ}\bigl(\phi_{HI}( U_{HK}^{(k-1)} )\bigr) \\
&= \phi_{IJ}\bigl(U_{IK}^{(k-1)}\cap s_I^{-1}(E_H) \bigr)
\subset  \phi_{IJ}(U^{(k)}_{IJ}) .
\end{align*}
Hence no points of $ U_{JK}^{(k-1)}\cap s_J^{-1}(E_H)$ are removed from
$U_{JK}^{(k-1)}$ when we construct $U_{JK}^{(k)}$.
This finishes Step B and hence the $k$-th iteration step.

\MS
Since the order of $I\in\Ii_\Kk$ is bounded $|I|\leq N$ by the number of basic Kuranishi charts, this iteration provides a complete construction of the shrinking domains after at most $N$ steps. In fact, we obtain $U'_I=U^{(|I|-1)}_{II}$ and $U'_{IJ}=U^{(|I|)}_{IJ}$ since the iteration does not alter these domains in later steps.
\end{proof}

\begin{lemma} \label{le:set}
Let $U$ be a complete metric space, $U'\subset U$ a precompact open set, and $Z\subset U'$ a relatively closed subset. Suppose we are given a finite collection of relatively open subsets $Z_i\subset Z$ for $i=1,\ldots,N$ and open subsets $W_K\subset U'$ with
$$
W_K\cap Z= Z_K: = {\textstyle\bigcap_{i\in K}} Z_i
$$
 for all index sets $K\subset\{1,\ldots,N\}$.
Then there exist open subsets $U_K\subset W_K$ with $U_K\cap Z=Z_K$ and
$U_J\cap U_K = U_{J\cup K}$ for all $J,K\subset\{1,\ldots,N\}$.
\end{lemma}

\begin{proof}
Let us first introduce a general construction of an open set $U_f\subset U'$ associated to any lower semi-continuous function $f:\overline{Z}\to [0,\infty)$, where $\overline{Z}\subset U$ denotes the closure in $U$.
By assumption, $\overline{Z}$ is compact, hence the distance function $\overline{Z}\to [0,\infty), z\mapsto d(x,z)$ for fixed $x\in U'$ achieves its minumum on a nonempty compact set $M_x\subset \overline{Z}$. Hence we have
$$
d(x,Z) := \inf_{z\in Z} d(x,z) = d(x,\overline{Z}) = \min_{z\in\overline Z} d(x,z) = \min_{z\in M_x} d(x,z)
$$
for the distance between $x$ and the set $Z$, or equivalently the closure $\overline{Z}$.
\MS

\NI {\bf Claim:}  {\it For any lower semi-continuous function $f:\overline{Z}\to [0,\infty)$ the
set}
$$
U_f := \bigl\{ x\in U' \,\big|\, d(x,Z) < \inf f|_{M_x} \bigr\} \subset U'
$$
{\it is open (in $U'$ or equivalently in $U$) and satisfies}
\begin{equation} \label{eq:supp}
U_f \cap Z = \supp f \cap Z = \bigl\{ z\in Z \,\big|\, f(z)>0 \bigr\}.
\end{equation}

\NI {\it Proof of Claim.}
For $x\in Z$ we have $d(x,Z)=0$ and $M_x=\{x\}$, so $d(x,Z) < \inf f|_{M_x}$ is equivalent to $0<f(x)$. We prove openness of $U_f\subset U'$ by checking closedness of $U'\less U_f$.
Thus, we consider a convergent sequence $U'\ni x_i\to x_\infty\in U'$ with $d(x_i, Z)\geq \inf f|_{M_{x_i}}$ and aim to prove $d(x_\infty, Z)\geq \inf f|_{M_{x_\infty}}$.
Since $f$ is lower semi-continuous and each $M_{x_i}$ is compact, we may choose a sequence $z_i\in\overline{Z}$ with $z_i\in M_{x_i}$ and $f(z_i)=\inf f|_{M_{x_i}}$.
(Indeed, for fixed $i$ any minimizing sequence $z^\nu_i\in M_{x_i}$ with $\lim_{\nu\to\infty}f(z^\nu_i) = \inf f|_{M_{x_i}}$ has a convergent subsequence $z^\nu_i\to z_i\in M_{x_i}$ and the limit satisfies $f(z_i) \leq \lim f(z^\nu_i) = \inf f|_{M_{x_i}}$, hence $f(z_i) = \inf f|_{M_{x_i}}$.)
Since $\overline{Z}$ is compact, we may moreover choose a subsequence, again denoted by $(x_i)$ and $(z_i)$, such that $z_i\to z_\infty\in\overline{Z}$ converges. Then by continuity of the distance functions we deduce $z_\infty\in M_{x_\infty}$ from
$$
d(x_\infty, z_\infty) = \lim d(x_i,z_i) = \lim d(x_i, Z) = d(x_\infty, Z) ,
$$
and finally the lower semi-continuity of $f$ implies the claim
$$
d(x_\infty, Z) = \lim d(x_i,Z) \geq \lim f(z_i) \geq f(z_\infty) \geq \inf f|_{M_{x_\infty}} .
$$
\medskip

We now use this general construction to define the sets $U_K:=\bigcap_{i\in K} U_{f_i}$ as intersections of the subsets $U_{f_i}\subset U'$ arising from functions $f_i: \overline{Z}\to [0,\infty)$ defined
by
$$
f_i(z) := \min \bigl\{  d( z, U' \less W_J) \,\big|\, J\subset\{1,\ldots,N\} : \;  i\in J, \; d(z,Z\less Z_i) = d(z,Z\less Z_J) \bigr\}.
$$
To check that $f_i$ is indeed lower semi-continuous, consider a sequence $z_\nu\to z_\infty\in \overline{Z}$. Then $f_i(z_\nu)= d(z_\nu,U'\less W_{J^\nu})$ for some index sets $J^\nu$ with $i\in J^\nu$ and $d(z_\nu,Z\less Z_i) = d(z_\nu,Z\less Z_{J^\nu})$. Since the set of all index sets is finite, we may choose a subsequence, again denoted $(z_\nu)$, for which $J^\nu=J$ is constant. Then in the limit we also have $d(z_\infty,Z\less Z_i) = d(z_\infty,Z\less Z_J)$ and hence
$$
f_i(z_\infty) \leq d(z_\infty, U'\less W_J) = \lim d(z_\nu, U'\less W_J) = \lim f_i(z_\nu).
$$
Thus $f_i$ is lower semi-continuous.
Therefore, the above Claim implies that each $U_{f_i}$ is open, and hence also that each $U_K$ is open as the finite intersection of open sets.

The intersection property holds by construction:
$$
U_J\cap U_K = \bigcap_{i\in J} U_{f_i} \cap \bigcap_{i\in K} U_{f_i}
= \bigcap_{i\in J\cup K} U_{f_i} = U_{J\cup K}.
$$
To obtain $U_K\cap Z = \bigcap_{i\in K} \bigl( U_{f_i} \cap Z \bigr) = Z_K$ it suffices to verify that $U_{f_i}\cap Z = Z_i$.
In view of \eqref{eq:supp}, and unravelling the meaning of $f_i(z)>0$ for $z\in Z$, that means we have to prove the
following equivalence for $z\in Z$,
$$
z\in Z_i \quad \Longleftrightarrow \quad  d(z, U'\less W_J)>0  \quad \forall J\subset\{1,\ldots N\} : i\in J ,\; d(z,Z\less Z_i) = d(z,Z\less Z_J).
$$
Assuming the right hand side, we may choose $J=\{i\}$ to obtain $d(z,U'\less W_{\{i\}})>0$, and hence, since $U'\less W_{\{i\}}$ is closed, $z\in Z \less (U'\less W_{\{i\}})= Z_i$.
On the other hand, $z\in Z_i$ implies $d(z,Z\less Z_i)>0$ since $z\in Z$ and $Z\less Z_i \subset Z$ is relatively closed. So for any $J$ with $d(z,Z\less Z_i) = d(z,Z\less Z_J)$ we obtain
$d(z,Z\less Z_J)>0$.  Hence $z\in Z_J\subset W_J$, so that $d(z, U'\less W_J)>0$.
This proves
the desired equivalence,
and hence $U_K\cap Z = \bigcap_{i\in K}Z_i = Z_K$.

Finally, we need to check that $U_K\subset W_K$.
Unravelling the construction, note that $U_K$ is the set of all $x\in U'$ that satisfy
\begin{equation}\label{eq:UK}
d(x, Z) <  d( z , U' \less W_J)
\end{equation}
for all $z\in M_x$ and all $J\subset\{1,\ldots,N\}$ such that there exists $i\in J\cap K $ satisfying
$d(z,Z\less Z_i) = d(z,Z\less Z_J)$.
Now suppose by contradiction that there exists a point $x\in U_K\less W_K$, and pick $z\in M_x$.
Then $d(x,Z) = d(x, z) \geq d( z , U' \less W_K)$ since $x\in U'\less W_K$.
This contradicts \eqref{eq:UK} with $J=K$. On the other hand, the condition $d(z,Z\less Z_i) = d(z,Z\less Z_J)$ for $J=K$ is always satisfied for some $i\in K$ since we have $d(z,Z\less Z_K) = \min_{j\in K} d(z,Z\less Z_j)$.
This provides the contradiction and hence proves $U_K\subset W_K$.
\end{proof}

This lemma completes the proof that every weak additive $\Kk$ has a tame shrinking.
We will return to these ideas in Section~\ref{ss:Kcobord} when discussing cobordisms.
We end this section by Proposition~\ref{prop:metric}, which constructs admissible metrics on certain 
 tame shrinkings by pullback with the map in the following lemma.

\begin{lemma}\label{le:injtameshr}  
Let $\Kk'$ be a tame shrinking of a tame Kuranishi atlas $\Kk$. Then the natural map $\io:|\Kk'|\to |\Kk|$
induced by the inclusion of domains $\io_I:U'_I\to U_I$ is injective.
\end{lemma}
\begin{proof} 
We write $U_I, U_{IJ}$ for the domains of the charts and coordinate changes of $\Kk$ and $U_I', U_{IJ}'$  for those of $\Kk'$, so that $U_I'\subset U_I, U_{IJ}'\subset U_{IJ}$ for all $I,J\in \Ii_\Kk = \Ii_{\Kk'}$.
Suppose that $\pi_\Kk(I,x) = \pi_\Kk(J,y)$ where $x\in U_I', y\in U_J'$.  Then we must show that $\pi_{\Kk'}(I,x) = \pi_{\Kk'}(J,y)$.  Since $\Kk$ is tame, 
Lemma~\ref{le:Ku2}~(a) implies that there is $w\in U_{I\cap J}$ such that
$\phi_{(I\cap J)I} (w)$ is defined and equal to $x$.  Hence $x\in s_I^{-1}(E_{I\cap J})\cap U_I' =\phi_{(I\cap J)I} (U'_{(I\cap J)I})$ by the tameness equation \eqref{eq:tame3}  for $\Kk'$.
Therefore $w\in U'_{(I\cap J)I}$. 
Similarly, because $\phi_{(I\cap J)J} (w)$ is defined and equal to $y$, we have $w\in U'_{(I\cap J)J}$.
Then by definition of $\pi_{\Kk'}$ we deduce
 $\pi_{\Kk'}(I,x) = \pi_{\Kk'}(I\cap J,w) =\pi_{\Kk'}(J,y)$.
\end{proof}

In order to construct metric tame Kuranishi atlases, we will find it
useful to consider tame shrinkings $\Kk_{sh}$ of a
weak Kuranishi atlas $\Kk$ that are obtained as shrinkings of an intermediate tame shrinking $\Kk'$ of $\Kk$.
For short we will call such $\Kk_{sh}$ a 
{\bf preshrunk tame shrinking} of $\Kk$.

\begin{prop}\label{prop:metric}  
Let $\Kk$ be an additive weak Kuranishi atlas.
Then every preshrunk tame shrinking of $\Kk$ is metrizable. In particular, 
$\Kk$ has a metrizable tame shrinking.
\end{prop}

\begin{proof}  
First use Proposition~\ref{prop:proper} to construct a  tame shrinking $\Kk'$ of $\Kk$ with domains $(U_I'\subset U_I)_{I\in \Ii_\Kk}$, and then use this result again to construct a tame shrinking $\Kk_{sh}$ of $\Kk'$ with domains $(U_I^{sh}\sqsubset U_I')_{I\in \Ii_\Kk}$. We claim that $\Kk_{sh}$ is metrizable. 

For that purpose we apply Proposition~\ref{prop:Ktopl1}~(iv) to the precompact subset $\Aa: = \bigcup_{I\in\Ii_\Kk} U_I^{sh}$ of $\Obj_{\bB_{\Kk'}}$ to obtain a metric $d'$ on $\pi_{\Kk'}(\ov\Aa)$ that induces the relative topology on the subset $\pi_{\Kk'}(\ov\Aa)$ of $|\Kk'|$, that is $\bigl( \pi_{\Kk'}(\ov\Aa) , d'\bigr) = \|\ov\Aa\|$.
Further, since $\pi_{\Kk'}(\ov\Aa)$ is compact, the metric $d'$ must be bounded. 
Now, by Lemma~\ref{le:injtameshr} the natural map $\io: |\Kk_{sh}|\to |\Kk'|$ is injective, with image $\pi_{\Kk_{sh}}(\Aa)$, so that the pullback $d:=\io^* d'$ is a bounded metric  on $|\Kk_{sh}|$ that is compatible with the relative topology induced by $|\Kk'|$; in other words $\io : \bigl(|\Kk_{sh}|, d \,\bigr) \to \|\Aa\|$ is an isometry.

Next, note that the pullback metric $d_I$ on $U^{sh}_I$ does give the usual topology since 
$\pi_{\Kk^{sh}}: U^{sh}_I \to \pi_{\Kk^{sh}}(U^{sh}_I) \subset \bigl( |\Kk^{sh}| , d \bigr)$  is a homeomorphism to its image.  
Indeed, by Lemma~\ref{le:injtameshr} it can also be written as $\pi_{\Kk^{sh}}|_{U^{sh}_I} = \io^{-1} \circ \pi_{\Kk'} \circ \io_I$ with the embedding $\io_I: U^{sh}_I \to U'_I$. The latter is a homeomorphism to its image, as is $\pi_{\Kk'} : U'_I \to \pi_{\Kk'}(U'_I) \subset |\Kk'|$ by Proposition~\ref{prop:Khomeo}, and $\io^{-1}$ by the definition of the metric topology on $|\Kk^{sh}|$.  
\end{proof}

\subsection{Cobordisms of Kuranishi atlases}\label{ss:Kcobord}  \hspace{1mm}\\ \vspace{-3mm}

Since there are many choices involved in constructing a Kuranishi atlas, and holomorphic curve moduli spaces in addition depend on the choice of an almost complex structure, it is important to have suitable notions of equivalence.
Since we are only interested here in constructing the VMC as cobordism class, resp.\ the VFC as a homology class, a notion of uniqueness up to cobordism will suffice for our purposes, and should e.g.\ arise from paths of almost complex structures.
We will defer this general notion of cobordism to \cite{MW:ku2} and concentrate here on developing tools for constructing well defined VMC/VFC for a fixed compact moduli space. This requires several types of results.
Within the abstract theory, we firstly need to prove the uniqueness part of Theorem~\ref{thm:K}, saying that metric tame shrinkings of additive weak Kuranishi atlases are unique up to a suitable notion of cobordism, which will be the content of this section. Secondly, we prove in Theorems~\ref{thm:VMC1} and \ref{thm:VMC2}  that cobordant Kuranishi atlases induce the same VMC/VFC.

On the other hand, for any given holomorphic curve moduli space, we need to construct Kuranishi atlases that are canonical up to a suitable notion of equivalence.
Here the most natural notion of equivalence would be along the lines of Morita equivalence, cf.\ Remark~\ref{rmk:Morita}.
However, our constructions of Kuranishi atlases for a fixed Gromov--Witten moduli space in  \cite{MW:gw} will depend on choices (in particular slicing conditions and obstruction spaces for the basic charts) only up to the following simpler notion of commensurability, in the sense that any two choices will yield Kuranishi atlases that are both commensurate to a third.
Finally, we require another abstract result to imply that commensurate Kuranishi atlases induce the same VFC.
In this section we hence fix a compact metrizable space $X$ and introduce the notions of commensurability and cobordism of Kuranishi atlases on $X$.

\begin{defn}\label{def:Kcomm}
Two (weak) Kuranishi atlases $\Kk^0,\Kk^1$ on the same compact space $X$ are {\bf commensurate} if there exists a common extension $\Kk^{01}$. This means that $\Kk^{01}$ is a (weak) Kuranishi atlas on $X$ with basic charts $(\bK^\al_i)_{(\al,i)\in\Nn^{01}}$, where
$$
\Nn^{01} := \Nn^0 \sqcup \Nn^1 ; \qquad \Nn^\al := \bigl\{ (\al,i) \,\big|\, 0 \leq i \leq N^\al \bigr\} ,
$$
and transition data $(\bK^{01}_I,\Hat\Phi^{01}_{IJ})_{I\subset\Nn^{01}, I\subsetneq J}$ such that $\bK^{01}_I=\bK^0_I$ and $\Hat\Phi^{01}_{IJ}=\Hat\Phi^\al_{IJ}$ whenever $I,J\subset\Nn^\al$ for fixed $\al=0$ or $\al=1$.

Moreover, if $\Kk^0,\Kk^1$ are additive, we say they are {\bf additively commensurate} if there exists a common extension $\Kk^{01}$ that in addition is additive.
\end{defn}

We will see that commensurability is stronger than cobordism, but note that it does not satisfy transitivity, hence is not an equivalence relation.
In order to define the notion of cobordism such that it is transitive, we will need a special form of charts and coordinate changes at the boundary that allows for gluing of cobordisms.
Here, in order to avoid having to deal with general Kuranishi atlases with boundary, we restrict our work to a notion of cobordism of Kuranishi atlases on the same space $X$, which involves Kuranishi atlases on the space $X\times[0,1]$, whose 
``boundary'' $X\times\{0,1\}$
has a natural collar structure. We deal with this boundary by requiring all charts and coordinate changes in a sufficiently small collar to be of product form as introduced below.
These notions of collars and product forms will also be used for introducing a more general notion of ``cobordism with Kuranishi atlas'' in \cite{MW:ku2}.

\begin{defn} \label{def:Cchart}
\begin{itemlist}
\item
Let $\bK^\al=(U^\al,E^\al,s^\al,\psi^\al)$ be a Kuranishi chart on $X$, and let $A\subset[0,1]$ be a relatively open interval. Then we define the {\bf product chart} for $X\times[0,1]$ with footprint $F_I^\al\times A$ as
$$
\bK^\al \times A :=\bigl(U^\al \times A, E^\al, s^\al\circ{\rm pr}_{U^\al}  ,\psi^\al\times\id_{A} \bigr) .
$$
\item
A {\bf Kuranishi chart with collared boundary} on $X\times[0,1]$ is a tuple $\bK = (U,E ,s ,\psi )$ of domain, obstruction space, section, and footprint map as in Definition~\ref{def:chart}, with the following boundary variations and collar form requirement:
\begin{enumerate}
\item
The footprint $F =\psi (s ^{-1}(0))\subset X\times[0,1]$ intersects the boundary $X\times\{0,1\}$.
\item
The domain $U $ is a smooth manifold whose boundary splits into two parts $\partial U  = \partial^0 U  \sqcup \partial^1 U $ such that $\partial^\al U $ is nonempty iff $F $ intersects $X\times\{\al\}$.
\item
If $\partial^\al U \neq\emptyset$ then there is a relatively open neighbourhood $A^\al\subset [0,1]$ of $\al$ and an embedding $\iota^\al:\partial^\al U \times A^\al  \hookrightarrow  U $ onto a neighbourhood of $\partial^\al U \subset U $ such that
$$
\bigl( \, \partial^\al U \times A^\al \,,\, E  \,,\, s \circ \iota^\al  \,,\, \psi \circ \iota^\al  \, \bigr)
\; = \; \partial^\al\bK  \times A^\al
$$
is the product of $A^\al$ with some Kuranishi chart $\partial^\al\bK $ for $X$
with footprint $F\subset X$ such that $(X\times A^\al) \cap F  = F\times A^\al$.
\end{enumerate}
\item
For any Kuranishi chart with collared boundary $\bK $ on $X\times[0,1]$ we call the uniquely determined Kuranishi charts $\partial^\al\bK $ for $X$ the {\bf restrictions of $\bK $ to the boundary} for $\al=0,1$.
\end{itemlist}
\end{defn}

We now define a coordinate change between charts on $X\times [0,1]$ that may have boundary.   Because in a Kuranishi atlas there is a coordinate change $\bK_I\to \bK_J$ only when $F_I\supset F_J$, we will  restrict to this case here.  (Definition~\ref{def:change} considered a more general scenario.)
In other words, we need not consider coordinate changes from a chart without boundary to a chart with boundary.

\begin{defn} \label{def:Ccc}
\begin{itemlist}
\item
Let $\Hat\Phi^\al_{IJ}:\bK^\al_I\to\bK^\al_J$ be a coordinate change between Kuranishi charts for $X$, and let $A_I,A_J\subset[0,1]$ be relatively open intervals.
Then the {\bf product coordinate change}
$\Hat\Phi^\al_{IJ}\times \id_{A_I\cap A_J} : \bK^\al_I\times A_I \to \bK^\al_J\times A_J$ is given by
$$
\phi^\al_{IJ}\times \id_{A_I\cap A_J}: \; U^\al_{IJ}\times (A_I\cap A_J) \;\to\; U^\al_J\times A_J,
\qquad
\Hat\phi^\al_{IJ}:  E_I^\al \to  E_J^\al.
$$
\item
Let $\bK _I,\bK _J$ be Kuranishi charts on $X\times[0,1]$ such that only $\bK _I$ or both
$\bK _I,\bK _J$ have collared boundary.
Then a {\bf coordinate change with collared boundary} $\Hat\Phi_{IJ} :\bK _I\to\bK _J$ is a tuple $\Hat\Phi_{IJ}  = (U_{IJ} ,\phi_{IJ} ,\Hat\phi_{IJ} )$ of domain and embeddings as in Definition~\ref{def:change}, with the following boundary variations and collar form requirement:
\begin{enumerate}
\item
The domain is a relatively open subset $U _{IJ}\subset U _I$ with boundary components
$\partial^\al U _{IJ}:= U _{IJ} \cap \partial^\al U _I$;
\item
If $F_J\cap (X\times \{\al\}) \ne \emptyset$ for  $\al = 0$ or $1$, there
is a relatively open neighbourhood $B^\al\subset [0,1]$
of $\al$
such that
\begin{align*}
(\iota_I^\al)^{-1}(U _{IJ})
\cap \bigl(\partial^\al U _I \times B^\al\bigr)
&\;=\; \partial^\al U _{IJ} \times B^\al , \\
(\iota_J^\al)^{-1}(\im\phi _{IJ})
\cap \bigl(\partial^\al U _J \times B^\al \bigr)
&\;=\;
 \phi _{IJ}(\partial^\al U _{IJ}) \times B^\al ,
\end{align*}
and
$$
\bigl( \,
\partial^\al U _{IJ} \times B^\al \,,\, (\iota_J^\al)^{-1}\circ \phi _{IJ} \circ \iota_I^\al  \,,\,  \Hat\phi _{IJ} \, \bigr)
\;=\; \partial^\al\Hat\Phi_{IJ} \times {\rm id}_{B^\al} ,
$$
where $\partial^\al\Hat\Phi_{IJ} : \partial^\al\bK _I \to \partial^\al\bK _J$ is a coordinate change.
\item
If $\p^\al F_J= \emptyset$ but $\p^\al F_I\ne \emptyset$ for  $\al = 0$ or $1$
there is a neighbourhood $B^\al\subset [0,1]$
of $\al$
such that
$$
U _{IJ}\cap \iota_I^\al \bigl(\partial^\al U _I \times B^\al \bigr) = \emptyset.
$$
\end{enumerate}
\item
For any coordinate change with collared boundary $\Hat\Phi_{IJ} $ on $X\times[0,1]$ we call the uniquely determined coordinate changes $\p^\al \Hat\Phi_{IJ}$ for $X$ the {\bf restrictions of $\Hat\Phi_{IJ} $ to the boundary} for $\al=0,1$.
\end{itemlist}
\end{defn}

\begin{defn}\label{def:CKS}
A {\bf (weak) Kuranishi cobordism} on $X\times [0,1]$ is a tuple
$$
\Kk^{[0,1]} = \bigl( \bK^{[0,1]}_{I} , \Hat\Phi_{IJ}^{[0,1]} \bigr)_{I,J\in \Ii_{\Kk^{[0,1]}}}
$$
of basic charts and transition data as in Definition~\ref{def:Ku} resp.\ \ref{def:Kwk}, with the following boundary variations and collar form requirements:
\begin{itemlist}
\item
The charts of $\Kk^{[0,1]}$ are either Kuranishi charts with collared boundary or standard Kuranishi charts whose footprints are precompactly contained in $X\times(0,1)$.
\item
The coordinate changes $\Hat\Phi_{IJ}^{[0,1]}: \bK^{[0,1]}_{I} \to \bK^{[0,1]}_{J}$ are either standard coordinate changes on $X\times(0,1)$ between pairs of standard charts, or coordinate changes with collared boundary between pairs of charts, of which at least the first has collared boundary.
\end{itemlist}
Moreover, we call $\Kk^{[0,1]}$ {\bf additive resp.\ tame} if it satisfies the additivity condition of Definition~\ref{def:Ku2}, resp.\ the additivity and tameness conditions of Definition~\ref{def:tame}.
\end{defn}

\begin{remark} \label{rmk:restrict}\rm
Any (weak) Kuranishi cobordism $\Kk^{[0,1]}$ induces by restriction two 
(weak) 
Kuranishi atlases $\partial^\al\Kk^{[0,1]}$ on $X$ for $\al=0,1$ with
\begin{itemize}
\item 
basic charts $\p^\al\bK_i$ given by restriction of basic charts of $\Kk^{[0,1]}$ with $F_i\cap X\times\{\al\}\neq\emptyset$;
\item 
index set $\Ii_{\p^\al\Kk^{[0,1]}}=\{I\in\Ii_{\Kk^{[0,1]}}\,|\, F_I\cap X\times\{\al\}\neq\emptyset\}$;
\item 
transition charts $\p^\al\bK_I$ given by restriction of transition charts of $\Kk^{[0,1]}$;
\item
coordinate changes $\p^\al\Hat\Phi_{IJ}$ given by restriction of coordinate changes of $\Kk^{[0,1]}$.
\end{itemize}
If $\Kk^{[0,1]}$ is additive or tame, then so are the restrictions $\partial^\al\Kk^{[0,1]}$.
Finally, $\Kk^{[0,1]}$ provides a cobordism from $\p^0\Kk^{[0,1]}$ to $\p^1\Kk^{[0,1]}$ in the sense of the following Definition~\ref{def:Kcobord}.
\end{remark}

With this language in hand, we can now introduce the cobordism relation between Kuranishi atlases. 
While such notions exist (and generally are equivalence relations) for all flavours of (additive/weak/tame) Kuranishi atlases, we restrict ourselves here to the cobordism relation under which the VMC/VFC will be well defined, namely additive cobordism for weak atlases.
In contrast, it will be important to  prove the existence of many kinds of cobordisms as in Propositions~\ref{prop:cobord2} and~\ref{prop:cov2}.

\begin{defn}\label{def:Kcobord}
Two additive weak Kuranishi atlases $\Kk^0, \Kk^1$ on $X$ are {\bf additively cobordant} if there exists an additive weak Kuranishi cobordism $\Kk^{[0,1]}$ from $\Kk^0$ to $\Kk^1$.
That is, $\Kk^{[0,1]}$ is an additive weak Kuranishi cobordism on $X\times[0,1]$, which restricts to $\partial^0\Kk^{[0,1]}=\Kk^0$ on $X\times\{0\}$ and to $\p^1\Kk^{[0,1]}=\Kk^1$ on $X\times\{1\}$.
More precisely, there are injections $\iota^\al:\Ii_{\Kk^\al} \hookrightarrow \Ii_{\Kk^{[0,1]}}$ for $\al=0,1$ such that $\im\iota^\al=\Ii_{\partial^\al\Kk}$ and for all $I,J\in\Ii_{\Kk^\al}$ we have
$$
\bK^\al_I = \p^\al \bK^{[0,1]}_{\iota^\al(I)}, \qquad
\Hat\Phi^\al_{IJ} = \p^\al \Hat\Phi^{[0,1]}_{\iota^\al(I) \iota^\al (J)} .
$$
 \end{defn}

In the following we will usually identify the index sets $\Ii_{\Kk^\al}$ of cobordant Kuranishi atlases with the restricted index set $\Ii_{\partial^\al\Kk}$ in the cobordism index set $\Ii_{\Kk^{[0,1]}}$, so that $\Ii_{\Kk^0}, \Ii_{\Kk^1}\subset \Ii_{\Kk^{[0,1]}}$ are the not necessarily disjoint subsets of charts whose footprints intersect $X\times\{0\}$ resp.\ $X\times \{1\}$.

\begin{example} \label{ex:triv}\rm
Let $\Kk= \bigl( \bK_I, \Hat\Phi_{IJ}\bigr)_{I,J\in\Ii_\Kk}$ be an additive weak Kuranishi atlas on $X$.
Then the {\bf product Kuranishi cobordism} $\Kk\times [0,1]$ from $\Kk$ to $\Kk$ is the weak Kuranishi  
cobordism on $X\times [0,1]$ consisting of the product charts $\bK_I\times [0,1]$ and the product coordinate changes $\Hat\Phi_{IJ}\times \id_{[0,1]}$ for $I,J\in\Ii_\Kk$.
Note that $\Kk\times [0,1]$ inherits additivity from $\Kk$. 
Similarly, if $\Kk$ is tame, then so is $\Kk\times [0,1]$.

If $|\Kk|$ is a Kuranishi atlas, so that the Kuranishi neighbourhood $|\Kk|$ is defined, 
then there is a natural bijection from the quotient $|\Kk\times [0,1]|$ to the product $|\Kk|\times [0,1]$.  
This map is continuous.  However, it is not clear that it is always a homeomorphism.\footnote{At this time, we have neither a proof nor a counter example. We hope to resolve this question, but none of the following depends on this.}
Further, when we come to put metrics on the product $\Kk\times [0,1]$ we will certainly sometimes consider metrics that define a topology on $|\Kk\times [0,1]|$ that is not a product.  
Nevertheless we will often denote the realization of $\Kk\times [0,1]$ as  $|\Kk|\times [0,1]$ with the understanding that this is an equivalence of sets, not of topological spaces.
\end{example}

In order to discuss further the collar structure near the boundary of Kuranishi cobordisms, we denote for 
$1\ge\eps>0$
the collar neighbourhoods of $0,1\in[0,1]$ by
\begin{equation}\label{eq:Naleps}
A_\eps^0: = [0,\eps)  \qquad\text{and}\qquad A_\eps^1: = (1-\eps,1] .
\end{equation}
We will see that the footprints of the Kuranishi charts in a Kuranishi cobordism are collared in the following sense.

\begin{defn}\label{def:collared}
We say that an open subset $F\subset X\times [0,1]$ is {\bf collared} 
if there is $\eps>0$ such that for $\al\in \{0,1\}$ we have
$$
F \cap (X\times A^\al_\eps)\ne \emptyset
\;\; \Longleftrightarrow\;\;
F \cap (X\times A^\al_\eps)
= \p^\al F \times A^\al_\eps .
$$
Here we denote by
$$
\partial^\al F :=  \pr_{X} \bigl( F \cap (X\times\{\al\}) \bigr)
$$
the image of the intersection with the ``boundary component'' $X\times\{\al\}$ under its projection to $X$,
noting that this is just a convenient notation, not a topological boundary.
\end{defn}

\begin{remark} \rm \label{rmk:Ceps} 
Since the index set $\Ii_{\Kk^{[0,1]}}$ in Definition~\ref{def:CKS} is finite, there exists a uniform 
{\bf collar width} $\eps>0$ such that all collar embeddings  $\io^\al_I$ are defined on $A^\al = A_\eps^\al$, all coordinate changes are of collar form on $B^\al=A_\eps^\al$, and all charts without collared boundary have footprint contained in $X\times (\eps,1-\eps)$. In particular, all footprints are collared in the sense of Definition~\ref{def:collared}.
\end{remark}

\begin{rmk}\rm  \label{rmk:cobordreal}
Let $\Kk^{[0,1]}$ be a Kuranishi cobordism from $\Kk^0$ to $\Kk^1$.
Its associated categories $\bB_{\Kk^{[0,1]}}, \bE_{\Kk^{[0,1]}}$ with projection, section, and footprint functor, as well as their realizations $|\Kk^{[0,1]}|, |\bE_{\Kk^{[0,1]}}|$ are defined as for Kuranishi atlases without boundary in Section~\ref{ss:Ksdef}.
We will see that these also have collared boundaries with $\eps>0$ from Remark~\ref{rmk:Ceps}~(i).

\begin{itemlist}
\item
We can think of the virtual neighbourhood $|\Kk^{[0,1]}|$ of $X\times[0,1]$ as a ``cobordism'' from $|\Kk^0|$ to $|\Kk^{1}|$ in the following sense:
There are natural functors $\bB_{\Kk^\al}\times A^\al_\eps\to \bB_{\Kk^{[0,1]}}$ given by the inclusions $\iota^\al_I : U^\al_I\times A^\al_\eps \hookrightarrow U^{[0,1]}_{I}$
on objects and $\iota^\al_I : U^\al_{IJ}\times A^\al_\eps\hookrightarrow U^{[0,1]}_{IJ}$
on morphisms, where $A^\al_\eps$ is defined in \eqref{eq:Naleps}.
The axioms on the interaction of the coordinate changes with the collar neighbourhoods
then imply that the functors map to full subcategories
that split $\bB_{\Kk^{[0,1]}}$ in the sense that there are no morphisms between any other object and this subcategory.
Hence they induce ``collar neighbourhoods'' of the ``boundaries'' $|\Kk^\al|$ on the topological realization $|\Kk^{[0,1]}|$, 
 i.e.\ topological embeddings
$$
\rho^0:  |\Kk^0| \times [0,\eps) \longhookrightarrow |\Kk^{[0,1]}| , \qquad
\rho^1:  |\Kk^1| \times (1-\eps,1] \longhookrightarrow |\Kk^{[0,1]}|
$$
such that for all $0<\eps'<\eps$ at $\al=0$ (and similarly for $\al=1$) we have
$$
\partial \bigl(  |\Kk^{[0,1]}| \less  \rho^0(|\Kk^0| \times [0,\eps') ) \bigr) =  \rho^0(|\Kk^0| \times \{\eps'\}).
$$
However, recall from Example~\ref{ex:triv} that
the topology on the collars $|\Kk|\times A^\al_\eps$ is the quotient topology from $|\Kk\times A^\al_\eps|$
which may not be the product topology.
In view of this remark, we shall write 
$$
\p^\al |\Kk^{[0,1]}| \,:= \; \rho^\al\bigl(|\Kk^\al|\times \{\al\}\bigr) \;\subset\; |\Kk^{[0,1]}|  
$$ 
for the $\al$-boundary of the Kuranishi cobordism neighbourhood $|\Kk^{[0,1]}|$, which is homeomorphic to the Kuranishi neighbourhood $|\Kk^\al|$ of the boundary $\p^\al\Kk^{[0,1]}$ via $\rho^\al(\cdot,\al)$.

\item
The obstruction bundle ``with boundary'' $|\pr_{\Kk^{[0,1]}}|: |\bE_{\Kk^{[0,1]}}|\to |\Kk^{[0,1]}|$ can also be thought of as a \lq\lq cobordism" between the obstruction bundles
$|\pr_{\Kk^\al}|: |\bE_{\Kk^\al}|\to |\Kk^\al|$ in the sense that $\rho^{\al\;*}  |\bE_{\Kk^{[0,1]}}| =  |\bE_{\Kk^{\al}}| \times A^\al_\eps$.
\smallskip
\item
The embeddings $\rho^\al$ extend the natural map between footprints
$$
|s_{\Kk^{[0,1]}}|^{-1}(0)
\;\;
\xleftarrow{\iota_{\Kk^{[0,1]}}}\;\;
X\times A_\eps^\al
\;\;\xrightarrow{\iota_{\Kk^{\al}}\times{\rm id}}\;\;
|s_{\Kk^{\al}}|^{-1}(0)\times A_\eps^\al  .
$$
\item
The footprint functor to $X\times[0,1]$ for $\Kk^{[0,1]}$ induces a continuous surjection
$$
{\rm pr}_{[0,1]}\circ \psi_{\Kk^{[0,1]}}
: \; s_{\Kk^{[0,1]}}^{-1}(0) \;\to\; X\times [0,1]  \;\to\;  [0,1] .
$$
In general we do not assume that this extends to a functor $\bB_{\Kk^{[0,1]}}\to [0,1]$.
However, all the Kuranishi cobordisms that we construct explicitly do have this property.
\end{itemlist}
\end{rmk}

\begin{lemma}\label {le:cob0}
Let $\Kk^{[0,1]}$ be a tame Kuranishi cobordism.
Then its realization $|\Kk^{[0,1]}|$ has the Hausdorff, homeomorphism, and linearity properties stated in Theorem~\ref{thm:K}.
\end{lemma}
\begin{proof}
These properties are proven by precisely the same arguments as in Proposition~\ref{prop:Khomeo} and Proposition~\ref{prop:linear}.  The fact that some charts have collared boundaries is irrelevant in this context.
\end{proof}

We now turn to the question of constructing additive cobordisms.
As we saw, additivity is an essential hypothesis in Proposition~\ref{prop:Khomeo}, which establishes that $|\Kk|$ has the Hausdorff and homeomorphism properties.  However, note that when manipulating charts
other than by shrinking the domains, one may easily destroy this property.  For example, when constructing cobordisms one must guard against taking two
products of the same basic chart, e.g.\ $\bK_i\times [0,\frac 13)$ and $\bK_i\times (\frac 14,\frac 12)$.
The natural transition chart for the overlap of footprints $F_i\times (\frac 14,\frac 13)$ is $\bK_i\times (\frac 14,\frac 13)$, but it fails additivity since $E_i \not\cong E_i \times E_i$ unless the obstruction space is trivial.
This means that in the following we have to construct cobordisms with great care.

\begin{lemma}\label{le:cobord1}
\begin{enumerate}
\item
Additive cobordism is an equivalence relation on the set of additive weak Kuranishi atlases on a fixed compact space $X$.
\item
Any two commensurate additive weak Kuranishi atlases are additively cobordant.
\end{enumerate}
\end{lemma}
\begin{proof} 
Given an additive weak Kuranishi atlas $\Kk$ on $X$, the product atlas $\Kk\times [0,1]$ on $X\times[0,1]$ of Example~\ref{ex:triv} is an additive weak Kuranishi cobordism from $\Kk$ to $\Kk$. This shows that the cobordism relation is reflexive.

Next, suppose that $\Kk^{[0,1]}$ is an additive weak Kuranishi cobordism from $\Kk^0$ to $\Kk^1$ as in Definition~\ref{def:Kcobord}. We may compose every footprint map $\psi^{[0,1]}_{I'}$ with the reflection $X\times[0,1]\to X\times[0,1], (x,t) \mapsto (x, 1-t)$ to define another additive weak Kuranishi atlas for $X\times [0,1]$, which restricts to $\Kk^1$ near $X\times\{0\}$ and to $\Kk^0$ near $X\times\{1\}$, thus proving that the cobordism relation is symmetric.

Similarly, given an additive weak Kuranishi cobordisms from $\Kk^1$ to $\Kk^2$, we may compose the footprint maps with the shift $X\times[0,1]\to X\times[1,2], (x,t) \mapsto (x, 1+t)$ to construct an additive weak Kuranishi atlas with boundary $\Kk^{[1,2]}$ for $X\times [1,2]$, which restricts to $\Kk^1$ on $X\times\{1\}$ and to $\Kk^2$ on $X\times\{2\}$.
We may concatenate this with any Kuranishi cobordism $\Kk^{[0,1]}$ from $\Kk^0$ to $\Kk^1$ to obtain a Kuranishi atlas with boundary $\Kk^{[0,2]}$ on $X\times [0,2]$, which restricts to $\Kk^0$ near $X\times\{0\}$ and to $\Kk^2$ near $X\times\{2\}$.
More precisely, we define $\Kk^{[0,2]}$ as follows.
\begin{itemlist}
\item
The index set $\displaystyle \; \Ii_{\Kk^{[0,2]}} :=\;\Ii_{[0,1)}  \;\sqcup\; \Ii_{\Kk^1} \;\sqcup\; \Ii_{(1,2]} \;$ is given by $\Ii_{[0,1)}:=\Ii_{[0,1]}\less\iota^1(\Ii_{\Kk^1})$, $\Ii_{(1,2]}:=\Ii_{[1,2]}\less\iota^1(\Ii_{\Kk^1})$.
\item
The charts are
$\bK^{[0,2]}_{I} := \bK^{[0,1]}_{I}$ for $I\in\Ii_{[0,1)}$, and $\bK^{[0,2]}_{I} := \bK^{[1,2]}_{I}$ for $I\in\Ii_{(1,2]}$.
For $I\in\Ii_{\Kk^1}$ denote by
$I^{01}=\iota^1(I)\in\Ii_{\Kk^{[0,1]}}$, $I^{12}=\iota^1(I)\in\Ii_{\Kk^{[1,2]}}$
the labels of the charts that restrict to $\bK_I$.
In particular, this implies $\partial^1 U^{[0,1]}_{I^{01}} = U^1_I =\partial^1 U^{[1,2]}_{I^{12}}$.
Then define the glued chart (possibly with collared boundary at $X\times\{0\}$ or $X\times\{2\}$)
$$
\qquad
\bK^{[0,2]}_{I} :=
 \bK^{[0,1]}_{I^{01}}\underset{\scriptscriptstyle U^1_I}{\cup} \bK^{[1,2]}_{I^{12}} :=
\left( U^{[0,1]}_{I^{01}} \underset{\scriptscriptstyle U^1_I}{\cup} U^{[1,2]}_{I^{12}} \, , \, E^1_I \, ,
\left\{ \begin{aligned}
s^{[0,1]}_{I^{01}}\;,\; \psi^{[0,1]}_{I^{01}}  \quad &\text{on} \; U^{[0,1]}_{I^{01}}\\
s^{[1,2]}_{I^{12}}\;,\; \psi^{[1,2]}_{I^{12}}  \quad &\text{on} \;U^{[1,2]}_{I^{12}}
\end{aligned} \right\}
\right).
$$
Here the domain is the boundary connected sum
$$
U^{[0,2]}_{I} \;:=\; U^{[0,1]}_{I^{01}} \underset{\scriptscriptstyle U^1_I}{\cup}  U^{[1,2]}_{I^{12}}
\;:=\;
\quotient{ U^{[0,1]}_{I^{01}} \sqcup U^{[1,2]}_{I^{12}} }{ \scriptstyle
 \iota^1_{I^{01}}(x,1) \sim \iota^1_{I^{12}}(x,1)\quad \forall x\in U^1_I
} .
$$
Due to the collar requirements, this domain inherits a smooth structure with an embedded
product $U^1_I\times (1-\eps,1+\eps)$ for some $\eps>0$.
Moreover the sections and footprint maps fit smoothly since their pullbacks to this product agree
on $U^1_I\times \{1\}$.
Finally, the bundles are identical $E^{[0,1]}_{I^{01}} = E^1_I = E^{[1,2]}_{I^{12}}$.
\vspace{.09in}
\item
The coordinate changes are $\Hat\Phi^{[0,2]}_{IJ} := \Hat\Phi^{[0,1]}_{IJ}$ for $I,J\in\Ii_{[0,1)}$, and $\Hat\Phi^{[0,2]}_{IJ} := \bK^{[1,2]}_{IJ}$ for $I,J\in\Ii_{(1,2]}$, and the following.
\vspace{.09in}
\begin{itemize}
\item[-]
For $I,J\in\Ii_{\Kk^1}$ the
coordinate charts corresponding to
$I^{01}, J^{01}\in\Ii_{\Kk^{[0,1]}}$, $I^{12},J^{12}\in\Ii_{\Kk^{[1,2]}}$
fit together to give a glued coordinate change
(possibly with collared boundary at $X\times\{0\}$ or $X\times\{2\}$)
$$
\qquad
\Hat\Phi^{[0,2]}_{IJ} := \left( U^{[0,1]}_{I^{01}J^{01}} \underset{\scriptstyle U^1_{IJ}}\cup U^{[1,2]}_{I^{12}J^{12}} \; , \,
\left\{ \begin{aligned}
\phi^{[0,1]}_{I^{01}}  \;\quad &\text{on} \; U^{[0,1]}_{I^{01}J^{01}}\\
\phi^{[1,2]}_{I^{12}}  \;\quad &\text{on} \; U^{[1,2]}_{I^{12}J^{12}}
\end{aligned}
\right\}
 \, , \,
\Hat\phi^1_{IJ} \,
\right) .
$$
Here the embeddings of domains fit smoothly since as before their pullbacks to the product $U^1_{IJ}\times (1-\eps,1+\eps)\subset U^{[0,2]}_{IJ}$ agree on $U^1_{IJ}\times \{1\}$.
Moreover, the linear embeddings are identical $\Hat\phi^{[0,1]}_{I^{01}J^{01}} = \Hat\phi^1_{IJ} = \Hat\phi^{[1,2]}_{I^{12}J^{12}}$.
\vspace{.09in}
\item[{-}]
For $J\in\Ii_{[0,1)}$ and $I\in\Ii_{\Kk^1}$ corresponding to $I^{01}\in\Ii_{\Kk^{[0,1]}}$ with $I^{01}\subsetneq J$ the coordinate change $\Hat\Phi^{[0,2]}_{IJ} := \Hat\Phi^{[0,1]}_{I^{01} J}$
is well defined with domain $U^{[0,1]}_{I^{01}J} \subset U^{[0,2]}_{IJ}$;
similarly for $J\in\Ii_{(1,2]}$,  $I\in\Ii_{\Kk^1}$.\vspace{.09in}
\end{itemize}
\end{itemlist}

Note that we need not construct coordinate changes from $I\in\Ii_{[0,1)}$ (or $I\in\Ii_{(1,2]}$) to $J\in\Ii_{\Kk^1}$
since in these cases $F_J$ is not a subset of  $F_I$.
Now we may define the basic charts in $\Kk^{[0,2]}$ to consist of the basic charts in $\Ii_{[0,1)}$ and $\Ii_{(1,2]}$ whose footprints are disjoint from $X\times\{1\}$, together with one glued chart for each basic chart in $\Ii_{\Kk^1}$ (which is constructed from a pair of charts in ${\Kk^{[0,1]}}$ and ${\Kk^{[1,2]}}$ with matching collared boundaries).
The further charts and coordinate changes constructed above then cover exactly the overlaps of the new basic charts, since the charts from $\Ii_{[0,1)}$ have no overlap with those arising from $\Ii_{(1,2]}$.
The weak cocycle condition  for charts or coordinate changes in $\Ii_{[0,1)}\sqcup\Ii_{(1,2]}$ then follows directly from the corresponding property of $\Kk^{[0,1]}$ and $\Kk^{[1,2]}$.
Furthermore,
for $I\in\Ii_{\Kk^{1}}$ the glued chart $\bK^{[0,2]}_{I}=\bK^{[0,1]}_{I^{01}}\underset{\bK^1_I}{\cup} \bK^{[1,2]}_{I^{12}}$ has restrictions (up to natural pullbacks)
\begin{align*}
&\bK^{[0,2]}_{I}|_{{\rm int}(U^{[0,1]}_{I^{01}})} = \bK^{[0,1]}_{I^{01}}|_{{\rm int}(U^{[0,1]}_{I^{01}})}, \qquad
\bK^{[0,2]}_{I}|_{{\rm int}(U^{[1,2]}_{I^{12}})} = \bK^{[1,2]}_{I^{12}}|_{{\rm int}(U^{[1,2]}_{I^{12}})}, \\
&\bK^{[0,2]}_{I}|_{U^1_I \times (1-\eps,1+\eps)} = \bK^1_{I} \times (1-\eps,1+\eps)  .
\end{align*}
The cocycle  condition for any tuple of coordinate changes can be checked separately for these restrictions (which cover the entire domain of $\bK^{[0,2]}_{I}$) and hence follow from the corresponding property of $\Kk^{[0,1]}$, $\Kk^{[1,2]}$, and $\Kk^{1}$.

Similarly, the additivity condition for the charts in $\Ii_{[0,1)}$ or $\Ii_{(1,2]}$ follows directly from the additivity of $\Kk^{[0,1]}$ or $\Kk^{[1,2]}$.
However, the additivity for a chart $\bK^{[0,2]}_I$ with $I\in\Ii_{\Kk^1}$ is a little more subtle in that it requires additivity of the obstruction bundle $E^{02}_{I}$ with respect to all basic charts $\bK^{[0,2]}_i$ whose footprint contains $F^{[0,2]}_{I}=F^{[0,1]}_{I^{01}}\cup F^{[1,2]}_{I^{12}}$. However, note that these are exactly the glued basic charts corresponding to the basic charts $\bK^1_i$ whose footprint contains $F^{1}_{I}$. So additivity follows from additivity for $\bK^1_I$,
$$
E^{[0,2]}_{I} \;=\; E^{1}_{I} \;=\; \underset{F^1_i \supset F^1_I}{\bigoplus}  \Hat\phi^{1}_{i I}(E^{1}_i) \;=\; \underset{F^{[0,2]}_i \supset F^{[0,2]}_{I}}{\bigoplus}  \Hat\phi^{[0,2]}_{i I}(E^{[0,2]}_i).
$$
Thus $\Kk^{[0,2]}$ is an additive weak Kuranishi cobordism with restrictions 
$$
\p^0\Kk^{[0,2]}=\p^1\Kk^{[0,1]}=\Kk^0,\qquad \p^2\Kk^{[0,2]}=\p^2\Kk^{[1,2]}=\Kk^2.
$$
Here we write $\p^2$ for the restriction over $X\times\{2\}$ defined analogously to $\p^0,\p^1$.
Finally, we compose the footprint maps of $\Kk^{[0,2]}$ with the rescaling $X\times[0,2]\to X\times[0,1], (x,t) \mapsto (x, \tfrac 12 t)$ to obtain an additive weak Kuranishi cobordism from $\Kk^0$ to $\Kk^2$. 

\MS

To prove (ii) consider additive weak Kuranishi atlases $\Kk^0,\Kk^1$ and a common additive weak extension $\Kk^{01}$.  Then an additive weak Kuranishi cobordism $\Kk^{[0,1]}$ from $\Kk^0$ to $\Kk^1$ is given by
\begin{itemize}
\item
index set
$\displaystyle \;
\Ii_{\Kk^{[0,1]}} := \Ii_{\Kk^{01}}  =  \bigl\{ I\subset\Nn^{01} \,\big|\, {\textstyle\bigcap_{i\in I}} F^{01}_i \neq \emptyset \bigr\}$;
\item
charts $\displaystyle\; \bK^{[0,1]}_I := \bK^{01}_I \times A_I$
with $A_I= [0,\tfrac 23)$ for $I\subset\Nn^0$, $A_I = (\tfrac 13,1]$ for $I \subset \Nn^1$, and $A_I=(\tfrac 13,\tfrac 23)$ otherwise;
\item
coordinate changes
$\displaystyle\; \Hat\Phi^{[0,1]}_{IJ} := \Hat\Phi^{01}_{IJ} \times (A_I\cap A_J)$.
\end{itemize}
This proves (ii) since additivity and 
weak
cocycle conditions follow from the corresponding properties of $\Kk^{01}$. In particular, note that additivity makes use of the fact that $\Kk^{[0,1]}$ has the same index set as the additive Kuranishi atlas $\Kk^{01}$.
\end{proof}

The final  task in this section is to construct tame cobordisms between different tame shrinkings in order to establish the uniqueness claimed in Theorem~\ref{thm:K}.  It will also be useful to have suitable metrics on these cobordisms, since they are used in the construction of perturbations.  We therefore begin by discussing the notion of metric tame Kuranishi cobordism.   One difficulty here is that we are dealing with an arbitrary distance function, not a length metric such as a Riemannian metric. Hence we must prove various elementary results that would be clear in the Riemannian case.

\begin{defn}\label{def:mCKS}
A {\bf metric tame Kuranishi cobordism} on $X\times [0,1]$ is a tame Kuranishi cobordism $\Kk^{[0,1]}$ equipped with a  metric $d^{[0,1]}$ on $\bigl|\Kk^{[0,1]}\,\!\bigr|$ that satisfies the admissibility conditions of Definition~\ref{def:metric} and has a metric collar as follows:

There is $\eps>0$ such that for $\al=0,1$ 
with $\Kk^\al:=\p^\al\Kk^{[0,1]}$
the collaring maps 
$\rho^\al: |\Kk^\al|\times A^\al_\eps\to \bigl|\Kk^{[0,1]}\bigr|$ of Remark~\ref{rmk:cobordreal} are defined and pull back $d^{[0,1]}$ to the product metric 
\begin{equation} \label{eq:epsprod}
(\rho^\al)^* d^{[0,1]}
\bigl((x,t),(x',t')\bigr) \;=\; d^\al(x,x')+ |t'-t| \qquad \text{on} \;\; |\Kk^\al|\times A^\al_\eps ,
\end{equation}
where the metric $d^\al$ on $|\Kk^\al| = |\p^\al\Kk^{[0,1]}|$ is given by pullback of the restriction of $d^{[0,1]}$ to $\p^\al |\Kk^{[0,1]}| = \rho^\al\bigl(|\Kk^\al|\times\{\al\}\bigr)$,
which we denote by
$$
d^\al \,:=\; d^{[0,1]}|_{|\p^\al \Kk^{[0,1]}|} \,:=\; \rho^\al(\cdot, \al)^* d^{[0,1]}.
$$ 
In addition, we require
for all $y\in |\Kk^{[0,1]}|\less \rho^\al  \bigl( |\Kk^\al|\times A^\al_\eps\bigr)$ 
\begin{equation}\label{eq:epscoll}
d^{[0,1]}\bigl( y , \rho^\al(x,\al + t )\bigr) \;\ge\; 
\eps - |t|
\qquad \forall \;
(x,\al + t)\in   |\Kk^\al|\times A^\al_\eps  .
\end{equation}

More generally, we call a metric on $\bigl|\Kk^{[0,1]}\,\!\bigr|$ {\bf admissible} if it satisfies the conditions of Definition~\ref{def:metric}, and {\bf $\eps$-collared} if it satisfies \eqref{eq:epsprod} and \eqref{eq:epscoll}.
\end{defn}

Condition \eqref{eq:epscoll} controls the distance between points $\rho^\al(x,\al + t)$ in the collar and 
points $y$ outside of the collar.  
In particular, if $\de<\eps-|t|$, then the $\de$-ball around $\rho^\al(x,\al + t)$ is contained in the $\eps$-collar, while the $\de$-ball around $y$ does not intersect the $|t|$-collar $\rho^\al\bigl(|\Kk^{\al}|\times A^\al_{|t|}\bigr)$.

\begin{example} \label{ex:mtriv}\rm
\NI (i) 
Any admissible metric $d$ on $|\Kk|$ for a Kuranishi atlas $\Kk$ induces an admissible 
collared metric $d + d_\R$ on $|\Kk\times[0,1]|\cong |\Kk| \times [0,1]$, given by 
$$
\bigr(d + d_\R\bigl)\bigl( (x,t) , (x',t') \bigr) = d(x,x') + |t'-t| .
$$ 
For short, we call $d + d_\R$ a {\bf product metric.} 

\NI (ii)
Let $d$ be an admissible collared metric on $|\Kk|$ for a general Kuranishi cobordism $\Kk$, and let $b$ be an upper bound of $d^\al:=d\big|_{|\p^\al\Kk|}$ for $\al=0,1$. Then we claim that for any $\ka>b$ the truncated metric $\min(d,\ka)$ given by $(x,y) \mapsto \min(d(x,y),\ka)$ is an admissible $\eps'$-collared metric for $\eps':=\min(\eps,\ka-b)$. 
Indeed, the metric $\min(d,\ka)$  is also admissible because it induces the same topology on each $U_I$ as $d$.
Further, the product form on the collar is preserved since
\begin{equation}\label{eq:trunc}
|t-t'| < \ka-b \qquad\Longrightarrow\qquad  d^\al(x,x') + |t-t'| \;\le\; b + |t-t'| \;<\; \ka ,
\end{equation}
and \eqref{eq:epscoll} holds since for $(x,\al + t)\in  |\p^\al\Kk|\times A^\al_{\eps'}$ we have 
\[
d\bigl( y , \rho^\al(x,\al + t) \bigr) \;\ge\; 
\begin{cases}
\eps - |t| \ge \eps' - |t|  &\quad\text{if} \; y\in |\Kk|\less \rho^\al\bigl( |\p^\al\Kk|\times A^\al_\eps\bigr) , \\
|t' - t| \ge \eps' - |t|  &\quad\text{if} \; y=\rho^\al(x',\al+t')\in \rho^\al\bigl( |\p^\al\Kk|\times ( A^\al_\eps\less A^\al_{\eps'} )\bigr) 
\end{cases}
\]
by \eqref{eq:epscoll} for $d$ resp.\ the product form of the metric on the $\eps$-collar, and moreover $\ka > \ka - b - t \geq \eps' - t$.
Finally, the restrictions of this truncated metric are by $\ka>b$
$$
\min(d,\ka)\big|_{|\p^\al\Kk|} \;=\; \min(d^\al,\ka) \;=\; d^\al \qquad\text{for}\; \al=0,1 .
$$

\NI (iii)  
Let $(\Kk,d)$ be a metric tame Kuranishi cobordism with collar width $\eps>0$. Then for any $0<\de<\eps$ the $\de$-neighbourhood of the inclusion of the Kuranishi space $X\times[0,1]$, 
\begin{equation}\label{eq:metcoll2}
\Ww_\de: = B_\de\bigl(\io_\Kk(X\times [0,1])\bigr) = \bigl\{y\in |\Kk| \, \big| \, 
\exists z\in  \io_\Kk(X\times [0,1]) : d(y, z)<\eps\bigr\}
\end{equation}
has collar form with collars of width $\eps-\de$, that is
$$
\Ww_\de \cap \rho^\al\bigl(|\p^\al\Kk|\times A^\al_{\eps-\de}\bigr) 
= \rho^\al\bigl(B_\de(\io_{\p^\al\Kk}(X))\times A^\al_{\eps-\de}\bigr).
$$
This holds because  \eqref{eq:epscoll} implies that $B_\de\bigl(\io_\Kk\bigl(X\times ([0,1]\less A^\al_\eps) \bigr)\bigr)$ does not intersect the collar $\rho^\al\bigl(\p^\al\Kk\times A^\al_{\eps-\de}\bigr)$, and on the other hand
$B_\de\bigl(\io_\Kk\bigl(X\times A^\al_\eps \bigr)\bigr)\cap \rho^\al\bigl(|\p^\al\Kk|\times A^\al_\eps\bigr)$ has product form because the metric in the collar is a product metric.
\end{example}

Recall that 
an admissible metric $d$ on the virtual neighbourhood $|\Kk|$ of a Kuranishi cobordism 
has the property that  its pullback to each chart $U_I$ defines the given topology on that manifold.
However, 
even if the metric is also collared,
this implies little else about the induced topology on $|\Kk|$.
For example, if we consider a product cobordism $\Kk\times [0,1]$, then an admissible 
(collared)
metric $d$ on $|\Kk\times [0,1]|$ need not define the product topology on  $|\Kk\times [0,1]| \cong |\Kk|\times [0,1]$;
 all we know is that the pullback metrics $d_I$ on each domain $U_I\times [0,1]$ give the product topology.
The next lemma gives useful techniques for converting such a metric to one of product form in the collar,
and for proving uniqueness of admissible metrics up to cobordism.
Here, as in Lemma~\ref{le:cobord1}, we will consider product Kuranishi atlases $\Kk\times A$ for various intervals $A\subset \R$, that is with domain category $\Obj_{\bB_\Kk}\times A$. Their virtual neighbourhoods $|\Kk\times A|$ are canonically identified with $|\Kk|\times A$ by Remark~\ref{rmk:cobordreal}, 
and hence have continuous injections  $|\Kk|\times B \hookrightarrow |\Kk\times A|$ for any $B\subset A$ with respect to the quotient topologies on the realizations of the categories $\Obj_\Kk$ resp.\ $\Obj_\Kk\times A$.
However, for admissible metrics on $|\Kk|$ and $|\Kk\times A|$, we do not require these injections to remain continuous.

\begin{prop} \label{prop:metcoll}
Let $\Kk$ be a metrizable tame Kuranishi atlas.
\begin{enumerate}
\item 
Given any admissible metric $d$ on 
$|\Kk \times [0,1]|$
and $\eps>0$, 
there is 
another admissible metric $D$ on 
$|\Kk \times [0,1]|$
that restricts to $d$ on  
$|\Kk| \times [\eps,1-\eps]$
and restricts to the product 
$d_\al + d_\R$ on $|\Kk|\times A^\al_\eps$, where $d_\al$ is the restriction of $d$ to $|\Kk|\times \{\al\}\cong |\Kk|$.
Moreover, $D$
is $\eps$-collared in the sense of \eqref{eq:epscoll}.

\item  
Let $d_1$ be an admissible metric on $|\Kk|$, and suppose for $A = [0,1], [1,2]$ that $d_A$ are admissible metrics on $|\Kk \times A|$ that for some $\ka>0$ restrict on $|\Kk| \times \bigl([1-\ka,1+\ka]\cap A\bigr)$ to the product metric $d_1+d_\R$. Then 
there exists 
an admissible 
$\frac\ka2$-collared
metric $D$ on $|\Kk \times [0,2]|$
whose boundary restrictions  are
$$
\qquad
D\big|_{|\Kk|\times \{0\}} = \min(d_{[0,1]}\big|_{|\Kk|\times \{0\}}, \tfrac \ka 2)
\quad\text{and}\quad
D\big|_{|\Kk|\times \{2\}} = \min(d_{[1,2]}\big|_{|\Kk|\times \{2\}}, \tfrac \ka 2).
$$

\item  
Suppose that $d$ and $d'$ are admissible metrics on $|\Kk|$, where $d'$ is bounded by~$1$.
Then, for any $0<\eps
<\frac 14$,
there is an admissible $\eps$-collared metric $D$ on  $|\Kk\times [0,1]|$ that 
restricts to  $d +  d_\R$ on $|\Kk|\times [0,\eps]$ and to $d+d' + d_\R$ on $|\Kk|\times [1-\eps,1]$.

\item  
If $d^0$ and $d^1$ are any two admissible metrics on $|\Kk|$, then 
there exists an admissible collared metric $D$ on $|\Kk\times[0,1]|$ with restrictions $D|_{|\Kk\times\{\al\}|}=d^\al$ at the boundaries $\p^\al |\Kk\times[0,1]|=|\Kk|\times\{\al\}\cong|\Kk|$ for $\al=0,1$.

\item 
Finally, suppose that $\Kk^{[0,1]}$ is a Kuranishi cobordism and $d$ is an admissible (not necessarily collared) metric on $|\Kk^{[0,1]}|$.
Then there exists an admissible collared metric $D$ on $|\Kk^{[0,1]}|$ with $D|_{|\p^\al \Kk^{[0,1]}|}=d|_{|\p^\al \Kk^{[0,1]}|}$ for $\al=0,1$.
\end{enumerate}

\end{prop}

\begin{proof}  
The metric required by (i) can be obtained by first rescaling the given admissible metric to $|\Kk|\times [\eps,1-\eps]$, and then making collaring constructions to extend it to $|\Kk|\times[\eps,1]$ and then to $|\Kk|\times[0,1]$.
We will moreover see that the collar width $\eps$ plays no special role other than complicating the notation. So it suffices to consider a given admissible metric $d$ on $|\Kk|\times[0,1]$ and extend it by a collaring construction to a metric $D$ on $|\Kk|\times[0,2]$ that restricts to $d_1 + d_\R$ on $|\Kk|\times[1,2]$ with the metric $d_1:=d|_{|\Kk|\times\{1\}}$ on $|\Kk|\times\{1\}\cong|\Kk|$ and satisfying \eqref{eq:epscoll} with $\eps=1$, that is $D(y,(x,2-t))\ge 1 - t$ for $y \in |\Kk|\times[0,1]$ and $0\le t < 1$.
This metric can be constructed by symmetric extension of
\begin{equation}\label{eq:DDD}
D\bigl( (x,t) , (x',t') \bigr) \; :=\; 
\left\{\begin{array}{ll}
 d\bigl( (x,t) , (x',t') \bigr)  &\mbox{ if } t,t'\le 1,\\ 
 d_1(x,x') + |t-t'|, 
 &\mbox{ if } t,t'\ge 1,\\
d\bigl( (x,t) , (x',1) \bigr) + |t'-1| &\mbox{ if } t \le 1 \le t'.
\end{array}\right.
\end{equation}
This is well defined, positive definite, and symmetric. The triangle inequality 
\begin{equation}\label{eq:tri}
D\bigl( (x,t) , (x'',t'') \bigr)\le D\bigl( (x,t) , (x',t') \bigr) + D\bigl( (x',t') , (x'',t'') \bigr)
\end{equation}
directly follows from the construction in case $t,t',t''\le 1$ or $t,t',t''\ge 1$.
In the case $t' \le 1 < t, t''$ it follows from the triangle inequalities for both $d$ and $d_\R$,
\begin{align*}
 D\bigl( (x,t) , (x'',t'') \bigr)
&\;=\;
d\bigl( (x,1) , (x'',1) \bigr) + |t''-t| \\
&\;\le\;
d\bigl( (x,1) , (x',t') \bigr)  + d\bigl( (x',t') , (x'',1) \bigr) + |t-1|+ |t''-1| \\
&\;=\; D\bigl( (x,t) , (x',t') \bigr) + D\bigl( (x',t') , (x'',t'') \bigr) .
\end{align*}
Similarly, if $t\le 1<t', t''$ we have
\begin{align*}
D\bigl( (x,t) , (x'',t'') \bigr)&\; =\; d\bigl( (x,t) , (x'',1) \bigr) + |t''-1| \\
&\;\le
d\bigl( (x,t) , (x',1) \bigr) + d\bigl( (x',1) , (x'',1) \bigr)+ |t'-1|+ |t''-t'| \\
&\; =\; D\bigl( (x,t) , (x',t') \bigr) + D\bigl( (x',t') , (x'',t'') \bigr).
\end{align*}
In the case $t, t''\le 1< t'$ the triangle inequality for $d$ implies that for $D$,
\begin{align*}
D\bigl( (x,t) , (x'',t'') \bigr)&\; =\; d\bigl( (x,t) , (x'',t'') \bigr) \\
&\;\le\; d\bigl( (x,t) , (x',1) \bigr) + d\bigl( (x',1) , (x'',t'') \bigr)+ 2|t'-1|\\
&\; =\; D\bigl( (x,t) , (x',t') \bigr) + D\bigl( (x',t') , (x'',t'') \bigr).
\end{align*}
Finally, for  $t', t''\le 1< t$ we have
\begin{align*}
D\bigl( (x,t) , (x'',t'') \bigr)&\; =\; d\bigl( (x,1) , (x'',t'') \bigr)  + |t-1| \\
&\;\le\; d\bigl( (x,1) , (x',t') \bigr) + d\bigl( (x',t') , (x'',t'') \bigr)+ |t-1| \\
&\; =\; D\bigl( (x,t) , (x',t') \bigr) + D\bigl( (x',t') , (x'',t'') \bigr).
\end{align*}
This proves that $D$ is a well defined metric on $|\Kk|\times [0,2]$. Boundedness follows from that of $d$.
We next check that the topology given by the pullbacks $D_I$ of $D$ to $U_I\times [0,2]$ is the standard product topology. Since $d$ pulls back to the product topology on $U_I\times [0,1]$ by hypothesis, 
as does $d_1+d_\R$ on $U_I\times [1,2]$, the topology given by $D_I$ restricts to the product topology on both $U_I\times [0,1]$ and on $U_I\times [1,2]$.  
On the other hand, the product topology on 
$U_I\times [0,2]$ is 
the quotient topology obtained from
the product topologies 
on the disjoint union
$U_I\times [0,1]
\;\sqcup\;
U_I\times [1,2]$ by identifying the two copies of $U_I\times \{1\}$. 
Since this is precisely the topology given by  $D_I$, the metric $D$ is admissible as claimed.
Finally, the collaring requirement \eqref{eq:epscoll} holds since for $y=(x,t)\in |\Kk|\times [0,1]$ and 
$(x',2-t)\in  |\Kk|\times (1,2]$ we have $D\bigl( y , (x',2-t) \bigr) = d\bigl( y , (x',1) \bigr) + |2-t - 1| \ge 1-t$.
This completes the proof of (i).

To prove (ii), we first replace each metric $d_A$ by $\min(d_A,\frac {\ka}2)$
so that the metrics $d_A$ and $d_1$ are bounded by $\frac \ka 2$.
We next replace each $d_A$ 
by the metric $D_A'$ constructed as in (i) that 
restricts to $d_A$ on 
$|\Kk| \times \bigl(A\less [1-\ka,1+\ka]\bigr)$,
and to  $d_1+d_\R$ on $|\Kk| \times \bigl([1-\ka,1+\ka]\cap A\bigr)$, 
and for $A=[0,1]$ (and similarly for $A=[1,2]$) satisfies
$$
D_A'\bigl( (x,t),(x',t') \bigr) =  d_A\bigl( (x,t),(x,1-\ka)\bigr) + |t'-(1-\ka)| \qquad \forall\;  t\le 1-\ka \le t'\le 1.
$$ 
Next we set $D_A: = \min(D_A',\ka)$, which by Example~\ref{ex:mtriv} is still $\frac \ka 2$-collared. 
Finally, we claim that
$$
D\bigl((x,t),(x',t')\bigr): = \left\{
\begin{array}{ll} D_A\bigl((x,t),(x',t')\bigr)  &\mbox{ if } t,t'\in A,\\
\min
\bigl( \, d_1(x,x') +|t-t'| \, ,\, \ka \,\bigr) \phantom{\int_A^B} \!\!\!\!
 &\mbox{ if } 
t,t'\in [1-\ka,1+\ka], \\
\ka &\mbox{ otherwise }
\end{array}\right.
$$
is the required metric on $|\Kk\times [0,2]|$.
Indeed, $D$ is well defined since $D_A$ restricts to $\min(d_1+d_\R,\ka)$ on 
 $|\Kk| \times \bigl([1-\ka,1+\ka]\cap A\bigr)$ by construction.
Further $D$ has the required restrictions, and is symmetric,
positive definite, 
and bounded by $\ka$. To see  
that it satisfies the triangle inequality \eqref{eq:tri} we need only check triples with $D\bigl( (x,t) , (x',t') \bigr) + D\bigl( (x',t') , (x'',t'') \bigr)<\ka$, and (by symmetry) $t\le t''$. 
The proof of (i) shows that $D$ satisfies \eqref{eq:tri} whenever $t,t',t''\le 1+\ka$ or $t,t',t''\ge 1-\ka$.  
Otherwise at least one of $t,t',t''$ is less than $1-\ka$ while another is larger than $1+\ka$. 
This means that the points in at least one of the pairs $\{t,t'\}$, $\{t,t''\}$, and $\{t',t''\}$ lie in different components of the complement of the interval $[1-\ka,1+\ka]$,
and thus the corresponding points in $|\Kk\times[0,2]|$
have distance $\ka$. 
In particular, \eqref{eq:tri} holds by the upper bound on $D$ except in the second case. Assuming w.l.o.g.\ $t\le t''$, this reduces our considerations to the case $t<1-\ka\le t' \le 1+\ka < t''$ when either $|t'-(1-\ka)|\ge\ka$ or $|t'-(1+\ka)|\ge\ka$, and thus $D\bigl( (x,t) , (x',t') \bigr) \ge \ka$ or $D\bigl( (x',t') , (x'',t'') \bigr) \ge \ka$, which again proves \eqref{eq:tri}.
Thus in all cases $D$ satisfies the triangle inequality,  and so is a metric as claimed.
Further, $D$ is admissible because, 
by (i) its pullback induces the product topology on each $U_I\times [0,1+\ka]$ and $U_I\times [1-\ka,2]$, and hence 
 on $U_I\times [0,2]$. 
This proves (ii).

To prove (iii), we choose a smooth nondecreasing function $\be: [0,1]\to [0,1]$  
such that $\be|_{[0,\eps]}=0$, $\be|_{[1-\eps,1]}=1$, and the derivative is bounded by $\sup \be'<2$. 
(At this point we need to know that $(1-\eps)-\eps > \frac 12$, which holds by assumption $\eps<\frac 14$.)
For $r\in [0,1]$ we then obtain a metric $d_r$ on $|\Kk|$ by
$$
d_r(x,x'): = d(x,x') + \be(r) d'(x,x') 
$$
and note that $d_r(x,x') \le d_s(x,x')$ whenever $r\le s$.
Moreover, each $d_r$ is admissible on $|\Kk|$ since their pullback to the charts are analogous sums, and the sum of two metrics that induce the same topology also induces this topology.
Now we claim that 
$$
D\bigl( (x,t) , (x',t') \bigr) = d_{\min(t,t')}(x,x') + |t-t'|
$$
provides the required metric $D$ on $|\Kk \times [0,1]|$. This is evidently symmetric and positive definite, and by symmetry it suffices to check the triangle inequality \eqref{eq:tri} for $t\le t''$. 
In the case $t'< t\le t''$ we use $0\le \be(t) - \be(t') \le 2 (t-t')$ 
and $d'\le 1$
to obtain
\begin{align*}
& \; D\bigl( (x,t) , (x',t') \bigr) + D\bigl( (x',t') , (x'',t'') \bigr)\\
&\qquad \qquad  = \;
d(x,x') + d(x',x'') + \be(t') d'(x,x') +  \be(t') d'(x',x'') + |t-t'| + |t'-t''|\\
&\qquad \qquad \ge\;  d(x,x'') + \be(t') d'(x,x'') +  2|t-t'| \\
&\qquad \qquad\ge\; d(x,x'') + \be(t) d'(x,x'') 
\; =\; D\bigl( (x,t) , (x'',t'') \bigr).
\end{align*}
In the other cases $t\le t'\le t''$ resp.\ $t\le t''\le t'$, we can use the monotonicity $\be(t')\geq \be(t)$ resp.\ $\be(t'')\geq \be(t)$ to check \eqref{eq:tri}. Therefore $D$ is a metric on the product virtual neighbourhood. It is admissible because each $d_r$ is admissible on $|\Kk|$ so that the pullback metric on each set $U_I\times [0,1]$ induces the product topology. 
Finally it is $\eps$-collared by construction. In particular, it satisfies \eqref{eq:epscoll} due to the term $|t-t'|$ in its formula.
This proves (iii).

To prove (iv), let 
$C>0$ be a common upper bound for $d^0$ and $d^1$ and set $\ka:=\frac 16$.
By (iii) there is a $\ka$-collared metric $d_{[0,1]}$ on $|\Kk|\times [0,1]$ that equals 
$\frac \ka{3 C} d^0 +d_\R$ on 
$|\Kk|\times [0,\ka]$ and  $\frac \ka {3 C} (d^0 +d^1)+d_\R$ on $|\Kk|\times [1-\ka, 1]$. 
Similarly, there is a $\ka$-collared metric $d_{[1,2]}$ on $|\Kk|\times [1,2]$ that equals $\frac \ka{3C} (d^0 +d^1)+d_\R$ on $|\Kk|\times [1,1+\ka]$ and  $\frac \ka{3 C} d^1+d_\R$ on $|\Kk|\times [2-\ka, 2]$.
These satisfy the assumptions of (ii), 
so that we obtain a collared metric $D'$ on $|\Kk|\times [0,2]$ that restricts to $\min\{ \frac \ka {3 C} d^0+d_\R, \frac\ka 2 \}= \frac \ka{3 C } d^0+d_\R$ on 
$|\Kk|\times [0,\frac\ka6]$ since $\frac \ka {3 C} d^0 \le \frac\ka3$.
Similarly, it restricts  to $\min\{ \frac \ka{3C} d^1+d_\R , \frac{\ka}2 \} =  \frac \ka{3 C} d'+d_\R$ on 
$|\Kk|\times [2-\frac\ka 6,2]$.
Moreover, $D'$ satisfies \eqref{eq:epscoll} with $\eps=\ka=\frac 16$ by (iii).
Now define $D$ on $|\Kk|\times [0,1]$ to be the pullback of $\frac{3C}\ka D'$ by a rescaling map $(x,t)\mapsto (x,\be(t))$ with $\be(t) = \frac {\ka t}{3C}$ for $t$ near $0$ and $\be(t) = 2- \frac {\ka(1- t)}{3C}$ for $t$ near $1$.  Then $D$ restricts to $d^0+d_\R$ near $|\Kk|\times \{0\}$ and to $d^1+d_\R$ near $|\Kk|\times \{1\}$, and is collared by construction.  Hence it provides the required metric tame cobordism $(\Kk\times[0,1],D)$ from $(\Kk,d^0) $ to  $(\Kk,d^1)$.

To prove (v) let $2\eps>0$ be the collar width of $\Kk:=\Kk^{[0,1]}$. 
Then we first push forward the metric $d$ on $|\Kk|$ by the bijection
$$
F : |\Kk| \;\overset{\cong}{\longmapsto}\; |\Kk| \less \Bigl( \rho^0\bigl(|\p^0\Kk|\times [0,\eps)\bigr) \;\cup\; \rho^1\bigl(|\p^1\Kk|\times (1-\eps,\eps]\bigr) \Bigr)
$$
given by $F : \rho^0(x, t) \mapsto \rho^0(x, \eps + \frac 12 t )$ and $F : \rho^1(x, 1-t) \mapsto \rho^1(x, 1-\eps - \frac 12 t )$ for $t\in[0,2\eps)$, and the identity on the complement of the $2\eps$-collars.
The push forward $F_*d$ is admissible since $F$ pulls back to homeomorphisms supported in the collars of the domains $U_I$ of $\Kk$.
Next, let us denote $d_\al:=d|_{|\p^\al\Kk|}$, such that $d_0$ equals to the restriction $(F_*d)|_{\rho^0(|\p^0\Kk|\times\{\eps\})}$, pulled back via $\rho^0(\cdot, \eps)$, and similar for $d_1$ via $\rho^1(\cdot, 1-\eps)$.
Then we apply the same collaring construction as in (i) to extend $F_*d$ to an admissible collared metric on $|\Kk|$ given by symmetric extension of
\[
D\bigl( y , y' \bigr) \; :=\; 
\left\{\begin{array}{ll}
 d\bigl( y , y' \bigr)  &\mbox{ if } y,y' \in \im F ,\\ 
 \rho^\al_*(d_\al + d_\R) (y,y') &\mbox{ if } y,y'\in \rho^\al\bigl(|\p^\al\Kk|\times A^\al_\eps\bigr)  ,\\
d\bigl( y , \rho^\al(x', \eps) \bigr) + |\eps - t' | &\mbox{ if } y\in \im F, y'=\rho^0(x',t')
, \\
d\bigl( y , \rho^1(x', 1-\eps) \bigr) + |1-\eps - t' | &\mbox{ if } y\in \im F, y'=\rho^1(x',t')
, \\
d\bigl( \rho^0(x, \eps) , \rho^1(x', 1-\eps) \bigr)  &\mbox{ if } y=\rho^0(x,t), y'=\rho^1(x',t'). \\ 
\qquad + |\eps - t |+ |1-\eps - t' | &
\end{array}\right.
\]
Viewing this as a two stage extension to $\rho^\al\bigl(|\p^\al\Kk|\times A^\al_\eps\bigr)$ for $\al=0,1$, the proof of the triangle inequality and admissibility is the same as in (i), and the restrictions are $\rho^\al_*d_\al$ on $\rho^\al\bigl(|\p^\al\Kk|\times \{\al\}\bigr)$, as required. This finishes the proof.
\end{proof}

We are now in a position to prove the uniqueness part of Theorem~\ref{thm:K}, namely that different tame shrinkings of the same additive weak Kuranishi atlas are cobordant. 
This is a crucial ingredient in establishing that the virtual fundamental class associated to a given Kuranishi atlas is well defined. Since in practice one only associates a well defined cobordism class of Kuranishi atlases to a given moduli space, we need the full generality of the following result.
For this, we must revisit the construction of shrinkings in the proof of Proposition~\ref{prop:proper}.

\begin{remark}\rm
In the case of shrinkings $\Kk^0,\Kk^1$ of a fixed Kuranishi atlas $\Kk$ there is an easier construction 
of a cobordism in one special case:
If the shrinkings of the footprint covers $(F^\al_I)$ are compatible in the sense that their intersection
$(F^0_I\cap F^1_I)$ is also a shrinking (i.e.\ covers $X$ and has the same index set of nonempty intersections of footprints), then by Remark~\ref{rmk:shrink} the intersection of domains $U^0_{IJ}\cap U^1_{IJ}$ defines another shrinking of~$\Kk$. Thus one obtains an additive tame shrinking of the product Kuranishi cobordism $\Kk\times [0,1]$ by
$$
U_{IJ}^{[0,1]} :=  \bigl( U^0_{IJ} \times [0,\tfrac 13) \bigr)
\;\cup\; \bigl( \bigl( U^0_{IJ}\cap U^1_{IJ} \bigr) \times [\tfrac 13,\tfrac 23] \bigr)
\;\cup\; \bigl( U^1_{IJ} \times (\tfrac 23,1] \bigr) .
$$
\end{remark}

\begin{prop}\label{prop:cobord2}
Let $\Kk^{[0,1]}$ be an additive weak Kuranishi cobordism on $X\times [0,1]$, and let 
$\Kk^0_{sh}, \Kk^1_{sh}$ be preshrunk tame shrinkings of $\p^0\Kk^{[0,1]}$ and $\p^1\Kk^{[0,1]}$, 
that are hence metrizable as in Proposition~\ref{prop:metric}.
Then there is a preshrunk tame shrinking of $\Kk^{[0,1]}$ that provides a metrizable tame Kuranishi cobordism from $\Kk^0_{sh}$ to $\Kk^1_{sh}$.  
\end{prop}

\begin{proof} 
As in the proof of Proposition~\ref{prop:metric} we first construct a tame shrinking between any pair of  
tame shrinkings  $\Kk^0, \Kk^1$ of $\p^0\Kk^{[0,1]}$ and $\p^1\Kk^{[0,1]}$.  
We will use this first to obtain a tame shrinking $\Kk'$ of $\Kk^{[0,1]}$ with $\p^\al\Kk'=\Kk^\al$ 
(the tame shrinkings of which $\Kk^\al_{sh}$ are precompact shrinkings), 
and then to obtain a 
precompact
tame shrinking $\Kk^{[0,1]}_{sh}$ of $\Kk'$ with $\p^\al\Kk^{[0,1]}_{sh}=\Kk^\al_{sh}$.
This Kuranishi cobordism $\Kk^{[0,1]}_{sh}$ supports an admissible metric $d_{sh}$ by the same argument as in Proposition~\ref{prop:metric}. 
Finally we may arrange that it is collared by 
Proposition~\ref{prop:metcoll}~(v).
Hence it remains to carry out the first construction.

We write the index set as the union $\Ii_{\Kk^{[0,1]}}= \Ii_0 \cup \Ii_{(0,1)}\cup \Ii_1$ of
$\Ii_\al := \Ii_ {\p^\al\Kk^{[0,1]}} \subset \Ii_{\Kk^{[0,1]}}$
and $\Ii_{(0,1)}:=\Ii_{\Kk^{[0,1]}}\less (\Ii_0\cup\Ii_1)$.
Since the footprint of a chart $\Kk^{[0,1]}$ might intersect both $X\times\{0\}$ and $X\times\{1\}$, the sets $\Ii_0$ and $\Ii_1$ may not be disjoint, though they are both disjoint from $\Ii_{(0,1)}$,
which indexes the charts with precompact footprint in $X\times (0,1)$.
We will denote the charts of the Kuranishi cobordism $\Kk^{[0,1]}$ by $\bK_I=(U_I,\ldots)$, while $\bK^\al_I=(U^\al_I,\ldots) = \p^\al \bK_I |_{U^\al_I}$ denotes the charts of the shrinking $\Kk^\al$ of $\p^\al\Kk^{[0,1]}$ with domains $U^\al_{IJ}\subset \partial^\al U_{IJ}$.
Recall moreover that by definition of a shrinking the index sets $\Ii_ {\Kk^\al}=\Ii_ {\p^\al\Kk}=\Ii_\al$ coincide.
We suppose that the charts and coordinate changes of $\Kk^{[0,1]}$ have uniform collar width $5\eps>0$ as in Remark~\ref{rmk:Ceps}. Then the footprints have induced $5\eps$-collars
$$
(X\times A^\al_{5\eps}) \cap F_I = \partial^\al F_I \times A^\al_{5\eps}
\qquad\text{with}\qquad
\partial^\al F_I = \pr_X\bigl(  F_I\cap (X\times\{\al\}) \bigr).
$$
By construction of the shrinkings $\Kk^\al$ of $\p^\al\Kk^{[0,1]}$, 
we have  precompact inclusions $F^\al_I \sqsubset \partial^\al F_I$.
Now one can form a  $3\eps$-collared shrinking $(F'_i\sqsubset F_i)_{i=1,\ldots,N}$ of the cover of $X\times[0,1]$ by the footprints of the basic charts by first choosing an arbitrary shrinking $(F_i'')$ as in Definition~\ref{def:shr0}, and then adding $3\eps$-collars to the boundary charts.
Namely, for $i\in \Ii_0\cup \Ii_1$ we define
$$
F_i': = \Bigl(F_i''\cup \bigl(F^\al_i\times A^\al_{4 \eps} \bigr)\Bigr)\;
\less \; \Bigl(
\bigl( X\less F^\al_i \bigr)\times \overline{A^\al_{3\eps}} \Bigr)  .
$$
By construction, these sets still cover $X\times A^\al_{4\eps}$, together with $F'_i:=F''_i$ for $i\in\Ii_{(0,1)}$ cover $X \times (3\eps, 1-3\eps)$, and hence cover all of $X\times[0,1]$. 
Moreover, each $F'_i$ 
is open with $3\eps$-collar $F_i'\cap \bigl(X\times A^\al_{3\eps} \bigr) = F^\al_i\times A^\al_{3\eps}$ and has compact closure in $F_i$ because $F_i''$ does by construction and $F^\al_i \times A^\al_{4\eps}
\sqsubset \partial^\al F_i \times A^\al_{5\eps} \subset F_i $.
Next, the induced footprints $F'_I=\bigcap_{i\in I} F'_i$ also have $3\eps$-collars, and $F_I\neq\emptyset$ implies $F'_I\neq\emptyset$ since either $F_I\cap A^\al_{5\eps}\neq \emptyset$ so that $\emptyset\neq F^\al_I\times A^\al_{4\eps} \subset F'_I$, or $F_I\subset X \times (5\eps, 1-5\eps)$ so that $\emptyset\neq F''_I\subset F'_I$. 
Hence $(F_i')_{i=1,\ldots,N}$ is a  shrinking of the footprint cover with $3\eps$-collars.

We now carry through the proof of Proposition~\ref{prop:proper},
in the $k$-th step choosing domains $  U^{(k)}_{IJ}\subset U^{(k-1)}_{IJ}\subset U_{IJ}$ for $I, J\in \Ii_{\Kk^{[0,1]}}$ satisfying the conditions
(i$'$), (ii$'$),(iii$'$) as well as
the following collar requirement which ensures that the resulting shrinking of $\Kk^{[0,1]}$ is a Kuranishi cobordism  between the given tame atlases $\Kk^0$ and $\Kk^1$:
\begin{equation}\label{collar}
 (\iota^\al_I)^{-1} \bigl( U^{(k)}_{IJ} \bigr) \;=\; U^{\al}_{IJ} \times  A^\al_\eps
\qquad\forall \; \al \in\{0,1\}, \; I\subset J\in \Ii_\al  .
\end{equation}
For $k=0$ we first must choose precompact sets $U^{(0)}_I\sqsubset U_I$ satisfying the zero set condition \eqref{eq:U(0)}, namely $U_I^{(0)}\cap s_I^{-1}(0) = \psi_I^{-1}(F_I')$.
For that purpose we apply Lemma~\ref{le:restr0} to
$$
F_I'\cap (X\times(2\eps,1-2 \eps)) \;\sqsubset \; \psi_I \Bigl(s_I^{-1}(0)\less  \bigcup_{\al=0,1} \io^\al_I(\p^\al U_I\times \ov{A^\al_{\eps}})\Bigr)
$$
to find
$$
U_I' \;\sqsubset \; U_I \;\less\;  \bigcup_\al \io^\al_I(\p^\al U_I\times \ov{A^\al_{\eps}})
\quad\text{with}\quad
U_I'\cap s_I^{-1}(0) = \psi_I^{-1}\bigl(F_I'\cap (X\times(2\eps,1-2\eps)\bigr)  .
$$
Then we add the image under $\io^\al_I$ of the precompact subsets
$U_I^\al\times A^\al_{3\eps} \sqsubset \partial^\al U_I \times A^\al_{4\eps} \subset U_I$, 
which have footprint $F_I^\al\times A^\al_{3\eps} = F_I' \cap (X\times   A^\al_{3\eps})$, to obtain the required domains
$$
U_I^{(0)} := U_I'\;\cup\; \bigcup_{\al=0,1} \io^\al_I(U_I^\al\times A^\al_{ 3\eps}) \quad\sqsubset\; U_I 
$$ 
with boundary $\p^\al U_I^{(0)} = U_I^\al$ and collar width $\eps$.
Next, the domains $U_{IJ}^{(0)}$ for $I\subsetneq J$ are determined by \eqref{eq:UIJ(0)} and satisfy \eqref{collar} since 
both $U_{IJ}$ and $U_{I}^{(0)}, U_{J}^{(0)}$ have $\eps$-collars,
and $\phi_{IJ}$ has product form on the collar.
For $I\in\Ii_{(0,1)}$ these constructions also apply, and reproduce the construction without boundary, if we denote $\im\io^\al_I:=\emptyset$ and recall that the footprints are contained in $X\times (5\eps,1-5\eps)$. In the following we will use the same conventions and hence need not mention $\Ii_{(0,1)}$ separately.

Now in each iterative step for $k\geq 1$ there are two adjustments of the domains. First, in Step~A the domains $U^{(k)}_{IK}\subset W_{K'}$ for $|I|=k$ are chosen using Lemma~\ref{le:set}, where $W_{K'}$ is given by~\eqref{eq:WwK}.
In order to give these sets $\eps$-collars, we denote the sets provided by Lemma~\ref{le:set} by $V_{IK}^{(k)}\subset W_{K'}$ and define 
\begin{equation}\label{eq:UIJeps1}
U_{IK}^{(k)} \,:=\; V_{IK}^{(k)} \cup U_{IK}^{0,1,\eps}
\qquad \text{with}\quad
U_{IK}^{0,1,\eps} \,:=\; 
\iota^0_I\bigl(U_{IK}^0\times A^0_\eps)  \;\cup\;  \iota^1_I\bigl(U_{IK}^1\times A^1_\eps ) .
\end{equation}
This set is open and satisfies \eqref{collar} because $V_{IK}^{(k)}$ is a subset of $U^{(k-1)}_{IK}$, which by induction hypothesis has the required $\eps$-collar.
We now show that  \eqref{eq:UIJeps1} satisfies the requirements of Step A.

\begin{itemlist}
\item[(i$'$)]
holds since $U_{IJ}^{(k-1)}\cap \bigl(s_I\bigr)^{-1}(E_H) \;\subset\; V_{IJ}^{(k)}\subset
U_{IJ}^{(k)}$, where the first inclusion holds by construction of $V_{IJ}^{(k)}$.
\item[(ii$'$)]
holds since $V_{IJ}^{(k)}\cap V_{IK}^{(k)}= V_{I (J\cup K)}^{(k)}$ by construction and
$U_{IJ}^{0,1,\eps}\cap U_{IK}^{0,1,\eps} =  U_{I (J\cup K)}^{0,1,\eps}$ by the tameness of the collars, so
\begin{align*}
U_{IJ}^{(k)}\cap U_{IK}^{(k)}
&\;=\;
\bigr(V_{IJ}^{(k)}\cap V_{IK}^{(k)}\bigl)
\;\cup\;
\bigr(U_{IJ}^{0,1,\eps}\cap V_{IK}^{(k)}\bigl)
\;\cup\;
\bigr(V_{IJ}^{(k)}\cap U_{IK}^{0,1,\eps} \bigl)
\;\cup\;
\bigr(U_{IJ}^{0,1,\eps} \cap U_{IK}^{0,1,\eps} \bigl)  \\
&\;=\;
V_{I(J\cup K)}^{(k)}
\;\cup\;
U_{I(J\cup K)}^{0,1,\eps}  \;=\; U_{I (J\cup K)}^{(k)} .
\end{align*}
Here the two mixed intersections are subsets of the collar $U^{0,1,\eps}_{IJ}\cap U^{0,1,\eps}_{IK}$ due to $V_{I\bullet}^{(k)}\subset U_{I\bullet}^{(k-1)}$.
\item[(iii$''$)]
holds since $V_{IK}^{(k)} \subset (\phi_{IJ})^{-1}(V_{JK}^{(k-1)})$ by construction and
$U^{0,1,\eps}_{IK} \subset (\phi^{[0,1]}_{IJ})^{-1}(U^{0,1,\eps}_{JK})$
by the tameness of the shrinkings $U^0_\bullet, U^1_\bullet$.
\end{itemlist}

\NI
This completes Step A.
In Step B the domains $U^{(k)}_{JK}$ for $|J|>k$ are constructed by \eqref{eq:UJK(k)}, namely
$$
U_{JK}^{(k)}\, :=\; U_{JK}^{(k-1)}\less \bigcup_{I\subset J, |I|= k}
 \bigl( s_J^{-1}(E_I)\less \phi_{IJ}(U^{(k)}_{IJ}) \bigr) .
$$
We must check that this removes no points in the collars, i.e.\
$$
\iota^\al_J(U_{JK}^\al\times A^\al_\eps)\cap s_J^{-1}(E_I) \;\subset\; \phi_{IJ}(U^{(k)}_{IJ}) .
$$
But in this collar $s_J$ and $\phi_{IJ}$ have product form induced from the corresponding maps in the Kuranishi atlases $\Kk^\al$, where tameness implies $U_{JK}^\al\cap (s_J^\al)^{-1}(E_I) =
\phi_{IJ}^\al (U^\al_{IJ})$.
Since the $U^{(k)}_{IJ}$ already have $\eps$-collar by construction, this guarantees the above inclusion. Thus, with these modifications, the $k$-th step in the proof of Proposition~\ref{prop:proper} carries through. After a finite number of iterations, we find a tame shrinking $\Kk'$ of $\Kk^{[0,1]}$ with given restrictions $\p^\al\Kk^{[0,1]}=\Kk^\al$ for $\al=0,1$.
This completes the proof.
\end{proof}

\begin{rmk}\label{rmk:Morita}\rm  
In some sense, the ``correct" notion of equivalence between Kuranishi atlases should generalize that of Morita equivalence for groupoids.  In other words, one should develop an appropriate notion of ``refinement'' of a Kuranishi atlas (e.g.\ by
replacing each chart  by a tuple of charts obtained by restriction to a finite cover of its footprint) and then should say that two Kuranishi atlases $\Kk, \Kk'$ on $\bX$ are equivalent if there is a diagram
$$
\Kk'\longleftarrow \Kk''\longrightarrow \Kk,
$$
where $\Kk''$ is a ``refinement" of $\Kk$, and an arrow $\Kk^1\to\Kk^2$
means (at a minimum) that there are functors $\bB_{\Kk^1} \to \bB_{\Kk^2}$ and $\bE_{\Kk^1} \to \bE_{\Kk^2}$,  which commute with the section
functors
$s_{\Kk^\al}:\bB_{\Kk^\al}\to\bB_{\Kk^\al}$ and the footprint
functors $\psi_{\Kk^\al}: s_{\Kk^\al}^{-1}(0)\to X$.
We do not pursue this formal line of reasoning here.  However, it will be useful to develop the notion of a particular kind of refinement (called a reduction) in order to construct sections; see 
Proposition~\ref{prop:red}.
\end{rmk}

\section{From Kuranishi atlases to  the Virtual  Fundamental Class}\label{s:VMC}

In this section we assume that $\Kk$ is an oriented, tame Kuranishi atlas (as throughout with trivial isotropy, and with the notion of orientation to be defined)  of dimension $d$ on a compact metrizable space $X$, and construct the virtual moduli cycle (VMC) and virtual fundamental class (VFC).

As a preliminary step, Section~\ref{ss:red} provides reductions of the cover of the Kuranishi neighbourhood $|\Kk|$ by the images of the domains $\pi_\Kk(U_I)$. The goal here is to obtain a  cover by a partially ordered set of Kuranishi charts, with coordinate changes governed by the partial order. This will allow for an iterative construction of perturbations.
In Section~\ref{ss:sect} we introduce the notion of transverse perturbations in a reduction, and 
-- assuming their existence --
construct the VMC as a closed manifold 
up to compact cobordism, so far unoriented, from the associated perturbed zero sets. One difficulty here is to ensure compactness of the zero set despite the fact that $\io_\Kk(X)\subset |\Kk|$ may not have a precompact neighbourhood.
Transverse perturbations are then constructed in Section~\ref{ss:const},
and orientations will be established in Section~\ref{ss:vorient}.
Finally, we construct the virtual fundamental cycle in Section~\ref{ss:VFC}.

\subsection{Reductions and Covers}\label{ss:red}  \hspace{1mm}\\ \vspace{-3mm}

The cover of $X$ by the footprints $(F_I)_{I\in \Ii_\Kk}$ of all the Kuranishi charts
(both the basic charts and those that are part of the transitional data) is closed under intersection. This makes it easy to express compatibility of the charts, since the overlap of footprints of any two charts $\bK_I$ and $\bK_J$ is covered by another chart $\bK_{I\cup J}$.
However, this yields so many compatibility conditions that a construction of compatible perturbations in the Kuranishi charts may not be possible. For example, a choice of perturbation in the chart $\bK_I$ also fixes the perturbation in each chart $\bK_J$ over $\phi_{J (I\cup J)}^{-1}\bigl( \im \phi_{I (I\cup J)}\bigr) \subset U_J$, whenever $I\cup J\subset\Ii_\Kk$.
Since we do not assume transversality of the coordinate changes, this subset of $U_J$ need not be a submanifold, and hence the perturbation may not extend smoothly to $U_J$.
We will avoid these difficulties, and also make a first step towards compactness, by reducing the domains of the Kuranishi charts to precompact subsets $V_I\sqsubset U_I$ such that all compatibility conditions between $\bK_I|_{V_I}$ and $\bK_J|_{V_J}$ are given by direct coordinate changes $\Hat\Phi_{IJ}$ or $\Hat\Phi_{JI}$.
The left diagram in Figure~\ref{fig:1} illustrates a typical family of sets $V_I$ for $I\subset\{1,2,3\}$ with the appropriate intersection properties.
As we explain in Remark~\ref{rmk:nerve} this reduction process is analogous to replacing the star cover of a simplicial set by the star cover of its first barycentric subdivision.
This method was used in the current context by Liu--Tian~\cite{LiuT}.

\begin{figure}[htbp] 
   \centering
   \includegraphics[width=4in]{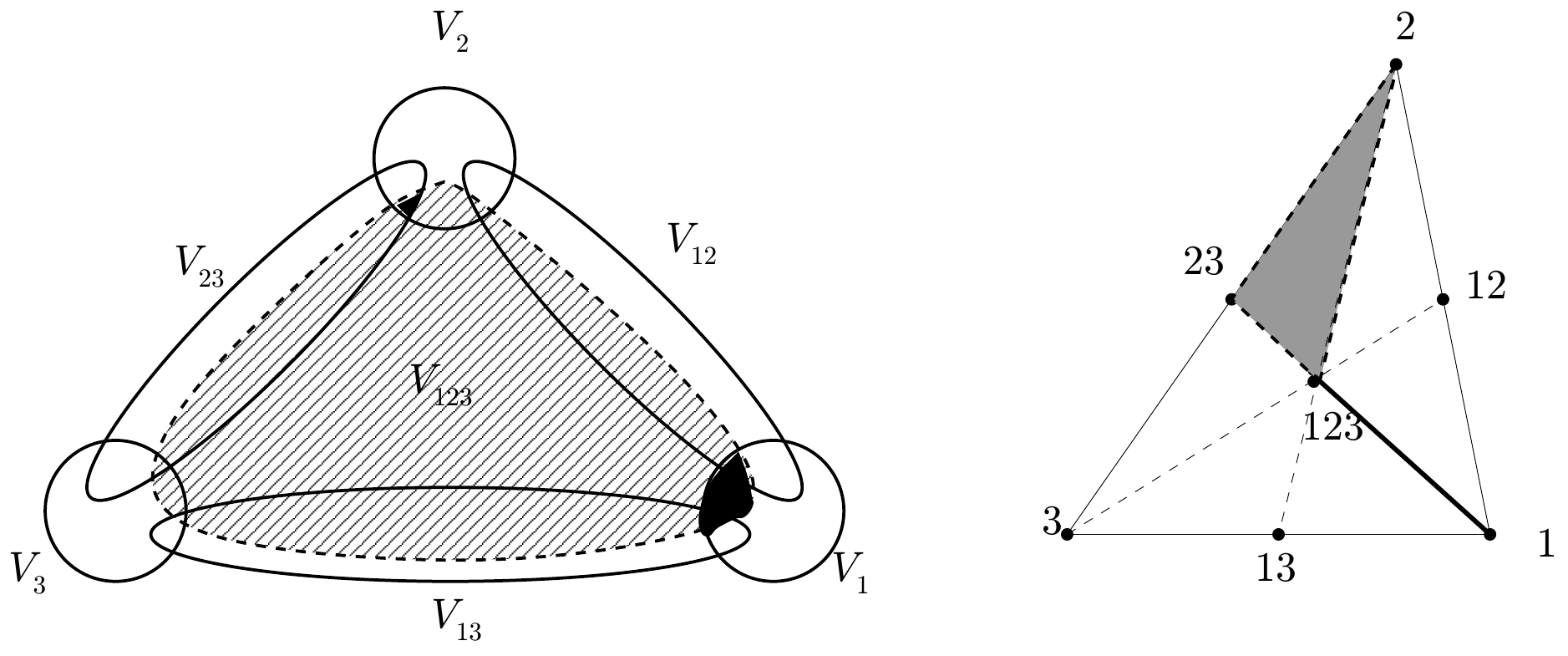}
     \caption{
The right diagram shows the
first barycentric subdivision of the triangle with vertices $1,2,3$.
It has three new vertices labelled $ij$ at the barycenters of the three edges and one vertex labelled $123$ at the barycenter of the triangle.
The left is a schematic picture of a 
reduction of the cover as in Lemma~\ref{le:cov0}.
The black sets are examples of multiple intersections of the new cover, which correspond to the simplices in the barycentric subdivision. E.g.\ $V_2\cap V_{23}\cap V_{123}$ corresponds to the triangle with vertices $2, 23, 123$, whereas $V_1\cap V_{123}$ corresponds to the edge between $1$ and $123$. }  \label{fig:1}
\end{figure}

\begin{rmk}\label{rmk:nerve}\rm
In algebraic topology it is often useful to consider the nerve $\Nn: = \Nn(\Uu)$ of an
open cover $\Uu: = (F_i)_{i=1,\ldots,N}$ of a space $X$, namely the
simplicial complex with one vertex for each open subset $F_i$ and a $k$-simplex for each nonempty intersection of $k+1$ subsets.\footnote
{
A simplicial complex is defined in \cite[\S2.1]{Hat} as a finite set of vertices (or $0$-simplices) $V$ and a subset of the power set $\Ii\subset 2^V$, whose $(k+1)$-element sets are called $k$-simplices for $k\geq 0$. The only requirements are that any subset $\tau \subsetneq \si$ of a simplex $\si\in\Ii$ is also a simplex $\tau\in\Ii$, and that each simplex is linearly ordered, compatible with a partial order on $V$.
In our case the ordering is provided by the linear order on $V=\{1,\dots,N\}$.
Then the $j$-th face of a $k$-simplex $\si: = \{i_0,\dots, i_k\}$, where $i_0<\dots<i_k$,
is given by the subset of $\si$ obtained by omitting its $j$-th vertex $i_j$. This provides the order in which faces are identified when constructing the realization of the simplicial complex.
}
We denote its set of simplices by
$$
\Ii_\Uu: = \bigl\{ I\subset \{1,\ldots,N\} \,\big|\, \cap _{i\in I} F_i\ne \emptyset \bigr\} .
$$
This combinatorial object is often identified with its realization,
the topological space
$$
|\Nn| :=\; \quotient{{\textstyle \coprod_{I\in \Ii_\Uu}} \{I\} \times \De^{|I|-1} }{\sim}
$$
where $\sim$ is the
equivalence relation under which the $|I|$ codimension $1$ faces of the simplex $\{I\}\times\De^{|I|-1}$ are identified with $\{I\less \{i\}\} \times \De^{|I|-2}$ for $i\in I$.
The realization $|\Nn|$ has a natural open cover by the stars $St(v)$ of its vertices $v$,
where $St(v)$ is the union of all (open) simplices whose closures contain $v$.
Notice that the nerve of the star cover of $|\Nn|$
can be identified with $\Nn$.

Next, let $\Nn_1:=\Nn_1(\Uu)$ be the first barycentric subdivision of $\Nn$.
That is, $\Nn_1$ is a simplicial complex with one vertex $v_I$ at the barycenter of each simplex
$I\in\Ii_\Uu$ and a $k$ simplex for each {\it chain} $I_0\subsetneq I_1\subsetneq\ldots\subsetneq I_k$ of simplices $I_0,\ldots,I_k\in\Ii_\Uu$.
This linear order on each simplex is induced from the partial order on the set of vertices
$\Ii_\Uu$ given by the inclusion relation for subsets of $\{1,\ldots,N\}$.
Further, the star $St(v_I)$ of the vertex $v_I$ in $\Nn_1$ is the union of all simplices given by chains that contain $I$ as one of its elements.
Hence two stars $St(v_I), St(v_J)$ have nonempty intersection if and only if $I\subset J $ or $J\subset I$, because this is a necessary and sufficient condition for there to be a chain containing both $I$ and $J$.
For example in the right hand diagram in Figure~\ref{fig:1} the stars of the vertices $v_{12}$ and $v_{13}$ are disjoint, as are the stars of $v_1$ and $v_2$.
As before the nerve of the star cover of $|\Nn_1|$ can be identified with $\Nn_1$ itself.
In particular, each nonempty intersection of sets in the star cover of $|\Nn_1|$  corresponds to a
simplex in $|\Nn_1|$, namely to  a chain $I_0\subset \ldots \subset I_k$ in the poset $\Ii_\Uu$.
Therefore the indexing set for this cover is the set $\Cc$ of chains in $\Ii_\Uu$; cf.\
Hatcher~\cite[p.119ff]{Hat}.

Now suppose that $\Uu = (F_i)_{i=1,\ldots,N}$ is the footprint cover provided by the basic charts of a tame Kuranishi atlas. Then $\Ii_\Uu=\Ii_\Kk$ is the index set of the Kuranishi atlas.
Hence $\Kk$ consists of one basic chart for each vertex of the nerve $\Nn(\Uu)$ and one transition chart $\bK_I$ for each simplex in $\Nn(\Uu)$.  We are aiming to construct from the original cover
$(F_i)$ of $X$ a {\it reduced} cover $(Z_I)_{I\in \Ii_\Uu}$ of $X$ whose pattern of intersections mimics that of the star cover of $|\Nn_1(\Uu)|$.
In particular, we will
require $\ov{Z_I}\cap \ov{Z_J}= \emptyset$ unless $I\subset J$ or $J\subset I$.
Next, we will aim to construct corresponding subsets $V_I\subset U_I$
with $V_I\cap s_I^{-1}(0)=\psi_I^{-1}(Z_I)$
and $\pi_\Kk(\ov{V_I})\cap \pi_\Kk(\ov{V_J})= \emptyset$ 
unless $I\subset J$ or $J\subset I$.
We will see in Proposition~\ref{prop:red} that such a reduction gives rise to a
Kuranishi atlas $\Kk^\Vv$ that has one {\it basic} chart for each vertex in $\Nn_1(\Uu)$, i.e.\ for
each element in $\Ii_\Kk$, and one transition chart for each simplex in $\Nn_1(\Uu)$, i.e.\ for each chain $C$ of elements in the poset $\Ii_\Kk$.
\end{rmk}

We will prove the existence of the following type of reduction in Proposition~\ref{prop:cov2} below.
As always, we denote the closure of a set $Z\subset X$ by $\ov Z$.

\begin{defn}\label{def:vicin}  
A {\bf reduction} of a tame Kuranishi atlas 
$\Kk$ is an open subset $\Vv=\bigcup_{I\in \Ii_\Kk} V_I \subset \Obj_{\bB_\Kk}$ i.e.\ a tuple of (possibly empty) open subsets $V_I\subset U_I$, satisfying the following conditions:
\begin{enumerate}
\item
$V_I\sqsubset U_I $ for all $I\in\Ii_\Kk$, and if $V_I\ne \emptyset$ then $V_I\cap s_I^{-1}(0)\ne \emptyset$;
\item
if $\pi_\Kk(\ov{V_I})\cap \pi_\Kk(\ov{V_J})\ne \emptyset$ then
$I\subset J$ or $J\subset I$;
\item
the zero set $\iota_\Kk(X)=|s_\Kk|^{-1}(0)$ is contained in 
$
\pi_\Kk(\Vv) \;=\; {\textstyle{\bigcup}_{I\in \Ii_\Kk}  }\;\pi_\Kk(V_I).
$
\end{enumerate}
Given a reduction $\Vv$, we define the {\bf reduced domain category} $\bB_\Kk|_\Vv$ and the {\bf reduced obstruction category} $\bE_\Kk|_\Vv$ to be the full subcategories of $\bB_\Kk$ and $\bE_\Kk$ with objects $\bigcup_{I\in \Ii_\Kk} V_I$ resp.\ $\bigcup_{I\in \Ii_\Kk} V_I\times E_I$, and denote by $s|_\Vv:\bB_\Kk|_\Vv\to \bE_\Kk|_\Vv$ the section given by restriction of $s_\Kk$. 
\end{defn}

Uniqueness of the VFC will be based on the following relative notion of reduction.

\begin{defn} \label{def:cvicin}
Let $\Kk$ be a tame Kuranishi cobordism. 
Then a {\bf cobordism reduction} of $\Kk$ is an open subset $\Vv=\bigcup_{I\in\Ii_{\Kk}}V_I\subset \Obj_{\bB_{\Kk}}$ that satisfies the conditions of Definition~\ref{def:vicin} and, in the notation of Section~\ref{ss:Kcobord}, has the following collar form:
\begin{enumerate}
\item[(iv)]
For each $\al\in\{0,1\}$ and $I\in 
\Ii_{\p^\al\Kk}\subset\Ii_{\Kk}$ 
there exists $\eps>0$ and a subset $\partial^\al V_I\subset \partial^\al U_I$ such that $\partial^\al V_I\ne \emptyset$ iff $V_I \cap \psi_I^{-1}\bigl( \partial^\al F_I \times \{\al\}\bigr)\ne \emptyset$,
and 
$$
(\iota^\al_I)^{-1} \bigl( V_I \bigr) \cap \bigl( \partial^\al U_I \times A^\al_\eps \bigr)
 \;=\; \partial^\al V_I \times A^\al_\eps .
$$
\end{enumerate}
We call 
$\partial^\al\Vv := \bigcup_{I\in\Ii_{\p^0\Kk}} \partial^\al V_I \subset \Obj_{\bB_{\p^\al\Kk}}$  
the {\bf restriction} of $\Vv$ to 
$\p^\al\Kk$.
\end{defn}

\begin{remark}\rm 
The restrictions $\partial^\al\Vv$ of a reduction $\Vv$ of a Kuranishi cobordism $\Kk$ are reductions of the restricted Kuranishi atlases $\p^\al\Kk$ for $\al=0,1$.
In particular condition (i) holds because part~(iv) of Definition~\ref{def:cvicin} implies that if $\p^\al V_I\ne \emptyset$ then $\p^\al V_I \cap \psi_I^{-1}\bigl( \partial^\al F_I\bigr)\ne \emptyset$
\end{remark}

The notions of reductions make sense for general Kuranishi atlases and cobordisms, however we will throughout assume additivity and tameness.
In some ways, the closest we come in this paper to constructing a ``good cover" in the sense of \cite{FO,J} is the category $\bB_\Kk|_\Vv$.  However, it is not a Kuranishi atlas.
For completeness, we show in Proposition~\ref{prop:red} that there is an associated
Kuranishi atlas  $\Kk^\Vv$ together with a faithful functor $\io^\Vv:\bB_{\Kk^\Vv}\to \bB_\Kk|_\Vv$ that induces an injection $|\Kk^\Vv|\to \pi_\Kk(\Vv)
\subset|\Kk|$. Since the extra structure in $\Kk^\Vv$ has no real purpose for us,  we use the simpler category $\bB_\Kk|_\Vv$ instead.
Its realization $|\bB_\Kk|_\Vv|$ also injects into $|\Kk|$, with image $|\Vv|=\pi_\Kk(\Vv)$ by a special case of the following result. In particular, this identifies the quotient topologies $|\Vv|\cong |\bB_\Kk|_\Vv|$ given by $\pi_\Kk$ resp.\ generated by the morphisms of $\bB_\Kk|_\Vv$.
Here, as before, we define the realization $|\bC|$ of a category $\bC$ to be the quotient $\Obj_{\bC}/\!\!\sim$, where $\sim$ is the equivalence relation generated by the morphisms in $\bC$.

\begin{lemma}\label{lem:full}  
Let $\Kk$ be a tame Kuranishi atlas with reduction $\Vv$, and suppose that $\bC$ is a full subcategory of the  reduced domain category $\bB_\Kk|_\Vv$. Then the map $|\bC|\to |\Kk|$, induced by the inclusion of object spaces, is
a continuous injection.
In particular, the realization $|\bC|$ is homeomorphic to its image $|\Obj_{\bC}|=\pi_\Kk(\Obj_{\bC})$ with the quotient topology in the sense of Definition~\ref{def:topologies}. 
\end{lemma}
\begin{proof}  
The map $|\bC|\to |\Kk|$ is well defined because $(I,x) \sim_{\bC} (J,y)$ implies $(I,x) \sim_{\bB_\Kk} (J,y)$ since the morphisms in $\bC$ are a subset of those in $\bB_\Kk$.
In order for $|\bC|\to |\Kk|$ to be injective we need to check the converse implication, that is we consider objects $(I,x),(J,y)\in \Obj_{\bC}$, identify them with points $x\in V_I$ and $y\in V_J$, and assume $\pi_\Kk(I,x) = \pi_\Kk(J,y)$. 
Then we have $I\subset J$ or $J\subset I$ by Definition~\ref{def:vicin}~(ii), so that Lemma~\ref{le:Ku2}~(a) implies either $y=\phi_{IJ}(x)$ or $x=\phi_{JI}(y)$.  Since $\bC$ is a full subcategory of $\bB_\Kk|_\Vv$ and hence of $\bB_\Kk$, the corresponding morphism $(I,J,x)$ (or $(J,I,y)$) belongs to $\bC$. 
Hence $(I,x) \sim_{\bB_\Kk} (J,y)$ implies $(I,x) \sim_{\bC} (J,y)$, so that $|\bC|\to |\Kk|$ is injective.
In fact, this shows that the relations $\sim_{\bB_\Kk}$ and $\sim_{\bC}$ agree on $\Obj_{\bC}\subset\Obj_{\bB_\Kk}$, 
and thus $|\bC|\to |\pi_\Kk(\Obj_{\bC})|$ is a homeomorphism with respect to the quotient topology on $\pi_\Kk(\Obj_{\bC})$. Finally, Proposition~\ref{prop:Ktopl1}~(i) asserts that the identity map $|\pi_\Kk(\Obj_{\bC})| \to \|\pi_\Kk(\Obj_{\bC})\|\subset|\Kk|$ is continuous from this quotient topology to the relative topology induced by $|\Kk|$, which finishes the proof.
\end{proof}

\begin{example}\rm 
The inclusion $|\bC| \hookrightarrow |\Kk|$ does {\it not} hold for arbitrary full subcategories of $\bB_\Kk$.  For example, the full subcategory $\bC$ with objects $\bigcup_{i=1,\dots, N} U_i$ (the union of the domains of the basic charts) has only identity morphisms, so that $|\bC| = \Obj_{\bC}$ equals $|\Kk|$ only if there are no transition charts.
\end{example}

In order to prove the existence and uniqueness up to cobordism of reductions, we start by analyzing the induced footprint cover of $X$.
Since the induced vicinity contains the zero set $\io_\Kk(X)$, the further conditions on reductions imply that the reduced footprints $Z_I = \psi_I(V_I\cap s_I^{-1}(0))$ form a reduction of the footprint cover $X=\bigcup_{i=1,\ldots,N} F_i$ in the sense of the following lemma.
This lemma makes the first step towards existence of reductions by showing how to reduce the footprint cover. We will use the fact that every compact Hausdorff space is a {\it shrinking space} in the sense
that every open cover has a precompact
 shrinking -- in the sense of Definition~\ref{def:shr0} without requiring condition \eqref{same FI}.

\begin{lemma}\label{le:cov0}
For any finite open cover of a compact Hausdorff
space $X=\bigcup_{i=1,\ldots,N} F_i$ there exists a {\bf cover reduction} $\bigl(Z_I\bigr)_{I\subset \{1,\ldots,N\}}$ in the following sense:
The $Z_I\subset X$ are (possibly empty) open subsets satisfying
\begin{enumerate}
\item
$Z_I\sqsubset F_I
:= \bigcap_{i\in I} F_i$
for all $I$;
\item
if $\ov{Z_I}\cap \ov{Z_J}\ne \emptyset$ then $I\subset J$ or $J\subset I$;
\item
$X\,=\, \bigcup_{I} Z_I$.
\end{enumerate}
\end{lemma}

\begin{proof}
Since $X$ is compact Hausdorff, we may choose precompact open subsets $F_i^0\sqsubset F_i$ that still cover $X$.
Next, any choice of precompactly nested sets
\begin{equation}\label{eq:FGI}
F_i^0\,\sqsubset\, G_i^1 \,\sqsubset\, F_i^1
\,\sqsubset\, G_i^2 \,\sqsubset\,\ldots \,\sqsubset\,
F_i^{N} = F_i
\end{equation}
yields further open covers  
$X=\bigcup_{i=1,\ldots,N} F_i^n$ and $X=\bigcup_{i=1,\ldots,N} G_i^n$ 
for $n=1,\ldots,N$.
Now we claim that the required cover reduction can be constructed by
\begin{equation}\label{eq:ZGI}
Z_I \,: =\; \Bigl( {\textstyle\bigcap_{i\in I}} G_i^{|I|} \Bigr) \;\less\; {\textstyle \bigcup_{j\notin I}} \ov{F^{|I|}_j} .
\end{equation}
To prove this we will use the following notation: Given any open cover $X=\bigcup_{i=1,\ldots,N} H_i$ of $X$, we denote the intersections of the covering sets by $H_I: = \bigcap_{i\in I} H_i$ for all $I\subset \{1,\ldots,N\}$. This convention will apply to define $F_I^k$ resp.\ $G_I^k$ from the $F_i^k$ resp. $G_i^k$, but it does not apply to the sets $Z_I$ constructed above, since in particular the $Z_i$ generally do not cover $X$. With this notation we have
$$
Z_I \,: =\; G_I^{|I|} \;\less\; {\textstyle \bigcup_{j\notin I}} \ov{F^{|I|}_j}
\qquad\text{for all}\;\; I\subset \{1,\ldots,N\}.
$$
These sets are open since they are the complement of a finite union of closed sets in the open set $G_I^{|I|}$. The precompact inclusion $Z_I\sqsubset F_I$ in (i) holds since $G_I^{|I|}\sqsubset F_I$.

To prove the covering in (iii)  let $x\in X$ be given. Then we claim that $x \in Z_{I_x}$ for
$$
I_x :=  \underset{I\subset\{1,\ldots,N\}, x \in G^{|I|}_I}{\textstyle \bigcup}  I  \;\;\;\subset\;\;\; \{1,\ldots, N\} .
$$
Indeed, we have $x\in G^{|I_x|}_{I_x}$ since $i\in I_x$ implies $x\in G^{|I|}_i$ for some $|I|\leq |I_x|$, and hence $x\in G^{|I_x|}_i$
since $G_i^{|I|}\subset G_i^{|I_x|}$.
On the other hand, for all $j\notin I_x$ we have $x\notin G^{|I_x|+1}_{I_x\cup j}$ by definition.
However, $x\in G^{|I_x|+1}_{I_x}$ by the nesting of the covers, so for every $j\notin I_x$ we obtain $x\in X\less G^{|I_x|+1}_j$, which is a subset of $X\less \ov{F^{|I_x|}_j}$. This proves $x \in Z_{I_x}$ and hence~(iii).

To prove the intersection property (ii), suppose to the contrary that $x\in \ov{Z_I}\cap\ov{Z_J}$ where $|I|\le |J|$ but $I\less J\ne \emptyset$.  Then given $i\in I\less J$, we have $x\in \ov{Z_I} \subset G^{|I|}_I\subset F^{|J|}_i$ since $|I|\le |J|$, which contradicts $x\in \ov{Z_J} \subset X\less F^{|J|}_i$.
Thus the sets $Z_I$ form a cover reduction.
\end{proof}

To construct cobordism reductions with given boundary restrictions we need the following notion of collared cobordism of cover reductions.

\begin{defn}\label{def:cobred}
Given a finite open cover $X=\bigcup_{i=1,\ldots,N} F_i$ of a compact Hausdorff space, we say that two {\bf cover reductions 
$(Z_I^0)_{I\subset \{1,\ldots,N\}}, (Z_I^1)_{I\subset \{1,\ldots,N\}}$  
are collared cobordant} if there exists a family of open subsets $Z_I\sqsubset F_I\times [0,1]$ satisfying conditions (i),(ii),(iii) in Lemma~\ref{le:cov0} for the cover $X\times[0,1]=\bigcup_{i=1,\ldots,N} F_i\times[0,1]$, 
and in addition 
are collared
in the sense of Definition~\ref{def:collared}, namely:
\begin{enumerate}
\item[(iv)]
There is $\eps>0$ such that $Z_I\cap \bigl(X\times A^\al_\eps\bigr)= Z_I^\al\times A^\al_\eps$
for all $I\subset\{1,\ldots,N\}$ and $\al=0,1$.
\end{enumerate}
\end{defn}

\begin{lemma} \label{le:cobred1}
The relation of collared cobordism for cover reductions is reflexive, symmetric, and transitive.
\end{lemma}

\begin{proof}
The proof is similar to (but much easier than) that of Lemma~\ref{le:cobord1}.
\end{proof}

With these preparations we can prove uniqueness of cover reductions up to collared cobordism, and also provide reductions for footprint covers of Kuranishi cobordisms.

\begin{lemma}\label{le:cobred}
\begin{enumerate}
\item
Any cover $X\times[0,1]=\bigcup_{i=1,\ldots,N} F_i$ by collared open sets $F_i\subset X\times [0,1]$ has a cover reduction $(Z_I)_{I\subset\{1,\ldots,N'\}}$ by collared sets $Z_I\subset X\times [0,1]$.
\item
Any two cover reductions $(Z_I^0)_{I\subset\{1,\ldots,N\}}$, $(Z_I^1)_{I\subset\{1,\ldots,N\}}$ of an open cover $X=\bigcup_{i=1,\ldots,N} F_i$ are collared cobordant.
\end{enumerate}
\end{lemma}

\begin{proof}
To prove (i) first note that any finite open cover $X\times [0,1]=\bigcup_{i=1,\ldots,N} F_i$ by sets with collar form near the boundary $\partial^\al F \times A_\eps$ can be shrunk to sets $F_i'\sqsubset F_i$ that also have collar form near the boundary.
Indeed, taking a common $\eps>0$, one can first choose a shrinking $F^\al_i \sqsubset\partial^\al F_i$ of the covers of the ``boundary components'' $X\times\{\al\}$ and a general shrinking $F''_i \sqsubset F_i$, then the required shrinking is given by
\begin{equation}\label{eq:cobred1}
F'_i:=  F^0_i \times [0,\tfrac \eps 2) \;\cup\;  F^1_i \times (1-\tfrac \eps 2,1]  \;\cup\;
F''_i \cap X \times (\tfrac \eps 4, 1-\tfrac \eps 4 ) .
\end{equation}
Hence we may choose nested covers $F_i^0\ldots F_i^k\sqsubset G_i^{k+1}\sqsubset \ldots F_i$ as in \eqref{eq:FGI} of $X\times [0,1]$ that have collar form near the boundary.
Then the sets $(Z_I)$ defined
by intersections in \eqref{eq:ZGI}
also have collar form near the boundary, and the arguments of Lemma~\ref{le:cov0} prove (i).

To prove (ii), we first transfer to the standard form constructed in Lemma~\ref{le:cov0}.

\MS
\noindent
{\bf Claim A:} {\em
Any cover reduction $(Z_I)$ of a finite open cover $X=\bigcup_i F_i$ is collared cobordant to a cover reduction constructed from nested covers $F_i^0\ldots F_i^k\sqsubset G_i^{k+1}\sqsubset \ldots F_i$
by \eqref{eq:ZGI}.
}
\MS

To prove this claim, choose a shrinking $Z_I^0\sqsubset Z_I$ such that $X = \bigcup_I Z^0_I$.  Then
these covers induce precompactly nested open covers
$$
F_i^0: = {\textstyle \bigcup_{i\in I}} Z_I^0  \quad \sqsubset \quad F_i': = {\textstyle \bigcup_{i\in I}} Z_I   \quad \sqsubset \quad F_i = {\textstyle \bigcup_{i\in I}} F_I .
$$
As in \eqref{eq:FGI} we can choose interpolating sets
$$
F_i^0\sqsubset \ldots F_i^k\sqsubset G_i^{k+1}\sqsubset \ldots \sqsubset F_i^{2N} = F_i',
$$
and let $\bigl(Z_I' := G_I^{|I|}\less \bigcup_{j\notin I} \ov{F_j^{|I|}}\bigr)$ be the resulting cover reduction of $F_i'$ and hence of $F_i$.
We claim that the union $(Z_I'':= Z^0_I\cup Z'_I)$ is a cover reduction of $(F_i)$ as well.
Since $Z_I^0\sqsubset Z_I\sqsubset F_I$ we only have to check the mixed terms in the intersection axiom (iii) in Lemma~\ref{le:cov0}, i.e.\ we need to verify
$$
\Bigl(J\less I\ne
\emptyset,\; I\less J\ne \emptyset\Bigr)\; \Longrightarrow\;
\Bigl((\ov{Z^0_I}\cup \ov{Z'_I})\cap (\ov{Z^0_J}\cup \ov{Z'_J}) = \emptyset\Bigr).
$$
Indeed, for $j\in J\less I$ we obtain $Z^0_J\subset F_j^0\sqsubset F_j^{|I|}\subset X\less Z'_I$
so that $\ov{Z^0_J}\cap \ov{Z_I'}=\emptyset$.  Conversely, $\ov{Z^0_I}\cap \ov{Z_J'}=\emptyset$ follows from the existence of $i\in I\less J$.
Now we have a chain of inclusions between cover reductions of $(F_i)$, namely $Z_I^0\subset Z_I$, $Z_I^0\subset Z_I''$, and  $Z_I'\subset Z_I''$. We claim that this induces collared cobordisms $(Z_I^0)\sim(Z_I)$, $(Z_I^0)\sim(Z_I'')$, and $(Z_I')\sim(Z_I'')$, so that Lemma~\ref{le:cobred1} implies that $(Z_I)$ is collared cobordant to $(Z_I')$, which is constructed by \eqref{eq:ZGI}.
To check this last claim, consider any two cover reductions $Z_I'\subset Z_I''$ of $(F_i)$ and note that a collared cobordism is given by
$$
 \bigl(Z_I'\times [0,\tfrac 23)\bigr)\cup \bigl(Z_I''\times (\tfrac 13,1]\bigr)  \;\subset\; X\times [0,1] .
$$
Indeed, these open subsets are collared and form a cover reduction since each of $(Z_I'),(Z_I'')$ satisfies the axioms (i),(ii), and the mixed intersection in (iii) is
$$
\Bigl(\ov{Z_I'}\times [0,\tfrac 23]\Bigr)\cap \Bigl(\ov{Z_J''}\times [\tfrac 13,1]\Bigr)
\subset
\Bigl(\ov{Z_I'} \cap \ov{Z_J''} \Bigr) \times [\tfrac 13,\tfrac 23] ,
$$
which is empty unless $I\subset J$ or $J\subset I$. This proves Claim A.
\MS

Now to prove (ii) it suffices to consider cover reductions $(Z^\al_I)$ that are constructed from nested covers $F_i^{0,\al} \ldots F_i^{k,\al} \sqsubset G_i^{k+1,\al}\sqsubset \ldots\partial^\al F_i$ as in \eqref{eq:FGI}.
We extend these to
nested covers of the product cover 
$X\times [0,1]=\bigcup_i F_i\times [0,1]$
by choosing constants
$$
\tfrac 13=\eps_{2N}<\ldots<\eps_{0} < \tfrac 12<\de_{0}<\ldots<\de_{2N} = \tfrac 23,
$$
and then setting
\begin{align*}
F_i^{k}: & = \Bigl(F_i^{k,0}\times [0,\de_{2k})\Bigr) \cup  \Bigl(F_i^{k,1}\times (\eps_{2k}, 1]\Bigr)
\;\sqsubset\; F_i\times [0,1],\\
G_i^{k}: & = \Bigl(G_i^{k,0}\times [0,\de_{2k-1})\Bigr) \cup \Bigl( G_i^{k,1}\times (\eps_{2k-1}, 1] \Bigr)\;\sqsubset\; F_i\times [0,1].
\end{align*}
Since these sets satisfy the nested property in \eqref{eq:FGI} and have collar form near the boundary given by the nested covers $G_i^{k,\al}\sqsubset F_i^{k,\al}$, the cover reduction $(Z_I)$ defined in \eqref{eq:ZGI} is a collared cobordism between the reductions $(Z^0_I)$ and $(Z^1_I)$.
\end{proof}

We now prove existence and uniqueness of reductions.

\begin{prop} \label{prop:cov2} 
\begin{itemize}
\item[(a)]
Every tame Kuranishi atlas $\Kk$ has a reduction $\Vv$.
\item[(b)] 
Every tame Kuranishi cobordism $\Kk^{[0,1]}$ has a cobordism reduction $\Vv^{[0,1]}$.
\item [(c)]  
Let $\Vv^0,\Vv^1$ be reductions of a tame Kuranishi atlas $\Kk$. Then there exists a cobordism reduction $\Vv$ of $\Kk\times[0,1]$ such that $\p^\al\Vv = \Vv^\al$ for $\al = 0,1$. 
\end{itemize}
\end{prop}

\begin{proof} 
For (a) we begin by using Lemma~\ref{le:cov0} to find a cover reduction 
$(Z_I)_{I\subset \{1,\ldots,N\}}$ of the footprint cover $X=\bigcup_{i=1,\ldots,N} F_i$. Since $Z_I\subset F_I=\emptyset$ for $I\notin\Ii_\Kk$, we can index the potentially nonempty sets in this cover reduction by $(Z_I)_{I\in\Ii_\Kk}$.
Then Lemma~\ref{le:restr0} provides precompact open sets $W_I\sqsubset U_I$ for each $I\in\Ii_\Kk$
with $Z_I\ne \emptyset$, satisfying
\begin{equation}\label{eq:wwI}
W_I\cap s_I^{-1}(0) = \psi_I^{-1}(Z_I),\qquad
\ov{W_I}\cap s_I^{-1}(0) = \psi_I^{-1}(\ov{Z_I}).
\end{equation}
The set $\Vv=(W_I)_{I\in\Ii_\Kk}$ now satisfies condition (iii) in Definition~\ref{def:vicin}, namely $\bigcup_I \pi_{\Kk}(W_I)$ contains $\bigcup_I \pi_{\Kk}\bigl(\psi_I^{-1}(Z_I)\bigr)$, which covers $\iota_\Kk(X)$.
We will construct the reduction by choosing $V_I \subset W_I$ so that (ii) is satisfied, while the intersection with the zero set does not change, i.e.\
\begin{equation}\label{eq:zeroV}
V_I\cap s_I^{-1}(0) = \psi_I^{-1}(Z_I),
\end{equation}
which guarantees (iii).
Further, condition (i) holds automatically since $V_I\subset W_I$,
and we define $V_I:=\emptyset$ when $Z_I=\emptyset$.
To begin the construction of the $V_I$, define 
$$
\Cc(I):=\{J\in \Ii_\Kk \,|\, I\subset J  \;\text{or}\; J\subset I \},
$$
and  for each $J\notin\Cc(I)$ define 
$$
Y_{IJ} \,:= \ov{W_I}\cap \pi_\Kk^{-1}(\pi_\Kk(\ov{W_J}))
\; = \; \ov{W_I}\cap \eps_I(\ov{W_J}) ,
$$
and note that $\eps_I(\ov{W_J})\subset U_J$ is closed by Lemma~\ref{le:Ku2}~(d).
Now if $J\notin\Cc(I)$, then using $\psi_I^{-1}(\ov{Z_I}) =s_I^{-1}(0)\cap \ov{W_I}$ we obtain
\begin{align*}
\psi_I^{-1}(\ov{Z_I}) \cap Y_{IJ}&\;\subset
\;\psi_I^{-1}(\ov{Z_I})\cap s_I^{-1}(0)\cap 
\eps_I(\ov{W_J}) \\
&\;=\; \psi_I^{-1}(\ov{Z_I}) \cap \eps_I\bigl(s_J^{-1}(0) \cap \ov{W_J} \bigr)\\
&\;= \; \psi_I^{-1}(\ov{Z_I})\cap \eps_I\bigl(\psi_J^{-1}(\ov{Z_J})\bigr)
\;= \; \psi_I^{-1}(\ov{Z_I}) \cap \psi_I^{-1}(\ov{Z_J})\;=\; \emptyset,
\end{align*}
where the first equality holds because $s$ is compatible with the coordinate changes.
The inclusion $Y_{IJ}\subset W_I\sqsubset U_I$ moreover ensures that $Y_{IJ}$ is compact, 
so has a nonzero Hausdorff distance from the closed set $\psi_I^{-1}(\ov{Z_I})$.
Thus we can find closed neighbourhoods 
$\Nn(Y_{IJ})\subset U_I$ of $Y_{IJ}$
for each $J\notin\Cc(I)$ such that  
$$
\Nn(Y_{IJ})\cap \psi_I^{-1}(\ov{Z_I})= \emptyset.
$$
We now claim that we obtain a reduction by removing these neighbourhoods,
\begin{equation}\label{eq:QIVI}
V_I := W_I \;\less {\textstyle\bigcup_{J\notin \Cc(I)}} \Nn(Y_{IJ}).
\end{equation}
Indeed, each $V_I\subset W_I$ is open, and 
$V_I\cap s_I^{-1}(0) = \psi_I^{-1}(Z_I) \;\less \bigcup_{J\notin \Cc(I)} \Nn({Y_{IJ}}) = \psi_I^{-1}(Z_I)$ 
by construction of $\Nn({Y_{IJ}})$.
Moreover, because $\ov{V_I}\subset \ov{W_I}$ for all $I$, 
we have for all $J\notin\Cc(I)$
$$
\ov{V_I}\cap \pi_\Kk^{-1}(\pi_\Kk(\ov{V_J}))\;\subset\; \ov{V_I}\cap \ov{W_I}\cap \pi_\Kk^{-1}(\pi_\Kk(\ov{W_J}))
\;\subset\; {\ov{V_I}}\cap  Y_{IJ} \;=\; \emptyset .
$$
This completes the proof of  a).

To prove (b) we use Lemma~\ref{le:cobred}~(i) to obtain a  collared cover reduction $(Z_I)_{I\in \Ii_{\Kk^{[0,1]}}}$ of the footprint cover of $\Kk^{[0,1]}$, which is collared by Remark~\ref{rmk:Ceps}.
Then the arguments for a) also show that $\Kk^{[0,1]}$ has a reduction $(V_I')_{I\in \Ii_{\Kk^{[0,1]}}}$.
It remains to adjust it to achieve collar form near the boundary as in Definition~\ref{def:vicin}~(iv).
For that purpose note that the footprint of the reduction is given by the cover reduction, that is \eqref{eq:zeroV} provides the identity
$$
V'_I\cap s_I^{-1}(0) = \psi_I^{-1}(Z_I) \qquad\forall I\in \Ii_{\Kk^{[0,1]}} .
$$
In particular, if 
$\p^\al Z_I\ne \emptyset$ then $\p^\al V'_I\ne \emptyset$,
though the converse may not hold.  
Now choose $\eps>0$ less or equal to half the collar width of $\Kk^{[0,1]}$ in Remark~\ref{rmk:Ceps} and so that condition (iv) in Definition~\ref{def:cobred} holds with $A^\al_{2\eps}$
for the footprint reduction $(Z_I)_{I\in \Ii_{\Kk^{[0,1]}} }$, 
in particular
\begin{equation}\label{eq:Zcoll}
\psi_I^{-1}(Z_I) \cap \io^\al_I\bigl(\p^\al U_I\times A^\al_{2\eps} \bigr)
\;=\;
\psi_I^{-1}(\partial^\al Z_I \times A^\al_{2\eps} )
\qquad \forall \al\in\{0,1\}, I\in \Ii_{\Kk^\al} .
\end{equation}
Then we set $V_I:=V'_I$ for interior charts 
$I\in \Ii_{\Kk^{[0,1]}}\less(\Ii_{\Kk^0}\cup \Ii_{\Kk^1})$.
If $I\in\Ii_{\Kk^\al}$ for $\al=0$ or $\al=1$ (or both) 
we define $V^\al_I \subset \p^\al U^\al_I$ so that, with $\eps_0: = \eps$ and $\eps_1:=1-\eps$, 
$$
\io^\al_I\bigl( V^\al_I \times \{\eps_\al\} \bigr) 
= \left\{ \begin{array}{ll} 
V_I' \cap \io^\al_I \bigl(\p^\al U_I\times \{\eps_\al\} \bigr) & \mbox{ if } \p^\al Z_I\ne \emptyset,\\
\emptyset &  \mbox{ if } \p^\al Z_I = \emptyset.
\end{array}\right.
$$
Since 
$(\p^\al Z_I)_{I\in\Ii_{\Kk^\al}}$ is a cover reduction of the footprint cover of $\Kk^\al$, and 
$$
\psi_I^{-1}(Z_I) \cap \io^\al_I \bigl(\p^\al U_I\times \{\eps_\al\} \bigr) = \psi_I^{-1}(\partial^\al Z_I \times \{\eps_\al\}),
$$ this defines reductions 
$(V^\al_I)_{I\in\Ii_{\Kk^\al}}$ of $\Kk^\al$.
With that, we obtain collared subsets of $U_I$ for each 
$I\in \Ii_{\Kk^0}\cup \Ii_{\Kk^1}$ by
\begin{equation}  \label{eq:Vcoll}
V_I:= \biggl( \, V_I' \;\less  \bigcup_{\al=0,1}  \io^\al_I\bigl(\p^\al U_I\times \ov{A^\al_\eps}\,\bigr) \biggr) 
\;\cup\;
\bigcup_{\al=0,1} \io^\al_I\bigl(V^\al_I\times \ov{A^\al_\eps}\bigr) \quad\subset\; U_I .
\end{equation}
These are open subsets of $U_I$ because, firstly, each $\p^\al U_I\times \ov{A^\al_\eps}$ is a relatively closed subset of the domain of the embedding $\io^\al_I$.
Secondly, each point in the boundary 
of $\io^\al_I\bigl(V^\al_I\times \ov{A^\al_\eps}\bigr)\subset U_I$ is of the form $\io^\al_I(x,\eps_\al)$ and, since the collar width is $2\eps$, has neighbourhoods $\io^\al_I\bigl(\Nn_x \times (\eps_\al -\eps, \eps_\al + \eps) \bigr)\subset V_I$ for any neighbourhood $\Nn_x\subset V^\al_I$ of $x$.
Moreover, $V_I$ has the same footprint as $V_I'$, since adjustment only happens on the collars $\iota^\al_I(\p^\al U_I\times \ov{A^\al_{\eps}})$, where the footprint is of product form, thus preserved by the construction.
Therefore the sets $(V_I)_{I\in \Kk^{[0,1]}}$ satisfy conditions (i) and (iii) in Definition~\ref{def:vicin}.
They also satisfy (ii) because, in the notation of Remark~\ref{rmk:cobordreal}, we can check this separately in the collars $\rho^\al(|\Kk^\al|\times{\ov{A^\al_\eps}})$ of $|\Kk^{[0,1]}|$ (where it holds because $(V^\al_I)$ is a reduction) and in their complements (where it holds because $(V'_I)$ is a reduction). 

Finally, condition (iv) in Definition~\ref{def:cvicin} holds by construction, namely 
$(\iota^\al_I)^{-1} \bigl( V_I \bigr) \cap \bigl( \partial^\al U_I \times A^\al_\eps \bigr) = V_I^\al \times A^\al_\eps$, and the condition $\bigl( \; \partial^\al V_I\ne \emptyset \; \Rightarrow \; V_I \cap \psi_I^{-1}\bigl( \partial^\al F_I \times \{\al\}\bigr)\ne \emptyset \;\bigr)$ holds since
$\partial^\al V_I = V_I^\al \ne \emptyset$ implies $\p^\al Z_I \neq \emptyset$ by definition of $V_I^\al$.
This completes the proof of (b).

To prove (c) we will use transitivity for cobordism reductions of the product cobordism $\Kk\times[0,1]$. More precisely, just as in the proof of Lemma~\ref{le:cobord1}, we can adjoin and rescale cobordism reductions within $\Kk\times[0,1]$.
In particular, given cobordism reductions $\Vv^{[0,1]}\subset \Obj_{\bB_\Kk}\times [0,1]$ and $\Vv^{[1,2]}\subset \Obj_{\bB_\Kk}\times [1,2]$ with identical collars $\p^1\Vv^{[0,1]}=\p^1\Vv^{[1,2]}$ near $\Obj_{\bB_\Kk}\times \{1\}$, we define a Kuranishi atlas $\Kk\times [0,2]$ with domain category $\Obj_{\bB_\Kk}\times [0,2]$ as in that lemma. Then, extrapolating notation to the reduction $V^{[0,2]}_I:=V^{[0,1]}_{I^{01}}\cup V^{[1,2]}_{I^{12}} \subset U_I \times [0,2]$ for $I\in\Ii_{\Kk^1}$ defines a cobordism reduction $\Vv^{[0,2]}$ of $\Kk\times [0,2]$ with $\p^0 \Vv^{[0,2]}=\p^0 \Vv^{[0,1]}$ and $\p^2 \Vv^{[0,2]}=\p^2 \Vv^{[1,2]}$.
Just as in the proof of additivity in Lemma~\ref{le:cobord1}, the resulting sets satisfy the separation condition (ii) of Definition~\ref{def:vicin} because the reductions $\Vv^{[0,1]}, \Vv^{[1,2]}$, and 
$\Vv^1=\partial^1 \Vv^{[0,1]} = \partial^1 \Vv^{[1,2]}$ do.
Now given two cobordism reductions $\Vv$ and $\Vv'$ of $\Kk\times[0,1]$ with $\p^1\Vv=\p^0\Vv$ we can shift $\Vv'$ in the domains to $\Obj_{\bB_\Kk}\times[1,2]$, glue it to $\Vv$ as above, and rescale the result back to a reduction $\Vv''\subset \Obj_{\bB_\Kk}\times[0,1]$ with $\p^0\Vv''=\p^0\Vv$ and $\p^1\Vv''=\p^1\Vv'$.
Similarly,  we can apply the isomorphism on $\Obj_{\bB_\Kk}\times[0,1]$ that reverses the inverval $[0,1]$ to turn any cobordism reduction $\Vv$ of $\Kk\times[0,1]$ into a cobordism reduction $\Vv'$ with $\p^0\Vv'=\p^1\Vv$ and $\p^1\Vv'=\p^0\Vv$.
Based on this, we will prove (c) in several stages.

\MS\NI
{\bf Step 1:} {\it The result holds if $\Vv^0 \subset \Vv^1$ and  $V^0_I\cap s_I^{-1}(0) = V^1_I\cap s_I^{-1}(0)$ for all $I\in \Ii_\Kk$.
}

\smallskip\NI
In this case the footprints $Z_I := \psi_I\bigl(V^\al_I\cap s_I^{-1}(0)\bigr)$ are the same for $\al=0,1$ by assumption.
Hence the sets for $I\in \Ii_\Kk$ 
$$
V_I \, 
 : =\; 
\bigl( V^0_I\times [0,\tfrac 23)\bigr) \cup\bigl( V^1_I\times (\tfrac 13,1]\bigr)
\;\subset\; U_I \times [0,1]
$$
form the required cobordism reduction. 
In particular, they satisfy the intersection condition (ii) over the interval $(\tfrac 13,\tfrac 23)$ because $V^0_I\subset V^1_I$ for all $I$. 
Their footprints are $Z_I\times [0,1]$ and so cover $X\times [0,1]$, and they satisfy the collar form requirement (iv) because $V_I\ne \emptyset$ iff $V^1_I\ne \emptyset$, and the latter implies $Z_I\ne \emptyset$  by condition (i) for $\Vv^1$.

\MS\NI
{\bf Step 2:} {\it The result holds if all footprints coincide, i.e.\ $V^0_I\cap s_I^{-1}(0) = V^1_I\cap s_I^{-1}(0)$.
}

\smallskip\NI
Note that $\Vv^{01}:= \bigcup_{I\in\Ii_\Kk} V^0_I\cap V^1_I$ is another reduction of $\Kk$ since it has the common footprints $Z_I$ as above, thus covers $\io_\Kk(X)$.
So Step 1 for $\Vv^{01}\subset\Vv^\al$ (together with reflexivity for $\al=0$) provides cobordism reductions $\Vv$ and $\Vv'$ of $\Kk\times[0,1]$ with $\p^0\Vv=\Vv^0$, $\p^1\Vv=\Vv^{01}=\p^0\Vv'$, and $\p^1\Vv'=\Vv^1$.
Now transitivity provides the required reduction with boundaries $\Vv^0,\Vv^1$.

\MS\NI
{\bf Step 3:} {\it The result holds for all reductions $\Vv^0,\Vv^1\subset\Obj_{\bB_\Kk}$.
}

\smallskip\NI
First use Lemma~\ref{le:cobred}~(ii) to obtain a family of collared sets $Z_I\sqsubset F_I\times [0,1]$ for $I\in\Ii_\Kk$ that form a cover reduction of $X\times[0,1]=\bigcup_{i=1,\ldots,N} F_i  \times[0,1]$ and restrict to the cover reductions $\partial^\al Z_I =  \psi_I(V_I^\al \cap s_I^{-1}(0))$ induced by the reductions $\Vv^\al$ for $\al=0,1$.
Next, as in the proof of (b), we construct a cobordism reduction $\Vv'$ of $\Kk\times [0,1]$ with footprints $Z_I$. 
Its restrictions $\partial^\al\Vv'$ are reductions of $\Kk^\al$ with the same footprint $\partial^\al Z_I$ as $\Vv^\al$  for $\al=0,1$.
Now Step 2 provides further cobordism reductions $\Vv$ and $\Vv''$ with $\Vv^0=\p^0\Vv$, $\p^1\Vv=\p^0\Vv'$, $\p^1\Vv'=\p^0\Vv''$, and $\p^1\Vv''=\Vv^1$, so that another transitivity construction provides the required cobordism reduction with boundaries $\Vv^0,\Vv^1$.
\end{proof}

Finally, we need to construct nested pairs of reductions $\Cc\sqsubset \Vv$, where either $\Cc$ or $\Vv$ is given. 
Since these will be used extensively in the construction of perturbations towards the VMC, we introduce this notion formally.

\begin{definition}\label{def:nest}
Let $\Kk$ be a Kuranishi atlas (or cobordism). Then we call a pair of subsets $\Cc,\Vv\subset\Obj_{\bB_\Kk}$ a {\bf nested (cobordism) reduction} if both are (cobordism) reductions of $\Kk$ and $\Cc\sqsubset \Vv$.
\end{definition}

In our later constructions, the roles of $\Vv$ and $\Cc$ will be quite different: 
we will define the perturbation $\nu$ on the domains of $\Vv$, while $\Cc$ will provide a precompact set $\pi_\Kk(\Cc)\subset |\Kk|$ that will contain the perturbed zero set $\pi_\Kk\bigl((s_\Vv + \nu)^{-1}(0)\bigr)$.   
As explained in Proposition~\ref{prop:Ktopl1}, the subset $\pi_\Kk(\Cc)
\subset|\Kk|$ 
has two different topologies, its quotient topology and the subspace topology.  
If $(\Kk,d)$ is metric, there might conceivably be a third topology, namely that induced by 
restriction of the metric.  
Although we will not use this explicitly, let us show that the metric topology on $\pi_\Kk(\Cc)$ agrees with the subspace topology, so that we only have two different topologies in play.

\begin{lemma}\label{le:Zz}
Let $\Cc$ be a reduction of a metric tame 
Kuranishi atlas 
$(\Kk,d)$. Then the metric topology on $\pi_\Kk(\Cc)$ equals the subspace topology. 
\end{lemma}
\begin{proof}  
Since every reduction $\Cc \subset\Obj_{\bB_\Kk}$ is precompact, the continuity of $\pi_\Kk:\Obj_{\bB_\Kk}\to |\Kk|$ (to $|\Kk|$ with its quotient topology) and of $\id_{|\Kk|}: |\Kk| \to (|\Kk|,d)$ from Lemma~\ref{le:metric} imply that $\pi_\Kk(\ov\Cc)\subset|\Kk|$ is compact in both topologies. Thus the identity map $\id_{\pi_\Kk(\ov\Cc)}: |\Kk|\supset \pi_\Kk(\ov\Cc) \to \bigl(\pi_\Kk(\ov\Cc), d \bigr)$ is a continuous bijection from the compact space $\pi_\Kk(\ov\Cc)$ with the subspace topology to the Hausdorff space $\bigl(\pi_\Kk(\ov\Cc), d \bigr)$ with the induced metric. But this implies that $\id_{\pi_\Kk(\ov\Cc)}$ is a homeomorphism, see Remark~\ref{rmk:hom}, and hence restricts to a homeomorphism $\id_{\pi_\Kk(\Cc)} :  |\Kk|\supset \pi_\Kk(\Cc) \to \bigl(\pi_\Kk(\Cc), d \bigr)$.
Thus, the relative and metric topologies on $\pi_\Kk(\ov\Cc)$ agree.
\end{proof}

\begin{lemma}\label{le:delred} 
Let  $\Vv$ be a (cobordism) reduction of a metric Kuranishi atlas (or cobordism) $\Kk$.
\begin{enumerate}
\item
There exists $\de>0$ such that $\Vv \sqsubset \bigcup_{I\in\Ii_\Kk} B^I_\de(V_I)$ is a nested (cobordism) reduction,
and moreover 
$$
B_\de(\pi_\Kk(\ov{V_I}))\cap B_\de(\pi_\Kk(\ov{V_J})) \neq \emptyset \qquad \Longrightarrow \qquad I\subset J \;\text{or} \; J\subset I. 
$$
\item  
If $\Vv$ is a reduction of a Kuranishi atlas $\Kk$, then 
there exists a nested reduction $\Cc\sqsubset \Vv$.
\item
If $\Vv$ is a cobordism reduction of the Kuranishi cobordism $\Kk$, and if $\Cc^\al\sqsubset \p^\al \Vv$ for $\al = 0,1$ are nested cobordism reductions of the boundary restrictions $\p^\al\Kk$, then there is a nested cobordism reduction $\Cc\sqsubset \Vv$ such that $\p^\al \Cc = \Cc^\al$ for $\al=0,1$.
\end{enumerate}
\end{lemma}

\begin{proof}
To prove (i) for a Kuranishi atlas $\Kk$ we need to find $\de>0$ so that
\begin{itemize} 
\item[a)]
$\displaystyle\phantom{\int\!\!\!\!\!\!}  B^I_\de(V_I)\sqsubset U_I$ for all $I\in\Ii_\Kk$;
\item[b)]
if $B_\de(\pi_\Kk(\ov{V_I}))\cap B_\de(\pi_\Kk(\ov{V_J})\ne \emptyset$ then $I\subset J$ or $J\subset I$.
\end{itemize}
The latter is a strengthened version of the separation condition (ii) in Definition~\ref{def:vicin}, due to the inclusion $\pi_\Kk(B^I_\de(W))\subset B_\de(\pi_\Kk(W))$ by compatibility of the metrics for any $W\subset U_I$.
The other conditions $B^I_\de(V_I)\cap s_I^{-1}(0)\neq\emptyset$ and $\io_\Kk(X)\subset\pi_\Kk(\Aa)$ for a reduction $\Aa := \bigcup_{I\in\Ii_\Kk} B^I_\de(V_I)\sqsubset\Obj_{\bB_\Kk}$ then follow directly from the inclusion $\Vv\subset\Aa$, and the latter is precompact since each component $V_I\sqsubset B^I_\de(V_I)$ is precompact.
In order to obtain cobordism reductions from this construction, recall that, by definition of a metric Kuranishi cobordism, it carries a product metric in the collar neighbourhoods $\iota^\al_I(\partial^\al U_I \times A^\al_\eps)\subset U_I$ of the boundary, which ensures that for $\de<\frac \eps 2$ the $\de$-neighbourhood of a collared set (i.e.\ with $(\iota^\al_I)^{-1}(V_I)=\partial^\al V_I \times A^\al_\eps$) is collared. More precisely, with $B^{\p^\al I}_\de(W)\subset \p^\al U_{I}$ denoting neighbourhoods in the corresponding domain of $\p^\al\Kk$ we have
$$
B^I_\de(V_I) \cap \iota^\al_I\bigl(\partial^\al U_I \times A^\al_{\frac \eps 2}\bigr) \;=\; \iota^\al_I\bigl( B^{\p^\al I}_\de(\partial^\al V_I) \times A^\al_{\frac\eps 2} \bigr) .
$$
So it remains to find $\delta>0$ satisfying a) and b).
Property a) for sufficiently small $\delta>0$ follows from the precompactness $V_I\sqsubset U_I$ in Definition \ref{def:vicin}~(i) and a covering argument based on the fact that, in the locally compact manifold $U_I$, every $p\in\ov{V_I}$ has a compact neighbourhood $\ov{B^I_{\delta_p}(p)}$ for some ${\delta_p>0}$.
To check b) recall that by Definition~\ref{def:vicin} the subsets $\pi_\Kk(\ov{V_I})$ and $\pi_\Kk(\ov{V_J})$ of $|\Kk|$ are disjoint unless $I\subset J$ or $J\subset I$. Since each $\pi_\Kk|_{U_I}$ maps continuously to the quotient topology on $|\Kk|$, and the identity to the metric topology is continuous by Lemma~\ref{le:metric}, the $\pi_\Kk(\ov{V_I})$ are also compact subsets of the metric space $|(\Kk,d)|$. Hence b) is satisfied if we choose $\de>0$ less than half the distance between each disjoint pair $\pi_\Kk(\ov{V_I}),\pi_\Kk(\ov{V_J})$.

For (ii) choose any shrinking $(Z_I')_{I\in \Ii_\Kk}$ of the footprint cover $\bigl(Z_I = \psi_I(V_I\cap s_I^{-1}(0))\bigr)_{I\in \Ii_\Kk}$ of $\Vv$ as in Definition~\ref{def:shr0}.  By Lemma~\ref{le:restr0} there are open subsets $C_I\sqsubset V_I$ such that $C_I\cap s_I^{-1}(0)  = \psi_I^{-1}(Z_I')$.  
This guarantees $\io_\Kk(X)\subset \pi_\Kk(\Cc)$, and  since the $(V_I)$ satisfy the separation condition (ii) in Definition~\ref{def:vicin}, so do the sets $(C_I)$.  
The same construction works if $\Kk$ is a cobordism, but we also require that $\Cc$ be collared.   
For this we must start with a collared shrinking $(Z_I')_{I\in \Ii_\Kk}$ of the footprint cover $\bigl(Z_I = \psi_I(V_I\cap s_I^{-1}(0)\bigr)_{I\in \Ii_\Kk}$ of the cobordism 
reduction $\Vv$, which exists by 
Lemma~\ref{le:cobred}~(i).
Then we choose $C_I'\sqsubset V_I$ such that $C_I'\cap s_I^{-1}(0)  = \psi_I^{-1}(Z_I')$ as above.
Finally we adjust each $C_I'$ to be a product in the collar by the method described in \eqref{eq:Vcoll}.

To prove (iii), we proceed as above, starting with a collared shrinking $(Z_I')_{I\in \Ii_\Kk}$ of the footprint cover $\bigl(Z_I = \psi_I(V_I\cap s_I^{-1}(0)\bigr)_{I\in \Ii_\Kk}$ of the cobordism 
reduction $\Vv$ that extends the shrinkings at $\al=0,1$ determined by the reductions $\Cc^\al$. 
This exists by the construction \eqref{eq:cobred1}
in the proof of Lemma~\ref{le:cobred}.
Then choose collared open subsets $C_I'\sqsubset V_I$ as above with $C_I'\cap s_I^{-1}(0)  = \psi_I^{-1}(Z_I')$ for all $I$.   
Finally, let $2\eps>0$ be less than the collar width of the sets $V_I$, $Z_I'$, and $C_I'$, then
we adjust $C_I'$ in these collars so that they have the needed restrictions by setting
$$
C_I: = \Bigl(C_I'\less \io^\al_I(\p^\al V_I\times [0,2\eps]\Bigr) \cup
 \Bigl(\io^\al_I\bigl(C^\al_I\times [0,2\eps)\cup
\p^\al(C_I')\times (\eps,2\eps]\bigr)\Bigr).
$$
Note that because $\Cc^\al$ and $\p^\al \Cc'$ are both contained in $\p^\al\Vv$, their union is also a reduction.
The proof of openness is the same as for \eqref{eq:Vcoll}.
\end{proof}

We end this subsection by showing that a reduction $\Vv$ of a tame Kuranishi atlas $\Kk$ canonically induces a (generally neither additive nor tame) Kuranishi atlas $\Kk^\Vv$ whose realization $|\Kk^\Vv|$ maps bijectively to $\pi_\Kk(\Vv)$.  This result is not used in the construction of the VMC or VFC.

\begin{prop}\label{prop:red}
Let $\Vv$ be a reduction of a tame Kuranishi atlas $\Kk$.
Then there exists a canonical Kuranishi atlas $\Kk^\Vv$ which satisfies the strong cocycle condition.
Moreover, there exists a canonical faithful functor $\io^\Vv: \bB_{\Kk^\Vv}\to \bB_{\Kk}$  which induces a continuous injection $|\Kk^\Vv|\to |\Kk|$ with image $\pi_\Kk(\Vv)$. 
\end{prop}

\begin{proof}
To begin the construction of $\Kk^\Vv$, note first that by condition (i) in Definition~\ref{def:vicin} the footprint $Z_I: = \psi_I\bigl(V_I\cap s_I^{-1}(0)\bigr)$ is nonempty whenever $V_I\ne \emptyset$.
Further by (iii) the sets 
$(Z_I)_{I\in\Ii_\Kk}$ cover $X$.  Hence we can use the tuple of nontrivial reduced Kuranishi charts $(\bK_I|_{V_I})_{I\in \Ii_\Kk, V_I\ne \emptyset}$ as the covering family of basic charts in $\Kk^\Vv$.
%
%
Then the index set of the new Kuranishi atlas is
\begin{equation}\label{eq:IKV}
\Ii_{\Kk^\Vv}= \bigl\{C\subset \Ii_\Kk \,\big|\, Z_C: =  {\textstyle \bigcap_{I\in C}} Z_I\neq \emptyset  \bigr\}.
\end{equation}
By Definition~\ref{def:vicin}~(ii), each such subset $C\subset \Ii_\Kk$ that indexes basic charts with $\bigcap_{I\in C} Z_I\neq \emptyset$ can be totally ordered into a chain $I_1\subsetneq I_2\ldots\subsetneq I_n$ of elements in $\Ii_\Kk$; cf.\ Figure~\ref{fig:1} and Remark~\ref{rmk:nerve}.   Therefore $\Ii_{\Kk^\Vv}$ can be identified with the set $\Cc\subset 2^{\Ii_\Kk}$ of linearly ordered chains
$C\subset \Ii_\Kk$ such that $Z_C\ne \emptyset$.
For $C = \bigl( I^C_1\subsetneq I^C_2\ldots\subsetneq I^C_{n^C}=:I^C_{\rm max} \bigr) \in\Ii_{\Kk^\Vv}$ we define the transition chart by restriction of the chart for
the maximal element $I^C_{\rm max}$
to the intersection of the domains of the chain:
\begin{equation}\label{eq:domKC}
\bK_C^\Vv := \bK_{I^C_{\rm max}}\big|_{V_C}
\qquad\text{with}\quad
V_C := \bigcap_{1\le k\le n_C}\; \phi_{I_k^C I^C_{\max}}\bigl(V_{I_k^C}\cap U_{I_k^C I^C_{\max}}\bigr) \;\subset\; V_{I^C_{\max}} .
\end{equation}
By Lemma~\ref{le:Ku2}~(a), this domain satisfies $\pi_\Kk(V_C)= \bigcap_{I\in C} \pi_\Kk(V_I)$
and hence can be expressed as 
\begin{equation}\label{eq:redu2}
V_C \;=\; {\textstyle \bigcap_{I\in C}} \; \eps_{I^C_{\max}}(V_I)
\;=\; U_{I^C_{\max}} \cap \pi_\Kk^{-1} \Bigl( {\textstyle \bigcap_{I\in C}} \, \pi_\Kk(V_I) \Bigr).
\end{equation}
Next, coordinate changes are required only between $C,D\in\Ii_{\Kk^\Vv}$ with $C\subset D$ so that $Z_C\supset Z_D\ne \emptyset$ and, by the above, $\pi_\Kk(V_C)\subset\pi_\Kk(V_D)$.
Since the inclusion $C\subset D$ implies the inclusion of maximal elements $I^C_{\rm max} \subset I^D_{\rm max}$, we can define the coordinate change as the restriction
$$
\Hat\Phi_{CD}^\Vv:=\Hat\Phi_{I^C_{\rm max} I^D_{\rm max}}\big|_{U^\Vv_{CD}}
\;: \;
\bK_C^\Vv = \bK_{I^C_{\rm max}}\big|_{V_C}  \; \longrightarrow\;
\bK_D^\Vv =\bK_{I^D_{\rm max}}\big|_{V_D}
$$
in the sense of Lemma~\ref{le:restrchange} with maximal domain
\begin{equation}\label{eq:UCD}
U^\Vv_{CD}: = V_C\cap (\phi_{I^C_{\rm max} I^D_{\rm max}})^{-1}(V_D)  \;=\;
V_C\cap \pi_\Kk^{-1}(\pi_\Kk(V_D)) .
\end{equation}
Using \eqref{eq:redu2} and the fact that $C\subset D$ we can also rewrite this domain as
\begin{equation}\label{domUCD}
U^\Vv_{CD} \;=\;  U_{I^C_{\max}}\cap \pi_\Kk^{-1}\Bigl({\textstyle \bigcap_{I\in D}} \pi_\Kk(V_I)\Bigr)
\;=\;  U_{I^C_{\max}}\cap \pi_\Kk^{-1}\bigl(\pi_\Kk(V_D)\bigr).
\end{equation}
Hence, again using Lemma~\ref{le:Ku2}~(a), the domain of the composed coordinate change $\Hat\Phi_{DE}\circ \Hat\Phi_{CD}$ for $C\subset D\subset E$ is
\begin{align*}
U^\Vv_{CD} \cap (\phi^\Vv_{I^C_{\rm max} I^D_{\rm max}})^{-1}\bigl( U^\Vv_{DE} \bigr)
&\;=\;
U_{I^C_{\max}} \;\cap\; \pi_\Kk^{-1}\bigl(\pi_\Kk(V_D)\bigr)
\;\cap\;
(\phi^\Vv_{I^C_{\rm max} I^D_{\rm max}})^{-1}\Bigl( \pi_\Kk^{-1}\bigl(\pi_\Kk(V_E)\bigr) \Bigr) \\
&\;=\;
U_{I^C_{\max}} \;\cap\; \pi_\Kk^{-1}\bigl(\pi_\Kk(V_D) \cap \pi_\Kk(V_E)\bigr)\\
& \;=\;
U_{I^C_{\max}} \;\cap\; \pi_\Kk^{-1}\bigl( \pi_\Kk(V_E)\bigr) ,
\end{align*}
which equals the domain of $\Hat\Phi_{CE}$.
Now the strong cocycle condition for $\Kk^\Vv$ follows from the strong cocycle condition for $\Kk$, which holds by Lemma~\ref{le:tame0}. In particular, $\Kk^\Vv$ is a Kuranishi atlas.

Next, the inclusions $V_C\hookrightarrow U_{I^C_{\max}}$ induce a continuous map on the object spaces
$$
\iota^\Vv \,:\;
\Obj_{\bB_{\Kk^\Vv}} = {\textstyle \bigcup_{C\in\Ii_{\Kk^\Vv}}} V_C \; \longrightarrow \;  {\textstyle \bigcup_{I\in\Ii_\Kk}} U_I = \Obj_{\bB_\Kk}.
$$
Since $V_C\subset V_{I^C_{\max}}$ for all $C$, this map  has image $\bigcup_{I\in \Ii_\Kk} V_I$.  It is generally not injective. However, because the coordinate changes in
${\Kk^\Vv}$ are restrictions of those in  $\Kk$, this map on object spaces extends to a functor 
$\iota^\Vv : \bB_{\Kk^\Vv}\to \bB_{\Kk}$. 
This shows that the induced map $|\iota^\Vv|:|{\Kk^\Vv}|\to |\Kk|$ is continuous with image $\pi_\Kk(\Vv)$. Moreover, the functor $\iota^\Vv$ is faithful, i.e.\ for each $(C,x), (D,y)\in   \Obj_{\bB_{\Kk^\Vv}}$ the map
$$
\Mor_{\Kk^\Vv}\bigl((C,x),(D,y)\bigr)\;\longrightarrow\; \Mor_{\Kk}\bigl(\io^\Vv(C,x), \io^\Vv(D,y)\bigr)
$$
is injective.
To prove that $|\iota^\Vv|$ is injective 
we need to show for $x\in V_C, y\in V_D$ that
$$
(I^C_{\max}, x)\sim_\Kk (I^D_{\max},y) \;\Longrightarrow\;  (C,x)\sim_{\Kk^\Vv} (D,y).
$$
To see this, note that by assumption and \eqref{eq:redu2} we have
$\pi_\Kk(V_I)\ni \pi_\Kk(x) = \pi_\Kk(y) \in  \pi_\Kk(V_J)$ for all 
$I\in C$ and $J\in D$.
In particular, for each $I^C_k\in C, I^D_\ell\in D$ the intersection $\pi_\Kk(V_{I^C_{k}})\cap \pi_\Kk(V_{I^D_{\ell}})$ is nonempty.  Hence, by Definition~\ref{def:vicin}~(ii), the elements  $I^C_1,\ldots,I^C_{\max}, I^D_1,\ldots,I^D_{\max}$ of $\Ii_\Kk$ can be ordered into a chain $E:=C\vee D$ (after removing repeated elements) with maximal element $I^{E}_{\max}=I^{C}_{\max}$ or $I^{E}_{\max}=I^{D}_{\max}$, and such that $\pi_\Kk(x)=\pi_\Kk(y)\in \bigcap_{I\in C\vee D} \pi_\Kk(V_I)=\pi_\Kk(V_{C\vee D})$. 
In particular, $V_{C\vee D}$ is nonempty, so we have $E=C\vee D \in\Ii_{\Kk^\Vv}$ and
$x\in V_C\cap \pi_\Kk^{-1}(V_E) \subset  U_{I^C_{\max}I^E_{\max}}$ lies in the domain of $\phi^\Vv_{C (C\vee D)}=\phi_{I^C_{\max}I^E_{\max}}$, whereas $y\in V_D\cap \pi_\Kk^{-1}(V_E) \subset  U_{I^D_{\max}I^E_{\max}}$ lies in the domain of $\phi^\Vv_{D (C\vee D)}=\phi_{I^D_{\max}I^E_{\max}}$. 
Now Lemma~\ref{le:Ku2}~(a) for $I^C_{\max}\subset I^E_{\max}$ and $I^D_{\max}\subset I^E_{\max}$ implies $\phi^\Vv_{C (C\vee D)}(x) = \phi^\Vv_{D (C\vee D)}(y)$.  
 This proves $(C,x)\sim_{\Kk^\Vv} (D,y)$ as required, and thus completes the proof.
\end{proof}

\begin{remark}\label{rmk:KVv} \rm   
The resulting Kuranishi atlas $\Kk^\Vv$ is far from additive. In fact, 
the above proof 
shows that $\Kk^\Vv$ has the property that for any two charts $\bK_C^\Vv, \bK_D^\Vv$ with intersecting footprints 
$Z_C\cap Z_D\ne \emptyset$,
we must have $I^C_{\rm max} \subset I^D_{\rm max}$ or $I^D_{\rm max} \subset I^C_{\rm max}$, though possibly neither $C\subset D$ nor $D\subset C$. 
Assuming w.l.o.g.\ that $I^C_{\rm max} \subset I^D_{\rm max}$, there is a direct coordinate change 
$$
\Hat\Phi_{I^C_{\rm max} I^D_{\rm max}}|_{V_C\cap \pi_\Kk^{-1}(\pi_\Kk(V_D))}
 \;: \;  \bK_C^\Vv = \bK_{I^C_{\rm max}}\big|_{V_C}  \; \longrightarrow\; \bK_D^\Vv = \bK_{I^D_{\rm max}}\big|_{V_D} .
$$
It embeds one of the obstruction bundles  as a summand of the other; in this case $E_C=_{I^C_{\rm max}}\hookrightarrow \Hat\Phi_{I^C_{\rm max} I^D_{\rm max}}(E_C)\subset E_D = E_{I^D_{\rm max}}$.
Such a coordinate change is not explicitly included in the Kuranishi atlas  
$\Kk^\Vv$ unless $C\subset D$. It does however appear, as in Lemma~\ref{le:Ku2}, as the composite of a coordinate change $\bK_C^\Vv \to \bK_{C\vee D}^\Vv$ with the inverse of a coordinate change $\bK_D^\Vv \to \bK_{C\vee D}^\Vv$.
In this respect the subcategory $\bB_{\Kk}|_\Vv$ has much simpler structure, since the components  $V_I$ of its space of objects correspond to chains $C=(I)$ with just one element.
\end{remark}

\subsection{Perturbed zero sets}\label{ss:sect}  \hspace{1mm}\\ \vspace{-3mm}

Throughout this section, $\Kk$ is a fixed tame Kuranishi atlas on a compact metrizable space $X$ or a Kuranishi cobordism on $X\times[0,1]$.
We begin by introducing sections in a reduction and an infinitesimal version of an admissibility condition 
for sections in \cite[A.1.21]{FOOO}.\footnote
{
In fact, even the canonical section $(s_I)$ may not satisfy an identity
$s_J = \Hat\phi_{IJ}(s_I) \oplus \id_{E_J/\Hat\phi_{IJ}(E_I)}$ in a tubular neighbourhood of $\phi_{IJ}(U_{IJ})\subset U_J$ identified with $U_{IJ}\times E_J/\Hat\phi_{IJ}(E_I)$ in the way described in  \cite[A.1.21]{FOOO}.  (The new definition used in \cite{FOOO12} is closer to ours.)
}

\begin{defn}\label{def:sect} 
A {\bf reduced section} of $\Kk$ is a smooth functor $\nu:\bB_\Kk|_\Vv\to\bE_\Kk|_\Vv$ between the reduced domain and obstruction categories of 
some reduction $\Vv$ of $\Kk$, 
such that $\pr_\Kk\circ\nu$ is the identity functor. 
That is, $\nu=(\nu_I)_{I\in\Ii_\Kk}$ is given by a family of smooth maps $\nu_I: V_I\to E_I$ such that for each $I\subsetneq J$ we have a commuting diagram
\begin{equation}\label{eq:comp}
\xymatrix{
 V_I\cap \phi_{IJ}^{-1}(V_J)   \ar@{->}[d]_{\phi_{IJ}} \ar@{->}[r]^{\qquad\nu_I}    &  E_I \ar@{->}[d]^{\Hat\phi_{IJ}}   \\
V_J \ar@{->}[r]^{\nu_J}  & E_J.
}
\end{equation}
We say that a reduced section $\nu$ is an {\bf admissible perturbation} of 
$s_\Kk|_\Vv$ if 
\begin{equation}\label{eq:admiss}
\rd_y \nu_J(\rT_y V_J) \subset\im\Hat\phi_{IJ} \qquad \forall \; I\subsetneq J, \;y\in V_J\cap\phi_{IJ}(V_I\cap U_{IJ}) .
\end{equation}
\end{defn}

\begin{rmk}\label{rmk:sect} \rm  
(i)
Each reduced section  $\nu:\bB_\Kk|_\Vv\to\bE_\Kk|_\Vv$ induces a continuous map $|\nu|: |\Vv|\to |\bE_\Kk|$ such that $|\pr_\Kk|\circ |\nu| = \id$, where $|\pr_\Kk|$ is as in Theorem~\ref{thm:K}.
Each such map has the further property that $|\nu|\big|_{\pi_\Kk(V_I)}$ takes values in $\pi_\Kk(U_I\times E_I)$.  
\MS

\NI (ii) 
More generally, a section $\si$ of $\Kk$ is a functor $\bB_\Kk\to\bE_\Kk$ that satisfies the conditions of Definition~\ref{def:sect} with $\Vv$ replaced by $\Obj_{\bB_\Kk}$. 
But these compatibility conditions are now much more onerous.  
For example, except in the most trivial cases, the set $V_{12}\cap \bigcap_{i=1,2} \phi_{i,12}(U_{i,12}\less V_i)$ is nonempty, so that there is $x\in V_{12}$ with $\pi_\Kk(x)\in \bigl(\pi_\Kk(U_1)\cup \pi_\Kk(U_2)\bigl)\less \bigl( \pi_\Kk(V_1) \cup \pi_\Kk(V_2)\bigr)$.
A reduced section $\nu$ could take any value $\nu_{12}(x) \in E_{12}\cong \Hat\phi_{1,12}(E_1)\oplus \Hat\phi_{2,12}(E_2)$. On the other hand, a section $\si$ of $\Kk$ would have $\si(x)\in \bigcap_{i=1,2} \Hat\phi_{i,12}(E_i)=\{0\}$ since the compatibility conditions imply that $\nu_{12}|_{\im\phi_{i,12}}$ takes values in $\Hat\phi_{i,12}(E_i)$.  
We cannot achieve transversality under such conditions, which explains why we consider reduced sections.
\end{rmk}

Note that the zero section $0_\Kk$, given by $U_I\to 0 \in E_I$, restricts to an admissible perturbation $0_\Vv:\bB_\Kk|_\Vv\to\bE_\Kk|_\Vv$ in the sense of the above definition. 
Similarly, the canonical section $s:= s_\Kk$ of the Kuranishi atlas restricts to a section $s|_\Vv: \bB_\Kk|_\Vv\to\bE_\Kk|_\Vv$ of any reduction.
However, the canonical section is generally not admissible. In fact, as we saw in Lemma~\ref{le:change}, for all $y\in V_J\cap \phi_{IJ}(V_I\cap U_{IJ})$ the map 
$$
{\rm pr}_{E_I}^\perp\circ \rd_y s_J \,: \;\;  \quotient{\rT_y U_J} {\rT_y (\phi_{IJ}(U_{IJ}))} \; \longrightarrow \; \quotient{E_J}{\Hat\phi_{IJ}(E_I)} 
$$
is an isomorphism by the index condition \eqref{tbc}, while for an admissible section it is identically zero.
So for any reduction $\Vv$ and admissible perturbation $\nu$, the sum 
$$
s+\nu:=(s_I|_{V_I}+\nu_I)_{I\in\Ii_\Kk} \,:\; \bB_\Kk|_\Vv \;\to\; \bE_\Kk|_\Vv
$$
is a reduced section that satisfies the index condition
$$
{\rm pr}_{E_I}^\perp\circ\rd_y (s_J+\nu_J)\,: 
\;\; \quotient{\rT_y U_J}{\rT_y (\phi_{IJ}(U_{IJ}))} \;\stackrel{\cong}\longrightarrow \; \quotient{E_J}{\Hat\phi_{IJ}(E_I)}
\qquad\forall \; y\in V_J\cap \phi_{IJ}(V_I\cap U_{IJ}).
$$
We use this in the following lemma to show that transversality of the sections in Kuranishi charts is preserved under coordinate changes.
However, admissibility only becomes essential in the discussion of orientations; cf.\ Proposition~\ref{prop:orient}.

\begin{lemma}\label{le:transv}
Let $\Vv$ be a reduction of $\Kk$, and $\nu$ an admissible perturbation of $s_\Kk|_\Vv$. 
If $z\in V_I$ and $w\in V_J$ map to the same point in the virtual neighbourhood $\pi_\Kk(z)=\pi_\Kk(w)\in|\Kk|$, then $z$ is a transverse zero of $s_I|_{V_I}+\nu_I$ if and only if $w$ is a transverse zero of $s_J|_{V_J}+\nu_J$.
\end{lemma}
\begin{proof}
Note that, since the equivalence relation $\sim$ on $\Obj_{\Bb_\Kk}$ is generated  by 
$\preceq$ and its inverse $\succeq$, it suffices to suppose that $(I,z)\preceq (J,w)$, i.e.\ $w=\phi_{IJ}(z)$. Now $s_J(w)=\Hat\phi_{IJ}(s_I(z))=0$ iff $s_I(z)=0$ since $\Hat\phi_{IJ}$ is injective.
Next, $z$ is a transverse zero of $s|_{\Vv}+\nu$ exactly if $\rd_z (s_I+\nu_I): \rT_z U_I\to E_I$ is surjective. On the other hand, we have splittings $\rT_w U_J \cong \im\rd_z\phi_{IJ} \oplus \tfrac{\rT_w U_J} {\im\rd_z\phi_{IJ}} $ and $E_J \cong \Hat\phi_{IJ}(E_I) \oplus \tfrac{E_J}{\Hat\phi_{IJ}(E_I)}$
with respect to which the differential at $w$ has product form
\begin{equation}\label{eq:dnutrans}
\rd_w (s_J + \nu_J) \cong \bigl(\; \Hat\phi_{IJ}\circ\rd_z (s_I+\nu_I) \circ (\rd_z\phi_{IJ})^{-1} \,,\, \rd_w s_J \, \bigr) ,
\end{equation}
by the admissibility condition on $\nu_J$. 
Here the second factor is an isomorphism by the index condition \eqref{tbc}. 
Since $\Hat\phi_{IJ}$ and $(\rd_z\phi_{IJ})^{-1}$ are isomorphisms on the relevant domains, 
this proves equivalence of the transversality statements.
\end{proof}

\begin{defn}\label{def:sect2}
A {\bf transverse perturbation} 
of $s_\Kk|_{\Vv}$
is a reduced section $\nu:\bB_\Kk|_\Vv\to\bE_\Kk|_\Vv$ 
whose sum with 
the canonical section $s|_\Vv$
is transverse to the zero section $0_\Vv$, that is $s_I|_{V_I}+\nu_I\pitchfork 0$ for all $I\in\Ii_\Kk$.

Given a transverse perturbation $\nu$, we  define the {\bf perturbed zero set} $|\bZ_\nu|$ to be the realization of the full subcategory $\bZ_\nu$ of $\bB_\Kk$ with object space 
$$
(s + \nu)^{-1}(0) := {\textstyle \bigcup_{I\in \Ii_\Kk}}(s_I|_{V_I}+\nu_I)^{-1}(0) \;\subset\;\Obj_{\bB_\Kk} . 
$$
That is, we equip
$$
|\bZ_\nu| : = \bigl|(s + \nu)^{-1}(0)\bigr| \,=\; \quotient{ {\textstyle\bigcup_{I\in\Ii_\Kk} (s_I|_{V_I}+\nu_I)^{-1}(0) }}{\!\sim} 
$$
with the quotient topology generated by the morphisms of $\bB_\Kk|_\Vv$.
By Lemma~\ref{lem:full} this is equivalent to the quotient topology induced by $\pi_\Kk$, 
and the inclusion $(s+\nu)^{-1}(0)\subset\Vv = \Obj_{\bB_\Kk|_\Vv}$ induces a continuous injection, which we denote by
\begin{equation}\label{eq:Zinject} 
i_\nu \,:\;  |\bZ_\nu| \;\longrightarrow\; |\Kk| .
\end{equation}
\end{defn}

To see that the above is well defined, recall that the canonical section restricts to a reduced section $s|_\Vv: \bB_\Kk|_\Vv\to\bE_\Kk|_\Vv$, so that the sum $s|_\Vv+\nu$ is a reduced section as well, with a well defined zero set $(s + \nu)^{-1}(0)$. 
Moreover, since $\bZ_\nu$ is the realization of a full subcategory of $\bB_\Kk|_\Vv$, Lemma~\ref{lem:full} asserts that the map $i_\nu$ is a continuous injection to $|\Kk|$, and moreover a homoeomorphism from $|\bZ_\nu|$ to $\pi_\Kk\bigl((s+\nu)^{-1}(0)\bigr)=|(s + \nu)^{-1}(0)|$ with respect to the quotient topology in the sense of Definition~\ref{def:topologies}. 
In particular, the continuous injection to the Hausdorff space $|\Kk|$ implies Hausdorffness of $|\bZ_\nu|$.
However, the image of $i_\nu$ is $\pi_\Kk\bigl((s+\nu)^{-1}(0)\bigr)$ with the relative topology induced by $|\Kk|$,  that is
$$
i_\nu ( |\bZ_\nu| )  =  \|(s+\nu)^{-1}(0)\| .
$$
So the perturbed zero set is equipped with two Hausdorff topologies -- the quotient topology on $|\bZ_\nu|\cong|(s+\nu)^{-1}(0)|$  and the relative topology on $\|(s+\nu)^{-1}(0)\|\subset|\Kk|$.
It remains to achieve local smoothness and compactness in one of the topologies. 
We will see below that local smoothness follows from transversality of the perturbation, though only in the topology of $|\bZ_\nu|$, which may contain smaller neighbourhoods than $\|(s + \nu)^{-1}(0)\|$.
On the other hand, compactness of $\|(s + \nu)^{-1}(0)\|$ is easier to obtain than that of $|\bZ_\nu|$, which may have more open covers.
For the first, one could use the fact that $\|(s+\nu)^{-1}(0)\|\subset\|\Vv\|$ is precompact in $|\Kk|$ by Proposition~\ref{prop:Ktopl1}~(iii), so it would suffice to deduce closedness of $\|(s+\nu)^{-1}(0)\|\subset|\Kk|$. This would follow if the continuous map $|s+\nu|:\|\Vv\| \to |\bE_\Kk|_\Vv|$ had a continuous extension to $|\Kk|$ with no further zeros. 
However, such an extension may not exist. In fact, generally $\|\Vv\|\subset |\Kk|$ fails to be open, 
$\io_\Kk(X)\subset |\Kk|$ does not have any precompact neighbourhoods (see Example~\ref{ex:Khomeo}),
 and even those assumptions would not guarantee the existence of an extension.
So compactness of either $\|(s+\nu)^{-1}(0)\|$ or $|\bZ_\nu|$ will not hold in general without further hypotheses on the perturbation that force its zero set to be ``away from the boundary"  of $\|\Vv\|$ in some appropriate sense.
For that purpose we introduce the following extra assumption on the perturbed section that will directly imply compactness of $|\bZ_\nu|$.

\begin{defn}\label{def:precomp}  
A reduced section $\nu: \bB_\Kk|_\Vv\to \bE_\Kk|_\Vv$ is said to be {\bf precompact} if its domain is part of a nested reduction $\Cc\sqsubset \Vv$ such that 
$
\pi_\Kk\bigl((s + \nu)^{-1}(0)\bigr)
\;\subset\; \pi_\Kk(\Cc).
$
\end{defn}

The smoothness properties follow more directly from transversality of the perturbation.
The next lemma shows that for transverse perturbations the object space $\bigcup_I (s_I|_{V_I}+\nu_I)^{-1}(0) \subset \bigcup_I V_I$ of $\bZ_\nu$ is a smooth submanifold of dimension $d: = \dim \Kk$, and that the morphisms spaces are given by local diffeomorphisms.
Hence the category $\bZ_\nu$ can be extended to a groupoid by adding the inverses 
to the space of morphisms.

\begin{lemma} \label{le:stransv}
Let $\nu: \bB_\Kk|_\Vv\to\bE_\Kk|_\Vv$ be a transverse perturbation of $s_\Kk|_\Vv$.
Then the domains of the perturbed zero set $(s_I|_{V_I}+\nu_I)^{-1}(0)\subset V_I$ are submanifolds for all $I\in\Ii_\Kk$; and for $I\subset J$ the map $\phi_{IJ}$ induces a diffeomorphism from ${V_J\cap U_{IJ}\cap (s_I|_{V_I}+\nu_I)^{-1}(0)}$ to an open subset of $(s_J|_{V_J}+\nu_J)^{-1}(0)$.
\end{lemma}
\begin{proof}
The submanifold structure of $(s_I|_{V_I}+\nu_I)^{-1}(0)\subset U_I$ follows from the transversality and the implicit function theorem, with the dimension given by the index $d:=\dim U_I-\dim E_I$.
For $I\subset J$ the embedding $\phi_{IJ}:U_{IJ}\to U_J$ then restricts to a smooth embedding  
\begin{equation}\label{eq:ZphiIJ}
\phi_{IJ}: U_{IJ}\cap(s_I|_{V_I}+\nu_I)^{-1}(0) \to (s_J|_{V_J}+\nu_J)^{-1}(0)
\end{equation}
by the functoriality of the perturbed sections.
Since this restriction of $\phi_{IJ}$ to this solution set is an embedding from an open subset of a $d$-dimensional manifold into a $d$-dimensional manifold, it has open image and is a diffeomorphism to this image.
\end{proof}

Assuming for now that precompact transverse perturbations exist (as we will show in Proposition~\ref{prop:ext}), we can now deduce smoothness and compactness of the perturbed zero set.

\begin{prop} \label{prop:zeroS0}
Let $\Kk$ be a tame $d$-dimensional Kuranishi atlas with a reduction $ \Vv\sqsubset \Kk$, and 
suppose that $\nu: \bB_\Kk|_\Vv \to \bE_\Kk|_\Vv$ is a precompact transverse perturbation.
Then $|\bZ_\nu| = |(s+ \nu)^{-1}(0)|$ is a smooth closed $d$-dimensional manifold. 
Moreover, its quotient topology agrees with the subspace topology on ${\|(s+ \nu)^{-1}(0)\|\subset|\Kk|}$.
\end{prop}
\begin{proof}  
By Lemma~\ref{le:stransv}, the realization $|\bZ_\nu|$ is made from the (disjoint) union $\bigcup_I \bigl(s_I|_{V_I}+\nu_I)^{-1}(0)\bigr)$ of $d$-dimensional manifolds via an equivalence relation given by the smooth local diffeomorphisms \eqref{eq:ZphiIJ}.   
From this we can deduce that $|\bZ_\nu|$ is second countable (i.e.\ its topology has a countable basis of neighbourhoods).
Indeed, a basis is given by the projection of countable bases of each manifold $(s_I|_{V_I}+\nu_I)^{-1}(0)$ to the quotient.
The images are open in the quotient space since the relation between different components of $(s+\nu)^{-1}(0)$ is given by local diffeomorphisms, taking open sets to open sets. In other words: The preimage of an open set in $|\bZ_\nu|$ is a disjoint union of open subsets of $(s_I|_{V_I}+\nu_I)^{-1}(0)$.
This can be used to express any open set as a union of the basis elements. 
It also shows that $|\bZ_\nu|$ is locally smooth, since any choice of lift $x\in (s_I|_{V_I}+\nu_I)^{-1}(0)$ of a given point $[x]\in |\bZ_\nu|$ lies in some chart $\Nn \hookrightarrow \R^d$, where $\Nn\subset (s_I|_{V_I}+\nu_I)^{-1}(0)$ is open; thus as above $[\Nn]\subset|\bZ_\nu|$ is open and provides a local homeomorphism to $\R^d$ near $[x]$.

Moreover, as noted above, the continuous injection $|\bZ_\nu| \to |\Kk|$ from Lemma~\ref{lem:full} transfers the Hausdorffness of $|\Kk|$ from Proposition~\ref{prop:Khomeo} to the realization $|\bZ_\nu|$.
Thus $|\bZ_\nu|$ is a second countable Hausdorff space that is locally homeomorphic to a $d$-manifold.
 Hence it is a $d$-manifold, where we understand all manifolds to have empty boundary since the charts are to open sets in $\R^d$.

In order to prove compactness, it now suffices to prove sequential compactness, since every manifold is metrizable. (In fact, second countability suffices for the equivalence of compactness and sequential compactness, see \cite[Theorem~5.5]{Kel}.) 
For that purpose recall that $\nu$ is assumed to be precompact,  
in particular
there exists a precompact subset $\Cc\sqsubset \Vv$ so that
\begin{equation}\label{eq:C}
(s_I|_{V_I}+\nu_I)^{-1}(0) \;\subset\; \pi_\Kk^{-1}\bigl(\pi_\Kk(\Cc)\bigr) \cap V_I \qquad  \forall I\in \Ii_\Kk.
\end{equation}
Now to prove sequential compactness, consider a sequence $(p_k)_{k\in\N}\subset |\bZ_\nu|$. In the following we will index all subsequences by $k\in\N$ as well.
By finiteness of $\Ii_\Kk$ there is a subsequence of $(p_k)$ that has lifts in $(s_I|_{V_I}+\nu_I)^{-1}(0)$ for a fixed $I\in\Ii_\Kk$. 
In fact, by \eqref{eq:C}, and using the language of Definition~\ref{def:preceq}, the subsequence lies in
$$
V_I \cap \pi_\Kk^{-1}\bigl(\pi_\Kk(\Cc)\bigr)   \;=\;  V_I \cap {\textstyle \bigcup_{J\in\Ii_\Kk}} \eps_I(C_J)
\;\subset\; U_I .
$$
Here $\eps_I(C_J)=\emptyset$ unless $I\subset J$ or $J\subset I$, due to the intersection property (ii) of Definition~\ref{def:vicin} and the inclusion $C_J\subset V_J$.
So we can choose another subsequence and lifts $(x_k)_{k\in\N}\subset V_I$ with $\pi_\Kk(x_k)=p_k$ such that either 
$$
(x_k)_{k\in\N}\subset V_I \cap \phi_{IJ}^{-1}(C_J)
\qquad\text{or}\qquad 
(x_k)_{k\in\N}\subset V_I \cap \phi_{JI}(C_J\cap U_{JI})
$$ 
for some $I\subset J$ or some $J\subset I$.
In the first case, compatibility of the perturbations \eqref{eq:comp} implies that there are other lifts
$\phi_{IJ}(x_k)\in (s_J|_{V_J}+\nu_J)^{-1}(0)\cap C_J$, which by precompactness $\ov{C_J}\sqsubset V_J$ have a convergent subsequence $\phi_{IJ}(x_k)\to y_\infty \in (s_J|_{V_J}+\nu_J)^{-1}(0)$.
Thus we have found a limit point in the perturbed zero set 
$p_k = \pi_\Kk(\phi_{IJ}(x_k)) \to \pi_\Kk(y_\infty) \in |\bZ_\nu|$,
as required for sequential compactness.

In the second case we use the relative closeness of $\phi_{JI}(U_{JI})=s_J^{-1}(E_I)\subset U_I$ and the precompactness $V_I\sqsubset U_I$ to find a convergent subsequence 
$x_k\to x_\infty \in \ov{V_I} \cap \phi_{JI}(U_{JI})$.
Since $\phi_{JI}$ is a homeomorphism to its image, this implies convergence of the preimages
$y_k:= \phi_{JI}^{-1}(x_k) \to \phi_{JI}^{-1}(x_\infty) =: y_\infty \in U_{JI}$.
By construction and compatibility of the perturbations \eqref{eq:comp}, this subsequence $(y_k)$ of lifts of $\pi_\Kk(y_k)=p_k$ moreover lies in $(s_J|_{V_J}+\nu_J)^{-1}(0)\cap C_J$.
Now precompactness of $C_J\sqsubset V_J$ implies $y_\infty\in V_J$, and continuity of the section implies $y_\infty\in (s_J|_{V_J}+\nu_J)^{-1}(0)$. Thus we again have a limit point 
$p_k = \pi_\Kk(y_k) \to \pi_\Kk(y_\infty) \in |\bZ_\nu|$.
This proves that the perturbed zero set $|\bZ_\nu|$ is sequentially compact, and hence compact. 
Therefore it is a closed manifold, as claimed. 

Finally, the map \eqref{eq:Zinject} now is a continuous bijection between the compact space $|\bZ_\nu|$ and the Hausdorff space $\|(s+\nu)^{-1}(0)\|\subset|\Kk|$ with the relative topology induced by $|\Kk|$. As such it is automatically a homeomorphism $|\bZ_\nu|\cong \|(s+\nu)^{-1}(0)\|$, see Remark~\ref{rmk:hom}
\end{proof}

We now extend these results to a tame Kuranishi cobordism $\Kk^{[0,1]}$ from $\Kk^0$ to $\Kk^1$ with cobordism reduction $\Vv$.
Recall from Definition~\ref{def:cvicin} that $\Vv$ induces reductions $\partial^\al\Vv := \bigcup_{I\in\Ii_{\Kk^\al}} \partial^\al V_I \subset \Obj_{\bB_{\Kk^\al}}$ of $\Kk^\al$ for $\al=0,1$, where we identify the index set $\Ii_{\Kk^\al}\cong\io^\al(\Ii_{\Kk^\al})$ with a subset of $\Ii_{\Kk^{[0,1]}}$.

\begin{defn} \label{def:csect}
A {\bf reduced cobordism section} of $s_{\Kk^{[0,1]}}|_{\Vv}$ is a reduced section $\nu:\bB_{\Kk^{[0,1]}}|_\Vv\to\bE_{\Kk^{[0,1]}}|_{\Vv}$ as in Definition~\ref{def:sect} that in addition has product form in a collar neighbourhood of the boundary. 
That is, for $\al=0,1$ and 
$I\in \Ii_{\Kk^\al}\subset\Ii_{\Kk^{[0,1]}}$ 
there is $\eps>0$ and a map $\nu_I^\al: \p^\al V_I\to E_I$ such that 
$$
\nu_I \bigl( \io_I^\al ( x,t ) \bigr) 
= \nu_I^\al (x)  \qquad
\forall\, x\in \p^\al V_I , \   t\in A^\al_\eps .
$$ 
A {\bf precompact, transverse cobordism perturbation} of $s_{\Kk^{[0,1]}}|_{\Vv}$ is a reduced cobordism section $\nu$ as above that satisfies the transversality condition $s|_{V_I}+\nu_I \pitchfork 0$ on the interior of the domains $V_I$, and whose domain is part of a nested cobordism reduction $\Cc\sqsubset \Vv$ such that 
$\pi_\Kk\bigl((s + \nu)^{-1}(0)\bigr)\subset \pi_\Kk(\Cc)$.
We moreover call such $\nu$ {\bf admissible} if it satisfies \eqref{eq:admiss}.
\end{defn}

The collar form ensures that the transversality of the perturbation extends to the boundary of the domains, as follows.

\begin{lemma}\label{le:ctransv}
If $\nu:\bB_{\Kk^{[0,1]}}|_\Vv\to\bE_{\Kk^{[0,1]}}|_{\Vv}$ is a 
precompact,
transverse  
cobordism 
perturbation of $s_{\Kk^{[0,1]}}|_{\Vv}$, then the {\bf restrictions} $\nu|_{\partial^\al\Vv} := \bigl(  \nu_I^\al \bigr)_{I\in\Ii_{\Kk^\al}}$ for $\al=0,1$ are precompact, transverse perturbations of the restricted canonical sections $s_{\Kk^\al}|_{\partial^\al\Vv}$. If in addition $\nu$ is admissible, then so are the restrictions $\nu|_{\partial^\al\Vv}$.

Moreover, each perturbed section $s|_{V_I}+\nu_I : V_I \to E_I$ for $I\in\Ii_{\Kk^0}\cup\Ii_{\Kk^1}\subset \Ii_{\Kk^{[0,1]}}$ is transverse to $0$ as a map on a domain with boundary. That is, the kernel of its differential is transverse to the boundary $\partial V_I = \bigcup_{\al=0,1}\iota^\al_I (\partial^\al V_I\times \{\al\})$.
\end{lemma}

\begin{proof}
Precompactness transfers to the restriction since the restrictions of the nested cobordism reduction are nested reductions $\p^\al\Cc \sqsubset \p^\al\Vv$ for $\al=0,1$.
Similarly, admissibility transfers immediately by pullback of \eqref{eq:admiss} to the boundaries via $\io^\al_I : \p^\al V_I \times \{\al\} \to V_I$. 
Transversality in a collar neighbourhood of the boundary $\io^\al_I(\partial^\al\Vv\times A^\al_\eps)\subset V_I$ is equivalent to transversality of the restriction $s|_{\partial^\al V_I}+\nu^\al_I \pitchfork 0$ since 
$\rd \bigl( \nu_I \circ \io_I^\al \bigr) =
 0 \rd t + \rd\nu_I^\al$.
Moreover, transversality of $s|_{V_I}+\nu_I : V_I \to E_I$
at the boundary of $V_I$, 
as a map on a domain with boundary, is equivalent under pullback with the embeddings $\io^\al_I$ to transversality of
$f := \bigl( s|_{V_I}+\nu_I \bigr)\circ\io^\al_I : 
\partial^\al V_I \times A^\al_\eps \to E_I$. 
For the latter, the kernel $\ker\rd_{s,x} f = \ker\rd_x \nu^\al_I   \times \R$ 
is indeed transverse to the boundary $\rT_x \partial^\al V_I\times \{\al\}$ in  $\rT_x \partial^\al V_I \times \R$.
\end{proof}

With that, we can show that 
precompact 
transverse perturbations of the Kuranishi cobordism induce smooth cobordisms (up to orientations) between the perturbed zero sets of the restricted perturbations. 

\begin{lemma} \label{le:czeroS0}
Let $\nu: \bB_{\Kk^{[0,1]}}|_{\Vv^{[0,1]}} \to \bE_{\Kk^{[0,1]}}|_{\Vv^{[0,1]}}$ be a precompact, transverse cobordism perturbation.
Then $|\bZ_{\nu}|$, defined as in Definition~\ref{def:sect2}, is  a compact manifold whose boundary $\p|\bZ_\nu|$ is diffeomorphic to the disjoint union of $|\bZ_{\nu^0}|$ and $|\bZ_{\nu^1}|$, where $\nu^\al :=\nu|_{\partial^\al\Vv}$ are the restricted transverse perturbations of $s_{\Kk^\al}|_{\p^\al \Vv}$ for $\al=0,1$.

\end{lemma}
\begin{proof}
The topological properties of $|\bZ_{\nu}|$ follow from the arguments in Proposition~\ref{prop:zeroS0},
and smoothness of the zero sets follows as in Lemma~\ref{le:stransv}. However, the zero sets $(s_I|_{V_I}+\nu_I)^{-1}(0)$ for $I\in\Ii_{\Kk^\al}\subset\Ii_{\Kk^{[0,1]}}$ are now submanifolds with boundary, by the implicit function theorem on the interior of $V_I$ together with the smooth product structure on the collar neighbourhoods $\io_I^\al (\p^\al V_I \times A^\al_\eps)$ of the boundary. The latter follows from the smoothness of 
$(s_I|_{\p^\al V_I}+\nu^\al_I)^{-1}(0)$
from Lemma~\ref{le:stransv} and the embedding
$$
(s_I|_{V_I}+\nu_I)^{-1}(0)\;\cap\; \io_I^\al (\p^\al V_I \times \{\al\}) \;=\; \io^\al_I \bigl( 
(s_I|_{\p^\al V_I}+\nu^\al_I)^{-1}(0) \times\{\al\} \bigr)   .
$$
This gives $(s+\nu)^{-1}(0)$ the structure of a compact manifold with two disjoint boundary components for $\al=0,1$ given by
$$
\partial^\al \bigl( (s+\nu)^{-1}(0) \bigr) \;=\;  \underset{{I\in\Ii_{\Kk^\al}}}{\textstyle\bigcup} 
\io^\al_I \bigl( (s_I|_{\p^\al V_I}+\nu^\al_I)^{-1}(0) \times\{\al\} \bigr) ,
$$
which are diffeomorphic via $\partial^\al \Vv \ni (I,x) \mapsto \iota^\al_I(x,\al)$ 
to the submanifolds $(s+\nu^\al)^{-1}(0)\subset \partial^\al \Vv$ given by the restricted perturbations 
$\nu^\al=\bigl(\nu^\al_I\bigr)_{I\in\Ii_{\Kk^\al}}$.
By the collar form of the coordinate changes in $\Kk^{[0,1]}$ this induces fully faithful functors $j^\al$ from $\bZ_{\nu^\al}$ to the full subcategories of $\bZ_\nu$ with objects $\partial^\al \bigl( (s+\nu)^{-1}(0) \bigr)$.

Moreover, as in Lemma~\ref{le:stransv}, the morphisms are given by restrictions of the embeddings $\phi_{IJ}$, which are in fact local diffeomorphisms, and hence can be inverted to give $\bZ_\nu$ the structure of a groupoid.
Again using the collar form of the coordinate changes, there are no morphisms between $\partial^\al \bigl( (s+\nu)^{-1}(0) \bigr)$ and its complement in $(s+\nu)^{-1}(0)$, so the realization $|\bZ_{\nu}|$ inherits the structure of a compact manifold with boundary
$\partial |\bZ_{\nu}| = \bigcup_{\al=0,1} \partial^\al |\bZ_{\nu}|$ with two disjoint boundary components  
$$
\partial^\al |\bZ_{\nu}|  
\,:=\; \partial^\al |\Kk^{[0,1]}| \cap |\bZ_{\nu}|  
\;=\; \Bigl|  {\textstyle\bigcup}_{I\in\Ii_{\Kk^\al}} 
\io^\al_I \bigl( (s_I|_{\p^\al V_I}+\nu^\al_I)^{-1}(0) \times\{\al\} 
\bigr) \Bigr| .
$$
Since the fully faithful functors $j^\al$ are diffeomorphisms between the object spaces, they then descend to diffeomorphisms to the boundary components,
$$
|j^\al| \,:\;  |\bZ_{\nu^\al}| \;=\;  
 \Bigl|  {\textstyle\bigcup}_{I\in\Ii_{\Kk^\al}} 
 (s_I^\al+\nu^\al_I)^{-1}(0)  \Bigr|
 \;\overset{\cong}{\longrightarrow}\; \partial^\al |\bZ_{\nu}|  .
$$
Thus $|\bZ_{\nu}|$ is a (not yet oriented) cobordism between $|\bZ_{\nu^0}|$ and $|\bZ_{\nu^1}|$, as claimed.
\end{proof}

\subsection{Construction of perturbations} \label{ss:const}   \hspace{1mm}\\ \vspace{-3mm}

In this section, we let $(\Kk,d)$ be a metric tame Kuranishi atlas (or cobordism) and $\Vv$ a (cobordism) reduction, and construct precompact transverse (cobordism) perturbations of the canonical section $s_\Kk|_\Vv$.
In fact, we will construct a transverse perturbation with perturbed zero set contained in $\pi_\Kk(\Cc)$ for any given nested (cobordism) reduction $\Cc \sqsubset \Vv$.
This will be accomplished by an intricate construction that depends on the choice of two suitable constants $\de,\si>0$ depending on $\Cc\sqsubset\Vv$.
Consequently, the corresponding uniqueness statement requires not only the construction of transverse cobordism perturbations $\nu$ of $s_\Kk|_\Vv$ in a nested cobordism reduction $\Cc \sqsubset \Vv$, and with given restrictions $\nu|_{\p^\al\Vv}$, but also an understanding of the dependence on the choice of constants $\de,\si>0$.
We begin by describing the setup, which will be used to construct perturbations for both Kuranishi atlases and Kuranishi cobordisms. It is important to have this in place before describing the iterative constructions because, firstly, the iteration depends on the above choice of constants, and secondly, even the statements about uniqueness and existence of perturbations in cobordisms need to take this intricate setup into consideration.

First, we construct compatible norms on the obstruction spaces by using the additivity isomorphisms given by Definition~\ref{def:Ku2},
\begin{equation} \label{eq:iI}
{\textstyle \prod_{i\in I}} \;\Hat\phi_{iI}: \; {\textstyle \prod_{i\in I}} \; E_i \;\stackrel{\cong}\longrightarrow \; E_I \;=\;\oplus_{i\in I} \Hat\phi_{iI}(E_i) ,
\end{equation}
which gives a unique decomposition of any map $f_I:{\rm dom}(f_I)\to E_I$ into components $\bigl(f^i_I : {\rm dom}(f_I) \to \Hat\phi_{iI}(E_i)\bigr)_{i\in I}$ determined by $f_I = \sum_{i\in I} f^i_I$.  
Now we choose norms $\|\cdot\|$ on the basic obstruction spaces $E_i$ for $i=1,\dots N$, and extend them to all obstruction spaces $E_I$ by
$$
\| e \| \;=\;
\Bigl\| {\textstyle \sum_{i\in I}} \Hat\phi_{iI} (e_i) \Bigr\| \,:=\; \max_{i\in I} \| e_i\|
\qquad
\forall e=  {\textstyle \sum_{i\in I}} \Hat\phi_{iI} (e_i) \in E_I. 
$$
Here we chose the maximum norm on the Cartesian product since this guarantees estimates of the components $\|e_i\| \leq \|e\|$.
This construction guarantees that each embedding $\Hat\phi_{IJ}:(E_I,\|\cdot\|) \to (E_J,\|\cdot\|)$ is an isometry by the cocycle condition $\Hat\phi_{IJ}\circ\Hat\phi_{iI}=\Hat\phi_{iJ}$. 
Moreover, we will throughout use the supremum norm for functions, that is for any $f_I:{\rm dom}(f_I)\to (E_I,\|\cdot\|)$ we denote
$$
\bigl\| f_I \bigr\| \,:=\; \sup_{x\in {\rm dom}(f_I)} \bigl\| f_I(x) \bigr\|  \;=\; \sup_{x\in {\rm dom}(f_I)} \max_{i\in I} \, \bigl\| f^i_I(x) \bigr\|   
 \;=:\, \max_{i\in I}\, \bigl\| f^i_I \bigr\|   .
$$
Next, recall from Lemma~\ref{le:metric} (which holds in complete analogy for metric Kuranishi cobordisms) that the metric $d$ on $|\Kk|$ induces metrics $d_I$ on each domain $U_I$ such that the coordinate changes $\phi_{IJ}:(U_{IJ},d_I)\to (U_J,d_J)$ are isometries.
Moreover, Lemma~\ref{le:delred}~(i) provides $\de>0$ so that $B_{2\de}(V_I)\sqsubset U_I$  for all $I\in\Ii_\Kk$,
and $B_{2\de}(\pi_\Kk(\ov{V_I}))\cap B_{2\de}(\pi_\Kk(\ov{V_J})) = \emptyset$ unless $I\subset J$ or $J\subset I$, and hence the precompact neighbourhoods 
\begin{equation}\label{eq:VIk}
V_I^k \,:=\; B^I_{2^{-k}\de}(V_I) 
\;\sqsubset\; U_I \qquad \text{for} \; k \geq 0
\end{equation}  
form further (cobordism) reductions, all of which contain $\Vv$.
Here we chose separation distance $2\de$ so that compatibility of the metrics ensures the strengthened version of the separation condition (ii) in Definition~\ref{def:vicin} for $I\not\subset J$ and $J\not\subset I$,
\begin{equation}\label{desep}
B_\de\bigl(\pi_\Kk(V^k_I)\bigr) \cap B_\de\bigl(\pi_\Kk(V^k_J)\bigr) \subset  
B_{\de + 2^{-k}\de}\bigl(\pi_\Kk(V_I)\bigr) \cap B_{\de + 2^{-k}\de}\bigl(\pi_\Kk(V_J)\bigr) = \emptyset .
\end{equation}
In case $I\subsetneq J$, Lemma~\ref{le:Ku2} then gives the identities
\begin{align}\label{eq:N}\notag
V^k_I \cap \pi_\Kk^{-1}(\pi_\Kk(V^k_J))&\;=\;V^k_I  \cap \phi_{IJ}^{-1}(V^k_J)  ,\\
V^k_J \cap \pi_\Kk^{-1}(\pi_\Kk(V^k_I))&\;=\;V^k_J \cap \phi_{IJ}(V^k_I \cap U_{IJ})  \;=:\, N^k_{JI} 
\end{align}
for the sets on which we will require compatibility of the perturbations $\nu_I$ and $\nu_J$.
The analogous identities hold for any combinations of the nested precompact open sets 
$$
C_I \;\sqsubset\; V_I \;\sqsubset\; \ldots V^{k'}_I \;\sqsubset\; V^{k}_I \ldots \;\sqsubset\; V^0_I  ,
$$
where $k'>k >0$ are any positive reals.
For the sets $N^k_{JI}\subset U_J$ introduced in \eqref{eq:N} above, note that by the compatibility of metrics we have inclusions for any $H\subsetneq J$,
$$
\pi_\Kk(B^J_\de(N^k_{JH}))
\subset 
B_\de \bigl(\pi_\Kk(N^k_{JH})\bigr)
\subset 
B_\de \bigl(\pi_\Kk(\phi_{HJ}(V^k_H\cap U_{HJ}))\bigr)
\subset B_{\de}\bigl( \pi_\Kk(V^k_H) \bigr) .
$$
So \eqref{desep} together with the injectivity of $\pi_\Kk|_{U_J}$ implies  
for any $H,I \subsetneq J$
\begin{equation}\label{Nsep}
B^J_\de(N^k_{JH}) \cap B^J_\de(N^k_{JI}) \neq \emptyset \qquad \Longrightarrow \qquad H\subset I \;\text{or} \; I\subset H.
\end{equation}
Moreover, we have precompact inclusions for any $k'>k\geq 0$
\begin{equation} \label{preinc}
N^{k'}_{JI} \;=\;V^{k'}_J \cap  \phi_{IJ}(V^{k'}_I\cap U_{IJ}) \;\sqsubset\; V^k_J \cap \phi_{IJ}(V^k_I\cap U_{IJ})  \;=\;N^k_{JI} ,
\end{equation}
since $\phi_{IJ}$ is an embedding to the relatively closed subset $s_J^{-1}(E_I)\subset U_J$ and thus $\ov{\phi_{IJ}(V^{k'}_I\cap U_{IJ})} = \phi_{IJ}\bigl(\,\ov{V^{k'}_I}\cap U_{IJ}\bigr) \subset \phi_{IJ}(V^k_I\cap U_{IJ})$.
Next, we abbreviate
$$
N^k_J \, := \;{\textstyle \bigcup_{J\supsetneq I}} N^k_{JI} \;\subset\; V^k_J ,
$$  
and will call the union $N^{|J|}_J$ the {\it core} of $V^{|J|}_J$, since it is the part of this set on which we will prescribe $\nu_J$ in an iteration by a compatibility condition with the $\nu_I$ for $I\subsetneq J$.
In this iteration we will be working with quarter integers between $0$ and 
$$
M \,:=\; M_\Kk \,:=\; \max_{I\in\Ii_\Kk} |I|  ,
$$
and need to introduce another constant $\eta_0>0$ that controls the intersection with $\im\phi_{IJ}=\phi_{IJ}(U_{IJ})$ for all $I\subsetneq J$ as in Figure~\ref{fig:4},
\begin{equation}\label{eq:useful} 
\im \phi_{IJ} \;\cap\;  B^J_{2^{-k-\frac 12}\eta_0} \bigl( N_{JI}^{k+\frac 34} \bigr) \;\subset\; N^{k+\frac 12}_{JI} 
\qquad \forall \; k\in  \{0,1,\ldots,M\}.
\end{equation}
Since $\phi_{IJ}$ is an isometric embedding, this inclusion holds whenever 
$2^{-k-\frac 12}\eta_0 + 2^{-k-\frac 34}\de \leq 2^{-k-\frac 12}\de$ 
for all $k$. To minimize the number of choices in the construction of perturbations, we may thus simply fix $\eta_0$ in terms of $\de$ by
\begin{equation}\label{eq:eta0}
\eta_0 \,:=\; (1 -  2^{-\frac 14} ) \de.
\end{equation} 
Then we also have $2^{-k}\eta_0 + 2^{-k-\frac 12}\de <  2^{-k}\de$,
which provides the inclusions
\begin{equation}\label{eq:fantastic}
B^I_{\eta_k}\Bigl(\;\ov{V_I^{k+\frac 12}}\;\Bigr) \;\subset\; V_I^k \qquad\text{for} \;\; k\geq 0, \; \eta_k:=2^{-k}\eta_0 .
\end{equation}

\begin{figure}[htbp] 
   \centering
   \includegraphics[width=3in]{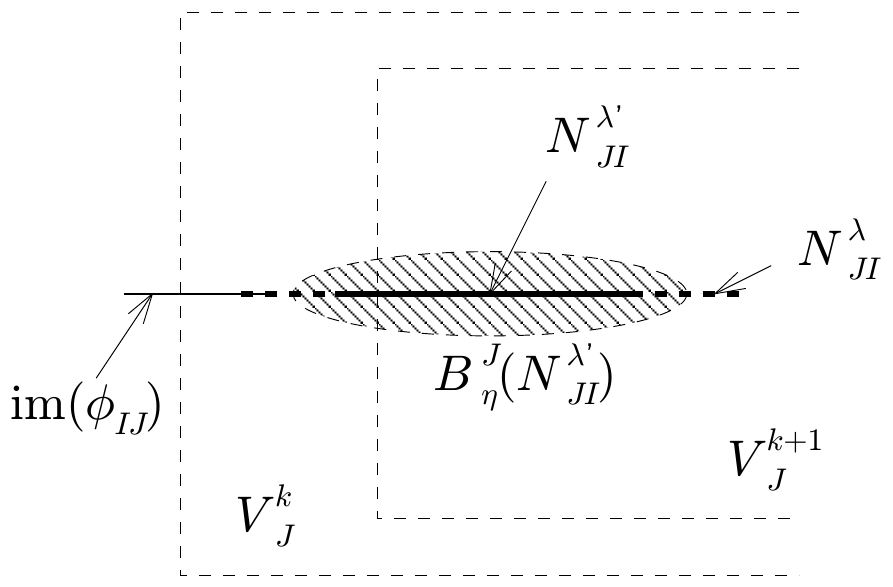} 
 \caption{
This figure illustrates the nested sets $V^{k+1}_J \sqsubset V^k_k$ and 
$ N^{\la'}_{JI}\sqsubset N^{\la}_{JI}\subset \im(\phi_{IJ})\cap V^k_J$ for $k+1>\la'=k+\frac 34 > \la =k+\frac 12>k$, the shaded neighbourhood $B^J_{\eta}(N^{\la'}_{JI})$ for $\eta = 2^{-\la}\eta_0$, and the inclusion given by \eqref{eq:useful}.
}
   \label{fig:4}
\end{figure}

Continuing the preparations, let a nested (cobordism) reduction $\Cc \sqsubset \Vv$ be given. Then we denote the open subset contained in $U_J\cap \pi_\Kk^{-1}(\pi_\Kk(\Cc))$ for $J\in\Ii_\Kk$ by
\begin{equation} \label{ticj}
\Ti C_J \,:=\;   {\textstyle \bigcup_{K\supset J}} \, \phi_{JK}^{-1}(C_K)   \;\subset\; U_J  .
\end{equation}
The assumption $\io_\Kk(X)\subset\pi_\Kk(\Cc)$ implies $s_J^{-1}(0)\subset \pi_\Kk^{-1}(\pi_\Kk(\Cc))$, and so by \eqref{eq:N} the zero set $s_J^{-1}(0)$ is contained in $\Ti C_J \;\cup\;  {\textstyle \bigcup_{I\subsetneq J}} \phi_{IJ}(C_I\cap U_{IJ})$, the union of an open set $\Ti C_J$ and a set in which the perturbed zero sets will be controlled by earlier iteration steps.

\begin{defn}  \label{def:admiss}
Given a reduction $\Vv$ of a metric Kuranishi atlas (or cobordism) $(\Kk,d)$, we set $\de_\Vv>0$ to be the maximal constant such that any $2\de< 2\de_\Vv$ satisfies the reduction properties of Lemma~\ref{le:delred}, that is 
\begin{align}
\label{eq:de1}
B_{2\de}(V_I)\sqsubset U_I\qquad &\forall I\in\Ii_\Kk , \\
\label{eq:dedisj}
B_{2\de}(\pi_\Kk(\ov{V_I}))\cap B_{2\de}(\pi_\Kk(\ov{V_J})) \neq \emptyset &\qquad \Longrightarrow \qquad I\subset J \;\text{or} \; J\subset I .
\end{align}
Given a nested reduction $\Cc\sqsubset\Vv$ of a metric Kuranishi atlas $(\Kk,d)$ and $0<\de<\de_\Vv$, we set $\eta_{|J|-\frac 12} :=  2^{-|J|+\frac 12} \eta_0 = 2^{-|J|+\frac 12} (1 -  2^{-\frac 14} ) \de$ and
\begin{equation*}
\si(\de,\Vv,\Cc) \,:=\; \min_{J\in\Ii_\Kk}
\inf \Bigl\{
\; \bigl\| s_J(x) \bigr\| \;\Big| \;
x\in \ov{V^{|J|}_J} \;\less\; \Bigl( \Ti C_J \cup {\textstyle \bigcup_{I\subsetneq J}} B^J_{\eta_{|J|-\frac 12}}\bigl(N^{|J|-\frac14}_{JI}\bigr) \Bigr) \Bigr\}  .
\end{equation*}
\end{defn}

In this language, the previous development of setup in this section shows that for any metric Kuranishi atlas or cobordism $(\Kk,d)$ we have $\de_\Vv>0$. We note some further properties of these constants. 
Note first that there is no general relation between $\si(\de,\Vv,\Cc)$ and $\si(\de',\Vv,\Cc)$ for $0<\de'<\de<\de_\Vv$ since both $V^{|J|}_J$ and $B^J_{\eta_{|J|-\frac 12}}\bigl(N^{|J|-\frac14}_{JI}\bigr)$ grow with growing $\de$, and hence the domains of the infimum are not nested in either way.

\begin{lemma}\label{le:admin}
\begin{enumerate}
\item
Let $\Cc\sqsubset\Vv$ be a nested reduction of a metric Kuranishi atlas or cobordism, and let $\de<\de_\Vv$, then we have $\si(\de,\Vv,\Cc)>0$.
\item 
For any reduction $\Vv$ of a metric Kuranishi atlas we have $\de_\Vv=\de_{\Vv\times[0,1]}$.
\item
Given a metric Kuranishi cobordism $(\Kk,d)$ we equip the Kuranishi atlases $\p^\al\Kk$ for $\al=0,1$ with the restricted metrics $d\big|_{|\p^\al\Kk|}$.
Then for any cobordism reduction $\Vv$ we have $\de_{\p^\al\Vv}\geq \de_\Vv$
for $\al=0,1$.
\item
In the setting of (iii), let $\eps>0$ be the collar width of $(\Kk,d)$.
Then the neighbourhood of radius $r<\eps$ of any $\eps$-collared set $W\subset U_I$ 
(i.e.\ with $W\cap \io^\al_I(\p^\al U_I \times A^\al_\eps)=\io^\al_I(\p^\al W \times A^\al_\eps)$) is $(\eps-r)$-collared,
\begin{equation} \label{eq:Wnbhd}
B^I_r(W) \cap \io^\al_I\bigl(\p^\al U_I \times A^\al_{\eps-r}\bigr)  = \io^\al_I\bigl( B^{I,\al}_r(\p^\al W)  \times A^\al_{\eps-r} \bigr) , 
\end{equation}
with $\p^\al B^I_r(W) = B^{I,\al}_r(\p^\al W)$, where we denote by $B^{I,\al}_r$ the neighbourhoods in $\p^\al U_I$ induced by pullback of the metric $d_I$ with $\io^\al_I:\p^\al U_I \times\{\al\} \to U_I$.
\item
If in (iv) the collared sets $W\subset U_I$ are obtained as products with $[0,1]$ in a product Kuranishi cobordism $\Kk=\Kk'\times[0,1]$, then \eqref{eq:Wnbhd} holds for any $r>0$ with $\eps-r$ replaced by $1$.
\end{enumerate}
\end{lemma}
\begin{proof}
To check statement (i) it suffices to fix $J\in\Ii_\Kk$ and consider the continuous function $\|s_J\|$ over the compact set $\ov{V^{|J|}_J} \;\less\; \bigl( \Ti C_J \cup {\textstyle \bigcup_{I\subsetneq J}} B^J_{\eta_{|J|-\frac 12}}\bigl(N^{|J|-\frac14}_{JI}\bigr) \bigr)$. We claim that its infimum is positive since the domain is disjoint from $s_J^{-1}(0)$. Indeed, the reduction property $\io_\Kk(X)\subset\pi_\Kk(\Cc)$ implies $s_J^{-1}(0)\subset \Ti C_J \cup \bigcup_{I\subsetneq J}  \phi_{IJ}(C_I\cap U_{IJ})$, the intersections $\ov{V^{|J|}_J}\cap \phi_{IJ}(C_I\cap U_{IJ})$ are contained in $N^{k}_{JI}$ for any $k<|J|$ since $V^{|J|}_J \sqsubset V^k_J$ and $C_I\sqsubset V_I \subset V^k_I$, and we have $N^{k}_{JI}\subset  B^J_{\eta_{|J|-\frac 12}}\bigl(N^{|J|-\frac14}_{JI}\bigr)$ for $k\geq -|J|+\frac 12$.

Statement (ii) holds since all sets and metrics
involved are of product form. 

Statement (iii) follows by pullback with $\io^\al_I|_{\p^\al U_I \times\{\al\}}$ since the $2^{-k}\de$-neighbourhood $(\p^\al V_I)^k$ of the boundary $\p^\al V_I$ within $\p^\al U_I$ is always contained in the boundary $\p^\al V^k_I$ of the $2^{-k}\de$-neighbourhood of the domain $V_I$.

To check (iv) note in particular the product forms $(\io^\al_I)^{-1}(V_I)=\p^\al V_I \times A^\al_\eps$ and
$(\io^\al_I)^*d_I = d^\al_I + d_\R$ on the $\eps$-collars, where $d^\al_I$ denotes the metric on $\p^\al U_I$ induced from the restriction of the metric on $|\Kk|$ to $|\p^\al\Kk|$.
Then for any $\p^\al W \subset \p^\al U_I$ the product form of the metric implies product form of the $r$-neighbourhoods 
$$
B^I_r\bigl(\io^\al_I(\p^\al W\times A^\al_\eps)\bigr)\cap \im\io^\al_I \;=\; \io^\al_I\bigl( B^{I,\al}_r(\p^\al W)\times A^\al_\eps\bigr) .
$$
Moreover, for any $W' \subset U_I\less\im\io^\al_I$ and $r < \eps$, 
the collaring condition \eqref{eq:epscoll} on the metric implies that 
\begin{equation}\label{eq:nob}
B^I_r(W') \cap \io^\al_I(\p^\al U_I\times A^\al_{\eps-r}) = \emptyset .
\end{equation}
The identity \eqref{eq:Wnbhd} now follows from applying the above identities with $W'=W \less \im\io^\al_I$.

In the product case (v), the complement of the (closed) collars is empty, so there is no need for the second identity and hence for the restriction $r<\eps$.
\end{proof}

In the case of a metric tame Kuranishi atlas we will construct transverse perturbations $\nu = \bigl(\nu_I : V_I \to E_I \bigr)_{I\in\Ii_\Kk}$ by an iteration which constructs and controls each $\nu_I$ over the larger set ${V_I^{|I|}}$.  In order to prove uniqueness of the VMC, we will moreover need to interpolate between any two such perturbations by a similar iteration. We will use the following definition to keep track of the refined properties of the sections in this iteration.

\begin{defn}  \label{a-e}
Given a nested reduction $\Cc\sqsubset\Vv$ of a metric tame Kuranishi atlas 
$(\Kk,d)$ and constants $0<\de<\de_\Vv$ and $0<\si\le\si(\de,\Vv,\Cc)$, we say that a perturbation $\nu$ of $s_\Kk|_\Vv$ is {\bf $(\Vv,\Cc,\de,\si)$-adapted} if the sections $\nu_I:V_I\to E_I$ extend to sections over ${V^{|I|}_I}$ (also denoted $\nu_I$) so that the following conditions hold for every $k=1,\ldots, M$ with 
$$
M_\Kk:= \max_{I\in\Ii_\Kk} |I|, \qquad
\eta_k:=2^{-k}\eta_0=2^{-k} (1-2^{-\frac 14})\de .
$$
\begin{itemize}
\item[a)]
The perturbations are compatible in the sense that the commuting diagrams in Definition~\ref{def:sect} hold on $\bigcup_{|I|\leq k} {V^k_I}$, that is
$$
\qquad
\nu_I \circ \phi_{HI} |_{{V^k_H}\cap \phi_{HI}^{-1}({V^k_I})} \;=\; \Hat\phi_{HI} \circ \nu_H |_{{V^k_H}\cap \phi_{HI}^{-1}({V^k_I})} 
\qquad \text{for all} \; H\subsetneq I , |I|\leq k .
$$
\item[b)]
The perturbed sections are transverse, that is $(s_I|_{{V^k_I}} + \nu_I) \pitchfork 0$ for each $|I|\leq k$.
\item[c)]
The perturbations are {\it strongly admissible} with radius $\eta_k$, that is for all $H\subsetneq I$ and $|I|\le k$ we have
$$
\qquad
\nu_I( B^I_{\eta_k}(N^{k}_{IH})\bigr) \;\subset\; \Hat\phi_{HI}(E_H) 
\qquad
\text{with}\;\;
N^k_{IH} = V^k_I \cap \phi_{HI}(V^k_H\cap U_{HI}) .
$$
In particular, the perturbations are admissible along the core $N^k_I$, that is we have $\im\rd_x\nu_I \subset \im\Hat\phi_{HI}$ at all $x\in N^k_{IH}$.
\item[d)]  
The perturbed zero sets are contained in $\pi_\Kk^{-1}\bigl(\pi_\Kk(\Cc)\bigr)$; more precisely
$$
(s_I |_{{V^k_I}}+ \nu_I)^{-1}(0) \;\subset\; {V^k_I} \cap \pi_\Kk^{-1}\bigl(\pi_\Kk(\Cc)\bigr)
\qquad
\forall |I|\leq k,
$$
or equivalently $s_I + \nu_I \neq 0$ on ${V^k_I} \less  \pi_\Kk^{-1}\bigl(\pi_\Kk(\Cc)\bigr)$.
\item[e)]
The perturbations are small, that is $\sup_{x\in {V^k_I}} \| \nu_I (x) \| < \si$
for $|I|\leq k$. 
\end{itemize}

Given a metric Kuranishi atlas $(\Kk,d)$, we say that a perturbation $\nu$  is {\bf adapted} if it is a $(\Vv,\Cc,\de,\si)$-adapted perturbation $\nu$ of $s_\Kk|_\Vv$ for some choice of nested reduction $\Cc\sqsubset\Vv$ and constants $0<\de<\de_\Vv$ and $0<\si\le\si(\de,\Vv,\Cc)$. 
\end{defn}

Next, we note some simple properties of these notions; in particular the fact that adapted perturbations are automatically admissible, precompact, and transverse.

\begin{lemma}\label{le:admin2}
\begin{enumerate}
\item
Any $(\Vv,\Cc,\de,\si)$-adapted perturbation $\nu$ of $s_\Kk|_\Vv$ is an admissible, precompact, transverse perturbation with $\pi_\Kk( (s+\nu)^{-1}(0))\subset\pi_\Kk(\Cc)$. 
\item 
If $\nu$ is a $(\Vv,\Cc,\de,\si)$-adapted perturbation, then it is also $(\Vv,\Cc,\de',\si')$-adapted for any $\de'\le\de$ and $\si'\in \bigl[\si, \si(\de',\Vv,\Cc)\bigr)$.
\end{enumerate}
\end{lemma}
\begin{proof}
To check statement (i), first note that $\nu$ is an admissible reduced section in the sense of Definition~\ref{def:sect} by c) and d), and is transverse by b). 
Restriction of a) implies that it satisfies the zero set condition $s_I+\nu_I \neq 0$ on $V_I \less \pi_\Kk^{-1}(\pi_\Kk(\Cc))$, and hence $\pi_\Kk( (s+\nu)^{-1}(0))\subset\pi_\Kk(\Cc)$, which in particular implies precompactness in the sense of Definition~\ref{def:precomp}.

Statement (ii) holds because the domains in a)-e) for $\de'$ are included in those for $\de$, and $\si$ only appears in the inequality of e).
\end{proof}

Using these notions, we now prove a refined version of the existence of admissible, precompact, transverse perturbations in every metric tame Kuranishi atlas.

\begin{prop}\label{prop:ext}
Let $(\Kk,d)$ be metric tame Kuranishi atlas with nested reduction $\Cc \sqsubset \Vv$.
Then for any $0<\de<\de_\Vv$ and $0<\si\le\si(\de,\Vv,\Cc)$ there exists a $(\Vv,\Cc,\de,\si)$-adapted perturbation $\nu$ of $s_\Kk|_{\Vv}$.  In particular, $\nu$ is admissible, precompact, and transverse, and its perturbed zero set $|\bZ_\nu|=|(s+\nu)^{-1}(0)|$ is compact with $\pi_\Kk\bigl((s+\nu)^{-1}(0)\bigr)$ contained in $\pi_\Kk(\Cc)$.
\end{prop}

\begin{proof}
We will construct $\nu_I: V^{|I|}_I\to E_I$ by an iteration over $k=0,\ldots,M= \max_{I\in\Ii_\Kk} |I|$, where in step $k$ we will define $\nu_I : V^k_I \to E_I$ for all $|I| = k$ that, together with the $\nu_I|_{V^k_I}$ for $|I|<k$ obtained by restriction from earlier steps, satisfy conditions a)-e) of Definition~\ref{a-e}.
Restriction to $V_I\subset V^{|I|}_I$ then yields a $(\Vv,\Cc,\de,\si)$-adapted perturbation $\nu$ of $s_\Kk|_\Vv$, which by Lemma~\ref{le:admin2}~(i) is automatically an admissible, precompact, transverse perturbation with $\pi_\Kk( (s+\nu)^{-1}(0))\subset\pi_\Kk(\Cc)$. 
Compactness of $|(s+\nu)^{-1}(0)|$ then follows from Proposition~\ref{prop:zeroS0}.
So it remains to perform the iteration.

For $k=0$ the conditions a)-e) are trivially satisfied since there are no index sets $I\in\Ii_\Kk$ with $|I|\leq 0$. Now suppose that $\bigl(\nu_I : V^k_I\to E_I\bigr)_{I\in\Ii_\Kk, |I|\leq k}$  are constructed such that a)-e) hold. In the next step we can then construct $\nu_J$ independently for each $J\in\Ii_\Kk$ with $|J|=k+1$, since for any two such $J,J'$ we have $\pi_\Kk(V_J^{k+1}) \cap \pi_\Kk(V_{J'}^{k+1})=\emptyset$ unless $J=J'$ by \eqref{desep}, and so the constructions for $J\neq J'$ are not related by the commuting diagrams in condition a).

\MS\NI
{\bf Construction for fixed $\mathbf {|J|=k+1}$:} 
We begin by noting that a) requires for all $I\subsetneq J$ 
\begin{equation} \label{some nu}
\nu_J|_{N^{k+1}_{JI}} \;=\; 
\nu_J|_{V^{k+1}_J \cap \phi_{IJ}(V^{k+1}_I\cap U_{IJ})} \;=\; \Hat\phi_{IJ}\circ \nu_I\circ\phi_{IJ}^{-1} .
\end{equation}
To see that these conditions are compatible, we note that for $H\neq I\subsetneq J$ with $\phi_{HJ}(V^k_H\cap U_{IJ}) \cap \phi_{IJ}(V^k_I\cap U_{IJ})\neq \emptyset $ property 
\eqref{Nsep}
implies $H\subsetneq I$ or $I\subsetneq H$. Assuming w.l.o.g.\ the first, we obtain compatibility from the strong cocycle condition \eqref{strong cocycle} and property a) for $H\subsetneq I$, 
\begin{align*}
&\Hat\phi_{HJ}\circ \nu_H\circ\phi_{HJ}^{-1} |_{V^k_J \cap \phi_{IJ}(V^k_I\cap U_{IJ})\cap \phi_{HJ}(V^k_H\cap U_{HJ})} \\
&=
\Hat\phi_{IJ}\circ \bigl(\Hat\phi_{HI}\circ \nu_H \bigr) |_{\phi_{HJ}^{-1}(V^k_J) \cap \phi_{HI}^{-1}(V^k_I) \cap V^k_H} \circ \phi_{HI}^{-1} \circ \phi_{IJ}^{-1} \\
&=
\Hat\phi_{IJ}\circ \bigl(\nu_I \circ \phi_{HI}\bigr) \circ \phi_{HI}^{-1} \circ \phi_{IJ}^{-1} 
\;=\;
\Hat\phi_{IJ}\circ \nu_I\circ\phi_{IJ}^{-1} .
\end{align*}
Here we checked compatibility on the domains $N^k_{JI}$, thus defining a map
\begin{equation}\label{eq:nuJ'}
\mu_J \,:\;  N^k_J = 
{\textstyle \bigcup_{I\subsetneq J}} 
N^k_{JI} \;\longrightarrow\; E_J , \qquad 
\mu_J|_{N^k_{JI}} := \Hat\phi_{IJ}\circ \nu_I\circ\phi_{IJ}^{-1} .
\end{equation}
Note moreover that by the compatible construction of norms on the obstruction spaces we have 
$$
\|\mu_J\| \,:=\; \sup_{y\in N^k_J} \|\mu_J(y)\| \;\leq\; \sup_{I\subsetneq J} \sup_{x\in V^k_I} \|\nu_I(x)\| \;<\; \si.
$$  
The construction of $\nu_J$ on $V^{k+1}_J$ then has three more steps.

\begin{itemize} 
\item {\bf Construction of extension:}
We construct an extension of the restriction of $\mu_J$ from \eqref{eq:nuJ'} to the enlarged core $N_J^{k+\frac 12}$. More precisely, we construct a smooth map $\Ti\nu_J : V^k_J \to E_J$ that satisfies
\begin{equation}\label{tinu}
\Ti\nu_J|_{N_J^{k+\frac 12}} \;=\; \mu_J|_{N_J^{k+\frac 12}} , \qquad\quad
\|\Ti\nu_J \| \;\leq \; \|\mu_J\| \;<\; \si ,
\end{equation}
and the strong admissibility condition on a larger domain than required in c),
\begin{equation}\label{value} 
\Ti\nu_J \bigl( B^J_{\eta_{k+\frac 12}}\bigl(N^{k+\frac 12}_{JI}\bigr) \bigr) \;\subset\; \Hat\phi_{IJ}(E_I) 
\qquad \forall\; I\subsetneq J .
\end{equation}
In case $k=1$ we achieve the analogous by setting $\Ti\nu_J:=0$.
\vspace{.03in}

\item {\bf Zero set condition:} 
We show that \eqref{value} and the control over $\|\Ti\nu_J\|$ imply the strengthened control of d) over the 
zero set of $s_J + \Ti\nu_J$, in particular
$$
\bigl(s_J|_{{V^{k+1}_J}} + \Ti\nu_J\bigr)^{-1}(0) \;\less\;  B^J_{\eta_{k+\frac 12}}\bigl(N^{k+\frac34}_J\bigr)   \;\subset\; \Ti C_J .
$$
In case $k=1$ this applies with the open set $\Ti C_J = C_J\subset U_J$.
\vspace{.03in}

\item {\bf Transversality:}  
We make a final perturbation $\nu_\pitchfork$ to obtain transversality for $s_J + \Ti\nu_J + \nu_\pitchfork$, 
while preserving conditions a),c),d), and then set $\nu_J: = \Ti\nu_J + \nu_\pitchfork$.
Moreover, taking $\|\nu_\pitchfork\|< \si - \|\Ti\nu_J\|$ ensures e).
\end{itemize}

\MS\NI
{\bf Construction of extensions:}  
 To construct $\Ti\nu_J$ in case $k\geq 2$ it suffices, in the notation of \eqref{eq:iI}, to extend each component $\mu^j_J$ for fixed $j\in J$.
For that purpose we iteratively construct smooth maps $\Ti\mu_\ell^j: W_\ell \to \Hat\phi_{jJ}(E_j)$ on the open sets 
\begin{equation}\label{eq:W}
W_\ell \,:=\; {\textstyle \bigcup _{I\subsetneq J,|I|\le \ell}}\, B^J_{r_\ell}(N^{k+\frac 12}_{JI}) 
\;=\;  B^J_{r_\ell}\bigl( {\textstyle \bigcup _{I\subsetneq J,|I|\le \ell}}\,N^{k+\frac 12}_{JI} \bigr) \;\subset\; U_J 
\end{equation}
with the radii 
$r_\ell:= \eta_k - \frac {\ell+1} {k+1} ( \eta_k-\eta_{k+\frac 12})$,
that satisfy the extended compatibility, admissibility, and smallness conditions
\vspace{.07in}
\begin{enumerate}
\item[(E:i)]
$\Ti\mu^j_\ell |_{N^{k+\frac 12}_{JI}} = \mu_J^j|_{N^{k+\frac 12}_{JI}}$ for all $I\subsetneq J$ with $|I|\leq \ell$ and $j\in I$;
\vspace{.07in}
\item[(E:ii)]
$\Ti\mu^j_\ell |_{B^J_{r_\ell}(N^{k+\frac 12}_{JI})} = 0$
for all $I\subsetneq J$ with $|I|\leq \ell$ and $j\notin I$;
\vspace{.07in}
\item[(E:iii)] 
$\bigl\|\Ti\mu^j_\ell \bigr\| \leq \|\mu^j_J\|$.
\end{enumerate}
\vspace{.07in}
Note here that the radii form a nested sequence $\eta_k=r_{-1} > r_0>r_1 \ldots > r_k = \eta_{k+\frac 12}$ and that when $\ell=k$ the function $\Ti\mu^j_k$ will satisfy  (E:i),(E:ii) for all $I\subsetneq J$,  and is defined on 
$W_k=B^J_{\eta_{k+\frac 12}}(N^{k+\frac 12}_J)\sqsupset N^{k+\frac 12}_J$.
So, after this iteration, we can define $\Ti\nu_J:= \beta {\textstyle\sum_{j\in J}} \, \Ti\mu^j_k$, where $\beta:U_J \to [0,1]$ is a smooth cutoff function with $\beta|_{N^{k+\frac 12}_J}\equiv 1$ and $\supp\beta\subset 
B^J_{\eta_{k+\frac 12}}(N^{k+\frac 12}_J)$, so that $\Ti\nu_J$ extends trivially to $U_J\less W_k$.
This has the required bound by (E:iii), satisfies \eqref{value} since $\Ti\nu_J^j |_{B^J_{\eta_{k+\frac 12}}(N^{k+\frac 12}_{JI})}\equiv 0$ for all $j\notin I$ by (E:ii). Finally, it has the required values on $N^{k+\frac 12}_J = \bigcup_{I\subsetneq J}N^{k+\frac 12}_{JI}$ since for each $I\subsetneq J$ the conditions (E:i), (E:ii) on $N^{k+\frac 12}_{JI}$ together with the fact $\mu_J(N_{JI}^{k+\frac 12}) \subset \Hat\phi_{IJ}(E_I)$ guarantee 
$$
\Ti \nu_J|_{N_{JI}^{k+\frac12}} \;=\; {\textstyle\sum_{j\in J}} \, \Ti\mu_k^j|_{N_{JI}^{k+\frac12}}
\;=\; {\textstyle\sum_{j\in I}} \, \mu_k^j|_{N_{JI}^{k+\frac12}} \;=\; \mu_J|_{N_{JI}^{k+\frac12}} .
$$ 
So it remains to perform the iteration over $\ell$, in which we now drop $j$ from the notation.
For $\ell=0$ the conditions are vacuous since $W_0=\emptyset$.
Now suppose that the construction is given on $W_\ell$. 
Then we cover $W_{\ell+1}$ by the open sets 
$$
B_L': = W_{\ell+1}\cap B^J_{r_{\ell-1}}(N^{k+\frac 12}_{JL})
\qquad
\text{for}\; L\subsetneq J, \; |L|=\ell+1,
$$ 
whose closures are pairwise disjoint by \eqref{Nsep} with $r_{\ell-1}<\delta$, and an open set 
$C_{\ell+1}\subset U_J$
covering the complement,
$$
C_{\ell+1} \,:=\; W_{\ell+1} \;\less\; {\textstyle \bigcup _{|L| = \ell+1}\, \ov{B^J_{r_{\ell}}(N^{k+\frac 12}_{JL})}} \;\sqsubset\; 
W_\ell \;\less\; {\textstyle \bigcup _{|L| = \ell+1}\, \ov{B^J_{r_{\ell+1}}(N^{k+\frac 12}_{JL})}} \;=:\, C_\ell ,
$$
which has a useful precompact inclusion into $C_\ell$,
as defined above,
by $r_{\ell+1}<r_\ell$.
This decomposition is chosen so that each $B^J_{r_{\ell+1}}(N^{k+\frac 12}_{JL})$ for $|L|=\ell+1$ (on which the conditions (E:i),(E:ii) for $I=L$ are nontrivial) has disjoint closure from $\ov{C_{\ell+1}}$ (a compact subset of the domain of $\Ti\mu_\ell$). 
Now pick a precompactly nested open set $C_{\ell+1} \sqsubset C' \sqsubset C_\ell$, 
in particular with $\ov{C'}\cap \ov{B^J_{r_{\ell+1}}(N^{k+\frac 12}_{JL})} = \emptyset$ for all $|L|=\ell +1$.
Then we will obtain a smooth map $\Ti\mu_{\ell+1}: W_{\ell+1} \to \Hat\phi_{jJ}(E_j)$ by setting $\Ti\mu_{\ell+1}|_{C_{\ell+1}} := \Ti \mu_\ell|_{C_{\ell+1}}$  and separately constructing smooth maps $\Ti\mu_{\ell+1}: B'_L \to \Hat\phi_{jJ}(E_j)$ for each $|L|=\ell+1$ such that $\Ti\mu_{\ell+1}=\Ti \mu_\ell$ on $B'_L \cap C'$.
Indeed, this ensures equality of all derivatives on the intersection of the closures $\ov{B_L'} \cap \ov{C_{\ell+1}}$, since this set is contained in $\ov{B_L'} \cap C'$, which is a subset of $\ov{B_L' \cap C'}$ because $C'$ is open, 
and by construction we will have $\Ti\mu_{\ell+1}=\Ti \mu_\ell$ with all derivatives on $\ov{B_L' \cap C'}$. So it remains to construct the extension $\Ti\mu_{\ell+1}|_{B'_L}$ for a fixed $L\subsetneq J$.
For that purpose note that the subset on which this is prescribed as $\Ti\mu_\ell$, can be simplified by the separation property \eqref{Nsep}, 
\begin{equation} \label{simply}
B'_L\cap C' \;\subset\; 
\Bigl( B^J_{r_{\ell-1}}(N^{k+\frac 12}_{JL}) \less \ov{B^J_{r_{\ell+1}}(N^{k+\frac 12}_{JL})} \; \Bigr) \;\cap\; {\textstyle \bigcup _{I\subsetneq L} }B^J_{r_\ell}(N^{k+\frac 12}_{JI}) \;\subset\; W_\ell.
\end{equation}
To ensure (E:i) and (E:ii) for $|I|\leq \ell+1$ first note that $\Ti\mu_{\ell+1}|_{C_{\ell+1}}$ inherits these properties from $\Ti\mu_\ell$ because $C_{\ell+1}$ is disjoint from $B^J_{r_{\ell+1}}(N^{k+\frac 12}_{JI})$ for all $|I|=\ell+1$.
It remains to fix $L\subset J$ with $|L|=\ell+1$ and construct the map $\Ti\mu_{\ell+1}: B'_L \to \Hat\phi_{jJ}(E_j)$ as extension of $\Ti\mu_\ell|_{B'_L\cap C'}$ so that it satisfies properties (E:i)--(E:iii) for all $|I|\le\ell+1$. 
 
In case $j\notin L$ we have $\Ti\mu_\ell|_{B'_L\cap C'}=0$ by iteration hypothesis (E:ii) for each $I\subsetneq J$. 
So we obtain a smooth extension by $\Ti\mu_{\ell+1}:=0$, which satisfies (E:ii) and (E:iii), whereas (E:i) is not relevant.

In case $j\in L$ the conditions (E:i),(E:ii) only require consideration of $I\subsetneq L$ since otherwise $B'_L \cap B^J_{r_\ell}(N^{k+\frac 12}_{JI})=\emptyset$ by \eqref{Nsep}. So we need to construct a bounded smooth map $\Ti\mu_{\ell+1} : B'_L=W_{\ell+1} \cap B^J_{r_{\ell-1}}(N^{k+\frac 12}_{JL}) \to \Hat\phi_{jJ}(E_j)$ that satisfies 
\begin{itemize}
\item[(i)]
$\Ti\mu_{\ell+1}|_{N^{k+\frac 12}_{JL}} = \mu_J^j|_{N^{k+\frac 12}_{JL}}$;
\vspace{.07in}
\item[(i$'$)]
$\Ti\mu_{\ell+1}|_{N^{k+\frac 12}_{JI}} = \mu_J^j|_{N^{k+\frac 12}_{JI}}$ for all $I\subsetneq L$ with $j\in I$;
\vspace{.07in}
\item[(ii)]
$\Ti\mu_{\ell+1}|_{B^J_{r_{\ell+1}}(N^{k+\frac 12}_{JI})} = 0$ for all $I\subsetneq L$ with $j\notin I$;
\vspace{.07in}
\item[(iii)] 
$\bigl\| \Ti\mu_{\ell+1}\bigr\| \leq \|\mu_J^j\|$;
\vspace{.07in}
\item[(iv)] 
$\Ti\mu_{\ell+1}|_{B'_L\cap C'} =\Ti\mu_{\ell}|_{B'_L\cap C'}$.
\end{itemize}
\vspace{.07in}
Because every open cover of $B'_L$ has a locally finite subcovering, such extensions can be patched together by partitions of unity. Hence it suffices, for the given $j\in L\subsetneq J$, to construct smooth maps $\Ti\mu_z: B^J_{r_z}(z)\to \Hat\Phi_{jJ}(E_j)$ on some balls of positive radius $r_z>0$ around each fixed $z\in B'_L$, that satisfy the above requirements.

\NI $\bullet$ 
For $z\in W_\ell \less \ov{N^{k+\frac 12}_{JL}}$ we find $r_z>0$ such that $B^J_{r_z}(z)\subset W_\ell \less \ov{N^{k+\frac 12}_{JL}}$ lies in the domain of $\Ti\mu_\ell$ and the complement of $N^{k+\frac 12}_{JL}$, so that $\Ti\mu_z:=\Ti\mu_\ell|_{B^J_{r_z}(z)}$ is well defined and satisfies all conditions with $\|\Ti\mu_z\|\leq \|\Ti\mu_\ell\|$.

\NI $\bullet$ 
For $z\in B'_L\less \bigl( W_\ell \cup \ov{N^{k+\frac 12}_{JL}}\bigr)$, we claim that there is $r_z>0$ such that $B^J_{r_z}(z)$ is disjoint from the closed subsets $\bigcup_{I\subset L} \ov{N^{k+\frac 12}_{JI}}$ and $\ov{C'}\subset C_\ell\subset W_\ell$. 
This holds because $\bigcup_{I\subsetneq L} \ov{N^{k+\frac 12}_{JI}}\subset W_\ell$ by \eqref{eq:W}.
 Then $\Ti\mu_z:=0$ satisfies all conditions since its domain is in the complement of the domains on which (i), (i$'$), and (iv) are relevant.

\NI $\bullet$ 
Finally, for $z\in \ov{N^{k+\frac 12}_{JL}}$ recall that $\ov{N^{k+\frac 12}_{JL}} \sqsubset N^k_{JL}$ is a compact subset of the smooth submanifold $N^k_{JL}=V^k_J\cap \phi_{LJ}(V^k_L\cap U_{LJ}) \subset A_J$. So we can choose $r_z>0$ such that $B^J_{r_z}(z)$ lies in a submanifold chart for $N^k_{JL}$. Then we define $\Ti\mu_z$ by extending $\mu_J^j|_{B^J_{r_z}(z)\cap N^k_{JL}}$ to be constant in the normal directions. This guarantees (i) and $\|\Ti\mu_z\|\leq \|\mu_J^j\|$, and we will choose $r_z$ sufficiently small to satisfy the further conditions.
First, $N^{k}_{JL}$ is disjoint from $C_{\ell}\sqsupset C'$, so we can ensure that $B^J_{r_z}(z)$ lies in the complement of $C'$, and hence condition (iv) does not apply.
To address (i$'$) and (ii) recall that for every $I\subsetneq L$ the strong cocycle condition of Lemma~\ref{le:tame0} implies that $N^{k+\frac 12}_{JI}\subset\im\phi_{IJ}  =\phi_{LJ}(U_{LJ}\cap\im\phi_{IL})$ is a submanifold of $\im\phi_{LJ}$, and by assumption $z$ lies in the open subset $N^k_{JL}\subset\im\phi_{LJ}$.

In case $j\in I$ and $z\in N^{k+\frac 12}_{JI}\cap \ov{N^{k+\frac 12}_{JL}}$, we can thus choose $r_z$ sufficiently small to ensure that $B^J_{r_z}(z)\cap N^{k+\frac 12}_{JI}$ is contained in the open neighbourhood $N^k_{JL}\subset\im\phi_{LJ}$ of~$z$.
Then $\Ti\mu_z$ satisfies (i$'$) by $\Ti\mu_z=\mu^j_J$ on $B^J_{r_z}(z)\cap N^{k+\frac 12}_{JI}$.

In case $j\notin I$ condition (ii) requires $\Ti\mu_z$ to vanish on $B^J_{r_z}(z)\cap B^J_{r_{\ell+1}}(N^{k+\frac 12}_{JI})$. Here we have $r_{\ell+1}\leq r_1 <\eta_k$, so if $z\notin B^J_{\eta_k}(N^{k+\frac 12}_{JI})$, then we can make this intersection empty by choice of $r_z$. It remains to consider the case 
$z\in B^J_{\eta_k}(N^{k+\frac 12}_{JI})\cap \ov{N^{k+\frac 12}_{JL}}$, where $I\subset L\subset J$ as above.
We pick $x_J\in N^{k+\frac 12}_{JI}$ with $d_J(z,x_J)\leq \eta_k$, then we have $x_J=\phi_{IJ}(x_I)$ for some $x_I\in V^{k+\frac 12}_I\cap U_{IJ}$. By tameness we also have $x_I\in U_{IL}$, and compatibility of the metrics then implies $d(z_L,x_L)=d(z,x_J)\leq \eta_k$ for $x_L:=\phi_{IL}(x_I)$ and $z_L:=\phi_{LJ}^{-1}(z)\in \ov{V^{k+\frac 12}_L}$.
This shows that $x_L$ lies in both $\phi_{IL}(V^{k+\frac 12}_I\cap U_{IJ})$ and $B_{\eta_k}( \ov{V^{k+\frac 12}_L} )$, where the latter is a subset of $V_L^k$ by \eqref{eq:fantastic}, and hence we deduce $x_L\in N^k_{LI}$.
From that we obtain $\nu^j_L|_{B^L_{\eta_k}(x_L)} = 0$ by the induction hypothesis d),
i.e.\ $\nu_L( B^L_{\eta_k}(N^{k}_{LI})\bigr) \;\subset\; \Hat\phi_{IL}(E_I)$.
This implies that the function $\mu^j_J$ of \eqref{eq:nuJ'} vanishes on 
$$
\phi_{LJ}(B^L_{\eta_k}(x_L)\cap U_{LJ})= B^J_{\eta_k}(x_J) \cap \phi_{LJ}(U_{LJ})
$$ 
Since $d_J(z,x_J) \leq r_{\ell+1} < \eta_k$ this set contains $z$, and thus $B^J_{r_z}(z)\cap\phi_{LJ}
(U_{LJ})$ for $r_z>0$ sufficiently small.
With that we have $\mu_J^j|_{B^J_{r_z}(z)\cap N^k_{JL}}=0$ and hence $\Ti\mu_z = 0$, so that (ii) is satisfied.
This completes the construction of $\Ti\mu_z$ in this last case, and hence of $\Ti\mu_{\ell+1}$, and thus by iteration
finishes the construction of the extension $\Ti\nu_J$.  
\MS

\NI {\bf Zero set condition:}
For the extended perturbation constructed above, we have $\bigl\|\Ti \nu_J\bigr\| \leq  \|\mu_J\| \le \sup_{I\subsetneq J} \sup_{x\in V^k_I} \|\nu_I(x)\|< \si$ by induction hypothesis e). We first consider the part of  the perturbed zero set near the core, and then look at the \lq\lq new part".  By \eqref{value}, 
the zero set near the core
$(s_J + \Ti\nu_J)^{-1}(0)\cap  B^J_{\eta_{k+\frac 12}}\bigl(N^{k+\frac34}_{JI}\bigr)$ 
consists of points with $s_J(x) = -\Ti\nu_J(x)\in \Hat\phi_{IJ}(E_I)$,
so must lie within $s_J^{-1}\bigl(\Hat\phi_{IJ}(E_I)\bigr) = \phi_{IJ}(U_{IJ})$.  
Hence \eqref{eq:useful} implies for all $I\subsetneq J$  the inclusion
\begin{equation} \label{eq:zeroset} 
(s_J + \Ti\nu_J)^{-1}(0)\;\cap\;  B^J_{\eta_{k+\frac 12}}\bigl(N^{k+\frac34}_{JI}\bigr)  
\;\subset\; N^{k+\frac12}_{JI} .
\end{equation}
Thus the inductive hypothesis d) together with the compatibility condition  $\Ti\nu_J =\mu_J$ on $N^{k+\frac12}_{JI}\subset \phi_{IJ}(V^k_I)$ from \eqref{tinu}, with $\mu_J$ given by \eqref{eq:nuJ'}, imply that
$s_J + \Ti\nu_J\ne 0$ on $N^{k+\frac12}_{JI}\less \pi_\Kk^{-1}(\pi_\Kk(\Cc))$. Therefore
$$
(s_J + \Ti\nu_J)^{-1}(0)\;\cap\;  B^J_{\eta_{k+\frac 12}}\bigl(N^{k+\frac34}_{JI}\bigr)  
\subset \pi_\Kk^{-1}(\pi_\Kk(\Cc)).
$$
Next, by Definition~\ref{def:admiss} we have
$$
\si < \si(\de,\Vv,\Cc) \le \| s_J(x) \|  \qquad\forall x\in \ov{V^{k+1}_J} \;\less\; \Bigl( \Ti C_J \cup {\textstyle \bigcup_{I\subsetneq J}} B^J_{\eta_{k+\frac 12}}\bigl(N^{k+\frac34}_{JI}\bigr) \Bigr) .
$$
Thus if $x$ is in the complement in $\ov{V^{k+1}_J}$ of the neighbourhoods $B^J_{\eta_{k+\frac 12}}\bigl(N^{k+\frac34}_{JI}\bigr)$ which cover the core, then either $x\in \Ti C_J$ or $\|s_J(x)\|\geq \si_{J,\eta_{k+1}} > \|\Ti\nu_J(x)\|$. 
In particular, we obtain the inclusion
\begin{equation}\label{eq:include}
\bigl(s_J|_{\ov{V^{k+1}_J}} + \Ti\nu_J\bigr)^{-1}(0) \;\less\; {\textstyle\bigcup_{I\subsetneq J} } B^J_{\eta_{k+\frac 12}}\bigl(N^{k+\frac34}_{JI}\bigr)   \;\subset\; \Ti C_J.
\end{equation}
From this we can deduce a slightly stronger version of a) at level $k+1$, namely 
$$
(s_J|_{\ov{V^{k+1}_J}}+\Ti\nu_J)^{-1}(0)  \; \subset\; \pi_\Kk^{-1}(\pi_\Kk(\Cc)) \qquad \forall \; |J|\le k+1 .
$$
Indeed, the zero set of $(s+\Ti\nu_J)|_{\ov{V^{k+1}_J}}$ consists of an ``old part'' given by \eqref{eq:zeroset}, which lies in the enlarged core $N^{k+\frac 12}_J$, where by the above arguments we have $s_J + \Ti\nu_J\ne 0$ on $N^{k+\frac12}_{JI}\less \pi_\Kk^{-1}(\pi_\Kk(\Cc))$. The ``new part'' given by \eqref{eq:include} is in fact contained in the open part $\Ti C_J\subset U_J$ of $\pi_\Kk^{-1}(\pi_\Kk(\Cc))$.

\MS
\NI {\bf Transversality:}
Since the perturbation $\Ti\nu_J$ was constructed to be strongly admissible and hence admissible, the induction hypothesis b) together with Lemma~\ref{le:transv} and \eqref{tinu} imply that the transversality condition is already satisfied on the enlarged core, $(s_J + \Ti\nu_J)|_{N^{k+\frac12}_J} \pitchfork 0$. 
In addition, \eqref{eq:zeroset} also implies that the perturbed section $s_J+\Ti\nu_J$ has no zeros on
$B^J_{\eta_{k+\frac 12}}\bigl(N^{k+\frac34}_{JI}\bigr) \less N^{k+\frac 12}_{JI}$,
so that we have transversality
$$
(s_J + \Ti\nu_J)|_{B^J_{\eta_{k+\frac 12}}(N^{k+\frac34}_J)} \; \pitchfork \; 0 
$$
on a neighbourhood $B:= B^J_{\eta_{k+\frac 12}}(N^{k+\frac34}_J) = \bigcup_{I\subsetneq J} B^J_{\eta_{k+\frac 12}}(N^{k+\frac34}_{JI})$ of the core $N:=N_J^{k+1}=\bigcup_{I\subsetneq J} N^{k+1}_{JI}$, on which compatibility c) requires $\nu_J|_N=\Ti\nu_J|_N$.
In fact, 
$B$ also precompactly contains the neighbourhood $B':= B^J_{\eta_{k+1}}(N^{k+1}_J)$ of $N$, so that strong admissibility c) can be satisfied by requiring $\nu_J|_{B'}=\Ti\nu_J|_{B'}$.

To sum up, the smooth map $\Ti\nu_J : V^{k+1}_J \to E_J$ fully satisfies the compatibility a), strong admissibility c), and strengthened zero set condition d). Moreover, $s_J+\Ti\nu_J$ extends to a smooth map on the compact closure $\ov{V^{k+1}_J}\subset U_J$, where it satisfies transversality $(s_J+\Ti\nu_J)|_B\pitchfork 0$ on the open set $B \subset \ov{V^{k+1}_J}$ and the zero set condition from \eqref{eq:include}, 
$$
(s_J+\Ti\nu_J)^{-1}(0) \cap (\ov{V^{k+1}_J}\less B)\; \subset\; O: = \ov{V^{k+1}_J} \cap \Ti C_J .
$$ 
The latter can be phrased as $\| s_J+\Ti\nu_J \| > 0$ on $( \ov{V^{k+1}_J}\less  B ) \less  O$, which is compact since $O$ is relatively open in $\ov{V^{k+1}_J}$.
Since $z \mapsto \| s_J(z) +\Ti\nu_J(z) \|$ is continuous,  it remains nonvanishing on $W\less O$ for some relatively open neighbourhood $W\subset \ov{V^{k+1}_J}$ of $\ov{V^{k+1}_J}\less B$. This extends the zero set condition to $(s_J+\Ti\nu_J)^{-1}(0) \cap W \subset O$.
We can moreover choose $W$ disjoint from the neighbourhood of the core $B' \sqsubset B$.
Now  we wish to find
a smooth perturbation $\nu_\pitchfork:\ov{V^{k+1}_J} \to E_J$ 
supported in $W$ that satisfies the following:
\begin{enumerate}
\item[(T:i)] it provides transversality $(s_J+\Ti\nu_J+\nu_\pitchfork)|_W \pitchfork 0$;
\item[(T:ii)] the perturbed zero set satisfies the inclusion  $(s_J+\Ti\nu_J+\nu_\pitchfork)^{-1}(0) \cap W \subset O$;
\item[(T:iii)] the perturbation is small:  $\|\nu_\pitchfork\|< \si - \|\Ti\nu_J\|$.
\end{enumerate}
To see that this exists, note that for $\nu_\pitchfork=0$ transversality holds outside the compact subset $\ov{V^{k+1}_J}\less B$ of $W$.  Hence by the Transversality Extension theorem in \cite[Chapter~2.3]{GuillP} we can 
fix a nested open precompact subset $\ov{V^{k+1}_J}\less B \subset P \sqsubset W$ and achieve transversality everywhere on $W$ by adding an arbitrarily small perturbation supported in $P$. This immediately provides (T:i).
Moreover, since $\|s_J +\Ti\nu_J\|$ has a positive maximum on the compact set $P\less O$, we can choose $\nu_\pitchfork$ sufficiently small to satisfy (T:ii) and (T:iii).
Setting 
$$
\nu_J:=\Ti\nu_J + \nu_\pitchfork \,:\; V^{k+1}_J \to E_J
$$
then finishes the construction since the choice of $\nu_\pitchfork$ ensures the zero set inclusion a) and transversality b) on $W$; the previous constructions for $\nu_J|_{V^{k+1}_J\less W}=\Ti\nu_J|_{V^{k+1}_J\less W}$ ensure a), b), and d) on $V^{k+1}_J\less W \supset B'$, and compatibility c) on the core $N \subset B'\subset V^{k+1}_J\less W$; and we achieve smallness e) by the triangle inequality 
$$
 \|\Ti\nu_J+\nu_\pitchfork\| \;\leq\;  \|\Ti\nu_J\| +  \si - \|\Ti\nu_J\| \;\le\; \|\mu_J\| \; \le\; \max_{I\subsetneq J} \|\nu_I\| \;<\; \si .
$$
This completes the iterative step and hence completes construction of the required $(\Vv,\Cc,\de,\si)$-adapted section.  The last claim follows from Proposition~\ref{prop:zeroS0}.
\end{proof}

In order to prove uniqueness up to cobordism of the VMC, we moreover need to construct transverse cobordism perturbations with prescribed boundary values as in Definition~\ref{def:csect}.
We will perform this construction by an iteration as in Proposition~\ref{prop:ext}, with adjusted domains $V^k_J$ obtained by replacing $\de$ with $\frac 12 \de$. 
This is necessary since as before the construction of $\nu_J$ will proceed by extending the given perturbations from previous steps, $\mu_J$, and now also the given boundary values $\nu^\al_J$, and then restricting to a precompact subset.
However, the $(\Vv,\Cc,\de,\si)$-adapted boundary values $\nu^\al_J$ on $\p^\al V_J$ only extend to admissible, precompact, transverse perturbations in a collar of $V^{|J|}_J$.
Hence the construction of $\nu_J$ by precompact restriction does not allow to define it on the whole of this collar. 
Instead, we do achieve this construction
by restriction to $V^{|J|+1}_J\sqsubset V^{|J|}_J$, which by \eqref{eq:VIk} is the analog of $V^{|J|}_J$ when $\de$ is replaced by $\frac 12 \de $.
This means that, firstly, we have to adjust the smallness condition for the iterative construction of perturbations by introducing a variation of the constant $\si(\de,\Vv,\Cc)$ of Definition~\ref{a-e}. 
Secondly, we need a further smallness condition on adapted perturbations if we wish to extend these to a Kuranishi cobordism. Fortunately, the latter construction will only be used on product Kuranishi cobordisms, which leads to the following definitions.

\begin{defn}  \label{a-e rel}
\begin{enumerate}
\item
Let $(\Kk,d)$ be a metric tame Kuranishi cobordism with nested cobordism reduction $\Cc\sqsubset\Vv$, and let $0<\de<\min\{\eps,\de_\Vv\}$, where $\eps$ is the collar width of $(\Kk,d)$ and the reductions $\Cc,\Vv$. Then we set
\begin{align*}
 \si' (\de,\Vv,\Cc) &\,:=\; \min_{J\in\Ii_\Kk} \inf \Bigl\{ \; \bigl\| s_J(x) \bigr\| \;\Big| \;
x\in \ov{V^{|J|+1}_J} \;\less\; \Bigl( \Ti C_J \cup {\textstyle \bigcup_{I\subsetneq J}} B^J_{\eta_{|J|+\frac 12}}\bigl(N^{|J|+\frac34}_{JI}\bigr) \Bigr) \Bigr\}  , \\
\qquad\si_{\rm rel}(\de,\Vv,\Cc) &\,:=\; \min\bigl\{ \si(\de,\p^0\Vv,\p^0\Cc), \,\si(\de,\p^1\Vv,\p^1\Cc), \,\si'(\de,\Vv,\Cc) \bigr\} .
\end{align*}
\item
Given a metric Kuranishi atlas $(\Kk,d)$, we say that a perturbation $\nu$ is {\bf strongly adapted} if it is a $(\Vv,\Cc,\de,\si)$-adapted perturbation $\nu$ of $s_\Kk|_\Vv$ for some choice of nested reduction $\Cc\sqsubset\Vv$ and constants $0<\de<\de_\Vv$ and 
$$
\qquad 0\;<\;\si\;\le\;\si_{\rm rel}(\de,\Vv\times[0,1],\Cc\times[0,1]) \;=\; \min\bigl\{
\si(\de,\Vv,\Cc) , \si'(\de,\Vv\times[0,1],\Cc\times[0,1]) \bigr\} .
$$

\end{enumerate}
\end{defn}

Recalling the definition of $\si(\de,\Vv,\Cc)$, and the product structure of all sets and maps involved in the definition of $\si'(\de,\Vv\times[0,1],\Cc\times[0,1])$, we may rewrite the condition on $\si>0$ in the definition of strong adaptivity as
$$
 \si \,<\;   \bigl\| s_J(x) \bigr\| \qquad\forall\; 
 x\in \ov{V^{k}_J} \;\less\; \Bigl( \Ti C_J \cup {\textstyle \bigcup_{I\subsetneq J}} B^J_{\eta_{k-\frac 12}}\bigl(N^{k-\frac14}_{JI}\bigr) \Bigr) \Bigr\} ,\;
 J\in\Ii_\Kk, \; k\in\{|J|,|J|+1\} .
$$
Although the construction of transverse cobordism perturbations with fixed boundary values in part (ii) of the following Proposition will be used only on product cobordisms, we state it here in generality, since we use it to construct transverse cobordism perturbations without fixed boundary values in part (i).

\begin{prop}\label{prop:ext2}
Let $(\Kk,d)$ be a metric tame Kuranishi cobordism with nested cobordism reduction $\Cc\sqsubset\Vv$, 
let $0<\de<\min\{\eps,\de_\Vv\}$, where $\eps$ is the collar width of $(\Kk,d)$
and the reductions $\Cc,\Vv$. 
Then we have $\si_{\rm rel}(\de,\Vv,\Cc)>0$ and the following holds.
\begin{enumerate}
\item
Given any 
$0<\si\le \si_{\rm rel}(\de,\Vv,\Cc)$, there exists an admissible, precompact, transverse cobordism perturbation $\nu$ of $s_\Kk|_\Vv$ with $\pi_\Kk\bigl((s+\nu)^{-1}(0)\bigr)\subset \pi_\Kk(\Cc)$, whose restrictions $\nu|_{\p^\al\Vv}$  for $\al=0,1$ are $(\p^\al\Vv,\p^\al\Cc,\de,\si)$-adapted perturbations of $s_{\p^\al\Kk}|_{\p^\al\Vv}$.
\item
Given any perturbations $\nu^\al$ of $s_{\p^\al\Kk}|_{\p^\al\Vv}$ for $\al=0,1$ that are $(\p^\al\Vv,\p^\al\Cc,\de,\si)$-adapted
with $\si\le \si_{\rm rel}(\de,\Vv,\Cc)$, the perturbation $\nu$ of $s_\Kk|_\Vv$ in (i) can be constructed to have boundary values $\nu|_{\p^\al\Vv}=\nu^\al$ for $\al=0,1$.
\item
In the case of a product cobordism $\Kk\times[0,1]$ with product metric and product reductions $\Cc\times[0,1]\sqsubset\Vv\times[0,1]$, both (i) and (ii) hold without requiring $\de$ to be bounded in terms of the collar width.
\end{enumerate}
\end{prop}
\begin{proof}
The positivity $\si_{\rm rel}(\de,\Vv,\Cc)>0$ follows from $\si(\de,\p^\al\Vv,\p^\al\Cc)>0$ by Lemma~\ref{le:admin}~(i), and $\si'>0$ by the arguments of Lemma~\ref{le:admin}~(i) applied to the shifted domains.

Next, we reduce (i) for given $0<\si \le \si_{\rm rel}(\de,\Vv,\Cc)$ to (ii). For that purpose recall that $\de<\de_{\p^\al\Vv}$ by Lemma~\ref{le:admin}~(iii) and $\si \le \si(\de,\p^\al\Vv,\p^\al\Cc)$ by definition of $\si_{\rm rel}(\de,\Vv,\Cc)$.
Hence Proposition~\ref{prop:ext} provides $(\p^\al\Vv,\p^\al\Cc,\de,\si)$-adapted perturbations $\nu^\al$ of $s_{\p^\al\Kk}|_{\p^\al\Vv}$ for $\al=0,1$. 
Now (ii) provides a cobordism perturbation $\nu$ with the given restrictions $\nu|_{\p^\al\Vv}=\nu^\al$,
which are $(\p^\al\Vv,\p^\al\Cc,\de,\si)$-adapted by construction.
So (i) follows from~(ii).

To prove (ii) recall that, by assumption, the given perturbations $\nu^\al$ of $s_{\p^\al\Kk}|_{\p^\al\Vv}$ for $\al=0,1$ extend to $\nu^\al_I : (\p^\al V_I)^{|I|} \to E_I$ for all $I\in\Ii_{\p^\al\Kk}$ which satisfy conditions a)-e) of Definition~\ref{a-e} with 
the given constant $\si$.
Here by Lemma~\ref{le:admin}~(iv)
the domains of $\nu^\al_I$ are $(\p^\al V_I)^{|I|} = \p^\al V_I^{|I|}$, and these are the boundaries of the reductions $V_I^k$ which have collars 
$$
V_I^k \cap \io^\al_I\bigl(\p^\al U_I \times A^\al_{\eps-2^{-k}\de}\bigr) 
 = \io^\al_I\bigl( \p^\al V_I^k  \times A^\al_{\eps-2^{-k}\de} \bigr)  ,
$$
where the requirement $2^{-k}\de<\eps$ of Lemma~\ref{le:admin} for $k>0$ is ensured by the assumption $\de<\eps$.
In the case of a product cobordism with product reduction this holds for any $\de>0$ with ${\eps-2^{-k}\de}$ replaced by $\eps:=1$.
The same collar form holds for $C_I \sqsubset V_I$, and hence for any set such as $N^k_{JI}$ or $\Ti C_I$ constructed from these.
Now $\de<\eps$ also ensures $2^{-k}\eps \le \eps - 2^{-k}\de$ for $k\geq 1$, so that we may denote the $2^{-k}\eps$-collar of $V_I^k$ by
$$
N^k_{I,\al}  \,:=\; \io^\al_I\bigl(\p^\al V_I^k \times A^\al_{2^{-k}\eps} \bigr) \;\subset\; V^k_I 
$$
and note the precompact inclusion $N^{k'}_{I,\al} \sqsubset N^k_{I,\al}$ for $k'>k$.

We will now construct the required cobordism perturbation $\nu$ by an iteration as in Proposition~\ref{prop:ext} with adjusted domains obtained by replacing $\de$ with $\frac 12 \de$. This is necessary since the given boundary value $\nu^\al_J$ by assumption only extends to a map $\nu^\al_J : V^{|J|}_J \to E_J$, but as before the construction of $\nu_J$ will proceed by restriction to a precompact subset of the domain of an extension $\Ti\nu_J$, where this agrees both with the push forward of previously defined  $(\nu_I)_{I\subsetneq J}$ and with the given boundary perturbations $\nu^\al_J$ in collar neighbourhoods. 
We achieve this by restriction to $V^{|J|+1}_J\sqsubset V^{|J|}_J$.
That is, in the $k$-th step we construct $\nu_J : V^{k+1}_J \to E_J$ for each $|J| = k$ that, together with the $\nu_I|_{V^{k+1}_I}$ for $|I|< k$ obtained by restriction from earlier steps, satisfies the following.

\begin{itemize}
\item[a)]
The perturbation is compatible with coordinate changes and collars, that is
$$
\quad
\nu_J |_{N^{k+1}_{JI}}  \;=\; \Hat\phi_{IJ} \circ \nu_I \circ \phi_{IJ}^{-1} |_{N^{k+1}_{JI}} 
\qquad \text{on}\quad
N^{k+1}_{JI} = V^{k+1}_J \cap \phi_{IJ}(V^{k+1}_I\cap U_{IJ})
$$
for all $I\subsetneq J$, and for each $\al\in\{0,1\}$ with $J\in\Ii_{\p^\al\Kk}$ we have 
$$
\nu_J |_{N^{k+1}_{J,\al}} \;=\; (\io^\al_J)^*\nu^\al_J
\qquad\text{on}\quad N^{k+1}_{J,\al}  = \io^\al_J\bigl(\p^\al V_J^{k+1} \times A^\al_{2^{-k-1}\eps} \bigr),
$$
where we abuse notation by defining $(\io^\al_J)^*\nu^\al_J : \io^\al_J(x,t) \mapsto \nu^\al_J(x)$. $\phantom{\bigg(}$
\item[b)]
The perturbed section is transverse, that is $(s_J|_{{V^{k+1}_J}} + \nu_J) \pitchfork 0$.
\item[c)]
The perturbation is {\it strongly admissible} with radius $\eta_{k+1}= 2^{-k-1}(1-2^{-\frac 14})$,
$$
\qquad
\nu_J( B^J_{\eta_{k+1}}(N^{k+1}_{JI})\bigr) \;\subset\; \Hat\phi_{IJ}(E_I)
\qquad\forall \;I\subsetneq J .
$$
\item[d)]  
The perturbed zero set is contained in $\pi_\Kk^{-1}\bigl(\pi_\Kk(\Cc)\bigr)$; more precisely
$$
(s_J |_{{V^{k+1}_J}}+ \nu_J)^{-1}(0) \;\subset\; {V^{k+1}_J} \cap \pi_\Kk^{-1}\bigl(\pi_\Kk(\Cc)\bigr) .
$$
\item[e)]
The perturbation is small, that is $\sup_{x\in {V^{k+1}_J}} \| \nu_J (x) \| <  \si$. 
\end{itemize}
The final perturbation $\nu=(\nu_I|_{V_I})_{I\in\Ii_\Kk}$ of $s_\Kk|_\Vv$ then has product form on collars of width $2^{-M_\Kk}\eps$ and thus is a cobordism perturbation, whose boundary restrictions are the given $\nu^\al$ by construction.
Moreover, $\nu$ will be admissible by c), transverse by b), and precompact by d) with $\pi_\Kk( (s+\nu)^{-1}(0))\subset\pi_\Kk(\Cc)$. Compactness of $|(s+\nu)^{-1}(0)|$ then follows from Lemma~\ref{le:czeroS0}.

For $k=0$, there are no indices $|J|=0$ to be considered.
Now suppose that $\bigl(\nu_I : V^{|I|+1}_I\to E_I\bigr)_{I\in\Ii_\Kk, |I|< k}$  are constructed such that a)-e) hold. Then for the iteration step it suffices as before to construct $\nu_J$ for a fixed $J\in\Ii_\Kk$ with $|J|=k$. In the following three construction steps we then unify the cases of $J\in\Ii_{\p^\al\Kk}$ for none, one, or both indices $\al$ by interpreting the collars $N^k_{J,\al}$ as empty sets unless $J\in\Ii_{\p^\al\Kk}$.

\MS\NI
{\bf Construction of extension for fixed $|J|=k$:}  
For each $k\geq 1$ we will construct an extension of a restriction of 
$$
\quad\mu_J : N_J^k \cup N^k_{J,0} \cup N^k_{J,1}   \;\longrightarrow\; E_J , \qquad
\mu_J|_{N^{k}_{JI}} := \Hat\phi_{IJ}\circ \nu_I\circ\phi_{IJ}^{-1}, \qquad
\mu_J|_{N^k_{J,\al}} :=  (\io^\al_J)^*\nu^\al_J ,
$$
where $N^k_{J,\al} = \io^\al_J\bigl(\p^\al V_J^k \times A^\al_{2^{-k}\eps} \bigr)$ is a collar of $V^k_J$.
More precisely, we construct a smooth map $\Ti\nu_J : V^k_J \to E_J$ that satisfies
\begin{equation}\label{ctinu}
\Ti\nu_J|_{N_{k+\frac 12}} = \mu_J|_{N_{k+\frac 12}} 
\qquad\text{on}\;\; N_{k+\frac 12} :=  N_J^{k+\frac 12} \cup N^{k+\frac 12}_{J,0} \cup N^{k+\frac 12}_{J,1} ,
\end{equation}
the bound $\|\Ti\nu_J \| \leq \|\mu_J\| < \si$, and the strong admissibility condition 
\begin{equation}\label{cvalue} 
\Ti\nu_J \bigl( B^J_{\eta_{k+\frac 12}}\bigl(N^{k+\frac 12}_{JI}\bigr) \bigr) \;\subset\; \Hat\phi_{IJ}(E_I) 
\qquad \forall\; I\subsetneq J .
\end{equation}

We proceed as in Proposition~\ref{prop:ext} for fixed $j\in J$ by iteratively constructing smooth maps $\Ti\mu^j_\ell: W_\ell \to \Hat\phi_{jJ}(E_j)$ for $\ell=0,\ldots,k-1$ on the adjusted open sets 
\begin{equation}\label{eq:cW}
W_\ell \,:=\; N^{k_\ell}_{J,0} \;\cup\; N^{k_\ell}_{J,1} \;\cup\;
 {\textstyle \bigcup _{I\subsetneq J,|I|\le \ell}}\, B^J_{r_\ell}(N^{k+\frac 12}_{JI}) 
\end{equation}
with $r_\ell:= \eta_k - \frac {\ell+1} {k} ( \eta_k-\eta_{k+\frac 13})$ and $k_\ell:= k + \frac {\ell+1} {3k}$, that satisfy the conditions
\begin{enumerate}
\item[(E:i)]
$\Ti\mu^j_\ell |_{N^{k+\frac 12}_{JI}} = \mu_J^j|_{N^{k+\frac 12}_{JI}}$ for all $I\subsetneq J$ with $|I|\leq \ell$ and $j\in I$;
\vspace{.07in}
\item[(E:ii)]
$\Ti\mu^j_\ell |_{B^J_{r_\ell}(N^{k+\frac 12}_{JI})} = 0$
for all $I\subsetneq J$ with $|I|\leq \ell$ and $j\notin I$;
\vspace{.07in}
\item[(E:iii)] 
$\bigl\|\Ti\mu^j_\ell \bigr\| \leq \|\mu^j_J\|$;
\vspace{.07in}
\item[(E:iv)]
$\Ti\mu^j_\ell = (\io^\al_J)^*\nu^{\al,j}_J$ on  $N^{k_\ell}_{J,\al} = \io^\al_J\bigl(\p^\al V_J^{k_\ell} \times A^\al_{2^{-k_\ell}\eps} \bigr)$ 
for $\al\in\{0,1\}$ with $J\in\Ii_{\p^\al\Kk}$.
\end{enumerate}
These requirements make sense because  $\eta_{k+\frac 12}< r_\ell < \eta_k$ and $B^J_{\eta_{k}}(N^{k+\frac 12}_J) \subset V^k_J$ by \eqref{eq:fantastic}, so that the domain in (E:ii) is included in $V^k_J$ and is larger than that in \eqref{cvalue}.
After this iteration, we then obtain the extension $\Ti\nu_J:= \beta {\textstyle\sum_{j\in J}} \, \Ti\mu^j_{k-1}$ by multiplication with a smooth cutoff function $\beta:V^k_J \to [0,1]$ with $\beta|_{N_{k+\frac 12}}\equiv 1$ and $\supp\beta\subset N^{k+\frac 13}_{J,0} \cup N^{k+\frac 13}_{J,1} \cup B^J_{\eta_{k+\frac 13}}(N^{k+\frac 12}_J)$, where the latter contains the closure of $N_{k+\frac 12}=N^{k+\frac 12}_{J,0} \cup N^{k+\frac 12}_{J,1} \cup N^{k+\frac 12}_J$ in $V^k_J$, so that $\Ti\nu_J$ extends trivially to $V^k_J\less W_{k-1}$.

For the start of iteration at $\ell=0$,
the domain is $W_0= N^{k_0}_{J,0} \;\cup\; N^{k_0}_{J,1}$ with $k_0 = k + \frac 1{3k}$.
Conditions (E:i) and (E:ii) are vacuous since there are no index sets with $|I|\le 0$, and 
we can satisfy (E:iii) and (E:iv), by setting $\Ti\mu^j_0 (\iota^\al(x,t)) := \nu^{\al,j}_J(x)$.
Next, if the construction is given on $W_\ell$, then we cover $W_{\ell+1}$ by the open sets 
$B_L': = W_{\ell+1}\cap B^J_{r_{\ell-1}}(N^{k+\frac12}_{JL})$
for $L\subsetneq J$, $|L|=\ell+1$ and $C_{\ell+1}\subset W_\ell$
given below, and pick an open subset $C'\subset V^k_J$ such that
$$
C_{\ell+1} \,:=\; W_{\ell+1} \;\less\; {\textstyle \bigcup _{|L| = \ell+1}\, \ov{B^J_{r_{\ell}}(N^{k+\frac12}_{JL})}} \;\sqsubset\; C' \;\sqsubset\; W_\ell \;\less\; {\textstyle \bigcup _{|L| = \ell+1}\, \ov{B^J_{r_{\ell+1}}(N^{k+\frac12}_{JL})}} \;=:\, C_\ell .
$$
As before, this guarantees that $C'$ and $B^J_{r_{\ell+1}}(N^{k+\frac 12}_{JL})$ have disjoint closures for all $|L|=\ell +1$. Then we set $\Ti\mu_{\ell+1}|_{C_{\ell+1}} := \Ti \mu_\ell|_{C_{\ell+1}}$, which inherits properties (E:i)--(E:iv) from $\Ti\mu_\ell$ because $C_{\ell+1}$ is still disjoint from $B^J_{r_{\ell+1}}(N^{k+\frac 12}_{JL})$ for any $|I|=\ell+1$, and we have $N^{k_{\ell+1}}_{J,\al}\subset N^{k_\ell}_{J,\al}$.
So it remains to construct $\Ti\mu_{\ell+1}: B'_L \to \Hat\phi_{jJ}(E_j)$ for a fixed $L\subset J$, $|L|=\ell+1$ such that $\Ti\mu_{\ell+1}=\Ti \mu_\ell$ on $B'_L \cap C'$.

In case $j\notin L$ condition (E:iv) prescribes $\Ti\mu^j_\ell = (\io^\al_J)^*\nu^{\al,j}_J$ on the intersection
$$
B'_L \cap N^{k_{\ell+1}}_{J,\al}
\subset 
\io^\al_J\bigl(B^{J,\al}_{r_{\ell-1}}(\p^\al N^{k+\frac12}_{JL}) \times A^\al_{2^{-k_{\ell+1}}\eps} \bigr)\bigr).
$$ 
Because $r_{\ell-1}<\eta_k$, strong admissibility for $\nu^\al_J$ on $B^{J,\al}_{\eta_k}(\p^\al N^k_{JL})$ implies that $(\io^\al_J)^*\nu^{\al,j}_J=0$ on this intersection.
Moreover, $B'_L\cap C'$ again is a subset of $\bigcup _{I\subsetneq L} B^J_{r_\ell}(N^{k+\frac 12}_{JI})$, where we have $\Ti\mu_\ell|_{B'_L\cap C'}=0$ by iteration hypothesis (E:ii) for each $I\subsetneq J$. 
Thus $\Ti\mu_{\ell+1}:=0$ satisfies all extension properties (E:i)--(E:iv) in this case.

In case $j\in L$ we may again patch together extensions by partitions of unity, so that it suffices to construct smooth maps $\Ti\mu_z: B^J_{r_z}(z)\to \Hat\Phi_{jJ}(E_j)$ on balls of positive radius $r_z>0$ around each fixed $z\in B'_L$, that satisfy
\begin{itemize}
\item[(i)]
$\Ti\mu_z = \mu_J^j$ on $B^J_{r_z}(z)\cap N^{k+\frac 12}_{JI}$ for all $I\subset L$ with $j\in I$ (including $I=L$);
\vspace{.07in}
\item[(ii)]
$\Ti\mu_z = 0$ on $B^J_{r_z}(z) \cap B^J_{r_{\ell+1}}(N^{k+\frac 12}_{JI})$ for all $I\subsetneq L$ with $j\notin I$;
\vspace{.07in}
\item[(iii)] 
$\bigl\| \Ti\mu_z\bigr\| \leq \|\mu_J^j\|$;
\vspace{.07in}
\item[(iv)] 
$\Ti\mu_z = (\io^\al_J)^*\nu^{\al,j}_J$ on  $B^J_{r_z}(z) \cap N^{k_{\ell+1}}_{J,\al}$ for $\al\in\{0,1\}$ with $J\in\Ii_{\p^\al\Kk}$;
\vspace{.07in}
\item[(v)] 
$\Ti\mu_z=\Ti\mu_{\ell}$ on $B^J_{r_z}(z)\cap B'_L\cap C'$.
\end{itemize}
\vspace{.07in}

For $z\in V^k_J \less \ov{N^{k_{\ell+1}}_{J,\al}}$, this is accomplished by the same constructions as in Proposition~\ref{prop:ext} by choosing $r_z>0$ such that $B^J_{r_z}(z)\cap N^{k_{\ell+1}}_{J,\al}=\emptyset$.
For $z\in \ov{N^{k_{\ell+1}}_{J,\al}}\subset N^{k_\ell}_{J,\al}$ we choose $r_z>0$ such that $B^J_{r_z}(z)\subset N^{k_{\ell}}_{J,\al}$. Then $\Ti\mu_z := \Ti\mu_\ell|_{B^J_{r_z}(z)}$ satisfies (v) by construction and (i)-(iv) by iteration hypothesis.

\MS

\NI {\bf Zero set condition:}
For the extended perturbation constructed above, we have $\bigl\|\Ti \nu_J\bigr\| \leq \max\{ \max_{I\subsetneq J} \|\nu_I\| , \|\nu^0_J\| , \|\nu^1_J\|  \} < \si $ by induction hypothesis e).
From \eqref{cvalue} and \eqref{eq:useful} we then obtain as in Proposition~\ref{prop:ext}
\begin{equation} \label{eq:czeroset} 
(s_J |_{V^k_J} + \Ti\nu_J)^{-1}(0)\;\cap\;  B^J_{\eta_{k+\frac 12}}\bigl(N^{k+\frac34}_{JI}\bigr)  
\;\subset\; N^{k+\frac12}_{JI} .
\end{equation}
Next, recall that we allowed only $\si>0$ such that
$$
\si \;\leq\; \inf \Bigl\{ \; \bigl\| s_J(x) \bigr\| \;\Big| \;
x\in \ov{V^{|J|+1}_J} \;\less\; \Bigl( \Ti C_J \cup {\textstyle \bigcup_{I\subsetneq J}} B^J_{\eta_{|J|+\frac 12}}\bigl(N^{|J|+\frac34}_{JI}\bigr) \Bigr) \Bigr\}  .
$$
Hence the same arguments as in the proof of Proposition~\ref{prop:ext} provide the inclusion
\begin{equation}\label{eq:cinclude}
\bigl(s_J|_{\ov{V^{k+1}_J}} + \Ti\nu_J\bigr)^{-1}(0) \;\less\; {\textstyle\bigcup_{I\subsetneq J} } B^J_{\eta_{k+\frac 12}}\bigl(N^{k+\frac34}_{JI}\bigr)   \;\subset\; \Ti C_J.
\end{equation}
Together with the induction hypothesis on $\Ti\nu_J =\mu_J=\Hat\phi_{IJ}\circ\nu_I\circ\phi_{IJ}^{-1}$ on $N^{k+\frac12}_{JI}$ this implies the zero set condition $(s_J|_{\ov{V^{k+1}_J}}+\Ti\nu_J)^{-1}(0) \subset\pi_\Kk^{-1}(\pi_\Kk(\Cc))$.

\MS
\NI {\bf Transversality:}
Admissibility together with induction hypothesis b) imply transversality $(s_J + \Ti\nu_J)|_{N^{k+\frac12}_J} \pitchfork 0$ on the enlarged core.
Together transversality of $\nu^\al_J$ and \eqref{eq:czeroset} we obtain transversality on the open set
$$
(s_J + \Ti\nu_J)|_{B} \; \pitchfork \; 0  , \qquad B:= B^J_{\eta_{k+\frac 12}}(N^{k+\frac34}_J) \cup N^{k+\frac 12}_{0,J} \cup N^{k+\frac 12}_{1,J} \;\subset\; V^k_J .
$$
Now $B$ precompactly contains the neighbourhood $B':= B^J_{\eta_{k+1}}(N^{k+1}_J)\cup N^{k+1}_{0,J} \cup N^{k+1}_{1,J} \subset V^k_J$ of the core and collar $N:= N^{k+1}_J \cup N^{k+1}_{0,J} \cup N^{k+1}_{1,J}$,
so that compatibility with the coordinate changes and collars in a) and strong admissibility in d) can be satisfied by requiring $\nu_J|_{B'}=\Ti\nu_J|_{B'}$.
In this abstract setting, we can finish the iterative step word by word as in Proposition~\ref{prop:ext}. This completes the construction of the required perturbation in case (ii) and thus finishes the proof.
\end{proof}

\subsection{Orientations} \label{ss:vorient}   \hspace{1mm}\\ \vspace{-3mm}

This section develops the theory of orientations of Kuranishi atlases.
We use the method of determinant line bundles as in e.g.\ \cite[App.A.2]{MS}.
but encountered compatibility issues of sign conventions in the literature, e.g.\ all editions of \cite{MS}.
We resolve these by using a different set of conventions most closely related to K-theory and thank Thomas Kragh for helpful discussions.
As shown in the recent work of  \cite{Z3}, these conventions are consistent with some important naturality properties, a fact which may prove useful in the future development of Kuranishi atlases.
 
While the relevant bundles and sections could just be described as tuples of bundles and sections over the domains of the Kuranishi charts, related by lifts of the coordinate changes, we take this opportunity to develop a general framework of vector bundles over Kuranishi atlases, which now no longer are assumed to be additive or tame.

\begin{defn} \label{def:bundle}
A {\bf vector bundle} $\La=\bigl(\La_I,\Ti\phi_{IJ}\bigr)_{I,J\in\Ii_\Kk}$ {\bf over a weak Kuranishi atlas} $\Kk$ is a collection $(\La_I \to U_I)_{I\in \Ii_\Kk}$ of vector bundles together with lifts $\bigl(\Tilde \phi_{IJ}: \La_I|_{U_{IJ}}\to \La_J\bigr)_{I\subsetneq J}$ of the coordinate changes $\phi_{IJ}$, that are linear isomorphisms on each fiber and satisfy the weak cocycle condition $\Tilde \phi_{IK} =  \Tilde \phi_{JK}\circ  \Tilde \phi_{IJ}$ on $\phi^{-1}_{IJ}(U_{JK})\cap U_{IK}$ for all triples $I\subset J\subset K$.  

A {\bf section} of a bundle $\La$ over $\Kk$ is a collection of smooth sections $\si=\bigl( \si_I: U_I\to \La_I \bigr)_{I\in\Ii_\Kk}$ that are compatible with the bundle maps $\Ti\phi_{IJ}$.
In particular, for a vector bundle $\La$ with section $\si$ there are commutative diagrams for each $I\subset J$,
\[
\xymatrix{
  \La_I|_{U_{IJ}}  \ar@{->}[d] \ar@{->}[r]^{\;\;\Tilde\phi_{IJ}}   &  \La_J \ar@{->}[d]   \\
U_{IJ}\ar@{->}[r]^{\phi_{IJ}}  & U_J
}
\qquad\qquad\qquad
\xymatrix{
  \La_I|_{U_{IJ}}  \ar@{->}[r]^{\;\;\Tilde\phi_{IJ}}    &  \La_J  \\
U_{IJ}  \ar@{->}[u]^{\si_I}   \ar@{->}[r]^{\phi_{IJ}}  & U_J  \ar@{->}[u]_{\si_J} .
}
\]
\end{defn}

The following notion of a product bundle will be the first example of a bundle over a Kuranishi cobordism. 

\begin{defn} \label{def:prodbun}
If $\La=\bigl(\La_I,\Ti\phi_{IJ}\bigr)_{I,J\in\Ii_\Kk}$ is a bundle over $\Kk$ and $A\subset [0,1]$ is an interval, then the {\bf product bundle} $\La\times A$ over $\Kk\times A$ is the tuple $\bigl(\La_I\times A,\Ti\phi_{IJ}\times \id_A\bigr)_{I,J\in\Ii_\Kk}$. 
Here and in the following we denote by $\La_I\times A\to U_I\times A$ the pullback bundle under the projection $U_I\times A\to U_I$.
\end{defn}

\begin{defn} \label{def:cbundle}
A {\bf vector bundle over a weak Kuranishi cobordism} $\Kk$ 
is a collection $\La=\bigl(\La_I,\Ti\phi_{IJ}\bigr)_{I,J\in\Ii_\Kk}$ of vector bundles and bundle maps as in Definition~\ref{def:bundle}, together with a choice of isomorphism from its restriction to a
collar of the boundary to a product bundle.
More precisely, this requires for $\al=0,1$ the choice of a {\bf restricted vector bundle} $\La|_{\p^\al\Kk}= \bigl( \La^\al_I \to \partial^\al U_I, \Ti\phi^\al_{IJ}\bigr)_{I,J \in \Ii_{\p^\al\Kk}}$ over $\p^\al\Kk$, and, for some $\eps>0$ less than the collar width of $\Kk$, a choice of lifts of the embeddings $\io^\al_I$ for $I\in\Ii_{\p^\al\Kk}$ to bundle isomorphisms $\ti\io^\al_I : \La^\al_I\times A^\al_\eps \to \La_I|_{\im\io^\al_I}$ such that,  with $A: = A^\al_\eps$, the following diagrams commute
\[
\xymatrix{ \La_I^\al\times A \ar@{->}[d]   \ar@{->}[r]^{\ti\io^\al_I}    &   \La_I|_{\im\io^\al_I} \ar@{->}[d]   \\
\partial^\al U_I\times A  \ar@{->}[r]^{\io^\al_I}   & \im\io^\al_I \subset U_I
}
\qquad\qquad
\xymatrix{
  \La_I^\al|_{\p^\al U_{IJ}} \times A
  \ar@{->}[r]^{\ti\io^{\al}_I} \ar@{->}[d]_{\Ti\phi^\al_{IJ}\times\id_A}    & 
 \La_I |_{\io^\al_I(\p^\al U_{IJ} \times A)}
  \ar@{->}[d]^{\Ti\phi_{IJ}}  \\
  \La_J^\al\times A \ar@{->}[r]^{\ti\io^\al_{J}}  &  \La_J|_{\im\io^\al_J}  
}
\]

A {\bf section} of a vector bundle $\La$ over a Kuranishi cobordism as above is a compatible collection $\bigl(\si_I:U_I\to \La_I\bigr)_{I\in\Ii_\Kk}$ of sections as in Definition~\ref{def:bundle} that in addition have product form in the collar. 
That is we require that for each $\al=0,1$ there is a {\bf restricted section} $\si|_{\p^\al\Kk}= ( \si^\al_I :\partial_\al U_I \to \La^\al_I)_{I\in\Ii_{\p^\al\Kk}}$ of $\La|_{\p^\al\Kk}$ such that for $\eps>0$ sufficiently small we have $(\ti\io^\al_I)^*\si_I  = \si^\al_I\times \id_{A^\al_\eps}$.
 \end{defn}

In Definition~\ref{def:bundle} we implicitly worked with an isomorphism $(\ti\io^\al_I)_{I\in\Ii_{\p^\al\Kk}}$, which satisfies all but the product structure requirements of the following notion of isomorphisms on Kuranishi cobordisms.

\begin{defn} \label{def:buniso}
An {\bf isomorphism} $\Psi: \La\to \La'$ between vector bundles over $\Kk$ is a collection
$(\Psi_I: \La_I\to \La'_I)_{I\in \Ii_\Kk}$ of bundle isomorphisms covering the identity on $U_I$, that intertwine the transition maps, i.e.\ $\Ti\phi'_{IJ}\circ\Psi_I|_{U_{IJ}} = \Psi_J \circ \Ti \phi_{IJ}|_{U_{IJ}}$ for all $I\subset J$.

If $\Kk$ is a Kuranishi cobordism then we additionally require $\Psi$ to have product form in the collar. That is we require that for each $\al=0,1$ there is a restricted isomorphism $\Psi|_{\p^\al\Kk}= ( \Psi^\al_I :\La^\al_I \to \La'_I\,\!\!^\al)_{I\in\Ii_{\p^\al\Kk}}$ from $\La|_{\p^\al\Kk}$ to $\La'|_{\p^\al\Kk}$ such that for $\eps>0$ sufficiently small we have 
$\ti\io'_I\,\!\!^\al \circ \bigl(\Psi^\al_I \times \id_A\bigr) = \Psi_I \circ \ti\io^{\al}_I$ on $\partial^\al U_I\times A^\al_\eps$.
\end{defn}

\begin{remark}\rm
In the newly available language, Definition~\ref{def:cbundle} of a bundle on a Kuranishi cobordism requires isomorphisms (without product structure on the collar) for $\al=0,1$ from the product bundle 
$\La|_{\p^\al\Kk}\times A^\al_\eps$ to the $\eps$-collar restriction 
$(\iota^\al_\eps)^*\La := \bigl((\iota^\al_I)^*\La_I , (\iota^\al_J)^* \circ \Ti\phi_{IJ} \circ (\iota^\al_I)_* \bigr)_{I,J\in\Ii_{\p^\al\Kk}}$,  
given by the collection of pullback bundles and isomorphisms 
under the embeddings $\iota^\al_I : \partial^\al U_I \times A^\al_\eps \to U_I$.
\end{remark}

Note that, although the compatibility conditions are the same, the canonical section 
$s_\Kk = ( s_I : U_I\to E_I)_{I\in\Ii_\Kk}$ of a Kuranishi atlas does not form a section of a 
vector bundle since the obstruction spaces $E_I$ are in general not of the same dimension, 
hence no bundle isomorphisms $\Ti\phi_{IJ}$ as above exist.
Nevertheless, we will see that, 
there is a natural bundle associated with the section $s_\Kk$, namely its determinant line bundle, 
and that this line bundle is isomorphic to a bundle constructed by combining the determinant lines of the obstruction spaces $E_I$ and the domains $U_I$.
 
Here and in the following we will exclusively work with finite dimensional vector spaces.
First recall that the determinant line of a vector space $V$ is its maximal exterior power $\lm V := \wedge^{\dim V}\,V$, with $\wedge^0\,\{0\} :=\R$.
More generally, the {\bf determinant line of a linear map} $D:V\to W$ is defined to be 
$$
\det(D):= \lm\ker D \otimes \bigl( \lm \bigl( \qu{W}{\im D} \bigr) \bigr)^*.
$$ 
In order to construct isomorphisms between determinant lines, we will need to fix various conventions, in particular pertaining to the ordering of factors in their domains and targets.
We begin by noting that every isomorphism $F: Y \to Z$ between finite dimensional vector spaces induces an isomorphism
\begin{equation}\label{eq:laphi}
\La_F :\; \lm Y   \;\overset{\cong}{\longrightarrow}\; \lm Z , \qquad
y_1\wedge\ldots \wedge y_k \mapsto F(y_1)\wedge\ldots \wedge F(y_k) .
\end{equation}
For example, if $I\subsetneq J$ and  $x\in U_{IJ}$, it follows from the index condition in Definition~\ref{def:change} that the map 
for $x\in U_{IJ}$
\begin{equation}\label{eq:bunIJ}
\La_{IJ}(x): = \La_{\rd_x\phi_{IJ}} \otimes 
\bigl(\La_{[\Hat\phi_{IJ}^{-1}]}\bigr)^*
\, :\; \det(\rd_x s_I) \to \det(\rd_{\phi_{IJ}(x)} s_J)
\end{equation}
is an isomorphism, induced by the isomorphisms $\rd\phi_{IJ}:\ker\rd s_I\to\ker\rd s_J$ and
$[\Hat\phi_{IJ}] : \qu{E_I}{\im\rd s_I}\to\qu{E_J}{\im\rd s_J}$.
With this, we can define the determinant bundle $\det(s_\Kk)$ of a Kuranishi atlas. A second, isomorphic, determinant line bundle $\det(\Kk)$ with fibers $\lm \rT_x U_I \otimes \bigl( \lm E_I \bigr)^*$ will be constructed in Proposition~\ref{prop:orient}.

\begin{defn} \label{def:det} 
The {\bf determinant line bundle} of a weak Kuranishi atlas (or cobordism) $\Kk$ is the vector 
bundle $\det(s_\Kk)$ given by the line bundles 
$$
\det(\rd s_I):=\bigcup_{x\in U_I} \det(\rd_x s_I) \;\to\; U_I \qquad 
\text{for}\; I\in\Ii_\Kk, 
$$
and the isomorphisms $\La_{IJ}(x)$ in \eqref{eq:bunIJ} for $I\subsetneq J$ and $x\in U_{IJ}$.
\end{defn}

To show that  $\det(s_\Kk)$ is well defined, in particular that $x\mapsto \La_{IJ}(x)$ is smooth, 
we introduce some further natural\footnote{
Here a ``natural" isomorphism is one that is functorial, i.e.\ it commutes with 
the action on both sides induced by a vector space isomorphism.}
isomorphisms 
and fix various ordering conventions.

\begin{itemlist}
\item
For any subspace $V'\subset V$ the {\bf splitting isomorphism}
\begin{equation}\label{eq:VW}
\lm V\cong \lm V'\otimes \lm\bigl( \qu{V}{V'}\bigr)
\end{equation}
is given by completing a basis $v_1,\ldots,v_k$ of $V'$ to a basis $v_1,\ldots,v_n$ of $V$ and
mapping $v_1\wedge \ldots \wedge v_n \mapsto (v_1\wedge \ldots \wedge v_k) \otimes ([v_{k+1}]\wedge \ldots\wedge [v_n])$.
\item
For each isomorphism $F:Y\overset{\cong}{\to} Z$ the {\bf contraction isomorphism} 
\begin{equation} \label{eq:quotable}
\mathfrak{c}_F \,:\; \lm Y  \otimes  \bigl( \lm Z \bigr)^* \;\overset{\cong}{\longrightarrow}\; \R , 
\end{equation}
is given by the map
$\bigl(y_1\wedge\ldots \wedge y_k\bigr) \otimes \eta \mapsto \eta\bigl(F(y_1)\wedge \ldots 
\wedge F(y_k)\bigr)$.
\item
For any space $V$ we use the {\bf duality isomorphism}
\begin{equation}\label{eq:dual} 
\lm V^* \;\overset{\cong}{\longrightarrow}\; (\lm V)^*, \qquad
 v_1^*\wedge\dots\wedge v_n^* 
\;\longmapsto\;
(v_1\wedge\dots\wedge v_n)^* ,
\end{equation}
which corresponds to the natural pairing
$$
 \lm V \otimes \lm V^*   \;\overset{\cong}{\longrightarrow}\;  \R , \qquad
 \bigl(v_1\wedge\dots\wedge v_n\bigr) \otimes \bigl(\eta_1\wedge\dots\wedge \eta_n\bigr)
 \;\mapsto\;
 \prod_{i=1}^n \eta_i(v_i) 
$$
via the general identification (which in the case of line bundles $A,B$ maps $\eta\neq0$ to a nonzero homomorphism, i.e.\ an isomorphism)
\begin{equation}\label{eq:homid}
\Hom(A\otimes B,\R) \;\overset{\cong}{\longrightarrow}\; \Hom(B, A^*)
,\qquad
H \;\longmapsto\; \bigl( \; b \mapsto H(\cdot \otimes b) \; \bigr) .
\end{equation}
\end{itemlist}

\MS\NI
Next, we combine the above isomorphisms to obtain a more elaborate 
contraction isomorphism.

\begin{lemma} \label{lem:get} 
Every linear map $F:V\to W$ together with an isomorphism $\phi:K\to \ker F$ induces an isomorphism
\begin{align}\label{Cfrak}
\mathfrak{C}^{\phi}_F \,:\; \lm V \otimes \bigl(\lm W \bigr)^* 
&\;\overset{\cong}{\longrightarrow}\;  \lm K \otimes \bigl(\lm \bigl( \qu{W}{F(V)}\bigr) \bigr)^*  
\end{align}
given by
\begin{align}
(v_1\wedge\dots v_n)\otimes(w_1\wedge\dots w_m)^* &\;\longmapsto\;
\bigl(\phi^{-1}(v_1)\wedge\dots \phi^{-1}(v_k)\bigr)\otimes \bigl( [w_1]\wedge\dots [w_{m-n+k}] \bigr)^* ,
\notag
\end{align}
where $v_1,\ldots,v_n$ is a basis for $V$ with ${\rm span}(v_1,\ldots,v_k)=\ker F$, and $w_1,\dots, w_m$ is a basis for $W$ whose last $n-k$ vectors are $w_{m-n+i}=F(v_i)$ for $i=k+1,\ldots,n$.

In particular, for every linear map $D:V\to W$ we may pick $\phi$ as the inclusion $K=\ker D\hookrightarrow V$ to obtain an isomorphism
$$
\mathfrak{C}_{D} \,:\;  \lm V \otimes \bigl(\lm W \bigr)^* \;\overset{\cong}{\longrightarrow}\;  \det(D) .
$$ 
\end{lemma}

\begin{proof}
We will construct $\mathfrak{C}^{\phi}_F$ by composition of several isomorphisms.
As a first step let $F(V)^\perp\subset W^*$ be the annihilator of $F(V)$ in $W^*$, then the splitting isomorphism \eqref{eq:VW} identifies $\lm W^*$ with $\lm ( F(V)^\perp )\otimes \lm \bigl(\qu{W^*}{F(V)^\perp}\bigr)$.
Next, we apply \eqref{eq:laphi} to the isomorphisms 
$F(V)^\perp \overset{\cong}{\to} \bigl(\qu{W}{F(V)}\bigr)^*$
and 
$\qu{W^*}{F(V)^\perp}\overset{\cong}{\to} F(V)^*$,
and apply the duality isomorphism \eqref{eq:dual} in all factors to obtain the isomorphism 
$$
S_W \,:\;
\bigl(\lm W\bigr)^* \;\overset{\cong}{\longrightarrow}\;  \bigl(\lm \bigl(\qu{W}{F(V)}\bigr)\bigr)^* \otimes \bigl( \lm F(V) \bigr)^*
$$
given by 
$(w_1\wedge \ldots  \wedge w_m)^* \mapsto ([w_1]\wedge \ldots\wedge [w_\ell])^* \otimes (w_{\ell+1}\wedge \ldots \wedge w_m)^*$
for any basis $w_1,\ldots,w_m$ of $W$ whose last elements $w_{\ell+i}$ for $i=1,\ldots,m-\ell=n-k$ span $F(V)$.
On the other hand, we apply the splitting isomorphism \eqref{eq:VW} for $\ker F\subset V$ and \eqref{eq:laphi} for $\phi^{-1}: \ker F\to K$ to obtain an isomorphism
$$
S_V \,:\;
\lm V \;\overset{\cong}{\longrightarrow}\;   \lm K \otimes \lm \bigl(\qu{V}{\phi(K)}\bigr)  
$$
given by 
$v_1\wedge \ldots \wedge v_n \mapsto (\phi^{-1}(v_1)\wedge \ldots\wedge \phi^{-1}(v_k)) \otimes ([v_{k+1}]\wedge \ldots \wedge [v_n])$
for any basis $v_1,\ldots,v_n$ of $V$ such that $v_1,\ldots,v_k$ spans $\ker F$.
Finally, note that $F$ descends to an isomorphism $[F] : \qu{V}{\phi(K)} \overset{\cong}{\to} F(V)$, so we wish to apply the contraction isomorphism 
$$
\mathfrak{c}_{[F]} : \lm \bigl(\qu{V}{\phi(K)}\bigr) \otimes \bigl( \lm F(V) \bigr)^* \to \R
$$
from \eqref{eq:quotable}.
Since these factors do not appear adjacent after applying $S_V\otimes S_W$, we compose $S_W$ with an additional reordering isomorphism -- noting that we do not introduce signs in switching factors here
$$
R \,:\; A \otimes B  \overset{\cong}{\longrightarrow}\;
B \otimes A , \qquad
 a \otimes b  \; \longmapsto\; b \otimes a .
$$
Finally, using the natural identification $\lm K \otimes \R\otimes  \bigl(\lm \bigl(\qu{W}{F(V)}\bigr)\bigr)^* \cong \lm K \otimes  \bigl(\lm \bigl(\qu{W}{F(V)}\bigr)\bigr)^*$ we obtain an isomorphism 
$$
\bigl( \id_{ \lm K }\otimes \mathfrak{c}_{[F]} \otimes \id_{ (\lm (\qu{W}{F(V)}))^*} \bigr) \circ \bigl( S_V \otimes (R\circ S_W) \bigr) .
$$ 
To see that it coincides with $\mathfrak{C}_F^\phi$ as described in the statement, note that -- using the bases as above -- it maps
$(v_1\wedge \ldots \wedge v_n) \otimes(w_1\wedge \ldots \wedge w_m)^*$ to $(\phi^{-1}(v_1)\wedge \ldots\wedge \phi^{-1}(v_k)) \otimes ([w_1]\wedge \ldots\wedge [w_\ell])^*$ multiplied with the factor
$(w_{\ell+1}\wedge \ldots \wedge w_m)^*\bigl(F(v_{k+1})\wedge \ldots \wedge F(v_n)\bigr)$, and that the latter equals $1$ if we choose $w_{\ell+i}=F(v_i)$ for $i=1,\ldots, n-k$. Note here that the existence of an isomorphism $F$ implies $m-\ell = n-k$, so that $m-n=\ell+k$, and hence $w_{m-n+(k+i)}=w_{\ell+i}$.
\end{proof}

\begin{prop}\label{prop:det0}  
For any weak Kuranishi atlas, $\det(s_\Kk)$ is a well defined line bundle over $\Kk$.
Further, if $\Kk$ is a weak Kuranishi cobordism, then $\det(s_\Kk)$ can be given product form on the collar of $\Kk$ with restrictions $\det(s_\Kk)|_{\p^\al\Kk} = \det(s_{\p^\al\Kk})$ for $\al= 0,1$.
The required bundle isomorphisms from the product $\det(s_{\p^\al\Kk})\times A^\al_\eps$ to the collar restriction  $(\io^\al_\eps)^*\det(s_\Kk)$ are given in \eqref{orient map}.
\end{prop}
\begin{proof}
To see that $\det(s_\Kk)$ is a line bundle over $\Kk$, we first note that each topological bundle $\det(\rd s_I)$ is a 
smooth line bundle, since it has compatible local trivializations $\det(\rd s_I)\cong\lm\ker(\rd s_I \oplus R_I)$ 
induced from constant linear injections $R_I:\R^{N}\to E_I$ 
which locally cover the cokernel, see  e.g.\ \cite[Appendix~A.2]{MS}. 
There are various natural ways to define these maps; the crucial choice is the sign in equation \eqref{Cfrak}.
 \footnote{
See \cite{Z3} for a discussion of the different  conventions.
Changing the sign in  \eqref{Cfrak} for example by the factor $(-1)^{n-k}$ affects the local trivializations (and hence the topology of the determinant bundle) because \eqref{Cfrak}  is applied below to the family of operators $F_x, x\in U_I,$ the dimension of whose kernels  varies with $x$.} 

At each point $x\in U_I$ we will use the contraction map $\mathfrak{C}^{\phi_x}_{F_x}$ of Lemma~\ref{lem:get} for the linear map $F_x$ and isomorphism to its kernel $\phi_x$, where
\begin{align*}
F_x  \,:\; \ker(\rd_x s_I \oplus R_I) &\; \to \; \im R_I \;\subset\; E_I , 
\qquad\qquad\qquad\qquad (v,r) \;\mapsto\; \rd_x s_I(v) ,
\\
\phi_x \,:\;\;\,  K := \ker \rd_x s_I &\;\to\; \ker( \rd_x s_I \oplus R_I)  \; \subset \; \rT_x U_I\oplus \R^N , 
\qquad
k\mapsto (k,0) .
\end{align*}
Note here that $\ker(\rd_x s_I \oplus R_I)=\bigl\{(v,r)\in \rT_x U_I\oplus \R^N \,\big|\, \rd_x s_I (v) = - R_I(r) \bigr\}$, so that $F_x$ indeed maps to $\im R_I$ with $F_x(v,r)=-R_I(r)$, and its image is $\im F_x = \im \rd_x s_I \cap \im R_I$.

If  we restrict $x$ to an open set $O\subset U_I$ on which $\rd_x s_I\oplus R_I$ is surjective, then the inclusion $\im R_I \hookrightarrow E_I$ induces an isomorphism
$$
\io_x \,: \; \qu{\im R_I}{\im \rd_x s_I \cap \im R_I} \; \overset{\cong}{\longrightarrow} \; \qu{E_I}{\im \rd_x s_I } .
$$
Indeed, $\io_x$ is surjective since $E_I = \im \rd_x s_I + \im R_I$ and injective by construction.
Hence \eqref{eq:laphi} together with dualization defines an isomorphism $\La_{\io_x}^* : \bigl( \lm \qu{E_I}{\im \rd_x s_I }\bigr)^* \to \bigl( \qu{\im R_I}{\im \rd_x s_I \cap \im R_I} \bigr)^*$, which we invert and compose with the contraction isomorphism of Lemma~\ref{lem:get} to obtain isomorphisms
\begin{align}\notag
T_{I,x} \, := \; \bigl(\id_{\lm\ker \rd_x s_I}\otimes (\La_{\io_x}^*)^{-1}\bigr) \circ {\mathfrak C}^{\phi_x}_{F_x} \;:\;\;\; & 
 \lm \ker(\rd_x s_I \oplus R_I) \otimes \bigl( \lm \im R_I \bigr)^*  \\
&\;\overset{\cong}{\longrightarrow}\;
\lm\ker \rd_x s_I \otimes \bigl(\lm \bigl(\qu{E_I}{\im \rd_x s_I }\bigr)\bigr)^*. \label{eq:TIx} 
\end{align}
Precomposing this with the isomorphism $\R\cong \lm \R^N\stackrel{\La_{R_I}}\cong \lm \im R_I$ from
\eqref{eq:laphi}, we obtain a trivialization of $\det(\rd s_I)|_O$ given by isomorphisms
\begin{align}\notag
\Hat T_{I,x} \,:\; 
\lm \ker(\rd_x s_I \oplus R_I) 
&\;\overset{\cong}{\longrightarrow}\; 
\qquad\qquad \det(\rd_x s_I)  \\ \label{eq:HatTx}
\ov v_1\wedge \ldots \wedge \ov v_n
&\;\longmapsto\; 
 (v_1\wedge\dots v_k)\otimes  \bigl( [R_I(e_1)]\wedge\dots [R_I(e_{N-n+k})] \bigr)^*,
\end{align}
where $\ov v_i=(v_i,r_i)$ is a basis of $\ker(\rd_x s_I \oplus R_I)$ such that $v_1,\ldots,v_k$ span $\ker \rd_x s_I$ (and hence $r_1=\ldots=r_k=0$), and $e_{1},\ldots, e_{N}$ is a positively ordered normalized basis of $\R^N$ (that is $e_1\wedge\ldots e_N = 1 \in \R \cong \lm\R^N$) such that $R_I(e_{N-n+i}) = \rd_x s_I(v_i)$ for $i = k+1,\ldots, n$. In particular, the last $n-k$ vectors span $\im \rd_x s_I \cap \im R_I \subset E_I$, and thus the first $N-n+k$ vectors $[R_I(e_1)],\dots, [R_I(e_{N-n+k})]$ span the cokernel $\qu{E_I}{\im\rd_x s_I}\cong\qu{\im R_I}{\im\rd_x s_I\cap \im R_I}$.

Next, we show that these trivializations do not depend on the choice of injection $R_I:\R^N\to E_I$.
Indeed, given another injection $R_I':\R^{N'}\to E_I$ that also maps onto the cokernel of $\rd s_I$,
we can choose a third injection $R_I'':\R^{N''}\to E_I$ that is surjective, and compare it to both of $R_I, R_I'$.  
Hence it suffices to consider the following two cases:

\begin{itemlist}
\item $N=N'$ and $R_I = R_I'\circ \io$ for a bijection $\io: \R^{N} \overset{\cong}{\to} \R^{N'}$;
\item $N<N'$ and $R_I=R_I'\circ\pr$ for the canonical projection $\pr: \R^{N'}\to \R^N\times\{0\} \cong\R^N$.
\end{itemlist}

In the second case denote by $\io: \R^{N}\to \R^N\times\{0\}\subset\R^{N'}$ the canonical injection, then in both cases we have $R_I = R_I'\circ \io$, and thus $\id \times \io$ induces an injection $\ker(\rd s_I\oplus R_I)\to
\ker(\rd s_I\oplus R_I')$ so that there is a well defined quotient bundle
$\qu{\ker(\rd s_I\oplus R_I')}{\ker(\rd s_I\oplus R_I)} \to U_I$.

In case $N<N'$ we claim that an appropriately scaled choice of local trivialization for this quotient over an open set $O\subset U_I$, on which both trivializations of $\det(\rd s_I)|_O$ are defined, induces a bundle isomorphism
$\Psi: \lm \ker(\rd s_I\oplus R_I)|_O\to \lm \ker(\rd s_I\oplus R_I')|_O$ 
that is compatible with the trivializations $\Hat T_I$ and $\Hat T_I': \lm \ker(\rd s_I\oplus R'_I)|_O \to \det(\rd s_I)|_O$ constructed as in \eqref{eq:HatTx}, that is $\Hat T_I = \Hat T_I' \circ \Psi$.

To define $\Psi$, let $n:=\dim \ker(\rd s_I\oplus R_I)$ and fix a trivialization of the quotient, that is a family of smooth sections $\bigl(\ov v^\Psi_{i}=(v^\Psi_{i},r^\Psi_{i})\bigr)_{i=n+1, \ldots,n'}$ of $\ker(\rd s_I\oplus R_I')|_O$ with 
$n':= n+N'-N$, that induces a basis for the quotient space at each point $x\in O$.
Here we may want to rescale $\ov v^\Psi_{n+1}$ by a nonzero real, as discussed below.
Then for fixed $x\in O$, any choice of basis $(\ov v_i)_{i=1,\ldots,n}$ of $\ker(\rd_x s_I\oplus R_I)$ induces a basis $(\id \times \io)(\ov v_1),\dots, (\id \times \io)(\ov v_n), \ov v^\Psi_{n+1}, \ldots, \ov v^\Psi_{n'}$ of $ \ker(\rd_x s_I\oplus R_I')$, and we define $\Psi$ by 
$$
\Psi_x \,: \; \ov v_1\wedge\dots\wedge \ov v_n \;\mapsto\; (\id \times \io)(\ov v_1)\wedge\dots\wedge (\id \times \io)(\ov v_n)\wedge \ov v^\Psi_{n+1}(x)\wedge \ldots \wedge \ov v^\Psi_{n'}(x) ,
$$
which varies smoothly with $x\in O$. It remains to show that, for appropriate choice of the sections $\ov v^\Psi_i$, we have $\Hat T_{I,x} = \Hat T_{I,x}' \circ \Psi_x$ for any fixed $x\in O$.
For that purpose we express the trivializations $\Hat T_{I,x}$ and $\Hat T'_{I,x}$ as in \eqref{eq:HatTx}. This construction begins by choosing a basis $(\ov v_i)_{i=1,\dots,n}$ of $\ker (\rd_x s_I \oplus R_I)$, where the first $k$ elements $\ov v_i=(v_i,0)$ span $\ker \rd_x s_I\times\{0\}$.
A compatible choice of basis $(\ov v'_i)_{i=1,\dots,n'}$ for $\ker (\rd_x s_I \oplus R'_I)$ is given by $\ov v'_i := (\id \times \io)(\ov v_i)$ for $i=1,\dots,n$, and $\ov v'_i:= \ov v^\Psi_i$ for $i=n+1,\ldots,n'$.
Note here that $\ov v'_i = \ov v_i$ for $i=1,\ldots,k$.
Next, one chooses a positively ordered normalized basis $e_{1},\ldots, e_{N}$ of $\R^N$ such that $R_I(e_{N-n+i}) = \rd_x s_I(v_i)$ for $i = k+1,\ldots, n$. 
Then the first $N-n+k$ vectors $[R_I(e_1)],\dots, [R_I(e_{N-n+k})]$ coincide with $[R'_I(\io(e_1))],\dots, [R'_I(\io(e_{N'-n'+k}))]$ and span the cokernel $\qu{E_I}{\im\rd_x s_I}$, and the last $n-k$ vectors span $\im \rd_x s_I \cap \im R_I \subset E_I$. So we obtain a corresponding basis $e'_1,\ldots, e'_{N'}$ of $\R^{N'}$ by taking $e'_i = \io(e_i)$ for $i=1,\ldots,N$ and $e'_{N+i} = (R'_I)^{-1}\bigl( \rd_x s_I(v^\Psi_{n+i}(x))$ for $i=1,\ldots, N'-N = n'-n$. 
To obtain the correct definition of $\Hat T'_{I,x}$, we then rescale $v^\Psi_{n'}$ by the reciprocal of 
\begin{align*}
\la(x) &\,:=\; \io(e_1)\wedge \ldots \wedge \io(e_N) \wedge (R'_I)^{-1}\bigl( \rd_x s_I(v^\Psi_{n+1}(x))\bigr) \wedge\ldots\wedge (R'_I)^{-1}\bigl( \rd_x s_I(v^\Psi_{n'}(x))\bigr) \\
& \;\in\; \lm \R^{N'} \;\cong \; \R ,
\end{align*}
such that $e'_1,\ldots, e'_{N'-1}, \la(x)^{-1} e'_{N'}$ becomes positively ordered and normalized.
Note here that $\la:O\to\R$ is a smooth nonvanishing function of $x$, depending only on the sections $v^\Psi_{n+1}(x), \ldots, v^\Psi_{n'}(x)$ since $\io(e_i)=(e_i,0)$ are a positively ordered normalized basis of $\R^N\times\{0\}\subset\R^{N'}$ for all $x\in O$.
Thus $v^\Psi_{n+1}(x), \ldots, \la(x)^{-1} v^\Psi_{n'}(x)$ defines a smooth trivialization of the quotient bundle $\qu{\ker(\rd s_I\oplus R_I')}{\ker(\rd s_I\oplus R_I)} \to O$, for which the induced map $\Psi$ now provides the claimed compatibility. Indeed, we have by construction
\begin{align}
\bigl( \Hat T'_{I,x}\circ\Psi_x \bigr) \bigl( \ov v_1\wedge\dots\wedge \ov v_n\bigr) &\;=\;  
(v_1\wedge\dots\wedge v_k)\otimes  \bigl([R_I'(e_1')]\wedge\dots\wedge [R_I'(e_{N'-n'+k}')] \bigr)^*  \notag \\ \notag
&\;=\; (v_1\wedge\dots\wedge v_k)\otimes 
\bigl([R_I(e_1)]\wedge\dots\wedge [R_I(e_{N-n+k})]\bigr)^* \\ \label{tpsit} 
&\;=\; \Hat T_{I,x}\bigl(\ov v_1\wedge\dots\wedge \ov v_n\bigr).
\end{align}

In case $N=N'$ we define an isomorphism $\Psi$ as above, which however does not depend on any choice of vectors $\ov v^\Psi_i$. 
Then in the above calculation of $\Hat T_{I,x}$ and $\Hat T'_{I,x}$, the factor $\la = \io(e_1) \wedge \ldots \wedge \io(e_{N})$ is constant (equal to the determinant of $\io = (R_I')^{-1}\circ R_I$), and hence $\la^{-1}\Psi$ intertwines the trivializations $\Hat T_I$ and $\Hat T_I'$. This completes the proof that the local trivializations of $\det(\rd s_I)$ do not depend on the choice of $R_I$.
In particular, $\det(\rd s_I)$ is a smooth line bundle over $U_I$ for each $I\in\Ii_\Kk$.

To complete the proof that $\det(s_\Kk)$ is a vector bundle we must check that the 
lifts $\La_{IJ}$ given in \eqref{eq:bunIJ} of the
coordinate changes  $\Phi_{IJ}$ 
are smooth bundle isomorphisms.
Since the $\La_{IJ}(x)$ are constructed to be fiberwise isomorphisms, and the weak cocycle condition for the coordinate changes transfers directly to these bundle maps,
the nontrivial step is to check that $\La_{IJ}(x)$ varies smoothly with $x\in U_{IJ}$. 
For that purpose note that  
any trivialization $\Hat T_I$ near a given point $x_0\in U_{IJ}$ using a choice of $R_I$ as above, induces a trivialization $\Hat T_J$ of $\det(\rd s_J)$ near $\phi_{IJ}(x_0)\in U_J$ 
using the injection $R_J:=\Hat\phi_{IJ}\circ R_I$, since by the index condition $\Hat\phi_{IJ}$ identifies the cokernels. 
We claim that these local trivializations transform $\La_{IJ}(x)$ into the isomorphisms $\La_{\rd_x \phi_{IJ} \oplus \id_{\R^{N}} }$ of \eqref{eq:laphi} induced by the smooth family of isomorphisms
$$
\rd_x\phi_{IJ} \oplus \id_{\R^{N}} \,:\;  \ker(\rd_x s_I \oplus R_I) \;\overset{\cong}{\longrightarrow}\; \ker\bigl(\rd_{\phi_{IJ}(x)} s_J \oplus (\Hat\phi_{IJ}\circ R_I)\bigr) .
$$
Note here that these embeddings are surjective since for $(v,z)\in \rT U_J \times \R^{N}$ with $\rd s_J(v)=-\Hat\phi_{IJ}(R_I(z))$ the tangent bundle condition $\im\rd s_J \cap \im\Hat\phi_{IJ}=\rd s_J (\im\rd\phi_{IJ})$
from Lemma~\ref{le:change}, the partial index condition $\ker\rd s_J \subset \im\rd\phi_{IJ}$, and injectivity of $\Hat\phi_{IJ}$ imply $v\in \im\rd s_I$ with $\rd s_I(v)=-R_I(z)$.
Moreover, $\rd_x\phi_{IJ}$ varies smoothly with $x\in U_{IJ}$, and hence $\La_{\rd_x\phi_{IJ} \oplus \id_{\R^{N}} }$  varies smoothly with $x\approx x_0$. 
So to prove smoothness of $\La_{IJ}$ near $x_0$ it suffices to prove the transformation as claimed, i.e.\ at fixed $x\in U_{IJ}$ 
\begin{equation}\label{eq:etrans}
\Hat T_{J,\phi_{IJ}(x)}\circ \La_{\rd_x\phi_{IJ} \oplus \id_{\R^{N}}} \;=\; \La_{IJ}(x) \circ \Hat T_{I,x} .
\end{equation}
For that purpose we may simply compare the explicit maps given in \eqref{eq:HatTx}.
So let $\ov v_i=(v_i,r_i)$ be a basis of $\ker(\rd_x s_I \oplus R_I)$ such that $v_1,\ldots,v_k$ span $\ker \rd_x s_I$.
Then, correspondingly, 
$\ov v'_i=\bigl(\rd_x\phi_{IJ} \oplus \id_{\R^{N}}\bigr)(\ov v_i)$ ia a basis of $\ker(\rd_{\phi_{IJ}(x)} s_J \oplus R_J)$ such that $v'_i=\rd_x\phi_{IJ}(v_i)$ for $i=1,\ldots,k$ span $\ker \rd_{\phi_{IJ}(x)} s_J$.
Next, let $e_{1},\ldots, e_{N}$ be a positively ordered normalized basis of $\R^N$ such that $R_I(e_{N-n+i}) = \rd_x s_I(v_i)$ for $i = k+1,\ldots, n$. Then, correspondingly, we have 
$$
R_J(e_{N-n+i}) = \Hat\phi_{IJ}\bigl(R_I(e_{N-n+i})\bigr) =   \Hat\phi_{IJ}\bigl(\rd_x s_I(v_i)\bigr)
 =  \rd_{\phi_{IJ}(x)}s_J \bigl( \rd_x\phi_{IJ}(v_i)\bigr)  =  \rd
 s_J ( v'_i) .
$$
Using these bases in \eqref{eq:HatTx} we can now verify \eqref{eq:etrans},
\begin{align*}
& \bigl(\La_{IJ}(x) \circ \Hat T_{I,x}\bigr)\bigl(\ov v_1\wedge \ldots \wedge \ov v_n\bigr) \\
&\;=\;
\La_{IJ}(x) \bigl(  (v_1\wedge\ldots \wedge v_k)\otimes  \bigl( [R_I(e_1)]\wedge\ldots\wedge [R_I(e_{N-n+k})] \bigr)^* \bigr) \\
&\;=\;
\bigl( \rd_x\phi_{IJ}(v_1)\wedge\ldots\wedge  \rd_x\phi_{IJ}(v_k)\bigr) \otimes  \bigl( [\Hat\phi_{IJ}(R_I(e_1))]\wedge\ldots \wedge [\Hat\phi_{IJ}(R_I(e_{N-n+k}))] \bigr)^* \bigr) \\
&\;=\;
( v'_1\wedge\ldots\wedge  v'_k) \otimes  \bigl( [R_J(e_1)]\wedge\ldots\wedge [R_J(e_{N-n+k})] \bigr)^* \bigr) \\
&\;=\;
\Hat T_{J,\phi_{IJ}(x)}\bigr)\bigl(\ov v'_1\wedge \ldots \wedge \ov v'_n\bigr) 
\;=\;
\bigl( \Hat T_{\phi_{IJ}(x)}\circ \bigl(\rd\phi_{IJ}(x) \oplus \id_{\R^{N}}\bigr) \bigr)\bigl(\ov v_1\wedge \ldots \wedge \ov v_n\bigr).
\end{align*}
This finishes the construction of $\det(s_\Kk)$ for a weak Kuranishi atlas $\Kk$.

In the case of a weak Kuranishi cobordism $\Kk$, we moreover have to
construct bundle isomorphisms from collar restrictions to the product bundles $\det(s_{\p^\al\Kk})\times A^\al_\eps$ to prove that $\det(s_\Kk)$ is a line bundle in the sense of Definition~\ref{def:bundle} with the claimed restrictions.
That is, we have to construct bundle isomorphisms $\ti\io^\al_I : \det(\rd s^\al_I)\times A^\al_\eps \to \det(\rd s_I)|_{\im\io^\al_I}$ for $\al=0,1$, $I\in\Ii_{\p^\al\Kk}$, and $\eps>0$ less than the collar width of $\Kk$, and check the identities $\La_{IJ} \circ \ti\io^{\al}_I = \ti\io^\al_{J} \circ \bigl(\La^\al_{IJ}\times\id_{A^\al_\eps}\bigr)$.

For that purpose recall that $(s_I\circ\io^\al_I )(x,t) = s^\al_I(x)$ for $(x,t)\in\partial^\al U_I \times A^\al_\eps$, so that we have a trivial identification $\id_{E_I} : \im \rd_x s^\al_I \to  \im\rd_{\io^\al_I(x,t)} s_I$ of the images and an isomorphism 
$\rd_{(x,t)}\io^\al_I :\ker \rd_x s^\al_I \times\R \to  \ker\rd_{\io^\al_I(x,t)} s_I$.
The latter gives rise to an isomorphism given by wedging with the canonical positively oriented unit vector $1\in\R =\rT_t A^\al_\eps$,
\begin{equation} \label{wedge 1}
 \wedge_1 \,:\; \lm \ker \rd_x s^\al_I \;\to\; \lm \bigl( \ker \rd_x s^\al_I \times\R \bigr) , \qquad
\eta \;\mapsto\; 1\wedge \eta.
\end{equation}
Here and throughout we identify vectors $\eta_i\in\ker \rd s^\al_I$ with 
$(\eta_i,0)\in \ker \rd s^\al_I \times\R$ and also abbreviate 
$1:= (0,1)\in \ker \rd s^\al_I \times\R$.
This map now composes with the induced isomorphism $\La_{\rd_{(x,t)}\io^\al_I}$ from \eqref{eq:laphi} and can be combined with the identity on the cokernel factor to obtain fiberwise isomorphisms
\begin{equation}\label{orient map}
\ti\io^\al_I (x,t) :=  \bigl( \La_{\rd_{(x,t)}\io^\al_I} \circ \wedge_1 \bigr) 
\otimes \bigl(\La_{\id_{E_I}}\bigr)^*
\;:\; \det(\rd_x s^\al_I)\times A^\al_\eps \;\to\;  \det(\rd_{\io^\al_I(x,t)} s_I) .
\end{equation}
These isomorphisms vary smoothly with $(x,t)\in \p^\al U_I\times A^\al_\eps$ since the compatible local trivializations $\lm\ker(\rd s^\al_I \oplus R_I) \to \det(\rd s^\al_I)$ and $\lm\ker(\rd s_I \oplus R_I) \to \det(\rd s_I)$ transform $\ti\io^\al_I(x,t)$ to $\La_{\rd_{(x,t)}\io^\al_I\oplus\id_{\R^N}}\circ \wedge_1$.
Moreover, $(x,t)\to \ti\io^\al_I(x,t)$ lifts $\io^\al_I$ and thus defines the required bundle isomorphism
$\ti\io^\al_I : \det(\rd s^\al_I)\times A^\al_\eps \to  \det(\rd s_I)|_{\im\io^\al_I}$ for each $I\in\Ii_{\p^\al\Kk}$.
Finally, the isomorphisms \eqref{orient map} intertwine $\La_{IJ} = \La_{\rd\phi_{IJ}} \otimes 
\bigl(\La_{[\Hat\phi_{IJ}]^{-1}}\bigr)^*$
and $\La^\al_{IJ} = \La_{\rd\phi^\al_{IJ}} \otimes \bigl(\La_{[\Hat\phi_{IJ}]^{-1}}\bigr)^*$ by the product form of the coordinate changes $\phi_{IJ}\circ\io^\al_I = \io^\al_ J\circ (\phi^\al_{IJ}\times \id_{A^\al_\eps})$, 
and because $\rd_{\io^\al_I(x,t)}\phi_{IJ}$ maps $\rd_{(x,t)} \io_I^\al (1)$ to $\rd_{(\phi^\al_{IJ}(x),t)} \io_J^\al (1)$,
both of which are wedged on by \eqref{orient map} from the left hand side.
(For an example of a detailed calculation see the end of the proof of Proposition~\ref{prop:orient1}.)
This finishes the proof.
\end{proof}

We next use the determinant bundle $\det(s_\Kk)$ to define the notion of an orientation of a Kuranishi atlas.

\begin{defn}\label{def:orient} 
A  weak Kuranishi atlas or Kuranishi cobordism $\Kk$ is {\bf orientable} if there exists a nonvanishing section $\si$ of the bundle $\det(s_\Kk)$ (i.e.\ with $\si_I^{-1}(0)=\emptyset$ for all $I\in\Ii_\Kk$).  An {\bf orientation} of $\Kk$ is a choice of nonvanishing section $\si$ of $\det(s_\Kk)$. 
An {\bf oriented Kuranishi atlas or cobordism} is a pair $(\Kk,\si)$ consisting of a  Kuranishi atlas or cobordism and an orientation $\si$ of $\Kk$.

For an oriented Kuranishi cobordism $(\Kk,\si)$, the {\bf induced orientation of the boundary} $\p^\al\Kk$ for $\al=0$ resp.\ $\al=1$ is the orientation of $\p^\al\Kk$,
$$
\p^\al\si \,:=\; \Bigl( \bigl( (\ti\io^\al_I)^{-1} \circ\si_I \circ \io^\al_I \bigr)\big|_{\partial^\al U_I \times\{\al\} } \Bigr)_{I\in\Ii_{\p^\al\Kk}}
$$
given by the isomorphism $(\ti\io^\al_I)_{I\in\Ii_{\p^\al\Kk}}$ in \eqref{orient map} between a collar neighbourhood of the boundary in $\Kk$ and the product Kuranishi atlas $\p^\al\Kk\times A^\al_\eps$, followed by restriction to the boundary $\p^\al \Kk=\p^\al\bigl(\p^\al \Kk\times A^\al_\eps\bigr)$, where we identify $\partial^\al U_I \times\{\al\} \cong \partial^\al U_I$.

With that, we say that two oriented weak Kuranishi atlases $(\Kk^0,\si^0)$ and $(\Kk^1,\si^1)$ are {\bf oriented cobordant} if there exists a weak Kuranishi cobordism $\Kk^{[0,1]}$ from $\Kk^0$ to $\Kk^1$ and a section $\si$ of $\det(s_{\Kk^{[0,1]}})$ such that  $\partial^\al\si=\si^\al$ for $\al=0,1$.
\end{defn}

\begin{rmk}\label{rmk:orientb}\rm  
Here we have defined the induced orientation on the boundary $\p^\al \Kk$ of a cobordism so that it is completed to an orientation of the collar by adding the  positive unit vector $1$ along $A^\al_\eps\subset \R$  rather than the more usual outward normal vector.  Further, although we write the collar as $U^\al_I\times A^\al_\eps$, formula \eqref{wedge 1} above shows that if $\eta_1,\dots, \eta_n$ is a positively ordered basis for $\rT_x U^\al_I$ then $1, \eta_1,\dots, \eta_n$ is a positively ordered basis for $\rT_x (U^\al_I\times A^\al_\eps)$. 
While the first convention merely simplifies notation, the second is necessary for compatibility checks in Proposition~\ref{prop:det0} as well as Proposition~\ref{prop:orient1}~(i) below; cf.\ the proof of \eqref{oclaim} below.
The alternative convention of adding the normal vector as last vector leads to sign ambiguities between $\det(s_\Kk)$ and the determinant line bundle $\det(s|_\Vv+\nu)$ for a perturbation, since the dimensions of the kernels of $\rd s_I$ and $\rd s_I+\nu_I$ need not be the same.
\end{rmk}

\begin{lemma}\label{le:cK}
Let $(\Kk,\si)$ be an oriented weak Kuranishi atlas or cobordism. 
\begin{enumerate}\item
The orientation  $\si$ induces a canonical orientation $\si|_{\Kk'}:=(\si_I|_{U'_I})_{I\in\Ii_{\Kk'}}$ on each shrinking $\Kk'$ of $\Kk$ with domains $\bigl(U'_I\subset U_I\bigr)_{I\in\Ii_{\Kk'}}$.
\item
In the case of a Kuranishi cobordism $\Kk$, the restrictions to boundary and shrinking commute, that is
$(\si|_{\Kk'})|_{\p^\al\Kk'} = (\si|_{\p^\al\Kk})|_{\p^\al\Kk'}$.
\item 
In the case of a weak Kuranishi atlas $\Kk$, the orientation $\si$ on $\Kk$ induces an 
orientation  
$\si^{[0,1]}$ 
on $\Kk\times[0,1]$, 
which induces the given orientation $\p^\al\si^{[0,1]}=\si$ of the boundaries $\p^\al(\Kk\times[0,1]) = \Kk$ for $\al=0,1$.
\end{enumerate}
\end{lemma}

\begin{proof}  
By definition, $\det(s_{\Kk'})$ is the line bundle over $\Kk'$ consisting of the bundles $\det(\rd s'_I)=\det(\rd s_I)|_{U'_I}$ and the transition maps $\La'_{IJ}=\La_{IJ}|_{U'_{IJ}}$.
The restricted sections $\si_I|_{U'_I}$ of $\det(\rd s'_I)$ are hence compatible with the transition maps $\La'_{IJ}$ and have product form near the boundary in the case of a cobordism. Since they are nonvanishing, they define an orientation of $\Kk'$.
Commutation of restrictions holds since both $(\si|_{\Kk'})|_{\p^\al\Kk'}$ and $(\si|_{\p^\al\Kk})|_{\p^\al\Kk'}$ are given by $\si_I|_{\p^\al U'_I}$ with $\p^\al U'_I = \p^\al U_I \cap U'_I$.

For  part~(iii)  we consider an oriented, additive, weak Kuranishi atlas $(\Kk,\si)$ and begin by constructing an induced orientation of the product cobordism $\Kk\times[0,1]$. For that purpose we use the bundle isomorphisms 
$$
\ti \io_I:=\wedge_1 \otimes (\La_{\id_{E_I}})^* \,:\; \det(\rd s_I)\times [0,1] \;\to\; \det(\rd s'_I)
$$
with $s'_I(x,t)=s_I(x)$, covering $\io_I=\id_{U_I\times[0,1]}$.
These coincide with the maps defined in \eqref{orient map} for the interval $[0,1]$ instead of $A^\al_\eps$, so the proof of Proposition~\ref{prop:det0}
 shows that they provide an isomorphism $\ti \io$ from the product bundle $\det(s_\Kk)\times [0,1]$ to the determinant bundle of the product $\det(s_{\Kk\times [0,1]})$.
Now an orientation $\si$ of $\Kk$ determines an orientation 
$\si^{[0,1]}:=\ti\io_*\si$ of the product $\Kk\times [0,1]$
given by $(\ti\io_*\si)_I(x,t)=\ti\io_I(x,t)\bigl(\si_I(x)\bigr)$.
Further, using $\ti\io^\al:=\ti\io$ to define the collar structure on $\det(s_{\Kk\times [0,1]})$, the restrictions to both boundaries are $\p^\al( \ti\io_*\si) =\si$ since $\ti\io_I^{-1}\circ (\ti\io_*\si)_I
\circ \io^\al_I|_{U_I\times\{\al\}} = \si_I$.
\end{proof}

The arguments of Proposition~\ref{prop:det0} equally apply for any reduction $\Vv$ of a
 Kuranishi atlas $\Kk$ and an admissible perturbation $\nu$ to define a line bundle $\det(s|_\Vv + \nu)$ over $\Vv$ (or more accurately over the Kuranishi atlas $\Kk^\Vv$ defined in Proposition~\ref{prop:red}).  
Instead of setting up a direct comparison between the bundles $\det(s|_\Vv + \nu)$  for different $\nu$,
we will work with a 
``more universal" determinant bundle $\det(\Kk)$ over $\Kk$.
This will allow us to obtain compatible orientations of the determinant bundles over the perturbed zero set $\det(s|_\Vv + \nu)|_{(s+\nu)^{-1}(0)}$ for different transverse perturbations $\nu$.
We will construct the bundle $\det(\Kk)$ from the determinant bundles of the zero sections in each chart. 
However, since the zero section $0_\Kk$ does not satisfy the index condition, 
we need to construct different transition maps for $\det(\Kk)$, which will now depend on the section $s_\Kk$.
For this purpose, we
again use contraction isomorphisms from Lemma~\ref{lem:get}.  
On the one hand, this provides families of isomorphisms
\begin{equation}\label{Cds}
\mathfrak{C}_{\rd_x s_I} \,:\; \lm \rT_x U_I \otimes \bigl( \lm E_I \bigr)^* \;\overset{\cong}{\longrightarrow}\;  \det(\rd_x s_I)
\qquad\text{for} \; x\in U_I .
\end{equation}
In fact, as we will see in the proof of Proposition~\ref{prop:orient} below, these maps are essentially the special cases of $T_{I,x}$ in \eqref{eq:TIx} in which $R_I$ is surjective.
On the other hand, recall that the tangent bundle condition \eqref{tbc} implies that $\rd s_J$ restricts to an isomorphism $\qu{\rT_y U_J}{\rd_x\phi_{IJ}(\rT_x U_I)}\overset{\cong}{\to} \qu{E_J}{\Hat\phi_{IJ}(E_I)}$ for $y=\phi_{IJ}(x)$. 
Therefore, if we choose a smooth normal bundle $N_{IJ} =\bigcup_{x\in U_{IJ}} N_{IJ,x}
\subset \phi_{IJ}^* \rT U_J
$ to 
the submanifold 
$\im \phi_{IJ}
\subset U_J$, then 
the subspaces 
$\rd_y s_J(N_{IJ,x})$ 
(where we always denote $y:=\phi_{IJ}(x)$ and vary $x\in U_{IJ}$)
form a smooth family of subspaces of $E_J$ that are complements  
to
$\Hat\phi_{IJ}(E_I)$.
Hence letting 
$\pr_{N_{IJ}}(x) : E_J \to \rd_y s_J(N_{IJ,x}) \subset E_J$
be the 
smooth family of 
projections with kernel  
$\Hat\phi_{IJ}(E_I)$,
we obtain 
a smooth family of 
linear maps
$$
F_x \,:= \;
\pr_{N_{IJ}}(x) 
\circ \rd_y s_J \,:\; \rT_y U_J \;\longrightarrow\; 
E_J
\qquad\text{for}\; x\in U_{IJ}, y=\phi_{IJ}(x)
$$
with images $\im F_x=\rd_y s_J(N_{IJ,x})$,
and isomorphisms to their kernel
$$
\phi_x  \,:= \;  \rd_x\phi_{IJ} \,:\; \rT_x U_I \;\overset{\cong}{\longrightarrow}\;  \ker F_x =  \rd_x\phi_{IJ}(\rT_x U_I) \;\subset\; \rT_y U_J .
$$
By Lemma~\ref{lem:get} these induce isomorphisms 
$$
\mathfrak{C}^{\phi_x}_{F_x} \,:\; \lm \rT_{\phi_{IJ}(x)} U_J \otimes \bigl(\lm E_J \bigr)^* 
 \;\overset{\cong}{\longrightarrow}\;  \lm \rT_x U_I \otimes  \Bigl(\lm \Bigl(\qq{E_J}
{\im F_x}
 \Bigr) \Bigr)^* .
$$
We may combine this with the dual of the  isomorphism $\lm \bigl(\qu{E_J}
{\rd_y s_J(N_{IJ,x})}
\bigr) \cong \lm E_I$ induced 
via \eqref{eq:laphi} by $\pr^\perp_{
N_{IJ}}(x) \circ\Hat\phi_{IJ} : E_I \to \qu{E_J}
{\rd_y s_J(N_{IJ,x})}
$ to obtain isomorphisms
\begin{align} \label{CIJ}
\mathfrak{C}_{IJ}(x) \,: \;
 \lm \rT_{\phi_{IJ}(x)} U_J \otimes \bigl(\lm E_J \bigr)^*  
\;\overset{\cong}{\longrightarrow}\;  \lm \rT_x U_I \otimes \bigl(\lm E_I \bigr)^*  
\end{align}
for $x\in U_{IJ}$, given by $\mathfrak{C}_{IJ}(x) :=  \bigl( \id_{ \lm \rT_x U_I} \otimes 
\La_{(\pr^\perp_{
N_{IJ}}(x)\circ\Hat\phi_{IJ})^{-1}}^*  \bigr) \circ \mathfrak{C}^{\phi_x}_{F_x}$.

\begin{prop}\label{prop:orient} 
\begin{enumerate}
\item 
Let $\Kk$ be a weak Kuranishi atlas. 
Then there is a well defined line bundle $\det(\Kk)$ over $\Kk$ given by the line bundles $\La^\Kk_I := \lm \rT U_I\otimes \bigl(\lm E_I\bigr)^* \to U_I$ for $I\in\Ii_\Kk$ and the transition maps $\mathfrak{C}_{IJ}^{-1}: \La^\Kk_I |_{U_{IJ}} \to \La^\Kk_J |_{\im\phi_{IJ}}$ from \eqref{CIJ} for $I\subsetneq J$. 
In particular, the latter isomorphisms are independent of the choice of 
normal bundle $N_{IJ}$.

Furthermore, the contractions $\mathfrak{C}_{\rd s_I}: \La^\Kk_I \to \det(\rd s_I)$ from \eqref{Cds} define an isomorphism $\Psi^{s_\Kk}:=\bigl(\mathfrak{C}_{\rd s_I}\bigr)_{I\in\Ii_\Kk}$ from $\det(\Kk)$ to $\det(s_\Kk)$.
\item 
If $\Kk$ is a weak Kuranishi cobordism, then the determinant bundle $\det(\Kk)$ defined as in (i)  
can be given a product structure on the collar 
such that its 
boundary restrictions are
$\det(\Kk)|{_{\p^\al\Kk}} = \det(\p^\al\Kk)$ for $\al= 0,1$. 

Further, the isomorphism $\Psi^{s_\Kk}: \det(\Kk) \to \det(s_\Kk)$ defined as in (i) has product structure on the collar with restrictions $\Psi^{s_\Kk}|_{\p^\al \Kk}=\Psi^{s_{\p^\al\Kk}}$ for $\al=0,1$.
\end{enumerate}
\end{prop}

\begin{proof}
To begin, note that each $\La^\Kk_I = \lm \rT U_I \otimes \bigl(\lm E_I\bigr)^*$ is a smooth line bundle over $U_I$, since it inherits local trivializations 
from the tangent bundle $\rT U_I\to U_I$.

Next, we will show that the isomorphisms $\mathfrak{C}_{\rd_x s_I}$ from \eqref{Cds} are smooth in this trivialization, where $\det(\rd s_I)$ is trivialized via the maps $\Hat T_{I,x}$ as in Proposition~\ref{prop:det0} using an isomorphism $R_I: \R^{N}\to E_I$. For that purpose we introduce the isomorphisms
$$
G_x: \rT_x U_I\to \ker(\rd_x s_I \oplus  R_I),\qquad   v\mapsto \bigl(v, -R_I^{-1}(\rd_x s_I(v)\bigr),
$$  
and claim that the associated maps on determinant lines fit
into a commutative diagram
with $\mathfrak{C}_{\rd_x s_I}$ and the version of the trivialization $T_{I,x}$ from \eqref{eq:TIx}
\begin{equation}\label{ccord}
\xymatrix{  
\lm  \rT_x U_I \otimes \bigl( \lm E_I \bigr)^*  
\ar@{->}[d]_{\La_{G_x}\otimes \id}
 \ar@{->}[r]^{ \qquad\quad\mathfrak{C}_{\rd_x s_I}}  
 &
\det(\rd_x s_I)  \ar@{->}[d]^{\id} \\
\lm  \ker (\rd_x s_I \oplus R_I) \otimes \bigl(\lm E_I \bigr)^* \ar@{->}[r]^{ \qquad\qquad\quad\;\; T_{I,x}} 
& \det (\rd_x s_I).
}
\end{equation}
Here the trivialization $\Hat T_{I,x}$ of $\det(\rd_x s_I)$ is given by precomposing $T_{I,x}$ with the $x$-independent isomorphism $\La_{R_I^{-1}}^* : \bigl( \lm \R^N \bigr)^* \to \bigl(\lm E_I\bigr)^*$ in the second factor.
Since $G_x$ varies smoothly with $x$ in any trivialization of $\rT U$, this will prove that $\mathfrak C_{\rd s_I}$ is smooth with respect to the given trivializations.

To prove \eqref{ccord} we use the explicit formulas from Lemma~\ref{lem:get} and \eqref{eq:HatTx} at a fixed $x\in U_I$. 
So let $v_1,\ldots,v_n$ be a basis for $\rT_x U_I$ with ${\rm span}(v_1,\ldots,v_k)=\ker \rd_x s_I$, and let $w_1,\dots, w_N$ be a basis for $E_I$ with $w_{N-n+i}=\rd_x s_I(v_i)$ for $i=k+1,\ldots,n$.
Then $\ov v_i:=G_x(v_i)=\bigl(v_i,-R_I^{-1}(\rd_x s_I(v_i)\bigr) $ is a corresponding basis of $\ker(\rd_x s_I \oplus R_I)$. 
In this setting we can verify \eqref{ccord},
\begin{align*}
\mathfrak{C}_{\rd_x s_I} \bigl( (v_1\wedge\dots v_n)\otimes(w_1\wedge\dots w_N)^* \bigr)
&\;=\;
\bigl(v_1\wedge\dots v_k\bigr)\otimes \bigl( [w_1]\wedge\dots [w_{N-n+k}] \bigr)^* \\
&\;=\; T_{I,x}\bigl( (\ov v_1\wedge \ldots  \ov v_n)\otimes(w_1\wedge\dots w_N)^* \bigr) \\
&\;=\; T_{I,x}\bigl( \La_{G_x}(v_1\wedge \ldots  v_n)\otimes(w_1\wedge\dots w_N)^* \bigr).
\end{align*}
This proves the smoothness of the isomorphisms $\mathfrak C_{\rd s_I}$
so that we can
define preliminary transition maps 
\begin{equation}\label{tiphi}
\Ti\phi_{IJ}:= \mathfrak{C}_{\rd s_J}^{-1} \circ \La_{IJ} \circ \mathfrak{C}_{\rd s_I}
\,:\; \La^\Kk_I|_{U_{IJ}}\to \La^\Kk_J \qquad\text{for}\; I\subsetneq J \in \Ii_\Kk
\end{equation}
by the transition maps \eqref{eq:bunIJ} of $\det(s_\Kk)$ and the isomorphisms \eqref{Cds}.
These define a line bundle $\La^\Kk:=\bigl(\La^\Kk_I, \Ti\phi_{IJ} \bigr)_{I,J\in\Ii_\Kk}$ since the weak cocycle condition follows directly from that for the $\La_{IJ}$. Moreover, this automatically makes the family of bundle isomorphisms $\Psi^\Kk:=\bigl(\mathfrak{C}_{\rd s_I}\bigr)_{I\in\Ii_\Kk}$ an isomorphism from $\La^\Kk$ to $\det(s_\Kk)$. 
It remains to see that $\La^\Kk=\det(\Kk)$ and $\Psi^\Kk=\Psi^{s_\Kk}$, i.e.\ we claim equality of transition maps $\Ti\phi_{IJ}=\mathfrak{C}_{IJ}^{-1}$. This also shows that $\mathfrak{C}_{IJ}^{-1}$ and thus $\det(\Kk)$ is independent of the choice of 
normal bundle $N_{IJ}$ in \eqref{CIJ}.
So to finish the proof of (i), it suffices to establish the following commuting diagram at a fixed $x\in U_{IJ}$ with $y=\phi_{IJ}(x)$,
\begin{equation}\label{cclaim}
\xymatrix{
\lm  \rT_x U_I \otimes \bigl( \lm E_I \bigr)^* 
 \ar@{->}[r]^{ \qquad\quad\mathfrak{C}_{\rd_x s_I}}  
 &
\det(\rd_x s_I)  \ar@{->}[d]^{\La_{IJ}
(x)} \\
 \lm \rT_y U_J \otimes \bigl( \lm E_J \bigr)^*
 \ar@{->}[r]^{\qquad\quad\mathfrak{C}_{\rd_y s_J}}  
 \ar@{->}[u]^{\mathfrak{C}_{IJ}
 (x)} 
&
\det(\rd_y s_J) .
}
\end{equation}
Using \eqref{ccord}, for surjective maps $R_I:\R^N\to E_I$ and $R_J:\R^{N'}\to E_J$, and the compatibility of the trivialization $\Hat T_{J,y}$ with $R'_J:=\Hat T'_{J,y}$ arising from $\Hat\phi_{IJ}:\R^N\to E_J$, we can expand this diagram to 
 \[
 \xymatrix{
\lm  \rT_x U_I \otimes \bigl( \lm E_I \bigr)^* 
 \ar@{->}[r]^{\La_{G_x}\otimes \id \qquad}   
 & \;\; \lm  \ker (\rd_x s_I\oplus R_I) \otimes  \bigl( \lm E_I \bigr)^*  \ar@{->}[r]^{ \qquad\qquad\qquad T_{I,x} } 
 \ar@{->}[d]_{ \La_{\rd_x \phi_{IJ} \oplus \id_{\R^N}} \otimes ( \La_{R'_J\,\!^{-1}}^* \circ \La_{R_I}^*) }&
\det(\rd_x s_I)  \ar@{->}[d]^{\La_{IJ}
(x)} \\
 & \;\; \lm  \ker (\rd_y s_J\oplus R'_J) \otimes  \bigl( \lm \im R'_J \bigr)^*  \ar@{->}[d]_{\Psi_y\otimes  ( \La_{R_J^{-1}}^* \circ \La_{R'_J}^*) } \ar@{->}[r]^{ \qquad\qquad\qquad T'_{J,y}} 
&
\det(\rd_y s_J) \\
 \lm \rT_y U_J \otimes \bigl( \lm E_J \bigr)^* \ar@{->}[uu]^{\mathfrak{C}_{IJ}(x)} 
 \ar@{->}[r]^{\La_{G_y}\otimes \id \qquad }  
 & \;\; \lm  \ker (\rd_y s_J\oplus R_J) \otimes  \bigl( \lm E_J \bigr)^*  \ar@{->}[r]^{ \qquad\qquad\qquad T_{J,y}}
&
\det(\rd_y s_J)   \ar@{->}[u]_{\id} .
}
\]
Here the upper right square commutes by \eqref{eq:etrans}. To make the lower right square precise, and in particular to choose suitable $R_J$, we note that $E_J=\rd_y s_J + \im R'_J$ 
and $\rd_y s_J (\rT_y\im\phi_{IJ}) \subset \Hat\phi_{IJ}(E_I) =  \im R'_J$, so that given any normalized basis $e_1,\ldots,e_N\in \R^N$ we can complete the corresponding vectors $R'_J(e_i)$ to a basis for
$E_J$ by adding the vectors $\rd_y s_J(v^\Psi_{N+1}), \ldots, \rd_y s_J(v^\Psi_{N'})$,
where
$v^\Psi_{N+1}, \ldots, v^\Psi_{N'}\in N_{IJ,x}$ is a basis of the normal space $N_{IJ,x}\subset \rT_y U_J$ to $\rT_y\im\phi_{IJ}$ that was used to define $\mathfrak{C}_{IJ}(x)$.
Thus $R_J'$ extends  to a smooth family of bijections
\begin{align*}
R_J: = R_{J,x} 
 \,:\quad \R^N\times \R^{N'-N}  &\; \overset{\cong}{\longrightarrow}\;  \Hat\phi_{IJ}(E_I) \oplus 
 \rd s_y(N_{IJ,x})
 \;=\; E_J , \\
 (\ul r ; r_{N+1},\ldots,r_{N'}) &\;\longmapsto \;  \Hat\phi_{IJ}\bigl( R_I(\ul r) \bigr) + {\textstyle \sum_{i=N+1}^{N'}} r_i \cdot \rd_y s_J(v^\Psi_i) .
\end{align*}
We may choose the vectors $v^\Psi_{i}$ so that 
the $e_{i}:= R_J^{-1}\bigl(\rd_y s_J(v^\Psi_{i})\bigr)$ for $i=N+1,\ldots N'$ extend $e_1,\ldots e_N \in \R^N\cong \R^N \times\{0\}$ to a normalized basis of $\R^{N'}$.  Further, the vectors
 $\ov v^\Psi_{i}:= \bigl( v^\Psi_i , - e_i \bigr)$ span the complement of the embedding $\io: \ker (\rd_y s_J\oplus R'_J) \hookrightarrow  \ker (\rd_y s_J\oplus R_J)$, $(v,\ul r) \mapsto (v,\ul r,0)$. 
Hence \eqref{tpsit} (with $\la(x)=1$) shows that the isomorphism 
\begin{align*}
\Psi_y \,:\;  \lm \ker (\rd_y s_J\oplus R'_J) &\; \overset{\cong}{\longrightarrow}\;  \lm \ker (\rd_y s_J\oplus R_J) ,\\ 
\ov v_1 \wedge\ldots \wedge \ov v_n  &\;\longmapsto \;  
 \io(\ov v_1)\wedge\ldots \io(\ov v_n) \wedge \ov v^\Psi_{N+1} \wedge \ldots \ov v^\Psi_{N'}
\end{align*}
intertwines the trivializations $T'_{J,y}$ and $T_{J,y}$, that is the lower right square in the above diagram commutes. 

Now to prove that the entire diagram commutes it remains to identify $\mathfrak{C}_{IJ}(x)$ with the map given by composition of the other isomorphisms,  which is
 the tensor product of $\La_{R_I^{-1}}^*\circ \La_{R_J}^*$ (composed via $\lm \R^N\cong \R\cong\lm \R^{N'}$) on the obstruction spaces with the inverse of  
\begin{align*}
\lm \rT_x U_I &\;\overset{\cong}{\longrightarrow}\;  \lm \rT_y U_J   \\
v_1 \wedge\ldots \wedge v_n 
  &\;\longmapsto \;  
 \rd_x\phi_{IJ}(v_1) \wedge\ldots \rd_x\phi_{IJ} (v_{n}) \wedge v^\Psi_{N+1} \wedge \ldots v^\Psi_{N'} .
\end{align*}
Here we used the fact that $\io \circ \bigl(\rd_x\phi_{IJ}\oplus\id_{\R^N}\bigr) \circ G_x = G_y \circ \rd_x\phi_{IJ}$ and $\ov v^\Psi_{i}= G_y( v^\Psi_i )$.
Note moreover that we chose the vectors $v^\Psi_{i}\in \rT_y U_J$ to span the complement of $\im\rd_x\phi_{IJ}$, and hence $\rd_x\phi_{IJ}(v_1),\ldots,\rd_x\phi_{IJ} (v_{n}), v^\Psi_{N+1},\ldots ,v^\Psi_{N'}$ forms a basis of $ \rT_y U_J$. 
Moreover, note that $w_i=R_J(e_i)$ is a basis for $E_J$ whose last $N'-N$ vectors are $w_{i}=\rd_y s_J(v^\Psi_i)\in 
\rd_y s_J(N_{IJ,x})$
for $i=N+1,\ldots,N'$. In these bases the explicit formulas \eqref{CIJ} and \eqref{Cfrak} give
\begin{align*}
\mathfrak{C}_{IJ}(x) \;:\; &  \bigl( \rd_x\phi_{IJ}(v_1) \wedge\ldots \rd_x\phi_{IJ} (v_{n}) \wedge v^\Psi_{N+1} \wedge \ldots v^\Psi_{N'}\bigr) \otimes \bigl(R_J(e_1)\wedge\dots R_J(e_{N'})\bigr)^*  \\
&\;\mapsto\;
(v_1\wedge\dots v_n)\otimes  \La_{(\pr^\perp_{
N_{IJ}}\circ\Hat\phi_{IJ})^{-1}}^*  \bigl( [\Hat\phi_{IJ}(R_I(e_1))]\wedge\dots [\Hat\phi_{IJ}(R_I(e_N))] \bigr)^* \\
&\;=\;
(v_1\wedge\dots v_n)\otimes  \bigl( R_I(e_1) \wedge\dots R_I(e_N) \bigr)^* .
\end{align*}
Here in the second factor we have $\La_{R_I^{-1}}\bigl( R_I(e_1) \wedge\dots R_I(e_N) \bigr) =1\in \lm\R^N$
and $\La_{R_J^{-1}}\bigl( R_J(e_1) \wedge\dots R_J(e_{N'}) \bigr)= 1 \in \lm\R^{N'}$, so this proves that  \eqref{cclaim} commutes.

For part (ii) the same arguments apply to define a bundle $\det(\Kk)$ and isomorphism $\Psi^{s_\Kk}$, for which it remains to establish the product structure on a collar. However, we may use the isomorphisms $\mathfrak{C}_{\rd s_I}:\La^\Kk_I \to \det(\rd s_I)$ and $\mathfrak{C}_{\rd s_I|_{\p^\al U_I}}:\La^{\p^\al \Kk}_I \to \det(\rd s_I|_{\p^\al U_I})$ to pull back the isomorphisms $\ti\io^\al_I : \det(\rd s_I)|_{\im\io^\al_I} \to \det(\rd s_I|_{\p^\al U_I}) \times A^\al_\eps$ from Proposition~\ref{prop:det0} to isomorphisms
$$
\bigl( \mathfrak{C}_{\rd s_I|_{\p^\al U_I}} \times \id_{A^\al_\eps} \bigr) ^{-1}  \circ \ti\io^\al_I  \circ \mathfrak{C}_{\rd s_I} \,:\;
\La^\Kk_I |_{\im\io^\al_I} \;\overset{\cong}{\longrightarrow}\; \La^{\p^\al\Kk}_I \times A^\al_\eps .
$$
This provides the product structure for $\det(\Kk)$, and moreover this construction was made such that $\Psi^{s_\Kk}$ has product form in the same collar, and the restrictions are, as claimed, given by pullback of the restrictions of $\det(s_\Kk)$. This completes the proof.
\end{proof}

\begin{prop}\label{prop:orient1} 
\begin{enumerate}
\item
Let $(\Kk,\si)$ be an oriented, tame Kuranishi atlas with reduction $\Vv$, and  let $\nu$ be an admissible, precompact, transverse perturbation of $s_\Kk|_\Vv$.  Then the zero set $|\bZ_\nu|$ is an oriented closed manifold.
\item
Let $(\Kk^{[0,1]},\si)$ be an 
oriented,  
tame Kuranishi cobordism with reduction $\Vv$ and admissible, precompact, transverse perturbation $\nu$. Then the corresponding zero set $|\bZ_{\nu}|$ is an oriented cobordism from $|\bZ_{\nu^0}|$ to $|\bZ_{\nu^1}|$ for $\nu^\al:=\nu|_{\partial^\al \Vv}$, with boundary orientations induced as in (i) by  $\si^\al:=\partial^\al \si$.
\end{enumerate}
\end{prop}
\begin{proof} 
We first show that the local zero sets $Z_I: = (s|_{V_I} + \nu)^{-1}(0)\subset V_I$  have a natural orientation.
By Lemma~\ref{le:stransv} they are submanifolds, and by transversality we have $\im (\rd_z s_I + \rd_z\nu_I)=E_I$ for each $z\in Z_I$, and thus $\lm \qu{E_I}{\im (\rd_z s_I + \rd_z\nu_I)}= \lm\{0\} = \R$, so that we have a natural isomorphism between the orientation bundle of $Z_I$ and the restriction of the determinant line bundle
\begin{align*}
\lm \rT Z_I &\;=\; {\textstyle \bigcup_{z\in Z_I}} \lm \ker (\rd_z s_I + \rd_z\nu_I) \\
&\; \cong\;  {\textstyle \bigcup_{z\in Z_I}} \lm \ker (\rd_z s_I + \rd_z\nu_I) \otimes \R 
\;=\; \det(s_I|_{V_I}+\nu_I)|_{Z_I} .
\end{align*}
Combining this with Proposition~\ref{prop:orient} and Lemma~\ref{lem:get} we obtain isomorphisms
$$
\mathfrak{C}^\nu_I(z) := \mathfrak{C}_{\rd (s_I +\nu_I)} \circ \mathfrak{C}_{\rd s_I}^{-1} \,:\;
\det(\rd s_I)|_{z} \;\longrightarrow\; \lm \rT_z Z_I  \qquad \text{for} \; z\in Z_I .
$$
To see that these are smooth,
recall that smoothness of $\mathfrak{C}_{\rd s_I}$ was proven in Proposition~\ref{prop:orient}.
The same arguments apply to $\mathfrak{C}_{\rd (s_I +\nu_I)}$.
Further, for $I\subsetneq J$ and $z\in Z_I\cap U_{IJ}$ these
isomorphisms 
are intertwined by the transition maps 
$$
\La_{IJ}(z)=\La_{\rd_z\phi_{IJ}} \otimes \bigl( \La_{\Hat\phi_{IJ}^{-1}}\bigr)^* : \det(\rd_z s_I) \to \det(\rd_{\phi_{IJ}(z)} s_J)
$$
and $\La_{\rd_z\phi_{IJ}} : \lm \rT_z Z_I \to \lm \rT_{\phi_{IJ}(z)} Z_J$.
To see this, one combines the commuting diagram \eqref{cclaim} with the analogous 
diagram over $Z_I\cap U_{IJ}$
\[
\xymatrix{
\lm  \rT U_I \otimes \bigl( \lm E_I \bigr)^* 
 \ar@{->}[rr]^{ \mathfrak{C}_{\rd (s_I+\nu_I)}\qquad}  
 & &
\det(\rd (s_I+\nu_I)) = \lm \rT Z_I \otimes \R
 \ar@{->}[d]^{\La_{IJ} = \La_{\rd\phi_{IJ}} \otimes \id_\R } \\
 \lm \rT U_J \otimes \bigl( \lm E_J \bigr)^*
 \ar@{->}[rr]^{\mathfrak{C}_{\rd (s_J+\nu_J)}\qquad}  
 \ar@{->}[u]^{\mathfrak{C}_{IJ}} 
& &
\det(\rd (s_J+\nu_J)) = \lm \rT Z_J \otimes \R.
}
\]
The latter diagram commutes by the arguments in Proposition~\ref{prop:orient} applied to $s_\bullet+\nu_\bullet$ because $\rd s_J$ and $\rd (s_J+\nu_J)$ induce the same map 
$\mathfrak{C}_{IJ}(z)$   
for $z\in Z_I\cap U_{IJ}$.
 Indeed, the admissibility of $\nu$ implies that  
$\im\rd_{\phi_{IJ}(z)} \nu_J\subset \Hat\phi_{IJ}(E_I)$ so that $F_z = \pr_{N_{IJ}}(z) \circ \rd_{\phi_{IJ}(z)} s_J =\pr_{N_{IJ}}(z) \circ \rd_{\phi_{IJ}(z)} (s_J+\nu_J)$ in the construction of $\mathfrak{C}_{IJ}(z)$.
Now the orientation $\bigl( \si_I : U_I \to \det(\rd s_I) \bigr)_{I\in\Ii_\Kk}$ of $\Kk$ induces nonvanishing sections $\si^\nu_I := \mathfrak{C}^\nu_I \circ \si_I : Z_I \to \lm \rT Z_I$ which, by the above discussion and the compatibility $\La_{IJ}\circ \si_I = \si_J$ are related by $\La_{\rd\phi_{IJ}} \circ\si^\nu_I = \si^\nu_J$, i.e.\ the orientations $\si^\nu_I$ in the charts of $|\bZ_\nu|$ are compatible with the transition maps $\phi_{IJ}|_{Z_I\cap U_{IJ}}$.
Hence this determines an orientation of $|\bZ_\nu|$. This proves~(i).

For a Kuranishi cobordism, the above constructions provide an orientation $|\si^\nu| : |\bZ_\nu| \to \lm \rT |\bZ_\nu|$ on the manifold with boundary $|\bZ_\nu|$. Moreover, Lemma~\ref{le:czeroS0} provides diffeomorphisms to the boundary components for $\al=0,1$
$$
|j^\al| \,:\;  |\bZ_{\nu^\al}| \;=\;   \Bigl| \underset{{I\in\Ii_{\Kk^\al}}}{\textstyle\bigcup} 
 (s_I+\nu^\al_I)^{-1}(0)  \Bigr| 
 \;\overset{\cong}{\longrightarrow}\; 
 \partial^\al |\bZ_{\nu}|  \;\subset \; \partial |\bZ_{\nu}|   ,
$$
which in the charts are given by $j^\al_I := \iota^\al_I (\cdot,\al) :  Z^\al_I  \to  Z_I$.
The latter lift to isomorphisms of determinant line bundles
$$
{\ti j}^\al_I := \La_{\rd\io^\al_I}\circ \wedge_1 \,:\; 
\lm \rT Z^\al_I \;\overset{\cong}{\longrightarrow}\; \lm \rT Z_I |_{\im\io^\al_I} ,
$$
given by the same expression as the restriction to $Z^\al_I\times \{\al\}$ of the map \eqref{orient map} on $\det(\rd s^\al_I)\times A^\al_\eps$ in the case of trivial cokernel.
These are the expressions in the charts of an isomorphism of determinant line bundles
$$
|\ti j^\al|  := \La_{\rd |\io^\al|}\circ \wedge_1 \,:\;  \lm \rT |\bZ_{\nu^\al}|  \;\overset{\cong}{\longrightarrow}\; \lm \rT |\bZ_\nu| \bigr|_{|j^\al|( |\bZ_{\nu^\al}|) } ,
$$
which consists of the isomorphism induced by the collar neighbourhood embedding $|\io^\al| :  |\bZ_{\nu^\al}|\times A^\al_\eps  \to  |\bZ_\nu|$ given by $\io^\al_\eps|_{Z^\al_I \times A^\al_\eps}$ in the charts together with 
the canonical isomorphism between the determinant line bundle of the boundary and the boundary restriction of the determinant line bundle of the collar neighbourhood,
$$
\wedge_1 \,:\; \lm \rT |\bZ_{\nu^\al}| \;\to\;  \lm \rT  \bigl( |\bZ_{\nu^\al}| \times A^\al_\eps \bigr) \big|_{|\bZ_{\nu^\al}| \times \{\al\}}
\;=\;  \bigl( \lm \rT  |\bZ_{\nu^\al}|  \bigr) \times \R
, \quad
\eta \mapsto  
1\wedge\eta .
$$
Here, as before, we identify vectors $\eta_i \in \rT |\bZ_{\nu^\al}|$ with $(\eta_i ,0) \in \rT |\bZ_{\nu^\al}| \times \R$ and abbreviate $1:=(0,1)  \in \rT |\bZ_{\nu^\al}| \times \R$.
The latter corresponds to the exterior normal $\rd |\io^1| (0,1) \in \rT |\bZ_\nu| \big|_{\p^1 |\bZ_\nu|}$ and the interior normal $\rd |\io^0| (0,1) \in \rT |\bZ_\nu| \big|_{\p^0 |\bZ_\nu|}$. Hence the boundary orientations\footnote
{
Here, contrary to the special choice for Kuranishi cobordisms discussed in Remark~\ref{rmk:orientb}, 
we use a more standard orientation convention for manifolds with boundary. Namely, a positively ordered basis $\eta_1,\dots,\eta_k$ for the tangent space to the boundary is extended to a positively ordered basis $\eta_{out},\eta_1,\dots,\eta_k$ for the whole manifold by adjoining an outward unit vector $\eta_{out}$ as its first element.
 }
$\p^\al |\si^\nu| :|\bZ_{\nu^\al}| \to \lm \rT |\bZ_{\nu^\al}|$ induced from $|\si^\nu|$ on the two components for $\al=0,1$ differ by a sign,
$$
\p^0 |\si^\nu| := - \, |\ti j^0|^{-1} \circ  |\si^\nu| \circ |j^0| ,\qquad
\p^1 |\si^\nu| :=  |\ti j^1|^{-1} \circ  |\si^\nu| \circ |j^1| .
$$
On the other hand, the restricted orientations $\si^\al:=\partial^\al \si$ of $\Kk^\al$ also induce orientations $|\si^{\nu_\al}|$ of the boundary components $|\bZ_{\nu^\al}|$ by the construction in (i). 
Now to prove the claim that $|\bZ_\nu|$ is an oriented cobordism from $\bigl(|\bZ_{\nu^0}|,|\si^{\nu_0}|\bigr)$ to $\bigl(|\bZ_{\nu^1}|,|\si^{\nu_1}|\bigr)$, it remains to check that $\p^0 |\si^\nu| = - |\si^{\nu_0}|$ and $\p^1 |\si^\nu| = |\si^{\nu_1}|$. This is equivalent to $|\si^{\nu^\al}| = |\ti j^\al|^{-1} \circ  |\si^\nu| \circ |j^\al|$ for both $\al=0,1$. So, recalling the construction of $\p^\al\si_I=(\ti\io^\al_I)^{-1} \circ \si_I \circ \io^\al_I\big|_{Z^\al_I \times\{\al\} }$ and $\si^\nu_I=\mathfrak{C}\nu_I \circ \si_I \big|_{Z_I}$ in local charts, we must show the following identity over
$Z^\al_I \times\{\al\}\cong Z^\al_I$ for all $I\in\Ii_{\p^\al\Kk}$
\begin{equation} \label{oclaim}
\mathfrak{C}_{\rd (s^\al_I +\nu^\al_I)} \circ \mathfrak{C}_{\rd s^\al_I}^{-1} \circ 
\bigl((\ti\io^\al_I)^{-1} \circ \si_I \circ \io^\al_I\bigr)
\; =\;  
(\ti j^\al_I)^{-1} \circ  \bigl( \mathfrak{C}_{\rd (s_I +\nu_I)} \circ \mathfrak{C}_{\rd s_I}^{-1} \circ \si_I \bigr) \circ j^\al_I  .
\end{equation}
We will check this at a fixed point $z\in Z^\al_I$ in two steps. We first show that the contraction isomorphisms $\mathfrak{C}_{\rd s^\al_I}(z)$ and $\mathfrak{C}_{\rd s_I}(\io^\al_I(z,\al))$ intertwine the collar isomorphism $\ti\io^\al_I = \bigl( \La_{\rd \io^\al_I} \circ \wedge_1 \bigr) \otimes \bigl(\La_{\id_{E_I}}\bigr)^*$ from $\det(\rd s^\al_I)$ to $\det(\rd s_I)$ with the analogous collar isomorphism
$$
\Ti I^\al_I  := \bigl( \La_{\rd \io^\al_I} \circ \wedge_1 \bigr) \otimes \bigl(\La_{\id_{E_I}}\bigr)^*
\,:\; \lm \rT \partial^\al U_I \otimes \bigl( \lm E_I \bigr)^* 
\;\to\;
\lm \rT U_I \otimes \bigl( \lm E_I \bigr)^* .
$$
Indeed, we can use the product form of $s_I$ in terms of $s^\al_I$ to check the corresponding identity of maps 
$\lm \rT \partial^\al U_I \otimes \bigl( \lm E_I \bigr)^* \big|_{Z^\al_I}\to \det(\rd s_I)\big|_{\io^\al_I(Z^\al_I)}$ at a fixed vector of the form $\bigl( \eta_{\ker} \wedge \eta_{\perp} \bigr) \otimes  \bigl
( \zeta_{\coker} \wedge \zeta_{\im} \bigr)$ 
with $\eta_{\ker} \in \lm\ker\rd s^\al_I$ and $\zeta_{\im} \in \bigl(\lm \im\rd s^\al_I \bigr)^*$: 
\begin{align*}
&
\Bigl( \ti\io^\al_I \circ \mathfrak{C}_{\rd s^\al_I} \Bigr)
\Bigl( \bigl( \eta_{\ker} \wedge \eta_{\perp} \bigr) \otimes  \bigl(
\zeta_{\coker} \wedge \zeta_{\im}
\bigr) \Bigr)  \\
&\;=\;
\ti\io^\al_I\Bigl(
\mathfrak{c}_{\rd s^\al_I}\bigl( \eta_{\perp} ,  \zeta_{\im}   \bigr) \cdot 
\eta_{\ker} \otimes \zeta_{\coker} \Bigr)\\
&\;=\;
\mathfrak{c}_{\rd s^\al_I}\bigl( \eta_{\perp} ,  \zeta_{\im}   \bigr) \cdot 
\La_{\rd \io^\al_I}(1\wedge \eta_{\ker}) \otimes \zeta_{\coker} \\
&\;=\; 
\mathfrak{c}_{\rd s_I}\bigl( \La_{\rd \io^\al_I} (\eta_{\perp}) , \zeta_{\im}\bigr) \cdot 
 \La_{\rd \io^\al_I} \bigl( 1\wedge \eta_{\ker}  \bigr) \otimes  \zeta_{\coker} \\
&\;=\; 
\mathfrak{C}_{\rd s_I} \Bigl( \bigl( \La_{\rd \io^\al_I}(1\wedge \eta_{\ker} \bigr) 
\wedge  \La_{\rd \io^\al_I}(\eta_{\perp}) \bigr) \otimes  \bigl( \zeta_{\coker} \wedge \zeta_{\im}
 \bigr) \Bigr) \\
&\;=\; 
\mathfrak{C}_{\rd s_I} \Bigl( \La_{\rd \io^\al_I} 
\bigl( 1\wedge \eta_{\ker} \wedge \eta_{\perp} \bigr) \otimes 
 \bigl( \zeta_{\coker} \wedge \zeta_{\im}
  \bigr) \Bigr) \\
&\;=\;  
\Bigl( \mathfrak{C}_{\rd s_I} \circ \Ti I^\al_I   \Bigr) \Bigl( \bigl( \eta_{\ker} \wedge \eta_{\perp} \bigr) \otimes  \bigl(
\zeta_{\coker} \wedge \zeta_{\im} \bigr) \Bigr). 
\end{align*}
Secondly we check that the contraction isomorphisms for the surjective maps 
$\rd_{\io^\al_I(z,\al)}(s_I +\nu_I)$  and $\rd_z (s^\al_I +\nu^\al_I)$ intertwine $\Ti I^\al_I (z)$ with the boundary isomorphism ${\ti j}^\al_I = \La_{\rd\io^\al_I}\circ \wedge_1$ from $\lm \rT_z Z^\al_I$ to $\lm \rT_{\io^\al_I(z,\al)}Z_I $.
For that purpose we also use the product form of $\nu_I$ in terms of $\nu^\al_I$ to check the corresponding identity of maps $\lm \rT U_I \otimes \bigl( \lm E_I \bigr)^* \big|_{\io^\al_I(Z^\al_I)}  \to \lm \rT Z^\al_I$ at a fixed vector of the form $\bigl( \La_{\rd\io^\al_\eps}(1\wedge \eta_{\ker}) \wedge  \La_{\rd\io^\al_\eps}
(\eta_{\perp}) \bigr) \otimes  \zeta$ with  $\eta_{\ker} \in \lm\ker\rd (s^\al_I + \nu^\al_I)$: 
\begin{align*}
&  \Bigl( (\ti j^\al_I)^{-1} \circ  \mathfrak{C}_{\rd (s_I +\nu_I)} \Bigr) \Bigl( \bigl( \La_{\rd\io^\al_\eps}(1\wedge \eta_{\ker}) \wedge  \La_{\rd\io^\al_\eps}(\eta_{\perp}) \bigr) \otimes  \zeta \Bigr) 
  \\
&\;=\; 
\mathfrak{c}_{\rd (s_I +\nu_I)} \bigl(\La_{\rd\io^\al_\eps}( \eta_{\perp}) , \zeta \bigr)
 \cdot (\ti j^\al_I)^{-1} \bigl( \La_{\rd\io^\al_\eps}(1\wedge \eta_{\ker})\bigr) 
\\
&\;=\;  \mathfrak{c}_{\rd (s^\al_I +\nu^\al_I)} \bigl( \eta_{\perp} ,\zeta \bigr) \cdot
\eta_{\ker} 
\\
&\;=\;  \mathfrak{C}_{\rd (s^\al_I +\nu^\al_I)} \Bigl(  \bigl(\eta_{\ker} \wedge \eta_{\perp} \bigr) \otimes  \zeta   \Bigr)
\\
& \;=\; 
 \Bigl( \mathfrak{C}_{\rd (s^\al_I +\nu^\al_I)} \circ  \bigl(\Ti I^\al_I \bigr)^{-1} \Bigr) \Bigl( \bigl( \La_{\rd\io^\al_\eps}(1\wedge \eta_{\ker}) \wedge  \La_{\rd\io^\al_\eps}(\eta_{\perp}) \bigr) \otimes  \zeta \Bigr) .
\end{align*}
This proves \eqref{oclaim} and hence finishes the proof.
\end{proof}

\subsection{Construction of Virtual Moduli Cycle and Fundamental Class}\label{ss:VFC}  \hspace{1mm}\\ \vspace{-3mm}

We are finally in a position to prove Theorem~B. We begin with its first part, which defines the virtual moduli cycle (VMC) as a cobordism class of closed oriented manifolds. 
After a discussion of \v{C}ech homology, we then construct the virtual fundamental class (VFC) as \v{C}ech homology class.

\begin{thm}\label{thm:VMC1}
Let $X$ be a compact metrizable space.
\begin{enumerate}
\item
Let $\Kk$ be an oriented, additive weak Kuranishi atlas of dimension $D$ on $X$.
Then there exists a preshrunk tame shrinking $\Kk_{\rm sh}$ of $\Kk$, an admissible metric on $|\Kk_{\rm sh}|$,
a reduction $\Vv$ of $\Kk_{\rm sh}$, and a 
strongly adapted, admissible, precompact, transverse perturbation $\nu$ of $s_{\Kk_{\rm sh}}|_\Vv$. 
\item
For any choice of data as in (i), the perturbed zero set $|\bZ_\nu|$ is an oriented compact manifold (without boundary) of dimension $D$.
\item
Let $\Kk^0,\Kk^1$ be two oriented, additive weak Kuranishi atlases on $X$ that are oriented, additively cobordant. Then, for any choices of 
strongly adapted perturbations $\nu^\al$ as in (i) for $\al=0,1$, the perturbed zero sets are cobordant (as oriented closed manifolds),
$$
|\bZ_{\nu^0}|  \; \sim \;|\bZ_{\nu^1}|.
$$
\end{enumerate}
\end{thm}
\begin{proof}  
Proposition~\ref{prop:proper} provides a shrinking $\Kk'$ of $\Kk$, which is a tame Kuranishi atlas.
Then Proposition~\ref{prop:metric}, again using this theorem, provides a precompact tame shrinking $\Kk_{\rm sh}$ of $\Kk'$, in other words a preshrunk shrinking of $\Kk$, and equips it with an admissible metric. The orientation of $\Kk$ then induces an orientation of $\Kk_{\rm sh}$ by Lemma~\ref{le:cK}. 
Moreover, Proposition~\ref{prop:cov2}~(a) provides a reduction $\Vv$ of $\Kk_{\rm sh}$, and by Lemma~\ref{le:delred} we find another reduction $\Cc$ with precompact inclusion $\Cc\sqsubset \Vv$,
i.e.\ a nested reduction. 
Then we may apply Proposition~\ref{prop:ext} with $\si=\si_{\rm rel}(\de,\Vv\times[0,1],\Cc\times[0,1])$ to obtain a strongly adapted, admissible, transverse perturbation $\nu$ with $\pi_\Kk\bigl((s+\nu)^{-1}(0)\bigr) \subset \pi_{\Kk}(\Cc)$. This proves (i).

Part (ii) holds in this setting since Proposition~\ref{prop:zeroS0} shows that $|\bZ_\nu|$ is a smooth closed $D$-dimensional manifold, which is oriented by Proposition~\ref{prop:orient1}~(i). 

To prove (iii) we will use transitivity of the cobordism relation for oriented closed manifolds to prove increasing independence of choices in the following Steps 1--4.

\MS\NI
{\bf Step 1:} 
{\it
For a fixed oriented, metric, tame Kuranishi atlas $(\Kk,d)$, nested reductions $\Cc\sqsubset\Vv$, and $0<\de<\de_\Vv$, $0<\si\leq \si_{\rm rel}(\de,\Vv\times[0,1],\Cc\times[0,1])$, the cobordism class of $|\bZ_\nu|$ is independent of the choice of $(\Vv,\Cc,\de,\si)$-adapted perturbation~$\nu$.
}

\MS
To prove this we fix $(\Kk,d)$, $\Cc\sqsubset\Vv$, $\de$, and $\si$, consider two $(\Vv,\Cc,\de,\si)$-adapted perturbations $\nu^0,\nu^1$, and need to find an oriented cobordism $|\bZ_{\nu^0}|  \sim |\bZ_{\nu^1}|$.
For that purpose we apply Proposition~\ref{prop:ext2}~(ii) to the Kuranishi cobordism $\Kk\times[0,1]$ with product metric and nested product reductions $\Cc\times[0,1]\sqsubset \Vv\times[0,1]$ to obtain an admissible, precompact, transverse cobordism perturbation $\nu^{01}$ of $s_{\Kk\times[0,1]}|_{\Vv\times[0,1]}$ with boundary restrictions $\nu^{01}|_{\Vv\times\{\al\}}=\nu^\al$ for $\al=0,1$.
Here we use the fact that $\de_{\Vv\times[0,1]}=\de_\Vv>\de$ by Lemma~\ref{le:admin}~(ii). 
Moreover, by Lemma~\ref{le:cK}~(iii) the orientation of $\Kk$ induces an orientation of $\Kk\times[0,1]$, whose restriction to the boundaries $\p^\al( \Kk\times[0,1]) =\Kk$ equals the given orientation on $\Kk$. 
Finally, Lemma~\ref{le:czeroS0} together with Proposition~\ref{prop:orient1}~(ii) imply that $|\bZ_{\nu^{01}}|$ is the required oriented cobordism from $|\bZ_{\nu^{01}|_{\Vv\times\{0\}}}|=|\bZ_{\nu^0}|$ to $|\bZ_{\nu^{01}|_{\Vv\times\{1\}}}|=|\bZ_{\nu^1}|$.

\MS\NI
{\bf Step 2:} {\it 
For a fixed oriented, metric, tame Kuranishi atlas $(\Kk,d)$ and nested reductions $\Cc\sqsubset\Vv$, the cobordism class of $|\bZ_\nu|$ is independent of the choice of strongly adapted perturbation~$\nu$ with respect to $\Cc\sqsubset\Vv$.}

\MS
To prove this we fix $(\Kk,d)$ and $\Cc\sqsubset\Vv$ and consider two $(\Vv,\Cc,\de^\al,\si^\al)$-adapted perturbations $\nu^\al$ for $0<\de^\al<\de_\Vv$, $0<\si^\al\leq \si_{\rm rel}(\de^\al,\Vv\times[0,1],\Cc\times[0,1])$, and  $\al=0,1$.
Then we need to find an oriented cobordism $|\bZ_{\nu^0}| \sim |\bZ_{\nu^1}|$.
To do this, first note that we evidently have $\de:=\max(\de^0,\de^1)<\de_\Vv=\de_{\Vv\times[0,1]}$  
by Lemma~\ref{le:admin}~(ii) with respect to the product metric on $|\Kk|\times[0,1]$.
Next, we have $\de=\de^\al$ for some $\al\in\{0,1\}$ and hence $\si_{\rm rel}(\de,\Vv\times[0,1],\Cc\times[0,1])=\si_{\rm rel}(\de^\al,\Vv\times[0,1],\Cc\times[0,1])\ge\si^\al\geq \min(\si^0,\si^1)$. 
Now choose $\si \leq \min\{ \si^0,\si^1, \si_{\rm rel}(\de,\Vv\times[0,1],\Cc\times[0,1]) \}$. Then Proposition~\ref{prop:ext2}~(i) provides an admissible, precompact, transverse cobordism perturbation $\nu^{01}$ of $s_{\Kk\times[0,1]}|_{\Vv\times[0,1]}$, whose restrictions $\nu^{01}|_{\Vv\times\{\al\}}$ for $\al=0,1$ are $(\Vv,\Cc,\de,\si)$-adapted perturbations of $s_\Kk|_\Vv$.
Since $\de^\al\leq \de$ and $\si \leq \si^\al \leq \si_{\rm rel}(\de^\al,\Vv\times[0,1],\Cc\times[0,1])$, they are also $(\Vv,\Cc,\de^\al,\si^\al)$-adapted.
Then, as in Step~1, the perturbed zero set $|\bZ_{\nu^{01}}|$ is an oriented cobordism from $|\bZ_{\nu^{01}|_{\Vv\times\{0\}}}|$ to $|\bZ_{\nu^{01}|_{\Vv\times\{1\}}}|$.
Moreover, for fixed $\al\in\{0,1\}$ both the restriction $\nu^{01}|_{\Vv\times\{\al\}}$ and the given perturbation $\nu^\al$ are $(\Vv,\Cc,\de^\al,\si^\al)$-adapted, so that Step 1 provides cobordisms $|\bZ_{\nu^0}|  \sim |\bZ_{\nu^{01}|_{\Vv\times\{0\}}}|$ and $|\bZ_{\nu^{01}|_{\Vv\times\{1\}}}| \sim |\bZ_{\nu^1}|$. By transitivity of the cobordism relation this proves $|\bZ_{\nu^0}| \sim |\bZ_{\nu^1}|$ as claimed.

\MS\NI
{\bf Step 3:} {\it
For a fixed oriented, tame Kuranishi atlas $\Kk$, the cobordism class of $|\bZ_\nu|$ is independent of the choice of admissible metric and strongly adapted perturbation~$\nu$.
}

\MS
To prove this we fix $\Kk$ and consider two $(\Vv^\al,\Cc^\al,\de^\al,\si^\al)$-adapted perturbations $\nu^\al$ with respect to nested reductions $\Cc^\al\sqsubset\Vv^\al$ and constants $0<\de^\al<\de_\Vv$, $0<\si^\al\leq \si_{\rm rel}(\de^\al,\Vv\times[0,1],\Cc\times[0,1])$ for $\al=0,1$.
To find an oriented cobordism $|\bZ_{\nu^0}| \sim |\bZ_{\nu^1}|$ we begin by using Proposition~\ref{prop:metcoll}~(iv) to find an admissible metric $d$ on $|\Kk\times[0,1]|$ with 
$d|_{|\Kk|=|\p^\al(\Kk\times[0,1])|}=d^\al$.
Next, we use Proposition~\ref{prop:cov2}~(c) and Lemma~\ref{le:delred} to find a nested cobordism reduction $\Cc\sqsubset\Vv$ of $\Kk\times[0,1]$ with $\p^\al\Cc=\Cc^\al$ and $\p^\al\Vv=\Vv^\al$.
If we now pick any $0<\de<\de_\Vv$ smaller than the collar width of $d$, $\Vv$, and $\Cc$, then we automatically have $\de<\de_{\Vv^\al}$ Lemma~\ref{le:admin}~(iii).
Then for any $0<\si\leq\si_{\rm rel}(\de,\Vv,\Cc)$ Proposition~\ref{prop:ext2}~(i) provides an admissible, precompact, transverse cobordism perturbation $\nu^{01}$ of $s_{\Kk\times[0,1]}|_\Vv$ whose restrictions $\nu^{01}|_{\p^\al\Vv}$ for $\al=0,1$ are $(\Vv^\al,\Cc^\al,\de,\si)$-adapted perturbations of $s_{\Kk}|_{\Vv^\al}$.
As in Step 1, the perturbed zero set $|\bZ_{\nu^{01}}|$ is an oriented cobordism from $|\bZ_{\nu^{01}|_{\p^0\Vv}}|$ to $|\bZ_{\nu^{01}|_{\p^1\Vv}}|$. 
Moreover, we may pick $\si$ such that $\si\leq\si_{\rm rel}(\de,\Vv^\al\times[0,1],\Cc^\al\times[0,1])$ for $\al=0,1$, then each $\nu^{01}|_{\p^\al\Vv}$ is strongly adapted with respect to $d^\al$ and $\Cc^\al\sqsubset\Vv^\al$. 
Then Step 2 applies for $\al=0$ as well as $\al=1$ to provide cobordisms $|\bZ_{\nu^0}|  \sim |\bZ_{\nu^{01}}|_{\p^0\Vv}|$ and $|\bZ_{\nu^{01}}|_{\p^1\Vv}| \sim |\bZ_{\nu^1}|$, which proves the claim by transitivity.

\MS\NI
{\bf Step 4:} 
{\it
Let $\Kk^{01}$ be an oriented, additive, weak Kuranishi cobordism, and for $\al=0,1$ let $\nu^\al$ be strongly adapted perturbations of some preshrunk tame shrinking $\Kk_{\rm sh}^\al$ of $\p^\al\Kk^{01}$ with respect to some choice of admissible metric on $|\p^\al\Kk^{01}|$. Then there is an oriented cobordism of compact manifolds $|\bZ_{\nu^0}|  \sim |\bZ_{\nu^1}|$.
}

This is proven along the lines of (i) and (ii) by first using Proposition~\ref{prop:cobord2} to find a preshrunk tame shrinking $\Kk_{\rm sh}$ of $\Kk$ with $\p^\al\Kk_{\rm sh}=\Kk^\al_{\rm sh}$, and an admissible metric $d$ on $|\Kk_{\rm sh}|$. If we equip $\Kk_{\rm sh}$ with the orientation induced by $\Kk$, then by Lemma~\ref{le:cK} the induced boundary orientation on $\p^\al\Kk_{\rm sh}=\Kk^\al_{\rm sh}$ agrees with that induced by shrinking from $\Kk^\al$.
Next, Proposition~\ref{prop:cov2}~(c) provides a reduction $\Vv$ of $\Kk_{\rm sh}$, and by Lemma~\ref{le:delred}~(ii) we find a nested cobordism reduction $\Cc\sqsubset \Vv$. 
Now we may apply Proposition~\ref{prop:ext2}~(i) with 
$$
\si\;=\; \min\bigl\{ \si_{\rm rel}(\de,\p^0\Vv\times[0,1],\p^0\Cc\times[0,1]) , \si_{\rm rel}(\de,\p^1\Vv\times[0,1],\p^1\Cc\times[0,1]) , \si_{\rm rel}(\de,\Vv,\Cc) \bigr\}
$$
to find an admissible, precompact, transverse cobordism perturbation $\nu^{01}$ of $s_{\Kk_{\rm sh}}|_\Vv$, whose restrictions $\nu^{01}|_{\p^\al\Vv}$ for $\al=0,1$ are $(\p^\al\Vv,\p^\al\Cc,\de,\si)$-adapted perturbations of $s_{\Kk^\al_{\rm sh}}|_{\p^\al\Vv}$. In particular, these are strongly adapted by the choice of $\si$.
Also, as in the previous steps, $|\bZ_{\nu^{01}}|$ is an oriented cobordism from $|\bZ_{\nu|_{\p^0\Vv}}|$ to $|\bZ_{\nu|_{\p^1\Vv}}|$. 
Finally, Step 3 applies to the fixed oriented, tame Kuranishi atlases $\Kk^\al_{\rm sh}$ for fixed $\al\in\{0,1\}$ to provide cobordisms $|\bZ_{\nu^0}|  \sim |\bZ_{\nu^{01}}|_{\p^0\Vv}|$ and $|\bZ_{\nu^{01}}|_{\p^1\Vv}| \sim |\bZ_{\nu^1}|$. By transitivity, this finishes the proof of Theorem~\ref{thm:VMC1}.
\end{proof}

One possible definition of the virtual fundamental class (VFC) is as the cobordism class of the zero set $|\bZ_\nu|$
constructed in the previous theorem. If we think of this as an abstract manifold and hence as representing an element in the $D$-dimensional oriented cobordism ring, it contains rather little information.
These notions are barely sufficient for the basic constructions of e.g.\ Floer differentials $\p_F$ from counts of moduli spaces with $D=0$, and proofs of algebraic relations such as $\p_F\circ\p_F=0$ by cobordisms with $D=1$.
However in the case of interest to us, in which $X = \oMm_{g,k}(A,J)$ is the Gromov--Witten moduli space of $J$-holomorphic curves of genus $g$, homology class $A$, and with $k\geq 1$ marked points, we explained in Section~\ref{s:construct} how to construct the domains $U_I$ of the Kuranishi charts for $X$ to have elements that are $k$-pointed stable maps to $(M,\om)$, so that there are evaluation maps $\ev_I: U_I \to M^k$.
Further, the coordinate changes can be made compatible with these evaluation maps, and Kuranishi cobordisms can be constructed so that the evaluation maps extend over them. Hence, after shrinking to a tame Kuranishi atlas
(or cobordism) $\Kk_{sh}$, there is a continuous evaluation map
$$
\ev: |\Kk_{sh}| \to M^k
$$
both for the fixed tame shrinking used to define $|\bZ_\nu|$ and for any shrinking of a weak Kuranishi cobordism compatible with evaluation maps.  Therefore, for any admissible, precompact, transverse perturbation $\nu$, the map $\ev:|\bZ_\nu|\to M^k$ can be considered as a cycle (the virtual moduli cycle VMC) in the singular homology $H_D(M^k)$ of $M^k$, or even as a cycle in the bordism theory of $M^k$. Similarly, in this case a possible definition of the VFC is as the corresponding singular homology (or bordism) class in $M^k$.
Finally, we will explain how to interpret the VFC more intrinsically as an element in the rational \v{C}ech homology $\check{H}_D(X;\Q)$ of
the compact metrizable space $X$, i.e.\ in the Gromov-Witten example the moduli space itself.
As a first step, we associate to every oriented, metric, tame Kuranishi atlas a $D$-dimensional homology class in any open neighbourhood $\Ww\subset |\Kk|$ of $\io_\Kk(X)$ within the virtual neighbourhood.

For that purpose recall from \eqref{eq:Zinject} that for any precompact, transverse perturbation $\nu$ of $s_\Kk|_\Vv$, the inclusion $(s+\nu)^{-1}(0)\subset\Vv = \Obj_{\bB_\Kk|_\Vv}$ induces a continuous injection $i_\nu: |\bZ_\nu| \to |\Kk|$, which we now compose with the continuous bijection $|\Kk| \to (|\Kk|,d)$ from Lemma~\ref{le:metric} to obtain a continuous injection
\begin{equation}\label{ionu}
\io_\nu \,:\;  |\bZ_\nu| \;\longrightarrow\; \bigl(|\Kk|,d\bigr) .
\end{equation}
Since $ |\bZ_\nu|$ is compact and the restriction of the metric topology to the image $\io_\nu(|\bZ_\nu|)\subset(|\Kk|,d)$ is Hausdorff, this map is in fact a homeomorphism to its image; see Remark~\ref{rmk:hom}, and compare with Proposition~\ref{prop:zeroS0} which notes that $i_\nu:|\bZ_\nu|\to |\Kk|$ is a homeomorphism to its image.
If moreover $|\bZ_\nu|$ is oriented, then it has a fundamental class $\bigl[|\bZ_\nu|\bigr]\in H_D(|\bZ_\nu|)$.
Now we obtain a homology class by push forward into any appropriate subset of $(|\Kk|,d)$,
$$
[\io_\nu] :=  (\io_\nu)_* \bigl[|\bZ_\nu|\bigr] \in H_D(\Ww) \qquad\text{for} \quad \io_\nu(|\bZ_\nu|) \subset \Ww \subset  \bigl(|\Kk|,d\bigr) .
$$
Analogously, any precompact, transverse perturbation $\nu^{01}$ of a metric, tame Kuranishi cobordism $\bigl(\Kk^{01}, d^{01}\bigr)$ gives rise to a topological embedding
\begin{equation}\label{cionu}
\io_{\nu^{01}} \,:\;  |\bZ_{\nu^{01}}| \;\longrightarrow\; \bigl( |\Kk^{01}| , d^{01}\bigr) .
\end{equation}
Now by Lemma~\ref{le:czeroS0} the boundary of the cobordism $ |\bZ_{\nu^{01}}|$ has two disjoint (but not necessarily connected) components
$$
\p |\bZ_{\nu^{01}}| \;=\;  \p^0 |\bZ_{\nu^{01}}| \;\cup\;  \p^1|\bZ_{\nu^{01}}|, \qquad
\p^\al |\bZ_{\nu^{01}}| \,:=\;  \p^\al |\Kk^{01}|  \cap |\bZ_{\nu^{01}}| .
$$
In fact, we also showed there that the embeddings $J^\al:=\rho^\al(\cdot,\al): |\p^\al\Kk^{01}|  \to \p^\al|\Kk^{01}| \subset |\Kk^{01}|$ restrict to diffeomorphisms
$$
J^\al|_{|\bZ_{\nu^\al}|}= |j^\al|  \,:\; |\bZ_{\nu^\al}| \;\longrightarrow\;\p^\al|\bZ_{\nu^{01}}|
\qquad\text{with}\quad
\io_{\nu^{01}} \circ |j^\al| = J^\al \circ \io_{\nu^\al} ,
$$
where $\nu^\al:= \nu^{01}|_{\p^\al\Kk^{01}}$ are the restricted perturbations of the Kuranishi atlases $\p^\al\Kk^{01}$.
Moreover, Proposition~\ref{prop:orient1}~(ii) asserts that the boundary orientations on $\p^\al |\bZ_{\nu^{01}}|$ 
(which are induced by the orientation of $|\bZ_{\nu^{01}}|$ arising from the orientation of $\Kk^{01}$) are related to 
the orientations of $|\bZ_{\nu^\al}|$ (which are induced by the orientation of $\p^\al\Kk^{01}$ obtained by restriction from the orientation of $\Kk^{01}$) by
$$
|j^0| \,:\; |\bZ_{\nu^0}|^-  \;\overset{\cong}{\longrightarrow}\;  \p^0 |\bZ_{\nu^{01}}| \qquad
\text{and}\qquad
|j^1| \,:\;  |\bZ_{\nu^1}| \;\overset{\cong}{\longrightarrow}\; \p^1 |\bZ_{\nu^{01}}|.
$$
In terms of the fundamental classes this yields the identity
\begin{align*}
|j^1|_*\bigl[ |\bZ_{\nu^1}|\bigr]  - |j^0|_*\bigl[ |\bZ_{\nu^0}|\bigr] &\;=\; 
\bigl[\p^1 |\bZ_{\nu^{01}}|\bigr]  + \bigl[\p^0 |\bZ_{\nu^{01}}|\bigr] \\
&\;=\; \bigl[\p |\bZ_{\nu^{01}}|\bigr]  \;=\; \delta \bigl[|\bZ_{\nu^{01}}|\bigr] \;\in\; H_D( \p |\bZ_{\nu^{01}}|)
\end{align*}
for the boundary map $\delta: H_{D+1}(\bigl[ |\bZ_{\nu^{01}}|\bigr] , \p \bigl[ |\bZ_{\nu^{01}}|\bigr] ) \to H_D( \p |\bZ_{\nu^{01}}|)$ that is part of the long exact sequence for $\p |\bZ_{\nu^{01}}|\subset  |\bZ_{\nu^{01}}|$.
Inclusion to $|\bZ_{\nu^{01}}|$ now provides, by exactness of this sequence,
$|j^1|_*\bigl[ |\bZ_{\nu^1}|\bigr]  - |j^0|_*\bigl[ |\bZ_{\nu^0}|\bigr] = 0  \in H_D( |\bZ_{\nu^{01}}|)$.
Finally, we can push this forward with $\io_{\nu^{01}}$ to $H_D(|\Kk^{01}|)$ and use the fact that $\io_{\nu^{01}} \circ |j^\al| = J^\al \circ \io_{\nu^\al}$ to obtain
\begin{align*}
0 &\;=\; 
(\io_{\nu^{01}})_*|j^1|_*\bigl[ |\bZ_{\nu^1}|\bigr]  - (\io_{\nu^{01}})_*|j^0|_*\bigl[ |\bZ_{\nu^0}|\bigr] \\
&\;=\;
|J^1|_*(\io_{\nu^{1}})_*\bigl[ |\bZ_{\nu^1}|\bigr]  - |J^0|_*(\io_{\nu^{0}})_*\bigl[ |\bZ_{\nu^0}|\bigr] 
\;=\;
|J^1|_*\bigl[\io_{\nu^{1}}\bigr]  - |J^0|_*\bigl[\io_{\nu^{0}}\bigr] .
\end{align*}
The same holds in $H_D(\Ww^{01})$ for any subset $\Ww^{01}\subset\bigl(|\Kk^{01}|,d^{01}\bigr)$ that contains $\io_{\nu^{01}}(|\bZ_{\nu^{01}}|)$, that is
\begin{equation} \label{homologous}
J^0_* [\io_{\nu^0}] \;=\; J^1_* [\io_{\nu^1}] \;\in\; H_D(\Ww^{[0,1]}) \qquad\text{when} \quad \io_{\nu^{01}}(|\bZ_{\nu^{01}}|) \subset \Ww^{01} \subset |\Kk^{01}| .
\end{equation}
This will be crucial for proving independence of the VFC from choices.

In the case of a product cobordism $\Kk^{01}=\Kk\times[0,1]$ with product metric we can identify $|\Kk^{01}|\cong |\Kk|\times [0,1]$ so that \eqref{cionu} also induces a cycle  
\begin{equation} \label{pionu}
\pr_{ |\Kk|}\circ \io_{\nu^{01}} \, :\;  |\bZ_{\nu^{01}}|\; \longrightarrow\; (|\Kk|,d) ,
\end{equation}
whose boundary restrictions are $\io_{\nu^\al} \circ |j^\al|^{-1}$, so that the above argument directly gives
\begin{equation} \label{pomologous}
[ \io_{\nu^0} ] = [\io_{\nu^1} ] \in H_D(\Ww) \qquad\text{when} \quad \io_{\nu^{01}}(|\bZ_{\nu^{01}}|) \subset \Ww\times[0,1].
\end{equation}
Now we can associate a well define virtual fundamental class to any choice of open neighbourhood  $\Ww$ of $X$ in the virtual neighbourhood induced by a fixed tame Kuranishi atlas.

\begin{lemma}\label{le:VMC1}
Let $(\Kk,d)$ be an oriented, metric, tame Kuranishi atlas and let $\Ww\subset |\Kk|$ be an open subset with respect to the metric topology such that $\io_\Kk(X)\subset \Ww$. Then there exists a strongly adapted perturbation $\nu$ in a suitable reduction of $\Kk$ such that $\pi_\Kk((s+\nu)^{-1}(0))\subset \Ww$. More precisely, there exists $\nu$ adapted with respect to a nested reduction $\Cc\sqsubset\Vv$ such that $\pi_\Kk(\Cc)\subset\Ww$.
For any such perturbation, the inclusion of the perturbed zero set $\io_\nu : |\bZ_\nu|\hookrightarrow \Ww \subset (|\Kk|,d)$ defines a singular homology class
$$
A^{(\Kk,d)}_{\Ww} \,:=\; \bigl[\io_\nu : |\bZ_\nu| \to \Ww \bigr] \;\in\; H_D(\Ww)
$$
that is independent of the choice of reductions and perturbation.
\end{lemma}

\begin{proof}
To see that the required perturbations exist, choose any nested reduction $\Cc \sqsubset \Vv$ of $\Kk$.
Then we obtain a further precompact open set
$$
\Cc_\Ww \,:=\; \Cc \cap \pi_{\Kk_{\rm sh}}^{-1}(\Ww) \;\sqsubset\; \Vv
$$
which, after discarding components $C_I \cap \pi_{\Kk_{\rm sh}}^{-1}(\Ww_k)$ that have empty intersection with $s_I^{-1}(0)$, forms another nested reduction since $\io_\Kk(X)\subset\pi_\Kk(\Cc)\cap \Ww$.
Now Proposition~\ref{prop:ext} guarantees the existence of a $(\Vv,\Cc,\de,\si)$-adapted perturbation $\nu$ for sufficiently small $\de,\si>0$,
and by choice of $\si$ we can ensure that $\nu$ is also strongly adapted.
By Proposition~\ref{prop:zeroS0} and Proposition~\ref{prop:orient1}~(i) its perturbed zero set is an oriented manifold $|\bZ_\nu|$. Moreover, the image of $\io_\nu : |\bZ_\nu| \to (|\Kk|,d)$ is $\pi_\Kk\bigl((s+\nu)^{-1}(0)\bigr) \subset\pi_\Kk(\Cc) \subset \Ww$, so that by the discussion above $\io_\nu : |\bZ_\nu| \to \Ww$ defines a cycle $\bigl[\io_\nu \bigr] \in H_D(\Ww)$.

To prove independence of the choices, we need to show that $\bigl[\io_{\nu^0}\bigr]=\bigl[\io_{\nu^1}\bigr]$ for any given $(\Vv^\al,\Cc^\al,\de^\al,\si^\al)$-adapted perturbation $\nu^\al$ of $s_\Kk|_{\Vv^\al}$ with $\pi_\Kk(\Cc^\al)\subset \Ww$, by adapting Steps 1--3 in the proof of Theorem~\ref{thm:VMC1} so that the cycles $\io_{\nu^{01}}:|\bZ_{\nu^{01}}| \to |\Kk\times[0,1]|$ given by \eqref{cionu} for the respective cobordism perturbations $\nu^{01}$ of $s_{\Kk\times[0,1]}|_{\Vv^{[0,1]}}$ take values in $\Ww\times[0,1]\subset |\Kk\times[0,1]|$.
Note that here we use the product metric on $|\Kk|\times [0,1]$ so that $\Ww\times [0,1]$ is open.
Then in each step the composite map $\pr_{|\Kk|} \circ\io_{\nu^{01}} : |\bZ_{\nu^{01}}| \to |\Kk|$ takes values in $\Ww$, so that \eqref{pomologous} applies to give $\bigl[\io_{\nu^{01}|_{\p^0\Vv^{[0,1]}}}\bigr]=\bigl[\io_{\nu^{01}|_{\p^1\Vv^{[0,1]}}}\bigr]\in H_D(\Ww)$.
By transitivity of equality in $H_D(\Ww)$, Steps 1--3 then prove $\bigl[\io_{\nu^0}\bigr]=\bigl[\io_{\nu^1}\bigr]$.

In Steps 1 and 2, the required inclusion is automatic since the perturbations are constructed so that $|(s+\nu^{01})^{-1}(0)| \subset \pi_{\Kk\times[0,1]}(\Cc\times[0,1])\subset\Ww\times[0,1]$, where the second inclusion follows from $\pi_\Kk(\Cc)\subset\Ww$.
In Step 3, we have a fixed metric $d^0=d^1=d$ but are given nested reductions $\Cc^\al\sqsubset\Vv^\al$ for $\al=0,1$ with $\pi_\Kk(\Cc^\al)\subset\Ww$.
Then we first equip $|\Kk\times[0,1]|\cong |\Kk|\times[0,1]$ with the product metric and choose a nested cobordism reduction $\Cc^{[0,1]}\sqsubset \Vv^{[0,1]}$ such that $\p^\al \Cc^{[0,1]}=\Cc^\al$ and $\p^\al \Vv^{[0,1]}=\Vv^\al$, and then replace $\Cc^{[0,1]}$ by $\Cc': = \Cc^{[0,1]}\cap \pi_{\Kk\times [0,1]}^{-1}( \Ww\times [0,1])$, discarding components $C'_I$ with ${C'_I\cap s_I^{-1}(0)=\emptyset}$.
Note that this construction preserves the collar boundary $\p^\al\Cc'=\p^\al \Cc^{[0,1]}=\Cc^\al$ since $\pi_\Kk(\Cc^\al) \subset\Ww$. In fact, $\Cc'\sqsubset\Vv^{[0,1]}$ is another nested cobordism reduction since $\io_{\Kk\times[0,1]}(X\times[0,1]) \subset \Ww\times[0,1]$.
Using the nested reduction $\Cc'\sqsubset\Vv^{[0,1]}$ in choosing the cobordism perturbation $\nu$ then ensures that $\io_\nu: |\bZ_\nu|\to |\Kk\times[0,1]|=|\Kk|\times[0,1]$  takes values in $\pi_{\Kk\times [0,1]}(\Cc')\subset \Ww\times [0,1]$, as required to finish the proof.
\end{proof}

To construct the VFC as a homology class in $\io_\Kk(X)$ for tame Kuranishi atlases, and later in $X$,
we use rational \v{C}ech homology, rather than integral \v{C}ech or singular homology, because it has the following continuity property.

\begin{remark} \label{rmk:Cech}
Let $X$ be a compact subset of a metric space $Y$, and let
$(\Ww_k\subset Y)_{k\in\N}$ be a sequence of open subsets that is nested, $\Ww_k\subset\Ww_{k-1}$, such that $X=\bigcap_{k\in\N} \Ww_k$.
Then the system of maps $\check{H}_n(X;\Q)  \to \check{H}_n(\Ww_{k+1};\Q)\to \check{H}_n(\Ww_k;\Q)$
induces an isomorphism
$$
\check{H}_n(X;\Q) \;\overset{\cong}{\longrightarrow}\; \underset{\leftarrow }\lim\, \check{H}_n(\Ww_k;\Q)  .
$$

{\rm
To see that singular homology does not have this property, let $X\subset \R^2$ be the union of  the line segment $\{0\}\times [-1,1]$, the graph $\{(x,\sin \tfrac \pi x) \,|\, 0<x\le 1\}$, and an embedded curve joining $(0,1)$ to $(1,0)$ that is otherwise disjoint from the line segment and graph.
Then $H_1^{sing}(X;\Q) = 0$ since it is the abelianization of the trivial fundamental group.  However, $X$ has arbitrarily small neighbourhoods $U\subset\R^2$ with $H_1^{sing}(U;\Q) =\Q$.

Note that we cannot work with integral \v{C}ech homology since it does not even satisfy the exactness axiom (long exact sequence for a pair), because of problems with the inverse limit operation; see the discussion of \v{C}ech cohomology in Hatcher~\cite{Hat}, and \cite[Proposition~3F.5]{Hat} for properties of inverse limits.
However, rational \v{C}ech homology does satisfy the exactness axiom, and because it is dual to
\v{C}ech cohomology has the above stated continuity property by \cite[Ch.6~Exercises~D]{Span}.

Further, rational \v{C}ech homology equals rational singular homology for finite simplicial complexes.
Hence the fundamental class of a closed oriented $n$-manifold $M$ can be considered as
an element $[M]\in \check{H}_n(M;\Q)$ in rational \v{C}ech homology and therefore pushes forward under a continuous map $f:M\to X$ to a well defined element $f_*([M])\in \check{H}_n(X;\Q)$.
Note finally that if one wants an integral theory with this continuity property, the correct theory to use is the Steenrod homology theory developed in \cite{Mi}.
}
\end{remark}

Using this continuity property of rational \v{C}ech homology, we can finish the proof of Theorem~B.

\begin{thm}\label{thm:VMC2}
Let $\Kk$ be an oriented, additive weak Kuranishi atlas of dimension $D$ on a compact, metrizable space $X$.
\begin{enumerate}
\item
Let $\Kk_{\rm sh}$ be a preshrunk tame shrinking of $\Kk$ and $d$ an admissible metric on $|\Kk_{\rm sh}|$. Then there exists a nested sequence of open sets $\Ww_{k+1}\subset \Ww_k\subset \bigl(|\Kk_{\rm sh}|, d\bigr)$ such that $\bigcap_{k\in\N}\Ww_k = \io_{\Kk_{\rm sh}}(X)$.
Moreover, for any such sequence there is a sequence $\nu_k$ of strongly adapted perturbations of $s_{\Kk_{\rm sh}}$ with respect to nested reductions $\Cc_k\sqsubset\Vv_k$ such that $\pi_\Kk(\Cc_k)\subset\Ww_k$ for all $k$. Then the embeddings
$$
\io_{\nu_k} \,:\; |\bZ_{\nu_k}| \;\hookrightarrow \;   \Ww_k \;\subset\; \bigl(|\Kk_{\rm sh}|, d\bigr)
$$
induce a \v{C}ech homology class
$$
\underset{\leftarrow}\lim\, \bigl[ \,
\io_{\nu_k} \, \bigr] \;\in\; \check{H}_D\bigl(\io_{\Kk_{\rm sh}}(X);\Q \bigr),
$$
for the subspace $\io_{\Kk_{\rm sh}}(X) =|s_{\Kk_{\rm sh}}|^{-1}(0)$ of the metric space $\bigl(|\Kk_{\rm sh}|,d\bigr)$.
\vspace{0.05in}
\item
The bijection $|\psi_{\Kk_{\rm sh}}| = \io_{\Kk_{\rm sh}}^{-1}: \io_{\Kk_{\rm sh}}(X) \to X$ from Lemma~\ref{le:Knbhd1} is a homeomorphism with respect to the metric topology on $\io_{\Kk_{\rm sh}}(X)$ so that we can define the {\bf virtual fundamental class} of $X$ as the pushforward
$$
[X]^{\rm vir}_\Kk \,:=\; |\psi_{\Kk_{\rm sh}}|_* \bigl( \, \underset{\leftarrow}\lim\, [\, \io_{\nu_k} \,] \, \bigr)
\;\in\;
\check{H}_D(X;\Q) .
$$
It is independent of the choice of shrinkings, metric, nested open sets, reductions, and perturbations in (i),
and in fact depends on the weak Kuranishi atlas $\Kk$ on $X$ only up to oriented, additive cobordism.
\end{enumerate}
\end{thm}

\begin{proof}
The existence of shrinkings and metric is guaranteed by Theorem~\ref{thm:VMC1}~(i).
We then obtain nested open sets converging to $\io_{\Kk_{\rm sh}}(X)$ by e.g.\ taking the $\frac 1k$-neighbourhoods $\Ww_k = B_{\frac 1 k}(\io_{\Kk_{\rm sh}}(X))$.
Given any such nested open sets $(\Ww_k)_{k\in\N}$, the existence of adapted perturbations $\nu_k$ with respect to some nested reductions $\Cc_k\sqsubset\Vv_k$ with $\pi_{\Kk_{\rm sh}}(\Cc_k)\subset\Ww_k$ is proven in Lemma~\ref{le:VMC1}.
The latter also shows that the embeddings $\io_{\nu_k} : |\bZ_{\nu_k}|\to \Ww_k$ define homology classes $A^{(\Kk_{\rm sh},d)}_{\Ww_k} = [\io_{\nu_k} ]\in H_D(\Ww_k)$, which are independent of the choice of reductions $\Cc_k\sqsubset \Vv_k$ and adapted perturbations $\nu_k$.
In particular, the push forward $H_D(\Ww_{k+1})\to H_D(\Ww_k)$ by inclusion $\Ii_{k+1}:\Ww_{k+1}\to\Ww_k$ maps $A^{(\Kk_{\rm sh},d)}_{\Ww_{k+1}}=[\io_{\nu_{k+1}} ]$ to $A^{(\Kk_{\rm sh},d)}_{\Ww_k}$ since any adapted perturbation $\nu_{k+1}$ with respect to a nested reduction $\Cc_{k+1}\sqsubset\Vv_{k+1}$ with $\Cc_{k+1}\subset\pi_{\Kk_{\rm sh}}^{-1}(\Ww_{k+1})$ can also be used as adapted perturbation $\nu_k:=\nu_{k+1}$.
Then we obtain $\io_{\nu_k} = \Ii_{k+1} \circ \io_{\nu_{k+1}}$, and hence
$A^{(\Kk_{\rm sh},d)}_{\Ww_k}= [\io_{\nu_k}] = (\Ii_{k+1})_* [\io_{\nu_{k+1}}]$. This shows that the homology classes $A^{(\Kk_{\rm sh},d)}_{\Ww_k}$ form an inverse system and thus have a well defined inverse limit, completing the proof of (i),
$$
A^{(\Kk_{\rm sh},d)}_{(\Ww_k)_{k\in\N}}
\,:=\; \underset{\leftarrow}\lim\, \bigl[ \,\io_{\nu_k} \, \bigr]
\;\in\; \check{H}_D\bigl(\io_{\Kk_{\rm sh}}(X);\Q \bigr).
$$
This defines $A^{(\Kk_{\rm sh},d)}_{(\Ww_k)_{k\in\N}}$ as \v{C}ech homology class in the topological space $\bigl( \io_{\Kk_{\rm sh}}(X), d\bigr)$.

Towards proving (ii), recall first that $|\psi_{\Kk_{\rm sh}}|:\io_{\Kk_{\rm sh}}(X)\to X$ is a homeomorphism with respect to the relative topology induced from the inclusion $\io_{\Kk_{\rm sh}}(X)\subset|\Kk_{\rm sh}|$ by Lemma~\ref{le:Knbhd1}.
That the latter is equivalent to the metric topology on $\io_{\Kk_{\rm sh}}(X)$ follows as in Remark~\ref{rmk:hom} from the continuity of the identity map $|\Kk_{\rm sh}| \to \bigl(|\Kk_{\rm sh}|,d\bigr)$ (see Lemma~\ref{le:metric}), which restricts to a continuous bijection from the compact set $\io_{\Kk_{\rm sh}}(X)\subset|\Kk_{\rm sh}|$ to the Hausdorff space $\bigl( \io_{\Kk_{\rm sh}}(X), d\bigr)$, and thus is a homeomorphism.
To establish the independence of choices, we then argue as in Steps 2--4 in the proof of Theorem~\ref{thm:VMC1}~(iii), with Lemma~\ref{le:VMC1} playing the role of Step 1.

\MS\NI
{\bf Step 2:} {\it
Let $(\Kk,d)$ be an oriented, metric, tame Kuranishi atlas, and let $(\Ww^\al_k)_{k\in\N}$ for $\al=0,1$ be two nested sequences of open sets $\Ww^\al_{k+1}\subset \Ww^\al_k\subset \bigl(|\Kk|,d\bigr)$ whose intersection is  $\bigcap_{k\in\N}\Ww^\al_k = \io_{\Kk}(X)$.
Then we have $A^{(\Kk,d)}_{(\Ww^0_k)_{k\in\N}} =A^{(\Kk,d)}_{(\Ww^1_k)_{k\in\N}}$, and hence
$$
A^{(\Kk,d)} \,:=\; A^{(\Kk,d)}_{(\Ww_k)_{k\in\N}}  \;\in\; \check{H}_D\bigl(\io_\Kk(X);\Q \bigr) ,
$$
given by any choice of nested open sets $(\Ww_k)_{k\in\N}$ converging to $\io_\Kk(X)$, is a well defined \v{C}ech homology class.
}

\MS
To see this note that the intersection $\Ww_k:=\Ww^0_k\cap \Ww^1_k$ is another nested sequence of open sets with $\bigcap_{k\in\N}\Ww_k = \io_{\Kk}(X)$.
We may choose a sequence of adapted perturbations $\nu_k$ with respect to nested reductions $\Cc_k\sqsubset\Vv_k$ with $\pi_\Kk(\Cc_k)\subset\Ww_k$ to define $A^{(\Kk,d)}_{\Ww_k}=[\io_{\nu_k}]$. The perturbations $\nu_k$ then also fit the requirements for the larger open sets $\Ww_k^\al$ and hence the inclusions $\Ii^\al_k : \Ww_k\to\Ww^\al_k$ push $A^{(\Kk,d)}_{\Ww_k}=[\io_{\nu_k}]\in H_D(\Ww_k)$ forward to $A^{(\Kk,d)}_{\Ww^\al_k}=[\Ii^\al_k\circ\io_{\nu_k}]\in H_D(\Ww^\al_k)$. Hence, by the definition of the inverse limit, we have equality
$$
A^{(\Kk,d)}_{(\Ww^0_k)_{k\in\N}} \;=\; A^{(\Kk,d)}_{(\Ww_k)_{k\in\N}} \;=\; A^{(\Kk,d)}_{(\Ww^1_k)_{k\in\N}}
 \;\in\; \check{H}_D\bigl(\io_{\Kk}(X);\Q \bigr) .
$$

\MS\NI
{\bf Step 3:} {\it
Let $\Kk$ be an oriented, metrizable, tame Kuranishi atlas with two admissible metrics $d^0,d^1$.
Then we have $A^{(\Kk,d^0)}= A^{(\Kk_,d^1)}$, and hence
$$
[X]^{\rm vir}_\Kk \,:=\; |\psi_\Kk|_* A^{(\Kk_{\rm sh},d)} \;\in\; \check{H}_D\bigl(X;\Q \bigr) ,
$$
given by any choice of metric, is a well defined \v{C}ech homology class.
}

\MS
As in Step 3 of Theorem~\ref{thm:VMC1}, we find an admissible collared metric $d$ on $|\Kk\times [0,1]|$ with $d|_{|\Kk|\times\{\al\}}=d^\al$.
Next, we proceed exactly as in the following Step 4 in the special case $\Kk_{\rm sh}=\Kk\times[0,1]$ and $\Kk^0_{\rm sh}=\Kk^1_{\rm sh}=\Kk$ to find strongly admissible perturbations $\nu^\al_k$ of $(|\Kk|,d^\al)$ that define the \v{C}ech homology classes $A^{(\Kk,d^\al)} =\underset{\leftarrow}\lim\, \bigl[ \,\io_{\nu^\al_k} \, \bigr]  \in \check H_D(\io_\Kk(X))$ and satisfy the identity
$$
J^0_*\,\Bigl(  \underset{\leftarrow}\lim\, \bigl[ \io_{\nu^0_k} \bigr] \Bigr)\;=\; J^1_* \,\Bigl(\underset{\leftarrow}\lim\, \bigl[ \io_{\nu^1_k} \bigr] \Bigr)
\; \in\; \check H_D( \io_{\Kk_{\rm sh}}(X\times[0,1]) )
$$
with the topological embeddings $J^\al=\rho^\al(\cdot,\al): ( |\Kk| , d^\al ) \to ( |\Kk_{\rm sh}|, d)$.
To proceed we claim that the push forwards by $J^\al$ restrict to the same isomorphism
\begin{equation}\label{J01}
\bigl(J^0\big|_{\io_\Kk(X)}\bigr)_* = \bigl(J^1\big|_{\io_\Kk(X)}\bigr)_*
 \;: \; \check H_D( \io_\Kk(X) ) \;\longrightarrow\; \check H_D( \io_{\Kk\times[0,1]}(X\times[0,1]) )
\end{equation}
on the compact set $\io_\Kk(X)$,
on which the two metric topologies induced by $d^0,d^1$ are the same,
since they both agree with the relative topology from $\io_\Kk(X)\subset|\Kk|$.
Indeed, the restrictions $J^\al\big|_{\io_\Kk(X)}$ for $\al = 0,1$ are homotopic  via the
continuous family of maps $J^t :\io_\Kk(X)\to \io_{\Kk_{sh}}(X\times [0,1])$ that arises from the continuous family of maps\footnote
{
We pointed out after
Example~\ref{ex:mtriv} that the metric
topology
on $|\Kk\times [0,1]|$ may not
be a product topology in the canonical identification with $|\Kk|\times [0,1]$.
In fact, the metrics $d^0$ and $d^1$ may well
induce different topologies on $|\Kk|$.
We avoid these issues by homotoping maps to $X\times[0,1]$, which
always has the product topology by Lemma~\ref{le:Knbhd1} and the remarks just before Step 2.
}
$I^t :X\to X\times [0,1]$, $x\mapsto (x,t)$
by composition with the embeddings $\io_\Kk$ and $\io_{\Kk_{sh}}$, i.e.\
\[
J^t \,:
\xymatrix{
\io_\Kk(X) \ar@{->}[r]^{\quad |\psi_\Kk|} & X \ar@{->}[r]^{I^t\quad\;\;} & X\times [0,1] \ar@{->}[r]^{\io_{\Kk_{sh}}\quad\;\;} &\io_{\Kk_{sh}}(X\times [0,1]).
}
\]
These maps are continuous because $\io_\Kk=|\psi_\Kk|^{-1}$ and similarly $\io_{\Kk_{sh}}$ are homeomorphisms to their image with respect to the metric topology by the argument at the beginning of
the proof of (ii).
Moreover,
each $J^t$ is a homotopy equivalence
because, up to homeomorphisms, it equals to the homotopy equivalence $I^t$.
This proves \eqref{J01}, which we can then use to
deduce the claimed identity
$$
A^{(\Kk,d^0)} \;=\;
\underset{\leftarrow}\lim\, \bigl[ \,
\io_{\nu^0_k} \, \bigr]
 \;=\;
\underset{\leftarrow}\lim\, \bigl[ \,
\io_{\nu^1_k} \, \bigr]
\;=\;A^{(\Kk,d^1)}
\quad\in \check H_D( \io_\Kk(X) ).
$$

\MS\NI
{\bf Step 4:} {\it
Let $\Kk^{01}$ be an oriented, additive, weak Kuranishi cobordism, and let $\Kk_{\rm sh}^\al$ be preshrunk tame shrinkings of $\p^\al\Kk^{01}$ for $\al=0,1$. Then we have
$$
[X]^{\rm vir}_{\Kk_{\rm sh}^0}=[X]^{\rm vir}_{\Kk_{\rm sh}^1}.
$$
}

\MS
As in Step 4 of Theorem~\ref{thm:VMC1}, we find a preshrunk tame shrinking $\Kk_{\rm sh}$ of $\Kk^{01}$ with $\p^\al\Kk_{\rm sh}=\Kk^\al_{\rm sh}$, and an admissible collared metric $d$ on $|\Kk_{\rm sh}|$.
We denote its restrictions to $|\Kk^\al_{\rm sh}|$ by $d^\al:= d|_{|\p^\al\Kk_{\rm sh}|}$.

Next, we proceed as in Lemma~\ref{le:VMC1} by choosing a nested cobordism reduction $\Cc \sqsubset \Vv$ of $\Kk_{\rm sh}$ and constructing nested cobordism reductions $\Cc_k \sqsubset \Vv$ by
$$
\Cc_k \,:=\; \Cc \cap \pi_{\Kk_{\rm sh}}^{-1}\bigl(\Ww^{01}_k) \;\sqsubset\; \Vv\qquad\text{with}\quad \Ww^{01}_k: =
B_{\frac 1k}(\io_{\Kk_{\rm sh}}(X\times[0,1]) \bigr) \;\subset\;|\Kk_{\rm sh}|,
$$
in addition discarding components $C_k\cap V_I$ that have empty intersection with $s_I^{-1}(0)$.
Indeed, each $\Ww^{01}_k$ and hence $\Cc_k$ is collared by Example~\ref{ex:mtriv}, with boundaries given by the $\frac 1k$-neighbourhoods $\p^\al \Ww^{01}_k = B_{\frac 1k}^{d^\al}(\io_{\Kk^\al_{\rm sh}}(X) \bigr)\subset|\Kk^\al_{\rm sh}|$ with respect to the metrics $d^\al$ on $|\Kk^\al_{\rm sh}|$.
With that, Proposition~\ref{prop:ext2}~(i) guarantees the existence of admissible, precompact, transverse cobordism perturbations $\nu^{01}_k$ with $|(s + \nu^{01}_k)^{-1}(0)| \subset \Ww^{01}_k$, and with boundary restrictions $\nu^\al_k:= \nu^{01}_k|_{\p^\al\Vv}$ that are strongly admissible perturbations of $(\Kk^\al_{\rm sh},d^\al)$ for $\al=0,1$.
Note here that these boundary restrictions satisfy the requirements of (i) since $\bigcap_{k\in\N} \p^\al \Ww^{01}_k  = \io_{\Kk^\al_{\rm sh}}(X)$, thus they define the \v{C}ech homology classes
$$
A^{(\Kk^\al_{\rm sh},d^\al)} \;=\;
\underset{\leftarrow}\lim\, \bigl[ \,
\io_{\nu^\al_k} \, \bigr] \;\in\; \check{H}_D\bigl(\io_{\Kk^\al_{\rm sh}}(X);\Q \bigr) .
$$
On the other hand, the homology classes $J^\al_* \bigl[ \io_{\nu^\al_k}\bigr]$ also form two inverse systems in $H_D(|\Kk_{\rm sh}|)$, and as in \eqref{homologous} the chains $\io_{\nu^{(01)}_k}: |\bZ_{\nu^{(01)}_k}| \to \Ww^{01}_k$ induce identities in the singular homology of $\Ww^{01}_k$,
$$
J^0 _* \bigl[ \io_{\nu_k^0} \bigr] \;=\; J^1_* \bigl[ \io_{\nu^1_k} \bigr] \; \in\; H_D(\Ww^{01}_k) ,
$$
with the topological embeddings $J^\al=\rho^\al(\cdot,\al): ( |\Kk^\al_{\rm sh}| , d^\al ) \to ( |\Kk_{\rm sh}|, d)$.
Thus taking the inverse limit -- which commutes with push forward -- we obtain
$$
J^0_*\,\Bigl(  \underset{\leftarrow}\lim\, \bigl[ \io_{\nu^0_k} \bigr] \Bigr)\;=\; J^1_* \,\Bigl(\underset{\leftarrow}\lim\, \bigl[ \io_{\nu^1_k}  \bigr] \Bigr)
\; \in\; \check H_D( \io_{\Kk_{\rm sh}}(X\times[0,1]) ) .
$$
So further push forward with the inverse $|\psi_{\Kk_{\rm sh}}|$ of the homeomorphism $\io_{\Kk_{\rm sh}}$ implies
$$
\bigl(|\psi_{\Kk_{\rm sh}}| \circ J^0\bigr)_*\,\Bigl(  \underset{\leftarrow}\lim\, \bigl[ \io_{\nu^0_k} \bigr] \Bigr)\;=\; \bigl(|\psi_{\Kk_{\rm sh}}| \circ J^1\bigr)_* \,\Bigl(\underset{\leftarrow}\lim\, \bigl[ \io_{\nu^1_k}  \bigr] \Bigr)
\; \in\; \check H_D(X\times[0,1] ) ,
$$
where the homeomorphism $|\psi_{\Kk_{\rm sh}}|$ is related to the analogous $|\psi_{\Kk^\al_{\rm sh}}| : \io_{\Kk^\al_{\rm sh}}(X) \to X$ by $J^\al$ and the embedding $I^\al:X\to X\times \{\al\}$ by
$$
|\psi_{\Kk_{\rm sh}}| \circ J^\al \big|_{\io_{\Kk_{\rm sh}}(X)}
\;=\;
I^\al \circ |\psi_{\Kk^\al_{\rm sh}}| .
$$
Now $I^0_* = I^1_* : \check H_D( X) \to \check H_D( X\times[0,1] )$ are the same isomorphisms,
because the two maps
$I^0, I^1$ are both homotopy equivalences and homotopic to each other.
Hence the equality of
$I^0_*  |\psi_{\Kk^0_{\rm sh}}| _* \, \underset{\leftarrow}\lim\, \bigl[ \io_{\nu^0_k}\bigr]
= I^1_*  |\psi_{\Kk^1_{\rm sh}}| _* \, \underset{\leftarrow}\lim\, \bigl[ \io_{\nu^1_k}\bigr]$ in $\check H_D(X\times[0,1] )$ implies as claimed
$$
[X]^{\rm vir}_{\Kk^0_{\rm sh}} \;=\;
|\psi_{\Kk^0_{\rm sh}}| _* \,\Bigl(  \underset{\leftarrow}\lim\, \bigl[ \io_{\nu^0_k} \bigr] \Bigr)
\;=\;
|\psi_{\Kk^1_{\rm sh}}| _* \,\Bigl(  \underset{\leftarrow}\lim\, \bigl[ \io_{\nu^1_k}  \bigr] \Bigr)
\;=\;
[X]^{\rm vir}_{\Kk^1_{\rm sh}}.
$$
This proves Step 4.

\MS
Finally, Step 4 implies uniqueness of the virtual fundamental cycle $[X]^{\rm vir}_\Kk\in \check H_D( X) $ for an oriented, additive weak Kuranishi atlas $\Kk$, since for any two choices of preshrunk tame shrinkings $\Kk^\al_{\rm sh}$ of $\Kk$ we can apply Step 4 to $\Kk^{01}=\Kk\times[0,1]$ to obtain $[X]^{\rm vir}_{\Kk_{\rm sh}^0}=[X]^{\rm vir}_{\Kk_{\rm sh}^1}$.
Moreover, given cobordant oriented, additive, weak Kuranishi atlases $\Kk^0,\Kk^1$ there exists by assumption an oriented, additive, weak Kuranishi cobordism  $\Kk^{01}$ with $\p^\al\Kk^{01}=\Kk^\al$.
If we pick any preshrunk tame shrinkings  $\Kk_{\rm sh}^\al$ of $\Kk^\al$ to define $[X]^{\rm vir}_{\Kk^\al}=[X]^{\rm vir}_{\Kk_{\rm sh}^\al}$, then Step 4 implies the claimed uniqueness under cobordism,
$$
[X]^{\rm vir}_{\Kk^0}\;=\;[X]^{\rm vir}_{\Kk_{\rm sh}^0}\;=\;[X]^{\rm vir}_{\Kk_{\rm sh}^1}\;=\;[X]^{\rm vir}_{\Kk^1}.
$$
This completes the proof of Theorem~\ref{thm:VMC2}.
\end{proof}

\bibliographystyle{alpha}

\appendix

\section{Some comments on recent discussions} \label{app}

Unfortunately, the topic of regularization and Kuranishi structures has recently appeared more like a ``political mine field" than a research question in need of clarification. We are working on taking the non-mathematical parts of this discussion offline since that seems much more appropriate to us.
At this point in time, however, we still feel the need to clarify some impressions given by \cite{FOOO12}, and will thus comment briefly on the recent discussions.

\begin{itemlist}
\item
In connection with talks at IAS, Princeton in March 2012 --- as our work was nearing completion --- the second author raised some basic questions (see below) in a small online discussion group including the cited authors. The purpose of these questions was to pinpoint some foundational issues in the work of Fukaya--Ono \cite{FO} and the subsequent work of Fukaya--Oh--Ohta--Ono \cite{FOOO} and Joyce \cite{J}. 
During this discussion, and with our feedback on various versions, Fukaya et al developed a revision of their approach which can now be found in the mathematical parts of \cite{FOOO12}.
Our manuscript was sent to the group at several stages prior to posting on the arxiv. However, we only learned from the arxiv about \cite{FOOO12} and its criticism as well as portrayal of the discussion. 
We suggest that -- should these private email communications be of scientific interest -- then a complete transcript ought to be released, with permission of all authors.

\item
In \cite{FOOO12}, the authors have reworked many of the details of the 
foundations of their approach to Kuranishi structures.  In particular, as in \cite{FOOO}, they no longer use germs. They also made significant changes to the definition of a good coordinate system in order to deal with issues mentioned in Section~\ref{ss:top} and give a much more detailed construction of the VMC. 
(Here  \cite{FO} only gives a brief analogy with the construction for orbifold bundles, in which e.g.\ the allowed size of perturbation $\frac{\epsilon}{10}$ is not shown to be positive, and Hausdorffness resp.\ compactness are not addressed. Note however, that the Kuranishi setting does not immediately induce an ambient space, let alone a locally compact Hausdorff space.)
We have read \cite{FOOO12} and its earlier versions to some extent and 
could imagine that their definitions are now adequate to prove the results concerning the topic of this paper, i.e.\ smooth Kuranishi structures with trivial isotropy.  However their discussion of some foundational issues, for example the role of reductions, shrinkings and cobordisms, and the different topologies on the Kuranishi quotient neighbourhood, are still so convoluted or inexplicit that we were unable to verify all proofs.
Similar comments apply to the construction of sum charts for Gromov-Witten moduli spaces.
The further issues raised by our questions were discussed in much less detail in the email group. In particular, we cannot comment on the issues of smoothness of gluing and $S^1$-equivariant regularization.
More to the point, we feel that it is not our place to referee \cite{FOOO12}. Instead, there should be a wider engagement with these issues in the community.

\item
We have not removed criticisms of the arguments in \cite{FO, FOOO}, since for many years these have been the only available references on this topic (and still are the only published sources).
So we think it important to give a coherent account of the construction in the simplest possible case, showing where arguments have been lacking and how one might hope to fill them.  

\item
There are still some significant differences between the  notion of a Kuranishi structure in \cite{FOOO12} and ours. 
To clarify these, we have changed the name of the object we construct, calling it a ``Kuranishi atlas" instead of a Kuranishi structure. We have also added some explanatory remarks to the beginning of this paper (cf.\ the paragraph in Section 1 called ``Relation to other Kuranishi notions") and have rewritten Remark~\ref{rmk:otherK}.
We will comment more on this in \cite{MW:ku2}, once we have extended our definitions to the case with isotropy, since in this case the approaches diverge further.
We believe that comparisons of the ease of the different approaches only make sense once their rigor is established and hope that a refereeing process for all Kuranishi type approaches can do the latter.

\item
Finally, we will provide the list of questions that we posed initially, since we hope that these may serve as guiding questions for anyone who wishes to evaluate the literature for themselves.
\end{itemlist}

\NI
{\bf 1.)} Please clarify, with all details, the definition of a Kuranishi structure. And could you confirm that a special case of your work proves the following?
\begin{enumerate}
\item
The Gromov-Witten moduli space $\oMm_1(J,A)$ of $J$-holomorphic curves of genus 0, fixed homology class $A$, with 1 marked point has a Kuranishi structure.
\item
For any compact space $X$ with Kuranishi structure and continuous map $f:X\to M$ to a manifold $M$ (which suitably extends to the Kuranishi structure), there is a well defined $f_*[X]^{\rm vir}\in H_*(M)$.
\end{enumerate}

\MS\NI
{\bf 2.)}  The following seeks to clarify certain parts in the definition of Kuranishi structures and the construction of a cycle.
\begin{enumerate}
\item
What is the precise definition of a germ of coordinate change?
\item
What is the precise compatibility condition for this germ with respect to different choices of representatives of the germs of Kuranishi charts?
\item
What is the precise meaning of the cocycle condition?
\item
What is the precise definition of a good coordinate system?
\item
How is it constructed from a given Kuranishi structure?
\item
Why does this construction satisfy the cocycle condition?
\end{enumerate}

\MS\NI
{\bf 3.)}  Let $X$ be a compact space with Kuranishi structure and good coordinate system. Suppose that in each chart the isotropy group $\Ga_p=\{{\rm id}\}$ is trivial and $s^\nu_p:U_p\to E_p$ is a transverse section. What further conditions on the $s^\nu_p$ do you need (and how do you achieve them) in order to ensure that the perturbed zero set $X^\nu=\cup_p (s^\nu_p)^{-1}(0) / \sim$ carries a global triangulation, in particular
\begin{enumerate}
\item $X^\nu$ is compact,
\item $X^\nu$ is Hausdorff,
\item $X^\nu$ is closed, i.e.\ if $X^\nu=\bigcup_n \Delta_n$ is a triangulation then $\sum_n f(\partial \Delta_n) = \emptyset$.
\end{enumerate}

\MS\NI
{\bf 4.)} For the Gromov-Witten moduli space $\oMm_1(J,A)$ of $J$-holomorphic curves of genus 0 with 1 marked point, suppose that $A\in H_2(M)$ is primitive so that $\oMm_1(J,A)$ contains no nodal or multiply covered curves.
\begin{enumerate}
\item
Given two Kuranishi charts $(U_p, E_p, \Ga_p=\{{\rm id}\}, \ldots)$ and $(U_q, E_q, \Ga_q=\{{\rm id}\},\ldots)$ with overlap at $[r]\in\oMm_1(J,A)$, how exactly is a sum chart $(U_r,E_r,\ldots)$ with $E_r \simeq E_p\times E_q$ constructed?
\item
How are the embeddings $U_p \supset U_{pr} \hookrightarrow U_r$ and $U_q \supset U_{qr} \hookrightarrow U_r$ constructed?
\item
How is the cocycle condition proven for triples of such embeddings?
\end{enumerate}

\MS\NI
{\bf 5.)} How is equality of Floer and Morse differential for the Arnold conjecture proven?
\begin{enumerate}
\item
Is there an abstract construction along the following lines: Given a compact topological space $X$ with continuous, proper, free $S^1$-action, and a Kuranishi structure for $X/S^1$ of virtual dimension $-1$, there is a Kuranishi structure for $X$ with $[X]^{\rm vir}=0$.
\item
How would such an abstract construction proceed?
\item
Let $X$ be a space of Hamiltonian Floer trajectories between critical points of index difference $1$, in which breaking occurs (due to lack of transversality).
How is a Kuranishi structure for $X/S^1$ constructed?
\item
If the Floer differential is constructed by these means, why is it chain homotopy equivalent to the Floer differential for a non-autonomous Hamiltonian?
\end{enumerate}

\end{document}